\documentclass[11pt]{amsbook}
\usepackage[utf8]{inputenc}

\usepackage{amsxtra,amssymb,amsmath,amscd,url,enumitem}
\usepackage[refpage,noprefix]{nomencl}

\usepackage{csquotes}

\DeclareRobustCommand{\gobblefive}[5]{}
\newcommand*{\SkipTocEntry}{\addtocontents{toc}{\gobblefive}}

\usepackage{verbatim}

\setenumerate{itemsep=5pt}
\setitemize{itemsep=5pt}

\newlist{enumth}{enumerate}{1}
\setlist[enumth]{label=\emph{(\arabic*)}, ref=(\arabic*)}

\usepackage[utf8]{inputenc}
\usepackage[russian, english]{babel}
\usepackage{fullpage}
\usepackage{array}
\usepackage{scrtime}
\usepackage[colorlinks]{hyperref}
\usepackage{datetime}
\usepackage[all]{xy}
\usepackage{graphicx}
\usepackage{tikz-cd}
\usepackage{mathrsfs}
\usepackage{xspace}

\usepackage{tensor}
\usepackage{braket}

\usepackage{todonotes}
\usepackage{tensor}

\shortdate

\DeclareMathSymbol{A}{\mathalpha}{operators}{`A}%
\DeclareMathSymbol{B}{\mathalpha}{operators}{`B}%
\DeclareMathSymbol{C}{\mathalpha}{operators}{`C}%
\DeclareMathSymbol{D}{\mathalpha}{operators}{`D}%
\DeclareMathSymbol{E}{\mathalpha}{operators}{`E}%
\DeclareMathSymbol{F}{\mathalpha}{operators}{`F}%
\DeclareMathSymbol{G}{\mathalpha}{operators}{`G}%
\DeclareMathSymbol{H}{\mathalpha}{operators}{`H}%
\DeclareMathSymbol{I}{\mathalpha}{operators}{`I}%
\DeclareMathSymbol{J}{\mathalpha}{operators}{`J}%
\DeclareMathSymbol{K}{\mathalpha}{operators}{`K}%
\DeclareMathSymbol{L}{\mathalpha}{operators}{`L}%
\DeclareMathSymbol{M}{\mathalpha}{operators}{`M}%
\DeclareMathSymbol{N}{\mathalpha}{operators}{`N}%
\DeclareMathSymbol{O}{\mathalpha}{operators}{`O}%
\DeclareMathSymbol{P}{\mathalpha}{operators}{`P}%
\DeclareMathSymbol{Q}{\mathalpha}{operators}{`Q}%
\DeclareMathSymbol{R}{\mathalpha}{operators}{`R}%
\DeclareMathSymbol{S}{\mathalpha}{operators}{`S}%
\DeclareMathSymbol{T}{\mathalpha}{operators}{`T}%
\DeclareMathSymbol{U}{\mathalpha}{operators}{`U}%
\DeclareMathSymbol{V}{\mathalpha}{operators}{`V}%
\DeclareMathSymbol{W}{\mathalpha}{operators}{`W}%
\DeclareMathSymbol{X}{\mathalpha}{operators}{`X}%
\DeclareMathSymbol{Y}{\mathalpha}{operators}{`Y}%
\DeclareMathSymbol{Z}{\mathalpha}{operators}{`Z}%

\renewcommand{\emptyset}{\varnothing}

\newcommand{\respup}[1]{\emph{(}\resp {#1}\emph{)}}

\renewcommand{\leq}{\leqslant}
\renewcommand{\geq}{\geqslant}
 
\newcommand{\uple}[1]{\text{\boldmath${#1}$}}

\def\setminus{\mathchoice
    {\mathbin{\vrule height .62ex width 1.61ex depth -.38ex}}
    {\mathbin{\vrule height .62ex width 1.61ex depth -.38ex}}
    {\mathbin{\vrule height .50ex width 0.85ex depth -.28ex}}
    {\mathbin{\vrule height .20ex width 0.570ex depth -.24ex}}
}

\let\oldcal\mathcal

\renewcommand{\mathcal}{\mathscr}

\newcommand{\Cc}{\mathbf{C}}
\newcommand{\Nn}{\mathbf{N}}
\newcommand{\Aa}{\mathbf{A}}

\newcommand{\Dd}{\mathbf{D}}
\newcommand{\Zz}{\mathbf{Z}}
\newcommand{\Pp}{\mathbf{P}}
\newcommand{\Rr}{\mathbf{R}}
\newcommand{\Gg}{\mathbf{G}}
\newcommand{\Gm}{\mathbf{G}_{m}}
\newcommand{\Hh}{\mathbf{H}}
\newcommand{\Qq}{\mathbf{Q}}
\newcommand{\Fp}{{\mathbf{F}_p}}
\newcommand{\Fpt}{{\mathbf{F}^\times_p}}

\newcommand{\Ff}{\mathbf{F}}

\newcommand{\bQl}{\overline{\Qq}_{\ell}}
\newcommand{\Tt}{\mathbf{T}}

\newcommand{\mmu}{\boldsymbol{\mu}}

\newcommand{\mcV}{\mathscr{V}}

\newcommand{\mcO}{\mathscr{O}}

\newcommand{\mcH}{\mathscr{H}}

\newcommand{\Lb}{\mathscr{L}}
\newcommand{\proba}{\mathbf{P}}
\newcommand{\expect}{\mathbf{E}}

\def\sfrac#1#2{{#1}/{#2}}

\def\loccit{loc.\kern3pt cit.{}\xspace}
\def\cf{see\kern.3em}
\def\Cf{See\kern.3em}
\def\eg{e.g.\kern.3em}
\def\ie{i.e.,\ }
\def\resp{\text{resp.}\kern.3em}

\newcommand{\mods}[1]{\,(\mathrm{mod}\,{#1})}

\newcommand{\what}{\widehat}

\DeclareMathOperator{\frob}{Fr}

\DeclareMathOperator{\hypk}{Kl}

\newcommand{\ra}{\rightarrow}
\newcommand{\lra}{\longrightarrow}

\newcommand{\injecte}{\hookrightarrow}

\newcommand{\fleche}[1]{\stackrel{#1}{\lra}}

\DeclareMathOperator{\Spec}{Spec}

\DeclareMathOperator{\rank}{rank}

\DeclareMathOperator{\Pic}{Pic}
\DeclareMathOperator{\Jac}{Jac}

\DeclareMathOperator{\Imag}{Im}

\DeclareMathOperator{\syms}{Sym^{2}}

\DeclareMathOperator{\Frob}{\mathrm{Fr}}
\DeclareMathOperator{\Fr}{\mathrm{Fr}}
\DeclareMathOperator{\Kl}{\mathrm{Kl}}

\DeclareMathOperator{\tr}{\mathrm{Tr}}

\DeclareMathOperator{\Gal}{Gal}

\DeclareMathOperator{\supp}{supp}

\DeclareMathOperator{\Tr}{Tr}
\DeclareMathOperator{\Hom}{Hom}
\DeclareMathOperator{\End}{End}

\DeclareMathOperator{\swan}{swan}

\DeclareMathOperator{\Drop}{drop}

\DeclareMathOperator{\codim}{codim}
\DeclareMathOperator{\ft}{FT}

\DeclareMathOperator{\Ad}{Ad}
\DeclareMathOperator{\dual}{D}

\DeclareMathOperator{\ccodim}{ccodim}
\DeclareMathOperator{\FS}{FS}

\newcommand{\eps}{\varepsilon}
\renewcommand{\rho}{\varrho}
\renewcommand{\tilde}{\widetilde}

\DeclareMathOperator{\hH}{H}

\DeclareMathOperator{\SL}{\mathbf{SL}}
\DeclareMathOperator{\GL}{\mathbf{GL}}

\DeclareMathOperator{\Sp}{\mathbf{Sp}}

\DeclareMathOperator{\SO}{\mathbf{SO}}
\DeclareMathOperator{\Ort}{\mathbf{O}}
\DeclareMathOperator{\SU}{\mathbf{SU}}
\DeclareMathOperator{\Un}{\mathbf{U}}
\DeclareMathOperator{\USp}{\mathbf{USp}}
\DeclareMathOperator{\bfG}{\mathbf{G}}

\newcommand{\sheaf}[1]{\mathscr{{#1}}}

\DeclareMathSymbol{\gena}{\mathord}{letters}{"3C}
\DeclareMathSymbol{\genb}{\mathord}{letters}{"3E}

\def\dblsum{\mathop{\sum \sum}\limits}

\theoremstyle{plain}
\newtheorem{theorem}{Theorem}[chapter]
\newtheorem*{theorem*}{Theorem}
\newtheorem{lemma}[theorem]{Lemma}
\newtheorem{corollary}[theorem]{Corollary}

\newtheorem{proposition}[theorem]{Proposition}

\newtheorem{theointro}{Theorem}

\theoremstyle{remark}

\theoremstyle{definition}

\newtheorem{definition}[theorem]{Definition}

\newtheorem*{question}{Question}
\newtheorem{example}[theorem]{Example}
\newtheorem{remark}[theorem]{Remark}

\newtheorem{remarkintro}{Remark}
\newtheorem{exampleintro}{Example}

\newcommand{\abs}[1]{\left\lvert#1\right\rvert}

\newcommand{\mcL}{\mathscr{L}}
\newcommand{\mcLd}{\mathscr{L}^{\Delta}}
\newcommand{\mcS}{\mathscr{S}}
\newcommand{\mcT}{\mathscr{T}}

\newcommand{\mcF}{\mathscr{F}}
\newcommand{\mcK}{\mathscr{K}}

\newcommand{\mcG}{\mathscr{G}}

\newcommand{\mcE}{\mathscr{E}}
\newcommand{\mcU}{\mathscr{U}}
\newcommand{\mcX}{\mathscr{X}}
\newcommand{\mcY}{\mathscr{Y}}

\DeclareMathOperator{\mfm}{\mathfrak{m}}

\newcommand{\rmH}{\mathrm{H}}
\newcommand{\mcP}{\mathscr{P}}

\renewcommand{\geq}{\geqslant}
\renewcommand{\leq}{\leqslant}

\newcommand{\ov}[1]{\overline{#1}}

\newcommand\sumsum{\mathop{\sum\sum}\limits}

\newcommand{\lhat}{\widehat{L}}
\DeclareMathOperator{\wtr}{wt}

\newcommand{\Ga}{\mathbf{G}_{a}}

\newcommand{\Qlb}{{\ov{\Qq}_{\ell}}}
\newcommand{\Qellb}[1]{{\ov{\Qq}_{{#1}}}}
\newcommand{\Qlbt}{\ov{\Qq}_{\ell}^{\times}}
\newcommand{\et}{{\mathrm{\acute{e}t}}}
\newcommand{\Der}{{\mathrm{D}_c^{\mathrm{b}}}}
\newcommand{\Dcoh}{{\mathrm{D}_{\mathrm{coh}}^{\mathrm{b}}}}

\newcommand{\Perv}{{\mathbf{Perv}}}

\newcommand{\Pervint}{{{\mathrm{Perv}_\mathrm{int}}}}
\newcommand{\Ppb}{{\overline{\Pp}}}
\newcommand{\Ppintun}{\Pp^1_{\intt}}
\newcommand{\Ddb}{{\overline{\Dd}}}
\newcommand{\Ppint}{{{\Pp_\mathrm{int}}}}
\newcommand{\Ppintarith}{{{\Pp_\mathrm{int}^\arith}}}

\newcommand{\Ppbarith}{{\overline{\Pp}^\arith}}

\newcommand{\unram}[1]{\mathcal{X}({#1})}
\newcommand{\wunram}[1]{\mathcal{X}_w({#1})}
\newcommand{\nunram}[1]{\mathcal{N}({#1})}
\newcommand{\funram}[1]{\mathcal{X}_F({#1})}

\newcommand{\Thetaf}{\uple{\Theta}}
\newcommand{\Frf}{\mathbf{Fr}}
\DeclareMathOperator{\Std}{Std}

\newcommand{\NegP}{{\mathrm{Neg}_{\Pp}}}
\newcommand{\NegD}{{\mathrm{Neg}_{\Dd}}}
\newcommand{\NegParith}{{\mathrm{Neg}_{\Pp}^\arith}}

\newcommand{\DD}{\dual}
\newcommand{\inv}{{\mathrm{inv}}}
\newcommand{\pH}{\tensor*[^{\mathfrak{p}}]{\mathscr{H}}{}}
\newcommand{\ptau}{\tensor*[^{\mathfrak{p}}]{\tau{}}{}}
\newcommand{\pt}{\tensor*[^{\mathfrak{p}}]{t{}}{}}
\newcommand{\tac}{tac}
\newcommand{\tacs}{tacs}

\newcommand{\un}{{\mathbf{1}}}
\newcommand{\ev}{{\mathrm{ev}}}
\newcommand{\coev}{{\mathrm{coev}}}
\newcommand{\id}{{\mathrm{id}}}
\newcommand{\ab}{\mathrm{ab}}
\newcommand{\can}{\mathrm{can}}
\newcommand{\FM}{\mathrm{FM}}

\newcommand{\arith}{{\mathrm{ari}}}
\newcommand{\geo}{{\mathrm{geo}}}
\newcommand{\intt}{\mathrm{int}}
\newcommand{\bks}{{\setminus}}
\newcommand{\Supp}{{\mathrm{Supp}}}

\newcommand{\Hyp}{\mathrm{Hyp}}

\newcommand{\charg}[1]{{\widehat{#1}}}
\newcommand{\chag}[2]{\langle {#1},{#2}\rangle}

\newcommand{\Ks}{\Kl}

\newcommand{\garith}[1]{{ \tensor*[]{\mathbf{G}}{_{#1}^{\arith}}}}
\newcommand{\ggeo}[1]{{ \tensor*[]{\mathbf{G}}{_{#1}^{\geo}}}}

\def\llb{\mathopen{[\![}}

\def\rrb{\mathopen{]\!]}}

\makeindex
\makenomenclature

\begin{document}

\title[Arithmetic Fourier transforms]{Arithmetic Fourier transforms over
  finite fields\\
  [0.5cm]{\Large Generic vanishing, convolution, and equidistribution}}

\author{Arthur Forey}
\address[A. Forey]{Univ. Lille, CNRS, UMR 8524 - Laboratoire Paul Painlevé,  \newline F-59000 Lille, France} 
  \email{arthur.forey@univ-lille.fr}

\author{Javier Fres\'an}
\address[J. Fres\'an]{Sorbonne Universit\'e and Universit\'e Paris Cit\'e, CNRS, IMJ-PRG, \newline F-75005 Paris, France}
\email{javier.fresan@imj-prg.fr}

\author{Emmanuel Kowalski}
\address[E. Kowalski]{D-MATH, ETH Z\"urich, R\"amistrasse 101,  \newline CH-8092 Z\"urich, Switzerland} 
  \email{kowalski@math.ethz.ch}

\begin{abstract}
  We study the arithmetic Fourier transforms of trace functions on
  general connected commutative algebraic groups. To do so, we first
  prove a generic vanishing theorem for twists of perverse sheaves by
  characters, and using this tool, we construct a tannakian category
  with convolution as tensor operation. Using Deligne's Riemann
  hypothesis, we show how this leads to a general equidistribution
  theorem for the discrete Fourier transforms of trace functions of
  perverse sheaves, generalizing the work of Katz in the case of the
  multiplicative group. We then give some concrete examples of
  applications of these results and raise a number of questions.
\end{abstract}


\maketitle

\cleardoublepage
\thispagestyle{empty}
\vspace*{13.5pc}
\begin{center}
  \textit{Dedicated to Nick Katz, with the greatest admiration}
\end{center}
\cleardoublepage

\setcounter{tocdepth}{1}
\tableofcontents
\mainmatter

\chapter*{Preface}

The Fourier transform, and the whole collection of its variants whose
study is summarized under the heading of ``harmonic analysis'', is one
of the most important tools of mathematics. In its many forms, its
applications cover the whole range not only of mathematics, but also
physics, computer science, chemistry and indeed of all sciences where
quantitative tools are applied.

In 1976, P. Deligne observed in a letter to D. Kazhdan (which is
reproduced in Appendix~\ref{ch-app-letter}) that the
formalism of algebraic geometry, and especially of $\ell$-adic
cohomology and the derived category of $\ell$-adic sheaves, provided a
new ``geometric'' form of the Fourier transform. Instead of the
familiar integral formula
$$
\widehat{f}(y)=\int_{\Rr}f(x)e^{-2i\pi xy}dx
$$
associating to a function $f$ (say $f\colon \Rr\to\Cc$ in the Schwartz
space) its Fourier transform~$\widehat{f}$, Deligne's version takes as
input an $\ell$-adic constructible sheaf~$M$, or a complex of those, on the
one-dimensional affine space over a finite field~$k$ of
characteristic~$p$, and outputs a Fourier transform $\widehat{M}$
which is of the same kind.

We note that although the most general and convenient category of
input objects~$M$, which we will also call ``coefficients'', is given
by the formalism of derived categories of $\ell$-adic complexes with~$\ell$ prime different from~$p$, there is a simpler definition in the
case considered here, where~$M$ can (in almost all cases) be thought
of as being a continuous finite-dimensional representation
$$
\rho\colon\Gal(k(T)^{\mathrm{sep}}/k(T))\to \GL_r(\bQl)
$$
of the absolute Galois group of the field $k(T)$ of rational functions
with coefficients in~$k$.

The crucial point for the interpretation of this construction as a
Fourier transform is that to each object $M$ is associated classically
a sequence of ``trace functions'', which are functions
$$
t_M(\cdot;k_n)\colon k_n\to\Cc\simeq \bQl
$$
defined on the finite extensions $k_n$ of~$k$ of degree~$n$, for all
integers $n\geq 1$, and Deligne's Fourier transform then satisfies
$$
t_{\widehat{M}}(y;k_n)=\sum_{x\in k_n}t_M(x;k_n)e^{2i\pi
  \Tr_{k_n/\Ff_p}(xy)/p}. 
$$
Thus, the trace functions of $\widehat{M}$ coincide with the discrete
Fourier transforms of those of~$M$.

Deligne's Fourier transform shares many features with the classical
euclidean Fourier transform, once properly interpreted in terms of the
coefficients~$M$. For instance:
\begin{itemize}
\item It satisfies a form of the Fourier inversion formula
  $$
  f(x)=\int_{\Rr}\widehat{f}(y)e^{2i\pi xy}dy,
  $$
  in the sense that applying the (similarly defined) analogue of the
  inverse Fourier transform to~$\widehat{M}$ recovers~$M$. 
\item It satisfies analogues of the Plancherel formula, which are
  however less obvious: one interpretation is that if the
  representation $\rho$ above is irreducible, then so is the
  representation associated to~$\widehat{M}$. 
\item It satisfies a geometric analogue of the fundamental algebraic
  relation $\widehat{f*g}=\widehat{f}\ \widehat{g}$, which relates the
  Fourier transform and the convolution product
  $$
  (f*g)(x)=\int_{\Rr}f(y)g(x-y)dy
  $$
  of functions (this property is often taken as the key feature of
  Fourier analysis and especially Pontryagin duality~\cite{bourbaki-ts-2}). Indeed, to two coefficients $M_1$
  and~$M_2$, another geometric construction associates a third
  one~$M_3$, such that the trace function of~$M_3$ is given by
  $$
  t_{M_3}(x;k_n)=\sum_{y\in k_n}t_{M_1}(y;k_n)t_{M_2}(x-y;k_n),
  $$
  the discrete convolution of those of $M_1$ and~$M_2$. 
\item And there is a subtle analogue, due to Laumon, of the stationary
  phase principle for estimating oscillatory integrals.
\end{itemize}

There are however also special features related to the geometric
nature of trace functions:
\begin{itemize}
\item Deligne's Fourier transform preserves a particularly important
  subcategory of coefficients, that of \emph{perverse sheaves} -- this
  extremely important fact has no obvious classical analogue.
\item If a coefficient object $M$ is a perverse sheaf, and hence also
  its transform $\widehat{M}$, then one can associate to it a natural
  intrinsic \emph{symmetry group}, also called its \emph{monodromy
    group}, which is an algebraic group over~$\bQl$ (or over~$\Cc$). The
  definition of this group can be seen as a wide-ranging generalization
  of that of the Galois group of a polynomial. (In the one-dimensional
  case, where $\widehat{M}$ can be identified, in most cases, with a
  Galois
  representation~$\rho\colon\Gal(k(T)^{\mathrm{sep}}/k(T))\to
  \GL_r(\bQl)$ as above, this symmetry group is nothing but the Zariski\nobreakdash-closure of
  the image of~$\rho$.)
\end{itemize}

Deligne's Fourier transform has found a number of very important
applications in arithmetic and algebraic geometry, as well as number
theory. In the former direction, Laumon~\cite{laumon-signes} used it
to obtain a product formula for the epsilon factors of Artin-type
$L$-functions on curves over finite fields. In number theory, Katz
used it extensively to study in depth the
distribution properties of families of exponential sums, which are
obtained as discrete Fourier transforms of simple trace functions (see,
for instance,~\cite{gkm} and \cite{esde}); the symmetry group of the
Fourier transform $\widehat{M}$ plays an essential role here. A
prominent example of such sums are the Kloosterman sums
$$
\Kl_2(a;p)=\frac{1}{\sqrt{p}}\sum_{x\in \Ff_p^{\times}}
e\Bigl(\frac{\bar{x}+ax}{p}\Bigr)
$$
which are the values of the trace function of the Fourier transforms of
a one-dimensional Galois representation, and are omnipresent in modern
analytic number theory (here and below, we use the notation $e(z)=\exp(2i\pi z)$,
and $\bar{x}$ is the inverse of $x$ modulo~$p$). Results about these and
similar sums, which often rely on properties of the $\ell$-adic Fourier
transform, have by now become essential in many fundamental results of
analytic number theory -- some concrete examples appear
in Zhang's famous work on bounded gaps between
primes~\cite[Lemma\,12]{zhang}, and systematic use of the Fourier
transform begins in various papers of Fouvry, Kowalski and Michel (see,
for instance,~\cite{fkm1}).

Deligne's transform is the geometric analogue of the classical euclidean
Fourier transform on~$\Rr$ and can be generalized to $n$ variables. But,
in recent years, a number of applications have led to questions
concerning similar properties of other discrete Fourier transforms, for
instance those related to the multiplicative group $k_n^{\times}$, which
are functions on the group of multiplicative characters~$\chi\colon k_n^{\times}\to \bQl$. The study of the distribution, or
average properties, of these sums is \emph{outside} of the realm of
applications of Deligne's Fourier transform, and these
functions \emph{cannot} be expressed as trace functions of complexes of
$\ell$-adic sheaves on an algebraic variety over $k$.

The fundamental motivation for this book is the search for a
definition of the analogue of Deligne's Fourier transform on an
arbitrary commutative algebraic group over a finite field, and for the
general theory and applications of this form of harmonic analysis. In
particular, we believe that these \emph{arithmetic Fourier transforms}
can be interpreted in the context of much more general arithmetic or
geometric avatars of harmonic and functional analysis.

The basic examples of commutative algebraic groups are the
multiplicative groups (or tori), and abelian varieties, and these can
be combined (together with additive groups) in various ways. The
choice of an input object $M$ on such a group~$G$ leads to its
arithmetic Fourier transforms, which are the functions of the form
$$
\widehat{t}_M(\chi;k_n)=\sum_{x\in G(k_n)}\chi(x)t_M(x;k_n),
$$
defined for any~$n\geq 1$, where the parameter $\chi$ ranges over
characters of the finite group $G(k_n)$.

The simplest example beyond the additive case is that of
$G(k_n)=k_n^{\times}$, in which case the characters are multiplicative
characters of~$k_n$, and $\widehat{t}_M$ is called an \emph{arithmetic Mellin transform}. N.~Katz, in a striking breakthrough, succeeded a
few years ago in finding an interpretation of these functions in his book~\cite{mellin}. He
exploited the formalism of tannakian categories, and the fact that the
\emph{convolution product} extends to any commutative algebraic group:
given coefficients~$M_1$ and~$M_2$ on~$G$, there exists a
geometrically-defined object~$M_3$ such that, for all $n\geq 1$ and $x\in G(k_n)$, their respective trace
functions~satisfy
$$
t_{M_3}(x;k_n)=\sum_{y\in G(k_n)}t_{M_1}(y;k_n)t_{M_2}(xy^{-1};k_n).
$$

Although Katz's interpretation of the arithmetic Mellin transforms is
\emph{not} fully geometric (there is no analogue of the object
$\widehat{M}$ which ``is'' Deligne's Fourier transform for the
additive group), Katz shows that it is enough to define a \emph{symmetry
  group} for the arithmetic Mellin transform. In combination with
another fundamental tool, Deligne's general form of the Riemann
hypothesis over finite fields~\cite{D-WeilII}, this allowed Katz to
prove an equidistribution theorem which controls the distributions of
arithmetic Mellin transforms. A number of significant applications
followed, including the work of Keating and
Rudnick~\cite{kr1} and Hall, Keating and Roditty-Gershon~\cite{variance}.

One of the main theoretical achievements of this book is the extension
of these ideas of Katz to \emph{any} connected commutative algebraic
group.  This  is far from routine, since certain necessary
tools, such as generic cohomological vanishing, or estimates for Betti
numbers, which are very elementary in the case considered by Katz, were
not known previously for groups of dimension at least~$2$. Indeed, we
rely in an essential way on the very recent \emph{quantitative sheaf
  theory} due to Sawin~\cite{sawin_conductors} (which was partly
motivated by this work and drafted in final form jointly with the
authors).

For any suitable coefficient object on the group~$G$, our construction
provides the fundamental invariant of its arithmetic Fourier transform,
its \emph{intrinsic symmetry group}. Combined again with other tools
such as Deligne's Riemann hypothesis over finite fields, this is already
sufficient to prove a very general form of equidistribution theorem,
which encompasses the previously known cases of Deligne and Katz (and in
fact sharpens these in certain aspects). In turn, we can use this
equidistribution theorem for a number of first applications, including
strengthening and simplifying the results of~\cite{variance}. But
there remain also many open questions and problems, both on the
theoretical side and on that of applications -- we will discuss
briefly some of these at the end of this book.

After this preface, the book will continue with a more technical
introduction, which contains precise statements of some of the key
results and a quick description of some of the crucial points which are
involved in the proofs.  We then split the remainder of the book in two
parts, one containing the main theoretical results, and the other
devoted to a variety of applications. These are complemented by
appendices recalling important material, and Deligne's letter to Kazhdan. A more precise outline of each chapter will be found at the end of the
introduction.

\emph{Readers with a background in analytic number theory who are not
  familiar with the theory of trace functions and the underlying
  geometric objects are invited to first read
  Appendix~\ref{ch-app-analytic}, where we attempt to present them in a
  concrete and intuitive way.}

\SkipTocEntry\section*{Acknowledgements}

We are particularly thankful to Will Sawin, not only for many
insightful comments on parts of this work, but especially for sharing
his ideas on the complexity of $\ell$-adic sheaves on algebraic
varieties, and for allowing us to assist in the final write-up of his
work, resulting in the paper~\cite{sawin_conductors}.
\par
We thank the referees who read various parts of the book, and gently pointed
out some incorrect or confusing statements, as well as some incomplete
arguments, in the first draft. 
\par
We thank K. Soundararajan for pointing out to us the definition of Sidon
sets.  We thank Thomas Krämer for useful comments concerning some of his
papers, especially~\cite{kramer-e6}. We thank Arnaud Beauville for his
help with some questions concerning cubic hypersurfaces, especially for
the proof of Lemma~\ref{lm-beauville}. We also thank Sophie Morel for
providing us with a proof of Lemma~\ref{lm-morel}. We thank Florent
Jouve for his careful reading of Appendix~\ref{ch-app-analytic}; we also
thank Jon Keating for his encouragements, and Zeev Rudnick for his
constructive criticism, concerning this same appendix.
\par
We thank N. Katz for forwarding to us the notes of his
talk~\cite{KatzETH}, and for his comments on various parts of the text. 
\par
We also thank the organizers of the various seminars where we have
been able to present parts of this work over the last few years.
\par
Finally, we thank Pierre Deligne for kindly allowing us to reproduce his
letter to Kazhdan in an appendix.
\par
During the preparation of this work, A.\,F. and E.\,K. were partially
supported by the DFG-SNF lead agency program grant 200020L\_175755,
A.\,F. was partially supported by the SNF Ambizione grant
PZ00P2\_193354, and J.\,F. was partially supported by the grant
ANR-18-CE40-0017 of the Agence Nationale de la Recherche.


\chapter*{Introduction}

\section{Statement of results}

Since Deligne's proof of his equidistribution theorem\index{Deligne's equidistribution theorem} for traces of
Frobenius of $\ell$-adic local systems on varieties over finite
fields~\cite{D-WeilII}, it has been known that any family of exponential
sums parameterized by an algebraic variety satisfies some form of
equidistribution, and that the concrete expression of this
equidistribution statement depends on the determination of the geometric
monodromy group\index{monodromy group} of the $\ell$-adic sheaf that underlies the family of
exponential sums.

The best known result of this kind is probably the computation by
Katz~\cite{gkm} of these monodromy groups in the case of Kloosterman
sums\index{Kloosterman sums}\index{hyper-Kloosterman sums} in several variables over finite fields, which are defined for some
fixed non-trivial additive character~$\psi \colon \Ff_q \to \mathbf{C}^\times$ and
$a \in \Ff_q^\times$ as
\[
\hypk_m(a;q)=\frac{1}{q^{(m-1)/2}}
\sum_{\substack{(x_1,\ldots, x_m)\in (\Ff_q^{\times})^m\\
    x_1\cdots x_m=a}}\psi(x_1+\cdots+x_m).
\]
This computation led him in particular to the proof of the average
version of the Sato\nobreakdash--Tate law\index{Sato--Tate law} for classical Kloosterman sums,
namely the equidistribution of the sets
$\{ \hypk_2(a; q)\,|\,a \in \Ff_q^\times \}$ with respect to the Sato--Tate measure\index{Sato--Tate measure} on the interval~$[-2, 2]$ as~$q \to +\infty$ among prime powers. Further deep investigations by
Katz, especially in his monograph~\cite{esde}, provide a cornucopia of
examples of equidistribution~statements.

Among other things, this framework allows for the study of exponential
sums of the form
\begin{displaymath}
S(M, \psi)=\sum_{x \in \Ff_{q^n}} t_M(x;\Ff_{q^n})\psi(x),
\end{displaymath}
where $t_M$ is the trace function of a perverse sheaf $M$ on the
additive group~$\Gg_a$ over $\Ff_q$ and $\psi$ ranges over characters of
$\Ff_{q^n}$. These sums are the discrete Fourier transform
$\psi \mapsto S(M, \psi)$ of the function~$x \mapsto t_M(x, \Ff_{q^n})$
on the finite group $\Ff_{q^n}=\Gg_a(\Ff_{q^n})$, and the key point is that they
are themselves the trace functions of another perverse sheaf on the dual
group parameterizing additive characters, namely Deligne's Fourier transform of $M$. 

In a more recent conceptual breakthrough, Katz~\cite{mellin} succeeded
in proving equidistribution results for families of exponential sums
parameterized by multiplicative characters, despite the fact that the
set of multiplicative characters of a finite field $\Ff_q$ does not
naturally arise as the set of~$\Ff_q$\nobreakdash-points of an algebraic
variety. In analogy with the above, such sums are of the form
\begin{displaymath}
 S(M, \chi)=\sum_{x \in \Ff_{q^n}^{\times}} t_M(x;\Ff_{q^n})\chi(x),
\end{displaymath}
except that $M$ is now a perverse sheaf on the multiplicative
group~$\Gg_m$ over~$\Ff_q$ and~$\chi$ ranges over characters of
$\Ff_{q^n}^{\times}$. Katz's beautiful insight was to replace points of
algebraic varieties by fiber functors of tannakian categories as
parameter spaces, and produce the groups governing equidistribution by
means of the tannakian formalism (see \cite{fresan} for an accessible
survey). Further work of Katz generalized this to elliptic curves
\cite{katz_elliptic} and certain abelian varieties
(unpublished).

\textbf{\emph{The primary goal of this book is to extend these ideas to
    exponential sums (arithmetic Fourier transforms) parameterized by
    the characters of the points of any connected commutative algebraic
    group over a finite field.}}
\par
\medskip
\par
More precisely, let $k$ be a finite field and $\bar k$ an algebraic
closure of~$k$. For each $n \geq 1$, we denote by~$k_n$ the extension
of~$k$ of degree~$n$ inside~$\bar k$. Let~$\ell$ be a prime number
different from the characteristic of~$k$ and $\Qlb$ an algebraic closure
of the field of $\ell$-adic numbers. Let $G$ be a connected commutative
algebraic group over $k$. We denote by $\widehat{G}(k_n)$ the group of
$\Qlb$-valued characters of~$G(k_n)$ and, for each
$\chi\in \widehat{G}(k_n)$, by $\mcL_{\chi}$ the $\ell$-adic lisse
character sheaf of rank one associated to~$\chi$ by means of the Lang
torsor construction, as briefly recalled in
Section~\ref{sec:character-groups-Lang-torsor}. By perverse sheaves, we
always understand $\Qlb$-perverse sheaves.

In rough outline, we establish the following types of theoretical
results:

\begin{itemize}
\item We prove generic and stratified vanishing theorems for the
  cohomology of twists of perverse sheaves on~$G$ by the sheaves
  $\mcL_{\chi}$ associated to characters~$\chi \in \widehat{G}(k_n)$.
  
\item Using the stratified vanishing theorems, we construct a tannakian
  category of perverse sheaves on~$G$ with the convolution coming from the group law as tensor product.

\item We prove that the tannakian group of a semisimple object~$M$ of
  this category that is pure of weight zero controls the distribution
  properties of the sums
  \[
  S(M,\chi)=\sum_{x \in G(k_n)} t_M(x;k_n)\chi(x),
  \]
  where~$\chi$ ranges over the set~$\widehat{G}(k_n)$. Under some
  assumptions on $G$ (e.g., for tori and abelian varieties), we prove
  the stronger result that the unitary conjugacy classes of which these
  sums are traces become equidistributed in a maximal compact subgroup
  of the tannakian group as $n \to +\infty$, as is customary since
  Deligne's work.
\end{itemize}

Once this is done, we provide a number of applications, both of a
general nature and for concrete groups and perverse sheaves.

The following statements are special cases of our main results, which we
formulate in simplified form in order to make it possible to present
self-contained statements at this stage.

\begin{theointro}\label{thm:vanishing-thm-intro}
Let $M$ be a perverse sheaf on a connected commutative algebraic group~$G$ of
dimension $d$ over a finite field $k$.
  \begin{enumth}
  \item \textup{(Generic vanishing)}\index{generic vanishing theorem} The sets
    \begin{align*}
      \mcX(k_n)=\big\{\chi\in \charg{G}&(k_n)\mid \rmH_c^{i}(G_{\bar
        k},M\otimes\mcL_{\chi})= \rmH^{i}(G_{\bar
        k},M\otimes\mcL_{\chi})=0 \text{ for all } i\neq 0
      \\
      &\text{ and } \rmH_c^{0}(G_{\bar k},M\otimes\mcL_{\chi}) \text{ is
        isomorphic to } \rmH^{0}(G_{\bar k},M\otimes\mcL_{\chi}) \big\}
    \end{align*}
    are \emph{generic}, in the sense that the estimate
    \[
      \abs{\charg{G}(k_n) \setminus \mcX(k_n)}\ll \abs{k_n}^{d-1}
    \]
   holds for~$n\geq 1$, with an implied constant that only depends on~$M$.
    
 \item \textup{(Stratified vanishing)}\index{stratified vanishing theorem} For~$-d\leq i\leq d$
   and~$n\geq 1$, the estimate
    \[
      \abs{\big\{\chi\in \charg{G}(k_n)\mid \rmH_c^{i}(G_{\bar k},
        M\otimes\mcL_{\chi})\neq 0\text{ or } \rmH^{i}(G_{\bar k},
        M\otimes\mcL_{\chi})\neq 0\big\}}\ll \abs{k_n}^{d-i}
    \]
    holds, with an implied constant that only depends on~$M$.
  \end{enumth}
\end{theointro}

The most general vanishing statements that we prove appear as
Theorems~\ref{thm-gen-vanish} and~\ref{thm-high-vanish}. Applications to
``stratification'' estimates for exponential sums are then given in
Chapter~\ref{sec-stratification}.

\begin{remarkintro}
  (1) With variations in the definition of generic set of characters,
  such statements were proved by Katz--Laumon
  \cite{KL-fourier-exp-som} for powers of the additive group, Saibi
  \cite{saibi_FD_unipotent} for unipotent groups,
  Gabber\nobreakdash--Loeser \cite{GL_faisc-perv} for tori, Krämer-Weissauer \cite{KW_vanishing_AV}, Weissauer
  \cite{weissauer_vanishing_2016} for abelian varieties and Krämer
  \cite{kramer_perverse_2014} for semiabelian varieties (see Remark
  \ref{rem:versionsvanishing} for more precise references).
  \par
  (2) In characteristic zero, and especially over the complex
  numbers, theorems of this type have also been proved for abelian and
  semiabelian varieties by Schnell~\cite{schnell},
  Bhatt--Scholze--Schnell~\cite{bss} and Liu--Maxim--Wang~\cite{lmw2}
  (see also~\cite{lmw} for a survey of some applications of such
  results). Over arbitrary algebraically closed fields, there has also
  been recent works of Esnault and Kerz~\cite{esnault-kerz}.
  \par
  (3) In contrast with the case of abelian varieties, lack of properness
  and wild ramification phenomena are the reason one must formulate
  conditions on cohomology groups both with and without compact support.
\end{remarkintro}

Using the vanishing theorems, and ideas going back to Gabber--Loeser and
Katz, we can construct tannakian categories with the convolution on~$G$
as tensor operation. Using these, and Deligne's Riemann hypothesis over
finite fields, we obtain the following equidistribution theorem for the
Fourier transforms of trace functions on~$G$, i.e., for families of
exponential sums parameterized by characters of~$G$.

\begin{theointro}[Equidistribution on average for arithmetic Fourier transforms]
  \label{thm:equidis-thm-intro}
  Let $G$ be a connected commutative algebraic group over~$k$.
  Let~$M$ be a geometrically simple $\ell$-adic perverse sheaf
  on~$G$ that is pure of weight zero, with complex-valued trace
  functions $t_M(\cdot;k_n) \colon G(k_n) \to \Cc$ for~$n\geq 1$. There
  exists an integer~$r\geq 0$ and a compact
  subgroup~$K\subset \Un_r(\Cc)$ of the unitary group such that
  the sums
  \[
    S(M,\chi)=\sum_{x\in G(k_n)}t_M(x;k_n)\chi(x)
  \]
  for complex-valued characters $\chi$ of $G(k_n)$ become
  \emph{equidistributed on average} in~$\Cc$ with respect to the image
  by the trace of the Haar probability measure $\mu$ on~$K$. That is,
  for any bounded continuous function~$f \colon \Cc \to \Cc$, the
  following equality holds:
  \begin{equation}\label{eqn:equidistribution-intro}
  \lim_{N\to+\infty} \frac{1}{N}\sum_{1\leq n\leq N} \frac{1}{|G(k_n)|}
  \sum_{\chi} f(S(M,\chi))= \int_K f(\Tr(x))d\mu(x),
  \end{equation}
  where $\chi$ runs over all characters of $G(k_n)$.
\end{theointro}

The general version of this theorem appears as Theorem~\ref{th-3}. Under
an additional assumption (which holds for tori, abelian varieties
and~$\Gg_a$, at least), we can also deduce it from Theorem~\ref{th-2},
which is a more precise equidistribution result for unitary conjugacy
classes of Frobenius in the compact group~$K$. (The difference between
these two statements is similar to that between the
Frobenius equidistribution theorem for cycle types of Frobenius classes
in the Galois group of a polynomial, viewed as a permutation group, and
the more precise Chebotarev density theorem.)

\begin{remarkintro}
  (1) In the classical setting of~$\Gg_a$ and the Fourier transform, the
  group~$K$ is a maximal compact subgroup of the \emph{arithmetic}
  monodromy group of the (lisse sheaf underlying the)~$\ell$\nobreakdash-adic
  Fourier transform of~$M$ (see Proposition~\ref{pr-gm-unip-known}).

  Note that this is in contrast with more usual versions of Deligne's
  equidistribution theorem, without the extra Cesàro average over~$n$,
  where the focus is on the \emph{geometric} monodromy group (see,
  e.g., the versions of Katz~\cite[Ch.\,3]{gkm} and
  Katz--Sarnak~\cite[Ch.\,9]{katz-sarnak}).  This slight change of
  emphasis extends to the general situation, and means that we can
  avoid additional (necessary) assumptions such as the equality of the
  geometric and arithmetic monodromy groups, which occur frequently
  otherwise (see, e.g.,~\cite[\S\,3.3]{gkm}), and are not always easy
  to check.
  \par
  The Cesàro average can of course be interpreted as a form of
  ``smoothing'' (a ``summation method'' in the classical
  terminology). Although it is quite natural, it can be replaced by
  many others (see Remark~\ref{rm-smoothing}).
  \par
  (3) We will also discuss a ``horizontal'' version, where we consider
  suitable families $(M_p)_p$ of perverse sheaves over $\Ff_p$ for primes
  $p\to+\infty$. However, such results depend on a more quantitative
  version of the stratified vanishing theorem, which we have not
  established in full generality yet.
  \par
  (4) As already mentioned, this equidistribution theorem is essentially
  Deligne's equidistribution theorem on average for the $\ell$-adic
  Fourier transform of~$M$ when $G=\Ga$. When~$G$ is the multiplicative
  group (or its non-split form), one obtains (an average version of)
  Katz's equidistribution
  theorem~\cite{mellin}. In~\cite{katz_elliptic}, Katz proves a similar
  theorem for elliptic curves.
  \par
  (5) The assumption that $G$ is connected arises from the fact that the
  Lang torsor construction is only applicable in this case.  For the
  purpose of equidistribution results, however, one can easily handle a
  non-connected algebraic group by considering one by one the
  restrictions to the neutral component of $G$ of the objects
  $([x\mapsto c^{-1}x]^*M)$, where $c$ runs over representatives of the
  connected components of~$G$. (Note that different connected components
  might give rise to exponential sums with different distributions.)
\end{remarkintro}

\begin{exampleintro}
  A simple concrete class of examples where we obtain equidistribution
  statements is the following (in the case when~$G$ is not an abelian
  variety): assume that~$k=\Ff_p$, and let~$d$ be the dimension
  of~$G$; then for any non-constant function $f\colon G\to \Aa^1$,
  there exists a perverse sheaf~$M_f$ on $G$ with trace functions
  $$
  t_{M_f}(x;\Ff_{p^n})= \frac{(-1)^d}{p^{nd/2}}
  e\Bigl(\frac{\Tr_{\Ff_{p^n}/\Ff_p}(f(x))}{p}\Bigr)
  $$
  for all~$n\geq 1$ and~$x\in G(\Ff_{p^n})$ (where
  $e(z)=\exp(2i\pi z)$), so that Theorem~\ref{thm:equidis-thm-intro}
  shows that the exponential sums
  $$
  \frac{1}{p^{nd/2}}\sum_{x\in G(\Ff_{p^n})}\chi(x)
  e\Bigl(\frac{\Tr_{\Ff_{p^n}/\Ff_p}(f(x))}{p}\Bigr)
  $$
  (which are intuitively sums over~$d$ variables) become
  equidistributed on average, with limiting measure of a very specific
  kind. Specializing even more to $G=\Gg_m^d$, the function~$f$ is a Laurent
  polynomial in $x_1$, \ldots, $x_d$ and their inverses,
  and these exponential sums become the sums
  $$
  \frac{1}{p^{nd/2}}\sum_{x_1,\ldots, x_d\in \Ff_{p^n}^{\times}}
  \chi_1(x_1)\cdots \chi_d(x_d)
  e\Bigl(\frac{\Tr_{\Ff_{p^n}/\Ff_p}(f(x))}{p}\Bigr)
  $$
  parameterized by a tuple $(\chi_1,\ldots,\chi_d)$ of characters
  of~$\Ff_{p^n}^{\times}$.
\end{exampleintro}

As a further concrete application, we will see how to deduce
statements like the following, which considerably strengthens earlier
work of Hall, Keating and Roddity-Gershon~\cite{variance}.

\begin{theointro}[Variance of the von Mangoldt function of the Legendre
  elliptic curve]\label{th-variance}
  Let~$k$ be a finite field of characteristic~$\geq 3$. Let~$\mathcal{E}$ be the Legendre elliptic curve with affine model
  \[
    y^2=x(x-1)(x-t)
  \]
  over the field~$k(t)$. Let~$\Lambda_{\mathcal{E}/k(t)}$ be the von
  Mangoldt function of~$\mathcal{E}$, defined by the generating series
  \[
    -T\frac{L'(\mathcal{E}/k(t), T)}{L(\mathcal{E}/k(t), T)} =\sum_g
    \Lambda_{\mathcal{E}/k(t)}(g) T^{\deg(g)}
  \]
  over monic polynomials $g \in k[t]$.  
  
  Let~$f\in k[t]$ be a square-free polynomial of degree $\geq 4$ and set
  $B=k[t]/fk[t]$. Let~$m\geq 1$ be an integer. For any $a\in B^\times$,
  consider the sum
  \[
    \psi_{\mathcal{E}}(m;f,a)=\sum_{\substack{\deg(g)=m\\g\equiv a\mods{f}}}
    \Lambda_{\mathcal{E}/k(t)}(g)
  \]
  over monic polynomials $g \in k[t]$ of degree~$m$. 
  Then the following equality holds:
  \begin{align*}
    \lim_{|k|\to +\infty} \frac{1}{|k|^2} \frac{1}{|B^\times|} \sum_{a\in
      B^\times}\Bigl| \psi_{\mathcal{E}}(m;f,a)-&\frac{1}{|B^\times|}\sum_{b\in B^\times}\psi_{\mathcal{E}}(m;f, b) \Bigr|^2 \\
      &= \min\big(m,2\deg(f)-2+\deg \gcd(t(t-1), f)\big).
  \end{align*}
\end{theointro}

The meaning of the above limit is that we replace $k$ by its extensions
  $k_n$ of degree~$n\geq 1$, compute the
  variance for the base change of~$\mathcal{E}$ to~$k_n$ (note that $B$
  depends on~$k$, so it is also replaced by~$k_n[t]/fk_n[t]$), and let~$n\to+\infty$. This theorem is proved at the end of Chapter~\ref{sec-variance}. 

\begin{remarkintro} The version in~\cite{variance} requires the assumptions
  $\deg(f)>900$ and $\gcd(t(t-1), f)=t$. We have greatly relaxed the
  former condition and fully removed the latter, which was recognized as
  being quite artificial (see~\cite[Rem.\,11.0.2]{variance}). These improvements are due to the consideration of the problem in its
  natural setting, involving characters of a torus of dimension
  $\deg(f)$, whereas the authors of~\cite{variance} used cosets of a
  one-dimensional torus together with Katz's work on~$\Gg_m$.
\end{remarkintro}

We also give a proof of an unpublished theorem of Katz~\cite{KatzETH}
answering a question of Tsimerman about equidistribution of Artin
$L$-functions on curves over finite fields.

\begin{theointro}[Katz] Let $C$ be a smooth projective geometrically
  connected curve of genus $g \geq 2$ over a finite field $k$ and let
  $D=\sum n_i x_i$ be a divisor of degree one on $C$. For each
  geometrically non-trivial character
  $\rho \colon \pi_1(C)^{\ab} \to \Cc^\times$ of finite order satisfying 
  $\prod \rho(\Fr_{k(x_i), x_i})^{n_i}=1$, we write its normalized Artin
  $L$-function as
  \[
    L(\rho, T/\sqrt{\abs{k}})=\det(1-T \Theta_{C/k, \rho})
  \]
  for a conjugacy class $\Theta_{C/k, \rho}$ in the unitary group
  $\Un_{2g-2}(\Cc)$.
\begin{enumth}
\item If $C$ is non-hyperelliptic and $(2g-2)D$ is a canonical divisor
  on $C$, then the classes $\Theta_{C/k, \rho}$ lie in $\SU_{2g-2}(\Cc)$
  and become equidistributed with respect to the image on the space of
  conjugacy classes of the Haar probability measure of
  $\SU_{2g-2}(\Cc)$.
\item If $C$ is hyperelliptic, the hyperelliptic involution has a fixed
  point $O \in C(k)$ and $D=O$, then the classes $\Theta_{C/k, \rho}$
  lie in~$\USp_{2g-2}(\Cc)$ and become equidistributed with respect to
  the image on the space of conjugacy classes of the Haar probability 
  measure on~$\USp_{2g-2}(\Cc)$.
\end{enumth}
\end{theointro}

See Chapter~\ref{sec-jacobian} for the proof of this result, as well as
some more general statements (including, in
Theorem~\ref{th-generalized-jacobians}, a result where the algebraic
group~$G$ occurring may involve abelian, toric and unipotent parts).

\section{Outline}\label{ssec-outline}

In this section, we present the plan of the book, and we sketch
one of the main ideas of the proof of
Theorem~\ref{thm:equidis-thm-intro}, in order to point out the key
difficulties for groups of dimension bigger than one, which are solved
using Sawin's quantitative sheaf theory~\cite{sawin_conductors}.

The book is organized as follows:
\begin{itemize}
\item In Chapter~\ref{sec-prelim}, we state some preliminary results;
  these include a survey of the formalism of quantitative sheaf
  theory~\cite{sawin_conductors}, as well as basic structural results
  concerning commutative algebraic groups and character sheaves.
  
\item In Chapter~\ref{sec-vanishing}, we prove the generic and
  stratified vanishing theorems for commutative algebraic groups over
  finite fields.  The very rough idea is to prove a \emph{relative
    version} of the vanishing theorems for the various basic types of
  commutative groups, with a good control of the implicit
  constants. These relative statements are of independent interest. For
  example, in the case of tori, Gabber--Loeser~\cite{GL_faisc-perv}
  prove the stratified vanishing theorem as stated above only under the
  assumption that resolution of singularities over~$k$ holds for varieties of dimension up to
  that of the torus. We remove this assumption using
  alterations. For abelian varieties, we extend Weissauer's
  work~\cite{weissauer_vanishing_2016} by proving a relative version of
  the theorem, which relies on Orgogozo's
  work~\cite{Orgo_constr_mod} on constructibility and moderation.
    
\item In Chapter~\ref{sec-tannakian}, we construct a suitable tannakian
  category of perverse sheaves on a commutative group over a finite
  field with convolution as tensor operation, and establish its basic
  properties, as well as those of the corresponding tannakian monodromy
  group. We will see that some subtleties arise when defining
  ``Frobenius conjugacy classes'' corresponding to characters of~$G$.
  
\item In Chapter~\ref{sec:equidis}, we combine these two ingredients to
  establish a number of ``vertical'' equidistribution theorems; there
  are some issues when we want to refine the statements at the level of
  conjugacy classes (related to those of the previous sections), which
  we are not currently able to solve in full generality, although we can
  always establish equidistribution for the characteristic polynomials.
  
\item The beginning of Part~\ref{part-applications} introduces a selection of first
  applications of a general nature. These include the following:
  \par
  \smallskip
  \par
  \begin{enumerate}
  \item the definition of the analogue of the $L$-function for
    arithmetic Fourier transforms, which is used to study
    finite tannakian groups over abelian varieties
    (Chapter~\ref{sec-L});
  \item a stratification result for exponential sums, similar to those
    of Katz, Laumon and Fouvry, although currently often restricted to
    the ``vertical'' direction (Chapter~\ref{sec-stratification});
  \item a ``generic Fourier invertibility'' result
    (Chapter~\ref{sec-stratification});
  \item some preliminary results of independence of $\ell$ for the
    tannakian group when working with perverse sheaves which are part of
    a compatible system (Chapter~\ref{sec-indep});
  \item various results of ``Diophantine group
    theory'',\index{Diophantine group theory} where averages
    of exponential sums are related to invariants of the tannakian
    group; this includes in particular Larsen's alternative, but also
    some criteria to recognize the exceptional group $\mathbf{E}_6$
    (Chapter~\ref{sec-larsen}).
  \end{enumerate}
  
\item Chapters~\ref{sec-product},~\ref{sec-variance}
  and~\ref{sec-jacobian} contain applications to concrete cases. The
  algebraic groups involved are, respectively, the product
  $\Gg_a\times \Gg_m$, higher-dimensional tori, and jacobians of curves,
  as well as the intermediate jacobian of a smooth cubic threefold
  (where the relevant tannakian group is $\mathbf{E}_6$, as first shown
  in the complex setting by Krämer).

\item In Chapter~\ref{sec-problems}, we list some open
  questions and problems. The title   ``Much remains to be done'' paraphrases Katz (\cite[p.\,18]{mellin}). 

\item Finally, we include appendices to survey the basic theory of
  perverse sheaves (Appendix~\ref{ch-app-perverse}), as well as to recall the most
  important results of Katz concerning the arithmetic Mellin transform
  on $\Gg_m$ (Appendix~\ref{ch-app-mellin}) and the product
  formula of Laumon for the epsilon factor of $L$-functions over finite
  fields (Appendix~\ref{ch-app-product}). In Appendix~\ref{ch-app-letter}, we reproduce, with Deligne's permission, the letter to Kazhdan in which the
  $\ell$-adic Fourier transform was first discussed. To conclude, we attempt to sketch the intuitive
  nature of the theory of general trace functions, to provide some
  intuition for analytic number theorists in
  Appendix~\ref{ch-app-analytic}.
\end{itemize}

We now survey the key analytic step in the proof of Theorem
\ref{thm:equidis-thm-intro} (see Proposition~\ref{pr-weyl-sum}).

By fixing an isomorphism $\iota_0\colon \Qlb\to \Cc$, we can work with trace functions and characters with values in $\Qlb$. The first step, following from the generic
vanishing theorem, will be to prove that there exist
subsets~$\mathcal{Y}(k_n)\subset \widehat{G}(k_n)$ of characters and conjugacy
classes $\Theta_{M,k_n}(\chi)$ in some unitary group $\Un_r(\Cc)$ such
that $\Tr(\Theta_{M,k_n}(\chi))=S(M,\chi)$ holds for all $\chi\in\mcY(k_n)$ and
$$
|\mathcal{Y}(k_n)|\sim |G(k_n)|
$$
as $n\to +\infty$. The second step (an application of the theory of
tannakian categories) will be an intrinsic a priori definition of the
compact group~$K$ for which equidistribution should hold.
\par
By (essentially) the Weyl criterion\index{Weyl criterion} for
equidistribution, Theorem~\ref{thm:equidis-thm-intro} will follow from
the proof that, for every non-trivial irreducible representation $\rho$
of the unitary group~$\Un_r(\Cc)$, the limit
\[
  \lim_{N\to+\infty} \frac{1}{N}\sum_{1\leq n\leq N} \frac{1}{|G(k_n)|}
  \sum_{\chi \in \mathcal{Y}(k_n)} \Tr\rho(\Theta_{M,k_n}(\chi))
\]
exists and is equal to the multiplicity of the trivial representation in
the restriction of~$\rho$ to~$K$. Now, the tannakian formalism associates to each $\rho$ a perverse sheaf $\rho(M)$ on $G
$ such that the equality
$$
\Tr\rho(\Theta_{M,k_n}(\chi))= \sum_{x\in G(k_n)}\chi(x)
t_{\rho(M)}(x;k_n)
$$
holds for $n\geq 1$ and $\chi\in\mathcal{Y}(k_n)$. The Grothendieck--Lefschetz trace formula
yields then the
equality
\begin{equation}\label{eq-trace}
  \sum_{x\in G(k_n)}\chi(x)
  t_{\rho(M)}(x;k_n)
  = \sum_{|j|\leq d} (-1)^j \Tr\big(\frob_{k_n}\mid
  \rmH^j_c(G_{\bar{k}},\rho(M)\otimes\mcL_{\chi})\big)
\end{equation}
for~$n\geq 1$ and \emph{any} character $\chi$ of $G(k_n)$, where
$\frob_{k_n}$ is the geometric Frobenius automorphism. 

The definition of the set $\mathcal{Y}(k_n)$ 
implies the property that for~$\chi\in\mathcal{Y}(k_n)$,
the only possibly non-zero term in the right-hand side
of~(\ref{eq-trace}) is the one with~$j=0$. Thus we have
$$
\sum_{\chi \in\mathscr{Y}(k_n)} \Tr\rho(\Theta_{M,k_n}(\chi)) =\sum_{\chi
  \in\mathscr{Y}(k_n)} \Tr\big(\frob_{k_n}\mid
\rmH^0_c(G_{\bar{k}},\rho(M)\otimes\mcL_{\chi})\big).
$$
\par
If we add to the right-hand side of this last expression the two sums
\begin{gather*}
  S_1= \sum_{\chi \notin\mathscr{Y}(k_n)}\Tr\big(\frob_{k_n}\mid
  \rmH^0_c(G_{\bar{k}},\rho(M)\otimes\mcL_{\chi})\big),
  \\
  S_2=\sum_{1\leq |j|\leq d}\sum_{\chi
    \notin\mathscr{Y}(k_n)}(-1)^j\Tr\big(\frob_{k_n}\mid
  \rmH^j_c(G_{\bar{k}},\rho(M)\otimes\mcL_{\chi})\big),
\end{gather*}
then the resulting quantity is
\begin{align*}
  \sum_{|j|\leq d}\sum_{\chi\in \widehat{G}(k_n)} (-1)^j
  \Tr\big(\frob_{k_n}\mid \rmH^j_c(G_{\bar{k}},\rho(M)\otimes\mcL_{\chi})\big)&=
  \sum_{\chi\in \widehat{G}(k_n)}\sum_{x\in G(k_n)}\chi(x)
  t_{\rho(M)}(x;k_n)\\
  &=\sum_{x\in G(k_n)}t_{\rho(M)}(x;k_n)
  \sum_{\chi\in\widehat{G}(k_n)}\chi(x) \\
  &= |G(k_n)|\ t_{\rho(M)}(1;k_n)
\end{align*}
by the trace formula again, followed by an exchange of the sums and an
application of the orthogonality of characters of finite abelian groups.
This is a single value of the trace function, and it is relatively
straightforward to show that it gives the desired multiplicity as
limit. So \emph{the key difficulty is to control the two auxiliary
  sums~$S_1$ and~$S_2$}.
\par
This can be done if:
\begin{enumerate}
\item We have some bound on the size of the individual traces
  $\Tr\big(\frob_{k_n}\mid
  \rmH^j_c(G_{\bar{k}},\rho(M)\otimes\mcL_{\chi})\big)$;
\item We have some bound on the number of characters $\chi$ such that $\rmH^j_c(G_{\bar{k}},\rho(M)\otimes\mcL_{\chi})$ can be
  non-zero in a given degree~$j$.
\end{enumerate}

The second bound is given by the stratified vanishing theorem for
$\rho(M)$. For the first, Deligne's Riemann hypothesis (see
Theorem~\ref{th-deligne-riemann}) implies the inequality
$$
\big|\Tr\big(\frob_{k_n}\mid \rmH^j_c(G_{\bar{k}},\rho(M)\otimes\mcL_{\chi})\big)\big|
\leq |k_n|^{(j-d)/2}\dim
\rmH^j_c(G_{\bar{k}},\rho(M)\otimes\mcL_{\chi}),
$$
and we see that we require a bound on the dimension of the cohomology
spaces, which should be independent of~$\chi$. We obtain such bounds as
special cases of Sawin's quantitative sheaf
theory~\cite{sawin_conductors}, which is a quantitative form of the
finiteness theorems for the six operations on the derived category of
$\ell$-adic sheaves on quasi-projective algebraic varieties.

\begin{remarkintro}
  If~$G$ is one-dimensional, then the Euler--Poincaré characteristic
  formula (see Theorem~\ref{th-euler-poincare}) easily implies precise
  bounds on the dimension of the cohomology spaces that arise, and hence
  this critical issue does not arise for the additive or multiplicative
  groups, or for elliptic curves (for such groups,
  Theorem~\ref{thm:vanishing-thm-intro} is also straightforward). It
  also does not arise if the set of ``good'' characters $\mathscr{Y}(k_n)$ is the whole
  group~$\widehat{G}(k_n)$, which is the case in some instances
  considered by Katz for higher-dimensional abelian varieties.
\end{remarkintro}


\section{Conventions and notation}
\label{sec-conventions}

We summarize the notation that we use, as well as some typographical
conventions that we follow consistently unless otherwise specified. 

Given complex-valued functions $f$ and $g$ defined on a set $S$, we
write $f \ll g$\nomenclature{$f\ll g$}{asymptotic notation} if there exists a real number $C \geq 0$ (called an
``implicit constant'') such that the inequality $|f(s)|\leq C g(s)$
holds for all $s \in S$. We write $f\asymp g$ if $f\ll g$ and~$g\ll
f$.\nomenclature{$f\asymp g$}{asymptotic notation}
If $f$ and $g$ are defined on a topological space~$X$, and
$\mathfrak{F}$ is a filter on~$X$, then we say that $f\sim g$ along
$\mathfrak{F}$ if $\lim_{\mathfrak{F}} f(x)/g(x)=1$.\nomenclature{$f\sim
  g$}{asymptotic notation}

For any complex number $z$, we write $e(z)=\exp(2i\pi z)$; the value $e(a/q)$ is well-defined\nomenclature{$e(z)$}{$\exp(2i\pi z)$} for $q\geq 1$
and $a\in\Zz/q\Zz$. 

By a \emph{variety over a field} $k$,\index{algebraic variety} we mean a reduced separated
$k$-scheme of finite type. In particular, an \emph{algebraic
  group},\index{algebraic group} as
opposed to a group scheme, is always supposed to be reduced, and hence
smooth if the field $k$ is perfect.

Let~$S$ be a scheme. We say that a pair $(X,u)$ is \emph{a
  quasi-projective scheme over~$S$}\index{quasi-projective scheme} if $X$ is a scheme over~$S$ and~$u$
is a locally-closed immersion $u\colon X\to \Pp^n_S$ for some
integer~$n \geq 0$. We call $n$ the embedding dimension\index{embedding dimension} of $(X,u)$, or
simply of~$u$, and we say that $u$ is a \emph{quasi-projective
  embedding} of~$X$.  When $S$ is the spectrum of a field~$k$ and~$X$ is a
variety over~$k$, we will speak of quasi-projective
varieties\index{quasi-projective variety}
over~$k$. In some cases, we omit the mention of~$u$, when it is clear from the
context which locally-closed immersion is used. By a morphism
$f\colon (X,u)\to (Y,v)$ of quasi-projective schemes over~$S$, we mean an $S$-morphism of the underlying schemes. 

An algebraic group $\mathbf{G}$ over an algebraically closed field of
characteristic zero is called \emph{reductive}\index{reductive algebraic
  group} if all its
finite-dimensional representations are completely reducible (that is, we
do not require $\mathbf{G}$ to be connected).

Let $X$ be a scheme and $\ell$ a prime number invertible on $X$.

By a $\Qlb$-sheaf on~$X$, we always mean a constructible étale $\Qlb$-sheaf. Perverse sheaves (when~$X$ is an algebraic variety defined over a
field~$k$) are always considered with respect to the middle
perversity. We include a short survey of the most important properties
of perverse sheaves in Appendix~\ref{ch-app-perverse}, but recall here
some of the definitions. An $\ell$-adic complex is said to be
\emph{semiperverse} if the inequality 
$$\dim \supp \mcH^i(M) \leq -i$$
holds for any integer $i$. This is equivalent to
asking that the perverse cohomology sheaves~$\pH^i(M)$ are zero for
$i\geq 1$ (see~\cite[Prop.\,1.3.7]{BBD-pervers}).

We say that a complex $M$ in $\Der(X,\Qlb)$ has
\emph{perverse amplitude}\index{perverse amplitude} $[a, b]$ if its perverse cohomology sheaves
$\pH^i(M)$ are zero for $i \notin [a, b]$.


A \emph{stratification}\index{stratification} $\mcX$ of $X$ is a finite set-theoretic
partition of the associated reduced scheme $X^{\mathrm{red}}$ into
non-empty reduced locally-closed subschemes of $X$, called the
\emph{strata} of~$\mcX$.

Let~$\mcX$ be a stratification of~$X$, and let~$\mcF$ be an
$\ell$-adic sheaf on~$X$. The sheaf~$\mcF$ is said to be \emph{tame
  and constructible along $\mcX$}\index{sheaf tame and constructible
  along $\mcX$} if it is tamely ramified, as in~\cite[\S 1.3.1]{Orgo_constr_mod}, and if its
restriction to any strat of $\mcX$ is a lisse sheaf. More generally, a
complex $M\in \Der(X,\Qlb)$ is said to be tame and constructible along
$\mcX$ if all its cohomology sheaves are tame and constructible along
$\mcX$.

Let $f\colon X\to Y$ be a morphism of schemes. For an object $M$ of
$\Der(X,\Qlb)$, we write $Rf_!M=Rf_*M$ to indicate that the canonical
``forget supports'' morphism $Rf_!M\to Rf_*M$ is an isomorphism (and
similarly for equality of cohomology groups with and without compact
support).

Let $q\geq 1$ and $w\in\Zz$ be integers. A complex number $\alpha$ is
called a \emph{$q$-Weil number}\index{Weil number} of weight~$w$ if $\alpha$ is algebraic
over~$\Qq$ and all its Galois conjugates have modulus $q^{w/2}$. If~$k$
is a finite field, then a $k$-Weil number is a $|k|$-Weil number. 

\par
\medskip
\par

\emph{Throughout, for any prime $\ell$, we consider a fixed isomorphism
  $\iota_0\colon \Qlb\to\Cc$. Trace functions of $\ell$-adic perverse
  sheaves are thus always identified with complex-valued functions
  through~$\iota_0$, and similarly~$\ell$\nobreakdash-adic characters are identified with
  complex characters. On the other hand, purity of perverse sheaves (or
  lisse sheaves or $\ell$-adic complexes) refers to purity in the sense
  of Deligne, i.e., pointwise purity means that the eigenvalues of
Frobenius are Weil numbers of some weight; see the survey in
Section}~\ref{sec-weights}. 

\par
\medskip
\par
The following notation is used consistently in all the book, although
frequently with reminders (some objects, such as character sheaves, will
be defined later).

\begin{itemize}
\item $X\setminus Y$: difference set (elements of~$X$ that are not
  in~$Y$); also used in scheme-theoretic settings.
\item $M|X$ or $M_{\mid X}$: restriction of an object~$M$ (or a section
  of a sheaf) to a subset or subscheme~$X$.
\item $\abs{X}$: cardinality of a set~$X$.
\item $\tau(\chi,\psi)$: (unnormalized) Gauss sum attached to a
  multiplicative character~\hbox{$\chi\colon k^{\times}\to \Qlb^{\times}$}
  and an additive character~$\psi\colon k\to \Qlb^{\times}$ for a finite
  field~$k$, i.e.
  \begin{equation}\label{eq-gauss-sum}
    \tau(\chi,\xi)=\sum_{x\in k^{\times}}\chi(x)\psi(x).
  \end{equation}%
  \index{Gauss sum}
  \nomenclature[$tau$]{$\tau(\xi,\psi)$}{Gauss sum}
\item $\Der(X)=\Der(X,\Qlb)$\nomenclature[$D$]{$\Der(X)=\Der(X,\Qlb)$}{category
    of bounded constructible complexes of $\Qlb$-sheaves on $X$}:
  category of bounded constructible complexes of $\Qlb$-sheaves on a
  scheme $X$ such that the prime~$\ell$ is invertible in~$X$.
\item $K(X)=K(X,\Qlb)$\nomenclature{$K(X)=K(X,\Qlb)$}{Grothendieck group}: the Grothendieck group (or ring) of $\Der(X)$;
  it has a basis consisting of classes of simple perverse sheaves
  (see~\cite[\S\,0.8]{laumon-signes}).
\item $\alpha^{\deg}$\nomenclature[$alpha$]{$\alpha^{\deg}$}{geometrically
    trivial lisse sheaf}: for $k$ a finite field and $\alpha$ an
  $\ell$-adic unit, the $\ell$-adic sheaf of rank~$1$ on $\Spec(k)$ on
  which the geometric Frobenius acts by multiplication by $\alpha$; more
  generally, for $f\colon X\to \Spec(k)$ a scheme over~$k$, the
  pullback of $\alpha^{\deg}$ to~$X$.  
\item $M\otimes N$: derived tensor product of objects of $\Der(X)$.
\item $M\boxtimes N$: for~$M$ an object of~$\Der(X)$ and~$N$ an object
  of~$\Der(Y)$, the object $p_1^*M\otimes p_2^*N$ on~$X\times Y$,
  where~$p_1$ and~$p_2$ are the two
  projections.\nomenclature{$M\boxtimes N$}{external tensor product}
\item
  $\Perv(X)=\Perv(X,\Qlb)$\nomenclature[$P$]{$\Perv(X)=\Perv(X,\Qlb)$}{category
  of perverse sheaves}: the category of $\ell$-adic perverse
  sheaves on $X$. A simple perverse sheaf will also sometimes be called
  an irreducible perverse sheaf.
\item $\DD(M)$\nomenclature[$D$]{$\DD(M)$}{Verdier dual}: the Verdier dual of a complex $M$.
\item $\mcH^i(M)$: for $M\in \Der(X)$, the $i$-th cohomology sheaf of
  $M$.
\item $\pH^i(M)$: for $M\in \Der(X)$, the $i$-th perverse cohomology
  sheaf of $M$.
\item $\rmH^i(M)=\rmH^i(X_{\bar{k}},M)$: the étale cohomology groups
  of the pull-back of~$M$ to $X\times_k \bar k$.
\item $\rmH^i_c(M)=\rmH^i_c(X_{\bar{k}},M)$: the étale cohomology groups
  with compact support of $M$.
\item $h^i(X_{\bar{k}},M)=\dim \rmH^i(X_{\bar{k}},M)$.
\item $h^i_c(X_{\bar{k}},M)=\dim \rmH^i_c(X_{\bar{k}},M)$.
\item $\rmH^{*}(X_{\bar{k}},M)$ or $\rmH^{*}_c(X_{\bar{k}},M)$: the
  graded vector space which is the direct sum of all cohomology spaces
  $\rmH^i(X_{\bar{k}},M)$ or $\rmH^i_c(X_{\bar{k}},M)$.
\item $\chi(X_{\bar{k}},M)$ or $\chi_c(X_{\bar{k}},M)$: Euler--Poincaré
  characteristic for cohomology or cohomology with compact support.
\item $t_M(x;k_n)$\nomenclature{$t_M(x; k_n)$}{trace function of~$M$
    on~$X(k_n)$}: Frobenius trace function of an object $M$ of~$\Der(X)$
  for~$x\in X(k_n)$. For $x \in X(k)$, we usually abbreviate it by $t_M(x)=t_M(x;k)$.\nomenclature{$t_M(x)=t_M(x;k)$}{trace function
    of~$M$ on~$X(k)$}
\item $\braket{M}$\nomenclature{$\braket{M}$}{tannakian category generated by~$M$}: tannakian category generated by $M$.
\item $\garith{M}\ (\resp \ggeo{M}$): arithmetic (\resp geometric)
  tannakian group associated with a perverse sheaf $M$.
\item $\charg{G}(k_n)$: group of $\Qlb$-characters of the finite group
  $G(k_n)$.
\item $\charg{G}$: disjoint union of the sets $\charg{G}(k_n)$ for
  $n\geq 1$.
\item $\Pi(G)$: for a semiabelian variety $G$, the $\Qlb$-scheme of
  $\ell$-adic characters of $G$.
\item $\mcL_\chi$: character sheaf on $G_{k_n}$ associated to a
  character $\chi\in \charg{G}(k_n)$.
\item $\mcL_{\chi(f)}$: for $f\colon X\to G$ and $\chi \in\charg{G}(k_n)$, the sheaf $f^*\mcL_\chi$ on $X$. 
\item $M_\chi$: for an object $M$ of $\Der(G)$ and a character $\chi$,
  the object $M\otimes \mcL_\chi$.
\end{itemize}

Moreover, the following notational conventions will be used (often
with reminders).

\begin{itemize}
\item $k$: a finite field of characteristic $p$.
\item $\ell$: a prime different from $p$.
\item $\bar k$: an algebraic closure of $k$.
\item $k_n$: the extension of degree $n$ of $k$ inside~$\bar k$.
\item $G$: a connected commutative algebraic group (in particular of
  finite type) defined over $k$.
\item $T$: a torus; 
\item $U$: a unipotent group; 
\item $A$: an abelian variety.
\item $\mcF$: a $\Qlb$-sheaf; 
\item $\mcL$: a $\Qlb$-lisse sheaf of rank one.
\item $M, N$: objects of~$\Der(X)$ or~$\Perv(X)$.
\end{itemize}







\part{Theoretical foundations}

\numberwithin{section}{chapter}
\numberwithin{equation}{chapter}

\chapter{Preliminaries}\label{sec-prelim}

In this chapter, we summarize some tools we use throughout the book,
especially the properties of Sawin's quantitative sheaf
theory~\cite{sawin_conductors} with an emphasis on commutative algebraic
groups.

\section{Specializations of perverse sheaves}

We will frequently use the following result concerning specializations
of perverse sheaves. 

\begin{proposition}\label{pr-kl}
  Let~$k$ be a field, $f\colon Y\to X$ a surjective affine
  morphism of varieties over~$k$, and $M$ a perverse sheaf
  on~$Y$. For all closed points~$x$ outside of a closed strict subvariety of~$X$, the
  object $M|_{f^{-1}(x)}[-1]$ is perverse on~$f^{-1}(x)$.
\end{proposition}

This follows directly
from~\cite[Ch.\,III, Lemma\,6.3]{kiehlWeilConjecturesPerverse2001}.






  

\section{Review of quantitative sheaf theory}

Let $k$ be a field, $\bar{k}$ an algebraic closure of $k$, and $\ell$
a prime different from the characteristic of $k$.

\begin{definition}[Complexity]\label{def-complexity}\mbox{}\index{complexity} 
  Let $M_{n+1, m+1}$ be the
  variety of $(n+1)\times (m+1)$ matrices of maximal rank, viewed as
  an affine scheme over~$k$. For each $0 \leq m \leq n$, consider a
  geometric generic point~$a_m$ of $M_{n+1,m+1}$ defined over an
  algebraically closed extension $K$ of~$k$, and let
  $l_{a_m}\colon \Pp^{m}_K\to \Pp^{n}_K$ denote the associated linear
  map.
  
  \begin{enumerate}
  \item[(a)] The \emph{complexity} of an object $M$ of
    $\Der(\Pp^n_{k})$ is defined as
    \[
      c(M)=\max_{0\leq m\leq n}\sum_{i\in \Zz}h^i(\Pp^n_K,M\otimes
      l_{a_{m}\ast}\Qlb)=\max_{0\leq m\leq n}\sum_{i\in \Zz}h^i(\Pp^m_K, l_{a_m}^*M),
    \]
    where the last equality follows from the projection formula.
    \nomenclature{$c(M)$}{complexity on projective space}
  \item[(b)] Let $(X,u)$ be a quasi-projective variety over~$k$. For
    any object $M$ of $\Der(X)$, the \emph{complexity} of $M$ with respect
    to $u$ is defined as
    $c_{u}(M)=c(u_{!}M)$.\nomenclature{$c_u(M)$}{complexity on
      quasi-projective variety}
  \end{enumerate} 
\end{definition}

The invariance of étale cohomology under base change between
algebraically closed fields implies that the complexity is
well-defined (\ie it does not depend on the choice of fields of
definition of the generic points~$a_m$).

\begin{lemma}\label{lm-sums-betti}
  Let $(X,u)$ be a quasi-projective variety over~$k$ and let~$M$ be an
  object of~$\Der(X)$. The following inequality holds:
  \begin{equation}\label{eqn:boundBetti}
    \sum_{i \in \Zz} h^i_c(X_{\bar{k}}, M) \leq c_u(M).
  \end{equation}
\end{lemma}

\begin{proof}
  This follows from the equality
  $h^i_c(X_{\bar{k}}, M)=h^i(\Pp^n_{\bar k}, u_{!}M)$ and the invariance of
  étale cohomology under extension of scalars between algebraically
  closed fields, combined with the fact that
  $l_{a_n} \colon \Pp^n_K \to \Pp^n_K$ is an isomorphism.
\end{proof}

\begin{definition}
  Let $f\colon (X,u)\to (Y,v)$ be a morphism of quasi-projective
  varieties over $k$ with embedding dimensions $n_X$ and $n_Y$
  respectively. For all integers $0 \leq m_X \leq n_X$ and~\hbox{$0 \leq m_Y \leq n_Y$}, consider geometric generic points
  $a_{m_X}$ of $M_{n_X+1,m_X+1}$ and $b_{m_Y}$ of $M_{n_Y+1,m_Y+1}$ defined
  over an algebraically closed extension $K$ of~$k$, and let
  $l_{a_{m_X}}\colon\Pp^{m_X}_K\to \Pp^{n_X}_K$ and
  $l_{b_{m_Y}}\colon \Pp^{m_Y}_K\to \Pp^{n_Y}_K$ denote the associated
  linear maps. The \emph{complexity} of $f$ is defined as
  \[
    c_{u,v}(f)=\max_{0\leq m_X\leq n_X}\max_{0\leq m_Y\leq
      n_Y}\sum_{i\in \Zz}h^i_c(X_K,u^*{l_{a_{m_X\ast}}}\Qlb\otimes f^*v^*
    {l_{b_{m_X\ast}}}\Qlb).
\]
\end{definition}

The main result of~\cite{sawin_conductors} establishes, among other things, the
``continuity'' of the six operations on the derived category with respect
to the complexity. In this result and the remainder of this
section, the implicit constants only depend on the embedding dimensions
of the quasi-projective varieties, unless otherwise specified.

\begin{theorem}[{\cite[Th.\,6.8,\,Prop.\,6.14,\,Prop.\,6.12]{sawin_conductors}}]
  \label{thm-conductors}
  Let $f\colon (X,u)\to (Y,v)$ be a morphism of quasi-projective
  varieties over~$k$. Let $M,N,P$ be objects of $\Der(X)$ and let
  $Q$ be an object of $\Der(Y)$. The following inequalities hold:
  \begin{enumth}
  \item $c_u(M\oplus N) \leq c_u(M)+c_u(N)$.
  \item\label{thm-conductors:item2} $c_u(M\otimes N)\ll c_u(M)c_u(N)$.
  \item If $M\to N\to P$ is a distinguished triangle, then
    $c_u(N)\leq c_u(M)+c_u(P)$.
  \item $c_u(M[k])= c_u(M)$ for any~$k\in\Zz$.
  \item $c_u(\mathrm{RHom}(M,N))\leq c_{u,\id}(u)c_u(M)c_u(N)$.
  \item\label{thm-conductors:item6} $c_v(Rf_!M)\ll c_{u,v}(f)c_u(M)$ and
    $c_v(Rf_*M)\ll c_{u,\id}(u)c_{v,\id}(v)c_{u,v}(f)c_u(M)$.
  \item $c_u(f^*Q)\ll c_{u,v}(f)c_v(Q)$ and
    $c_u(f^!Q)\ll
    c_{u,\id}(u)c_{v,\id}(v)c_{u,v}(f)c_v(Q)$.\label{eq-7}
  \item\label{thm-boxtimes}
    $c_{u\boxtimes v}(M\boxtimes Q)\ll c_{u\boxtimes
      v,u}(p_1)c_{u\boxtimes v,v}(p_2)c_u(M)c_v(Q)$.\label{eq-8}
  \end{enumth}
  
  In the last of these inequalities, $u\boxtimes v$ is the composition
  of~$u\times v$ with the Segre embedding and $p_1$, $p_2$ are the
  projections $X\times Y\to X$ and $X\times Y\to Y$, respectively.
\end{theorem}

\begin{remark}\label{rm-change-embedding}
  Although the notion of complexity on a quasi-projective scheme $(X,u)$
  depends on the quasi-projective immersion~$u$, note that if $v$ is
  another quasi-projective immersion of~$X$, then applying the
  property~\ref{eq-7} to the identity morphism between~$(X,u)$
  and~$(X,v)$, we get
  $$
  c_u(M)\asymp c_v(M)
  $$
  for all objects~$M$ of~$\Der(X)$, where the implied constants are
  essentially $c_{u,v}(\mathrm{Id})$ and~$c_{v,u}(\mathrm{Id})$, up to
  constants depending on the embedding dimensions of~$u$ and~$v$. Thus,
  as long as we only consider on~$X$ an absolutely bounded number of
  different quasi-projective immersions, we can think of the complexity
  as being essentially independent of them. (This is reminiscent of
  similar properties of height functions in diophantine geometry.)
\end{remark}

The complexity can also be used to control the degree of the locus where
a complex of sheaves is lisse, and of the locus where the generic base
change theorem holds.

\begin{theorem}[{\cite[Th.\,6.23]{sawin_conductors}}]
  \label{thm-lisse-locus-conductor}
  Let $(X,u)$ be an irreducible quasi-projective variety over~$k$. Let~$M$ be an object of $\Der(X)$. Let~$Z$ be the complement of the
  maximal open subset where~$X$ is smooth and~$M$ is lisse.  Then the
  estimate
  \[
    \deg (u(Z)) \ll (3 + s)c(u) c_u(M)
  \]
  holds, where the degrees are computed in the projective space target
  of~$u$, and~$s$ is the degree of the codimension $1$ part of the
  singular locus of $X$.
 \end{theorem}

\begin{theorem}[{\cite[Th.\,6.27]{sawin_conductors}}]
  \label{thm-gen-bc-conductors}
  Let $(X,u)$, $(Y,v)$ and $(S,w)$ be quasi-projective algebraic
  varieties over~$k$. Let~$f\colon X\to Y$ and $g\colon Y\to S$ be
  morphisms.
  \par
  For any object $M$ of~$\Der(X)$, there exists an integer $C\geq 0$,
  depending only on $c_u(M)$ and the morphisms $(f,g,u,v,w)$, and a dense open set
  $U\subset S$ such that: 
  \par
  \emph{(i)} The image of the complement of $U$ has degree~$\leq C$.
  \par
  \emph{(ii)} The object $f_*M$ is of formation compatible with any base
  change $S'\to U\subset S$.
\end{theorem}

\begin{proposition}[{\cite[Th.\,6.15]{sawin_conductors}}]
  \label{prop-JorHol-cond}
  Let $(X,u)$ be a quasi-projective variety over~$k$. Let $M$ be an
  object of $\Der(X)$. For each integer~$i$, let
  $M_{i, 1},\dots, M_{i, n_i}$ denote the Jordan\nobreakdash--Hölder
  factors of the perverse cohomology sheaf $\pH^i(M)$. Then the following estimate holds: 
  \[
    \sum_{i\in \Zz} \sum_{1\leq j\leq n_i} c_u(M_{i, j})\ll
    c_{u,\id}(u)c_u(M).
  \]
\end{proposition}

We also recall the quantitative statement of the Riemann
hypothesis\index{Riemann Hypothesis} over finite fields when
interpreted as a quasi-orthogonality statement.\index{quasi-orthogonality}

\begin{theorem}[{\cite[Th.\,7.13\,(2)]{sawin_conductors}}]
  \label{th-rh}
  Let~$k$ be a finite field and $\ell$ a prime different from the
  characteristic of~$k$. Let~$(X,u)$ be a geometrically irreducible
  quasi-projective algebraic variety over~$k$. Let~$M$ and~$N$ be
  geometrically simple $\ell$-adic perverse sheaves on~$X$ that are
  pure of weight zero, with complex trace functions~$t_M$ and~$t_N$
  respectively. Then the estimate
  $$
  \sum_{x\in X(k)} t_M(x)\overline{t_N(x)} \ll
  c(u)c_u(M)c_u(N)|k|^{-1/2}
  $$
  holds if~$M$ and~$N$ are not geometrically isomorphic, whereas
  $$
 \sum_{x\in X(k)} |t_M(x)|^2= 1+ O(
  c(u)c_u(M)^2|k|^{-1/2}).
  $$
  In both estimates, the implied constants only depend on the
  embedding dimension and are effective.
\end{theorem}

Finally, we have pointwise bounds for the trace functions.

\begin{proposition}[{\cite[Prop.\,7.11\,(2)]{sawin_conductors}}]
  \label{pr-pointwise-bound}
  Let~$k$ be a finite field and $\ell$ a prime different from the
  characteristic of~$k$. Let~$(X,u)$ be a quasi-projective algebraic
  variety over~$k$, and let~$M$ be a non-punctual simple perverse sheaf
  on~$X$ which is pure of weight zero. For any $n\geq 1$
  and~$x\in X(k_n)$, the following estimate holds: 
  $$
  t_M(x;k_n)\ll \frac{1}{|k_n|^{1/2}}.
  $$
\end{proposition}

\section{Existence of rational points}

The following lemma is standard.

\begin{lemma}\label{lm-lang-weil}
  Let $(X,u)$ be a non-empty quasi-projective variety over a finite
  field~$k$ with embedding dimension~$n$. There exists a finite
  extension $k'$ of $k$ with degree bounded in terms of~$(\dim(X), \deg(\overline{u(X)}),n)$ such that $X(k')$ is
  non-empty.
\end{lemma}

\begin{proof}
  Write the variety $u(X)$ as $Z\setminus Y$ for some closed
  subvarieties~$Z$ and~$Y$. Then the degrees of $Z$ and $Y\cap Z$ are
  bounded in terms of $c_u(X)$ by~\cite[Lemma\,6.26]{sawin_conductors},
  and the result then follows from the Lang--Weil bound 
  (see~\cite[Th.\,1]{lang-weil}) applied to~$Z$ and to $Y\cap Z$.
\end{proof}

\section{Structure of commutative algebraic groups}
\label{sec_com_alg_gp}

Let~$k$ be a field and let $G$ be a commutative algebraic group
over~$k$. The algebraic variety~$G$ is quasi-projective (see,
e.g.,~\cite[Prop.\,A.3.5]{cgp} or~\cite[\href{https://stacks.math.columbia.edu/tag/0BF7}{Lemma 0BF7}]{stacks}. We will
always assume that a quasi-projective immersion~$u$
of~$G$ is given, and the complexity of $\ell$-adic complexes will be understood
with respect to~$u$ (so that we sometimes write just~$c(M)$ instead
of~$c_u(M)$). If $G$ is either a power of~$\Gg_a$ or of~$\Gg_m$, we
assume that~$u$ is simply the obvious embedding in the projective space
of the same dimension. We will on occasion use auxiliary quasi-projective
immersions and rely on Remark~\ref{rm-change-embedding} to compare complexities. 

Smooth connected commutative algebraic groups can be built as successive extensions of abelian
varieties, tori, unipotent\footnote{In this book, ``unipotent'' only applies to commutative groups} and finite commutative group
schemes. The most convenient formulation of this fact for us is the following statement,
which follows from results of Barsotti--Chevalley and Rosenlicht (see
for instance the account in the book of Brion, Samuel and Umae,
combining~\cite[Cor.\,5.5.2]{BSU_alg_gps} with the structure theorem for
connected affine commutative algebraic groups over perfect fields as a
product of a unipotent group and a torus; see, e.g.,~\cite[Th.\,5.3.1\,(2)]{brion}).

\begin{proposition}\label{pr-devissage}
  Let~$k$ be a finite field and let~$G$ be a connected commutative
  algebraic group over~$k$. There exist an abelian variety~$A$, a
  torus~$T$, a unipotent group~$U$ and a finite commutative subgroup
  scheme~$N$ of~$A\times U\times T$, all defined over~$k$, such that $G$
  is isomorphic to $(A\times U\times T)/N$. 
\end{proposition}

We further recall that a finite commutative group scheme~$N$ over a
perfect field has a unique direct product decomposition
$N=N_r\times N_l$ where~$N_r$ is reduced and~$N_l$ is local (i.e., equal
to its connected component of the identity; see,
e.g.,~\cite[Prop.\,2.5.4]{brion}).


\section{Convolution}\label{sec-convolution}

Let~$G$ be a commutative algebraic group
over a field~$k$. We denote by 
\[
m \colon G \times G \to G, \quad \mathrm{inv} \colon G \to G, \quad e \in G(k)
\] the group law, the inversion morphism, and the neutral element respectively. 
\nomenclature{$e$}{neutral element of a group}
\nomenclature{$m$}{group law}
\nomenclature[$inv$]{$\mathrm{inv}$}{inverse map on a group}

\begin{definition}[Convolution]\index{convolution}
  \index{convolution with compact support}
The
\emph{convolution product} and the \emph{convolution product with compact support} on~$G$
are the functors from $\Der(G)\times \Der(G)$ to~$\Der(G)$ defined as
\[
  M \ast_\ast N=Rm_\ast(M \boxtimes N), \quad\quad M \ast_! N=Rm_!(M
  \boxtimes N)
\]
for objects~$M$ and~$N$ of~$\Der(G)$.
\nomenclature{$M*_*N$}{convolution}
\nomenclature{$M*_{#!} N$}{convolution with compact support}
\end{definition}

If $G$ is projective, then so is the morphism $m$, and hence the two
convolutions agree. In general, there is a canonical ``forget supports''
\index{``forget support'' morphism}
morphism
\[
M \ast_! N \longrightarrow M \ast_\ast N. 
\]
We will write $M\ast_! N=M\ast_{\ast} N$ when this morphism is an
isomorphism.

If~$u$ is a quasi-projective immersion of~$G$, then we deduce from
Theorem~\ref{thm-conductors}\,\ref{thm-boxtimes} that for any
objects~$M$ and~$N$, the following estimates hold:
\[
c_u(M*_* N)\ll c_u(M)c_u(N),\quad\quad c_u(M *_! N)\ll c_u(M)c_u(N),
\]
where the implied constant depends on~$G$, and is uniform in families
(see~\cite[\S\,6.5]{sawin_conductors}).

For an object~$M$ of~$\Der(G)$, we define
\[
  M^\vee=\inv^*\DD(M), 
\]
where $\DD(M)$ is the
Verdier dual.\nomenclature{$M^\vee$}{tannakian dual of $M$}  Since $\inv^*=\inv^!$ commutes with $\DD$, the functor $M \mapsto M^\vee$
is an involution, in the sense that the functor $M\mapsto (M^\vee)^\vee$
is canonically isomorphic to the identity. 

We denote by $\un$\nomenclature[$1$]{$\un$}{skyscraper sheaf, unit for
  convolution}
the skyscraper sheaf\index{skyscraper sheaf} supported at the neutral
element $e$ of $G$.

The basic formal properties of the convolution products are given by the
following lemma:

\begin{lemma}
\label{lem-prop-conv}
Let $M$ and $N$ be objects of $\Perv(G)$. There exist canonical
isomorphisms
\begin{gather}
  \Hom(\un, M^\vee*_* N)\simeq \Hom(M, N)\simeq \Hom(M*_!N^\vee,\un),
  \label{eqn1.2lemma}\\ 
  \DD(M*_*N)\simeq\DD(M)*_!\DD(N),\quad \DD(M*_!N)\simeq\DD(M)*_*\DD(N),
  \\
  \rmH^*_c(G_{\bar k}, M)\otimes_\Qlb \rmH^*_c(G_{\bar k},
  N)\simeq \rmH^*_c(G_{\bar k}, M*_!N), 
  \\
  \rmH^*(G_{\bar k}, M)\otimes_\Qlb \rmH^*(G_{\bar k},
  N)\simeq \rmH^*(G_{\bar k}, M*_*N).
\end{gather}
In the first isomorphisms, the Hom spaces are taken in the category~$\Der(G)$.
\end{lemma}

\begin{proof}
  All these are consequences of the formal properties of the six
  operations on~$\Der(G)$. More precisely, all can be found
  in~\cite[8.1.8,\,8.1.9]{esde}, except for the first statement. This is
  proved for tori in~\cite[p.\,533]{GL_faisc-perv}; however, the
  argument applies verbatim to any~$G$, since it only uses formal properties of
  the six operations on~$\Der(G)$.
\end{proof}

From the adjunctions in \eqref{eqn1.2lemma} of Lemma~\ref{lem-prop-conv}, we see that for
all $M\in \Perv(G)$, the identity morphism $\id_M \colon M \to M$ defines
evaluation and coevaluation morphisms\nomenclature[$e$]{$\ev$}{evaluation
  map}\nomenclature[$c$]{$\coev$}{coevaluation map}\index{evaluation
  map}\index{coevaluation map} 
\[
  \ev \colon M*_!M^\vee \longrightarrow \un\quad\quad\text{ and }\quad\quad \coev
  \colon \un \longrightarrow M^\vee*_*M.
\]
As a consequence of these properties, we note that $\Der(G)$  is
a symmetric monoidal $\Qlb$-linear category with respect to either the convolution $(A,B)\mapsto A*_!B$ or $(A,B)\mapsto A*_* B$,

\section{Character groups}\label{sec:character-groups-Lang-torsor}

In this section, we denote by~$k$ a finite field, by~$\bar{k}$ an
algebraic closure of~$k$, and by~$k_n$ the extension of degree~$n$
of~$k$ in~$\bar{k}$.  Let $\ell$ be a prime number distinct from the
characteristic of $k$.

Let $G$ be a connected commutative algebraic group defined over~$k$. For each $n\geq 1$, the \emph{norm map}\index{norm map} is the group homomorphism\nomenclature{$N_{k_n/k}$}{norm map} $N_{k_n/k}\colon G(k_n)\to G(k)$ defined as
\[
N_{k_n/k}(x)=\prod_{i=0}^{n-1} \Fr_{k_n}^{i}(x).
\]

For any~$n\geq 1$, let $\charg{G}(k_n)$ be the group of characters
$\chi\colon G(k_n)\to \Qlbt$.
\nomenclature[$G$]{$\charg{G}(k_n)$}{characters of $G(k_n)$}
We denote by $\charg{G}$ the
\emph{disjoint union}
\[
  \charg{G}=\bigsqcup_{n\geq 1} \charg{G}(k_n)
\]
(note that this is not a group; we also omit the dependency on~$\ell$ from
this notation).
\nomenclature[$G$]{$\charg{G}$}{disjoint union of $\charg{G}(k_n)$}

Given any set $S\subset \charg{G}$, we also define
$S(k_n)=S\cap \charg{G}(k_n)$, so that~$S$ is the disjoint union of
the subsets $S(k_n)$.

Since~$G$ is geometrically irreducible (see, \eg, \cite[Cor.\,1.35]{milne-groups}), the estimate
 $$
|\charg{G}(k_n)|=|G(k_n)|=|k|^{n\dim(G)}+O(|k|^{(n-1/2)\dim(G)})
$$
holds for~$n\geq 1$ by the Lang--Weil estimates. If $G$ is
an abelian variety, we have more precisely
$$
(|k|^{1/2}-1)^{2n\dim(G)}
\leq |\charg{G}(k_n)|
\leq(|k|^{1/2}+1)^{2n\dim(G)}
$$
and if $G$ is a torus, then
$$
(|k|-1)^{n\dim(G)} \leq |\charg{G}(k_n)| \leq(|k|+1)^{n\dim(G)}.
$$
These can be derived from the computation of the étale cohomology of abelian
varieties combined with the trace formula, or from Steinberg's formula
for tori; see, for instance,~\cite[Th.\,15.1, Th.\,19.1]{milne-ab} for the
case of abelian varieties and~\cite[Prop.\,3.3.5]{carter} for the case
of tori.

We now recall from \cite[Sommes trig.,\,1.4]{SGA4_1-2} the \emph{Lang
  torsor}\index{Lang torsor} construction and the basic properties of
the associated character sheaves. There is an exact sequence of
commutative algebraic groups\footnote{\ Note that it is here that the
  assumption that $G$ is connected plays a role, since in general the
  image of the morphism $\mathfrak{L}$ is equal
  to the connected component of the neutral element.}
\[
  1\longrightarrow G(k)\longrightarrow G
  \overset{\mathfrak{L}}{\longrightarrow} G\longrightarrow 1,
\]
where $\mathfrak{L}$ is the Lang isogeny
$x\mapsto \Frob_k(x)\cdot x^{-1}$. The Lang
isogeny is a Galois \'etale covering with Galois group $G(k)$, and hence induces a surjective map
$\pi_1^{\et}(G,e)\to G(k)$. Given a character~$\chi\in \charg{G}(k)$, we
denote by $\mcL_\chi$ the $\ell$-adic lisse sheaf of rank one on $G$
obtained by composing this map with $\chi^{-1}$ and we say
that~$\mcL_{\chi}$ is the \emph{character sheaf} on~$G$ associated
to~$\chi$.\index{character sheaf}
\nomenclature[$L$]{$\mcL_{\chi}$}{character sheaf}

For $x\in G(k)$, the geometric Frobenius automorphism\index{geometric Frobenius automorphism}
at~$x$ acts on the
stalk of $\mcL_\chi$ at~$x$ by multiplication by $\chi(x)$. In
particular, the lisse sheaf~$\mcL_{\chi}$ is pure of weight zero.

If $\chi$ is the trivial character, then $\mcL_\chi$ is the constant sheaf  $\Qlbt$.

The dual $\DD(\mcL_{\chi})$ of a character sheaf is isomorphic
to~$\mcL_{\chi^{-1}}$, and for any two characters~$\chi_1$ and~$\chi_2$ there are canonical isomorphisms
$$
\mcL_{\chi_1}\otimes\mcL_{\chi_2}\simeq \mcL_{\chi_1\chi_2}.
$$

If $n\geq 1$ and $\chi\in\charg{G}(k_n)$ is non-trivial, then for
all~$i\in \Zz$, the cohomology space $H^i_c(G_{\bar{k}},\mcL_{\chi})$
vanishes (see~\cite[Sommes trig.,\,Th.\,2.7*]{SGA4_1-2}).  More
generally, we have the following relative version.

\begin{lemma}\label{lm-image-character}
  Let $f\colon G\to H$ be a surjective morphism of commutative algebraic
  groups over~$k$. Let $\chi\in\charg{G}(k)$. The complex
  $Rf_!\mcL_{\chi}$ vanishes unless $\mcL_{\chi}|\ker(f)^{\circ}$ is
  the constant sheaf, i.e., unless $\chi$ is trivial on $\ker(f)^{\circ}$.
\end{lemma}

\begin{proof}
  Let $M=Rf_!\mcL_{\chi}$. Let $y\in H$ and let $z\in G$ be such that
  $f(z)=y$. By the proper base change theorem, the stalk of~$M$ at~$y$
  is given by
  $$
  M_y=H^*_c(f^{-1}(y)_{\bar{k}},\mcL_{\chi})=
  H^*_c((z\ker(f))_{\bar{k}},\mcL_{\chi})=H^*_c(\ker(f)_{\bar{k}},[x\mapsto
  xz]^*\mcL_{\chi}|\ker(f)).
  $$
  \par
  We write $\ker(f)$ as the disjoint union of cosets $u\ker(f)^{\circ}$
  where~$u$ runs over a set of representatives of the group of connected
  components of~$\ker(f)$. Thus,
  $$
  H^*_c(\ker(f)_{\bar{k}},[x\mapsto
  xz]^*\mcL_{\chi}|\ker(f))=\bigoplus_u
  H^*_c(\ker(f)^{\circ}_{\bar{k}},[x\mapsto
  xuz]^*\mcL_{\chi}|\ker(f)^{\circ}).
  $$
  \par
  Since $\mcL_{\chi}$ is a character sheaf, the sheaf
  $[x\mapsto xuz]^*\mcL_{\chi}$ is geometrically
  isomorphic to $\mcL_{\chi}$, so that we have an isomorphism
  $$
  M_y\simeq \bigoplus_u
  H^*_c(\ker(f)^{\circ}_{\bar{k}},\mcL_{\chi}|\ker(f)^{\circ}),
  $$
  and the result now follows from \cite[Sommes
  trig.,\,Th.\,2.7*]{SGA4_1-2} as recalled above. 
\end{proof}

Let~$n\geq 1$ and~$\chi\in\charg{G}(k)$. The base change of
$\mcL_\chi$ to $G_{k_n}$ is the character sheaf on~$G_{k_n}$
associated to the character~$\chi\circ N_{k_n/k}$ of~$G(k_n)$. In
particular, the trace function of~$\mcL_{\chi}$ on~$k_n$ is given by
$$
t_{\mcL_{\chi}}(x;k_n)=\chi(N_{k_n/k}(x)). 
$$

When there is no risk of confusion, we will still denote
by~$\mcL_{\chi}$ the pullback of the character sheaf associated
to~$\chi$ to~$\bar{k}$. The previous remark shows that~$\chi$ and
$\chi\circ N_{k_n/k}$ give rise to the same base change to~$\bar{k}$.

Let $f\colon G\to H$ be a homomorphism of commutative algebraic groups
defined over~$k$. For any integer~$n\geq 1$, let us denote by $f_n$ the induced
morphism $G(k_n)\to H(k_n)$; then we have dual homomorphisms
\index{dual homomorphism}
$\widehat{f}_n\colon \charg{H}(k_n)\to \charg{G}(k_n)$ defined by
$\chi\mapsto \chi\circ f_n$.
\nomenclature[$f$]{$\widehat{f}_n$}{dual homomorphism}%
\nomenclature[$f$]{$\widehat{f}$}{dual homomorphism}%
The combination of all these maps gives a
map $\widehat{f}\colon \charg{H}\to\charg{G}$, which we will often
denote simply by $\chi\mapsto \chi\circ f$. We will sometimes say that a
character $\chi\in\charg{G}$ \emph{arises} from~$H$ if~$\chi$ belongs to
the image of~$\widehat{f}$.

For~$\chi\in\charg{H}(k_n)$, there is a canonical isomorphism
$\mcL_{\widehat{f}(\chi)}\simeq f^*\mcL_{\chi}$.

For any object $M$ of $\Der(G)$ and any character~$\chi$ of~$G(k)$, we
denote by 
\[
M_{\chi}=M\otimes\mcL_{\chi}
\]
\nomenclature{$M_{\chi}$}{twist of~$M$ by~$\mcL_{\chi}$}
the ``twist'' of
$M$ by the character sheaf $\mcL_{\chi}$.

For all $\chi\in\charg{G}$, and all objects $M$ and $N$ of $\Der(G)$ (or
$\Der(G_{\bar{k}})$), there are canonical isomorphisms
\begin{gather}
  \DD(M_{\chi})\simeq \DD(M)_{\chi^{-1}}, \label{eqn:isomDualChar}
  \\
  \left(M_{\chi}\right)^\vee\simeq(M^\vee)_{\chi}, \label{eqn:isomTannakianDual}
  \\
  (M*_* N)_{\chi}\simeq(M_\chi*_* N_\chi),\quad (M*_!
  N)_{\chi}\simeq(M_\chi*_!  N_\chi).
\end{gather}

The first two properties follow from duality from
$\dual(\mcL_{\chi})=\mcL_{\chi^{-1}}$, and the third from the projection
formula combined with the canonical isomorphism
$m^*\mcL_{\chi}\simeq \mcL_{\chi}\boxtimes \mcL_{\chi}$, where
$p_1$ and $p_2$ are the projections $G\times G\to G$
(see~\cite[8.1.10\,(4)]{esde}).

More generally, for any
algebraic variety $X$ over $k$, any morphism $f\colon X\to G$, and any
object~$M$ of~$\Der(X)$, we set
$$
M_{\chi}=M\otimes f^*\mcL_{\chi},
$$
and we use the same notation for objects in~$\Der(G_{\bar{k}})$ and
$\Der(X_{\bar{k}})$, or in~$\Der(G_{k_n})$ and $\Der(X_{k_n})$.

We will extensively (and often without comment) use the following
standard lemma.

\begin{lemma}
  \label{lem-characters-t-exact}
  Let $f\colon X\to G$ be a morphism from an algebraic variety $X$ to a
  connected commutative algebraic group $G$, both defined over~$k$. Let
  $\chi\in \charg{G}$ be a character. Then the functor
  $M\mapsto M_{\chi}$ on $\Der(X)$ or~$\Der(X_{\bar{k}})$ is t-exact for the standard and perverse $t$-structures.
  In particular, if~$M$ is perverse
  \textup{(}\resp semiperverse\textup{)} then so is~$M_{\chi}$.
\end{lemma}

\begin{proof}
  Let $i\in\Zz$.  Since $\mcL_\chi$ is a lisse sheaf on~$G$, the
  pullback $f^*\mcL_{\chi}$ is lisse on~$X$, and hence tensoring with
  $f^*\mcL_\chi$ is exact for the standard t-structure on $\Der(X)$
  or~$\Der(X_{\bar{k}})$ (i.e., the t\nobreakdash-structure whose heart is the
  category of constructible sheaves concentrated in degree~$0$). There are thus canonical isomorphisms
  $\mcH^i(M\otimes f^*\mcL_\chi)\simeq \mcH^i(M)\otimes
  f^*\mcL_\chi$ for all $i$. Hence, looking at the support, we see that the functor
  $M\mapsto M_{\chi}$ is right t-exact for the perverse
  t-structure. It is also left t-exact since $\DD(M_{\chi})$ is
  isomorphic to $\DD(M)_{\chi^{-1}}$, hence the result.
\end{proof}

\section{Complexity estimates for character sheaves}

We keep the notation of the previous section. 
The first essential new ingredient for our work is the fact that the
complexity of character sheaves on~$G$ is uniformly bounded.

\begin{proposition}
  \label{prop-cond-char}
  Let~$G$ be a connected commutative algebraic group over~$k$ together with a
  quasi-projective immersion~$u$.  There exists a real number~$C\geq 0$
  such that, for every~$n\geq 1$ and for every character
  $\chi\in \charg{G}(k_n)$, the inequality $c_u(\mcL_{\chi})\leq C$
  holds.
\end{proposition}

\begin{proof}
  We will proceed in several steps, first noting that we may assume
  that~$n=1$.
  \par
  (1) If the result is true for the groups $G_1$ and $G_2$, then it is
  true for their product \hbox{$G=G_1\times G_2$}. Indeed, let
  $p_i \colon G \to G_i$ denote the two projections. Since any character
  $\chi$ of $G(k)$ takes the form~$(x_1,x_2)\mapsto \chi_1(x_1)\chi_2(x_2)$ for some characters $\chi_i$
  of $G_i(k)$, the corresponding character sheaf is the external product
  $\mcL_{\chi}=\mcL_{\chi_1}\boxtimes \mcL_{\chi_2}$, which has
  complexity bounded in terms of the complexity of~$\mcL_{\chi_1}$ and
  that of $\mathcal{L}_{\chi_2}$, and hence bounded uniformly by assumption.
  \par
  (More precisely, this is one case where we use
  Remark~\ref{rm-change-embedding}, since we most easily bound the
  complexity of $\mcL_{\chi_1}\boxtimes \mcL_{\chi_2}$ with
  respect to the composition~$v$ of the given quasi-projective
  immersions~$u_1$ and~$u_2$ of $G_1$ and~$G_2$ and the Segre embedding
  using Theorem~\ref{thm-conductors}, as
  in~\cite[Prop.\,6.12]{sawin_conductors}.) 
   \par
  (2) If the result holds for a group~$G$, then for any finite subgroup
  scheme $H$ (defined over~$k$), the results holds for the quotient
  $G/H$ (if this quotient is an algebraic group). To see this, we can
  further decompose $H=H_r\times H_l$ where~$H_r$ is reduced and~$H_l$
  is local, so that we may assume that~$H$ is either reduced or
  local. Let $v$ be a quasi-projective embedding of~$G/H$ and let
  $\pi\colon G\to G/H$ be the quotient morphism.

  If $H$ is reduced, then $\pi$ is a finite étale covering, so for any
  lisse sheaf $\mcL$ on~$G/H$, the sheaf $\mcL$ is a direct factor of
  $\pi_*\pi^*\mcL$, and we deduce
  $$
  c_v(\mcL)\leq c_v(\pi_*\pi^*\mcL)\ll c_u(\pi^*\mcL).
  $$
  This implies the result since $\pi^*\mcL$ is a character sheaf on~$G$
  if $\mcL$ is a character sheaf on~$G/H$.
  \par
  If~$H$ is local, then the quotient morphism~$\pi$ is finite and radicial, and hence the adjunction map $\mcL \to \pi_\ast \pi^\ast \mcL=\pi_! \pi^\ast \mcL$ is an isomorphism (see, \eg,
  \cite[Cor.\,5.3.10]{fu_etale_co}). By Theorem \ref{thm-conductors}\,\ref{thm-conductors:item6}, the complexity $c_v(\mcL)=c_v(\pi_!\pi^\ast \mcL)$ is hence $\ll c_u(\pi^\ast \mcL)$, and the result again follows. 
  \par
  (3) The result is valid for tori and unipotent groups. For the
  former, since complexity is a geometric invariant, we may assume
  that we have a split torus, and the result then follows from~(1) and
  the case of $G=\Gg_m$, which is established
  in~\cite[Prop.\,7.5]{sawin_conductors}.
  \par
  Assume then that~$G$ is a unipotent group. Let $G^\vee$ be its Serre
  dual (or more precisely, an algebraic group model of it; see
  Section~\ref{sec-unipotent} for details). There exists a lisse
  $\ell$-adic sheaf~$\mcL$ of rank one on~$G^{\vee}\times G$ such that the character sheaves associated to characters of $G(k)$ are in
  bijection with the points $a\in G^\vee(k)$ by mapping
  $a\in G^{\vee}(k)$ to the restriction of the sheaf $\mcL$ to
  $\{a\}\times G$. Hence, by Theorem~\ref{thm-conductors}, the
  complexity of any character sheaf of~$G$ is bounded in terms of the
  complexity of the single sheaf $\mcL$.
  \par
  (4) The result holds for abelian varieties
  by~\cite[Prop.\,7.9]{sawin_conductors}, since abelian varieties are
  projective and any character sheaf is lisse on~$G$.
  \par
  (5) The general case now follows using the previous results and the
  dévissage of Proposition~\ref{pr-devissage}.
This completes the proof of the proposition.   
\end{proof}

\begin{remark}
  A potential alternative (more conceptual) approach to this result
  would be the following. For a character sheaf~$\mcL$ on~$G$, there is
  an isomorphism
  $$
  m^*\mcL\simeq p_1^*\mcL\otimes p_2^*\mcL
  $$
  (recall that~$m$ is the multiplication map $G\times G\to G$).  If one
  could prove directly the estimate
  \begin{equation}\label{eq-better}
    c(\mcL)^2\ll c(p_1^*\mcL\otimes p_2^*\mcL),
  \end{equation}
  then we would deduce from Theorem~\ref{thm-conductors} that
  $$
  c(\mcL)^2\ll c(m^*\mcL)\ll c(\mcL),
  $$
  and hence $c(\mcL)\ll 1$. Note that Proposition~\ref{prop-cond-char}
  shows that~(\ref{eq-better}) is indeed true, and it is maybe not out
  of the question that one could provide a direct proof.
\end{remark}

\section{Arithmetic Fourier
  transforms}\label{sec:discreteFouriertransform}

We continue with the notation of the previous section. Given an
$\ell$-adic complex $M$ in $\Der(G)$, we can consider for any fixed
$n\geq 1$ the discrete Fourier transform of the trace function
$x\mapsto t_M(x;k_n)$ on $G(k_n)$, which we normalize to be the function
from $\charg{G}(k_n)$ to $\bQl$, or~$\Cc$, defined by
$$
\chi\mapsto S(M,\chi)=\sum_{x\in G(k_n)}\chi(x)t_M(x;k_n).
$$
\nomenclature{$S(M,\chi)$}{arithmetic Fourier transform}

This Fourier transform satisfies the usual formalism of commutative
harmonic analysis (see, e.g.,~\cite{bourbaki-ts-2}). For instance the
Fourier inversion formula\index{Fourier inversion formula}
\begin{equation}\label{eq-fourier-inv}
  t_M(x;k_n)=\frac{1}{|G(k_n)|}
  \sum_{\chi\in\charg{G}(k_n)}S(M,\chi)\bar{\chi}(x)
\end{equation}
holds for any $x\in G(k_n)$, and there is also a 
Plancherel formula\index{Plancherel formula}
$$
\sum_{x\in G(k_n)}|t_M(x;k_n)|^2=
\frac{1}{|G(k_n)|}\sum_{\chi\in\charg{G}(k_n)}|S(M,\chi)|^2. 
$$

Putting together the data of these discrete Fourier transforms on
$G(k_n)$ for all $n\geq 1$, we obtain what we call the \emph{arithmetic
  Fourier transform of the complex~$M$},
\index{arithmetic Fourier transform}
an element of the product set
$$
\mathcal{C}(\charg{G}, \bQl)=\prod_{n\geq 1}\mathcal{C}(\charg{G}(k_n), \bQl),
$$
where, for any set~$X$ and ring~$A$, we denote by $\mathcal{C}(X,A)$ the
$A$-module of functions $f\colon X \to A$.

Combining the Fourier inversion formula~(\ref{eq-fourier-inv}) with the
known injectivity theorem for trace functions (see
Proposition~\ref{pr-injectivity}), we deduce a corresponding injectivity
property of the discrete Fourier transform of complexes:

\begin{proposition}\label{prop:discreteFouriertransform} 
  Let $M_1$ and $M_2$ be complexes in $\Der(G)$ such that for all
  $n\geq 1$ and all characters $\chi\in\charg{G}(k_n)$, the
  equality
  $$
  \sum_{x\in G(k_n)}\chi(x)t_{M_1}(x;k_n)= \sum_{x\in
    G(k_n)}\chi(x)t_{M_2}(x;k_n)
  $$
  holds. Then the classes of $M_1$ and $M_2$ in the Grothendieck group
  $K(G)=K(G,\bQl)$ are equal.
\end{proposition}

\begin{remark}
  In Chapter~\ref{sec-stratification}, we will establish a more
  refined statement where the equality of discrete Fourier transforms is
  only assumed to hold for characters in a ``generic'' set, as described in the next section. 
\end{remark}

\section{Generic sets of characters}

For an arbitrary connected commutative algebraic group, there is no
obvious topology (or measure) on the set~$\charg{G}$ of characters which
would lead to a natural notion of sets containing ``almost all''
characters.  We will use instead the following definition of a generic
set of characters.

\begin{definition}
  \label{def-generic}
  Let~$k$ be a finite field and let~$G$ be a connected commutative
  algebraic group of dimension~$d$ over~$k$. Let~$S$ be a subset
  of~$\charg{G}$.

  Let~$i\geq 0$ be an integer. We say that~$S$ has \emph{character
    codimension at least~$i$},
  \index{character codimension}which we denote sometimes by
  $\ccodim(S)\geq i$, if the estimate 
  \begin{equation}\label{eq-c-codim}
    \abs{S(k_n)}\ll |k|^{n(d-i)}
  \end{equation}
  holds for all integers $n \geq 1$.
  \nomenclature[$c$]{$\ccodim(S)$}{character codimension}
  
  We say that~$S$ is \emph{generic}\index{generic set of characters}
  if $\charg{G}\setminus S$ has
  character codimension at least~$1$, i.e., if the estimate
    \begin{equation}\label{eq-generic}
    \abs{\charg{G}(k_n)\bks S(k_n)}\ll |k|^{n(d-1)} 
  \end{equation} holds for all integers $n \geq 1$.
\end{definition}

We now discuss the relation between the definition of generic sets
and other notions that appear in the literature, in the case of
unipotent and semiabelian varieties.

If~$G$ is unipotent, then the set of characters can be identified with
the $\bar k$-points of a $k$\nobreakdash-scheme~$G^{\vee}$; see again
Section~\ref{sec-unipotent}. If~$S\subset \charg{G}$ is algebraic (i.e., the disjoint union of the sets~$\widetilde{S}(k_n)$ for some
subvariety~$\widetilde{S}$ of~$G$), then the condition
$\ccodim(S)\geq i$ implies that the codimension of~$\widetilde{S}$
in~$G^{\vee}$ is at least~$i$. Conversely, if~$\widetilde{S}$ is a closed
subvariety of~$G^{\vee}$ over~$k$, then 
$\ccodim(\widetilde{S}(\bar{k}))\geq \codim_{G^{\vee}}(\widetilde{S})$.

Let $G$ be a semiabelian variety over~$k$. Let~$\ell$ be a prime
different from the characteristic of~$k$. The set of $\ell$-adic
characters of~$G$ can be naturally identified with the set of
$\Qlb$-points of a~$\Qlb$\nobreakdash-scheme, as we now recall. Let
$\pi^t_1(G_{\bar k})$ \nomenclature[$P$]{$\pi^t_1(G_{\bar k})$}{tame
  étale fundamental group} be the geometric tame étale fundamental
group of~$G$ (see, for instance, the paper~\cite{kerz-schmidt} of Kerz
and Schmidt for various equivalent definitions; note that it is
well-known that semiabelian varieties have good compactifications),
and let $\Pi(G,\Qlb)$ \nomenclature[$P$]{$\Pi(G,\Qlb)$}{continuous
  tame $\ell$-adic characters} be the group of continuous characters
$\chi \colon \pi_1^t(G_{\bar k})\to \Qlbt$.  For any~$n\geq 1$ and
$\chi\in \charg{G}(k_n)$, the character sheaf~$\mcL_\chi$ is tamely
ramified (indeed, only the case of tori requires proof; since the
question is geometric, we may assume that $G=\Gg_m^d$ for some integer
$d\geq 0$, and the result follows by induction from the
well\nobreakdash-known case of~$\Gg_m$ and the multiplicativity of the
tame fundamental group~\cite[Th.\,5.1]{orgogozo-tame}), and hence
corresponds to a point in $\Pi(G,\Qlb)$. For each $n\geq 1$, this
leads to a natural injective map
$$
\charg{G}(k_n)\injecte \Pi(G,\Qlb),
$$
and we will identify $\charg{G}(k_n)$ this way with a subset of
$\Pi(G,\Qlb)$.

There is a decomposition
\[
\Pi(G,\Qlb)= \Pi(G,\Qlb)_{\ell'}\times \Pi(G,\Qlb)_{\ell},
\] where
\nomenclature[$P$]{$\Pi(G,\Qlb)_{\ell'}$}{torsion characters of order prime
  to~$\ell$}
\nomenclature[$P$]{$\Pi(G,\Qlb)_{\ell}$}{characters factoring through the
  pro-$\ell$-quotient}
$\Pi(G,\Qlb)_{\ell'}$ is the group of torsion characters of order prime
to $\ell$ and $\Pi(G,\Qlb)_{\ell}$ is the group of characters that
factor through the maximal pro-$\ell$ quotient
$\pi_1^t(G_{\bar k})_{\ell}$ of $\pi_1^t(G_{\bar k})$. Since
$\pi_1^t(G_{\bar{k}})_{\ell}$ is a free $\Zz_{\ell}$-module of finite rank, by a
result of Brion and Szamuely~\cite{brion-szamuely-etale-covers}, we can
identify $\Pi(G,\Qlb)_{\ell}$ with the $\Qlb$-points of a scheme
$\Pi(G)_{\ell}$, following the arguments of Gabber and Loeser \hbox{\cite[Section 3.3]{GL_faisc-perv}}.\nomenclature[$P$]{$\Pi(G)_{\ell}$}{$\Qlb$-scheme whose $\Qlb$-points are $\Pi(G,\Qlb)_{\ell}$}

Letting $\Pi(G)$
\nomenclature[$P$]{$\Pi(G)$}{disjoint union of $\Pi(G)_{\ell}$}
be the disjoint union of the schemes $\Pi(G)_{\ell}$
indexed by $\chi\in \Pi(G,\Qlb)_{\ell'}$, we get
$$
\Pi(G,\Qlb)=\Pi(G)(\Qlb),
$$
and as above we will identify $\charg{G}$ with a subset of
$\Pi(G)(\Qlb)$.

Let~$G'$ be a semiabelian variety over~$k$ and~$f\colon G\to G'$ a
homomorphism. There is a dual morphism $\Pi(G')\to\Pi(G)$, denoted by $\chi\mapsto \chi\circ f$; if~$f$ is an inclusion, we
also write simply~$\chi\circ f=\chi_{|G}$.  The restriction of this
map to the subset $\charg{G}'$ is the map
$\widehat{f}\colon\charg{G}'\to \charg{G}$ previously defined.

\begin{definition} \label{def-most} Let $G$ be a semiabelian variety
  over a finite field~$k$, and let~$\ell$ be a prime different from
  the characteristic of~$k$.
  \begin{enumerate}
  \item A subset $S\subset \Pi(G)(\Qlb)$ is a \emph{translate of an
      algebraic cotorus} (abbreviated \emph{\tac})\index{tac} if there
    exists a surjective morphism $\pi\colon G_{\bar k}\to G'$ of semiabelian varieties over $\bar{k}$, with non-trivial \emph{connected}
    kernel, and a character $\chi_0\in \Pi(G)(\Qlb)$ such that
    \begin{displaymath}
      S=\{\chi_0\cdot (\chi'\circ \pi)\in \Pi(G)(\Qlb)\mid
      \chi'\in \Pi(G')(\Qlb)\}.
    \end{displaymath} We then say that~$S$ is \emph{defined by the quotient
      $G_{\bar{k}}\to G'$ and the character~$\chi_0$}, and that $S$ \emph{has dimension~$\dim(G'_{\bar{k}})$}. The kernel
    of~$\pi$ is also called the \emph{kernel of the~\tac}. If $G'$ and $\pi$ are defined over a finite extension $k'$ of $k$,
    then we say that $S$ is a \tac\ of~$G_{k'}$.
    
  \item We say that a subset $S\subset \Pi(G)(\Qlb)$ contains
    \emph{most} characters if the complement of $S$ is contained in a
    finite union of \tacs.
    \index{set containing most characters}

  \item We say that a subset $S\subset \Pi(G)(\Qlb)$ is \emph{weakly
      generic} if it is a generic set in the sense of the Zariski
    topology in $\Pi(G)$, \emph{i.e} it contains a dense open subset of $\Pi(G)$.\index{weakly generic set of characters}
  \end{enumerate}
  
  By extension, we shall say that a subset $S\subset \charg{G}$
  \emph{contains most characters}, or is \emph{weakly generic}, if its
  image in $\Pi(G)(\Qlb)$ satisfies this property.
\end{definition}

\begin{remark}
  \label{comp-generics}
  (1) The terminology ``most'' is used by Krämer and
  Weissauer~\cite{KW_vanishing_AV}; Esnault and Kerz~\cite{esnault-kerz}
  speak of ``quasi-linear'' subsets. What we call ``weakly generic'' is
  usually called ``generic'' (see, for example, the papers
  \cite{KW_vanishing_AV}, \cite{kramer_perverse_2014} and
  \cite{GL_faisc-perv}).
  
  (2) Let $S\subset \Pi(G)(\Qlb)$ be a subset that contains most
  characters. The Lang--Weil estimates imply that $S\cap \charg{G}$ is
  generic in the sense of (\ref{eq-generic}). Also, if $S\subset \charg{G}$ is a generic set and
  $\Pi(G)(\Qlb)\setminus S$ is not Zariski-dense,  then~$S$ is weakly
  generic.

  (3) The \tac\ defined by~$\pi$ and~$\chi_0$ can also be interpreted as
  the set of characters $\chi$ such that the restriction of~$\chi$ to
  $\ker(\pi)$ is equal to that of~$\chi_0$.

  (4) If a \tac\ $S$ of $G$ has dimension~$i$, then $S\cap \charg{G}$ has
  character codimension $\geq \dim(G)-i$ since
  $$
  |(S\cap \charg{G})(k_n)|\leq |G'(k_n)|\ll |k|^{ni}
  $$
  if~$S$ is defined by the quotient $G \to G'$ and the
  character~$\chi_0$.
\end{remark}

\begin{lemma}
\label{lem:inter:tac}
Let $G$ be a semiabelian variety over a finite field $k$. Let $\ell$ be
a prime different from the characteristic of~$k$. Let $I$ be a non-empty
finite set and let $(S_i)_{i\in I}$ be a family of \tacs\ in~$G$,
defined by quotient morphisms $\pi_i\colon G_{\bar k} \to G_{i,\bar k}$
and characters $\chi_i\in \Pi(G)(\Qlb)$.

Let~$K$ be the subgroup of $G_{\bar k}$ generated by the subgroups
$\ker(\pi_i)$.  The intersection $S=\bigcap S_i$ is non-empty if and
only if the restriction of $\chi_i$ to~$K$ is independent of~$i$.

If this is the case, then~$S$ is a \tac, which is defined by the
quotient morphism $\pi\colon G_{\bar k} \to G_{\bar k}/K$ and any of the
characters~$\chi_i$.
\end{lemma}

\begin{proof}
  We write $K_i=\ker(\pi_i)$ for $i\in I$. Since each $K_i$ is
  connected by definition, the subgroup~$K$ generated by the~$K_i$ is
  also connected.

  A character $\chi\in \Pi(G)(\Qlb)$ belongs to $S_i$ if and only if
  $\chi|_{K_i}=\chi_i|_{K_i}$. If $\chi \in S$, then the restriction
  of $\chi_i$ to~$K$ must coincide with the restriction of~$\chi$
  to~$K$, and is therefore independent of~$i$.

  Conversely, if this condition is satisfied, then pick any $i_0\in
  I$. The \tac\ defined by $G_{\bar{k}}\to G_{\bar{k}}/K$ and the
  character $\chi_{i_0}$ consists of characters $\chi$ such that
  $\chi|_K=\chi_{i_0}|_{K}$. This condition is equivalent to
  $\chi|_{K_i}=\chi_{i_0}|_{K_i}$ for all $i\in I$. Since
  $\chi_i|_{K_i}=\chi_{i_0}|_{K_i}$, this \tac\ is exactly
  the intersection of the~$S_i$.
  \end{proof}

\section{Fourier--Mellin transforms on semiabelian varieties}\label{sec:Fourier-Mellin}

Let~$k$ be a finite field and~$G$ a semiabelian variety
over~$k$. Let~$\ell$ be a prime different from the characteristic
of~$k$. We use the notation of the previous section.

We recall here some results of Gabber and Loeser for
tori~\cite{GL_faisc-perv}, generalized by
Krämer~\cite{kramer_perverse_2014} to semiabelian varieties.

Let $R$ be the ring of integers of a finite extension of $\Qq_\ell$ and
$\Omega_G=R\llb\pi_1^t(G_{\bar k})_{\ell}\rrb$. We have\nomenclature[$Omega$]{$\Omega_G$}{completed group algebra of ~$\pi_1^t(G_{\bar{k}})_{\ell}$}
\[
\Pi(G)_{\ell}=\Spec( \Qlb \otimes_R \Omega_G).
\]

Let $p\colon G_{\bar k}\to \Spec(\bar k)$ be the structural morphism.
We denote by $\can_G$ the tautological character\nomenclature[$c$]{$\can_G$}{tautological character}
$$
\can_G\colon \pi_1^t(G_{\bar{k}})_{\ell}\to \Omega_G^{\times},
$$
which defines a lisse $\Omega_G$-sheaf of rank one $\mcL_G$ on\nomenclature[$L$]{$\mcL_G$}{$\Omega_G$-sheaf associated to $\can_G$} $G_{\bar k}$.
Given an object $N$ of $\Der(G_{\bar k},R)$, one defines the \emph{Fourier--Mellin transforms of $N$},
\index{Fourier--Mellin transform on semi-abelian variety}
with and without
compact support, as the objects 
\begin{align*}
\FM_!(N)&=Rp_!(N\otimes_R \mcL_G) \\
\FM_*(N)&=Rp_*(N\otimes_R\mcL_G)
\end{align*} of the category $\Der(\bar
k,\Omega_G)=\Dcoh(\Omega_G)$. Inverting $\ell$ and passing to the direct limit over all 
$R\subset \Qlb$ and all $\chi\in \Pi(G,\Qlb)_{\ell'}$, we then get two
functors
$$
\FM_!, \FM_*\colon \Der(G_{\bar k})\to \Dcoh(\Pi(G)),
$$
where $\Dcoh(\Pi(G))$ is the derived category of the category of
coherent sheaves on~$\Pi(G)$.
\nomenclature[$F$]{$\FM_{#!}$}{Fourier-Mellin transform with compact
  support}
\nomenclature[$F$]{$\FM_{*}$}{Fourier-Mellin transform}
\par
By (the generalization of)~\cite[Cor.\,3.3.2]{GL_faisc-perv}, for an object 
$N$ of $\Der(G_{\bar k})$ and every $\chi\in \Pi(G)(\Qlb)$, viewed as a
closed immersion $i_\chi\colon \{\chi\}\to \Pi(G)$, there are canonical
isomorphisms
$$
Li_\chi^*\FM_!(N)\simeq Rp_!(N_\chi) \quad
\text{ and }
\quad
Li_\chi^*\FM_*(N)\simeq
Rp_*(N_\chi),
$$
where $Li_{\chi}$ indicates \emph{left-derived functors}.

\section{A geometric lemma}
A connected commutative algebraic group~$G$ is said to be \emph{almost
  simple}\index{almost simple group}
if it has no proper connected closed subgroup. Examples of
such groups are $\Gg_a$, $\Gg_m$ and simple abelian varieties.

We will use the following lemma in the proof of the general higher
vanishing theorem.
\begin{lemma}\label{lm-geometric}
  Let~$k$ be a field. Let~$s\geq 0$ be an integer. We denote
  $[s]=\{1,\ldots, s\}$.
  \nomenclature{$[s]$}{the set $\{1,\ldots, s\}$}
  Let
  $$
  G=\prod_{i=1}^s G_i
  $$
  be a product of almost simple connected commutative algebraic groups
  over~$k$. Let~$d=\dim(G)$.
  \par
  For any subset $I\subset [s]$, let
  $$
  G_I=\prod_{i\in I}G_i,
  $$
  which we identify with a subgroup of~$G$ in the obvious way. 
  
  Let $1\leq i \leq d$. Let $\mcE_i$ be the set of subsets~$I$ such
  that $\dim(G_I)>d-i$. For each $I\in \mcE_i$, let $H_I$ be a
  non-trivial connected subgroup of $G_I$. Then the algebraic
  subgroup~$H$ generated by all~$H_I$ has dimension at least~$i$.
\end{lemma}

\begin{proof}
  We denote~$d_i=\dim(G_i)$ for $1\leq i\leq s$.
  
  We work by induction on~$s$, and for each~$s$, by induction on
  $i$. The case $s=1$ is elementary, since~$\{1\}\in\mcE_i$ then,
  hence~$H=H_{[s]}=G$ in that case. For any $s$, the result is also
  elementary for $i=1$, since for $I=[s]\in \mcE_1$, we have
  $\dim(H)\geq \dim(H_I)\geq 1$.  Assume now that $2\leq i\leq g$ and
  that the result is known for $(s,i')$ for $i'<i$ as well as for
  $(s',i)$ for any $s'<s$.

  The subgroup $H_{[s]}\subset G$ is non-trivial, and hence there exists
  some integer $j\leq s$ such that the image of $H_{[s]}$ under the
  projection $G\to G_j$ is non-trivial; this means that this image
  must be equal to~$G_j$ since all~$G_i$ are almost simple. Up to
  reordering the factors, we may assume that the projection of
  $H_{[s]}$ on~$G_s$ is surjective.

  If $d_s\geq i$, then we are done since we then have
  $\dim(H_{[s]})\geq \dim(G_s)=d_s\geq i$. We therefore assume now
  that $d_s<i$.

  Let $G'=G_1\times \cdots \times G_{s-1}$ and $i'=i-d_s$. The
  dimension of~$G'$ is $d'=d-d_s$. We have $1\leq i'\leq d'$ and
  $d-i=d'-i'$. Each $J\subset [s-1]$ with $\dim(G'_J)>d'-i'=d-i$ is an
  element of $\mcE_i$. By induction, applied to the subgroups $H_J$
  for $J\in \mcE_{i'}$, the subgroup $H'$ of $G'$ generated
  by all $H_J$ has dimension $\geq i'=i-d_s$.

  To conclude, we observe that since $H'$ is a subgroup of $G'$ with
  dimension $\geq i-d_s$ and $H_{[s]}$ is a subgroup of $G=G'\times G_s$
  such that the projection of $H_{[s]}$ to $G_s$ is surjective, the
  subgroup~$H$ that they generate together satisfies
  \begin{align*}
    \dim(H)&=\dim(H')+\dim(H_{[s]})-\dim( H'\cap H_{[s]})
    \\
           &\geq
             \dim(H')+\dim(H_{[s]})-\dim( G'\cap H_{[s]})\geq  i-d_s+d_s=i
  \end{align*}
  since $\dim( G'\cap H_{[s]})+\dim(G_s)=\dim(H_{[s]})$.
\end{proof}

\section{Geometric and arithmetic
  semisimplicity}\label{sec-semisimple}

Let $k$ be a finite field, and $\bar{k}$ an algebraic closure
of~$k$. Let $\ell$ be a prime different from the characteristic of~$k$.

For an algebraic variety $X$ over~$k$ and a complex $M$ in
$\Der(X,\Qlb)$, we will sometimes refer to properties of $M$ (e.g., $M$
being a simple or semisimple perverse sheaf) as \emph{arithmetic}, and
to the analogue for the base change of~$M$ to $M_{\bar{k}}$ as being
\emph{geometric}. Thus we may speak of a geometrically simple perverse
sheaf, or an arithmetically semisimple perverse sheaf.

We collect here some facts about certain relations between such
properties.

\begin{lemma}\label{lm-semisimple}
  Let $X$ a geometrically irreducible algebraic variety over~$k$ and
  $\mcF$ a lisse $\ell$-adic sheaf on~$X$. If $\mcF$ is arithmetically
  semsimple, then it is geometrically semisimple.
\end{lemma}

\begin{proof}
  Using the correspondence between lisse sheaves and representations of
  the étale fundamental group, this follows, e.g.,
  from~\cite[Lem.\,5\,(a)]{serre-tenseurs}.
\end{proof}

\begin{lemma}
  \label{lem-arith-ss-geo}
  Let $(X,u)$ be a quasi-projective variety over~$k$. Let $M$ be an
  arithmetically simple perverse sheaf on~$X$. There exists a finite
  extension of $k$ of degree bounded in terms of $c_u(M)$ such that
  the base change of $M$ to $X_{k'}$ is a direct sum of geometrically
  simple perverse sheaves on~$k'$.
  \par
  In particular, $M$ is geometrically semisimple.
\end{lemma}

\begin{proof}
  By~\cite[Prop.\,5.3.9\,(ii)]{BBD-pervers}, there exists an
  integer~$n\geq 1$ and a geometrically simple perverse sheaf $N$ on
  $X_{k_n}$ such that $M=f_{n*} N$, where $f_n\colon X_{k_n}\to X$ is
  the base change morphism. Since $N$ is non-zero, we deduce that
  $n\ll c_u(M)$ by looking at the rank at a generic point of the
  support. The base change of~$M$ to~$k_n$ is then a direct sum of
  geometrically simple perverse sheaves. 
\end{proof}

\begin{lemma}
  \label{lm-arith-geo-iso}
  Let $k$ be a finite field and $\bar{k}$ an algebraic closure
  of~$k$. Let $\ell$ be a prime different from the characteristic
  of~$k$. Let $X$ be a smooth and geometrically connected
  quasi-projective variety over~$k$. Two 
  perverse sheaves~$M$ et~$N$ on~$X$ which are geometrically simple are
  geometrically isomorphic if and only if there exists $\alpha\in\Qlb$
  such that $M\simeq \alpha^{\deg}\otimes N$.
\end{lemma}

This is a standard fact (see,
e.g.,~\cite[Lemme\,4.4.4]{loeser-determinant}).





\section{A result from representation theory}

The following basic fact from the representation theory of reductive
groups will play a crucial role.

\begin{proposition}\label{pr-tensor-ab}
  Let~$F$ be a field of characteristic zero and let~$G$ be a reductive
  algebraic group over~$F$. Let~$V$ be a finite-dimensional faithful
  representation\index{faithful representation}
  of~$G$ over~$F$. Any finite-dimensional irreducible
  representation of~$G$ over~$F$ occurs in a tensor power
  $(V\oplus V^{\vee})^{\otimes m}$ for some integer~$m\geq 0$,
  where~$V^{\vee}$ is the contragredient of~$V$.
\end{proposition}

See, for instance,~\cite[Prop.\,3.1]{deligne-milneI} for the proof.

\chapter{Generic vanishing theorems}\label{sec-vanishing}

Throughout this chapter, $k$ denotes a finite field, $\bar{k}$ an algebraic closure of $k$, and~$k_n$ the extension of degree~$n$ of~$k$ inside~$\bar{k}$ for each $n \geq 1$. We also fix once for all a
prime number $\ell$ different from the characteristic of~$k$. All complexes of sheaves and characters are tacitly understood to be $\ell$-adic complexes and characters for this choice of $\ell$. 

\section{Statement of the vanishing theorems}

We now state our main vanishing theorems.
 
\begin{theorem}[Generic vanishing]
  \label{thm-gen-vanish}
  Let $G$ be a connected commutative algebraic group
  over~$k$ and let~$M$ be a perverse sheaf on~$G$. The set $\mcU$ of characters
  $\chi\in \charg{G}$ satisfying
  \begin{gather*}
    \rmH^i(G_{\bar k},M_\chi)=\rmH^i_c(G_{\bar k},M_\chi)=0 \quad\text{
      for all $i\not=0$,}
    \\
    \rmH^0_c(G_{\bar k},M_\chi)= \rmH^0(G_{\bar k},M_\chi)
  \end{gather*}
  is generic in the sense of Definition~\textup{\ref{def-generic}}.
\end{theorem}

This gives the first part of Theorem~\ref{thm:vanishing-thm-intro} from
the introduction.

\begin{remark}\label{rem:versionsvanishing}
  Versions of Theorem~\ref{thm-gen-vanish} have been proved by the
  following authors:
  \begin{enumerate}
  \item Katz--Laumon~\cite[Th.\,2.1.3,
    Scholie\,2.3.1]{KL-fourier-exp-som} in the case of powers of the
    additive group and Saibi~\cite[Th.\,3.1]{saibi_FD_unipotent} in the
    case of general unipotent groups; in both cases, the generic set is a
    Zariski\nobreakdash-dense open subset of the $k$-scheme parameterizing
    characters.
  \item Gabber--Loeser~\cite[Cor.\,2.3.2]{GL_faisc-perv} for tori, with
    ``generic'' replaced by a condition implying ``weakly-generic'' in
    the sense of Definition~\ref{def-most}; see also
    \cite[Th.\,7.2.1]{GL_faisc-perv}.
  \item Weissauer~\cite[Vanishing\,Th.,\,p.\,561]{weissauer_vanishing_2016} for abelian varieties, with
    ``generic'' replaced by ``most'' characters, and
    Krämer~\cite[Th.\,2.1]{kramer_perverse_2014} for semiabelian
    varieties, for ``weakly generic'' characters.
  \end{enumerate}
\end{remark}

For our main applications, the second part of
Theorem~\ref{thm:vanishing-thm-intro} is more important. It is
provided by the following result, which also controls the
``stratification'' arising from the non-vanishing of other cohomology
groups.

\begin{theorem}[Stratified vanishing]
  \label{thm-high-vanish}
  Let $G$ be a connected commutative algebraic group of dimension $d$
  over~$k$, and~$M$ a perverse sheaf on $G$. There exist subsets
  \[
    \mcS_d\subset\dots\subset \mcS_0=\charg{G}
  \]
  such
  that the following holds: 
  \begin{enumth}
  \item For~$0\leq i\leq d$, the subset $\mcS_i$ has character
    codimension at least~$i$.
  \item For~$0\leq i\leq d$, any~$\chi\in\charg{G}$ such that at least
    one of the cohomology groups
    \begin{equation}\label{eq-groups}
      \rmH^{i}(G_{\bar k},M_\chi),\quad\quad
      \rmH^{-i}(G_{\bar k},M_\chi),\quad\quad
      \rmH^{i}_c(G_{\bar k},M_\chi),\quad\quad
      \rmH^{-i}_c(G_{\bar k},M_\chi)
    \end{equation}
    is non-zero belongs to~$\mcS_i$.
  \item For $\chi\in \mcS_0\bks \mcS_1$, the equality
    $\rmH^0_c(G_{\bar k},M_\chi)=\rmH^0(G_{\bar k},M_\chi)$
    holds.\footnote{\ Recall that by convention, this means that the
      ``forget support'' map is an isomorphism.}
  \item If~$G$ is a torus or an abelian variety, then~$\mcS_i$ is a
    finite union of \tacs\ of~$G$ of dimen\-sion~$\leq d-i$.
  \item If~$G$ is a unipotent group, then $\mcS_i$ is the set of closed
    points of a closed subvariety of dimension $\leq d-i$ of the Serre
    dual $G^\vee$.
  \end{enumth}
\end{theorem}

Concretely, this implies that for~$0\leq i\leq d$, the estimate
\begin{align*}
  |\{\chi\in \charg{G}(k_n)\mid\, &\rmH_c^{i}(G_{\bar k},M_\chi)\neq
  0\text{ or } \rmH_c^{-i}(G_{\bar k},M_\chi)\neq 0
  \\
  &\text{ or }\rmH^{i}(G_{\bar k},M_\chi)\neq 0\text{ or }
  \rmH^{-i}(G_{\bar k},M_\chi)\neq 0 \}|\ll \abs{k_n}^{d-i}
\end{align*}
holds for all $n\geq 1$, and so this implies the second part of
Theorem~\ref{thm:vanishing-thm-intro}.

Note that Theorem~\ref{thm-gen-vanish} is a consequence of
Theorem~\ref{thm-high-vanish}, since the set of characters satisfying
the property of Theorem~\ref{thm-gen-vanish} contains the generic set
$\mcS_0\bks \mcS_1$.

\begin{remark}
  \label{rmk-HV-effective}
  We expect that this result should be true with the stronger
  information that the implied constants in~(\ref{eq-c-codim}) for the
  subsets~$\mcS_i$ depend only on the complexity of~$M$.  A result of
  this type would be especially useful for applications to
  ``horizontal'' equidistribution theorems.

  However, we can only prove this at the current time in the following
  cases:
  \begin{enumerate}
  \item if $G$ is a unipotent group (use the equality of Fourier
    transforms of~\cite[Th.\,3.1]{saibi_FD_unipotent} combined with
    Theorem~\ref{thm-lisse-locus-conductor});
  \item if $G$ is a geometrically simple abelian variety (see Corollary
    \ref{cor-gen-van-AV-4}).
  \end{enumerate}
  
  The issues that arise in attempting to handle the general case are:
  \begin{itemize}
  \item For tori, the use of de Jong's theorem on alterations, where we
    do not control the number of exceptional components that appear
    (thus, a suitably effective version of de Jong's theorem, or an
    effective form of embedded resolution of singularities, would
    probably imply the desired conclusion in this case).
  \item For abelian varieties, the need to find and control the
    complexity of an alteration that ``moderates'' certain perverse
    sheaves, to apply results of Orgogozo.
  \end{itemize}
  
\end{remark}

\begin{corollary}
\label{cor-gen-van-der}
Let~$G$ be a connected commutative algebraic group over~$k$ and let $M$ be
an object of~$\Der(G)$. Then for generic $\chi\in \charg{G}$ and any
$i\in \Zz$, there are canonical isomorphisms
\[
  \rmH^i_c(G_{\bar k},M_{\chi})\simeq \rmH^i(G_{\bar
    k},M_{\chi})\simeq \rmH^0_c(G_{\bar k},\pH^i(M_{\chi})).
\]
\end{corollary}

\begin{proof}
  The proof is similar to that
  of~\cite[Cor.\,2.3]{kramer_perverse_2014}; see also
  Corollary~\ref{cor-gen-van-AV-3} below.
\end{proof}

We will prove Theorems~\ref{thm-gen-vanish} and~\ref{thm-high-vanish} in
Section~\ref{ssec-proof-vanish}. Before doing this, we need to establish
some preliminaries concerning perverse sheaves on the basic building
blocks of Proposition~\ref{pr-devissage}, namely (in rough order of
difficulty) unipotent groups, tori and abelian varieties.

Note that proving either Theorem~\ref{thm:vanishing-thm-intro} or
Theorem~\ref{thm:equidis-thm-intro} for a given group~$G$ only
involves the corresponding material for groups of the types that
actually appear in Proposition~\ref{pr-devissage} applied to~$G$. In
particular, for instance, the proof of Theorem~\ref{th-variance} (and
other similar statements) only depends on the case of tori, i.e., on
Section~\ref{ssec-tori}.

To facilitate orientation, we list below the key statements about each
type of groups; Section~\ref{ssec-proof-vanish} only requires these
statements from the next three sections.

\begin{enumerate}
\item Unipotent groups: Proposition~\ref{prop-strat-unip}.
\item Tori: Corollary~\ref{cor-split-torus}.
\item Abelian varieties: Corollary~\ref{cor-av-hv-main} and the
  auxiliary Theorem~\ref{thm-strat-alt-orgo}, due to
  Orgogozo~\cite{Orgo_constr_mod}.
\end{enumerate}

To conclude this section, we also point out that we expect that in
Theorems~\ref{thm-gen-vanish} and~\ref{thm-high-vanish},
one can find a suitable set~$\mcU$ of ``geometric'' nature. For
example, if $G=A\times U \times T$, where $A$ is an abelian variety,
$U$ a unipotent group and $T$ a torus, we would expect that there
exists a finite extension $k'$ of $k$, a finite union of tacs $\mcS'$
of $A$ defined over $k'$, a finite union of tacs $\mcS$ of $T$ defined
over $k'$ and a strict closed subvariety $Z$ of $\what{U}$ such that
$\what{G}\setminus \mcU\subseteq \mcS' \times Z \times \mcS$. In this
direction, following a suggestion of one referee, we can prove the
following:

\begin{theorem}\label{th-vanishing-geometric}
  Let $G$ be a connected commutative algebraic group over~$k$ such
  that $G=A\times U$ \respup{$G=A\times T$} where~$A$ is an
  abelian variety, $U$ is a unipotent group and~$T$ is a torus.  Let
  let~$M$ be a perverse sheaf on~$G$.

  There exists a finite extension $k'$ of $k$, a finite union of tacs
  $\mcS'$ of $A$ defined over $k'$ and a strict closed subvariety $Z$
  of $\what{U}$ \respup{a finite union of tacs $\mcS$ of $T$ defined
    over $k'$} such that
  for~$\chi\in \what{G}\setminus( \mcS' \times Z)$ \respup{for
    $\chi\in \what{G}\setminus (\mcS'\times \mcS)$}, we have
  \begin{gather*}
    \rmH^i(G_{\bar k},M_\chi)=\rmH^i_c(G_{\bar k},M_\chi)=0 \quad\text{
      for all $i\not=0$,}
    \\
    \rmH^0_c(G_{\bar k},M_\chi)
    \text{ is isomorphic to } 
    \rmH^0(G_{\bar k},M_\chi).
  \end{gather*}
\end{theorem}

\begin{remark}
  We do not claim that the isomorphism between
  $ \rmH^0_c(G_{\bar k},M_\chi)$ and $\rmH^0(G_{\bar k},M_\chi)$ in
  this statement is the ``forget support'' map, since this does not
  follow from the proof. We expect that this should be true (maybe up
  to enlarging the set of exceptional characters). 
\end{remark}

\section{The case of unipotent groups}
\label{sec-unipotent}

We begin by summarizing the duality theory of commutative unipotent
groups; a good account can also be found
in~\cite[App.\,F]{boyarchenko-drinfeld}.

Let $U$ be a connected unipotent commutative algebraic group over a
finite field $k$ of characteristic~$p$. 
The functor that sends a perfect $k$-scheme $S$ (\ie a scheme for which
the absolute Frobenius is an automorphism) to the extension group
\[
  \mathrm{Ext}^1(U \times_k S, \Qq_p/\Zz_p)=\varinjlim_m\,\mathrm{Ext}^1(U \times_k S, p^{-m}\Zz_p/\Zz_p)
\]
in the category of commutative group schemes over $S$ (with
$\Qq_p/\Zz_p$ viewed as a constant group scheme) is representable by a
connected commutative group scheme $U^*$ over $k$, called the
\emph{Serre dual}\index{Serre dual}\nomenclature{$U^*$}{Serre dual, as
  group scheme} of $U$. This goes back to a remark by Serre
\cite[p.\,55]{Ser60} and was subsequently developed by Bégueri
in~\cite[Prop.\,1.2.1]{Beg} and Saibi~\cite{saibi_FD_unipotent}.
Morever, if $m_0\geq 1$ is such that $p^{m_0}U=0$, then the natural
map yields an isomorphism 
\[
\mathrm{Ext}^1(U \times_k S, p^{-m_0}\Zz_p/\Zz_p) \stackrel{\sim}{\longrightarrow}
\varinjlim_m\,\mathrm{Ext}^1(U \times_k S, p^{-m}\Zz_p/\Zz_p). 
\] 

Let $F$ be a finite abelian group. For each integer $n \geq 1$, the
short exact sequence
\[
  1\longrightarrow U(k_n)\longrightarrow U_{k_n} \xrightarrow{x
    \longmapsto \Frob_{k_n}(x)\cdot x^{-1}} U_{k_n} \longrightarrow 1
\]
induces an isomorphism
\[
\Hom(U(k_n), F) \stackrel{\sim}{\longrightarrow} \mathrm{Ext}^1(U_{k_n}, F)
\]
(see~\cite[Prop.\,F.2]{boyarchenko-drinfeld}).

Let $m\geq 1$ be such that $p^mU=0$. We take
$F=p^{-m}\Zz_p/\Zz_p\simeq \Zz/p^m\Zz$. For any integer $n\geq 1$, we
obtain an isomorphism
\[
\Hom(U(k_n), F) \stackrel{\sim}{\longrightarrow} U^{*}(k_n). 
\]

Fix now a faithful character
$\psi\colon p^{-m}\Zz_p/\Zz_p\to \bQl^{\times}$. We then obtain, for each integer $n\geq 1$, an isomorphism
\[
U^{*}(k_n) \stackrel{\sim}{\longrightarrow} \charg{U}(k_n).
\]

Saibi~\cite[Lemma\,1.5.4.1]{saibi_FD_unipotent} (see
also~\cite[Remark\, F.1\,(ii)]{boyarchenko-drinfeld}) proved that
there exists a connected commutative unipotent algebraic group
$U^{\vee}$\nomenclature{$U^{\vee}$}{algebraic group model of the Serre
  dual} and a biextension $\mcL_{U,U^{\vee}}$
\nomenclature[$L$]{$\mcL_{U,U^{\vee}}$}{bi-extension}
of~$U^{\vee}\times U$ by $\Qq_p/\Zz_p$ such that $\mcL_{U,U^{\vee}}$ induces an isomorphism between the perfectizations of~$U^{\vee}$
and~$U^*$. Together with the above character $\psi$, this induces
isomorphisms
\begin{equation}\label{eqn:isomorphismbeta}
\beta_n\colon U^{\vee}(k_n)\to \charg{U}(k_n)
\end{equation}
for all $n\geq 1$. (See
also~\cite[Remark\,F.4\,(ii)]{boyarchenko-drinfeld} for a different
approach to the construction of the model of finite type $U^{\vee}$.) We
also write $\psi_x$ for the character~$\beta_n(x)$.
\nomenclature[$psi$]{$\psi_x$}{character associated to $x\in U^{\vee}$}

We denote by $\mcL_{U,U^{\vee},\psi}$
\nomenclature[$L$]{$\mcL_{U,U^{\vee},\psi}$}{lisse sheaf on
  $U\times U^{\vee}$} the lisse $\ell$-adic sheaf of rank~$1$
on~$U^{\vee}\times U$ associated to~$\mcL_{U,U^{\vee}}$; its trace
functions are given by
\[
t_n(x,y;k_n)=\beta_n(x)(y)
\]
for all~$n\geq 1$ and $(x,y) \in U(k_n)\times U^{\vee}(k_n)$.

\begin{example}
  Fix a non-trivial additive character $\psi\colon k\to \Qlb$. Suppose
  that $U=\Gg_a^d$ for some~$d\geq 0$. For $(x,y)\in U\times U$, write
  \[
  x\cdot y=\sum_{i=1}^d x_iy_i. 
  \]

  There exists a choice of biextension with $U^{\vee}=U$, and the
  isomorphisms
  \[
  \beta_n\colon (\Gg_a^d)(k_n) \longrightarrow \charg{\Gg}_a^d(k_n)
  \]
  are given by $x\mapsto \psi_x$, where
  \[
  \psi_x(y)=\psi(\Tr_{k_n/k}(x\cdot y)).
  \]
\end{example}

We now also fix a quasi-projective scheme $S$ over $k$, and shall work relative to $S$. We denote by~$\mcL_{U,U^{\vee},\psi,S}$ the pull-back of $\mcL_{U,U^{\vee},\psi}$
to $U \times_{k} U^\vee \times_{k} S$.

Let $p \colon U \times_{k} U^\vee \times_{k} S \to U\times_{k} S$,
$p^\vee \colon U \times_{k} U^\vee \times_{k} S \to U^\vee \times_{k}
S$ and $q \colon U \times_{k} S \to S$ denote the
projections. The~\emph{Fourier transform} \index{Fourier transform} is
the equivalence of categories\nomenclature[$F$]{$\ft_{\psi,S}$}{Fourier transform on a unipotent
  group} 
\[
\ft_{\psi,S} \colon \Der(U\times_{k} S)\longrightarrow \Der(U^\vee\times_{k} S)
\]
defined by
\[
  \ft_{\psi,S}(M)=Rp^\vee_!(p^*(M)\otimes \mcL_{U,U^{\vee},\psi,S})=
  Rp^\vee_*(p^*(M)\otimes \mcL_{U,U^{\vee},\psi,S}),
\]
where the second equality (more precisely, the fact that the natural
transformation ``forget supports'' from the left-hand side to the
right-hand side is an isomorphism)
is~\cite[Th.\,3.1]{saibi_FD_unipotent}. A corollary of this is that the
Fourier transform is compatible with Verdier duality, in that there is a
canonical functorial isomorphism
\[
  \DD(\ft_{\psi,S}(M)) \simeq \ft_{\psi^{-1},S}(\DD(M))(\dim U)
\]
for each object $M$ of $\Der(U\times_{k} S)$;
see~\cite[Cor.\,3.2.1]{saibi_FD_unipotent}. We refer the reader to
Saibi's article~\cite{saibi_FD_unipotent} for the other main
properties of the $\ell$\nobreakdash-adic Fourier transform on
unipotent groups, such as the formula for the inverse Fourier
transform.

By the proper base change theorem and the definition of Fourier
transform using $Rp^{\vee}_!$, for all~$a\in U^\vee(\bar k)$ and
$i\in\Zz$, there are natural isomorphisms
\begin{equation}\label{eq-stalk-fourier}
  R^iq_!(M_{\psi_a})=\mcH^i(\ft_{\psi}(M))_a. 
\end{equation}
Since unipotent groups are affine, it follows from Artin's vanishing
theorem that the Fourier transform shifts the perverse degree by the
dimension of $U$. In particular, if $M$ is perverse, then so~is
\[
\ft_{\psi,S}(M)[\dim(U)].
\]

%
%

\begin{proposition}
\label{prop-strat-unip}
Let $U$ be a connected unipotent commutative algebraic group of
dimension $d$ over~$k$. Fix a locally-closed immersion~$u$
\respup{$u^\vee$} of $U$ \respup{$U^\vee$} into some projective space
to compute the complexity.  Let $M$ be an object of~$\Der(U)$ of
perverse amplitude $[a,b]$.
\par
There exists an integer~$C\geq 0$, depending only on $c_u(M)$, and a
stratification $(S_i)$ of $U^\vee$ such that every strat $S_i$ is
either empty or has dimension $d-i$, with the following properties:
\par
\begin{enumth}
\item The sum of the degrees of the irreducible components of $u(S_i)$
  is at most $C$.
\item\label{prop-strat-unip:item2} For each $\xi\in S_i(\bar k)$, the
  vanishing $ \rmH^j_c(U_{\bar k}, M_{\psi_\xi})=0$ holds for all
  $j\notin [a,b+i]$.
\end{enumth}

In particular, the estimate
\begin{equation}\label{eq-unipotent-bound}
  \abs{S_i(k_n)}\ll \abs{k_n}^{d-i}
\end{equation}
holds for all~$n\geq 1$, with an implicit constant that only depends on
$c_u(M)$.

Moreover, for any $\xi\in S_0(\bar k)$ and any $j\in \Zz$, we have
\[
\rmH^j_c(U_{\bar k}, M_{\psi_\xi})=\rmH^j(U_{\bar k}, M_{\psi_\xi}).
\]
\end{proposition}




\begin{proof} Since the Fourier transform shifts the perverse degree by
  $d$, the complex $\ft_{\psi}(M)$ has perverse amplitude $[a+d,b+d]$.
  By Theorem \ref{thm-conductors}, the complexity $c_u(\ft_{\psi}(M))$
  is bounded in terms of~$c_u(M)$.

  By Theorems~\ref{thm-lisse-locus-conductor}
  and~\ref{thm-gen-bc-conductors}, there exists a smooth
  open subscheme $S_0\subset U^\vee$, with closed complement $Y_0$ of
  degree bounded in terms of $c_u(\ft_{\psi}(M))$, and hence in terms of
  $c_u(M)$, such that the restriction of $\ft_{\psi}(M)$ to $S_0$ has
  lisse cohomology sheaves and such that $\ft_{\psi}(M)$ is of formation
  compatible with any base change $S'\to S_0\subset U^\vee$ (this
  follows from the formula for the Fourier transform in terms of
  $Rp^{\vee}_*$). Up to replacing $S_0$ by a smaller open subset, we may
  assume that $S_0$ is affine (and this does not increase the complexity
  of the complement).

  In particular, using~(\ref{eq-stalk-fourier}) and this compatibility,
  we obtain the following equality for $a\in S_0(\bar k)$:
  \begin{equation}\label{eq-apply-basechange}
    \rmH^i_c(U_{\bar k}, M_{\psi_a})=\mcH^i(\ft_{\psi}(M))_a
    =\rmH^i(U_{\bar k}, M_{\psi_a}).
  \end{equation}
  
  By a slight generalization
  of~\cite[Cor.\,4.1.10.\,\emph{ii}]{BBD-pervers}, the pullback by a
  closed immersion with affine complement of a complex of perverse
  amplitude $[a, b]$ has perverse amplitude $[a-1, b]$.  Therefore, the
  restriction of $\ft_{\psi}(M)$ to $Y_0$ has perverse amplitude
  $[a+d-1,b+d]$. Proceeding by induction, we construct a stratification
  $(S_i)_{0\leq i\leq d}$ of $U^\vee$ into strats $S_i$ such that
  \begin{enumerate}
  \item each $S_i$ is smooth, empty or equidimensional of
    dimension $d-i$;
  \item the closure of each $S_i$ has degree bounded in terms of
    $c_u(M)$;
  \item the restriction of $\ft_{\psi}(M)$ to each
    $S_i$ has lisse cohomology sheaves and is of perverse amplitude
    $[a+d-i,b+d]$.
  \end{enumerate}
  
  Let $0\leq i\leq d$. On each connected component of $S_i$, the support
  of the cohomology sheaves of~$\ft_{\psi}(M)$ is either empty or equal
  to $S_i$ (since these sheaves are lisse). However, the definition of
  perversity implies the inequality
  \[
    \dim \supp\mcH^j(\ft_{\psi}(M)_{\vert S_i})\leq -j+b+d
  \] for all integers $j$. Since $S_i$ has dimension $d-i$, the non-vanishing of $\mcH^j(\ft_{\psi}(M)_{\vert S_i})$ implies therefore the inequality
  \[
    d-i\leq-j+b+d,\quad\text{ i.e. } j\leq b+i.
  \]

  Since $S_i$ is smooth of dimension $d-i$ (so the dualizing complex
  on~$S_i$ is $\bQl(d-i)[d-i]$ and the Verdier dual of a lisse sheaf
  is the naive dual, up to shift) and the cohomology sheaves on~$S_i$
  are lisse, duality implies that $\dual(\ft_{\psi}(M)_{\vert S_i})$
  also has lisse cohomology
  sheaves, given by the formula
  \[
  \mcH^j(\dual(\ft_{\psi}(M))_{\vert
    S_i})=(\mcH^{-j-2d+2i}(\ft_{\psi}(M)_{\vert S_i}))^\vee(d-i)
  \] for all $j$. Thus, arguing as above, the perversity condition shows that  $\mcH^j(\ft_{\psi}(M)_{\vert S_i})\neq 0$ implies
  $$
  d-i \leq j+2d-2i-a-d+i, \quad\text{ i.e. } j\geq a.
  $$

  We conclude that the cohomology sheaves of the complex
  $\ft_{\psi}(M)_{\vert S_i}$ are concentrated in degrees~$[a,b+i]$. By~\eqref{eq-apply-basechange}, this implies
  assertion~\ref{prop-strat-unip:item2} of the proposition and
  concludes the proof.
\end{proof}

\begin{remark}
  This result is a generalization to all unipotent groups, and a
  quantification by means of the complexity, of some of the
  Fouvry--Katz--Laumon stratification results for additive exponential
  sums~\cite{KL-fourier-exp-som,fouvry-katz}. It may have interesting
  applications to analytic number theory, since the quantitative form
  means that it may be used over varying finite fields, \eg $\Ff_p$ as
  $p\to +\infty$ as in Chapter~\ref{sec-stratification} (although a
  referee pointed out that the complexity of the Fourier transform of a
  perverse sheaf on a non-additive unipotent group, such as Witt vectors
  of length~$2$, will usually have to depend on~$p$).
\end{remark}

For the proof of Theorem~\ref{th-vanishing-geometric}, we will also
require a relative version of the ``first step'' of
Proposition~\ref{prop-strat-unip}.

\begin{proposition}
\label{prop-strat-unip-rel}
Let $U$ be a connected unipotent commutative algebraic group of
dimension~$d$ over~$k$, and $S$ a quasi-projective scheme over
$k$. Fix a locally-closed immersion~$u$ \respup{$u^\vee$} of
$U\times S$ \respup{$U^\vee\times S$} into some projective space
to compute the complexity. Let $q\colon U\times S\to S$ be the
projection.
\par
There exists an integer~$C\geq 0$, depending only on $c_u(M)$, and a
dense open subset $U^\vee_0$ of $U^\vee$ with the following
properties:
\begin{enumth}
\item The sum of the degrees of the irreducible components of
  $u^\vee(U^\vee \setminus U^\vee_0)$ is at most $C$.
\item For any $\xi\in U^\vee_0(\bar k)$, we have
  \[
    Rq_!(M_{\psi_\xi})=Rq_*(M_{\psi_\xi}), 
  \] where $\psi_\xi$ denotes the character corresponding to $\xi$ by means of \eqref{eqn:isomorphismbeta}.\end{enumth}

In particular, the estimate
\begin{equation}
  \abs{(U^\vee \setminus U^\vee_0)(k_n)}\ll \abs{k_n}^{d-1}
\end{equation}
holds for all~$n\geq 1$, with an implicit constant that only depends
on $c_u(M)$.

Moreover, if $M$ is perverse, then
$Rq_!(M_{\psi_\xi})=Rq_*(M_{\psi_\xi})$ is perverse for every
$\xi\in U^\vee_0(\bar k)$.
\end{proposition}

\begin{proof}
  By Theorem~\ref{thm-gen-bc-conductors}, there exists a smooth open
  subscheme $U_0^\vee\subset U^\vee$, with closed complement $Y_0$ of
  degree bounded in terms of $c_u(\ft_{\psi,S}(M))$, and hence in
  terms of $c_u(M)$, such that~$\ft_{\psi,S}(M)$ is of formation
  compatible with any base change $S'\to U_0^\vee\subset U^\vee$ (this
  follows from the formula for the Fourier transform in terms of
  $Rp^{\vee}_*$). Combining this with proper base change, we get for every $\xi\in U^\vee_0(\bar k)$, the equality
  \[
    Rq_!(M_{\psi_\xi})=Rq_*(M_{\psi_\xi}).
  \]
  If $M$ is perverse, then Artin's vanishing theorem implies that
  $Rq_!(M_{\psi_\xi})$ is perverse since $q$ is affine.

  Finally, the estimates on the degrees of the irreducible components of $U^\vee \setminus U^\vee_0$ follow from
  Theorem~\ref{thm-gen-bc-conductors}\,(i) and the counting estimate
  follows from Theorem~\ref{th-rh}.
\end{proof}

\section{Perverse sheaves on tori}
\label{ssec-tori}

In this section, we generalize some of the results of Gabber and Loeser
\cite{GL_faisc-perv} about perverse sheaves on tori. We begin with a generalization of \cite[Th.\,4.1.1']{GL_faisc-perv}, which is proved in \loccit under the assumption that resolution of singularities
and simplification of~ideals hold for varieties of dimension at most the dimension of the torus in question. The structure of our proof is the same, but we are able to
replace the appeal to resolution of singularities with de Jong's theorem
on alterations~\cite{dejong-alt}.

\begin{theorem}
\label{thm-tori-vanishing}
Let $T$ be a torus over~$\bar{k}$ and let~$M$ be an object
of~$\Der(T)$. For all charac\-ters~$\chi\in \Pi(T)(\Qlb)$ outside of a
finite union of \tacs, the equality 
\[
H^i(T,M_{\chi})=H^i_c(T,M_{\chi})
\]
holds for all~$i\in\Zz$.
\end{theorem}

As in \cite{GL_faisc-perv}, the proof of
Theorem~\ref{thm-tori-vanishing} relies on the auxiliary Proposition
\ref{prop-tori-vanishing} stated below. We pick a smooth compactification of~$T$ by a simple
normal crossing divisor $j\colon T\to \bar{T}$ (for example, the
projective space), and denote by $i \colon \bar T\setminus T\to \bar T$
the complementary closed immersion. Given any morphism
$\varphi\colon W\to \bar T$ of varieties over~$k$, denote by
$j_W\colon \varphi^{-1}(T)\to W$ and
$i_W\colon \varphi^{-1}(\bar T \setminus T) \to W$ the corresponding
open and closed immersions. Recall the $\Omega_T$-sheaf of rank one
$\mcL_T$ on $T$ from Section \ref{sec:Fourier-Mellin}. In particular, for an object $N$ of $\Der(\varphi^{-1}(T))$ and a point $\xi$ of $ \varphi^{-1}(\bar T \setminus T)$, the stalk $j_{W*}(N\otimes \varphi^*(\mcL_T)))_\xi$ is a complex of coherent sheaves on $\Pi(T)$.

\begin{proposition}
  \label{prop-tori-vanishing}
  With notation as above, let~$N$ be an object of
  $\Der(\varphi^{-1}(T))$. There exists a finite union~$\mathcal{S}$ of
  \tacs\ in~$\charg{T}$ such that, for any $r \geq 0$ and any
  $\xi\in \varphi^{-1}(\bar T \setminus T)$, the support of the module
  $(R^rj_{W*}(N\otimes \varphi^*(\mcL_T)))_\xi$ is contained
  in~$\mathcal{S}$.
\end{proposition}

\begin{proof}
  The idea of the proof is to reduce to the situation
  of~\cite[Prop.~4.3.1']{GL_faisc-perv}.

  We use induction on the dimension of~$W$. We can then readily assume
  that $N$ is a lisse sheaf on a locally-closed irreducible subvariety
  $U$ of $\varphi^{-1}(T)$, extended by zero to $\varphi^{-1}(T)$. We
  can assume further that $U$ is dense in $W$. Now the monodromy of $N$
  can be assumed to be pro-$\ell$. Indeed, consider the finite \'etale
  cover $f\colon U'\to U$ associated to the $\ell$-Sylow subgroup of the
  monodromy group of $N$, and let $W'$ be the normalization of $W$ in
  the function field of $U$. The sheaf~$N$ is a direct factor of
  $f_*f^*N$, and it suffices to prove the theorem for $f^* N$ and
  $W'$. Hence, we assume that the monodromy of $N$ is pro-$\ell$.

  By de Jong's theorem~\cite[Th.\,4.1]{dejong-alt}, there exists an alteration\index{alteration}
  $f\colon W'\to W$ such that $W'$ is smooth and the reduction of the
  complement of $f^{-1}(U)$ in $W'$ is a strict normal crossing
  divisor. Since we are working over a perfect field, we can further
  assume that the alteration $f$ is generically \'etale. Hence, there
  exists a dense open subset $U_0$ of $U$ such that~$f$ is finite
  \'etale over $f^{-1}(U_0)$. By induction, it is enough to prove the
  result for $U_0$ and $N_{\vert U_0}$, and hence by the same argument
  as above, it is enough to prove it for $f_*f^*N_{\vert U_0}$. By
  proper base change, it is then enough to prove the result for $W'$ and
  $f^*N_{\vert U_0}$. By a last d\'evissage, it is finally enough to
  prove it for $f^*N$.

  We are now in a situation where we can suppose that $W$ is smooth,
  that the complements of~$\varphi^{-1}(T)$ and $U$ in $W$ are strict
  normal crossing divisors, and that the monodromy of~$N$ is
  pro-$\ell$. This is exactly the situation at the end of the proof
  of~\cite[Prop.~4.3.1', starting from p.\,544, line~-4]{GL_faisc-perv}
  (with~$N$ replacing~$A$ there) and the remaining argument is identical
  to that of \loccit
  \end{proof}

\begin{proof}[Proof of Theorem \ref{thm-tori-vanishing}]
  The fact that Proposition \ref{prop-tori-vanishing} implies
  Theorem~\ref{thm-tori-vanishing} is completely similar to the fact
  that Proposition~4.3.1' implies Théorème~4.1.1' in
  \cite{GL_faisc-perv}. We keep the notation introduced before the
  statement of Proposition \ref{prop-tori-vanishing}, and apply
  Proposition~\ref{prop-tori-vanishing} with $W=\bar T$, with~$\varphi$
  the identity morphism and~$N=M$, so that $j_W=j$ and~$i_W=i$.

  Let $\chi \in \charg{T}$ such that~$\chi$ does not belong to the
  finite number of \tacs\ of~$\charg{T}$ given by
  Proposition~\ref{prop-tori-vanishing}. According
  to~\cite[Prop.\,4.5.1\,(2)]{GL_faisc-perv}, this implies that the
  object $i^*Rj_*(M_\chi)\in \Der(\bar T\setminus T)$ is trivial, and
  hence its cohomology complex
  $$
  R\Gamma(\bar T\setminus T,i^*Rj_*(M_\chi))
  $$
  is also trivial. But this last complex is isomorphic to the cone of
  the morphism
  $$
  Rs_!(M_\chi)\to Rs_*(M_\chi),
  $$
  where~$s\colon T_{\bar{k}}\to\Spec(\bar{k})$ is the structure
  morphism, hence the theorem.
\end{proof}

We now use Proposition~\ref{prop-tori-vanishing} to deduce a relative
version of Theorem \ref{thm-tori-vanishing}.

\begin{theorem}
  \label{thm-tori-vanishing-rel}
  Let~$T$ be a torus over~$\bar{k}$, let $S$ be an arbitrary scheme
  over~$\bar{k}$, and let $G=S\times T$. Denote by $p\colon G\to S$ the
  projection.  Let $N$ be an object of $\Der(G)$.
  \par
  For $\chi\in \Pi(T)(\Qlb)$ away from a finite union of \tacs\
  $\mathcal{S}$, we have $Rp_!(N_{\chi})=Rp_*(N_{\chi})$.
  \par
  In particular, if $N$ is a perverse sheaf, then for~$\chi$ not
  in~$\mathcal{S}$, the complex $Rp_!(N_\chi)=Rp_*(N_\chi)$ is a
  perverse sheaf on~$S$.
\end{theorem}

\begin{proof}
  This is similar to Theorem \ref{thm-tori-vanishing}. We apply
  Proposition~\ref{prop-tori-vanishing} with
  \hbox{$W=S\times_{\bar k} \bar T$,} and check that, for each character
  $\chi$ away from the finite union of~\tacs\ given by the proposition,
  the object $i_W^*Rj_{W*}(N\otimes \mcL_\chi)$ is trivial, which
  follows from the immediate extension
  of~\cite[Prop.~4.7.2\,(ii)]{GL_faisc-perv} to an arbitrary base
  scheme~$S$ (instead of just tori).
\end{proof}

\begin{theorem}
\label{thm-high-van-torus}
Let $T$ be a $d$-dimensional torus over~$k$, let $S$ be an arbitrary
scheme over~$k$, and define $X=T\times S$.  Let~$i$ be an integer such
that~$1\leq i\leq d$.

Let $M$ be a perverse sheaf on~$X$.  There exist a finite
extension~$k'$ of~$k$ and a family $(S_f)_{f\in\mathcal{F}}$ of \tacs\
of~$T_{k'}$ of dimension~$\leq d-i$ with the property that for
any~$\chi\in\charg{T}_{k'}$ that does not belong to the union of
the~$S_f$ there exists a quotient torus $q\colon T_{k'}\to Z$ of
dimension~$i-1$ such that
$$
Rq_{S!}M_{\chi}=Rq_{S*}M_{\chi},
$$
with~$q_S$ denoting the projection
$q_S\colon T_{k'}\times S\to Z_{k'}\times S$, and this complex is
perverse on~$Z\times_{k'} S_{k'}$.
\end{theorem}

\begin{proof}
  Up to replacing $k$ by a finite extension, we can assume that $T_k$ is
  split, and thus reduce to $T=\Gm^d$. Now let~$1\leq i\leq d$. For each subset~$I$
  of~$[d]=\{1,\ldots, d\}$ of size~$i-1$, we apply Theorem~\ref{thm-tori-vanishing-rel} with~$(T,S)=(\Gg_m^{[d]\setminus I},\Gg_m^I\times S)$ over~$\bar{k}$,
  so that the projection~$p$ in the theorem is then the canonical
  projection
  \[
  q_I\colon \Gg_m^d\times S=\Gg_m^{[d]}\times S\to \Gg_m^{I}\times S.
  \]

  We obtain a finite union of \tacs\ of $\Gm^{[d]\setminus I}$ such that
  for characters $\chi$ of $\Gm^{[d]\setminus I}$ outside of this finite
  union, we have
  \[
  Rq_{I!}(M_\chi)=Rq_{I*}(M_\chi)
  \]
  and this complex is perverse.

  Let
  \[
  (\pi_{I,j}\colon \Gg_m^{[d]\setminus I}\to Y_{I,j},\chi_{I,j})_{j\in
    X_I}
  \]
  be the quotient morphisms and characters defining this finite family
  of \tacs. For~$j\in X_I$, we define~$K_{I,j}=\ker(\pi_{I,j})$; this is
  a non-trivial subtorus of~$\Gg_m^{[d]\setminus I}$, which we identify
  with a subtorus of~$\Gg_m^d$ using the canonical embedding
  $\Gg_m^{[d]\setminus I}\to \Gg_m^d$. In addition, we
  define~$\chi'_{I,j}\in\Pi(\Gg_m^d)(\Qlb)$ to be the character that is
  trivial on~$\Gg_m^I$ and coincides with~$\chi_{I,j}$
  on~$\Gg_m^{[d]\setminus I}$.
  \par
  Let~$\mathcal{F}$ be the set of all maps $f$ from the subsets of~$[d]$
  of size $i-1$ to the disjoint union of the~$X_I$ that send a subset
  $I$ to an element $j\in X_I$ for each~$I$; this set is finite. For
  $f\in\mathcal{F}$, let~$S_f$ be the intersection of the \tacs\
  of~$\Gg_m^d$ defined by
  $$
  (\Gg_m^d\to \Gg_m^d/K_{I,f(I)},\chi'_{I,f(I)}).
  $$
  
  We claim that the family $(S_f)_{f\in\mathcal{F}}$ (to be precise, the subfamily where $S_f$ is not empty) satisfies the assertions of the theorem.

  Indeed, first of all Lemma~\ref{lem:inter:tac} shows that~$S_f$ is
  either empty or is again a \tac; moreover, in the second case, it is
  defined by the projection $\Gg^m_d\to \Gg_m^d/T_f$ where $T_f$ is the
  subtorus of $\Gg_m^d$ generated by the $K_{I,f(I)}$ (as subtori of
  $\Gg_m^d$). By Lemma~\ref{lm-geometric} applied to~$G_i=\Gg_m$ for
  all~$i$ and the subgroups~$K_{I,f(I)}$, we have $\dim(T_f)\geq i$ for
  all such~$f$, and hence the quotient
  $$
  p_f\colon \Gg_m^d\to Y_f=\Gg_m^d/T_f
  $$
  has image of dimension $\leq d-i$, as desired.

  Finally, let $\chi\in\charg{\Gg}_m^d$ be a character that does not
  belong to any of the \tacs\ $S_f$. This implies that there exist some~$f\in\mathcal{F}$, some subset $I\subset [d]$ of size $i-1$ and some~$j\in X_I$ such that the restriction~$\chi_I$ of $\chi$ to~$\Gg_m^{[d]\setminus I}$ is not equal to~$\chi_{I,j}$.
  
  We can write~$\chi=\chi_I\chi'$ where $\chi'$ is a character
  of~$\Gg_m^I$. Considering the quotient
  $q\colon \Gg_m^d\to \Gg_m^I$, the base change $q_S$ is the canonical
  projection $q_I$ and from the application of
  Theorem~\ref{thm-tori-vanishing-rel} to $q_I$, we obtain
  \[
  R_{q_S*}(M_{\chi})=R_{q_S*}(M_{\chi_I})\otimes \mcL_{\chi'}=
  R_{q_{S!}}(M_{\chi_I})\otimes \mcL_{\chi'}=Rq_{S!}(M_{\chi}),
  \]
  and the fact that this object is perverse.
\end{proof}

We deduce two corollaries that are sometimes more convenient for
applications. The first one is Theorem~\ref{thm-high-vanish} for tori. 

\begin{corollary}
\label{cor-high-van-torus}
Let $T$ be a torus of dimension $d$ over~$k$ and let $M$ be a perverse sheaf on $T$. For
$-d\leq i\leq d$, the sets
\[
\{\chi\in\charg{T}\,\mid\,
\rmH^i(T_{\bar{k}},M_{\chi})\not=0\}\quad\text{and}\quad
\{\chi\in\charg{T}\,\mid\, \rmH^i_c(T_{\bar{k}},M_{\chi})\not=0\}
\]
are contained in a finite union of \tacs\ of~$T$ of
dimension~$\leq d-|i|$, and in particular they have character
codimension at least~$|i|$.
\end{corollary}

\begin{proof}
  We apply Theorem~\ref{thm-high-van-torus} to~$|i|$ and claim that the
  characters in either of these sets belong to the union of the \tacs\
  $S_f$ that arise. Indeed, if~$\chi$ is not in any~$S_f$, then there
  exists a quotient torus $T_{k'}\to Z$ of dimension~$i-1$ such that
  $ Rq_{S!}M_{\chi}=Rq_{S*}M_{\chi}$, and hence
  $$
  \rmH^i(T_{\bar{k}},M_{\chi})= \rmH^i(B_{\bar{k}},Rq_{S*}M_{\chi})=0
  $$
  since~$Rq_{S*}M_{\chi}$ is a perverse sheaf and~$\dim(B)=i-1$. The
  argument is similar for the cohomology with compact support.
\end{proof}

\begin{remark}
  We recall that, concretely, this corollary implies that for
  $|i|\leq d$, the estimate
  \[
    \abs{\set{\chi\in\charg{T}(k_n)\mid \rmH^{i}(T_{\bar{k}},M_\chi)\neq 0
        \text{ or } \rmH^{i}_c(T_{\bar{k}},M_\chi)\neq 0 }}\ll
    \abs{k_n}^{d-|i|}
  \]
  holds for all $n\geq 1$.
\end{remark}

The following ``stratified'' statement is also a useful formulation of
the result.

\begin{corollary}\label{cor-split-torus}
  Let~$T$ be a torus of dimension~$d$ over~$k$ and~$S$ a scheme
  over~$k$. Set~$X=T\times S$ and let~$q$ denote the projection
  $q\colon X\to S$. Let~$M$ be a perverse sheaf on~$X$. There exists a
  finite extension $k'$ of $k$ and a partition of $\charg{T}_{k'}$
  into subsets $(S_i)_{0\leq i\leq d}$ of character codimension
  $\geq i$ such that, for any~$i$ and~$\chi\in S_i$, the object
  $Rq_!(M_{\chi})$ of $\Der(S)$ has perverse amplitude $[0,i]$.
\end{corollary}

\begin{proof}
  Using the notation of the proof of Theorem~\ref{thm-high-van-torus}, for any integer~$i$
  with $1\leq i\leq d$, let~$k'_i$ be the finite extension arising from
  its application to~$i$ and let~$\mathcal{F}_i$ be the corresponding
  family of \tacs. Define~$\widetilde{S}_i$ to be the union of the~$S_f$
  for~$f\in\mathcal{F}_i$ for~$1\leq i\leq d$.
  \par
  Let~$k'$ be the compositum of all~$k'_i$.  Define
  $S_0=\charg{T}\setminus \widetilde{S}_1$ and
  $S_{i}=\widetilde{S}_i\setminus \widetilde{S}_{i+1}$
  for~$1\leq i\leq d$. These sets form a partition of~$\charg{T}_{k'}$,
  and since~$S_i\subset \widetilde{S}_i$ for~$i\geq 1$, they have
  character codimension $\geq i$. This property is also clear for~$i=0$.
  \par
  Let $0\leq i\leq d$, and let~$\chi\in S_i$.
  Then~$\chi\notin \widetilde{S}_{i+1}$, and hence the theorem
  provides a projection $q_S\colon \Gg_m^d\times S\to Z\times S$
  with~$\dim(Z)=i$ such that $Rq_{S!}M_{\chi}$ is perverse. Composing
  with the projection $r\colon Z\times S\to S$, which is affine and
  hence such that $Rr_!$ preserves objects with perverse amplitude
  $[0,+\infty]$ (by Artin's vanishing theorem), it follows that
  $Rq_!M_{\chi}$ has perverse amplitude $[0,i]$.
\end{proof}

\section{Perverse sheaves on abelian varieties}

In this section, we will review and extend some results of Kr\"amer and
Weissauer on perverse sheaves on abelian varieties.

\subsection{Statement of the results and corollaries}

Let~$k$ be a finite field, and~$\bar{k}$ an algebraic closure of~$k$. Let $X$ be an abelian variety over~$k$. We fix a projective embedding $u$ of $X$. For subvarieties of~$X$, the degree means the degree of the image by~$u$; for a \tac\ of~$S$ defined by~$\pi\colon X\to A$ and~$\chi$, we will say that the \emph{degree} of~$S$ is the degree of the image~$u(\ker(\pi))$.

For a perverse sheaf~$M$ on~$X$, a combination of the main result of Weissauer~\cite{weissauer_vanishing_2016} and of the machinery developped by Krämer and Weissauer~\cite{KW_vanishing_AV} implies that
for most characters $\chi\in\charg{X}$, we have
$ \rmH^i(X_{\bar k},M_\chi)=0$ for all~$i\not=0$; we will show here that
this result can be made quantitative using the complexity of~$M$, and
will then establish a relative version (see
Section~\ref{ssec-relative}).

\begin{theorem}
\label{thm-gen-van-AV}
Let $X$ be an abelian variety over~$k$ and let $M$ be a perverse sheaf
on~$X$.

There exist an integer~$c\geq 0$ depending only on~$c_u(M)$, a finite
extension~$k'$ of~$k$ of degree~$\leq c$, and a finite family $(S_f)_{f\in F}$ of \tacs\ of~$X_{k'}$ with $|F|\leq c$,
each of degree at most~$c$, such that any~$\chi\in\charg{X}_{k'}$ that does not belong to the union of the~$S_f$ satisfies
\[
\rmH^i(X_{\bar k},M_\chi)=0 
\]
for all~$i\not=0$.
\end{theorem}

We will prove this below, but first we establish some corollaries.

\begin{corollary}
\label{cor-gen-van-AV-2}
Let $M$ be an object of $\Der(X)$. 

There exist an integer~$c\geq 0$, depending only on~$c_u(M)$, a finite
extension~$k'$ of~$k$ of degree $\leq c$, and a finite family
$(S_f)_{f\in F}$ of \tacs\ of~$X_{k'}$, each of degree at most~$c$, with
$|F|\leq c$, such that for any~$\chi\in\charg{X}_{k'}$ that does not
belong to the union of the~$S_f$, there is a canonical isomorphism
$$
\rmH^i(X_{\bar k},M_\chi)\simeq \rmH^0(X_{\bar k},\pH^i(M)_\chi)
$$
for all~$i\in\Zz$.
\end{corollary}

\begin{proof}
  This is the same argument as in the proof of
  Corollary~\ref{cor-gen-van-der}; the dependency on $c_u(M)$ is
  obtained by means of Proposition~\ref{prop-JorHol-cond} to control the
  perverse cohomology sheaves of~$M$.
\end{proof}

Alternatively, the next corollary may be more convenient for
applications.

\begin{corollary}
\label{cor-gen-van-AV-3}
Let $M$ be an object of $\Der(X_k)$. The set~$\mathcal{S}$ of
characters $\chi\in \charg{X}$ such that
there are isomorphisms
\[
\rmH^i(X_{\bar k},M_\chi)\simeq \rmH^0(X_{\bar k},\pH^i(M)_\chi)
\]
for all~$i\in\Zz$ is generic, 
and the implicit constant in~\emph{(\ref{eq-generic})} depends only on
$c_u(M)$.

In particular, if~$M$ is a perverse sheaf, then the set
of~$\chi$ such that $H^i(X_{\bar{k}},M_{\chi})=0$ for
all~$i\not=0$ is generic and the implicit constant
in~\emph{(\ref{eq-generic})} only depends on $c_u(M)$.
\end{corollary}

\begin{proof}
  Assume first that~$M$ is a perverse sheaf. We apply
  Theorem~\ref{thm-gen-van-AV} to $M$, and use the notation there.
  For~$n\geq 1$, let~$k'_n=k'k_n$. For any
  $\chi\in \charg{X}(k_n)\setminus \mathcal{S}(k_n)$, the corresponding
  character in~$\charg{X}(k'_n)$ belongs to $S_f(k'_n)$ for
  some~$f\in F$. Let~$A_f$ be the abelian variety such that $S_f$ is
  defined by~$\pi_f\colon X_{k'}\to A_f$; we have
  $$
  |\charg{X}(k_n)\setminus \mathcal{S}(k_n)| \leq
  \sum_{f\in F}|A_f(k'_n)|\leq |F|\
  (|k'k_n|^{1/2}+1)^{2\dim(A_f)} \ll |k_n|^{\dim(X)-1},
  $$
  where the implied constant depends only on~$c_u(M)$ by the theorem.
  \par
  Now in the general case, recalling that $\pH^i(M_{\chi})$ is
  canonically isomorphic to $\pH^i(M)_{\chi}$ for all~$i$ and
  all~$\chi$, we have the convergent perverse spectral sequences
  $$
  E_2^{i,j}=H^i(X_{\bar{k}}, \pH^j(M)_{\chi})\Rightarrow
  H^{i+j}(X_{\bar{k}},M_{\chi}).
  $$
  By the previous case applied to each of the finitely many perverse
  cohomology sheaves, the set of~$\chi$ such that
  $H^i(X_{\bar{k}}, \pH^j(M)_{\chi})=0$ for all $i\not=0$ and all~$j$ is
  generic; for any such character, the spectral sequence degenerates and
  we obtain isomorphisms
  $$
  H^i(X_{\bar{k}},M_{\chi})\simeq H^0(X_{\bar{k}},\pH^i(M)_{\chi}).
  $$
  \par
  Applying Proposition~\ref{prop-JorHol-cond}, we see that the last
  statement concerning the implicit constant in~(\ref{eq-generic})
  holds.
\end{proof}

\begin{corollary}
\label{cor-gen-van-AV-4}
Let $X$ be a geometrically simple abelian variety over~$k$. Let $M$ be a
perverse sheaf on~$X$.  Then there exists a constant~$c$ depending only
on~$c_u(M)$ and a finite set~$\mcS\subset \Pi(X)(\Qlb)$ of cardinality
at most~$c$ such that for~$\chi\in \Pi(X)(\Qlb)\setminus \mcS$,
\[
  \rmH^i(X_{\bar k},M_\chi)=0 \text{ for } i\neq 0.
\]
\end{corollary}

\begin{proof}
  Since $X$ is a geometrically simple abelian variety, then a tac of $X$
  contains a single character. Hence, the result follows from
  Theorem~\ref{thm-gen-van-AV}.
\end{proof}

\subsection{Proof of the results}

We now proceed with the proof of Theorem~\ref{thm-gen-van-AV}.  As we
indicated, the first ingredient is a quantitative version of a result of
Weissauer~\cite{weissauer_vanishing_2016}.

\begin{proposition}
  \label{prop-chi=0-trans-inv}
  Let~$X$ be an abelian variety over~$k$ with a projective
  embedding~$u$, and let $M$ be a geometrically simple perverse sheaf
  on~$X$ such that $\chi(X_{\bar k}, M)=0$.
  \par
  There exists a \tac\ $S$ on~$X$ with kernel an abelian
  subvariety $A$ of degree bounded in terms of $c_u(M)$, such that
  $$
  \bigoplus_i\rmH^i(X_{\bar k}, M_\chi)\not=0
  $$
  if and only if~$\chi$ is in~$S$.
  
  Moreover, $M$ is invariant by translation by $A$.
\end{proposition}

\begin{proof}
  For any perverse sheaf~$N$ on~$X$, we denote by $\mathcal{S}_1(N)$ the set of
  characters~$\chi$ such that for some $i\neq 0$ the cohomology group
  $\rmH^i(X_{\bar{k}},N)$ is non-zero.
  
  By~\cite[Th.\,3 and Lem.\,6]{weissauer_vanishing_2016}, there exists a
  maximal abelian variety $A_{\bar{k}}$ of~$X_{\bar{k}}$ such that~$M$
  is invariant by translation by~$A_{\bar{k}}$, and this abelian variety
  is non-trivial.
  \par
  Denoting by~$q\colon X_{\bar{k}}\to X_{\bar{k}}/A_{\bar{k}}$ the
  quotient morphism, this is equivalent to the fact that~$M$ is
  isomorphic over~$\bar{k}$ to a perverse sheaf of the form
  $\mcL_{\chi_0}\otimes q^*(\widetilde{M})[\dim(A)]$ for some
  character
  $\chi_0\colon \pi_1(X_{\bar k})\to \overline{\Qq_{\ell}}^{\times}$
  and some simple perverse sheaf $\widetilde{M}$ on
  $X_{\bar k}/A_{\bar{k}}$. 
  \par
  We claim first that $A_{\bar{k}}$ is defined over $k$ and that the degree of~$A_{\bar{k}}$ in the image of~$u$ is
  bounded in terms of~$c_u(M)$.
  
  The fact that $A_{\bar{k}}$ is defined over $k$ is implicit in the
  proof of the existence of~$A$ by Weissauer. We recall his
  argument. First, a perverse sheaf~$\mcP_M$ is defined as follows
  (\cite[p.\,563]{weissauer_vanishing_2016}): the evaluation morphism
  $\ev\colon M^\vee * M\to \un$ (see Section~\ref{sec-convolution})
  induces morphisms of perverse sheaves $\pH^i(M^{\vee}*M)[-i]\to \un$
  for all~$i$, and each $\pH^i(M^{\vee}*M)[-i]$ is a direct sum of shifted irreducible perverse sheaves. Since by Lemma \ref{lem-prop-conv} $\dim\Hom (M^\vee * M, \un)=\dim \Hom (M,M)=1$, there is a unique integer~$\nu\geq 0$ and a unique perverse irreducible summand $\mcP_M[-\nu]$ of $\pH^\nu(M^\vee * M)[-\nu]$ on which the restriction of $\ev$ is not zero (see
  \cite[p.\,563\,and\,Remark\,(2),\,p.\,569]{weissauer_vanishing_2016} for details).

  
  Since $\chi(X_{\bar k}, M)=0$, we have $\nu\geq 1$, and
  by~\cite[Lemma\,2]{weissauer_vanishing_2016}, it follows that
  $\mcS_1(M)=\mcS_1(\mcP_M)$.
  Weissauer shows (see~\cite[Prop.\,2]{weissauer_vanishing_2016}) that
  there exists an abelian subvariety~$A$ of~$X$ of
  dimension~$\nu\geq 1$, with closed immersion $i\colon A\to X$, and a
  character~$\chi_0$ such that there is a geometric isomorphism
  $\mcP_M\simeq \mcL_{\chi_0}\otimes i_*\Qlb[\nu]$. Since~$A$ can therefore
  be recovered as the support of~$\mcP_M$, it is defined over~$k$.

  The perverse sheaf~$\mcP_M$ is invariant by translation
  by~$A_{\bar{k}}$, and then so is~$M$
  by~\cite[Remark\,(2)]{weissauer_vanishing_2016}. 
  
  Moreover, we have $c_u(\mcP_M)\ll c_u(M)$ by the definition
  of~$\mcP_M$ and Proposition~\ref{prop-JorHol-cond}, and
  $\deg(u\circ i(A))\ll c_u(\mcP_M)\ll c_u(M)$ by
  Theorem~\ref{thm-lisse-locus-conductor}.

  \par
  Let~$q\colon X\to X/A$ be the quotient morphism and~$\chi$ be a character not in the \tac\ $S$ of~$X_{k}$ defined
  by~$(q,\chi_0^{-1})$. We now compute for every~$i\in\Zz$ that
  \[
    \rmH^i(X_{\bar k}, \mcP_{M\chi})=\rmH^i((X/A)_{\bar k},
    Rq_*(\mcP_{M\chi})=\rmH^i((X/A)_{\bar k}, Rq_*(\mcL_{\chi\cdot
      \chi_0})\otimes \un[\dim(A)]).
  \]
  \par
  Since~$\chi$ is not in the \tac\ $S$, 
  the restriction of $\mcL_{\chi\cdot \chi_0}$ to $A_{\bar k}$ is
  non-trivial, and hence we have $Rq_*(\mcL_{\chi\cdot \chi_0})=0$ by
  Lemma~\ref{lm-image-character},
  and therefore $\rmH^i(X_{\bar k}, \mcP_{M\chi})=0$ for all~$i$.

  Conversely, if $\chi=\chi_0^{-1} \cdot (\tilde \chi\circ q)$, then we
  have
  \[
    \rmH^{*}(X_{\bar k}, \mcP_{M\chi})=\rmH^*(A_{\bar
      k},\Qlb)\otimes\rmH^{*}((X/A)_{\bar k},\mcL_{\tilde
      \chi}\un[\dim(A)]),
  \]
  by the Künneth formula, and this is non-zero.
\end{proof}


\begin{proof}[Proof of Theorem~\ref{thm-gen-van-AV}]
  We follow the method used by Krämer and Weissauer to
  prove~\cite[Th.~1.1]{KW_vanishing_AV}, keeping track of the
  complexity.

  Since~$X$ is an abelian variety, the two convolution products of
  Section~\ref{sec-convolution} coincide; for an object~$M$ of~$\Der(X)$
  and an integer $n\geq 1$, we denote by $M^{*n}$ the $n$-th iterated
  convolution product of~$M$.

  We recall the axiomatic framework of~\cite[Section
  5]{KW_vanishing_AV}, specialized to our situation as in
  \cite[Example~5.1]{KW_vanishing_AV}. Let $\Dd$ be the full subcategory
  of $\Der(X_{\bar k})$ whose objects are direct sums of shifts of
  geometrically semisimple perverse sheaves which are obtained by
  pullback from $X_{k_n}$ for some~$n\geq 1$. Let
  $\Pp\subset \Perv(X_{\bar k})$ be the corresponding subcategory of
  perverse sheaves, namely that with objects the geometrically
  semisimple perverse sheaves arising by pullback from $X_{k_n}$ for
  some~$n\geq 1$. Then the categories $\Pp$ and $\Dd$ satisfy the axioms
  (D1), (D2), and~(D3) of \cite[Section~5]{KW_vanishing_AV}, namely:
  \begin{itemize}
  \item[(D1)] The category $\Dd$ is stable under degree shift,
    convolution and perverse truncation functors; the category $\Pp$ is
    the heart of this $t$-structure, and is a semisimple abelian
    category.
  \item[(D2)] Any object $M$ of~$\Dd$ can be written (non-canonically) as
    a direct sum
    \[
      \bigoplus_{n\in \Zz} \pH^m(M)[-m].
    \]
  \item[(D3)] The Hard Lefschetz Theorem holds for objects of $\Dd$.
  \end{itemize}
  
  Let~$\Nn$ be the full subcategory of~$\Dd$ whose objects are the
  complexes~$N$ such that all geometrically simple constituents of all
  perverse cohomology sheaves $\pH^i(N)$ for $i\in \Zz$ have
  Euler--Poincaré characteristic equal to~$0$. By
  \cite[Cor.\,6.4]{KW_vanishing_AV}, the category~$\Nn$ satisfies the
  axioms (N1), (N2), (N3) and (N4) of \cite[Section 5]{KW_vanishing_AV},
  namely:
  \begin{itemize}
  \item[(N1)] We have $\Nn* \Dd\subset \Nn$ and the category $\Nn$ is
    stable under direct sums, retracts, degree shifts, perverse
    truncation and duality;
  \item[(N2)] If $N$ is an object of $\Nn$, then for most
    characters~$\chi$, we have $\rmH^i(X_{\bar k}, N_\chi)=0$ for
    all~$i$;
  \item[(N3)] The category $\Nn$ contains all objects $M$ of~$\Dd$ such
    that $\rmH^i(X_{\bar k},N)=0$ for all~$i\in\Zz$;
  \item[(N4)] The category $\Nn$ contains all simple objects of $\Pp$
    with zero Euler--Poincaré characteristic.
  \end{itemize}
  (Note that we will not make use of this version of (N2).)
  
  By~\cite[Theorem~9.1]{KW_vanishing_AV}, every $M\in \Pp$ is an
  $\Nn$-multiplier, meaning that for all integers $i\not=0$ and any
  integer $r\geq 1$, every subquotient of
  $\pH^i((M\oplus M^\vee)^{*r})$ lies in $\Nn$.
  
  We now argue as in the proof of \cite[Lemma~8.2]{KW_vanishing_AV} to
  prove Theorem~\ref{thm-gen-van-AV} for a perverse sheaf~$M$ on~$X$.

  \textbf{Step 1}. We assume that $M$ is arithmetically simple. By
  Lemma~\ref{lem-arith-ss-geo}, the base change
  of $M$ to $\bar{k}$ is an object of $\Pp$. We denote $g=\dim(X)$;
  by~(D2), we have
  \[
    M_{\bar{k}}^{*(g+1)}\simeq\bigoplus_{m\in \Zz} M_m[m],
  \]
  for some objects $M_m$ of~$\Pp$, which are in fact objects of~$\Nn$
  for~$m\not=0$ since $M$~is an $\Nn$-multiplier.

  By Proposition \ref{prop-JorHol-cond}, the number of integers~$m$
  such that $M_m$ is non-zero is bounded in terms of~$c_u(M)$, and
  similarly $c_u(M_m)$ is bounded in terms of~$c_u(M)$.  By the
  semisimplicity property in~(D1), each~$M_m$ is a direct sum of
  simple perverse sheaves in~$\Nn$, and by
  Proposition~\ref{prop-JorHol-cond}, the number and the complexity of
  these constituents are bounded in terms of $c_u(M)$. We denote
  by~$\mathscr{C}$ the finite set of all these simple perverse
  sheaves. By Lemma~\ref{lem-arith-ss-geo}, there exists a finite
  extension~$k'$ of~$k$, of degree bounded in terms of~$c_u(M)$, such
  that any element $C$ of~$\mathcal{C}$ is defined over~$k'$.

  We apply Proposition~\ref{prop-chi=0-trans-inv} to each
  $C\in\mathcal{C}$. Let~$\mcS_{C}$ denote the corresponding \tac; it is
  of degree bounded in terms of~$c_u(M)$.

  We claim that if $\chi\in \charg{X}$ does not belong to the union of
  the \tacs\ $\mcS_C$, then we have
  \[
    \rmH^i(X_{\bar k},M_\chi)=0
  \]
  for all~$i\not=0$.  This statement will conclude the proof of
  Theorem~\ref{thm-gen-van-AV} for~$M$.
  
  Let $\chi$ be a character that is not in any of the \tacs\
  $\mcS_C$. Since $M_{\chi}^{*(g+1)}$ is isomorphic to
  $(M^{*(g+1)})_{\chi}$ and $ H^*(X_{\bar{k}},C_{\chi})=0$ for
  $\chi\notin \mcS_C$, we have
  \[
    \rmH^i(X_{\bar k},M^{*(g+1)}_\chi)=\rmH^i(X_{\bar k},{M_0}_\chi),
  \]
  for any $i\in\Zz$. The right-hand side vanishes if $\abs{n}>g$
  since~$M_0$ is perverse. Finally, by the compatibility between
  convolution and the Künneth formula (see Lemma~\ref{lem-prop-conv}
  below) we also have an isomorphism
  \[
    \rmH^*(X_{\bar k},M^{*(g+1)}_\chi) \simeq
    \rmH^*(X_{\bar k},M_\chi)^{\otimes(g+1)},
  \]
  and by comparing we see that only the space
  $\rmH^0(X_{\bar k},M_\chi)$ may be non-zero, which establishes the
  claim.

  \textbf{Step 2.} Now let~$M$ be an arbitrary perverse sheaf on~$X$. By
  Proposition~\ref{prop-JorHol-cond}, the number of geometric
  Jordan--Hölder factors of $M$ is bounded in terms of~$c_u(M)$, and
  hence also the number of arithmetic Jordan-Hölder factors; we then
  apply the first step to each of the terms of a composition series
  for~$M$, and deduce the corresponding result for~$M$.
\end{proof}

\subsection{The relative version}\label{ssec-relative}

Our next goal is to establish a relative version of
Theorem~\ref{thm-gen-van-AV}.  The arguments over the complex numbers of Krämer and Weissauer
in~\cite[Section 2]{KW_vanishing_AV} do not
apply to our situation over finite fields, since they rely on Verdier stratifications. We instead use a
constructibility result of Orgogozo~\cite{Orgo_constr_mod}, which is a
stratification result, locally for the alteration
topology.\index{alteration topology}

\begin{theorem}
\label{thm-gen-van-AV-rel}
Let $S$ be a quasi-projective scheme over~$k$, and let $A$ be an abelian
variety over~$k$. Let $X=A\times S$, and denote by $f\colon X\to S$ the
canonical morphism. Fix a projective embedding~$u$ of~$X$.

Let $\alpha\colon X'\to X$ be an alteration defined over~$k$, and
$\mcX'$ a stratification of $X'$.

Let $a\leq b$ be integers. Let $M$ be an object of~$\Der(X)$ with
perverse amplitude $[a,b]$ such that $\alpha^*M$ is tame and
constructible along $\mcX'$.

There exist an integer $d\geq 1$, a finite extension $k'$ of~$k$ and a
finite family $(S_f)_{f\in \mathcal{F}}$ of \tacs\ of~$A_{k'}$, such that
\begin{enumth}
\item The integer $d$ and the size of~$\mathcal{F}$ are bounded in
  terms of~$c_u(M)$ and the data $(X,X',\alpha,\mcX')$,
\item Each \tac\ $S_f$  has degree at most~$d$,
\item The degree of $k'$ is at most~$d$,
\end{enumth}
with the property that for any $\chi\in\charg{A}_{k'}$ which does not
belong to the union of the~$S_f$, the object $Rf_*(M_{\chi})$ has
perverse amplitude~$[a,b]$.
\end{theorem}

By~\cite[Prop.\,1.6.7]{Orgo_constr_mod}, for any object~$M$ of
$\Der(X)$, there does exist an alteration $\alpha\colon X'\to X$ (in
fact, a finite surjective morphism) and a stratification $\mcX'$ of $X'$
such that $\alpha^*M$ is tame and constructible along $\mcX'$.
In particular, the following corollary follows.

\begin{corollary}
\label{cor-gen-van-AV-rel}
Let $S$ be a quasi-projective scheme over~$k$ and let~$A$ be an abelian
variety over~$k$. Define $X=A\times S$ and denote
$f\colon A\times S\to S$ the projection.

Let~$a\leq b$ be integers and let~$M$ be an object of~$\Der(X)$ with
perverse amplitude $[a,b]$. There exist a finite extension $k'$ of $k$
and a finite family $(S_f)_{f\in\mathcal{F}}$ of \tacs\ of $A_{k'}$ such
that for any character $\chi\in\charg{A}_{k'}$ that does not belong to
the union of the $S_f$, the object $Rf_*(M_{\chi})$ has
perverse amplitude~$[a,b]$.
\end{corollary}

For the proof of Theorem~\ref{thm-gen-van-AV-rel}, we use the following
special case of \cite[Th.\,3.1.1]{Orgo_constr_mod}.

\begin{theorem}[Orgogozo]
\label{thm-strat-alt-orgo}
Let $f\colon X\to Y$ be a proper morphism defined over $k$. Let
$\alpha\colon X'\to X$ be an alteration and $\mcX'$ a stratification of
$X'$. Then there exist an alteration $\beta\colon Y'\to Y$ and a
stratification $\mcY'$ of $Y'$ such that for any object $M$
of~$\Der(X)$, the condition that $\alpha^*(M)$ is tame and constructible
along $\mcX'$ implies that $\beta^* Rf_*M$ is tame and constructible
along $\mcY'$.
\end{theorem}

\begin{proof}[Proof of Theorem \ref{thm-gen-van-AV-rel}]
  By shifting and Verdier duality, it is enough to prove the weaker
  statement where ``$M$ is of perverse amplitude $[a,b]$'' is replaced
  by ``$M$ is semiperverse''.

  Apply Theorem~\ref{thm-strat-alt-orgo} to the proper morphism
  $f\colon A\times S\to S$ and to the alteration~$\alpha$. We obtain an
  alteration $\beta\colon S'\to S$ and a stratification $\mcS'$ of $S'$
  such that $\beta^*Rf_*M$ is tame and constructible along $\mcS'$. Note
  that since any~$\mcL_{\chi}$ is lisse and tame, $\alpha^*M_\chi$ is
  tame and constructible along $\mcX'$ (see
  \cite[5.2.5]{Orgo_constr_mod} for details), and hence the complex
  $\beta^*Rf_*M_{\chi}$ is also tame and constructible along~$\mcS'$ for
  any~$\chi\in\charg{A}$.

  Consider the image of the stratification $\mcS'$ by $\beta$. By
  Chevalley's theorem, it is a covering of $S$ by constructible sets,
  but not necessarily a partition. Refine this covering and remove
  redundant strats in order to obtain a stratification $\mcS$ of $S$
  where all strats are equidimensional. Then refine the stratification
  $\mcS'$ in such a way that preimages by $\beta$ of strats of $\mcS$
  are union of strats of $\mcS'$ and that $\beta$ induces surjective
  morphisms from each strat of~$\mcS'$ to a strat of~$\mcS$.

  Let~$\chi\in\charg{A}$. Even if the complex $Rf_*M_{\chi}$ is not
  necessarily constructible along $\mcS$, it has the property that for
  any strat $S_i$ of $\mcS$, the support of the restriction of each
  cohomology sheaf of $Rf_*M_{\chi}$ to $S_i$ is either $S_i$ or empty,
  since the analogue property holds for $\beta^*Rf_*M_{\chi}$ and the
  stratification $\mcS'$, and $\beta$ is surjective from a strat of
  $\mcS'$ to one of~$\mcS$.

  Consider now the preimage of the stratification $\mcS$ by $f$, and
  also the image of the stratification $\mcX'$ of $X'$ by
  $\alpha$. Choose a stratification~$\mcX$ of $X$ that refines both
  these coverings of~$X$, with the property that for any strats $X_i$
  and~$S_j$ of~$\mcX$ and~$\mcS$ such that $f(X_i)\subset S_j$, the
  restriction of~$f$ to~$X_i$ is smooth (in particular, that $X_i$ is
  equidimensional above~$S_j$). Now refine $\mcX'$ similarly to
  $\mcS'$, in such a way that preimages by $\alpha$ of strats of
  $\mcX$ are union of strats of $\mcX'$ and $\alpha$ induces
  surjective morphisms from any strat of $\mcX'$ to a strat of $\mcX$.

  By Lemma~\ref{lm-lang-weil}, up to replacing $k$ with a finite
  extension of degree bounded in terms of $c_u(M)$ (and the fixed data
  $(X,X',\alpha,\mcX')$), we can assume that each strat $S_i$ of~$\mcS$
  has a $k$-rational point $s_i$. We now apply
  Corollary~\ref{cor-gen-van-AV-2} for each~$i$ to the restriction
  $M_{s_i}$ of $M$ to~$f^{-1}(s_i)\simeq A$ for each $i$, obtaining
  extensions~$k_i$ of~$k$ and families $(S_{f,i})_{f\in\mathcal{F}_i}$
  of \tacs\ of~$A_{k_i}$ satisfying the properties of this corollary.

  Let~$k'$ be the compositum of all~$k_i$, which has degree bounded in
  terms of $c_u(M)$ and the fixed data. We claim that for any
  character $\chi\in\charg{A}_{k'}$ that belongs to none of the \tacs\
  $S_{f,i}$ for any~$i$, the object $Rf_*M_{\chi}$ is
  semiperverse. This will conclude the proof.

 
  Suppose that the claim fails for some~$\chi$. Then there exists an
  integer $k\in \Zz$ such that
  $$
  \dim\Supp(\mcH^k(Rf_*(M_\chi)))>-k.
  $$
  Since $\Supp(\mcH^k(Rf_*(M_\chi)))$ is a union of strats of $\mcS$,
  there is a strat $S_i\subset\Supp(\mcH^k(Rf_*(M_\chi)))$ of $\mcS$ of
  dimension $>-k$. In particular, we have
  $\mcH^k(Rf_*(M_\chi))_{s_i}\neq 0$. By proper base change, we have
  $\mcH^k(Rf_*(M_\chi))_{s_i}=\rmH^k(A_{\bar k}\times
  \set{s_i},{M_{s_i}}_\chi)$, and hence the latter is also non-zero. From
  the assumption on~$\chi$ and Corollary~\ref{cor-gen-van-AV-2}, we have
  \[
    \rmH^k(A_{\bar k}\times \set{s_i},{M_{s_i}}_\chi)\simeq
    \rmH^0(A_{\bar k}\times \set{s_i},\pH^k(M_{s_i})_\chi),
  \] 
  and hence $\pH^k(M_{s_i})=\pH^0(M_{s_i}[k])\neq 0$. By definition of the
  perverse t-structure, this implies that there exists some $r\in \Zz$
  such that
  \[
    \dim \Supp(\mcH^r(M_{s_i}))\geq-r+k.
  \]

  The support of $\mcH^{r}(M)$ is a union of strats of $\mcX$, so there
  exists a strat $X_j\subset \Supp(\mcH^{r}(M))$ of $\mcX$ with
  $\dim(X_j\cap A\times \set{s_i})=\dim \Supp(\mcH^r(M_{s_i}))$. Since $X_j$ is equidimensional
  over $S_i$ and $\dim(S_i)>-k$, we conclude that
  \[
    \dim\Supp(\mcH^r(M))\geq \dim(X_j)\geq-r+k+\dim(S_i)>-r,
  \]
  contradicting the semiperversity of $M$.
\end{proof}

We now prove a vanishing theorem for higher cohomology groups of
perverse sheaves on abelian varieties. We begin with an analogue of
Theorem~\ref{thm-high-van-torus}. 

\begin{proposition}
\label{prop-av-hv-main}
Let $A$ be a $g$-dimensional algebraic variety over~$k$, let $S$ be a
quasi-projective scheme over~$k$, and define $X=A\times S$.  Fix a
projective embedding~$u$ of $X$.

Let $\alpha\colon X'\to X$ be an alteration and $\mcX'$ a
stratification of $X'$.

Let $i$ be an integer with $1\leq i\leq g$. Let $a\leq b$ be integers.

Let~$M$ be an object of $\Der(X)$ with perverse amplitude $[a,b]$ such
that $\alpha^*M$ is tame and constructible along $\mcX'$.  There exist a
finite extension $k'$ of~$k$ and a family $(S_f)_{f\in\mathcal{F}}$ of
\tacs\ of~$A_{k'}$ of dimension~$\leq d-i$ with the property that for
any~$\chi\in\charg{A}_{k'}$ which does not belong to the union of
the~$S_f$ there exists a quotient abelian variety $q\colon A_{k'}\to B$
of dimension at most $i-1$ such that $Rq_{S*}M_{\chi}$ has perverse
amplitude $[a,b]$.

Moreover, the degree of $k'$ over~$k$ and the size of~$\mathcal{F}$
depend only on $c_u(M)$ and $(X,X',\alpha,\mcX')$.

\end{proposition}

\begin{proof}
  As in the proof of Theorem~\ref{thm-strat-alt-orgo}, we can work
  with each perverse cohomology sheaf, and it is therefore enough to
  prove the proposition for $a=b=0$, which means that $M$ is perverse.

  By Poincaré's complete reducibility theorem, up to replacing $k$
  with a finite extension, there exists an isogeny $f\colon A\to B$
  over~$k$ where $B$ is a product of geometrically simple abelian
  varieties. We first claim that it is enough to prove the proposition
  for $B$.

  To see this, we assume that the statement holds for~$B$. Consider the
  base change $f_B\colon X\to B\times S$.  Since $f$ is finite,
  $f_{B*}(M)$ is perverse for every perverse sheaf $M$ on $X$. By
  Theorem \ref{thm-strat-alt-orgo}, we find an alteration
  $\beta\colon B'\to B\times S$ and a stratification of $B'$ such that
  $\beta^* f_{B*}(M_\chi)$ is tame and adapted for every $M$ such that
  $\alpha^*M$ is tame and adapted to~$\mcX'$. Then the proposition can
  be applied to $f_{B*}(M_\chi)$. Let $N$ be the kernel of the isogeny
  $f$. Choose up to $\abs{N}$ characters of $A$ whose restrictions to
  $N$ run over the character group of~$N$. Then the proposition for $A$
  follows by applying the result for $B$ to the objects
  $f_{B*}(M_\chi)$, where $\chi$ varies among this finite set of
  characters. This proves the claim.

  So we assume that $A=A_1\times \cdots \times A_s$ is a product of
  geometrically simple abelian varieties. Set $g_j=\dim(A_j)$ for
  all~$j$. For any subset $I\subset [s]$, let
  $$
  A_I=\prod_{i\in I}A_i,
  $$
  viewed as a subvariety of~$A$, and let
  $A_I^{\bot}=A_{[s]\setminus I}$ be the kernel of the canonical
  projection $A\to A_I$.

  Fix an integer $1\leq i\leq g=\dim(A)$. Let $\mcE$ be the set of
  subsets $I\subset [s]$ such that $\dim(A_I)<i$; for $I\in\mcE$, we
  have $\dim(A_{I}^{\bot})>g-i$. 

  Fix $I\in\mcE$. Let $p\colon A\times S\to A_I\times S$ be the
  projection. We apply Theorem~\ref{thm-gen-van-AV-rel} to~$p$ and~$M$,
  i.e., with $(A,S)$ there equal to $(A_I^{\bot},A\times S)$. Up to
  replacing $k$ by a finite extension~$k'$, we obtain a finite family
  $(S_{I,j})_{j\in X_I}$ of \tacs\ of $A_{I,k'}^{\bot}$ such that the
  object $Rp_*(M_{\chi})$ is perverse on $A_I\times S$ for any
  $\chi\in\charg{A}^{\bot}_{I,k'}$ not in the union of these \tacs.
  Let
  $$
  (\pi_{I,j},\chi_{I,j})_{j\in X_I}
  $$
  be the projection and characters defining these \tacs, and
  let~$K_{I,j}=\ker(\pi_{I,j})$, viewed as a subgroup of~$A_{k'}$.


  Let~$\mathcal{F}$ be the set of all maps $f$ from $\mcE$ to the
  disjoint union of the $\mathcal{S}_I$ that send a subset $I$ to an
  element $j\in X_I$ for each~$I$; this set is finite. For
  $f\in\mathcal{F}$, let~$S_f$ be the intersection of the \tacs\
  of~$A_{k'}$ defined by
  $$
  (A_{k'}\to A_{k'}/K_{I,f(I)},\chi'_{I,f(I)})
  $$
  for~$I\in \mcE$.
  \par
  We claim that the family $(S_f)_{f\in\mathcal{F}}$ (to be precise, the
  subfamily where $S_f$ is not empty) satisfies the assertions of the
  theorem.
  \par
  Indeed, first of all Lemma~\ref{lem:inter:tac} shows that~$S_f$ is
  either empty or is again a \tac; moreover, in the second case, it is
  defined by the projection $A_{k'}\to A_{k'}/B_f$ where $B_f$ is the
  abelian subvariety in $A_{k'}$ generated by the $K_{I,f(I)}$, viewed
  as subvarieties of $A_{k'}$. For such~$f$, by Lemma~\ref{lm-geometric}
  applied to~$A$ and the subgroups~$K_{I,f(I)}$, we have
  $\dim(B_f)\geq i$, and hence the quotient
  $$
  p_f\colon A_{k'}\to A_{k'}/B_f
  $$
  has image of dimension $\leq d-i$.

  Finally, let $\chi\in\charg{A}_{k'}$ be a character that does not
  belong to any of the \tacs\ $S_f$. This implies that there exists
  some~$f\in\mathcal{F}$, some subset $I\subset \mcE$ and
  some~$j\in X_I$ such that the restriction $\chi_I$ of $\chi$
  to~$A_{I,k'}^{\bot}$ is not equal to~$\chi_{I,j}$.
  
  We can write~$\chi=\chi_I\chi'$ where $\chi'$ is a character
  of~$A_{I,k'}$. Then, considering the particular quotient
  $q\colon A_{k'}\to A_{I,k'}$, the base change $q_S$ is the canonical
  projection $q_I$ and hence
  $$
  R_{q_S*}M_{\chi}=R_{qS*}(M_{\chi_I})\otimes \mcL_{\chi'}
  $$
  is perverse.
\end{proof}

As in the case of tori, we state two further consequences that are
useful in applications.

\begin{corollary}
\label{cor-av-hv}
Let $A$ be an abelian variety defined over $k$ of dimension $g$. Let~$M$
be a perverse sheaf on~$A$.  For $-g\leq i\leq g$, the sets
$$
\{\chi\in\charg{A}\,\mid\,
\rmH^i(A_{\bar{k}},M_{\chi})\not=0\}
$$
are contained in a finite union of \tacs\ of~$A$ of
dimension~$\leq g-|i|$, and in particular they have character
codimension at least~$|i|$.
\end{corollary}

\begin{proof}
  We argue as in the proof of Corollary~\ref{cor-high-van-torus} using
  the previous theorem (with $a=b=0$), as we may since we have
  recalled that one can find an alteration~$\alpha$ of $A$ such that
  the pull-back $\alpha^*M$ is tame.
\end{proof}

\begin{corollary}
\label{cor-av-hv-main}
Let $A$ be a $g$-dimensional algebraic variety over~$k$, let $S$ be a
quasi-projective scheme over~$k$, and define $X=A\times S$.  Fix a
projective embedding~$u$ of $X$ and denote by~$q$ the projection $X\to
S$. 

Let $\alpha\colon X'\to X$ be an alteration and $\mcX'$ a
stratification of $X'$.

Let~$M$ be a perverse sheaf on~$X$ such that $\alpha^*M$ is tame and
constructible along $\mcX'$. 
There exists a finite extension $k'/k$ and a partition of
$\charg{A}_{k'}$ into subsets $(S_i)_{0\leq i\leq g}$ of character
codimension $\geq i$ such that for any~$i$ and~$\chi\in S_i$, the object
$Rq_!M_{\chi}$ has perverse amplitude~$[-i,i]$.

Moreover, for any integer~$n\geq 1$, we have
\begin{equation}\label{eq-abelian-bound}
  |S_i(k_n)|\ll |k|^{n(g-i)},
\end{equation}
where the implied constant depends only on $(c_u(M),X,X',\alpha,\mcX')$.
\end{corollary}

\begin{proof}
  We argue as in the proof of Corollary~\ref{cor-split-torus} for the
  first part; to deduce~(\ref{eq-abelian-bound}), we simply note for
  each~$i\leq g$, the number of \tacs\ in
  Proposition~\ref{prop-av-hv-main} is bounded in terms of the indicated
  data, and for each \tac\ $S$ of dimension~$i$, the number of
  characters in $S(k_n)$ is $\leq (|k_n|^{1/2}+1)^{2i}\ll |k|^{ni}$.
\end{proof}


\section{Proof of the general vanishing theorem}
\label{ssec-proof-vanish}

We can now prove Theorems~\ref{thm-gen-vanish} and~\ref{thm-high-vanish}.

We consider the dévissage of Proposition~\ref{pr-devissage}. Namely,
let~$A$ be an abelian variety, $T$ a torus,~$U$ a unipotent group
and~$N$ a finite commutative subgroup scheme of~$A\times U\times T$
such that $G$ is isomorphic to $(A\times U\times T)/N$. Further, we
write $N=N_r\times N_l$ where~$N_r$ is reduced and~$N_l$ is local.


Let~$M$ be a perverse sheaf on~$G$.

\textbf{Step 1.} We claim that it is enough to prove the theorems for the
group $\widetilde{G}=A\times U\times T$.

Indeed, since $N=N_r\times N_l$, the quotient morphism
$p \colon \widetilde{G}\to G$ can be factored as the composition of an
étale isogeny and a purely inseparable one. The latter is a universal
homeomorphism, and since universal homeomorphisms preserve the étale
site, and since pull-back by a finite étale map preserves perversity, it
follows that the pull-back $p^*(M)$ is perverse.

Assume that the result of Theorem~\ref{thm-high-vanish} holds for
$p^*(M)$ on $\widetilde{G}$. Then we obtain the vanishing theorem
for~$M$ as follows. Let~$\mcS'_i$ be the subsets of loc. cit. for
$p^*(M)$ on~$\widetilde{G}$, and define $\mcS_i$ to be the set of
$\chi\in\charg{G}$ such that $\chi\circ p\in\mcS'_i$. Since $G$ has the
same dimension as~$\widetilde{G}$ and $\mcS'_i$ has character
codimension~$i$, do does $\mcS_i$.

If~$\chi\in \charg{G}$, then the projection formula gives isomorphisms
\[
  \rmH^i(\widetilde{G}_{\bar{k}}, p^*(M_\chi))=
  \rmH^i(\widetilde{G}_{\bar{k}}, p^*(M)_{\chi\circ p})
\]
for all~$i\in\Zz$.
\par
The vanishing of $\rmH^i(\widetilde{G}, p^*(M_\chi)_{\chi \circ p})$
implies that of $\rmH^i(G_{\bar{k}}, M_\chi)$, since the latter space is
a direct summand of the former. A similar argument applies for
compactly-supported cohomology, which shows that the characters
$\chi\in\charg{G}$ such that any of the groups~(\ref{eq-groups}) is
non-zero belong to~$\mcS_i$.

Finally, suppose that $\chi\in\mcS_0\setminus \mcS_1$, so that
$\chi \circ p\in \mcS'_0\setminus \mcS'_1$. Since the forget support map
is functorial, the forget support morphism
$$
\rmH^0_c(\widetilde{G}_{\bar k},p^*(M_\chi))\to
\rmH^0(\widetilde{G}_{\bar k},p^*(M_\chi))
$$
induces by restriction the forget support morphism
$$
\rmH^0_c(G_{\bar k},M_\chi)\to \rmH^0(G_{\bar k},M_\chi),
$$
and since the former is an isomorphism (from our assumption that
Theorem~\ref{thm-high-vanish} holds for~$\widetilde{G}$), so is the
latter.  This concludes the proof of the claim of Step~$1$.

\par
\textbf{Step 2.} We now assume that $G=A\times U\times T$, and will
prove Theorem~\ref{thm-high-vanish}. We fix a quasi-projective
immersion~$u$ of $G$. Let $d_A=\dim(A)$, $d_U=\dim(U)$, $d_T=\dim(T)$,
and $d=d_A+d_U+d_T=\dim(G)$. We denote by
$p_T\colon A\times U\times T\to A\times U$ the canonical projection.

Up to replacing $k$ by a finite extension, we can assume that~$T$ is
split.  By applying
Theorem~\ref{thm-high-van-torus} and
Corollary~\ref{cor-split-torus} with~$S=A\times U$,
we can partition $\charg{T}$ into subsets $(S_i)_{0\leq i\leq d_T}$ of
character codimension $\geq i$ such that
\begin{enumerate}
\item if~$\chi\in S_i$,
then the complex $Rp_{T!}(M_\chi)\in \Der(A\times U)$ is of perverse
amplitude $[-i,i]$.
\item if $\chi\in \widehat{T}\setminus S_1$, then
  $Rp_{T!}(M_{\chi})=Rp_{T*}(M_{\chi})$.
\end{enumerate}

We now wish to apply Proposition~\ref{prop-av-hv-main} to $A\times U$,
but we first need to find an alteration that moderates all complexes
$Rp_{T!}(M_\chi)$.

Let $j \colon T\to \bar T=(\mathbf{P}^1)^{d_T}$ be the obvious
compactification of $T$. By~\cite[Prop.\,1.6.7]{Orgo_constr_mod}, there
exists an alteration $\alpha\colon X\to A\times U\times \bar T$ and a
stratification~$\mcX$ of $X$ such that $\alpha^*(j_!M)$ is tame and
constructible along~$\mcX$. For each character $\chi\in \charg{T}$, the
sheaf $j_!(\mcL_\chi)$ is tame,
and hence $\alpha^*(j_!M_\chi)$ is also constructible and tame along $\mcX$ (see
\cite[5.2.5]{Orgo_constr_mod} for details).

We apply Theorem~\ref{thm-strat-alt-orgo} to the proper projection
$A\times U\times \bar T\to A\times U$.  This provides us with an
alteration $\beta\colon X'\to A\times U$ and a stratification $\mcX'$ of
$X'$ such that the complex $\beta^*Rp_{T!}(M_\chi)$ is tame and
constructible along $\mcX'$ for every $\chi\in \charg{T}$. Moreover, by
Proposition~\ref{prop-cond-char} and
Theorem~\ref{thm-conductors}, the complexity of $Rp_{T!}(M_\chi)$ is
bounded independently of $\chi\in \charg{T}$.

We can now apply Corollary~\ref{cor-av-hv-main}
to~$S=U$ and the complexes $Rp_{T!}(M_\chi)$. For each character
$\chi\in \charg{T}$, we obtain a partition
$(S_{\chi,j})_{0\leq j\leq d_A}$ of~$\charg{A}$ into subsets such that
$S_{\chi,j}$ has character codimension at least~$j$, with the property
that for $(\chi,\xi)\in S_i\times S_{\chi,j}$, the complex
$Rp_{A!}(Rp_{T!}(M_\chi))_\xi)$ has perverse amplitude $[-i-j,i+j]$.

By Proposition \ref{prop-cond-char} and Theorem \ref{thm-conductors},
the complexity of the object $Rp_{A!}(Rp_{T!}(M_\chi))_\xi)$ is bounded
independently of $(\chi,\xi) \in S_i\times S_{\chi,j}$. Hence, by
applying Proposition~\ref{prop-strat-unip} to these
objects 
we find for each $(\chi,\xi)$ a partition
$(S_{\chi,\xi,m})_{0\leq m\leq d_U}$ of $\charg{U}$ such that the set
$S_{\chi,\xi,m}$ has character codimension at least~$m$ and, moreover,
we have
\[
  \rmH^n_c(G_{\bar{k}},M_{\chi\xi\psi})=0
\]
for each $\psi\in S_{\chi,\xi,m}$ unless $n\in [-i-j,i+j+m]$.

For~$0\leq r\leq d$, we now define $\widetilde{\mcS}_r$ to be the set of characters
$(\chi,\xi,\psi)\in \charg{G}$ such that 
$$
\psi\in S_{\chi,\xi,m},\quad\quad
\xi\in S_{\chi,j},\quad\quad
\chi\in S_i
$$
for some $i$, $j$, $m$ such that $i+j+m\geq r$.


For any integer~$n\geq 1$, we have
$$
|\widetilde{\mcS}_r(k_n)|=\sum_{i+j+m\geq r} \sum_{\chi\in S_i(k_n)} \sum_{\xi\in
  S_{\chi,j}(k_n)} |S_{\chi,\xi,m}(k_n)| \ll |k|^{n(d-(i+j+m))}\ll
|k|^{n(d-r)}
$$
by~(\ref{eq-unipotent-bound}) and~(\ref{eq-abelian-bound}) (note that
the uniformity with respect to the perverse sheaf in these estimates,
and the uniform bound on the complexity, are crucial to control the sums
over~$\chi$ and~$\xi$). Thus the set $\widetilde{\mcS}_r$ has character codimension
at least~$r$.

By construction of the sets $S_i$, $S_{\chi,j}$ and $S_{\chi,\xi,m}$,
the condition $\rmH^{i}_c(G_{\bar{k}},M_{\chi\xi\psi})\neq 0$, for
$(\chi,\xi,\psi)\in \charg{G}$, implies that
$(\chi,\xi,\psi)\in \mcS_{\abs{i}}$. We apply a similar argument
with~$\DD(M)$ to obtain the analogue conclusion for ordinary cohomology
and set $\mcS_i$ to be the intersection of the set $\widetilde{\mcS}_i$
for $M$ and of the analogue for $\DD(M)$. By construction, the sets
$\mcS_i$ satisfy the first two claims of Theorem~\ref{thm-high-vanish}.

We now establish the last claims of Theorem~\ref{thm-high-vanish}.

First, let $(\chi,\xi,\psi)\in \charg{G}\setminus \mcS_{1}$. By
construction of~$S_1$ through Theorem \ref{thm-high-van-torus} (see
point~(2) above), we have $Rp_{T!}(M_\chi)=Rp_{T*}(M_\chi)$. Moreover
$p_{A!}=p_{A*}$ since $p_A$ is proper, and by the last claim of
Proposition~\ref{prop-strat-unip}, we obtain
\[
  \rmH^0(G_{\bar{k}},M_{\chi\xi\psi})
  =\rmH^0(U_{\bar{k}},Rp_{A*}Rp_{T*}M_{\chi\xi\psi})
  =\rmH^0_c(U_{\bar{k}},Rp_{A!}Rp_{T!}M_{\chi\xi\psi})=
  \rmH^0_c(G_{\bar{k}},M_{\chi\xi\psi}).
\]

Finally, if~$G$ is a torus (resp. an abelian variety) then we use
Corollary~\ref{cor-high-van-torus} (resp. Corollary~\ref{cor-av-hv}) to
prove that the sets $\mcS_i$ are contained in a finite union of \tacs\
of~$G$ of dimension $\leq d-i$.

This finally concludes the proof. \qed



\begin{remark}
  Once we have reduced the proof of Theorem~\ref{thm-high-vanish} to a
  product, the order in which we handle the toric, unipotent and abelian
  variety parts of~$G$ in the proof is essentially dictated by the fact
  that the current versions of Theorem~\ref{thm-tori-vanishing-rel} (the
  relative vanishing theorem for tori) and its corollaries are not
  uniform in terms of the complexity of the input object~$M$.

  However, if the toric part has dimension~$1$, it is not difficult to
  obtain such a statement, and thus to vary the proof.  This is not
  entirely anecdotal, because the choice of order has implications on
  the structure of the sets~$\mcS_i$ in
  Theorem~\ref{thm-high-vanish}.

  We describe the special case of $G=\Gg_m\times \Gg_a$, which will be
  used in Chapter~\ref{sec-product}. We note first that if $M$ is a
  perverse sheaf on~$\Gg_m$ over~$k$, then the vanishing
  $$
  H^i(\Gg_{m,\bar{k}},M_{\chi})= H^i_c(\Gg_{m,\bar{k}},M_{\chi})=0
  $$
  holds for $i\not=0$ and $\chi$ outside of a set
  $\mcS\subset \charg{\Gg}_m$ such that $|\mcS(k_n)|\ll 1$ for all
  $n\geq 1$, where the implied constant depends only on the complexity
  of~$M$ (one reduces to the case of a simple perverse sheaf, and then
  one can apply Lemma~\ref{lm-morel}, for instance). In particular,
  for $\chi\notin \mcS$, the complex~$M_{\chi}$ is a perverse sheaf.

  \begin{theorem}[Stratified vanishing for $\Gg_m\times
    \Gg_a$]\label{thm-vanish-gmga}
  Let~$M$ a perverse sheaf on $G=\Gg_m\times\Gg_a$. There
  exist subsets
  \[
    \mcS_2\subset\mcS_1\subset \mcS_0=\charg{G} \quad \text{and} \quad   \mcT\subset \charg{\Gg}_a
  \]
  such that the following holds:
  \begin{enumth}
  \item For $n\geq 1$, we have $|\mcS_2(k_n)|\ll 1$,
    $|\mcT(k_n)|\ll 1$ and $|\mcS_1(k_n)|\ll |k_n|$.
  \item For~$0\leq i\leq 2$, any~$\chi\in\charg{G}$ such that at least
    one of the cohomology groups
    $$
    \rmH^{i}(G_{\bar k},M_\chi),\quad\quad
    \rmH^{-i}(G_{\bar k},M_\chi),\quad\quad
    \rmH^{i}_c(G_{\bar k},M_\chi),\quad\quad
    \rmH^{-i}_c(G_{\bar k},M_\chi)
    $$
    is non-zero belongs to~$\mcS_i$.
  \item For $\chi\in \mcS_0\bks \mcS_1$, the equality
    $\rmH^0_c(G_{\bar k},M_\chi)=\rmH^0(G_{\bar k},M_\chi)$ holds.
  \item For $\psi\in\charg{\Gg}_a\setminus \mcT$, the set
    $\mcT_{\psi}$ of all $\chi\in\charg{\Gg}_m$ such that
    $\psi\boxtimes\chi\in\mcS_1$ satisfies the bound
    $|\mcT_{\psi}(k_n)|\ll 1$ for all~$n\geq 1$, with an implied
    constant that only depends on~$c(M)$.
  \end{enumth}
\end{theorem}

\begin{proof}
  Let $p\colon G\to \Gg_m$ be the projection. For $\psi$ varying in
  $\charg{\Gg}_a$, the complexes $Rp_!(M_{\psi})$ on~$\Gg_m$ have
  bounded complexity. By Proposition~\ref{prop-strat-unip}, we can
  partition $\charg{\Gg}_a$ in subsets $S_0$ and $S_1$, with $S_1$ of
  character codimension $\geq 1$, such that $Rp_!(M_{\psi})$ is
  perverse if $\psi\in S_0$.

  Let $\psi\in S_0$. Then by the elementary remark before the
  statement, the set $S_{0,\psi}$ of $\chi\in\charg{\Gg}_m$ such that 
  $Rp_!(M_{\psi})_{\chi}$ is not perverse has the property that
  $|S_{0,\psi}(k_n)|\ll 1$ for all~$n\geq 1$, where the implied
  constant depends only on~$c(M)$. The result now follows with
  \begin{align*}
    \mcS_0&=\charg{G},\\
    \mcS_1&=(\charg{\Gg}_m\times S_0)\cup \{\chi\boxtimes\psi\,\mid\,
    \psi\in S_0\text{ and } \chi\in
    S_{\psi,1}\},\\
    \mcT&=S_0,
  \end{align*}
  and $\mcS_2$ the set of characters such that one of
  $$
  \rmH^{2}(G_{\bar k},M_\chi),\quad\quad
  \rmH^{-2}(G_{\bar k},M_\chi),\quad\quad
  \rmH^{2}_c(G_{\bar k},M_\chi),\quad\quad
  \rmH^{-2}_c(G_{\bar k},M_\chi)
  $$
  is non-zero (which satisfies $|\mcS_2(k_n)|\ll 1$ for all $n$ by the
  observation before the theorem; note that $\mcS_2\subset \mcS_1$
  because the existence for given $(\chi,\psi)\in\widehat{G}$ of
  non-zero $\rmH^2$ or $\rmH^2_c$ implies that $Rp_!(M_{\psi})$ is not
  perverse).
\end{proof}

\end{remark}

\section{Proof of Theorem~\ref{th-vanishing-geometric}}

We conclude this chapter with the proof of the partial geometric
version of the vanishing theorem. This section is not needed for the
rest of the results of this book, and in particular may be omitted by
readers interested in equidistribution problems.

We will first prove Theorem~\ref{th-vanishing-geometric} in the case
$G=A\times T$, where~$A$ is an abelian variety and $T$ is a torus. The
case when~$G=A\times U$ will be a simple adaptation of this argument. 

We denote by $p_T\colon T\ra \Spec(k)$ and $p_A\colon A\to \Spec(k)$ the
structural morphisms. We will use the same notation for base-change
morphisms, e.g. for the projection $p_T\colon G\to A$.

Let~$M$ be a perverse sheaf on~$G=A\times T$.  By
Corollary~\ref{cor-gen-van-AV-rel} applied to $p_A\colon G\to T$ and
$M$, up to replacing~$k$ by a finite extension, there is a finite
union~$\mcS'$ of tacs of $A$ such that the complex~$Rp_{A*}(M_\xi)$ is
perverse for all~$\xi\in \what{A} \setminus \mcS'$. By
Theorem~\ref{thm-tori-vanishing-rel} applied to $p_T\colon G\to A$ and
$M$, also up to replacing $k$ by a finite extension, there is a finite
union~$\mcS$ of tacs of $T$ such that
$R p_{T!}(M_\chi)\simeq R p_{T*} (M_\chi)$ for
all~$\chi\in \what{T} \setminus \mcS$, and in particular the
object~$Rp_{T!}(M_{\chi})$ is then perverse.

Let~$\xi\in \what{A} \setminus \mcS'$ and~$\chi\in \what{T} \setminus \mcS$. We will prove that $(\xi,\chi)$
satisfies the properties of Theorem~\ref{th-vanishing-geometric}, thus
concluding the proof for the case~$G=A\times T$.

We start from the isomorphism $Rp_{T*}M_{\chi}=Rp_{T!}M_{\chi}$ of
complexes on~$A$, and tensor with~$\mcL_{\xi}$. By the projection
formula, and the fact that~$\mcL_{\xi}$ is lisse, we obtain an
isomorphism
\[
  R p_{T*}(M_{\xi\chi}) \simeq R p_{T!}(M_{\xi\chi}).
\]
Applying the functor $R p_{A!}=R p_{A*}$ and using functoriality in
the square
\[
  \begin{tikzcd}
    A\times T \arrow{r}{p_A} \arrow{d}{p_T}
    & T  \arrow{d}{p_T} \\
    A \arrow{r}{p_A} &\Spec(k),
   \end{tikzcd}
\]
we obtain an isomorphism
\begin{equation}
  \label{eqn:vant:3}
  R p_{T*}Rp_{A*}(M_{\xi\chi}) \simeq   R p_{T!}R p_{A*}(M_{\xi\chi}).
\end{equation}

The projection formula gives an isomorphism
$Rp_{A*}(M_{\xi\chi})\simeq Rp_{A*}(M_{\xi})\otimes \mcL_{\chi}$, and this
object is perverse by the choice of $\xi$ and the fact that $\mcL_{\chi}$ is lisse.  So from~(\ref{eqn:vant:3}), we deduce that
the cohomology with and without compact support of the perverse sheaf
$Rp_{A*}(M_{\xi\chi})$ on $T$ are isomorphic. By Artin's vanishing
theorem, since $T$ is affine, this implies that only the degree zero
component may be non-zero.  Hence, we have proved that
$\rmH_c^i(G,M_{\xi\chi})=\rmH^i(G,M_{\xi\chi})=0$ for $i\neq 0$, and
that there is an isomorphism
$\rmH_c^0(G, M_{\xi\chi})\simeq\rmH^0(G, M_{\xi\chi})$.

In the case when~$G=A\times U$, we apply the same argument,
\emph{mutatis mutandis}, simply replacing the use of
Theorem~\ref{thm-tori-vanishing-rel} by that of
Proposition~\ref{prop-strat-unip-rel}.

\begin{remark}
  The proof of Theorem~\ref{th-vanishing-geometric} relies in an
  essential way on the fact that we have, for the factors~$A$ and~$T$
  (or~$U$), the isomorphism between cohomology with and without
  compact support. Characters that only satisfy vanishing properties
  do not always satisfy the conclusion, as the following example shows.
  
  Let $E$ be an elliptic curve over~$k$ and let $p\colon E\times\Gm\to E$
  and $q\colon E\times\Gm\to \Gm$ be the two projections. Let
  $i\colon C\to E\times\Gm$ be a closed one-dimensional irreducible
  subvariety of $E\times\Gm$, with dominant projections to both $E$
  and $\Gm$. Let $M=i_{!*}\Qlb[1]$ be the perverse sheaf on
  $E\times\Gm$ which is the intermediate extension of the shifted
  constant sheaf on~$C$.
 
  We claim that both $p_!M$ and $q_!M$ are perverse. Indeed, we have
  $q_!M=q_{\vert C!}i_{!*}\Qlb[1]$, which is perverse since
  $q_{\vert C}\colon C\to \Gg_m$ is a finite morphism, and hence is
  $t$-exact by~\cite[Cor.\,2.2.6]{BBD-pervers}. Moreover, 
  $p_!M=p_{\vert C!}i_{!*}\Qlb[1]$ and $p_{\vert C}$ can be factored
  as a finite morphism followed by an affine open immersion, both of
  which are $t$-exact by~\cite[Cor.\,4.1.3]{BBD-pervers}. It follows
  that the trivial characters on $\Gm$ and $E$ are in the generic sets
  for the relative vanishing Theorem~\ref{thm-tori-vanishing-rel}
  applied to $p$ and for Corollary~\ref{cor-gen-van-AV-rel} applied to
  $q$. However,
  $\rmH^1_c((E\times\Gm)_{\bar{k}},M)\simeq
  \rmH^2_c(C_{\bar{k}},\Qlb)\neq 0$, and hence the trivial character is
  \emph{not} in the generic set given by Theorem \ref{thm-gen-vanish}
  applied to $E\times \Gm$ and $M$.

\end{remark}


\chapter{Tannakian categories of perverse sheaves}\label{sec-tannakian}

\section{Introduction}

Throughout this chapter, $k$ denotes a finite field and~$\bar{k}$ an
algebraic closure of $k$. We denote by $\ell$ a prime number different
from the characteristic of~$k$. All complexes we consider are
$\ell$-adic complexes.

Let~$G$ be a connected commutative algebraic group over~$k$.

Let~$M$ be a perverse sheaf on~$G$.  We wish to define a ``symmetry
group'' that governs the statistical behavior of the arithmetic Fourier
transform
$$
S(M,\chi)=\sum_{x\in G(k_n)}\chi(x)t_M(x;k_n)
$$
for $\chi\in\charg{G}(k_n)$. The fundamental mechanism for this
is that the symmetry group~$\Gg$ should come with a faithful linear
representation $\Gg\subset \GL_r$ for some $r\geq 0$, and to almost
all characters~$\chi$ there should be assigned an element (or conjugacy class)
$\Fr_{\chi}\in \Gg$ such that $S(M,\chi)$ is the trace
of~$\Fr_{\chi}$.

The idea behind the construction of the group $\Gg$ (following
Katz~\cite{mellin}) is based on the fact that we have a ``geometric''
control on the algebra structure on the space of arithmetic Fourier
transforms through the link with convolution: for two objects~$M_1$
and~$M_2$ on~$G$, we have
$$
S(M_1,\chi)S(M_2,\chi)=\sum_{x\in\charg{G}(k_n)}\chi(x)(t_{M_1}*t_{M_2})(x;k_n),
$$
where
$$
(t_{M_1}*t_{M_2})(x;k_n)=\sum_{y\in
  G(k_n)}t_{M_1}(y;k_n)t_{M_2}(y^{-1}x;k_n),
$$
for $x\in G(k_n)$, is the convolution product in the classical sense of
Fourier analysis on $G(k_n)$.

It is fundamental that by the proper base change theorem and the trace
formula, we can view this function as a trace function, namely
$$
(t_{M_1}*t_{M_2})(x;k_n)=t_{M_1*_! M_2}(x;k_n),
$$
where $M_1 *_! M_2$ is the convolution with compact support
(Section~\ref{sec-convolution}).

This geometric interpretation suggests to use the convolution as
``tensor operation'' to define a \emph{tannakian category}, which would
be equivalent to the category of representations of the desired symmetry
group.

In essence, this is what we will do. However, there are some significant
issues to handle:

\begin{itemize}
\item The first one, already present in the work of Katz for~$\Gg_m$,
  has to do with the fact that convolution with compact support does not
  always preserve perverse sheaves (for instance, if $G$ has
  dimension~$d$, then the convolution
  $\mcL_{\chi_1}[d]*_!  \mcL_{\chi_2}[d]$ is not perverse) or duality
  (because duality transforms the convolution $*_!$ into the convolution
  $*_*$, which is different in general).

  We can solve this first problem using a suitable quotient category
  where the two geometric convolution products turn out to coincide
  (this idea goes back to Gabber and Loeser and was also used by Katz).
  
\item A related issue is that weights do not always behave well under
  convolution, in the case of affine groups at least. Since weights
  dictate the size of the sums $S(M,\chi)$, this is a crucial issue for
  our intended applications. This is again related to the difference
  between the two geometric convolutions, each of which leads in
  practice to inequalities in one direction for the weights.
  
\item Finally there is a major new difficulty in comparison with the
  work of Katz.  The link between the abstract tannakian ideas and the
  arithmetic Fourier transform is that for a
  character~$\chi\in\charg{G}(k_n)$, the formula
  $$
  S(M,\chi)=\Tr(\Fr_{k_n}\mid H^0_c(G_{\bar{k}},M_{\chi}))
  $$
  should hold. This is in fact (by the generic vanishing theorem) only
  true in general for a generic set of~$\chi$ where the contributions of
  $H^i_c$ in the trace formula vanish for~$i\not=0$. But we also want
  ``higher-order'' versions of this formula to hold, namely for instance
  $$
  S(M,\chi)^2=\Tr(\Fr_{k_n}\mid H^0_c(G_{\bar{k}},M_{\chi}*_! M_{\chi})),
  $$
  and so on for further powers (intuitively, this is because
  understanding the limits of averages of such expressions is necessary
  to apply the Weyl equidistribution criterion, as we will do in the
  next chapter). This amounts roughly to requesting that
  $M\mapsto H^0_c(G_{\bar{k}},M_{\chi})$ should be compatible with
  convolution and so should (roughly) the generic vanishing theorem.
  
  Thus we need to distinguish various types of characters depending on
  their behavior with respect to operations of this type.
\end{itemize}




\section{Categories of objects defined over finite fields}

We denote by~$\Dd(G)$
and~$\Pp(G)$\nomenclature[$D$]{$\Dd(G)$}{subcategory of $\Der(G_{\bar{k}})$ of objects defined over a finite field}
\nomenclature[$P$]{$\Pp(G)$}{subcategory of $\Perv(G_{\bar{k}})$ of objects defined over a finite field}%
the full subcategories of
$\Der(G_{\bar{k}})$ and $\Perv(G_{\bar{k}})$ respectively whose objects
are defined over some finite extension of the base field~$k$.  These
categories are abelian categories, stable by shifts and duality. Moreover, the
perverse cohomology sheaves of an object of~$\Dd(G)$ belong to~$\Pp(G)$.

We recall from Section~\ref{sec-convolution} the definition and
properties of the two convolution bifunctors $(M,N)\mapsto M *_* N$ and
$(M,N)\mapsto M *_! N$ for objects $M$ and $N$ of $\Der(G)$ or 
$\Der(G_{\bar{k}})$. These are compatible with base
change, so that the convolutions on $G_{\bar{k}}$ preserve the
category~$\Dd(G)$. In addition, the functor $M\mapsto M^{\vee}$ also
induces a functor on $\Dd(G)$ and $\Pp(G)$.

\section{Weakly unramified characters}

\begin{definition}[Weakly unramified characters]\label{def:unramified}
  Let $M$ be an object of $\Pp(G)$. A character $\chi \in \charg{G}$ is
  said to be \emph{weakly unramified}\index{weakly unramified character}
  for~$M$ if the following holds:
    \begin{gather*}
    \rmH^i(G_{\bar k},M_\chi)=\rmH^i_c(G_{\bar k},M_\chi)=0 \quad\text{
      for all $i\not=0$,}
    \\
    \rmH^0_c(G_{\bar k},M_\chi)=\rmH^0(G_{\bar k},M_\chi). 
  \end{gather*}
  We denote by $\wunram{M}$\nomenclature[$X$]{$\wunram{M}$}{set of weakly
    unramified characters for~$M$} the set of weakly unramified
  characters for~$M$.
\end{definition}

\begin{remark}
 (1) The terminology is suggested by analogy with the case of the additive
  group, in which the characters for which generic vanishing holds
  correspond to points at which the Fourier transform is lisse.
  However, we will see that the generic vanishing condition is not in
  general strong enough to obtain the properties we seek (namely, that
  the assignment $M\mapsto \rmH^0_c(G_{\bar{k}},M_{\chi})$ defines a
  fiber functor on a suitable tannakian category of perverse sheaves on
  $G_{\bar{k}}$), see Remark \ref{rem-unram-ex3} and Example \ref{ex-non-unram}. We will introduce unramified characters in
  Definition~\ref{def:unramified-fibfunct}, as well as the variant of
  Frobenius-unramified characters in Definition~\ref{def-funram}.
  \par
  (2) Since~$G$ is not proper in general, the condition on the
  cohomology groups with support is not implied by the one on the
  cohomology groups without support.
\end{remark}

With this definition, we can reformulate the Stratified Generic
Vanishing Theorem~\ref{thm-gen-vanish} as follows:

\begin{theorem}\label{th-generic-unramified}
  The subset $\wunram{M} \subset \charg{G}$ of weakly unramified
  characters for an object~$M$ of~$\Pp(G)$ is generic.
\end{theorem}

\section{Negligible objects}

In general, none of the two convolution bifunctors on the derived category preserves the subcategory of
perverse sheaves. As first observed in the case of tori by Gabber and
Loeser \cite{GL_faisc-perv}, there is however a suitable quotient of the category $\Pp(G)$ on
which both convolution functors induce the same bifunctor.

\begin{definition}
\label{def-neg}
An object $M$ of $\Pp(G)$ is said to be \emph{negligible}\index{negligible object} if the set of
characters~$\chi\in\charg{G}$ satisfying $H^0(G_{\bar{k}},M_{\chi})=0$ is
generic. An object $N$ of $\Dd(G)$ is said to be \emph{negligible} if all its perverse cohomology objects $\pH^i(N)$ are negligible.

We denote by $\NegP(G)$\nomenclature[$N$]{$\NegP(G)$}{negligible objects of
  $\mathbf{P}(G)$} and $\NegD(G)$\nomenclature[$N$]{$\NegD(G)$}{negligible
  complexes of~$\Dd(G)$} the full
subcategories of $\mathbf{P}(G)$ and $\Dd(G)$ respectively consisting of
negligible objects.

We denote by $K_{\mathrm{neg}}(G)$ the subgroup of the Grothendieck group $K(G)$
generated by classes of negligible perverse sheaves, or equivalently by
classes of negligible objects.
\nomenclature{$K_{\mathrm{neg}}(G)$}{subgroup of $K(G)$ generated by negligible objects}
\end{definition}

Given an object $M$ of $\Pp(G)$,
set\nomenclature[$N$]{$\nunram{M}$}{characters $\chi$ with $\rmH^*(G_{\bar{k}}, M_\chi)=\rmH_c^*(G_{\bar{k}},M_{\chi})=0$}
\[
  \nunram{M}=\{\chi \in \charg{G} \mid \rmH^i(G_{\bar{k}},
  M_\chi)=\rmH_c^i(G_{\bar{k}},M_{\chi})=0 \text{ for all } i \}.
\]

Using Theorem~\ref{thm-gen-vanish}, we see that~$M$ is negligible if
and only if $\nunram{M}$ is a generic subset of~$\charg{G}$. For
$M\in \NegD(G)$, we set
\[
\nunram{M}=\bigcup_i\nunram{\pH^{i}(M)}.
\]

It follows from the definition that, for each negligible perverse sheaf
$M$ (\resp object of $\NegD(G)$), the perverse sheaf $M^\vee$ is also
negligible (\resp the complex $M^{\vee}$ is negligible).


\begin{example}\label{ex-negligible}
  Any character sheaf~$\mcL_{\chi}$ on~$G$ is negligible. More
  generally, let $f\colon G\to H$ be a surjective morphism of
  algebraic groups such that the dimension $d$ of the kernel $\ker(f)$
  is positive. Let~$\eta\in\charg{G}$ and let $N$ be any object
  of~$\Der(H)$. We claim that the object $M=(f^*N)_{\eta}$ is
  negligible.

  Indeed, let~$i\in\Zz$. We can factor $f=f_1\circ f_2$, where~$f_2$ is
  smooth of relative dimension~$d$ and~$f_1$ is radicial. Then
  $f^*_2[d]$ is t-exact (see~\cite[\S\,4.2.4]{BBD-pervers}), and so is
  tensoring by~$\mcL_{\eta}$ (Lemma~\ref{lem-characters-t-exact}), so
  there is a canonical isomorphism
  $$
  \pH^i((f^*N)_{\eta})\simeq f_2^*(\pH^{i-d}(f_1^*(N)))_{\eta}.
  $$
  For~$\chi\in\charg{G}$, the projection formula leads to canonical
  isomorphisms
  $$
  H^*(G_{\bar{k}},M_{\chi})\simeq H^*(G_{\bar{k}},
  f_2^*(\pH^{i-d}(f_1^*(N)))\otimes\mcL_{\eta\chi})
  \simeq H^*(H_{\bar{k}},\pH^{i-d}(f_1^*(N))\otimes Rf_{2!}\mcL_{\eta\chi}).
  $$

  \par
  The complex $Rf_{2!}\mcL_{\eta\chi}$ is zero if the restriction of
  $\eta\chi$ to the subgroup $\ker(f_2)^{\circ}$ is not the trivial
  character (see Lemma~\ref{lm-image-character}).  Since this condition
  defines a generic set of characters~$\chi$, we deduce that~$\pH^i(M)$
  is negligible, and the result follows.
\end{example}

\begin{remark}
  Intuitively, to say that $M$ is negligible means that the arithmetic
  Fourier transform of~$M$ (see
  Section~\ref{sec:discreteFouriertransform}) satisfies $S(M,\chi)=0$
  for $\chi$ in a generic subset of~$\charg{G}$. To illustrate this
  concrete aspect, we show how it explains the previous example. Thus
  consider $M=(f^*N)_{\eta}$, with notation as above for some
  $\eta\in\charg{G}(k)$. Let $\chi\in\charg{G}(k_n)$; the corresponding
  value of the Fourier transform is
  \begin{align*}
    S(M,\chi)=\sum_{x\in G(k_n)}\chi(x)t_M(x;k_n)&= \sum_{x\in
      G(k_n)}\chi(x)(\eta\circ N_{k_n/k})(x) t_N(f(x);k_n)
    \\
    &=\sum_{y\in H(k_n)}t_N(y;k_n)\sum_{\substack{x\in G(k_n)\\f(x)=y}}
    \chi(\eta\circ N_{k_n/k})(x),
  \end{align*}
  and the inner sum is either empty or a sum of a character over the
  $k_n$-points of a coset of the kernel of~$f$, which vanishes unless
  $\chi=(\eta\circ N_{k_n/k})^{-1}$ on the kernel of~$f$.
  \par
  In some cases, one can show that, conversely, all simple negligible
  perverse sheaves are of the form $(f^*N)_{\eta}$ for some quotient
  morphism~$f$ with kernel of dimension at least~$1$. This is for
  instance the case for abelian varieties, by a result of
  Weissauer~\cite[Lemma\,6,\,Th.\,3]{weissauer_vanishing_2016} (see also
  Remark~\ref{rm-negligible}) and we will prove later that this is also
  the case for $\Gg_a\times\Gg_m$ (see Section~\ref{ssec-neg-one}).
  \par
  This structural property is however not always true. For instance, if
  $G$ is a unipotent group of dimension at least~$2$ (e.g., $G=\Gg_a^d$
  with $d\geq 2$), with Serre dual~$G^{\vee}$, then we can take any
  object $N\in \Der(G^{\vee})$ whose support~$S$ has codimension at
  least~$1$, and the inverse Fourier transform~$M$ of~$N$ will be a
  negligible object on~$G$. If~$S$ is not a translate of a subgroup
  of~$G$, then the object~$M$ is not obtained by pullback from any
  quotient of~$G$. (In the terminology of~\cite[\S\,4]{fouvry-katz}, in
  the case of~$\Gg_a^d$, such objects are said to have
  $A$-number\index{$A$-number} equal to~$0$, and they play a delicate
  role in certain analytic applications.)
\end{remark}


We recall that a full subcategory $S$ of an abelian category $C$ is said
to be a \emph{Serre subcategory}\index{Serre subcategory} if it is not empty, stable by extension
and by subobject and quotient. A strictly full triangulated subcategory
$S$ of a triangulated category $C$ is said to be
\emph{thick}\index{thick subcategory} if, for
any morphism $f\colon X\to Y$ in~$C$ which factors through an object
of~$S$, and which appears in a distinguished triangle
$$
X\fleche{f} Y\to Z
$$
with $Z$ object of~$S$, the objects~$X$ and~$Y$ are in~$S$.


\begin{lemma}
\label{lem-serre}
The category $\NegP(G)$ is a Serre subcategory of $\Pp(G)$, and
$\NegD(G)$ is a thick triangulated subcategory of $\Dd(G)$.
\end{lemma}

\begin{proof}
  Fix an exact sequence
  $X\to Y\to Z$ in $\Pp(G)$ such that $X$ and~$Z$ are objects
  of~$\NegP(G)$. By Theorem~\ref{th-generic-unramified}, there is a
  generic set of characters $\chi\in\charg{G}$ that are weakly
  unramified for $X$, $Y$, and $Z$. From the long exact sequence in
  cohomology, we find that for any such~$\chi$, the vanishing
  $\rmH^i(G_{\bar k}, Y_{\chi})=\rmH^i_c(G_{\bar k}, Y_{\chi})=0$ holds
  for all $i$, and hence $Y$ is negligible. The first statement follows
  easily. An argument of Gabber--Loeser
  (see~\cite[Prop.\,3.6.1(i)]{GL_faisc-perv}) then implies that
  $\NegD(G)$ is a thick triangulated subcategory of $\Dd(G)$.
\end{proof}

\begin{lemma}
\label{lem-mildeconv-quot}
For all objects $M$ and $N$ of $\Dd(G)$, the following properties hold:
\begin{enumth}
\item \label{lem-mildeconv-quot1} The cone of the canonical morphism
  $M*_!N\to M*_*N$ lies in $\NegD(G)$. 
\item \label{lem-mildeconv-quot2} If $M$ belongs to $\NegD(G)$, then so
  do $M*_!N$ and $M*_*N$ for each object $N$.
\item \label{lem-mildeconv-quot3} If $M$ and $N$ belong to $\Pp(G)$, then
  $\pH^i(M*_!N)$ and $\pH^i(M*_*N)$ lie in $\NegP(G)$ for each non-zero integer $i$.
\end{enumth}
\end{lemma}

We omit the proof, which is the same as that of
\cite[Lem.\,4.3]{kramer_perverse_2014}.

\section{Tannakian categories}

By results of Gabriel~\cite{gab_cat_abeliennes} for abelian categories
and Verdier (see the treatment in the book~\cite{nee_trig_cat} of
Neeman) for triangulated categories, we can define the quotient of an
abelian or triangulated category by a Serre or thick subcategory. This
allows us to make the following definition.

\begin{definition}[Convolution categories]
  The \emph{convolution category}\index{convolution category} of~$G$,
  denoted~$\Ddb(G)$, \nomenclature[$D$]{$\Ddb(G)$}{convolution category}
  is the
  quotient category of $\Dd(G)$ by $\NegD(G)$; it is a triangulated
  category.
  \par
  The \emph{perverse convolution category}\index{perverse convolution
    category} of~$G$, denoted~$\Ppb(G)$, is
  \nomenclature[$P$]{$\Ppb(G)$}{perverse convolution category}
  the quotient abelian category of~$\Pp(G)$ by~$\NegP(G)$.
\end{definition}

Those two constructions are compatible, in the sense that the
t-structure on $\Dd(G)$ induces a t-structure on $\Ddb(G)$ whose heart
is the category~$\Ppb(G)$ (see~\cite[Prop.\,3.6.1]{GL_faisc-perv}).

Since the functor $N\mapsto N^{\vee}$ preserves negligible objects,
it induces a functor on $\Ppb(G)$ (\resp on $\Ddb(G)$), which is still
an involution.

\begin{proposition}\label{pr-convol-pp}
  With notation as above, the following properties hold:
  \begin{enumth}
  \item The convolution products $*_!$ and $*_*$ induce bifunctors on
    $\Ddb(G)\times \Ddb(G)$.
  \item The canonical forget support morphisms $M *_! N\to M *_* N$
    induce isomorphisms in $\Ddb(G)$, and define by passing to the
    quotient a convolution bifunctor denoted
    \[
      * : \Ddb(G)\times \Ddb(G)\to \Ddb(G).
    \]
  \item The subcategory $\Ppb(G)$ of $\Ddb(G)$ is stable under the convolution $*$.
  \item The categories $\Ddb(G)$ and $\Ppb(G)$, endowed with the bifunctor
    $*$, are symmetric $\Qlb$-linear monoidal categories\index{monoidal
      category} with unit object\index{unit object}
    $\un$ the image of the skyscraper sheaf at the neutral element of $G$.
  \end{enumth}
\end{proposition}

\begin{proof}
  The fact that $*_!$ and $*_*$ induce functors on
  $\Ddb(G)\times \Ddb(G)$ follows from Lemma~\ref{lem-mildeconv-quot}
  \ref{lem-mildeconv-quot2}. That they agree is
  Lemma~\ref{lem-mildeconv-quot} \ref{lem-mildeconv-quot1}. The
  stability of $\Ppb(G)$ under $*$ is Lemma \ref{lem-mildeconv-quot}
  \ref{lem-mildeconv-quot3}. The fact that we obtain symmetric
  $\Qlb$-linear monoidal categories is now clear. The last assertion
  follows from the canonical isomorphisms
  $\un *_! M\simeq \un *_* M\simeq M$ which exist for any complex~$M$.
\end{proof}

It is also very useful that there exists a natural subcategory
of~$\Pp(G)$ that is equivalent to the perverse convolution category.

\begin{definition} The \emph{internal convolution
    category}\index{internal convolution category} of~$G$ is
  the full subcategory $\Ppint(G)$ of the category~$\Pp(G)$ whose
  objects are perverse sheaves that have no subobject or quotient in
  $\NegP(G)$.
\end{definition}

\begin{proposition}\label{pr-equiv-int}
  The localization functor $\Pp(G) \to \Ppb(G)$ restricts to an
  equivalence of categories
  \[
    \Ppint(G)\longrightarrow \Ppb(G),
  \]
  hence the convolution product bifunctor $*$ on~$\Ppb(G)$ induces a
  convolution bifunctor $*_{\intt}$  \nomenclature{$M*_{\intt} N$}{internal convolution} on $\Ppint(G)$.
\end{proposition}

\begin{proof}
  The argument is the same as that of Gabber and
  Loeser~\cite[Déf.-Prop.\,3.7.2]{GL_faisc-perv}.
\end{proof}

The convolution product on~$\Ppint(G)$ will sometimes be called the
\emph{internal} or \emph{middle} convolution.\index{internal convolution}


\begin{remark}\label{rem-Mint}
  One can give a more explicit form of the equivalence of categories
  above, and of the internal convolution.

  First, Gabber and Loeser (\loccit) give an explicit quasi-inverse
  functor \hbox{$M \mapsto M_{\intt}$}
  \nomenclature{$M_{\intt}$}{quasi-inverse of $\Ppint(G)\to \Ppb(G)$} to
  the equivalence of categories $ \Ppint(G)\to \Ppb(G)$. Namely, let $M$
  be an object of~$\Pp(G)$. Let $M_t$ be the largest subobject of $M$
  that belongs to $\NegP(G)$ and let $M^t$ be the smallest subobject of
  $M$ such that $M/M^t$ belongs to~$\NegP(G)$. Define
  $M_{\intt}=M^t/(M^t\cap M_t)$. Then we have canonical
  isomorphisms
  \[
    M_{\intt}\simeq (M^t+M_t)/M_t,
  \]
  and the assignment $M\mapsto M_{\intt}$ is a functor which factors
  through $\Ppb(G)$ and induces a quasi-inverse of the localization
  functor.
  \nomenclature{$M^t$}{smallest subobject such that $M/M^t$ is negligible}
  \nomenclature{$M_t$}{largest subobject such that $M_t$ is negligible}
  
  In particular, this implies that if $M$ is a semisimple object of
  $\Ppb(G)$, then $M_{\intt}$ is the sum of all the simple
  constituents of $M$ that are not in $\NegP(G)$.

  Second, it follows from the argument
  in~\cite[Déf.-Prop.\,3.7.3]{GL_faisc-perv} that for $M$ and $N$
  in~$\Ppint(G)$, there are canonical isomorphisms
  \begin{equation}\label{eq-comparison-convolutions}
    M *_{\intt} N \to \pH^0(M *_! N)_{\intt} \to \pH^0(M *_* N)_{\intt}.
  \end{equation}
\end{remark}

We continue using the notation of the remark.

\begin{lemma}
  \label{lem-H-M-Mint}
  Let $M$ be an object of $\Pp(G)$. For all
  $\nunram{M/M^t}\cup \nunram{M^t\cap M_t}$, the natural morphisms
  $M^t\fleche{i} M$ and $M^t\fleche{p} M^t/(M^t\cap M_t)=M_{\intt}$
  induce isomorphisms
  \[
    H^{*}(G_{\bar{k}},(M^t)_\chi)\fleche{i_*} H^{*i}(G_{\bar{k}},M_\chi)
  \]
  and
  \[
    H^{*}(G_{\bar{k}},(M^t)_\chi) \fleche{p_*}
    H^{*}(G_{\bar{k}},(M_{\intt})_{\chi}).
  \]

  In particular, the set of characters $\chi\in\what{G}$ such that the
  cohomology groups $H^{*}(G_{\bar{k}},(M_{\intt})_\chi)$ and
  $H^{*}(G_{\bar{k}},M_\chi)$ are isomorphic is generic.
\end{lemma}

\begin{proof}
  The first statement follows immediately from the exact sequences
  \[
    M^t\to M \to M/M^t,\quad\quad 
    M^t\cap M_t \to M^t \to M^t/(M^t\cap M_t)=M_{\intt},
  \]
  and the assumption on~$\chi$. Since $\nunram{M/M^t}$ and
  $\nunram{M^t\cap M_t}$ are both generic, the final assertion follows.
\end{proof}

Recall from Section~\ref{sec-convolution} that for $M\in \Pp(G)$, the identity morphism $\id_M \colon M \to M$ defines
evaluation and coevaluation maps
\[
  \ev \colon M*_!M^\vee \to \un\quad\quad\text{ and }\quad\quad \coev
  \colon \un \to M^\vee*_*M.
\]
They correspond to maps in $\Ppb(G)$ which we denote in the same
way.

\begin{proposition}
\label{prop-rigid-cat}
The monoidal category $\Ppb(G)$ is rigid.\index{rigid monoidal category}  That is, for each object~$M$ of~$\Ppb(G)$, the morphisms
\begin{align*}
  M&\simeq M*\un \xrightarrow{\id_M*\coev} M*M^\vee
  *M\xrightarrow{\ev*\id_M}\un *M\simeq M \\
    M^\vee&\simeq \un *M^\vee \xrightarrow{\coev *\id_M} M^\vee * M * M^\vee
  \xrightarrow{\id_M *\ev} M^\vee * \un \simeq M^\vee
\end{align*} are the identity on $M$ and on $M^\vee$ respectively. 
\end{proposition}

\begin{proof}
  The argument is the same as that of Krämer
  in~\cite[Th.\,5.2]{kramer_perverse_2014}.
\end{proof}

For any object~$M$ of~$\Ppb(G)$ (\resp of~$\Ppint(G)$), we denote by
$\braket{M}$\nomenclature{$\braket{M}$}{subcategory tensor-generated by~$M$} the subcategory of $\Ppb(G)$ (\resp of $\Ppint(G)$) which
is tensor-generated by~$M$, i.e., the full subcategory whose objects are
the subquotients of all convolution powers of $M\oplus M^\vee$.

Our next goal is to prove the following crucial result:

\begin{theorem}
  \label{thm-cat_neut_tan}
  The categories $\Ppb(G)$ and $\Ppint(G)$ are neutral tannakian
  categories.
  \par
  In particular, for any object $M$ of~$\Ppint(G)$ or of~$\Pp(G)$, the
  category $\braket{M}$ is a neutral tannakian category over~$\Qlb$.\index{neutral tannakian category}
\end{theorem}

Recall that this means that there exists a fiber functor, namely a
faithful exact tensor functor from $\Ppb(G)$ to the category
$\mathrm{Vect}_\Qlb$ of finite dimensional $\Qlb$-vector spaces.
  


We begin the proof with an auxiliary result.  Recall that the trace
$\tr(f)\in\Qlb=\End(\un)$ of an endomorphism $f$ of $M\in \Ppint(G)$ is
defined as the composition
\[
  \un\xrightarrow{\coev} M*_{\intt}
  M^\vee\xrightarrow{f*_{\intt} \id_{M^\vee}} M*_{\intt}
  M^\vee\xrightarrow{\ev}\un.
\]

The dimension of $M\in \Ppint(G)$ is then intrinsically defined as
$\dim(M)=\tr(\id_M)$. It is, a priori, an element of~$\Qlb$.

\begin{proposition}
\label{prop-dimPerv=dimVect}
Let $M$ be an object of~$\Ppint(G)$. Then there is a generic set of characters $\chi\in\charg{G}$ such that
the following equality holds: 
\begin{equation}\label{eq-dim-formula}
  \dim \rmH^0(G_{\bar k},M_\chi)=\dim(M).
\end{equation}
\par
In particular, $\dim(M)$ is a non-negative integer, and there exists a
generic set of characters $\chi$ such that the dimension of
$H^0(G_{\bar{k}},M_{\chi})$ is independent of $\chi$.
\end{proposition}

\begin{proof}
  We need to determine the morphism
  \[
    \un\xrightarrow{\coev} M*_{\intt}
    M^\vee\xrightarrow{\ev}\un.
  \]
  
  Twisting by~$\chi$ and taking cohomology, the sequence above induces a
  sequence 
  \[
    \Qlb\to H^{*}(G_{\bar{k}},(M*_{\intt} M^{\vee})_\chi)\to \Qlb.
  \]
  
  Let $P=\pH^0(M *_* M^{\vee})$.  By Lemma \ref{lem-mildeconv-quot}, the
  objects $\pH^i(M *_* M^{\vee})$ are negligible for all~$i\not=0$, so
  that for $\chi$ in a generic set of characters, we have isomorphisms
  \[
    H^{*}(G_{\bar{k}},(M*_* M^{\vee})_\chi)\simeq
    H^{*}(G_{\bar{k}},P_\chi).
  \]
  
  By Lemma~\ref{lem-H-M-Mint}, for $\chi$ in another generic set, we
  also have isomorphisms
  \[
    H^{*}(G_{\bar{k}},P_\chi)\simeq H^{*}(G_{\bar{k}},(P_{\intt})_\chi).
  \]

  By~(\ref{eq-comparison-convolutions}), the object $P_{\intt}$ is
  isomorphic to $M*_{\intt} M^{\vee}$, so these isomorphisms combine to
  imply that that, for $\chi$ generic, we have an isomorphism
  \[
    H^{*}(G_{\bar{k}},(M*_{\intt} M^{\vee})_\chi)\to
    H^{*}(G_{\bar{k}},(M*_* M^{\vee})_\chi).
  \]
  
  By Lemma~\ref{lem-prop-conv}, there are canonical isomorphisms
  \[
    H^{*}(G_{\bar{k}},(M*_* M^{\vee})_\chi)\simeq
    H^{*}(G_{\bar{k}},M_{\chi})\otimes
    H^{*}(G_{\bar{k}},(M_{\chi})^{\vee})
  \]
  for all~$\chi$.  Finally, if $\chi$ is also in $\wunram{M}$, then
  $H^{*}(G_{\bar{k}},M_{\chi})=H^0(G_{\bar{k}},M_{\chi})$ and
  $H^{*}(G_{\bar{k}},M^\vee_{\chi})=H^0(G_{\bar{k}},M_{\chi})^\vee$, and
  therefore there exists a generic set of characters for which the
  sequence above becomes
  \[
    \Qlb\to \End(H^0(G_{\bar{k}},M_{\chi}))\to \Qlb.
  \]
  
  Since the evaluation and coevaluation maps are sent to evaluation and
  covevaluation maps in vector spaces (see the proof of
  \cite[Th.\,5.2]{kramer_perverse_2014}), this composition is the
  multiplication by the dimension of $H^0(G_{\bar{k}},M_{\chi})$, which
  is therefore equal to the dimension of~$M$ in $\Ppint(G)$.
\end{proof}

\begin{proof}[Proof of Theorem~\ref{thm-cat_neut_tan}]
  By Proposition~\ref{prop-rigid-cat}, the equivalent categories
  $\Ppb(G)$ and~$\Ppint(G)$ are $\Qlb$-linear rigid tensor symmetric
  categories. Since the unit~$\un$ is (the image of) a skyscraper sheaf,
  we have $\mathrm{End}(\un)\simeq\Qlb$. 
  \par
  Proposition~\ref{prop-dimPerv=dimVect} and
  Theorem~\ref{th-generic-unramified} imply that the dimension~$\dim(M)$
  of every object~$M$ of $\Ppb(G)$ is a non-negative integer.  By a
  theorem of Deligne~\cite[Th.\,7.1]{del_cat_tan}, it follows that the
  category $\Ppb(G)$ is a tannakian category. A further theorem of
  Deligne (see the proof by Coulembier in~\cite[Th.6.4.1]{coulembier})
  implies that it is indeed neutral (i.e., there exists a fiber functor
  defined over~$\Qlb$).
\end{proof}

\begin{remark}\label{rm-chi}
  (1) In Example~\ref{ex-non-unram}, we will give examples to show that
  there may exist weakly unramified characters for which
  formula~(\ref{eq-dim-formula}) does not hold.
  \par
  (2) In this book, we will  exclusively consider from now on the
  categories $\braket{M}$ generated by a single object. A simpler proof
  that these are neutral tannakian categories is then provided by
  combining~\cite[Th.\,7.1]{del_cat_tan}
  with~\cite[Cor.\,6.20]{del_cat_tan}.
\end{remark}




\begin{corollary}\label{cor-def-group}
  Let $M$ be an object of $\Ppint(G)$. There exists an affine algebraic
  group $\Gg$ over~$\Qlb$ such that the category $\braket{M}$ is
  equivalent to the category~$\mathrm{Rep}_{\Qlb}(\Gg)$
  \nomenclature[$R$]{$\mathrm{Rep}_{\Qlb}(\Gg)$}{category of representations
    of~$\Gg$} of finite-dimensional $\Qlb$-representations
  of~$\Gg$. If~$M$ is semisimple, then the group~$\Gg$ is reductive and
  the category $\braket{M}$ is semisimple.
\end{corollary}

\begin{proof}
  The first part follows from the tannakian reconstruction
  theorem~\cite[Th.\,2.11]{deligne-milneII}.\index{tannakian
    reconstruction theorem} If~$M$ is semisimple then
  since the category of representations of~$\Gg$ is equivalent to the
  category $\braket{M}$ generated by the semisimple object~$M$, it
  follows, e.g., from~\cite[Th.\,22.42]{milne-groups} that the group
  $\Gg$ is reductive, and that every object $N\in \braket{M}$ is
  semisimple.
\end{proof}

\begin{definition}
  For any object~$M$ of~$\Ppint(G)$ or of~$\Ppb(G)$, we denote by
  $\ggeo{M}$ the affine algebraic group over~$\Qlb$ given by the
  corollary, and we say that it is the \emph{geometric tannakian group}
  of the object~$M$. \nomenclature[$G$]{$\ggeo{M}$}{geometric tannakian
    group of~$M$}\index{geometric tannakian group}
\end{definition}

\begin{example}\label{ex-gm-geo}
  (1) Let $G=\Gg_m$. A perverse sheaf $N$ on~$\Gg_m$ is negligible if
  and only if it is a successive extension of
  shifted Kummer sheaves\index{Kummer sheaf} $\mcL_{\chi}[1]$ for some characters~$\chi$, and
  it follows that the category $\Ppint(\Gg_m)$ is the same as the
  category $\oldcal{P}$\nomenclature[$P$]{$\mathcal{P}$}{category of perverse
    sheaves on~$\Gg_m$} of Katz~\cite[Ch.\,2]{mellin}  (see
  also Section~\ref{sec-category-p}).
  \par
  (2) Let $G=\Gg_a$.  Fix an additive character $\psi$ of~$k$. By the
  proper base change theorem, a perverse sheaf $N$ on~$\Gg_a$ is
  negligible if and only if its Fourier transform $\ft_{\psi}(N)$ is
  punctual, which means that $N$ is a finite direct sum of
  Artin--Schreier
  sheaves~$\sheaf{L}_{\psi(yx)}[1]$\index{Artin--Schreier sheaf} for
  some $y\in\Gg_a$. This implies that the category $\Ppint(\Gg_a)$
  coincides with the category of perverse sheaves on~$\Gg_{a}$ with
  ``property $\oldcal{P}$'', as defined by
  Katz~\cite[(2.6.2)]{katz-rls} (this follows by combining
  Cor.\,2.6.14, Cor.\,2.6.15 and Lemma\,2.6.13 of~\cite{katz-rls}; see
  Remark 2.10.4 in \loccit).
\end{example}

\section{Euler--Poincaré characteristic and Grothendieck groups}

Proposition~\ref{prop-dimPerv=dimVect} has some other useful corollaries
which we state now.

\begin{proposition}\label{pr-chi}
  Let $M$ be an object of~$\Der(G)$.
  \begin{enumth}
  \item There exists a generic set~$\mcX\subset\charg{G}$ such that the
    Euler--Poincaré characteristic $\chi(G_{\bar{k}},M_\chi)$ is
    independent of $\chi\in \mcX$.
  \item If $M$ is negligible, then $\chi(G_{\bar{k}},M_{\chi})=0$ for
    all $\chi$ in a generic set of characters. The converse holds if $M$
    is a perverse sheaf.
  \item If~$G$ is a semiabelian variety, then the Euler--Poincaré
    characteristic $\chi(G_{\bar{k}},M_\chi)$ is independent of
    $\chi\in \charg{G}$ and it is non-negative if $M$ is a perverse
    sheaf.
  \end{enumth}
\end{proposition}

\begin{proof}
  The decomposition
  $$
  M=\sum_{i\in\Zz} (-1)^i\ \pH^i(M)
  $$
  in the Grothendieck group $K(G)$, together with
  Lemma~\ref{lem-characters-t-exact}, implies that
  $$
  \chi(G_{\bar{k}},M_{\chi})=\sum_{i\in\Zz} (-1)^i\
  \chi(G_{\bar{k}},\pH^i(M)_{\chi})
  $$
  for all $\chi\in\charg{G}$.  Thus the first statement is an immediate
  consequence of Proposition~\ref{prop-dimPerv=dimVect}, combined wih
  the generic vanishing theorem, applied to each perverse cohomology
  sheaf.
  \par
  If $N$ is a negligible perverse sheaf, then by definition we
  get~$H^*(G_{\bar{k}},N_{\chi})=0$ for a generic set of characters,
  hence also $\chi(G_{\bar{k}},N_{\chi})=0$ for a generic set of
  characters. The previous formula shows that this is also true for any
  complex~$M$.

  Conversely, assume that~$M$ is a perverse sheaf
  and~$\chi(G_{\bar{k}},M_{\chi})=0$ for all~$\chi$ in a generic
  set. Combined with the generic vanishing theorem, this implies that
  $H^*(G_{\bar{k}},M_{\chi})=0$  for~$\chi$ generic, hence~$M$ is
  negligible.
  \par
  If~$G$ is a semiabelian variety, then the Euler--Poincaré
  characteristic $\chi(G_{\bar{k}},M_{\chi})$ is independent of~$\chi$
  by a result of Deligne (see~\cite{illu_charEP}), because all the
  $\chi\in \charg{G}$ are tame. In this case, the tannakian dimension of
  a perverse sheaf on~$G$ is therefore the same as its Euler--Poincaré
  characteristic.
\end{proof}



\begin{corollary}\label{cor-grothendieck}
  A perverse sheaf~$M$ in~$\Pp(G)$ is negligible if and only if its
  class in the Grothendieck group $K(G)$ belongs to the subgroup
  $K_{\mathrm{neg}}(G)$ generated by classes of negligible perverse sheaves.
\end{corollary}

\begin{proof}
  It suffices to prove that a perverse sheaf~$M$ is negligible if the
  class of~$M$ in~$K(G)$ can be expressed as a finite sum
  $$
  M=\sum_{i\in I} \eps_i M_i
  $$
  in~$K(G)$, where~$M_i$ is a negligible perverse sheaf for all~$i\in I$
  and $\eps_i\in\{-1,1\}$. Such a formula implies
  the equality
  $$
  \chi(G_{\bar{k}},M_{\chi})=\sum_{i\in I} \eps_i \chi(G_{\bar{k}},M_{i,\chi})
  $$
  for all $\chi\in\charg{G}$. For a generic set of characters we have
  $\chi(G_{\bar{k}},M_{i,\chi})=0$ for all~$i\in I$, since~$M_i$ is
  negligible by assumption, hence $\chi(G_{\bar{k}},M_{\chi})=0$ for a
  generic set of characters; thus~$M$ is negligible by
  Proposition~\ref{pr-chi}, (2).
\end{proof}

\begin{corollary}\label{cor-tacs-negligible}
  Suppose that~$G$ is a semiabelian variety. Let~$M$ be a negligible
  perverse sheaf on~$G$. The Euler--Poincaré characteristic of~$M$
  is~$0$ and the set of characters~$\chi\in\charg{G}$ such that the
  space $H^0(G_{\bar{k}},M_{\chi})$ is non-zero is contained in a finite
  union of \tacs.
\end{corollary}

\begin{proof}
  The fact that~$\chi(M)=0$ has been stated in
  Proposition~\ref{pr-chi}. By Theorem~\ref{thm-gen-van-AV}, there
  exists a finite family $(S_f)$ of \tacs\ of~$G$ such
  that~$H^i(G_{\bar{k}},M_{\chi})=0$ for all~$i\not=0$ and~$\chi$ not
  belonging to the union~$\mcS$ of these \tacs. For any $\chi$ not
  in~$\mcS$, we then deduce by loc. cit. that
  $$
  \dim H^0(G_{\bar{k}},M_{\chi})=\chi(M_{\chi})=\chi(M)=0.
  $$
\end{proof}

\section{Arithmetic fiber functors}

We now address the question of constructing arithmetic fiber functors that will be used to define conjugacy classes
of elements in the tannakian groups.\index{arithmetic fiber functor}

\begin{definition}[Unramified characters]\label{def:unramified-fibfunct}
  \index{unramified character} Let~$M$ be an object of $\Ppint(G)$. A
  weakly unramified character $\chi\in \charg{G}$ for $M$ is said to be
  \emph{unramified for $M$} if the functor
  \[
    N\longmapsto  \omega_\chi(N)=\rmH^0(G_{\bar k},N_\chi)
  \]
  is a fiber functor on the
  category~$\braket{M}\subset \Ppint(G)$.\nomenclature[$omega$]{$\omega_{\chi}$}{fiber functor
    defined by~$\chi$} We denote by 
  \[
  \unram{M}\subset \wunram{M}\subset \charg{G}
  \] the set of unramified characters for $M$.\nomenclature[$X$]{$\unram{M}$}{set of
  unramified characters}  We say that the perverse sheaf $M$ is \emph{generically unramified}\index{generically unramified} if the subset
$\unram{M} \subset \charg{G}$ is generic.\index{generically unramified
  perverse sheaf}
\end{definition}

\begin{remark}
\label{rem-unram-ex3}
In Example~\ref{ex-non-unram}, we will give examples to show that
  there may exist weakly unramified characters which are not unramified. An example is given by the sheaf $M=\mcL_{\eta(f)}(1/2)[1]$ on $\Gg_m\times \Gg_a$, where $f$ is a polynomial of degree $d$ such that $f(0)\not= 0$ and $\eta$ is a multiplicative character such that $\eta^d$ is non-trivial. 
  
We shall prove that every character $(\chi,a)$ is weakly unramified for $M$, that $\dim(\rmH^0(\Gg_m\times \Gg_a,M_{(\chi,a)})$ is $d+1$ if $a\not= 0$ but $d$ if $a=0$, implying that $(\chi,0)$ is not unramified for $M$. 

\end{remark}

We expect that all semisimple objects of $\Ppint(G)$ are generically
unramified. We can currently only prove this property for the three
fundamental types of algebraic groups.

\begin{theorem}\label{th-polite}
  If $G$ is a torus, an abelian variety or a unipotent group, then any
  semisimple object of $\Ppint(G)$ is generically unramified.
\end{theorem}

For tori or abelian varieties, we need a general technical criterion
ensuring that an object $M$ is generically unramified.

\begin{lemma}\label{lm-critere-unram}
  Let $M$ be a semisimple object of $\Ppint(G)$. Set
  $L=M\oplus M^{\vee}$. For each $m\geq 2$, let~$C_m$ be the cone of the
  canonical morphism $L^{*_!^m}\to L^{*_*^m}$. All characters $\chi$ in 
  \begin{equation}\label{eq-set-unram}
    \wunram{M}\cap \bigcap_{m\geq 2}\nunram{C_m}
  \end{equation}
  are unramified for~$M$.
\end{lemma}

\begin{proof}
  Let $\chi$ be a character in the set~(\ref{eq-set-unram}). By
  Proposition~\ref{pr-tensor-ab}, every object $N$ of $\braket{M}$ is a
  direct sum of direct factors of $m$-fold convolution products
  $L^{*_{\intt}^m}$ for some integers $m$. By the definition
  of~(\ref{eq-set-unram}) and Lemma~\ref{lem-prop-conv}, we have
  canonical isomorphisms
  $$
  \rmH^*(G_{\bar k},L^{*_{\intt}^m}_\chi)\simeq
  \rmH^*(G_{\bar k},L^{*_\mathrm{*}^m}_\chi)\simeq
  \rmH^*(G_{\bar k},L_\chi)^{\otimes^m}
  $$
  for any~$m$.

  By~(\ref{eq-set-unram}) again, we have
  $\rmH^*(G_{\bar k},L_\chi)=\rmH^0(G_{\bar k},L_\chi)$, and hence
  $\omega_\chi(L^{*_{\intt}^m})=\omega_\chi(L)^{\otimes^m}$. This
  proves that the functor $\omega_{\chi}$ is compatible with the tensor
  product; other compatibilities are elementary, and
  the functor~$\omega_{\chi}$ is exact on~$\braket{M}$, hence the result
  (see~\cite[Prop.\,1.19]{deligne-milneII}).
\end{proof}

\begin{proof}[Proof of Theorem~\ref{th-polite} for abelian varieties]
  If~$G$ is an abelian variety, then both convolution functors are
  canonically isomorphic; hence, all objects $C_m$ vanish and the
  set~(\ref{eq-set-unram}) is the same as $\wunram{M}=\unram{M}$,
  which is generic.
\end{proof}

\begin{remark}\label{rm-ab-var-unram}
  There is a more precise result if~$G$ is an abelian variety. Indeed,
  we have recalled that $\wunram{M}=\unram{M}$ for any semisimple object
  of~$\Ppint(G)$, and by the strong form of the Stratified Generic
  Vanishing Theorem (Theorem~\ref{thm-high-vanish}), it follows that the
  set of ramified characters is contained in a finite union of \tacs\
  of~$G$.
\end{remark}

\begin{proof}[Proof of Theorem~\ref{th-polite} for tori]
  We use the notation of the previous lemma. For a torus~$G$, a result
  of Gabber and Loeser~\cite[Prop.\,3.9.3\,(iv)]{GL_faisc-perv} implies
  that there is an inclusion $\nunram{C_2}\subset \nunram{C_m}$ for all
  integers $m\geq 2$. So the set
  \[
    \wunram{M}\cap \bigcap_{m\geq 2}\nunram{C_m}= \wunram{M}\cap
    \nunram{C_2}
  \]
  is generic, by the generic vanishing theorem and the definition of
  negligible objects.
\end{proof}



Finally we consider unipotent groups.

\begin{proof}[Proof of Theorem~\ref{th-polite} for~$G$ unipotent]
  We denote by $G^{\vee}$ a form of the Serre dual of~$G$, and we fix an
  additive character $\psi$ to compute the Fourier transform
  $\ft_{\psi}$ on~$G$ (see Section~\ref{sec-unipotent}).

  Let $M$ be a semisimple object of $\Ppint(G)$. We claim that there
  exists a dense open set $V\subset G^\vee$ such that for all objects
  $N$ and $N'$ of $\braket{M}$, the restriction of $\ft_{\psi}(N)$
  to~$V$ is lisse and there exists a canonical isomorphism
  \begin{equation}\label{eq-can-goal}
    \ft_{\psi}(N*_\intt N')|V \to (\ft_{\psi}(N)\otimes
    \ft_{\psi}(N'))|V.
  \end{equation}
  \par
  Indeed, if this claim holds, then it is elementary that for any
  $a\in V(\bar{k})$, the corresponding character $\psi_a\in\charg{G}$
  is unramified for~$M$.
  

  The claim above follows in turn from a more general statement: for
  all objects $M_1$ and~$M_2$ of $\Ppint(G)$, and for any open dense
  subset $W\subset G^{\vee}$ such that the Fourier transforms
  $\ft_{\psi}(M_1)$ and $\ft_{\psi}(M_2)$ are lisse on $W$, there
  exists a canonical isomorphism
  \[
    \ft_{\psi}(M_1*_{\intt} M_2)|W\to (\ft_{\psi}(M_1)\otimes
    \ft_{\psi}(M_2))|W.
  \]
  Indeed, the isomorphism shows in particular that the Fourier transform
  of $M_1*_{\intt} M_2$ is also lisse on~$W$; since the same is true of
  the dual $\DD(M_1)$, it follows that the Fourier transform of any
  object of $\braket{M_1}$ is lisse on~$W$, leading to the previous
  claim (with $V=W$).

  We now prove the general statement above. Let $M=\pH^0(M_1 *_!
  M_2)$. By definition of $M_1 *_\intt M_2$, we have
  $M_1 *_\intt M_2=M_{\intt}$ (see Remark~\ref{rem-Mint}).
 
  Let $\ptau_{\leq 0}$ and $\ptau_{\geq 0}$ be the perverse truncation
  functors. We have canonical morphisms
 \begin{equation}
   \ptau_{\leq 0}(M_1 *_! M_2)\to M_1 *_! M_2
  \end{equation}
  and 
  \begin{equation}
    \ptau_{\leq 0}(M_1 *_! M_2)\to
    \ptau_{\geq 0}(\ptau_{\leq 0}(M_1 *_! M_2))=\pH^0(M_1 *_! M_2)=M.
  \end{equation}
 
  By Lemma \ref{lem-mildeconv-quot}, the mapping cones of both morphisms
  are negligible. By the vanishing theorem for unipotent groups
  (Proposition~\ref{prop-strat-unip}), there is a dense open subset~$W'$
  of~$W$ such that the induced morphisms
  \begin{equation}\label{eqn:fibfunctunip1}
    \ft_\psi(\ptau_{\leq 0}(M_1 *_! M_2))|W'\to \ft_\psi(M_1 *_! M_2)|W'
  \end{equation}
  and 
  \begin{equation}\label{eqn:fibfunctunip2}
    \ft_\psi( \ptau_{\leq 0}(M_1 *_! M_2))|W'\to
    \ft_\psi(M)|W'
  \end{equation}
  are isomorphisms. Inverting~(\ref{eqn:fibfunctunip1}) and composing
  with~(\ref{eqn:fibfunctunip2}), we obtain a canonical isomorphism
  \begin{equation}
    \label{eqn:fibfunctunip3}
    \ft_\psi(M_1 *_! M_2)|W'\to  \ft_\psi(M)|W'.
  \end{equation}
  
  Let $M^t$ be the smallest subobject of $M$ such that $M/M^t$ is
  negligible. We then have a
  canonical injection $M^t\to M$ with negligible cokernel and a
  canonical surjection $M^t\to M_\intt$ with negligible kernel, by
  Remark~\ref{rem-Mint}. 
  By the
  vanishing theorem for unipotent groups
  (Proposition~\ref{prop-strat-unip}), up to replacing $W'$ by a smaller
  dense open subset, we can assume that the canonical morphisms
  \begin{equation}\label{eqn:fibfunctunip4}
    \ft_{\psi}(M^t)|W'\to    \ft_{\psi}(M)|W' 
  \end{equation}
  and
  \begin{equation}\label{eqn:fibfunctunip5}
    \ft_{\psi}(M^t)|W'\to    \ft_{\psi}(M_\intt)|W' 
  \end{equation}
  are isomorphisms. Inverting~(\ref{eqn:fibfunctunip4}) and composing
  with (\ref{eqn:fibfunctunip5}), we get a canonical isomorphism
  \begin{equation}\label{eqn:fibfunctunip6}
    \ft_{\psi}(M)|W'\to  \ft_{\psi}(M_\intt)|W'=\ft_{\psi}(M_1*_\intt M_2)|W'. 
  \end{equation}
  Composing (\ref{eqn:fibfunctunip3}) and (\ref{eqn:fibfunctunip6}), we
  get a canonical isomorphism
  \begin{equation}
    \label{eqn:fibfunctunip7}
    \ft_{\psi}(M_1*_!M_2)|W'\simeq \ft_{\psi}(M_1*_\intt M_2)|W'.
  \end{equation}
  
  Denote by $j\colon W'\to W$ the open immersion. By the definition of
  the category $\Ppint(G)$, the Fourier transform
  $\ft_{\psi}(M_1*_{\intt} M_2)$ (which is a perverse sheaf up to shift)
  has no shifted perverse component supported in $G^\vee\setminus W'$
  (such a component would be negligible), and therefore we have a
  canonical isomorphism
  \begin{equation}\label{eqn:fibfunctunip8}
    j_{!*}j^*(\ft_{\psi}(M_1*_\intt M_2)|W)\simeq \ft_{\psi}(M_1*_\intt M_2)|W
  \end{equation}
  by the properties of the intermediate extension functor $j_{!*}$ (see
  Proposition~\ref{pr-intermediate-ext}).
  
  By Lemma~\ref{lem-prop-conv}, there is a canonical isomorphism
  $\ft_{\psi}(M_1*_! M_2)\simeq \ft_\psi (M_1) \otimes
  \ft_\psi(M_2)$. Since $\ft_{\psi}(M_1)$ and $\ft_{\psi}(M_2)$ are
  lisse on $W$, we have also a canonical isomorphim
  $$
  j_{!*}j^*((\ft_\psi (M_1) \otimes \ft_\psi(M_2))|W)
  \simeq (\ft_\psi (M_1) \otimes \ft_\psi(M_2))|W,
  $$
  hence a canonical isomorphism
  \begin{equation}\label{eqn:fibfunctunip9}
    j_{!*}j^*(\ft_\psi (M_1*_! M_2)|W)
    \simeq (\ft_\psi (M_1) \otimes \ft_\psi(M_2))|W.
  \end{equation}
  
  We now apply the functor $j_{!*}$ to the
  isomorphism~(\ref{eqn:fibfunctunip7}), and use
  (\ref{eqn:fibfunctunip8}) and (\ref{eqn:fibfunctunip9}) to obtain the
  desired canonical isomorphism~(\ref{eq-can-goal}); this concludes the
  proof of the claim.
\end{proof}

\section{The arithmetic tannakian group}

In this section, we consider the situation over the finite field~$k$. Base change $M\mapsto M_{\bar{k}}$ gives a functor $\Perv(G)\to \Pp(G)$. For a perverse sheaf~$M$
on~$G$, we define the set of unramified characters for $M$ as
$\unram{M}=\unram{M_{\bar{k}}}$.

We denote by $\NegParith(G)$ (\resp $\Ppintarith(G)$)
\nomenclature[$N$]{$\NegParith(G)$}{arithmetic negligible objects}
\nomenclature[$N$]{$\Ppintarith(G)$}{arithmetic internal convolution category}
the full
subcategory of $\Perv(G)$ whose objects are the perverse sheaves~$M$
such that $M_{\bar{k}}$ is an object of $\NegP(G)$ (\resp of
$\Ppint(G)$). As in the geometric case, we find that $\NegParith(G)$ is
a Serre subcategory of $\Perv(G)$ and that the localization functor induces 
an equivalence from $\Ppintarith(G)$ to the quotient abelian category
$\Ppbarith(G)=\Perv(G)/\NegParith(G)$.\nomenclature[$P$]{$\Ppbarith(G)$}{arithmetic
  convolution category}

Also similarly to the geometric case, the two convolution bifunctors
on $\Perv(G)$ induce equivalent bifunctors on $\Ppbarith(G)$ (compare
with Proposition~\ref{pr-convol-pp}). The categories $\Ppbarith(G)$
and $\Ppintarith(G)$ are then rigid symmetric $\Qlb$-linear tensor
categories, with unit object~$\un$ still the skyscraper sheaf at the
unit of~$G$, which again satisfies $\End(\un)\simeq \Qlb$.

Let $M$ be a perverse sheaf on~$G$. To distinguish between the
arithmetic and geometric situations, we denote from now on by
$\braket{M}^\arith$ (\resp $\braket{M}^\geo$) the subcategory of
$\Ppintarith(G) \simeq \Ppbarith(G)$ (\resp
of~$\Ppint(G)\simeq \Ppb(G)$) that is tensor-generated by (the image
of) $M$ (resp. by~$M_{\bar{k}}$). Base change $N\mapsto N_{\bar{k}}$ gives a functor
from~$\braket{M}^{\arith}$ to~$\braket{M}^{\geo}$.
\nomenclature{$\braket{M}^{\arith}$}{subcategory of $\Ppintarith(G)$
tensor-generated by~$M$}
\nomenclature{$\braket{M}^{\geo}$}{subcategory of $\Ppint(G)$
tensor-generated by~$M_{\bar{k}}$}

\begin{theorem}
\label{thm-arith_cat_neut_tan}
Let $M$ be an object of $\Perv(G)$. The categories $\braket{M}^\arith$ and
$\braket{M}^\geo$ are neutral $\Qlb$\nobreakdash-linear tannakian categories. There
exist algebraic groups $\ggeo{M}$ and $\garith{M}$ over~$\Qlb$ such that
$\braket{M}^{\arith}$ is equivalent to the category
$\mathrm{Rep}_{\Qlb}(\garith{M})$ and
$\braket{M}^{\geo}$ is equivalent to the category
$\mathrm{Rep}_{\Qlb}(\ggeo{M})$.
\nomenclature[$G$]{$\garith{M}$}{arithmetic tannakian group of~$M$}
\par
Moreover, if~$r$ is the tannakian dimension of~$M$, then the objects~$M$
and~$M_{\bar{k}}$ of $\braket{M}^{\arith}$ and~$\braket{M}^{\geo}$,
respectively, correspond to faithful representations of~$\garith{M}$
and~$\ggeo{M}$  in~$\GL_r(\Qlb)$.
\end{theorem}

\begin{proof}
  The case of $\braket{M}^\geo$ is dealt with by
  Theorem~\ref{thm-cat_neut_tan} and Corollary~\ref{cor-def-group}. The
  case of $\braket{M}^\arith$ follows by the same argument because
  Proposition~\ref{prop-dimPerv=dimVect} also applies to
  $\Ppintarith(G)$.
  \par
  The last assertion is a tautological consequence of the formalism.
\end{proof}

\begin{definition}
  In the context of Theorem \ref{thm-arith_cat_neut_tan}, we call $\garith{M}$ the \emph{arithmetic tannakian
    group}\index{arithmetic tannakian group} of~$M$, and $\ggeo{M}$ its
  \emph{geometric tannakian group}.\index{geometric tannakian group} 
\end{definition}

\begin{proposition}
  Let $M$ be an object of $\Perv(G)$. The functor of base change to
  $\bar{k}$ is a tensor functor from $\braket{M}^{\arith}$ to
  $\braket{M}^{\geo}$ that induces a morphism
  $\varphi\colon \ggeo{M}\to \garith{M}$. This morphism is a closed
  immersion.
\end{proposition}

\begin{proof}
  The first assertion is immediate, and it implies by the tannakian
  formalism the existence of the homomorphism~$\varphi$. According
  to~\cite[Prop.\,2.21\,(b)]{deligne-milneII}, this morphism~$\varphi$ is a
  closed immersion if and only if every object of $\braket{M}^{\geo}$ is
  isomorphic to a subquotient of an object in the essential image of the
  base-change functor.
  \par
  Let $N$ be such an object of~$\braket{M}^{\geo}$, viewed as an object
  of $\Ppint(G)$. By definition of the category~$\Pp(G)$, there exists a
  finite extension $k_n$ of~$k$ in~$\bar{k}$ such that $N$ is the base
  change to~$\bar{k}$ of a perverse sheaf~$N_1$ on~$G_{k_n}$. Then $N$
  is a subquotient of the base change of the perverse sheaf $f_{n*}N_1$
  to $G_{\bar{k}}$, where $f_n\colon\Spec(k_n)\to \Spec(k)$ is the
  canonical morphism, hence the result.
\end{proof}

From now on, we will identify the geometric tannakian group of
a perverse sheaf~$M$ on~$G$ with its image in the arithmetic
tannakian group. 

We recall the convention from Section~\ref{sec-semisimple} concerning
properties over~$k$ and~$\bar{k}$. Let $M$ be a perverse sheaf on
$G$. We view $\braket{M}^{\arith}$ as a subcategory of $\Ppint(G)$, so
that the weights of an object $N\in \braket{M}^{\arith}$ are
well-defined.

\begin{theorem}
\label{thm-Harith-fib-funct}
Let~$M$ be a perverse sheaf on~$G$. Assume that~$M$ is arithmetically
semisimple and pure of weight zero. Let $r$ be the tannakian
dimension of~$M$.
\par
\emph{(1)} The groups $\garith{M}$ and $\ggeo{M}$ are reductive
subgroups of $\GL_r$.
\par
\emph{(2)} Every object $N$ of $\braket{M}^\arith$ is arithmetically
semisimple and pure of weight zero, and every object $N$ of
$\braket{M}^\geo$ is semisimple.
\end{theorem}

\begin{proof}
  Since any pure perverse sheaf on~$G$ is geometrically semisimple
  by~\cite[Th.\,5.3.8]{BBD-pervers}, the assertions for
  $\braket{M}^\geo$ follow. 
  The same proof is also valid for $\braket{M}^\arith$, since~$M$ is
  arithmetically semisimple, so that the group $\garith{M}$ is also
  reductive, and all objects of $\braket{M}^\arith$ are arithmetically
  semisimple.
  \par
  We now prove the purity statement. Since $M$ is pure of weight zero, it follows from the description of $ M_{\intt}$ in Remark~\ref{rem-Mint} that the corresponding object of~$\Ppintarith(G)$ is also pure of weight
  zero, and similarly for its dual 
 
  For any perverse sheaves~$N_1$ and~$N_2$ on~$G$ that are pure of
  weight zero, the convolution~$N_1*_{\intt} N_2$ is also pure of
  weight zero. Indeed, by Deligne's Riemann
  Hypothesis~\cite[3.3.1]{D-WeilII}, the object $N_1*_!N_2$ is mixed of
  weights~$\leq 0$. Hence, the quotient $N_1*_{\intt} N_2$ of $N_1*_! N_2$ is also mixed of weights $\leq 0$
  by~\cite[Prop.\,5.3.1]{BBD-pervers}.
  Thanks to Lemma~\ref{lem-prop-conv}, the same
  applies to the Verdier dual $\DD(N_1*_{\intt} N_2)$, which implies the
  claim.

  Hence, the property of being pure of weight zero is preserved by
  convolution, duality and taking subobjects. Thus we conclude that
  every object $N$ of~$\braket{M}^\arith$ is pure of weight zero.
\end{proof}


We now show that the tannakian groups coincide with those of Katz for
the multiplicative group using the category $\oldcal{P}$ (see~\cite[Ch.\,2]{mellin} and Section~\ref{sec-category-p}), and with monodromy groups of the Fourier
transform for unipotent groups.

\begin{proposition}\label{pr-gm-unip-known}
  Let~$M$ be a perverse sheaf on~$G$. Assume that~$M$ is arithmetically
  semisimple and pure of weight zero.
  \par
  \begin{enumth}
    \item If $G=\Gg_m$, then the arithmetic and geometric tannakian
      groups of~$G$ coincide with those defined by Katz using 
      the category $\oldcal{P}$.
    \item If $G$ is unipotent of dimension~$d$, and $\psi$ is a fixed
      additive character used to define its Fourier transform, then
      there exists a dense open subset~$U$ of the Serre dual~$G^{\vee}$
      such that $(\ft_{\psi}M_{\intt})|U$ is isomorphic to a lisse sheaf
      $\mcF$ on~$U$, pure of weight~$d$, placed in degree~$0$. The
      arithmetic and geometric tannakian groups of~$M$ coincide with the
      arithmetic and geometric monodromy groups of the lisse sheaf~$\mcF$.
  \end{enumth}
\end{proposition}

\begin{proof}
  In the case of $\Gg_m$, the statement follows directly from
  Example~\ref{ex-gm-geo} (1) (see also Section~\ref{sec-category-p} for
  the definition of $\oldcal{P}$).
  \par
  Suppose then that $G$ is unipotent. To prove the first assertion of (2), we
  may assume that $M$ is simple and non-negligible. Its Fourier
  transform is then a simple $d$-shifted perverse sheaf on the Serre
  dual~$G^{\vee}$, pure of weight~$d$, and with support equal
  to~$G^{\vee}$ (since the object $M$ would be negligible if the support
  were smaller). Thus it is a single lisse sheaf, pure of weight~$d$, on
  an open dense subset of~$G^{\vee}$.
  \par
  For the second part of~(2), we note that by (the proof of)
  Theorem~\ref{th-polite} for unipotent groups, the convolution product
  on $\braket{M}^{\arith}$ can be identified with the tensor product on
  the subcategory generated by~$\mcF$ of the category of lisse sheaves
  on~$U$. The result then follows.
\end{proof}



\section{Frobenius conjugacy classes}

We keep working over the finite field~$k$ and use the same notation as
in the previous subsection. For any finite extension $k_n$ of~$k$, we
denote by $\Fr_{k_n}$ the geometric Frobenius automorphism
of~$k_n$.\nomenclature[$F$]{$\Fr_{k_n}$}{geometric Frobenius automorphism
  of~$k_n$}\index{geometric Frobenius automorphism}
  
For an object~$M$ of~$\Der(X)$, an integer~$n\geq 1$ and a character
$\chi\in\charg{G}(k_n)$, we denote by $\Fr_{M,k_n}(\chi)$
\nomenclature[$F$]{$\Fr_{M,k_n}(\chi)$}{Frobenius action
  on~$H^0_c(G_{\bar{k}},\chi)$} the automorphism of the $\Qlb$-vector
space $H^0_c(G_{\bar{k}},M_{\chi})$ induced by the action
of~$\Fr_{k_n}$. Recall from \ref{sec-weights} the notions of weights and
purity.

Let~$r$ be the dimension of this space. If the automorphism
$\Fr_{M,k_n}(\chi)$ is pure of weight zero, for instance if~$M$ is
pure of weight~$0$ and $\chi$ is weakly unramified for~$M$, then there
is a unique conjugacy class $\Theta_{M,k_n}(\chi)$
\nomenclature[$T$]{$\Theta_{M,k_n}(\chi)$}{unitary conjugacy class for
  $\Fr_{M,k_n}(\chi)$} in the complex unitary group~$\Un_r(\Cc)$
containing the semisimple part of $\iota_0(\Fr_{M,k_n}(\chi))$.

We call~$\Fr_{M,k_n}(\chi)$ the
\emph{Frobenius automorphism of~$M$ associated to $\chi$ over $k_n$}
and $\Theta_{M,k_n}(\chi)$ the \emph{unitary Frobenius conjugacy class
  of~$M$ associated to~$\chi$ over~$k_n$}.
\index{Frobenius automorphism of~$M$ associated to~$\chi$}
\index{unitary Frobenius conjugacy class of~$M$ associated to$~\chi$}

Suppose now that $M$ is an arithmetically semisimple
perverse sheaf on~$G$.

Let~$n\geq 1$ and let $\chi\in\charg{G}(k_n)$ be an \emph{unramified}
character for~$M$, so that the functor
$\omega_{\chi} \colon N \mapsto \rmH^0(G_{\bar k},N_\chi)$ is a fiber
functor on the tannakian category $\braket{M}^\arith$.  For any object
$N$ of $\braket{M}^\arith$, the Frobenius automorphism $\Fr_{k_n}$ now
induces an automorphism of $\omega_{\chi}(N)$, and thus defines an
automorphism of the fiber functor $\omega_{\chi}$. By the tannakian
formalism, this corresponds to a unique
conjugacy class in~$\garith{M}(\Qlb)$. We denote by
$\Frf_{M,k_n}(\chi)$ the corresponding conjugacy class of
$\garith{M}(\Cc)$, and call it the \emph{Frobenius conjugacy class
  of~$M$ associated to $\chi$ over $k_n$}.
\nomenclature[$F$]{$\Frf_{M,k_n}(\chi)$}{Frobenius conjugacy class
  in~$\garith{M}$}
\index{Frobenius conjugacy class of~$M$ associated to~$\chi$}


Suppose furthermore that~$M$ is pure of weight zero. Let $K_M$ be a
maximal compact subgroup of the reductive group $\garith{M}(\Cc)$. Since
all objects of $\braket{M}^{\arith}$ are pure of weight zero (by
Theorem~\ref{thm-Harith-fib-funct}), the eigenvalues of any element of
the conjugacy class $\Frf_{M,k_n}(\chi)$ are complex numbers with
modulus~$1$, so that the semisimple part of this conjugacy class is a
unitary matrix. One can then deduce from the Peter--Weyl
Theorem\index{Peter--Weyl theorem} that the $\garith{M}(\Cc)$-conjugacy
class of the semisimple part of $\Frf_{M,k_n}(\chi)$ intersects $K_M$ in
a unique conjugacy class, which is denoted $\Thetaf_{M,k_n}(\chi)$, and
is called the \emph{unitary Frobenius conjugacy class} of~$M$ associated
to~$\chi$. (See, e.g.,~\cite[9.2.4]{katz-sarnak} for this argument.)
\nomenclature[$T$]{$\Thetaf_{M,k_n}(\chi)$}{unitary Frobenius conjugacy
  class in~$\garith{M}$} \index{unitary Frobenius conjugacy class
  associated to~$\chi$}

For an unramified character~$\chi$, the space $\omega_{\chi}(M)$ has
dimension~$r$, the tannakian dimension of~$M$, and the conjugacy class
of $\Frf_{M,k_n}(\chi)$ in the automorphism group
of~$H^0_c(G_{\bar{k}},M_{\chi})$ coincides with that of
$\Fr_{M,k_n}(\chi)$, and similarly for $\Thetaf_{M,k_n}(\chi)$.

When $k_n=k$, we will sometimes use simply the notation $\Fr_M(\chi)$,
$\Theta_M(\chi)$, $\Frf_M(\chi)$, $\Thetaf_M(\chi)$.
\nomenclature[$F$]{$\Fr_{M}(\chi)$}{$\Fr_{M,k}(\chi)$}
\nomenclature[$T$]{$\Theta_{M}(\chi)$}{$\Theta_{M,k}(\chi)$}
\nomenclature[$F$]{$\Frf_{M}(\chi)$}{$\Frf_{M,k}(\chi)$}
\nomenclature[$T$]{$\Thetaf_{M}(\chi)$}{$\Thetaf_{M,k}(\chi)$}

We have the following important consequences of the formalism.

\begin{lemma}\label{lm-trace}
  Let~$M$ be an arithmetically semisimple perverse sheaf on~$G$ that is
  pure of weight zero and of tannakian dimension $r\geq 0$.
  \begin{enumth}
  \item Let~$\chi\in\wunram{M}(k)$ be a weakly unramified character
    for~$M$. For any integer $n\geq 1$, we have
    $$
    \Tr(\Theta_{M,k_n}(\chi))=
    \Tr(\Theta_M(\chi)^n)=\sum_{x\in
      G(k_n)}\chi(N_{k_n/k}(x))t_M(x;k_n),
    $$
    where~$t_M$ is the trace function of~$M$ and the trace on the left
    is that on~$\GL_r$.
  \item Let $\chi\in\unram{M}(k)$ be an unramified character. Let~$\rho$
    be an algebraic $\Qlb$-representation of~$\garith{M}$ and denote
    by~$\rho(M)$ the corresponding object of $\braket{M}^{\arith}$. The
    character $\chi$ is unramified for~$\rho(M)$ and
    $$
    \Tr(\rho(\Frf_{M}(\chi)))=\Tr(\frob_{k}\mid
    H^0_c(G_{\bar{k}},\rho(M)_{\chi})). 
    $$
  \end{enumth}
\end{lemma}

\begin{proof}
  (1) By definition, we have
  $$
  \Tr(\Theta_M(\chi)^n)=\Tr(\Fr_M(\chi)^n)=\Tr(\Fr_{k}^n\mid
  H^0(G_{\bar{k}},M_{\chi})).
  $$
  \par
  Since~$\chi$ is weakly unramified, we
  have~$H^i_c(G_{\bar{k}},M_{\chi})=0$ for all~$i\not=0$
  and~$H^0(G_{\bar{k}},M_{\chi})=H^0_c(G_{\bar{k}},M_{\chi})$, so that
  we can write
  $$
  \Tr(\Theta_M(\chi)^n)=\sum_{i\in\Zz}(-1)^i\Tr(\Fr_{k}^n\mid
  H^i_c(G_{\bar{k}},M_{\chi}))=\sum_{x\in
    G(k)}\chi(N_{k_n/k}(x))t_M(x;k_n),
  $$
  by the trace formula.
  \par
  (2) The fact that $\chi$ is unramified for $\rho(M)$ follows from the
  definition 
  and Proposition~\ref{pr-tensor-ab}, and the formula follows then from
  the definition of the Frobenius conjugacy class of~$\chi$
  for~$\rho(M)$.
\end{proof}


\section{Frobenius-unramified characters}

Because weakly unramified characters do not always give rise to fiber
functors, and moreover we do not always know if there exist sufficiently
many (if any) unramified characters, we introduce an intermediate
notion.

\begin{definition}[Frobenius-unramified characters]\label{def-funram}
  Let~$M$ be an object of $\Perv(G)$ which is arithmetically semisimple
  and pure of weight zero, of tannakian dimension~$r$. Let~$\rho$ be a
  representation of $\GL_r$ and let~$N$ be the object of
  $\braket{M}^{\arith}$ corresponding to the restriction of~$\rho$
  to~$\garith{M}$.  Let $n\geq 1$ and let $\chi\in\wunram{M}(k_n)$ be a
  weakly unramified character for~$M$. We say that $\chi$ is
  \emph{Frobenius-unramified
    for~$\rho$}\index{Frobenius-unramified character} if $\chi$ is
  weakly unramified for~$N$ and if the formula
  $$
  \Tr(\rho(\Theta_{M,k_n}(\chi))^v)=\Tr(\frob_{k_n}^v\mid
  H^0_c(G_{\bar{k}},N_{\chi}))
  $$
  holds for all integers~$v\geq 1$, or equivalently if
  $$
  \det(1-\rho(\Theta_{M,k_n}(\chi))T)=\det(1-T\frob_{k_n}\mid
  H^0_c(G_{\bar{k}},N_{\chi})).
  $$
  \par
  The disjoint union over~$n$ of the
  set of Frobenius-unramified characters is denoted~$\funram{\rho}$.
  \nomenclature[$X$]{$\funram{\rho}$}{Frobenius-unramified characters for~$\rho$}
\end{definition}

\begin{remark}
  (1) The key point is that since $\rho$ is a representation of $\GL_r$,
  we can consider the conjugacy class of $\rho(\Fr_{M,k_n}(\chi))$ (in
  $\GL(V)$, where $\rho$ is a representation on~$V$); a priori, this is
  not meaningful for a representation of $\garith{M}$, unless we know
  that elements of the conjugacy class of the Frobenius automorphism of
  $H^0_c(G_{\bar{k}},M_{\chi})$ are conjugate to some
  element of the arithmetic tannakian group, which is unique up to
  conjugacy in $\garith{M}$. 
  \par
  (2) We will also sometimes write $\funram{\rho}=\funram{N}$,
  \nomenclature[$X$]{$\funram{N}$}{$\funram{\rho}$}
  although
  this set depends on~$M$, since we view $N$ as an object
  of~$\braket{M}^{\arith}$.  When confusion might arise, we may also
  write $\funram{N}_M$.\nomenclature[$X$]{$\funram{N}_M$}{Frobenius-unramified characters for $N\in\braket{M}$}
\end{remark}

Any unramified character for~$M$ is Frobenius-unramified for all objects
of $\braket{M}^{\arith}$, by Lemma~\ref{lm-trace}, (2). But in
contrast to unramified characters, we can prove in all cases that the
set of Frobenius-unramified characters is generic.

\begin{proposition}\label{pr-frob-unram}
  Let~$M$ be an object of $\Perv(G)$ which is arithmetically semisimple
  and pure of weight zero and of tannakian dimension~$r\geq 0$. For any
  representation $\rho$ of $\GL_r$, the set $\funram{\rho}$ is generic.
\end{proposition}

\begin{proof}
  We first observe that it is straightforward that if two
  representations~$\rho_1$ and $\rho_2$ of $\GL_r$ have the property
  that $\funram{\rho_1}$ and $\funram{\rho_2}$ are generic, then the
  sets $\funram{\rho_1\oplus\rho_2}$, $\funram{\rho_1\otimes\rho_2}$ and
  $\funram{\rho_1^{\vee}}$ are also generic. Indeed, consider the case
  of the tensor product, the others being similar (and in fact
  simpler). Let $N_i$ be the object corresponding to~$\rho_i$. For
  $\chi$ generic, we have
  $$
  H^*_c(G_{\bar{k}},(N_1*_{\intt}
  N_2)_{\chi})=H^*_c(G_{\bar{k}},(N_1*_!N_2)_{\chi}) \simeq
  H^*_c(G_{\bar{k}}, N_{1,\chi})\otimes H^*_c(G_{\bar{k}},N_{2,\chi})
  $$
  as well as
  $$
  H^*_c(G_{\bar{k}}, N_{i,\chi})=H^0_c(G_{\bar{k}},N_{i,\chi})
  $$
  for $i=1$ and $i=2$, all these isomorphisms being compatible with
  Frobenius. Thus
  \begin{multline*}
    \det(1-T\frob_{k_n}\mid H^0_c(G_{\bar{k}},(N_1*_{\intt}
    N_{2})_{\chi}))=
    \\\det(1-T\frob_{k_n}\mid
    H^0_c(G_{\bar{k}},N_{1,\chi})) \det(1-T\frob_{k_n}\mid
    H^0_c(G_{\bar{k}},N_{2,\chi}))
  \end{multline*}
  for $\chi$ generic, which then establishes the claim concerning
  $\funram{\rho_1\otimes\rho_2}$ using the definition of
  $\funram{\rho_i}$ and the assumption that these are generic sets.

  A first consequence of this observation is that we may assume that
  $\rho$ is irreducible to prove the proposition.  Recall next that
  every irreducible representation~$\rho$ of~$\GL_r$ is isomorphic to
  one of the form~$\rho=\rho_0\otimes \det(\cdot)^k$ for some
  representation~$\rho_0$ given by a Schur
  functor~$\mathbf{S}_{\lambda}$ and some integer~$k\in\Zz$ (see,
  e.g.,~\cite[Prop.\,15.47]{fulton-harris}). Since the determinant is
  itself a Schur functor, and~$\det(\cdot)^{-k}$ is the contragredient
  of~$\det(\cdot)^k$, the previous observation reduces the proof to the
  case where~$\rho=\mathbf{S}_{\lambda}$ for some~$\lambda$.

  In this case, $\rho$ is given by the image of an explicit projector
  (see, e.g.,~\cite[\S\,6.1,\,th.\,6.3]{fulton-harris}), and hence makes
  sense for any symmetric monoidal category where idempotents split. In
  particular, this applies to $\Der(G)$ with either of the two
  convolutions, since~$\Der(G)$ is known to have this property (e.g., by
  combining the fact that $\Der(G)$ is equivalent to the bounded derived
  category of the category of perverse sheaves, by a theorem of
  Beilinson~\cite[Th.\,1.3]{beilinson}, and the fact that the bounded
  derived category of an abelian category is idempotent complete, by a
  result of Balmer and
  Schlichting~\cite[Cor.\,2.10]{balmer-schlichting}). We will denote
  by~$\rho_!(M)$ (\resp $\rho_*(M)$) the action of these functors on~$M$
  for the symmetric monoidal structures given by the convolution
  $(A,B)\mapsto A*_{!}B$ (\resp by $(A,B)\mapsto A*_* B$).

  Since taking cohomologgy with compact support (\resp cohomology) is an
  additive monoidal functor for the convolution $A*_!B$ (\resp for
  $A *_* B$), by the Künneth formula, the explicit description of the
  idempotent defining~$\rho$ provides isomorphisms
  \begin{gather}\label{eq-schur-1}
    H^*_c(G_{\bar{k}},\rho_!(M)_{\chi})\simeq \rho(H^*_c(G_{\bar{k}},M_{\chi}))
    \\
    H^*(G_{\bar{k}},\rho_*(M)_{\chi})\simeq \rho(H^*(G_{\bar{k}},M_{\chi}))
    \label{eq-schur-2}
  \end{gather}
  for every~$\chi$, which are also compatible with Frobenius, where the
  Schur functor acts on the right-hand sides in the category of bounded
  complexes of~$\Qlb$-vector spaces.
  
  Let~$N=\rho(M)$ be the object of~$\braket{M}^{\arith}$ corresponding
  to~$\rho$. Since~$\rho$ is assumed to be a Schur functor, there exists
  an integer~$l\geq 0$ and an embedding
  $N\to M_{l}= M^{*_{\intt} l}$
  (see, e.g., ~\cite[\S\,6.1]{fulton-harris}). 
  We obtain a commutative square
  $$
  \begin{tikzcd}
    \rho_!(M)\arrow[r]\arrow[d] & \rho_*(M)\arrow[d]
    \\
    M_{l,!}\arrow[r] &M_{l,*}
  \end{tikzcd}
  $$
  where $M_{l,!}=M^{*_{!} l}$ and $M_{l,*}=M^{*_{*}
    l}$.  This implies, in particular, that the cone~$C$ of the morphism
  $$
  \rho_!(M)\to \rho_*(M)
  $$
  is negligible, since this is the case for the cone of the bottom
  morphism by Lemma~\ref{lem-mildeconv-quot}.  Applying
  Remark~\ref{rem-Mint}, there exists a generic set
  $\mcX$ of characters such that for
  $\chi\in\mcX$, we have isomorphisms
  \begin{equation}\label{eq-iso1}
    H^*_c(G_{\bar{k}},N_{\chi})\simeq H^*_c(G_{\bar{k}},\rho_!(M)_{\chi})
  \end{equation}
  which are compatible with Frobenius.

  Let finally~$\chi\in\mcX$ be a character which is weakly unramified
  for both~$M$ and~$N$. Then we have isomorphisms
  $$
  H^0_c(G_{\bar{k}},N_{\chi})\simeq
  H^*_c(G_{\bar{k}},N_{\chi})\simeq
  H^*_c(G_{\bar{k}},\rho_!(M)_{\chi})\simeq
  \rho(H^*_c(G_{\bar{k}},M_{\chi}))
  \simeq \rho(H^0_c(G_{\bar{k}},M_{\chi}))
  $$
  compatible with Frobenius (the first and fourth of these are given by
  the theorem, the second is~(\ref{eq-iso1}) and the third
  is~(\ref{eq-schur-1})), and hence
  $$
  \det(1-T\Frob_{k_n}\mid H^0_c(G_{\bar{k}},N_{\chi}))=
  \det(1-\rho(\Theta_{M,k_n}(\chi))T).
  $$

  Since this holds for a generic set of characters (by
  Theorem~\ref{th-generic-unramified}), we obtain the desired result.
\end{proof}

\begin{corollary}\label{cor-finite-gen-unram}
  Let~$M$ be an object of $\Perv(G)$ which is arithmetically semisimple
  and pure of weight zero and of tannakian dimension~$r\geq 0$. If the
  group~$\garith{M}$ is finite, then $M$ is generically unramified.
\end{corollary}

\begin{proof}
  The fact that the tannakian group is finite implies that any object
  of~$\braket{M}^{\arith}$ is a subobject of a \emph{direct sum} of
  copies of a single object
  $N=M^{*_{\intt} m}*_{\intt} (M^{\vee})^{*_{\intt} l}$ for some (fixed)
  integers $m$ and~$l$ (see~\cite[Prop.\,2.20\,(a)]{deligne-milneII}). Any
  Frobenius-unramified character for~$M$ is then an unramified character
  for~$M$.
\end{proof}

\section{Group-theoretic properties}

We continue with the notation of the previous sections.


The following basic proposition establishes the relation between the
geometric and arithmetic tannakian groups.

\begin{proposition}\label{pr:geom-vs-arith1}
  Let $M$ be a geometrically semisimple object of $\Perv(G)$.  The
  geometric tannakian group $\ggeo{M}$ is a normal subgroup of the
  arithmetic tannakian group $\garith{M}$.
\end{proposition}

\begin{proof}
  The proof is identical with that of~\cite[Lemma 6.1]{mellin}.
\end{proof}

\begin{proposition}\label{pr:geom-vs-arith2}
  Let $M$ be an arithmetically semisimple object of $\Perv(G)$. Assume
  that~$M$ is pure of weight zero.
  \begin{enumth}
  \item The quotient
    $\garith{M}/\ggeo{M}$ is of multiplicative type.
  \item Let~$V$ be a geometrically trivial object of
    $\braket{M}^{\arith}$ which corresponds to a faithful representation
    of the group $\garith{M}/\ggeo{M}$. Any character $\chi\in\charg{G}$
    is unramified for~$V$, and the class~$\xi$ of the Frobenius
    conjugacy class of any such character is independent of~$\chi$ and
    generates a Zariski\nobreakdash-dense subgroup of~$\garith{M}/\ggeo{M}$.
  \item For any $n\geq 1$ and any character $\chi\in\charg{G}(k_n)$
    unramified for~$M$, the image in $\garith{M}/\ggeo{M}$ of the
    Frobenius conjugacy class $\Frf_{M,k}(\chi)$ is equal to~$\xi^n$.
  \end{enumth}
\end{proposition}

\begin{proof}
  This follows by the same arguments as in~\cite[Lemma 7.1]{mellin}
  (checking first that, using the structure of geometrically trivial
  objects as direct sums of $\alpha^{\deg}\otimes \delta_1$ for suitable
  $\alpha$, it is indeed straightforward that all characters are
  unramified for such objects).
\end{proof}


  

We will also use the following result in Chapter~\ref{sec-product}.


\begin{proposition}\label{pr-specialize}
  Let $G_1$ and~$G_2$ be connected commutative algebraic groups over~$k$
  and let $p\colon G_1\to G_2$ be a morphism of algebraic groups.
  Let~$M$ be a perverse sheaf on $G_1$ which is arithmetically
  semisimple and pure of weight zero.

  Let~$\chi_1\in\charg{G}_1(k)$ be a character such that we have
  $Rp_!(M_{\chi_1})=Rp_*(M_{\chi_1})$. Assume further that
  $N=Rp_!(M_{\chi_1})$ is perverse and arithmetically semisimple.
  \par
  \begin{enumth}
  \item The object~$N$ is pure of weight zero.
  \item Let $n\geq 1$ and let $\chi\in \wunram{N}(k_n)$ be a character
    such that $\chi_1\cdot (\chi\circ p)$ is weakly unramified
    for~$M$. Then the conjugacy classes
    $\Theta_{M,k_n}(\chi_1\cdot (\chi \circ p))$ and
    $\Theta_{N,k_n}(\chi)$ satisfy
    $$
    \det(1-T\Theta_{M,k_n}(\chi_1\cdot(\chi\circ
    p)))=\det(1-T\Theta_{N,k_n}(\chi))\in \Cc[T]
    $$
    and in particular
    $$
    \det(\Theta_{M,k_n}(\chi_1\cdot (\chi\circ
    p))))=\det(\Theta_{N,k_n}(\chi)).
    $$
  \end{enumth}
\end{proposition}

\begin{proof}
  It suffices to consider the case where $\chi\in\charg{G}(k)$.  For
  any $n\geq 1$, the exponential sums
  \begin{align*}
    S_n&=\sum_{x\in G_1(k_n)}
    t_{M}(x;k_n)(\chi_1\cdot (\chi\circ p))(N_{k_n/k}(x))\\
    S'_n&= \sum_{y\in G_2(k_n)} t_{N}(y;k_n)\chi(N_{k_n/k}(y))
  \end{align*}
  are equal by the trace formula. Hence, the corresponding $L$-functions
  $$
  \exp\Bigl(\sum_{n\geq 1}
  S_n\frac{T^n}{n}
  \Bigr),\quad\quad
  \exp\Bigl(\sum_{n\geq 1}
  S'_n\frac{T^n}{n}
  \Bigr)
  $$
  are also equal.  But these
  $L$-functions coincide with the (reversed) characteristic polynomials
  of the conjugacy classes $\Theta_{M,k}(\chi_1\cdot (\chi\circ
  p))$ and
  $\Theta_{N,k}(\chi)$, by Lemma~\ref{lm-trace} (1), hence the result.
\end{proof}

\begin{remark}
  If the morphism $p\colon G_1\to G_2$ is affine, then the condition
  $Rp_!(M_{\chi_1})=Rp_*(M_{\chi_1})$ implies that $N$ is perverse.
\end{remark}

We will give an application when the group $G_2$ is the multiplicative
group. For this we need a lemma.

\begin{lemma}\label{lm-gm-det}
  Let~$N$ be a simple perverse sheaf on~$\Gg_m$ over~$k$ which is an
  object of the category~$\Ppintarith(\Gg_m)$. Assume that~$N$ is pure
  of weight~$0$ and of tannakian dimension~$1$. Suppose that there exists
  an integer~$d\geq 1$ and a finite set~$\mcY\subset\charg{\Gg}_m$
  such that for all~$n\geq 1$ and
  for~$\chi\in \charg{\Gg}_m(k_n)\setminus \mcY(k_n)$, the
  determinant~$\det(\Thetaf_{N,k_n})^d$ depends only on~$n$. Then~$N$
  is geometrically of finite order.
\end{lemma}

\begin{proof}
  If~$N$ is not geometrically of finite order, then the perverse
  sheaf~$N$ is a hypergeometric sheaf of generic rank at least~$1$ (see
  Section~\ref{sec-hypergeometric} and Theorem~\ref{th-hypergeometric}
  for reminders of the definition of hypergeometric sheaves and for this
  result, due to Katz).
  But these hypergeometric sheaves do not have the indicated property,
  e.g. because the $\Thetaf_{N,k_n}(\chi)$ become equidistributed
  in~$\mathbf{S}^1$ as~$\chi$ varies among unramified characters
  in~$\charg{\Gg}_m(k_n)$ (see Theorem~\ref{th-hypergeometric}, (3)
  and~\cite[Th.\,7.2]{mellin} or Theorem~\ref{th-2}).
\end{proof}

\begin{proposition}\label{pr-determinant-tori}
  Let $G$ be a connected commutative algebraic group over~$k$ and let
  $p\colon G\to \Gg_m$ be a non-trivial morphism of algebraic groups.
  Let~$M$ be a perverse sheaf on $G$ which is arithmetically
  semisimple and pure of weight zero.

  Let~$\chi_1\in\charg{G}(k)$ be a character such that the equality
  $Rp_!(M_{\chi_1})=Rp_*(M_{\chi_1})$ holds.  Assume further that the
  complex~\hbox{$N=Rp_!(M_{\chi_1})$} is a perverse sheaf on~$\Gg_m$
  and is arithmetically semisimple. It is then pure of weight zero.

  \par
  Suppose that the set of $\chi\in\charg{\Gg}_m$ such that
  $\chi_1(\chi\circ p)$ is unramified for the object $\det(M)$ is
  generic, and that the tannakian determinant of~$N$ is arithmetically
  \respup{geometrically} of infinite order. Then the tannakian
  determinant of~$M$ is arithmetically \respup{geometrically} of
  infinite order.

\end{proposition}

\begin{proof}
  We begin by proving that the determinant is arithmetically of infinite
  order in both cases.  Let $n\geq 1$ and let
  $\chi\in\charg{\Gg}_m(k_n)$ be a character such that
  $\chi_1(\chi\circ p)$ is unramified for the object $\det(M)$. We then
  have
  \begin{equation}\label{eq-det-special}
    \Thetaf_{\det(M),k_n}(\chi)=\det\big(\Theta_{M,k_n}(\chi_1(\chi\circ
    p))\big)=\det\big(\Theta_{N,k_n}(\chi)\big)
  \end{equation}
  by Proposition~\ref{pr-specialize}. By assumption this is valid for
  all but finitely many $\chi\in \charg{\Gg}_m$, and moreover $N$ has
  determinant which is arithmetically of infinite order, so that the
  arithmetic tannakian group of $\det(M)$ must be infinite.

  

  It remains to deduce that the geometric tannakian determinant of~$M$
  has infinite order if the same property holds for~$N$. If not, then
  $\det(M)^d$ would be geometrically trivial for some integer
  $d\geq 1$. In this case, for any $n\geq 1$ and any character
  $\chi\in\charg{G}(k_n)$ which is Frobenius-unramified for~$\det$,
  the determinant $\det(\Theta_{M,k_n}(\chi))^d$ only depends on~$n$
  (see Proposition~\ref{pr:geom-vs-arith2},
  (2)). By~(\ref{eq-det-special}) and Lemma~\ref{lm-gm-det}, the
  tannakian determinant of~$N$ (which is an object of tannakian
  dimension~$1$ on~$\Gg_m$) is geometrically of finite order, which
  contradicts the assumption.
\end{proof}

\begin{remark}\label{rm-application}
  If $G=T\times \Gg_m$ for some torus $T$ and $p$ is the projection on
  $\Gg_m$ then, according to Theorem~\ref{thm-tori-vanishing-rel}
  applied to~$p$ and $M$, the assumption that
  $Rp_!M_{\chi_1}=Rp_*M_{\chi_1}$ and that this complex is a perverse
  sheaf is true for all $\chi_1$ outside of a finite union of \tacs\
  of~$T$.  Moreover, by varying $\chi_1$, we can always find such a
  character for which $\chi_1(\chi\circ p)$ is unramified for generic
  $\chi$, since $M$ is generically unramified by
  Theorem~\ref{th-polite}.
\end{remark}

Using further work of Katz, we can give a sufficient criterion to apply
this proposition.

\begin{corollary}\label{cor-simpler}
  Let $G$ be a connected commutative algebraic group over~$k$ and let
  $p\colon G\to \Gg_m$ be a non-trivial morphism of algebraic groups.
  Let~$M$ be a perverse sheaf on $G$ which is arithmetically semisimple
  and pure of weight zero.
  \par
  Let~$\chi_1\in\charg{G}(k)$ be a character satisfying
  $Rp_!(M_{\chi_1})=Rp_*(M_{\chi_1})$.  Assume that~$N=Rp_!(M_{\chi_1})$ is a perverse sheaf on~$\Gg_m$, which is
  arithmetically semisimple and of the form~$\sheaf{F}[1]$ for some
  middle extension sheaf~$\mcF$ (see Example~\ref{ex-middle-ext} for the
  definition of middle extension sheaves).\index{middle extension sheaf} Let
  $$
  (e_1,\ldots, e_l),\quad\quad (f_1,\ldots, f_{m})
  $$
  be the sizes of the unipotent Jordan blocks\index{unipotent Jordan blocks}
  in the tame monodromy\index{tame monodromy representation}
  representation of~$\sheaf{F}$ at~$0$ and~$\infty$ respectively.
  \par
  Suppose that the set of $\chi\in\charg{\Gg}_m$ such that
  $\chi_1(\chi\circ p)$ is unramified for the object $\det(M)$ is
  generic.
  \par
  If we have
  $$
  \sum_{i}e_i-\sum_jf_j\not=0,
  $$
  then the tannakian determinant of~$M$ is geometrically of infinite
  order.
\end{corollary}

\begin{proof}
  According to the previous proposition, it suffices to show that the
  tannakian determinant of~$N$ is geometrically of infinite order.
  By~\cite[Th.\,16.1]{mellin}, the condition implies that the
  determinant of the Frobenius action on Deligne's fiber
  functor\index{Deligne's fiber functor}
  $\omega_{\mathrm{Del}}(N)$ is not unitary (see
  Section~\ref{sec-deligne-fiber-functor} for the definition of this
  functor), and the result follows from Katz's classification of objects
  of tannakian dimension~$1$ on~$\Gg_m$
  (Theorem~\ref{th-hypergeometric}).
\end{proof}

\section{External products}

The following proposition concerns objects on a product
$G=G_1\times G_2$, and is useful for constructing various examples (see
for instance Section~\ref{ssec-neg-one}).

\begin{proposition}\label{pr-product}
  Assume that~$G=G_1\times G_2$ for connected commutative algebraic
  groups~$G_1$ and~$G_2$. For any objects $M_i\in \Der(G_i)$, there
  exist natural isomorphisms
  $$
  M_1\boxtimes M_2\simeq (M_1\boxtimes \un_{G_2})*_{!}
  (\un_{G_1}\boxtimes M_2)\simeq
  (M_1\boxtimes \un_{G_2})*_{*}
  (\un_{G_1}\boxtimes M_2).
  $$

  Moreover, if $M_1$ and $M_2$ are perverse sheaves on~$G_1$ and~$G_2$
  with tannakian rank~$r_1$ and~$r_2$, respectively, then the object
  $M_1\boxtimes M_2\in\Der(G)$ is perverse and has tannakian
  rank~$r_1r_2$.
\end{proposition}

\begin{proof}
  We use coordinates $(x_1,x_1,y_1,y_2)$ on $G\times G$ with $x_i$ and
  $y_i$ coordinates on~$G_i$. Let $m_{12}\colon G\times G\to G$ be the
  multiplication map for~$G$, and $m_i\colon G_i\times G_i \to G_i$
  those for~$G_i$.

  Let~$N$ be the object on the right-hand side of the first isomorphism
  to be established. By definition, we have
  $$
  N=(M_1\boxtimes \un_{G_2})*_{!} (\un_{G_1}\boxtimes M_2)= m_{12,!}(
  (M_1\boxtimes \un_{G_2})\boxtimes (\un_{G_1}\boxtimes M_2)).
  $$

  Let $s\colon G\to G$ be the involution given by
  $(x_1,x_2,y_1,y_2)\mapsto (x_1,y_2,y_1,x_2)$. We have
  $m_{12}=m_{12}\circ s$, and hence
  $$
  N=m_{12,!}s_!  ((M_1\boxtimes \un_{G_2})\boxtimes (\un_{G_1}\boxtimes
  M_2)).
  $$

  Since $s$ is an involution, we have $s_!=s_*=s^*$, and therefore
  $$
  s_!  ((M_1\boxtimes \un_{G_2})\boxtimes (\un_{G_1}\boxtimes
  M_2))=p_{1}^*(M_1\boxtimes M_2)\otimes p_2^*(\un_G),
  $$
  where $p_1$, $p_2\colon G\times G\to G$ are the two projections.
  Thus, using the definition again, we obtain an isomorphism
  $$
  N\simeq m_{12!}(p_{1}^*(M_1\boxtimes M_2)\otimes
  p_2^*(\un_G))=(M_1\boxtimes M_2)*_! \un_{G},
  $$
  which is isomorphic to $M_1\boxtimes M_2$ since~$\un_G$ is the unit
  for convolution.

  We obtain similarly the second isomorphism
  $$
  M_1\boxtimes M_2\simeq (M_1\boxtimes \un_{G_2})*_{*}
  (\un_{G_1}\boxtimes M_2)\simeq
  (M_1\boxtimes \un_{G_2})*_{*}
  (\un_{G_1}\boxtimes M_2).
  $$

  It is classical that $M_1\boxtimes M_2$ is perverse if $M_1$ and $M_2$
  are, and the final assertion then results from the fact that
  $$
  H^*_c(G_{\bar{k}},M_1\boxtimes M_2)\simeq
  H^*_c(G_{1,\bar{k}},M_1)\otimes H^*_c(G_{2,\bar{k}},M_2),
  $$
  and the generic vanishing theorem.
\end{proof}

\begin{remark}
  Concretely, this proposition reflects the convolution formula
  $$
  f_1(x_1)f_2(x_2)= \sum_{(y_1,y_2)\in G(k_n)}
  f_1(y_1)\delta_2(y_2)\delta_1(x_1y_1^{-1})f_2(x_2y_2^{-1})
  $$
  for any functions $f_i\colon G_i(k_n)\to \Cc$, where the $\delta_i$ are
  Dirac masses at the unit element of~$G_i$. 
\end{remark}

\section{The rank~$1$ tannakian group}

Given the group~$G$ over~$k$, we can form the subcategory $\Ppb^1(G)$ of
$\Ppb(G)$ (\resp $\Ppintun(G)$ of $\Ppint(G)$) additively generated by
objects of tannakian rank~$1$. This is again a tannakian category, since
the convolution (\resp dual) of objects of rank~$1$ is of
rank~$1$. 
\nomenclature[$P$]{$\Ppb^1(G)$}{subtannakian category generated by rank~$1$ objects}
\nomenclature[$P$]{$\Ppintun(G)$}{subtannakian category generated by rank~$1$ objects}

\begin{proposition}\label{pr-rank1}
  Let~$\mathbf{L}(G)$ be the group of isomorphism classes of objects
  of~$\Ppint(G)$ of rank~$1$.
  \nomenclature[$L$]{$\mathbf{L}(G)$}{group of isomorphism classes of
    objects of rank~$1$}
  The tannakian group of~$\Ppb^1(G)$ and
  $\Ppintun(G)$ is the pro-algebraic subgroup of
  $$
  \GL_1^{\mathbf{L}(G)}
  $$
  defined by the equations
  $$
  \prod_{1\leq i\leq m}x_{L_i}^{n_i}=1
  $$
  for all integers~$m\geq 0$, all families $(L_i)_{1\leq i\leq m}$ of
  elements of $\mathbf{L}(G)$ and all families $(n_i)_{1\leq i\leq m}$
  of integers such that the object
  $$
  L_1^{*_{\intt} n_1}*_{\intt}\cdots*_{\intt}L_m^{*_{\intt} n_m}
  $$
  is isomorphic to~$\un_G$, or equivalently that the relation
  $$
  L_1^{n_1}\cdots L_m^{n_m}=1
  $$
  holds in the group~$\mathbf{L}(G)$.
\end{proposition}

\begin{proof}
  This amounts to proving that, for the object
  $$
  M=L_1\oplus \cdots \oplus L_m,
  $$
  the tannakian group~$\Gg$ is the subgroup of $\GL_1^m$ determined by
  the equations
  \begin{equation}\label{eq-relation-2}
    \prod_{1\leq i\leq m}x_{i}^{n_i}=1
  \end{equation}
    for all~$(n_1,\ldots,n_m)\in\Zz^m$ such that
  \begin{equation}\label{eq-relation-1}
    L_1^{n_1}\cdots L_m^{n_m}=1.
  \end{equation}
  
  Being a group of multiplicative type, $\Gg$ is
  characterized by its character group, and the character group
  of~$\Gg\subset \GL_1^m$ is $ \Zz^m/H$, where
  $$
  H=\{\chi\colon \GL_1^m\to\Gg_m\,\mid\, \Gg\subset \ker(\chi)\}.
  $$
A character~$\chi$ of~$\GL_1^m$ restricts to a character
  of~$\Gg$, so we can form the object $\chi(M)$
  in~$\Ppintun(G)$. Since the image $\chi(\Gg)$ of~$\Gg$ by~$\chi$ is the tannakian group of~$\chi(M)$, we then have
  $$
  \Gg\subset \ker(\chi)\text{ if and only if } \chi(M)\simeq \un_G. 
  $$
   On the other hand, there is a natural
  isomorphism
  $$
  \chi(M)\simeq L_1^{*_{\intt}
    n_1}*_{\intt}\cdots*_{\intt}L_m^{*_{\intt} n_m}
  $$
  (indeed, this holds when $\chi$ is the character
  $(x_1,\ldots,x_m)\mapsto x_i$, and then the general case follows by
  the compatibility with convolution and tensor product, which for
  characters is just the ordinary product), and therefore $H$ is the
  subgroup of~$\Zz^m$ formed by the tuples $(n_1,\ldots,n_m)$ which
  satisfy~(\ref{eq-relation-1}). This means that~$\Gg$ coincides with
  the subgroup of~$\GL_1^m$ determined by the
  equations~(\ref{eq-relation-2}), and concludes the proof.
\end{proof}

\begin{remark}
  This result is in fact valid, with the same proof, in any tannakian
  category. In particular, the analogue holds for the categories
  generated by objects of rank~$1$ in $\Ppintarith(G)$. Of course, as
  usual, the arithmetic tannakian group may be bigger than the geometric
  one.
\end{remark}

\begin{example}\label{ex-gl-rank-1}
  If~$G$ is a torus, then the group~$\mathbf{L}(T)$ has been
  determined by Gabber and Loeser~\cite[Th\,8.6.1]{GL_faisc-perv}, who
  denote it $\Hh_{\intt}(G)$.
  \nomenclature[$H$]{$\Hh_{\intt}(T)$}{hypergeometric group of Gabber and Loeser}

  Precisely, let~$r\geq 0$ be such that~$G$ is isomorphic
  to~$\Gg_m^r$. Let $\mathcal{S}$ be the set of one-dimensional
  subtori of~$\Gg_{m,\bar{k}}^r$. For each torus~$T\in\mathcal{S}$,
  denote by~$i_T$ the closed immersion~$T\to \Gg_{m,\bar{k}}^r$ and
  choose an isomorphism $\varphi_T\colon\Gg_{m,\bar{k}}\to T$. Recall
  that $\Pi(\Gg_{m,\bar{k}},\bQl)$ is the set of continuous tame
  characters of~$\Gg_{m,\bar{k}}$ (see
  Section~\ref{sec:Fourier-Mellin}), and write the basis vectors of
  the free abelian group
  $\Zz^{(\mathcal{S}\times \Pi(\Gg_{m,\bar{k}},\bQl))}$) as
  $(T, \chi)$. Then Gabber and Loeser prove that there is an
  isomorphism
  $$
  (\bar{k}^{\times})^r\times \Zz^{(\mathcal{S}\times
    \Pi(\Gg_{m,\bar{k}},\bQl))}\to \mathbf{L}(\Gg_{m,\bar{k}}^r)
  $$
  that maps $(\lambda, (T,\chi))$ to the object
  $$
  \delta_{\lambda}*_{\intt} R(i_T\circ
  \varphi_T)_*(j^*\mcL_{\psi}\otimes \mcL_{\chi})[1],
  $$
  where~$j\colon\Gg_{m,\bar{k}}\to\Aa^1$ is the open immersion.
\end{example}


\chapter{Equidistribution theorems}
\label{sec:equidis}


\section{Equidistribution on average}

Along with the classical form of equidistribution that goes back in
principle to Weyl and appears in Deligne's equidistribution
theorem,\index{Deligne's equidistribution theorem} we
will apply a useful variant that allows us to avoid the assumption that
the geometric and the arithmetic tannakian groups are equal, at the cost
of getting slightly weaker statements.

\begin{definition}\label{def-equid}
  Let $X$ be a locally compact topological space and let~$\mu$ be a
  Borel probability measure on~$X$. Let $(Y_n,\Theta_n)_{n\geq 1}$ be a
  sequence of pairs of finite sets $Y_n$ and maps
  $\Theta_n\colon Y_n\to X$.
  \par
  \begin{enumerate}
  \item We say that $(Y_n,\Theta_n)$, or simply $(Y_n)$ when the maps
    $\Theta_n$ are clear from the context, \emph{becomes
      $\mu$\nobreakdash-equidistributed on average as $n\to \infty$}
    \index{equidistribution on average} if
    the sets $Y_n$ are non-empty for all large enough~$n$ and if the
    sequence of probability measures
    $$
    \mu_N=\frac{1}{N'}\sum_{\substack{1\leq n\leq N\\Y_n\not=\emptyset}}
    \frac{1}{|Y_n|}\sum_{y\in Y_n}\delta_{\Theta_n(y)},\quad\quad
    N'=|\{n\leq N\,\mid\, Y_n\not=\emptyset\}|,
    $$
    defined on~$X$ for large enough $N$, converges weakly to~$\mu$ as~$N$
    goes to infinity, i.e., for any bounded continuous
    function~$f\colon X\to\Cc$, the following holds:
    \begin{equation}\label{eqn:equidis-average}
      \lim_{N\to+\infty} \frac{1}{N'}
      \sum_{\substack{1\leq n\leq N\\Y_n\not=\emptyset}}\frac{1}{|Y_n|}
      \sum_{y\in Y_n}f(\Theta_n(y))=\int_{X}fd\mu.
    \end{equation}
    \par
  \item The sequence $(Y_n,\Theta_n)$, or simply $(Y_n)$, \emph{becomes
      $\mu$\nobreakdash-equidistributed as $n\to \infty$} if the sets
    $Y_n$ are non-empty for all large enough~$n$ and if the sequence of
    probability measures
    $$
    \widetilde{\mu}_n=
    \frac{1}{|Y_n|}\sum_{y\in Y_n}\delta_{\Theta_n(y)},
    $$
    defined on~$X$ for large enough $n$, converges weakly to~$\mu$
    as~$n$ goes to infinity, i.e., for any bounded continuous
    function~$f\colon X\to\Cc$, the following holds:
    \begin{equation}\label{eqn:equidis}
      \lim_{n\to+\infty} 
      \frac{1}{|Y_n|}
      \sum_{y\in Y_n}f(\Theta_n(y))=\int_{X}fd\mu.
    \end{equation}
  \end{enumerate}
\end{definition}

\begin{remark}
  (1) In practice, since $N'\sim N$ as $N\to+\infty$, we will sometimes
  not distinguish between $N$ and~$N$', and use the convention that
  those terms for which $Y_n$ is empty are omitted from the sum over $n$
  when discussing equidistribution on average.
  \par
  (2) Since convergence of a sequence $(x_n)$ of complex numbers implies
  that of its Cesàro means $(N^{-1}\sum_{1 \leq n \leq N} x_n)$, with
  the same limit, equidistribution implies equidistribution on average.
\end{remark}

\section{The basic estimate}

We state here a preliminary estimate that will be the key analytic step in
the proof of our equidistribution results, including
Theorem~\ref{thm:equidis-thm-intro} from the introduction.

We denote as usual by $k$ a finite field with algebraic closure
$\bar{k}$, and by $k_n$ the extension of $k$ of degree~$n$
in~$\bar{k}$. We fix a prime $\ell$ distinct from the characteristic of~$k$.

\begin{proposition}\label{pr-weyl-sum-bis}
  Let $G$ be a commutative connected algebraic group over~$k$.  Let~$M$
  be an $\ell$-adic perverse sheaf on~$G$ that is arithmetically
  semisimple and pure of weight zero, and of tannakian dimension~$r$. Let~$N$
  be an object of $\braket{M}^{\arith}$.

  For all $n\geq 1$ such that $\funram{N}(k_n)$ is not empty, the following
  estimate holds: 
  \begin{equation}\label{eq-pr-weyl-sum-bis}
    \frac{1}{|\funram{N}(k_n)|}\sum_{\chi \in
      \funram{N}(k_n)}
    \Tr(\Fr_{k_n}\mid\rmH^0_c(G_{\bar k},N_\chi)) =
    t_{N}(e;k_n)+O(\abs{k_n}^{-1/2}). 
  \end{equation}
\end{proposition}

\begin{proof}
  We fix a quasi-projective embedding~$u$ of~$G$.
  Let $d$ denote the dimension of $G$, and put $\mcX=\funram{N}$. For
  each non-zero integer $i$, consider the subset
  \[
    \mathcal{A}_{i}=\{\chi \in \charg{G} \mid \rmH^i_c(G_{\bar{k}},
    N_\chi) \neq 0\}
  \]
  consisting of those characters $\chi$ such that $N_\chi$ has
  non-trivial cohomology with compact support in degree $i$.  Then the
  left-hand side of \eqref{eq-pr-weyl-sum-bis} is equal to
  \begin{equation}
  \begin{aligned}\label{eqn:proof1bis}
    \frac{1}{|\mcX(k_n)|}&\sum_{\chi \in \mcX(k_n)}
    \Tr(\Fr_{k_n}\,|\,\rmH^0_c(G_{\bar{k}}, N_\chi)) =
    \\
    &\frac{1}{|\mcX(k_n)|}\sum_{\chi\in \charg{G}(k_n)}\sum_{|i|\leq d}
    (-1)^i \Tr(\Fr_{k_n}\,|\,\rmH^i_c(G_{\bar{k}}, N_\chi))
    \\
    &\hspace{1cm} -\frac{1}{|\mcX(k_n)|} \sum_{0<|i|\leq d} (-1)^i
    \sum_{\chi \in
      \mathcal{A}_i(k_n)}\Tr(\Fr_{k_n}\,|\,\rmH^i_c(G_{\bar{k}},
    N_\chi))
    \\
    &\hspace{2cm}-\frac{1}{|\mathscr{X}(k_n)|}\sum_{\chi\in
      \charg{G}(k_n) \setminus \mcX(k_n)}
    \Tr(\Fr_{k_n}\,|\,\rmH^0_c(G_{\bar{k}}, N_\chi)).
  \end{aligned}
  \end{equation}

  By the Grothendieck--Lefschetz trace formula
  (see~(\ref{eq-trace-formula})), the equalities
  \[
    \sum_{|i|\leq d} (-1)^i \Tr(\Fr_{k_n}\,|\,\rmH^i_c(G_{\bar{k}},
    N_\chi))=\sum_{x \in G(k_n)} t_{N_{\chi}}(x;k_n)
    =\sum_{x \in G(k_n)} \chi(x) t_N(x;k_n)
  \]
  hold for any character~$\chi$. Combined with the orthogonality of
  characters of $G(k_n)$, this shows that the first summand in
  \eqref{eqn:proof1bis} is equal to
  \[
    \frac{|\charg{G}(k_n)|}{|\mcX(k_n)|} t_N(e;k_n)=
    t_N(e;k_n)+O(|k_n|^{-1})
  \]
  since the set $\mcX$ is generic, so that the estimate
  $\frac{|\charg{G}(k_n)|}{|\mcX(k_n)|}=1+O(\abs{k_n}^{-1})$
  holds.
  
  We now turn to bounding the second and the third summands in the
  right-hand side of \eqref{eqn:proof1bis}.
  
  Since $M$ is pure of weight zero, the same holds for $N$ and
  $N_{\chi}$ by Theorem~\ref{thm-Harith-fib-funct}. It then follows from
  Deligne's Riemann Hypothesis (see Theorem~\ref{th-deligne-riemann})
  that $\rmH^i_c(G_{\bar{k}}, N_\chi)$ is mixed of weights $\leq i$ for
  any~$i$, in particular the eigenvalues of $\Fr_{k_n}$ acting on this
  space have modulus at most $|k_n|^{i/2}$. Moreover,
  using~\eqref{eqn:boundBetti} and Theorem
  \ref{thm-conductors}\,\ref{thm-conductors:item2}, we get
  \[
    h^i_c(G_{\bar{k}}, N_\chi) \leq c_u(N_{\chi}) \ll
    c_u(N) c_u(\mcL_\chi)\ll c_u(N)
  \]
  since the complexity $c_u(\mcL_\chi)$ is bounded independently of
  $\chi$ by Proposition \ref{prop-cond-char}. So the second term
  in~(\ref{eqn:proof1bis}) can be bounded by
  \[
    \frac{1}{|\mcX(k_n)|} \sum_{0<|i|\leq d} \ \sum_{\chi \in
      \mathcal{A}_i(k_n)} h^i_c(G_{\bar{k}}, N_\chi) |k_n|^{i/2}
    \leq  \frac{1}{|\mcX(k_n)|} \sum_{0<|i|\leq d}\  \sum_{\chi \in
      \mathcal{A}_i(k_n)} |k_n|^{i/2}.
  \]
  
  The Stratified Vanishing Theorem~\ref{thm-high-vanish} applied to
  $N$ gives the estimate
  \begin{equation}\label{eq-comp2bis}
    |\mathcal{A}_i(k_n)|\ll |k_n|^{d-|i|}
  \end{equation}
  for $i$ such that~$0<|i|\leq d$.  We split the sum over~$i$ into
  that over~$1\leq i\leq d$ and that over~$-d\leq i\leq -1$, and
  obtain
  \begin{align}
    \frac{1}{|\mcX(k_n)|} \sum_{0<|i|\leq d}\  \sum_{\chi \in
    \mathcal{A}_i(k_n)} |k_n|^{i/2}
    &\ll
      \frac{1}{|\mcX(k_n)|} \sum_{1\leq i\leq d}
      |k_n|^{d-i/2}
      +
      \frac{1}{|\mcX(k_n)|} \sum_{-d\leq i\leq -1}
      |k_n|^{d+3i/2}
      \notag
    \\
    &\ll \frac{|k_n|^{d-1/2}}{|\mcX(k_n)|}.
      \label{eq-comp3bis}
  \end{align}
 
  Thanks to the estimate $|\mcX(k_n)|=|k_n|^d+O(|k_n|^{d-1})$, the last
  term is~$\ll|k_n|^{-1/2}$ and tends to~$0$ as $n\to+\infty$.
  \par
  Finally, the third term in~(\ref{eqn:proof1bis}) satisfies
  \begin{equation}\label{eq-comp4bis}
    \frac{1}{|\mcX(k_n)|}\sum_{\chi\in \charg{G}(k_n)
      \setminus \mcX(k_n)} \Tr(\Fr_{k_n}\,|\,\rmH^0_c(G_{\bar{k}},
    N_\chi))
    \ll \frac{| \charg{G}(k_n) \setminus
      \mcX(k_n)|}{|\mcX(k_n)|} \ll \frac{1}{|k_n|}
  \end{equation}
  since $H^0_c(G_{\bar{k}},N_{\chi})$ is mixed of weights $\leq 0$ and
  has dimension bounded for all~$\chi$, and the set $\mcX$ is
  generic. This finishes the proof.
\end{proof}

\section{Equidistribution for characteristic
  polynomials}\label{sec-exp-sums}

Let~$k$ be a finite field, with an algebraic closure~$\bar{k}$, and let
$G$ be a connected commutative algebraic group over $k$. Let~$\ell$ be a
prime number distinct from the characteristic of~$k$.


Our most general equidistribution result concerns the characteristic
polynomials of the unitary Frobenius conjugacy classes for weakly
unramified characters. Equivalently, this is about the conjugacy classes
in the ambient unitary group.

\begin{theorem}\label{th-4}
  Let~$M$ be an $\ell$-adic perverse sheaf on~$G$ that is
  arithmetically semisimple and pure of weight zero. Let~$r\geq 0$ be
  the tannakian dimension of~$M$.  Let~$K \subset \Un_r(\Cc)$ be a
  conjugate of 
  a maximal compact subgroup of the arithmetic tannakian group
  $\garith{M}(\Cc)\subset \GL_r(\Cc)$ of~$M$, and denote by~$\nu_{cp}$
  \nomenclature[$nu$]{$\nu_{cp}$}{image of Haar measure on~$K$ on the
    space of conjugacy classes in $\Un_r(\Cc)^{\sharp}$} the measure
  on the space~$\Un_r(\Cc)^{\sharp}$\nomenclature{$K^{\sharp}$}{space
    of conjugacy classes in~$K$} of conjugacy classes in the unitary
  group which is the direct image of the Haar probability
  measure~$\mu$ on~$K$ by the quotient map $K\to
  \Un_r(\Cc)^{\sharp}$. Then the families of unitary conjugacy classes
  $(\Theta_{M,k_n}(\chi))_{\chi\in\wunram{M}(k_n)}$ become
  $\nu_{cp}$-equidistributed on average in~$\Un_r(\Cc)^{\sharp}$ as
  $n\to+\infty$.
\end{theorem}

\begin{remark}
  (1) To be precise, in terms of Definition~\ref{def-equid}, we consider
  the equidistribution on average of pairs $(\wunram{M}(k_n),\Theta_n)$
  with $\Theta_n(\chi)=\Theta_{M,k_n}(\chi)$.
  \par
  (2) The set $\Un_r(\Cc)^{\sharp}$ can be identified with the set of
  characteristic polynomials of unitary matrices of size~$r$, or
  equivalently with the quotient topological space
  $(\mathbf{S}_1)^r/\mathfrak{S}_r$ (by mapping a matrix to the set of
  eigenvalues, with multiplicity) so the statement means that the
  characteristic polynomials of the Frobenius automorphisms for weakly
  unramified characters tend to be distributed like the characteristic
  polynomials of random elements of~$K$ (hence the notation $\nu_{cp}$).
\end{remark}

\begin{proof}
  Let $\mcX=\wunram{M}$. It suffices to check the
  equality~(\ref{eqn:equidis-average}) for $f$ taken in a set of
  continuous functions on~$\Un_r(\Cc)^{\sharp}$ that span a dense subset
  of the Banach space $\mathcal{C}(\Un_r(\Cc)^{\sharp})$ of all continuous
  complex-valued functions on~$\Un_r(\Cc)^{\sharp}$ (since probability
  measures on $\Un_r(\Cc)^{\sharp}$ are continuous functionals on
  $\mathcal{C}(\Un_r(\Cc)^{\sharp})$ by the Riesz representation
  theorem). Thanks to the Peter\nobreakdash--Weyl
  Theorem,\index{Peter--Weyl theorem} it suffices
  to prove the equality
  $$
  \lim_{N\to +\infty}\frac{1}{N}\sum_{1\leq n\leq N}
  \frac{1}{|\mcX(k_n)|}\sum_{\chi\in\mcX(k_n)}
  \Tr(\rho(\Theta_{M,k_n}(\chi)))=\int_K \Tr(\rho(g))d\mu(g)
  $$
  for any irreducible unitary representation $\rho$
  of~$\Un_r(\Cc)$. In fact, we will prove this for any unitary
  representation~$\rho$, not necessarily irreducible.
  \par
  By the Peter--Weyl Theorem again, the right-hand side is the
  multiplicity of the trivial representation in the representation
  of~$\garith{M}$ that corresponds to the restriction of~$\rho$
  to~$K$. We denote by $N=\rho(M)$ the object of~$\braket{M}^{\arith}$
  that corresponds to this restriction of~$\rho$.
  \par
  Let $\mcX_N=\funram{N}_M$ be the set of Frobenius-unramified characters
  for~$N$. We have
  $$
  \lim_{N\to +\infty}\frac{1}{N}\sum_{1\leq n\leq N}
  \frac{1}{|\mcX(k_n)|}\sum_{\chi\in (\mcX\setminus \mcX_N)(k_n)}
  \Tr(\rho(\Theta_{M,k_n}(\chi)))=0,
  $$
  since $\mcX_N$ is generic (by Proposition~\ref{pr-frob-unram}) and the
  upper-bound
  $$
  |\Tr(\rho(\Theta_{M,k_n}(\chi)))|\leq \dim(\rho)
  $$
  holds for all~$\chi\in\mcX(k_n)$.
  \par
  By the definition of Frobenius-unramified characters, we have
  $$
  \frac{1}{|\mcX(k_n)|}\sum_{\chi\in\mcX_N(k_n)}
  \Tr(\rho(\Theta_{M,k_n}(\chi)))=
  \frac{1}{|\mcX(k_n)|}\sum_{\chi\in\mcX_N(k_n)} \Tr(\frob_{k_n}\mid
  H^0_c(G_{\bar{k}},N_{\chi}))
  $$
  for $n\geq 1$. 
  Since $\mcX$ and $\mcX_N$ are both generic, we have
  $\frac{|\mcX_N(k_n)|}{|\mcX(k_n)|}=1+O(\frac{1}{\abs{k_n}})$.  By
  Proposition~\ref{pr-weyl-sum-bis}, we deduce that
  $$
  \frac{1}{|\mcX(k_n)|}\sum_{\chi\in\mcX_N(k_n)}
  \Tr(\rho(\Theta_{M,k_n}(\chi))) = t_{N}(e;k_n) +O(|k_n|^{-1/2}),
  $$
  where~$e$ is the identity of~$G$.

  We decompose the semisimple perverse sheaf~$N$ as a direct sum
  $$
  N=\bigoplus_{r\geq 0}\bigoplus_{i\in I(r)} N_{r,i}
  $$ of pairwise non-isomorphic arithmetically simple
  perverse sheaves $N_{r,i}$ of support of dimension~$r$. For $r\geq 1$,
  we get the pointwise bound
  $$
  t_{N_{r,i}}(e;k_n)\ll \frac{1}{\sqrt{|k_n|}}.
  $$
  using Proposition~\ref{pr-pointwise-bound}.
  \par
  The punctual objects $N_{0,i}$ are of the form
  $\alpha_i^{\deg}\otimes\delta_{x_i}$ for some unitary scalars
  $\alpha_i$ and some points~$x_i$. If~$x_i\not=e$, then
  $$
  t_{N_{0,i}}(e;k_n)=0.
  $$
  \par
  Thus, if we denote by $J\subset I(0)$ the subset where $x_i=e$
  (which has cardinality equal to the multiplicity of the trivial
  representation in the restriction of~$\rho$ to $\ggeo{M}$), then the
  formula
  \begin{equation}\label{eqn:computation-traces-proof}
    \frac{1}{|\mcX(k_n)|}\sum_{\chi\in\mcX_N(k_n)}
    \Tr(\rho(\Theta_{M,k_n}(\chi))) = \sum_{i\in J} \alpha_i^n
    +O(|k_n|^{-1/2})
  \end{equation}
  holds.  The subset~$J^0\subset J$ where $\alpha_i=1$ has cardinality
  equal to the multiplicity of the trivial representation in the
  restriction of~$\rho$ to~$\garith{M}$. Averaging over~$n$ and using
  $$
  \lim_{N\to+\infty}\frac{1}{N}\sum_{1\leq n\leq N} \alpha_i^n=0
  $$
  for $i\in J\setminus J^0$, we conclude that
  $$
  \lim_{N\to +\infty}\frac{1}{N}\sum_{1\leq n\leq N}
  \frac{1}{|\mcX(k_n)|}\sum_{\chi\in\mcX_N(k_n)}
  \Tr(\rho(\Theta_{M,k_n}(\chi))) = |J^0|+O(|k_n|^{-1/2}),
  $$
  which gives the desired result.
\end{proof}


It is useful to state the following corollary of the proof, which is a
diophantine version of Schur's Lemma\index{Schur's Lemma} in our context.

\begin{corollary}[Schur's Lemma]\label{cor-schur}
  Let~$M$ and~$N$ be geometrically simple $\ell$-adic perverse sheaves
  on~$G$ which are pure of weight zero and are objects of~$\Ppintarith(G)$.  Let $\mcX$ be the set of characters which are
  weakly unramified for $M\oplus N^{\vee}$. We have
  $$
  \lim_{N\to +\infty}\frac{1}{N} \sum_{n\leq N} \frac{1}{|G(k_n)|}
  \sum_{\chi\in\mcX(k_n)}S(M *_{\intt} N^{\vee},\chi)=
  \begin{cases}
    1& \text{if $M$ is arithmetically isomorphic to~$N$},
    \\
    0&\text{ otherwise.}
  \end{cases}
  $$
\end{corollary}

\begin{proof}
  Proposition~\ref{pr-weyl-sum-bis} applied to the perverse sheaf
  $M\oplus N^{\vee}$ and the object $Q=\Hom(N,M)$ of the category
  $\braket{M\oplus N^{\vee}}^{\arith}$ (the homomorphisms are in the
  category $\Ppintarith(G)$) implies that
  $$
  \frac{1}{|\funram{Q}(k_n)|}\sum_{\chi \in \funram{Q}(k_n)}
  S(Q,\chi) =
  t_{Q}(e;k_n)+O(\abs{k_n}^{-1/2})
  $$
  for any $n\geq 1$, where
  $$
  S(Q,\chi)=\sum_{x\in G(k_n)}\chi(x)t_Q(x;k_n).
  $$
  \par
  Since $\funram{Q}$ is generic, and since there is a canonical
  isomorphism $Q\to M*_{\intt} N^{\vee}$, we deduce that
  $$
  \frac{1}{N} \sum_{n\leq N} \frac{1}{|G(k_n)|}
  \sum_{\chi\in\mcX(k_n)}S(M *_{\intt} N^{\vee},\chi) =
  \frac{1}{N}\sum_{n\leq N}t_{Q}(e;k_n)+O(\abs{k_n}^{-1/2})
  $$
  for all $N\geq 1$. Arguing as in the last part of the proof of
  Theorem~\ref{th-4}, we see that the right-hand side converges to the
  multiplicity of the trivial representation in the representation
  corresponding to~$Q$; by the classical form of Schur's Lemma, this is
  either~$1$ or~$0$, depending on whether $M$ is isomorphic to~$N$
  or~not.
\end{proof}

\begin{remark}\label{rm-smoothing}
  The proof of Theorem~\ref{th-4} allows us to see clearly what is
  involved in the use of the Cesàro mean\index{Cesàro mean} in the average
  equidistribution.

  First, we can see that it is necessary in general, unless
  $\garith{M}=\ggeo{M}$ (see Section~\ref{sec-without-average} for
  statements under this assumption, in particular
  Proposition~\ref{pr-converse-equality}).

  Second, we see that the use of the Cesàro average can be generalized
  to establish the convergence to the limit $\nu_{cp}$ of any sequence
  of average measures of the form
  $$
  \sum_{n\geq 1}\frac{\varphi_N(n)}{|\mcX(k_n)|}\sum_{\chi\in \mcX(k_n)}
  \delta_{\Theta_{M,k_n}(\chi)}, 
  $$
  where $\varphi_N(n)$ are non-negative coefficients that are bounded
  and satisfy the equality
  \begin{equation}\label{eq-cesaro-enough}
    \lim_{N\to +\infty}
    \sum_{n\geq 1}\varphi_N(n)\alpha^n
    =
    \begin{cases}
      0&\text{ if } \alpha\not=1,\\
      1&\text{ if } \alpha=1
    \end{cases}
  \end{equation}
  for any complex number $\alpha$ of modulus~$1$. The Cesàro case
  corresponds to $\varphi_N(n)=1/N$ for all~$n\leq N$ and
  $\varphi_N(n)=0$ for $n\geq N$, but there are many other
  possibilities. (In classical terms, as expounded for instance by
  Hardy~\cite{hardy}, these $\varphi_N$ define a ``summation
  method'',\index{summation method}
  and it is elementary that the requirements amounts
  essentially\footnote{\ Precisely, we need that the series $\sum a_n$
    with $a_1=\alpha$ and $a_{n}=\alpha^n-\alpha^{n-1}$ for $n\geq 2$ has
    ``sum'' $\alpha+(\alpha-1)/(1-\alpha)=0$ for $|\alpha|=1$
    and~$\alpha\not=1$.} to asking that this summation method gives the
  ``right'' sum $1/(1-\alpha)$ to the geometric series for~$|\alpha|=1$
  and~$\alpha\not=1$.)

  It is also instructive to view the average probabilistically,
  interpreting $\varphi_N$ as the law of a random variable $X_N$ with
  values in positive integers. The condition above is the requirement
  that the equality
  $$
  \lim_{N\to +\infty} \expect(e^{i\theta X_N})=0
  $$
  holds for all $\theta\in \Rr/2\pi\Zz\setminus\{0\}$.
  \par
  Besides the Cesàro case, where $X_N$ is a random variable uniform
  on~$\{1,\ldots,N\}$, consider a Poisson distribution~$X_N$ with
  parameter $\lambda_N>0$, shifted to have support in the positive
  integers, i.e., let
  $$
  \proba(X_N=n)=\varphi_N(n)=e^{-\lambda_N}\frac{\lambda_N^{n-1}}{(n-1)!}
  $$
  for any positive integers $N$ and~$n$.  The condition above becomes
  the limit
  $$
  \expect(e^{i\theta X_N})= \exp(i\theta+\lambda_N(e^{i\theta}-1))\to 0
  $$
  as $N\to+\infty$ for $\theta\in\Rr/2\pi\Zz\setminus\{0\}$, which holds
  provided $\lambda_N\to +\infty$, since the modulus of the left-hand
  side is $\exp(\lambda_N(\cos(\theta)-1))$.
  \par
  Intuitively, this means that if we pick a positive integer $n$
  according to a Poisson distribution with large parameter, then pick
  uniformly a random $\chi\in\mcX(k_n)$, then the Frobenius conjugacy
  class $\Theta_{M,k_n}(\chi)$ will be distributed like a random
  $\Un_r(\Cc)$-conjugacy class of an element of the maximal compact
  subgroup~$K$. (A whimsical enough way to do this -- according to the
  Rényi--Turan form of the Erd\H os--Kac Theorem, see
  e.g.~\cite[Prop.\,4.14]{mod-gaussian} -- would be to pick a large integer
  $m\geq 1$ and to take $n$ to be the number of prime factors of~$m$,
  which corresponds roughly to having $\lambda_N=\log\log N$.)
  \par
  Note however that are also many cases where the
  condition~(\ref{eq-cesaro-enough}) is not true. The most obvious is
  when $\varphi_N(N)=1$ and $\varphi_N(n)=0$ for $n\not=N$,
  corresponding to a limit without extra average at all. In addition,
  the condition implies that for any integers $q\geq 1$ and $a\in \Zz$,
  we have
  $$
  \proba(X_N\equiv a\mods{q})=\frac{1}{q}\sum_{b\mods{q}}
  e^{-2i\pi ab/q}\expect(e^{2i\pi bX_N/q})\to \frac{1}{q},
  $$
  so there is a strong arithmetic restriction that $X_N\mods{q}$
  converge to the uniform probability measure modulo~$q$ for all~$q\geq
  1$.
  \par
  Similar remarks apply in an obvious manner to our other
  equidistribution statements, e.g. to
  Theorem~\ref{thm:equidis-thm-intro}.
\end{remark}

\section{Equidistribution for arithmetic Fourier transforms}

We now deduce from Theorem~\ref{th-4} the equidistribution of the
exponential sums defined by
$$
S(M,\chi)=\sum_{x\in G(k_n)}\chi(x)t_M(x;k_n).
$$
\par
In fact, note that these sums make sense for \emph{all} characters
$\chi\in\charg{G}(k_n)$, and we can indeed prove equidistribution for
all of them. This implies Theorem~\ref{thm:equidis-thm-intro} from the
introduction. As a final addition, we prove an equidistribution
statement for the arithmetic Fourier transforms of all objects $M$ of
$\Der(G)$ which are mixed semiperverse sheaves of weights $\leq 0$. This
is of interest especially in more analytic applications, since the
condition of being semiperverse and that of being mixed of weights
$\leq 0$ are much more flexible, and easier to check, than those of
being perverse and pure.

\begin{theorem}\label{th-3}
  Let~$k$ be a finite field and let $G$ be a connected commutative
  algebraic group over $k$. Let~$\ell$ be a prime number distinct from
  the characteristic of~$k$.
  \par
  Let~$M$ be an object of $\Der(G)$. Assume that $M$ is semiperverse
  and mixed of
  weights $\leq 0$. Let~$N$ be the maximal perverse subsheaf of weight~$0$ of
  the arithmetic semisimplification of the perverse cohomology sheaf
  $\pH^0(M)$.
  \par
  Let~$r\geq 0$ be the tannakian dimension of~$N$.
  Let~$K \subset \garith{N}(\Cc)\subset \GL_r(\Cc)$ be a maximal compact
  subgroup of the arithmetic tannakian group of~$N$. Denote by~$\mu$ the
  Haar probability measure on~$K$ and by $\nu$ its image by the trace.
  \par
  The families of exponential sums $S(M,\chi)$ for
  $\chi\in\charg{G}(k_n)$ become $\nu$-equidistri\-buted on average as
  $n\to+\infty$.

\end{theorem}

\begin{proof}
  Up to conjugation, we may assume that~$K\subset\Un_r(\Cc)$.
  
  We first assume that $M$ is perverse and pure of weight~$0$, so that
  the object~$N$ coincides with~$M$. We then observe that, by the
  generic vanishing theorem, it suffices to prove that the families of
  exponential sums associated to $\chi\in \wunram{M}$ become
  $\nu$-equidistributed on average, since for any bounded continuous
  function $f\colon\Cc\to\Cc$, we have
  $$
  \Bigl| \frac{1}{|\charg{G}(k_n)|} \sum_{\chi\in (\charg{G}\setminus
    \wunram{M})(k_n)}f(\Tr(\Theta_{M,k_n}(\chi)))\Bigr| \leq
  \|f\|_{\infty} \frac{|(\charg{G}\setminus \wunram{M})(k_n)|}
  {|\charg{G}(k_n)|}\to 0
  $$
  because $\wunram{M}$ is generic. But since $\Tr_*(\nu_{cp})=\nu$, 
  this equidistribution follows from Theorem~\ref{th-4} by considering
  the composition $ K\to \Un_{r}(\Cc)^{\sharp}\fleche{\Tr}\Cc$.
  \par
  We now consider the general case.  We denote by $M_0$ the arithmetic
  semisimplification of the perverse sheaf $\pH^0(M)$, and by $N'$ the
  perverse sheaf such that $M_0=N\oplus N'$, defined using the weight
  filtration on~$M_0$; the perverse sheaf~$N'$ is mixed of weights
  $\leq -1$.
  \par
  Since $M$ is semiperverse of weights~$\leq 0$, we have $\pH^i(M)=0$
  for $i\geq 1$, and $\pH^{-i}(M)$ is of weights $\leq -i\leq -1$ for
  all $i\geq 1$ (see~\cite[Th.\,5.4.1]{BBD-pervers}).
  \par
  For any $n\geq 1$ and $\chi\in\charg{G}(k_n)$, we have the equality
  \begin{equation}\label{eq-semip-equid}
    S(M,\chi)=S(N,\chi)+S(N',\chi)
    +\sum_{i\geq 1}(-1)^iS(\pH^{-i}(M),\chi)
  \end{equation}
  by~(\ref{eq-decomp-phi}).
  \par
  By generic vanishing and the trace formula (see
  Theorem~\ref{th-stratified} below, applied to $N'(-1/2)$ and
  $\pH^{-i}(M)(-1/2)$ for $i\geq 1$, which are mixed perverse sheaves of
  weights $\leq 0$), there exists a generic subset
  $\mcX\subset \charg{G}$ such that we have
  \begin{equation}\label{eq-semip-negl}
    S(N',\chi) +\sum_{i\geq 1}(-1)^iS(\pH^{-i}(M),\chi)\ll
    \frac{1}{|k_n|^{1/2}}
  \end{equation}
  for all $n\geq 1$ and $\chi\in\mcX(k_n)$. This implies that the
  sequence $(\varpi_n)$ of probability measures defined as averages of
  delta masses at the points
  $$
  S(N',\chi) +\sum_{i\geq 1}(-1)^iS(\pH^{-i}(M),\chi)
  $$
  for all $\chi\in\charg{G}(k_n)$ converges to zero in probability,
  i.e., that for any fixed real number $\eps>0$, the limit
  $$
  \lim_{n\to +\infty} \varpi_n(\{|t|>\eps\})=0
  $$
  holds.
  \par
  By the first case applied to the perverse sheaf~$N$, the sums
  $S(N,\chi)$ become $\nu$-equidistributed on average as $n\to +\infty$,
  and the formula~(\ref{eq-semip-equid})
  ensures then that the same holds for the $S(M,\chi)$ (see,
  e.g.,~\cite[Cor.\,B.4.2]{pnt} for the simple probabilistic argument
  that leads to this conclusion).
\end{proof}

\begin{remark}
  (1) As we will see later, it is often of interest to attempt to apply
  equidistribution of exponential sums to the test function
  $z\mapsto z^m$ or $z\mapsto |z|^m$ for some integer~$m\geq 1$. Such
  functions are continuous but not bounded on~$\Cc$, so that
  Theorem~\ref{th-3} does not apply, and Theorem~\ref{th-4} only gives
  the equidistribution for weakly unramified characters. In these
  attempts, the contribution of the other characters may therefore need
  to be handled separately (see for instance the proof of
  Theorem~\ref{th-sld}).
  \par
  (2) See Chapter~\ref{sec-indep} for an application of this theorem to
  a question of independence of~$\ell$ of tannakian groups.
  \par
  (3) The measure $\nu$ is also the image by the trace of
  measure~$\nu_{cp}$ on characteristic polynomials appearing in
  Theorem~\ref{th-4}. It is often called the \emph{Sato--Tate measure}
  associated to~$M$.\index{Sato--Tate measure}
\end{remark}

\begin{example}
  Let $k=\Ff_p$, and let $\psi$ be the additive character on~$k$ such
  that $\psi(x)=e(x/p)$ for $x\in k$.  Let $X\subset G$ be a
  locally-closed subvariety of~$G$ of dimension~$d\geq 1$, and let
  $f\colon X\to \Aa^1$ be a non-zero function on~$X$. Then there is a
  semiperverse sheaf~$M$ on~$G$, mixed of weights~$0$, such that the
  trace function of~$M$ is given by the formula
  $$
  t_M(x;\Ff_{p^n})=
  \begin{cases}
    (-1)^dp^{-nd/2}e(\Tr_{\Ff_{p^n}/\Ff_p}(f(x))/p)&\text{ if }
    x\in X(\Ff_{p^n})\\
    0&\text{ otherwise.}
  \end{cases}
  $$
  for $n\geq 1$ and $x\in G(k_n)$, namely
  $$
  M=j_!f^*\mcL_{\psi}[d](d/2),
  $$
  where $j\colon X\to G$ is the natural immersion.
  \par
  Hence Theorem~\ref{th-3} implies that the exponential sums
  $$
  \frac{1}{p^{nd/2}}\sum_{x\in X(\Ff_{p^n})}
  \chi(x)e\Bigl(\frac{f(x)}{p}\Bigr)
  $$
  for $\chi\in\charg{G}(\Ff_{p^n})$ \emph{always satisfy some}
  equidistribution theorem on average.
  \par
  A similar property holds if we fix a non-trivial multiplicative
  character $\eta$ of $\Ff_p^{\times}$ and an invertible function
  $g\colon X\to \Gg_m$, and consider the exponential sums
  $$
  \frac{1}{p^{nd/2}}\sum_{x\in X(\Ff_{p^n})} \chi(x)\eta(g(x))
  $$
  (using the object $j_!g^*\mcL_{\eta}[d](d/2)$, which is also mixed and
  semiperverse of weights $\leq 0$).
\end{example}
  
\section{Equidistribution for conjugacy classes}

We keep the notation of the previous sections. If the object~$M$ that we
consider is generically unramified, then we can prove equidistribution
at the level of the Frobenius conjugacy classes in the maximal compact
subgroup of the arithmetic tannakian group.


\begin{theorem}[Equidistribution on average]\label{th-2}
  Let~$k$ be a finite field and let $G$ be a connected
  commutative algebraic group over $k$. Let~$\ell$ be a prime number
  distinct from the characteristic of~$k$.
  \par
  Let~$M$ be an $\ell$-adic perverse sheaf on~$G$ that is
  arithmetically semisimple, pure of weight zero and generically
  unramified. Let $\mathcal{X}=\unram{M}$ be the set of unramified
  characters for~$M$.  Let~$K$ be a maximal compact subgroup of the
  arithmetic tannakian group $\garith{M}(\Cc)$ of~$M$, and denote
  by~$\mu^{\sharp}$ the direct image of the Haar probability
  measure~$\mu$ on~$K$ by the projection to the set~$K^{\sharp}$ of
  conjugacy classes of~$K$.
  \par
  The families of unitary Frobenius conjugacy classes
  $(\Thetaf_{M,k_n}(\chi))_{\chi\in\mathcal{X}(k_n)}$ become
  $\mu$-equidistri\-buted on average in~$K^{\sharp}$ as $n\to+\infty$.
\end{theorem}

Precisely, we are considering here the equidistribution on average of
the pairs $(\mathcal{X}(k_n),\Theta_n)$ where
$\Theta_n(\chi)=\Thetaf_{M,k_n}(\chi)$.

\begin{proof}
  By Theorem~\ref{thm-arith_cat_neut_tan} and the definition of generic
  sets, we know that $|\mathscr{X}(k_n)|\sim |G(k_n)|$ as $n\to+\infty$, and
  hence the sets of unramified conjugacy classes are non-empty for $n$
  large enough.
  \par
  By the Peter-Weyl theorem, any continuous central function
  $f \colon K \to \Cc$ is a uniform limit of linear combinations of
  characters of finite-dimensional unitary irreducible representations
  of~$K$, and hence it suffices to prove the formula
  \eqref{eqn:equidis-average} when~$f$ is such a character. For the
  trivial representation, both sides are equal to $1$. If the
  representation is non-trivial, then the integral on the right-hand
  side vanishes, and we are reduced to showing that the limit on the
  left-hand side exists and is equal to~$0$. We thus consider a
  non-trivial irreducible representation $\rho$ of~$K$, which we
  identify with a non-trivial irreducible algebraic
  $\Qlb$-representation of the arithmetic tannakian group $\garith{M}$
  by Weyl's unitarian trick (see, e.g.,~\cite[3.2]{gkm} for this step);
  applying the next proposition then completes the proof.
\end{proof}


\begin{proposition}\mbox{}\label{pr-weyl-sum}
  With notation as in Theorem~\emph{\ref{th-2}}, let~$\rho$ be a
  non-trivial irreducible unitary representation of~$K$, identified with
  a non-trivial irreducible representation of $\garith{M}$.
  \begin{enumth}
  \item If the restriction of $\rho$ to $\ggeo{M}$ is non-trivial, then
    \begin{equation}\label{eqn:nontrivialKg}
      \frac{1}{|\mathscr{X}(k_n)|}\sum_{\chi \in \mathscr{X}(k_n)}
      \Tr(\rho(\Thetaf_{M,k_n}(\chi))) \ll \frac{1}{\sqrt{|k_n|}}
    \end{equation}
    for all $n$ such that $\mathscr{X}(k_n)$ is not empty.
  \item If the restriction of $\rho$ to $\ggeo{M}$ is trivial, then
    \begin{equation}\label{eqn:trivialKg}
      \lim_{N \to +\infty}
      \frac{1}{N}
      \sum_{\substack{1 \leq n \leq N\\\mathcal{X}(k_n)\not=\emptyset}}
      \frac{1}{|\mathscr{X}(k_n)|}
        \sum_{\chi \in \mathscr{X}(k_n)} \Tr(\rho(\Thetaf_{M,k_n}(\chi)))=0.
    \end{equation}
  \end{enumth}
\end{proposition}

\begin{proof}
  (1) We assume that the restriction of~$\rho$ to the geometric
  tannakian group is non-trivial.

  Let $\rho(M)$ denote the object of the tannakian category
  $\braket{M}^{\arith}$ corresponding to the representation $\rho$ of
  the arithmetic tannakian group $\garith{M}$; this is a simple perverse
  sheaf on~$G$.

  We have $\mcX\subset \funram{\rho(M)}$. Applying
  Proposition~\ref{pr-weyl-sum-bis} to the object~$N=\rho(M)$, we obtain
  $$
  \frac{1}{|\mcX(k_n)|}\sum_{\chi \in \mcX(k_n)}
  \Tr(\rho(\Thetaf_{M,k_n}(\chi))) =
  t_{\rho(M)}(e;k_n)+O(\abs{k_n}^{-1/2})
  $$
  since the conjugacy class $\Thetaf_{M,k_n}(\chi)$ coincides with
  $\Theta_{M,k_n}(\chi)$ when $\chi$ is unramified for~$M$.
  
  Since $\rho(M)$ is a simple perverse sheaf on $G$, the classification
  of \cite[Th.\,4.3.1\,(ii)]{BBD-pervers} shows that there exist an
  irreducible closed subvariety $s \colon Y \to G$ of dimension~$r$, an
  open dense smooth subvariety $j \colon U \to Y$, and an irreducible
  lisse $\Qlb$-sheaf $\mcF$ on $U$ such that
  $\rho(M)=s_\ast j_{!\ast}\mcF[r]$. Since the functors $s_\ast$ and
  $j_{!\ast}$ are weight-preserving, the sheaf $\mcF$ is pure of
  weight~$-r$.

  If $r=0$, then $Y$ consists of a closed point of $G$, which must be
  different from the neutral element~$e$, since otherwise $\rho(M)$
  would be geometrically trivial, contrary to the assumption in~(1). In
  that case, we have therefore $t_{\rho(M)}(e;k_n)=0$. On the other
  hand, if $r\geq 1$ we get
  \begin{equation}\label{eq-comp1}
    t_{\rho(M)}(e;k_n)
    \ll \frac{1}{\sqrt{|k_n|}}
  \end{equation}
  (by Proposition~\ref{pr-pointwise-bound}), which concludes the proof
  of~(1).
  \par
  (2) We assume that the restriction of the representation~$\rho$
  to~$\ggeo{M}$ is trivial.  Then $\rho$ has dimension~$1$ since the
  quotient $\garith{M}/\ggeo{M}$ is abelian
  (Proposition~\ref{pr:geom-vs-arith2}).

  Let $Q$ be the set of integers $n\geq 1$ such that $\mathcal{X}(k_n)$
  is not empty; it contains all sufficiently large integers.  It follows
  from Proposition~\ref{pr:geom-vs-arith2} that there exists an element
  $\xi$ of $\garith{M}/\ggeo{M}$, generating a Zariski-dense subgroup of
  this group, such that $\rho(\Thetaf_{M,k_m}(\chi))=\rho(\xi)^n$ for any
  $n\geq 1$ and any $\chi\in\charg{G}(k_n)$ unramified
  for~$M$. Moreover, we have $\rho(\xi)\not=1$, since otherwise the
  representation~$\rho$ would be trivial.  We conclude that
  $$
  \frac{1}{N} \sum_{\substack{1\leq n\leq N\\\mcX(k_n)\not=\emptyset}}
  \frac{1}{|\mathcal{X}(k_n)|} \sum_{\chi\in \mathscr{X}(k_n)}
  \Tr(\rho(\Thetaf_{M,k_n}(\chi))) = \frac{1}{N} \sum_{\substack{1\leq
      n\leq N\\ n\in Q}} \rho(\xi)^n
  $$
  converges to~$0$ as $N\to+\infty$ by summing a non-trivial geometric
  progression.


  %
  \par

\end{proof}

\begin{remark}
  For certain reductive groups $\Gg\subset \GL_r(\Cc)$, a conjugacy
  class in a maximal compact subgroup $K$ of~$\Gg$ is determined by
  its characteristic polynomial (equivalently, the exterior powers of
  the standard representation\index{standard representation} generate the representation ring
  of~$\Gg$).  If $\garith{M}(\Cc)$ has this property, then
  Theorem~\ref{th-4} implies a version of Theorem~\ref{th-2}, even
  if~$M$ is not generically unramified.
  \par
  If $\Gg$ is semisimple, this property holds, for instance,
  for~$\SL_r(\Cc)\subset \GL_r(\Cc)$, for
  $\Sp_{2r}(\Cc)\subset \GL_{2r}(\Cc)$, and for
  $\Gg_2(\Cc)\subset\GL_7(\Cc)$. Indeed, the first two cases are
  explained by Katz in~\cite[Lemma\,13.1, Remark\,13.2]{gkm}; in the
  third case, we note that the second fundamental representation
  of~$\Gg_2(\Cc)$ is virtually $\bigwedge^2\Std-\Std$ (see,
  e.g.,~\cite[p.\,353]{fulton-harris}) so that the exterior powers of
  the standard $7$-dimensional representation generate the
  representation ring.
\end{remark}

We deduce immediately from Theorem~\ref{th-2} a useful corollary,
analogue to some classical consequences of the Chebotarev density
theorem.\index{Chebotarev density theorem}

\begin{corollary}\label{cor-generate}
  Let~$k$ be a finite field and let $G$ be a connected commutative
  algebraic group over~$k$. Let~$M$ be a perverse sheaf on~$G$ which
  is arithmetically semisimple, pure of weight zero and generically
  unramified.
  \par
  Let~$S$ be any finite subset of~$\charg{G}$. The union of the unitary
  Frobenius conjugacy classes of~$M$ associated to unramified characters
  in $\charg{G}\setminus S$ is dense in a maximal compact subgroup
  of~$\garith{M}(\Cc)$.
\end{corollary}

\section{Equidistribution without average}
\label{sec-without-average}

We continue again with the previous notation.
If we make the extra assumption that the geometric and the arithmetic
tannakian groups coincide, then the equidistribution of Frobenius
conjugacy classes holds without averaging over~$n$. We summarize the
variants of the previous theorems in this situation.

\begin{theorem}[Equidistribution without average]
  \label{th-2bis}
  Let~$M$ be an $\ell$-adic perverse sheaf on~$G$ that is arithmetically
  semisimple, pure of weight zero. We assume that the inclusion
  $\ggeo{M} \subset \garith{M}$ is an equality.
  \par
  Let~$r\geq 0$ be the tannakian dimension of~$M$.
  Let~$K \subset \Un_r(\Cc)$ be a conjugate of 
  a maximal compact subgroup of the arithmetic tannakian group
  $\garith{M}(\Cc)\subset \GL_r(\Cc)$ of~$M$.
  Denote
  by~$\mu$ the Haar probability measure on~$K$, by~$\nu_{cp}$ its
  direct image by the map $K\to \Un_r(\Cc)^{\sharp}$, by $\nu$ its
  image by the trace, and by~$\mu^{\sharp}$ its image by the map
  $K\to K^{\sharp}$.  \nomenclature[$mu$]{$\mu^{\sharp}$}{image of
    Haar measure of~$K$ on~$K^{\sharp}$}
  \par
  \begin{enumth}
  \item The families of unitary Frobenius conjugacy classes
    $(\Theta_{M,k_n}(\chi))_{\chi\in\wunram{M}(k_n)}$ become
    $\nu_{cp}$-equidistri\-buted as $n\to+\infty$.
  \item The families of exponential sums $S(M,\chi)$ for
    $\chi\in\charg{G}(k_n)$ become
    $\nu$-equidistri\-buted as $n\to+\infty$.
  \item If $M$ is generically unramified, then the family of conjugacy
    classes~$(\Thetaf_{M,k_n}(\chi))_{\chi\in\unram{M}(k_n)}$ become
    $\mu^{\sharp}$-equidistributed as $n$ goes to infinity.
  \end{enumth}
\end{theorem}

\begin{proof}
  This follows from the Weyl Criterion\index{Weyl criterion} as in the proof of
  Theorems~\ref{th-4}, \ref{th-3} and~\ref{th-2}; in the case of the
  last statement, for instance, we use only the first part of
  Proposition~\ref{pr-weyl-sum} (as we may since a non-trivial
  irreducible representation of~$\garith{M}$ is a non-trivial
  irreducible representation of~$\ggeo{M}$ under the assumption).
\end{proof}

\begin{remark}
  There is an obvious further variant of Theorems~\ref{th-2bis} and of
  the case of mixed semiperverse objects of weights $\leq 0$
  of~\ref{th-3}: if~$M$ is mixed semiperverse of weights $\leq 0$,
  with~$N$ as in Theorem~\ref{th-3} such that $\garith{N}=\ggeo{N}$,
  then the discrete Fourier transform becomes equidistributed towards
  the measure~$\nu$ without average over~$n$.
\end{remark}

There is a converse to Theorem~\ref{th-2bis}. In fact, there is a
statement which is valid for an individual representation of the unitary
group (this will be useful in Chapter~\ref{sec-larsen}).

\begin{proposition}\label{pr-conv-weyl-geo}
  Let~$M$ be an $\ell$-adic perverse sheaf on~$G$ that is arithmetically
  semisimple and pure of weight zero. Let~$r$ be the tannakian dimension
  of~$M$ and let $\mcX=\mcX_w(M)$ be the set of weakly unramified
  characters for~$M$. Let $\rho$ be a finite-dimensional unitary
  representation of~$\Un_r(\Cc)$. Assume that the sequence
  $$
  \frac{1}{|\mcX(k_n)|}\sum_{\chi\in\mcX(k_n)}
  \Tr(\rho(\Theta_{M,k_n}(\chi))),
  $$
  defined for all integers $n\geq 1$ such that $\mcX(k_n)$ is not
  empty, has a limit. Then this limit is equal to the multiplicity of
  the trivial representation in the restriction of $\rho$
  to~$\ggeo{M}$, and the latter equals the multiplicity of the trivial
  representation in~$\rho$.
\end{proposition}

\begin{proof}
  We use the notation in the proof of Theorem~\ref{th-4}. Taking the
  equality \eqref{eqn:computation-traces-proof} into account, the
  assumption of the statement means that the limit
  $$
  \lim_{n\to+\infty} \sum_{i\in J} \alpha_i^n
  $$
  exists, where the complex numbers $\alpha_i$ have modulus~$1$ and
  the set~$J$ has cardinality equal to the multiplicity of the trivial
  representation in the restriction of $\rho$ to $\ggeo{M}$.  We claim
  that the existence of this limit implies the equality $\alpha_i=1$
  for all~$i\in J$, so that the limit is equal to~$|J|$, as desired.
  \par
  Indeed, let $L\subset J$ be the set of $i$ where $\alpha_i\not=1$. The
  sequence
  $$
  \sum_{i\in L} \alpha_i^n
  $$
  converges as well, and its limit must be zero since it converges
  to~$0$ on average over~$n\leq N$. However, the lower bound
  $$
  \limsup_{n\to +\infty}\  \Bigl|\sum_{i\in L}\alpha_i^n\Bigr|\geq |L|^{1/2}
  $$
  holds (see, e.g.,~\cite[Lemma\,11.41]{ant}), so we deduce that~$L$ is
  empty, which proves the claim.
\end{proof}

A more global form of this converse, for generically unramified objects,
is the following:

\begin{proposition}\label{pr-converse-equality}
  Let~$M$ be an $\ell$-adic perverse sheaf on~$G$ that is arithmetically
  semisimple and pure of weight zero. Assume that $M$ is generically
  unramified. Let~$r$ be the tannakian dimension of~$M$ and let
  $\mcX=\mcX_w(M)$ be the set of unramified characters for~$M$. If the
  sequence of probability measures
  $$
  \frac{1}{|\mcX(k_n)|}\sum_{\chi\in\mcX(k_n)}\delta_{\Thetaf_{M,k_n}(\chi)},
  $$
  defined when $\mcX(k_n)$ is not empty, converges weakly to some
  probability measure, then we have the equality $\garith{M}=\ggeo{M}$.
\end{proposition}

\begin{proof}
  Suppose that $\ggeo{M}\not= \garith{M}$.  By
  Proposition~\ref{pr:geom-vs-arith2}, there exists an element
  $\xi\not=1$ of $\garith{M}/\ggeo{M}$ which generates a Zariski-dense
  subgroup of this group, which is abelian. Thus there exists an
  irreducible representation~$\rho$ of the quotient
  $\garith{M}/\ggeo{M}$ such that $\rho(\xi)\not=1$; for any $n\geq 1$
  and any $\chi\in\charg{G}(k_n)$ unramified for~$M$, the equality
  $\rho(\Thetaf_{M,k_n}(\chi))=\rho(\xi)^n$ holds.
  \par
  Let $\mcX$ be the set of characters unramified for~$M$.  Then
  $$
  \frac{1}{|\mcX(k_n)|}\sum_{\chi\in\mcX(k_n)}\Tr(\rho(\Thetaf_{M,k_n}(\chi)))
  =\rho(\xi)^n
  $$
  for all $n\geq 1$ for which $\mcX(k_n)$ is not empty. Since
  $\rho(\xi)\not=1$, this quantity does not converge as $n\to+\infty$,
  which implies the proposition by contraposition.
\end{proof}

\section{Horizontal equidistribution}

The proof of Theorem~\ref{th-2} relies crucially on the estimates in the
stratified vanishing theorem~\ref{thm-high-vanish}. We expect (see
Remark \ref{rmk-HV-effective}) that the implied constants in these
estimates depend only on the complexity of the perverse sheaf~$M$ (as is
the case for unipotent groups).

Under the assumption that such a statement is valid, and in fact that
this holds for the size of the set of unramified characters, one can
obtain equidistribution statements for finite fields when their size tends to infinity (for instance, for $\Ff_p$ as
$p\to+\infty$; compare with~\cite[Ch.\,28--29]{mellin}).

We include a conditional statement of this type, anticipating some
progress soon concerning the underlying uniformity question. We leave to
the interested reader the task of formulating variants similar to
Theorems~\ref{th-3} and~\ref{th-4}.


\begin{theorem}[Horizontal equidistribution]\label{th-horiz}
  Let $\ell$ be a prime number. Let~$N\geq 1$ be an integer and let
  $(G,u)$ be a quasi-projective commutative group scheme
  over~$\Zz[1/\ell N]$ such that, for all primes $p\nmid \ell N$, the
  fiber~$G_p$ of~$G$ over~$\Ff_p$ is a connected commutative algebraic
  group for which the estimate
  $$
  |\charg{G}_p(\Ff_{p^n})\setminus\unram{M}(\Ff_{p^n})|\ll c_u(M)
  p^{n(\dim(G_p)-1)}
  $$
  holds for all primes~$p$ and~$n\geq 1$ and all arithmetically
  semisimple objects~$M$ in~$\Pervint(G_p)$ which are generically
  unramified.

  Let $(M_p)_{p\nmid N\ell}$ be a sequence of arithmetically semisimple
  objects in $\Pervint(G_p)$ which are pure of weight zero.  Suppose
  that the tannakian dimension~$r$ of $M_p$ is independent of~$p$, and
  that for all~$p$, we have $\garith{M_p}=\ggeo{M_p}$, and that this
  common reductive group is conjugate to a fixed subgroup~$\Gg$
  of~$\GL_r(\Qlb)$.

  Let~$K$ be a maximal compact subgroup of~$\Gg(\Cc)$ and let
  $\mu^{\sharp}$ be the direct image of the Haar probability measure on
  $K$ to $K^\sharp$.

  Let $\mcX_p$ be the set of characters $\chi\in \charg{G}_p(\Ff_p)$
  which are unramified for the object $M_p$.
  
  If we have $ c_{u}(M_p)\ll 1$ for all $p\nmid N\ell$, then the
  families of conjugacy classes
  $(\Thetaf_{M_p,\Ff_p}(\chi))_{\chi\in\mcX_p}$ become
  $\mu^{\sharp}$-equidistributed in $K^\sharp$ as $p\to +\infty$.
\end{theorem}




\begin{proof}
  The argument follows that of Theorem~\ref{th-2}; it suffices to prove
  the estimate
  $$
  \frac{1}{|\mathscr{X}_p|}\sum_{\chi \in \mathscr{X}_p}
  \Tr\big(\rho(\Thetaf_{M_p,\Ff_p}(\chi))\big) \ll \frac{1}{\sqrt{p}}
  $$
  for all~$p\nmid N\ell$.  The proof of this is similar to the first
  part of Proposition~\ref{pr-weyl-sum}, noting that, under our
  assumptions, the implied constants in the key
  bounds~(\ref{eq-comp1}),~(\ref{eq-comp2bis}),~(\ref{eq-comp3bis})
  and~(\ref{eq-comp4bis}) are independent of~$p$, since the complexity
  of~$M_p$ is bounded independently of~$p$, and hence also that
  of~$\rho(M_p)$ by~\hbox{\cite[Prop.\,6.33]{sawin_conductors}}.
\end{proof}

\begin{remark}\label{rm-horizontal}
  (1) For $G$ unipotent, results of this form are unconditional by
  Proposition~\ref{prop-strat-unip} (the case of $\Gg_a$ essentially
  goes back to Katz~\cite{gkm}, whereas the case of an arbitrary power
  of~$\Gg_a$ follows from~\cite[Th.\,7.22]{sawin_conductors}). For
  $G=\Gg_m$, a similar statement is proved by Katz
  in~\cite[Th.\,28.1]{mellin}.
  \par
  (2) The result is also unconditional in the case of abelian varieties
  (see Remark~\ref{rmk-HV-effective}). We expect that a careful look at
  the proof of the generic vanishing theorem will also show that it is
  unconditional for $\Gg_m\times\Gg_a$. The case of tori of
  dimension~$\geq 2$ is however not yet known.
\end{remark}

\section{Objects of rank~$1$}

In this section, we apply the general equdistribution results to objects
in the tannakian subcategory of $\Ppintarith(G)$ additively generated by
objects of tannakian rank~$1$. The corresponding arithmetic tannakian
groups are computed (in principle) in Proposition~\ref{pr-rank1}.

\begin{proposition}\label{pr-rank1-equi}
  Let $r\geq 1$ be an integer and let
  $$
  M=L_1\oplus \cdots \oplus L_r
  $$
  where $L_i$ is an object of $\Ppintarith(G)$ of tannakian
  rank~$1$. Let
  $$
  H=\{(n_1,\ldots,n_r)\in\Zz^r\,\mid\,
  L_1^{*_{\intt} n_1}*_{\intt}\cdots*_{\intt}L_r^{*_{\intt} n_r}\simeq \un_G\},  
  $$
  and let
  $$
  K=\{(x_i)\in(\mathbf{S}^1)^r\,\mid\,     \prod_{1\leq i\leq
    r}x_{i}^{n_i}=1\text{ for all } (n_1,\ldots, n_r)\in H\}.
  $$

  Then the unitary Frobenius conjugacy classes of~$M$ are
  equidistributed on average in~$K$.
\end{proposition}

\begin{proof}
  This follows from Theorem~\ref{th-4}, on noting that the
  arithmetic tannakian~$\garith{M}$ group of~$M$ is abelian (it can be viewed as a subgroup of the diagonal subgroup
  $\GL_1^r\subset \GL_r$); hence, the conjugacy classes of elements of
  $\garith{M}$ are just singletons, and in particular are the same as
  the conjugacy classes in~$\GL_r$.
\end{proof}

As an application, we explain how to recover a theorem of
Rojas-León~\cite[Th.\,1]{rojas-leon}, which concerns the
equidistribution properties of Gauss sums.

\begin{theorem}[Rojas-León]\label{th-rojas-leon}
  Let~$r\geq 1$ be an integer. Let $(\alpha_i)_{1\leq i\leq r}$ be a family of non-constant morphisms $\alpha_i\colon \Gg_m\to \Gg_m^r$ defined over~$k$. Let $t=(t_i)\in (k^{\times})^r$ and let $(\eta_i)_{1\leq i\leq r}$ be
  characters of~$k^{\times}$.

  The tuples
  $$
  \Bigl( \frac{\chi(t)\tau(\psi,\eta_1\cdot (\chi\,
    \circ\alpha_1))}{|k_n|^{1/2}},\ldots,
  \frac{\chi(t)\tau(\psi,\eta_r\,\cdot
    (\chi\circ\alpha_r))}{|k_n|^{1/2}} \Bigr)_{\chi\in
    \widehat{\Gg_m^r}(k_n)}
  $$
  of Gauss sums are equidistributed on average in~$\Cc^r$ according
  to the probability Haar measure on a closed subgroup
  $K\subset (\mathbf{S}^1)^r\subset \Cc^r$.

  Moreover, factor $\alpha_i=[x\mapsto x^{N_i}]\circ \beta_i$ for some
  closed immersion $\beta_i$ and some integer $N_i\geq 1$. If, for each
  $i$ with $1\leq i\leq r$, the elements
  $$
  \sum_{\omega^{N_i}=\eta_i}\omega
  $$
  of $\Zz[\Pi(\Gg_{m,\bar{k}},\bQl)]$ are linearly independent
  over~$\Zz$, then $K=(\mathbf{S}^1)^r$ and equidistribution holds
  without average.
\end{theorem}

\begin{proof}
  For simplicity of notation, we will assume that $t=1$ and that each
  $N_i$ is coprime to~$p$. Let $j\colon \Gg_m\to \Aa^1$ be the open
  immersion. We recall that
  $\chi\mapsto \tau(\psi,\eta_i\cdot (\chi\circ \alpha_i))$ is the
  discrete Mellin transform on $\Gg_m$ of the trace function of the
  perverse sheaf
  \[
  L_i=\alpha_{i*}(j^*\mcL_{\psi}*_! \mcL_{\chi}[1])(1/2)
  \] on $\Gg_m^r$, which is pure of weight~$0$. This is an object of tannakian
  rank~$1$.

  Thus the first statement is a direct application of Proposition \ref{pr-rank1-equi}, with~$K$ a maximal compact subgroup for the arithmetic
  tannakian group of
  $$
  M=L_1\oplus\cdots\oplus L_r
  $$
  (note that we consider the Gauss sums in~$\Cc$ to avoid excluding
  those boundedly many $\chi$ where some
  $\eta_i\cdot(\chi\circ\alpha_i)$ is trivial, for which the modulus is
  $1/\sqrt{|k_n|}$ instead of~$1$; these do not affect the
  equidistribution property).

  Using Proposition \ref{pr-rank1-equi} again, for the second statement
  we need to prove that under the stated assumptions, there is no
  convolution relation between the objects $L_i$, or equivalently no
  relation between their classes in the group $L(\Gg_m^r)$; this implies
  that the arithmetic tannakian group of~$M$ is~$\Gg_m^r$ (i.e., is as
  large as possible).  In fact, we claim that the geometric tannakian
  group is already that large, which means that there are no geometric
  convolution relations between the objects $L_i$. This in particular
  also implies the equidistribution without average (see
  Theorem~\ref{th-2bis}).

  To prove the claim, we use the Gabber--Loeser isomorphism
  described in Example~\ref{ex-gl-rank-1} to express the class of $L_i$ in
  $L(\Gg^r_{m,\bar{k}})$ as
  $$
  \lambda_i=\Bigl(1,\sum_{\omega^{N_i}=\eta_i}(T_i,\omega)\Bigr),
  $$
  where $T_i$ is the image of~$\beta_i$ and $(T_i, \omega)$ is one of
  the basis vectors in the free abelian group generated by pairs of a
  one-dimensional subtorus and a tame character (see \loccit). (This
  fact is a form of the Hasse--Davenport relation;
  see~\cite[Th.\,8.9.1]{esde}.)

  By definition of a free abelian group, a non-trivial linear relation
  can only exist if, for some one-dimensional subtorus $T\subset
  \Gg_m^r$, the elements $\lambda_i$ with $T_i=T$ are linearly
  dependent, and this in turn is equivalent with the elements
  $$
  \sum_{\omega^{N_i}=\eta_i}\omega
  $$
  being linearly dependent, as claimed.
\end{proof}


\part{Applications}\label{part-applications}

\chapter*{Description of applications}
\label{sec-applications}

The remainder of the book is devoted to applications of the theoretical
results of the first part of this book. We split these applications in
further chapters as follows:

\begin{enumerate}
\item We define in Chapter~\ref{sec-L} the analogue of $L$-functions for
  the Fourier--Mellin transforms.  We establish with its help that the
  arithmetic tannakian group is infinite for many non-punctual objects
  on abelian varieties.
\item We present in Chapter~\ref{sec-stratification} the concrete
  analytic translation of the stratified vanishing theorem to
  stratification of estimates for exponential sums, in the spirit of
  Katz--Laumon~\cite{KL-fourier-exp-som} and
  Fouvry--Katz~\cite{fouvry-katz}. We also present a statement of
  ``generic Fourier invertibility'', which shows that two semsimple
  perverse sheaves are isomorphic in the category $\Ppbarith(G)$ if and
  only if the associated exponential sums coincide for a generic set of
  characters.
\item In Chapter~\ref{sec-indep}, we add a theoretical application of
  equidistribution in direction of \emph{independence of~$\ell$}
  properties of the tannakian groups associated to a compatible system
  of $\ell$-adic complexes.
\item In applications of equidistribution to concrete perverse sheaves,
  the main issue is to determine the tannakian group.  The main tool
  that we will use for this purpose is \emph{Larsen's Alternative}, and
  its link with equidistribution. We present this result (and a new
  variant for the exceptional group $\mathbf{E}_6$) in
  Chapter~\ref{sec-larsen}.
\item Then in the remaining chapters, we present examples of
  equidistribution for ``concrete'' groups, namely:
  \begin{itemize}
  \item the product $\Gg_m\times\Gg_a$, which (apart from unipotent
    groups) is probably the simplest group of dimension $\geq 2$
    (Chapter~\ref{sec-product}); this corresponds to rather natural
    families of exponential sums parameterized by both an additive
    character and a multiplicative character.
  \item higher-dimensional tori, with applications to the study of the
    variance of arithmetic functions on $k[t]$ in arithmetic
    progressions modulo square-free polynomials (see
    Chapter~\ref{sec-variance}).
  \item the jacobian of a curve (Chapter~\ref{sec-jacobian}); the
    application we present is a generalization of an unpublished result
    of Katz (which answered a question of Tsimerman).
  \item in the same chapter, the intermediate jacobian of a smooth
    projective cubic hypersurface of dimension~$3$, which is an abelian
    variety of dimension~$5$ (see Chapter~\ref{sec-cubic}).
  \end{itemize}
\end{enumerate}

\chapter{Über eine neue Art von $L$-Reihen}\label{sec-L}

\section{$\lhat$-functions}

Let $k$ be a finite field, with algebraic closure~$\bar{k}$ and
intermediate extensions $k_n$. We fix as usual a prime $\ell$ different
from the characteristic of~$k$.  Let $G$ be a connected commutative
algebraic group over~$k$, and let~$d$ be its dimension. We denote by~$e$
the neutral element of~$G$.

By analogy with algebraic varieties over~$k$, we can define
``$L$-functions'' for objects of~$\Der(G)$, where suitable
characters~$\chi\in\charg{G}$ play the role of primes in an ``Euler
product''.

We denote by $\charg{G}^*\subset \charg{G}$
\nomenclature[$G$]{$\charg{G}^*$}{primitive elements of~$\charg{G}$} the set of characters such
that $\chi\in \charg{G}^*(k_n)$ if and only if there is no $d\mid n$
with $d<n$ such that $\chi=\chi'\circ N_{k_n/k_d}$.  We say that
elements of $\charg{G}^*$ are \emph{primitive},\index{primitive characters} and for
$\chi\in \charg{G}^*(k_n)$, we put
$\deg(\chi)=n$.\nomenclature[$deg$]{$\deg(\chi)$}{degree of a primitive character}  We then denote by
$[\charg{G}]$\nomenclature[$G$]{$[\charg{G}]$}{primitive characters modulo
  Galois action} the quotient set of~$\charg{G}^*$ by the equivalence
relation defined by $\chi_1\sim \chi_2$ if and only if
$\deg(\chi_1)=\deg(\chi_2)$ and
$$
\chi_2=\chi_1\circ \Frob_{k_{\deg(\chi_1)}}^j
$$
for some integer~$j\in\Zz$. There are $\deg(\chi)$ primitive characters
equivalent to a given~$\chi\in \charg{G}^*$.


\begin{definition}[$\lhat$-function]
  Let $M$ be an object of~$\Der(G)$.  The
  \emph{Fourier-$L$-function},\index{Fourier
    $L$-function}\index{$\lhat$-function}\nomenclature[$L$]{$\lhat(M,T)$}{$\lhat$-function of~$M$}
  or \emph{$\lhat$-function}, of~$M$ is the formal power series
  $$
  \lhat(M,T)=\prod_{\chi\in [\widehat{G}]}
  \det(1-T^{\deg(\chi)}\Fr_{k_{\deg(\chi)}}\mid H^*_c(G_{\bar{k}},M_{\chi}))^{-1}\in
  \Qlb[[T]].
  $$
\end{definition}

This is similar to the definition
$$
L(M,T)=\prod_{x \in [X]} \det(1-T^{\deg(x)}\Fr_{k_{\deg(x)}}\mid
M_x)^{-1}\in \Qlb[[T]]
$$
of the $L$-function of~$M$ on an arbitrary algebraic variety~$X$
over~$k$, with primitive characters replacing the set $[X]$ of closed
points of~$X$.

Indeed, if $G$ is unipotent of dimension~$d$, and $\ft(M)$ denotes the
Fourier transform of~$M$ on the (or ``a'') Serre dual $G^{\vee}$ defined
with respect to some additive character~$\psi$, as in
Section~\ref{sec-unipotent}, then we obtain the identity
$$
\lhat(M,T)=L(\ft(M),|k|^dT),
$$
(e.g. by the formula~(\ref{eq-lhat2}) below, since the stalk of $\ft(M)$
at the origin is canonically isomorphic to~$M$ by the proper base change
theorem, and $|G(k_n)|=|k|^{nd}$ in this case).

In general, however, we obtain ``new'' $L$-functions.  Their fundamental
properties, including rationaliy, are given by the next proposition.

\begin{proposition}\label{pr-lhat}
  Let $M$ be an object of~$\Der(G)$. We denote as usual
  $$
  S(M,\chi)=\sum_{x\in G(k_n)}\chi(x)t_M(x;k_n)
  $$
  for $n\geq 1$ and $\chi\in\charg{G}(k_n)$.
  \begin{enumth}
  \item The $\lhat$-function satisfies
    \begin{align}\label{eq-lhat}
      \lhat(M,T)&= \exp\Bigl(\sum_{n\geq
        1}\Bigl(\sum_{\chi\in\charg{G}(k_n)}S(M,\chi)\Bigr)
      \frac{T^n}{n}\Bigr)\\
      &= \exp\Bigl(\sum_{n\geq 1}|G(k_n)|t_{M}(e;k_n)
      \frac{T^n}{n}\Bigr).
      \label{eq-lhat2}
    \end{align}
  \item\label{pr-lhat:item2} The $\lhat$-function is a rational
    function; if~$M$ is a mixed complex, then the zeros and poles of
    $\lhat(M,T)$ are $|k|$-Weil numbers of some weights.
  \item For any $\chi\in\charg{G}(k)$, the equality
    $\lhat(M_{\chi},T)=\lhat(M,T)$ holds.
  \end{enumth}
\end{proposition}

\begin{proof}
  The proof of the formula~(\ref{eq-lhat}), like in the classical case,
  is a simple consequence of the trace formula. Precisely, we apply the
  operator $f(T)\mapsto Td\log f(T) $ to both sides of this equality. On
  the left-hand side, after expressing the determinant as alternating
  product of the determinants on the various groups
  $H^i_c(G_{\bar{k}},M_{\chi})$, we obtain
  \begin{multline*}
    Td\log \lhat(M,T)= \sum_{\chi\in[\charg{G}]}\deg(\chi)\sum_{m\geq
      1}T^{m\deg(\chi)} \Tr(\Fr_{k_{\deg(\chi)}}^{m}\mid
    H^*_c(G_{\bar{k}},M_{\chi})) =
    \\
    \sum_{n\geq 1} T^n \sum_{d\mid n} \sum_{\chi\in
      [\widehat{G}](k_d)}d\Tr(\Fr_{k_d}^{n/d}\mid
    H^*_c(G_{\bar{k}},M_{\chi})).
  \end{multline*}
  \par
  On the right-hand side of~(\ref{eq-lhat}), we obtain
  $$
  \sum_{n\geq 1}T^n \sum_{\chi\in \charg{G}(k_n)} S(M,\chi),
  $$
  and hence the formula is equivalent with the fact that the identity
  \begin{equation}\label{eq-desired-identity}
    \sum_{d\mid n} \sum_{\chi\in
      [\widehat{G}](k_d)}d\Tr(\Fr_{k_d}^{n/d}\mid
    H^*_c(G_{\bar{k}},M_{\chi})) = \sum_{\chi\in \charg{G}(k_n)}
    S(M,\chi)
  \end{equation}
  holds for any integer~$n\geq 1$.
  \par
  Let~$n\geq 1$. To establish~(\ref{eq-desired-identity}) for~$n$, we
  begin with the trace formula~(\ref{eq-trace-formula}), which implies
  that
  $$
  S(M,\chi)=\Tr(\Fr_{k_n}\mid H^*_c(G_{\bar{k}},M_{\chi})),
  $$
  for any $\chi\in\charg{G}(k_n)$.

  There exists a unique divisor $d$ of~$n$ and a
  character~$\chi_0\in\charg{G}^*(k_d)$ such that
  $\chi=\chi_0\circ N_{k_n/k_d}$. The map sending $\chi$ to the
  equivalence class of~$\chi_0$ in~$[\charg{G}]$ has image the subset of
  classes~$[\eta]$ of primitive characters~$\eta$ with degree
  dividing~$n$, and for any such class~$[\eta]$, there are
  exactly~$\deg([\eta])$ characters~$\chi\in\charg{G}(k_n)$ mapping
  to~$[\eta]$. Moreover, there are canonical isomorphisms
  $$
  H^*_c(G_{\bar{k}},M_{\chi})\simeq H^*_c(G_{\bar{k}},M_{\eta}),
  $$
  with the actions of $\Fr_{k_n}$ corresponding to that
  of~$\Fr_{k_d}^{n/d}$, so that
  $$
  S(M,\chi)=\Tr(\Fr_{k_d}^{n/d}\mid H^*_c(G_{\bar{k}},M_{\eta}))
  $$
  for all~$\chi$ mapping to~$[\eta]$.
  This implies the desired
  identity~(\ref{eq-desired-identity}).
  \par
  The second formula~(\ref{eq-lhat2}) for $\lhat(M,T)$ follows
  immediately from~(\ref{eq-lhat}), since orthogonality of characters
  implies that the formula
  $$
  \sum_{\chi\in\charg{G}(k_n)}S(M,\chi)
  =|G(k_n)|t_M(e;k_n)
  $$
  holds for all~$n\geq 1$.
  \par
  Using next the trace formula and the Riemann Hypothesis to compute
  $|G(k_n)|$ as an alternating sum of $|k|$-Weil numbers, it follows
  that
  $$
  |G(k_n)|t_M(e;k_n)= \sum_{i\in I} \eps_i\alpha_i^n
  $$
  for some finite set $I$, some $\eps_i\in \{-1,1\}$, and some
  $|k|$-Weil numbers $\alpha_i$. The second assertion follows then from
  the usual power series expansion
  $$
  \exp\Bigl(\sum_{n\geq 1} \frac{\alpha^nT^n}{n}\Bigr) =
  \frac{1}{1-\alpha T}.
  $$
  \par
  The final assertion is clear either from the definition, or from the
  above, noting that $t_{M_{\chi}}(e;k_n)=t_M(e;k_n)$ for any
  $\chi\in\charg{G}(k)$ and $n\geq 1$.
\end{proof}

\begin{remark}\label{rm-lhat-different}
  To illustrate the differences with $L$-functions, we note that if~$G$
  is not unipotent, then the $\lhat$-function is very rarely a
  polynomial or the inverse of a polynomial, and does not satisfy in
  general any functional equation\index{functional equation} of the form
  $$
  \lhat(M,T)=(\text{simple quantities})
  \times  \lhat(M^{\vee},q^{\alpha}T^{-1}).
  $$
  as is the case for the standard $L$-function of~$M$ (this is related
  to the remark of Boyarchenko and
  Drinfeld~\cite[\S1.6,\,Example\,1.8]{boyarchenko-drinfeld}).
  \par
  To give a concrete example, take $G=\Gg_m$. In this case, we deduce
  from~(\ref{eq-lhat2}) the formula
  $$
  \lhat(M,T)= \exp\Bigl(\sum_{n\geq 1}(|k|^n-1)\,t_{M}(e;k_n)
  \frac{T^n}{n}\Bigr)=\frac{L(M_e,|k|T)}{L(M_e,T)}
  $$
  where $M_e$ is the stalk of~$M$ at~$e$ (where $L(M_e,T)$ is the
  $L$-function of the stalk of~$M$ at~$e$, viewed as a complex
  on~$\{e\}$). If the $L$-function $L(M_e,T)$ is not constant, then
  there can never be cancellation in this quotient to obtain a
  polynomial or the inverse of a polynomial. If (say) we have
  $$
  L(M_e,T)=(1-\alpha T)(1-\alpha^{-1}T),
  $$
  then
  $$
  \lhat(M,T)=\frac{(1-|k|\alpha T)(1-|k|\alpha^{-1}T)}{ (1-\alpha
    T)(1-\alpha^{-1}T)},
  $$
  and this satisfies no simple functional relation.
\end{remark}

We conclude with a result that will be useful in the next section when
performing induction.

\begin{proposition}\label{pr-tac}
  Let $G$ be a semiabelian variety over~$k$. Let $S$ be a \tac\
  of~$G_k$ defined by a morphism $\pi\colon G\to G'$ over~$k$ and a
  character $\chi_0\in\charg{G}(k)$, and let $[S]$ denote the classes
  in~$[\charg{G}]$ of elements of~$S$.  Let $M$ be an object
  of~$\Der(G)$. We then have
  $$
  \prod_{\chi \in [S]} \det(1-T^{\deg(\chi)}\Fr_{k_{\deg(\chi)}}\mid
  H^*_c(G_{\bar{k}},M_{\chi}))^{-1}= \lhat(R\pi_!M_{\chi_0},T).
  $$
\end{proposition}

\begin{proof}
  We have $\chi\in [S]$ if and only if $\chi=\chi_0\cdot (\pi^*\eta)$
  for some $\eta\in [\charg{G}']$, with $\deg(\chi)=\deg(\eta)$. By the
  projection formula, we have a canonical isomorphism
  $$
  H^*_c(G_{\bar{k}},M_{\chi}) =
  H^*_c(G_{\bar{k}},M_{\chi_0}\otimes\pi^*\mcL_{\eta}) \simeq
  H^*_c(G'_{\bar{k}},R\pi_!M_{\chi_0}\otimes\mcL_{\eta}),
  $$
  from which the identity
  $$
  \det(1-T^{\deg(\chi)}\Fr_{k_{\deg(\chi)}}\mid
  H^*_c(G_{\bar{k}},M_{\chi}))^{-1}=
  \det(1-T^{\deg(\eta)}\Fr_{k_{\deg(\eta)}}\mid
  H^*_c(G'_{\bar{k}},(R\pi_!M_{\chi_0})_{\eta}))^{-1}
  $$
  follows for any~$\chi\in [S]$.
\end{proof}

\section{Objects with finite arithmetic tannakian groups on abelian
  varieties} 

As a non-trivial application of $\lhat$-functions, we will show that
they lead to a characterization of objects with finite arithmetic
tannakian groups on abelian varieties. This is an analogue of a result
of Katz (see~\cite[Th.\,6.2]{mellin}, recalled in
Theorem~\ref{th-katz-finite}, (1)) for~$\Gg_m$, where in fact the
$\lhat$-function appears implicitly (more precisely, where the
logarithmic derivative $Td\log \lhat(M,T)$ appears); similar results
appear in a preprint of Weissauer~\cite{weissauer_connected}.

More generally, inspired by the formulation used by Katz, we can prove a
stronger statement.

\begin{definition}[Quasi-unipotent object]
  \index{quasi-unipotent object} Let~$G$ be a connected commutative
  algebraic group over~$k$.  An object~$M$ of~$\Der(G)$ is said to be
  \emph{quasi-unipotent} if it is generically unramified and if there
  exists an integer $m\geq 1$ such that for any unramified
  character~$\chi\in \charg{G}$, the eigenvalues of Frobenius on
  $H^0(G_{\bar{k}},M_{\chi})$ are roots of unity of order at most~$m$.
\end{definition}

\begin{remark}\label{rm-quasi-unipotent}
  (1) Any perverse sheaf~$M$ on~$G$ with~$\garith{M}$ finite is
  quasi-unipotent. Indeed, first $M$ is generically unramified by
  Corollary~\ref{cor-finite-gen-unram}. Let then $m$ be the size of
  $\garith{M}$. For any unramified character $\chi\in\charg{G}$, the
  Frobenius action on $H^0(G_{\bar{k}},M_{\chi})$ is ``conjugate'' to an
  element of $\garith{M}$, so its eigenvalues are $m$-th roots of unity.
  \par
  (2) If $M$ is a quasi-unipotent perverse sheaf on~$G$, then it follows
  from the definition that any object of $\braket{M}$ is also
  quasi-unipotent.
  \par
  (3) Let $M$ be a quasi-unipotent object of~$\Der(G)$. Let
  $g_0\in G(k)$. Then the translated object $M'=[g\mapsto gg_0]^*M$ is
  also quasi-unipotent. Indeed, since $M'$ is canonically isomorphic to
  the convolution $\delta_{g_0} *_! M$, we obtain for any
  $\chi\in\charg{G}$ a canonical isomorphism
  $$
  H^*_c(G_{\bar{k}},M'_{\chi})\simeq
  H^*_c(G_{\bar{k}},(\delta_{g_0^{-1}})_{\chi})\otimes
  H^*_c(G_{\bar{k}},M_{\chi}).
  $$
  \par
  Noting that
  $H^*_c(G_{\bar{k}},(\delta_{g_0^{-1}})_{\chi})
  =H^0_c(G_{\bar{k}},(\delta_{g_0^{-1}})_{\chi})$,
  this shows already that $\chi$ is weakly-unramified for~$M$ if and
  only if it is for~$M'$.
  \par
  If~$\chi$ is weakly-unramified for~$M'$, and belongs to
  $\charg{G}(k_n)$, then the Frobenius automorphism of~$k_n$ acts on
  $H^0(G_{\bar{k}},(\delta_{g_0^{-1}})_{\chi})$ by multiplication
  by~$\chi(g_0^{-1})$, which is a root of unity of order bounded by the
  order of~$g_0$ in $G(k)$. Since $M$ is quasi-unipotent, the
  eigenvalues of Frobenius on $H^0(G_{\bar{k}},M'_{\chi})$ are roots of
  unity of order bounded independently of~$\chi$.
\end{remark}

\begin{theorem}\label{th-finite-ab-1}
  Let~$A$ be an abelian variety over~$k$.  Let $M$ be an arithmetically
  semisimple perverse sheaf of weight zero in $\Ppintarith(A)$ which is
  non-zero. If $M$ is quasi-unipotent, for instance if the group
  $\garith{M}$ is finite, then~$M$ is punctual.
\end{theorem}

\begin{remark}
  As proved by Katz in the case of $\Gg_m$, one may expect that the
  conclusion of the theorem extends to objects with finite
  \emph{geometric} tannakian group (see~\cite[Th.\,6.4]{mellin} or
  Theorem~\ref{th-katz-finite}, (2)). We do not know how to prove this
  in general (Katz's deduction of this fact from the analogue of
  Theorem~\ref{th-finite-ab-1} for $\Gg_m$ uses the classification of
  objects of tannakian rank~$1$, for instance, which we do not have in
  this setting). We will however prove a weaker statement in
  Section~\ref{ssec-finite-ab} which turns out to be sufficient for many
  applications, including those of Chapter~\ref{sec-jacobian}.
\end{remark}

Before giving the proof, we state two corollaries.

\begin{corollary}\label{cor-translate-irred-var-ab}
  Let $M$ be an arithmetically simple perverse sheaf of weight zero on
  an abelian variety~$A$ over~$k$ of dimension $g\geq 1$.  Let $\bfG$ be
  the neutral component of~$\garith{M}$, and let~$S$ be the support
  of~$M$. The restriction of the standard representation of~$\garith{M}$
  to~$\bfG$ is irreducible unless there exists $x\not=e$ such that
  $M*\delta_x$ is isomorphic to~$M$. In particular, this holds unless
  there exists $x\in A$ with $x\not=e$ such that $x+S=S$.
\end{corollary}

\begin{proof}
  Let $P$ be an object of $\braket{M}^{\arith}$ which is a faithful
  representation of the finite component group $C=\garith{M}/\bfG$. Its
  tannakian group is isomorphic to~$C$, and hence the object~$P$ is punctual
  by Theorem~\ref{th-finite-ab-1}. The points appearing in the
  decomposition of~$P$ generate a finite subgroup~$B$ of~$A(\bar{k})$,
  and each skyscraper sheaf for $x\in B$ corresponds to a
  character~$\chi_x$ of~$\garith{M}$ trivial on~$\bfG$.
  \par 
  By a simple application of Frobenius reciprocity,\index{Frobenius
    reciprocity} a representation~$\rho$ of $\garith{M}$ restricts to an
  irreducible representation of~$\bfG$ unless there exists $x\in C$ such
  that $x\not=e$ and $\rho\otimes\chi_x$ is isomorphic to~$\rho$. In
  terms of perverse sheaves on~$A$, this condition (for the standard
  representation) means that $M*\delta_x$ is isomorphic to~$M$, which is
  the first assertion. Since it also implies that~$S+x=S$, this
  concludes the proof.
\end{proof}

\begin{corollary}\label{cor-support-fingp}
  Let~$A$ be an abelian variety over~$k$.  Let $M$ be a non-zero
  arithmetically semisimple perverse sheaf of weight zero in
  $\Ppintarith(A)$. If $M$ is quasi-unipotent, for instance if the group
  $\garith{M}$ is finite, then~$\ggeo{M}$ is a finite abelian group
  which is naturally isomorphic to the dual of the subgroup of
  $A(\bar k)$ generated by the support of $M$.
\end{corollary}

\begin{proof}
  By Theorem~\ref{th-finite-ab-1}, the object~$M$ is punctual. If~$F$
  denotes its support, we have an isomorphism
  $$
  M=\bigoplus_{x\in F}\alpha_x^{\deg} \otimes \delta_x
  $$
  for some unitary scalars~$\alpha_x$, and therefore a geometric
  isomorphism of~$M$ with the direct sum of the $\delta_x$
  for~$x\in F$. Let~$H$ be the subgroup of $A(\bar{k})$ generated
  by~$F$, which is a finite abelian group and let
  $\widehat{H}=\Hom(H,\Qlb^{\times})$ be the dual group of~$H$. We
  obtain an additive functor from the finite-dimensional
  $\Qlb$-representations of~$\widehat{H}$ to $\braket{M}^{\geo}$ by
  associating to the character ``evaluation at~$x$'' of $\widehat{H}$
  the object $\delta_x$. Since $\delta_x *\delta_y\simeq \delta_{x+y}$,
  this is a tensor functor, and it gives an equivalence of categories.
  Hence~$\ggeo{M}$ is
  isomorphic to~$\widehat{H}$.
\end{proof}

We will prove Theorem~\ref{th-finite-ab-1} in the next two sections. In
fact, since this case is somewhat easier, we will begin by assuming that
the abelian variety~$A$ is simple (which is in a reasonable sense the
generic case) before handling the general situation. The reader may skip
the first case to read directly the proof of the general result.

We first prove two lemmas that are used in both proofs.

\begin{lemma}\label{lm-lambda1}
  Let $R$ be a commutative ring with unit and $\lambda$ a
  non-archimedean valuation on~$R$. Assume that~$R$ is complete with the
  topology given by $\lambda$.
  \par
  Let $(\alpha_i)_{i\in I}$ be a family of elements of~$R$ such that
  $|\alpha_i|_{\lambda}\leq 1$ for all~$i\in I$, and let $(d_i)_{i\in
    I}$ be a family of positive integers such that
  $$
  \lim_{I} d_i= +\infty,
  $$
  where the limit is along the filter of the complements of finite
  subsets of~$I$.
  \par
  The product
  $$
  \prod_{i\in I}(1-\alpha_iT^{d_i})
  $$
  converges and is non-zero for $T$ such that $|T|_{\lambda}<1$.
\end{lemma}

\begin{proof}
  Let $J\subset K$ be finite subsets of~$I$.  Then for
  $|T|_{\lambda}\leq 1$, we compute that
  \begin{align*}
    \Bigl|\prod_{i\in K}(1-\alpha_iT^{d_i})-\prod_{i\in
      J}(1-\alpha_iT^{d_i})\Bigr|_{\lambda} &=\Bigl|\prod_{i\in
      J}(1-\alpha_iT^{d_i})\Bigl(\prod_{i\in K\setminus
      J}(1-\alpha_iT^{d_i})-1\Bigr)\Bigr|_{\lambda}
    \\
    &\leq \Bigl|\prod_{i\in K\setminus
      J}(1-\alpha_iT^{d_i})-1\Bigr|_{\lambda}=
    \Bigl|\sum_{\emptyset\not= L\subset K\setminus
      J}(-1)^{|L|}\sigma_{L}T^{d_{L}}\Bigr|_{\lambda}
  \end{align*}
  where 
  $$
  \sigma_L=\prod_{i\in L}\alpha_i,\quad\quad
  d_L=\sum_{i\in L}d_i.
  $$
  \par
  We note that $|\sigma_L|_{\lambda}\leq 1$ for all~$L$. Moreover, since
  the lower-bound
  $$
  d_L\geq \min_{i\in I\setminus J}d_i
  $$
  holds, the assumption that $d_i\to +\infty$ implies that for any
  integer $N\geq 1$, we can
  choose $J$ so that 
  $$
  \Bigl|\sum_{\emptyset\not= L\subset K\setminus
    J}(-1)^{|L|}\sigma_{L}T^{d_{L}}\Bigr|_{\lambda}
  \leq |T|_{\lambda}^{N}
  $$
  for any finite set $K$ containing $J$. The absolute convergence of the
  product follows when $|T|_{\lambda}<1$ using the Cauchy criterion.  In
  particular, the product can only be zero if some term is zero, and
  this is not the case if~$|T|_{\lambda}<1$.
\end{proof}

The next lemma gives basic structural information on zeros and poles of
$\lhat(M,T)$, refining the last part of Proposition~\ref{pr-lhat} in the
case of abelian varieties.

\begin{definition}
  Let $f\in \bQl(X)$ be a non-zero rational function, $k$ a finite field
  and $r\in\Zz$. We denote by
  $\wtr_{k,r}(f)$\nomenclature[$w$]{$\wtr_{k,r}(f)$}{part of $|k|$-weight~$w$
    of a rational function} the rational function
  $$
  \prod_{\alpha\text{ of $k$-weight~$-r$}}(1-\alpha T)^{v_{\alpha}(f)}
  $$
  where $\alpha$ runs over elements of~$\bQl$ which are $k$-Weil numbers
  of weight~$-r$, and $v_{\alpha}$ is the order of~$f$ at~$\alpha$.
\end{definition}

In other words (note the minus sign), the rational function
$\wtr_{k,r}(f)$ is (up to leading terms) ``the part of~$f$ with zeros
and poles of weight~$r$''. Below, we will sometimes write $\wtr_{r}$
when the finite field $k$ is clear from context.

The definition implies that the identity
$$
\wtr_{k,r}(f_1f_2)=\wtr_{k,r}(f_1)\wtr_{k,r}(f_2)
$$
holds for any rational functions~$f_1$ and~$f_2$.

\begin{proposition}\label{pr-lhat-weights1}
  Let $M$ be a
  complex on an abelian variety~$A$ over~$k$ of dimension $g\geq
  0$. Assume that $M$ is pure of weight zero and that $M_e$ has weights
  in $[a,b]$.

  \begin{enumth}
  \item The poles \respup{zeros} of $\lhat(M,T)$ are $k$-Weil
    numbers. Their weights are of the form $-w-i$ for some even
    \respup{odd} integer~$i$ with $0\leq i\leq 2g$ and some integer $w$
    with $a\leq w\leq b$.

    If there exists such a zero or pole then there exists an eigenvalue
    of weight~$w$ on~$M_e$, and the formula
    $$
    \wtr_{k,-w}(\lhat(M,T))=\wtr_{k,-w}(\det(1-T\Fr_k\mid
      M_e))^{-1}
    $$
    holds.

  \item If~$M$ is an arithmetically simple perverse sheaf, and if $e$
    belongs to the open set of the support of~$M$ where $M$ is lisse,
    then the poles \respup{zeros} of $\lhat(M,T)$ have $k$-weights equal
    to $\dim \Supp(M)-i$ for some integers $i$ with $0\leq i\leq 2g$
    such that
    $$
    \dim\Supp(M)\equiv i\mods{2},
    $$
    and there are poles and zeros of all these possible weights.
  \end{enumth}
\end{proposition}

\begin{proof}
  (1) 
  By
  Proposition~\ref{pr-lhat}, we have
  $$
  \lhat(M,T)=\exp\Bigl(\sum_{n\geq 1}|A(k_n)|t_M(e;k_n)\frac{T^n}{n}\Bigr).
  $$
  \par
  This expression, combined with the purity of~$M$ and the structure of
  the cohomology of~$A$, shows that $\lhat(M,T)$ has:
  \begin{enumerate}
  \item[(i)] Poles of the form
    $$
    T=\frac{1}{\alpha\beta},
    $$
    where $\alpha$ is an eigenvalue of Frobenius on the stalk of~$M$
    at~$e$, and $\beta$ is an eigenvalue of Frobenius on
    $H^i(A_{\bar{k}},\bQl)$ for some even integer~$i$ with
    $0\leq i\leq 2g$. Since $\alpha$ is pure of some weight~$w$ where
    $a\leq w\leq b$, and $\beta$ is of weight~$i$, such a pole is a
    $|k|$-Weil number of weight~$-w-i$.
  \item[(ii)] Zeros of the form
    $$
    T=\frac{1}{\alpha\beta},
    $$
    where $\alpha$ is an eigenvalue of Frobenius on the stalk of~$M$
    at~$e$, and $\beta$ is an eigenvalue of Frobenius on
    $H^i(A_{\bar{k}},\bQl)$ for some odd integer~$i$ with
    $1\leq i\leq 2g-1$. As above, such a zero is a $|k|$-Weil number of
    weight~$-w-i$ where $a\leq w\leq b$.
  \end{enumerate}

  The precise formula for the parts of weight~$-w$ follows from the
  above since~$\beta=1$ is the unique eigenvalue of weight~$0$
  on~$H^*(A_{\bar{k}},\bQl)$.
  
  (2) If $M$ is an arithmetically simple perverse sheaf and~$e$ is a
  point where $M$ is lisse, then the eigenvalues $\alpha$ above have
  weight~$w=-\dim(\supp(M))$, and there is at least one $\alpha$ since
  the stalk at~$e$ is non-zero. Thus the poles and zeros above have
  weight~$\dim(\supp(M))-i$.
\end{proof}

\section{Perverse sheaves with finitely many ramified characters}

In this section, we prove Theorem~\ref{th-finite-ab-1} in the case of an
arithmetically semisimple perverse sheaf of weight zero in
$\Ppintarith(A)$ which has the property that the set of ramified
characters for~$M$ is finite. This applies in particular, for instance,
if the abelian variety~$A$ is simple, since the set of ramified
characters is a finite union of \tacs\ of~$A$ (see
Remark~\ref{rm-ab-var-unram}), and each \tac\ is reduced to a single
character if~$A$ is simple.



Let~$M$ be an arithmetically semisimple perverse sheaf of weight zero
such that the set $\mcS$ of ramified characters for~$M$ is finite. We
will prove the following:

\begin{proposition}
  Under the above assumptions, if~$M$ is quasi-unipotent and
  non-punctual, then it is negligible.
\end{proposition}

In Theorem~\ref{th-finite-ab-1}, we assume that~$M$ is quasi-unipotent
and that $M_{\intt}$ is non-zero; comparing with the proposition, it
implies that~$M$ must be punctual.

We now prove the proposition.  After a finite extension of~$k$, we may
assume that each $\chi\in\mcS$ is in $\charg{A}(k)$.
  
One reduces using Lemma~\ref{lem-arith-ss-geo} to the case of~$M$
geometrically simple.  We denote by~$S$ the support of~$M$ and by $r$
its dimension; we have $r\geq 1$ since $M$ is not punctual.  We denote
by~$U$ a smooth open dense subset of~$S$ such that~$M$ is lisse on~$U$.

Let $n\geq 1$ and let~$a\in A(k_n)$. We denote
$M^{(a)}=[x\mapsto x+a]^*M$, which is a simple perverse sheaf
on~$A_{k_n}$. The stalk of~$M^{(a)}$ at~$e$ is canonically isomorphic to
the stalk~$M_a$ of~$M$ at~$a$. We note that the set of ramified
characters for~$M^{(a)}$ is also contained in~$\mcS$, and that~$M^{(a)}$
is quasi-unipotent (see Remark~\ref{rm-quasi-unipotent}, (3)).


We then write
$$
\lhat(M^{(a)},T)=\lhat_0(M^{(a)},T)\prod_{\chi\in\mcS}\det(1-T\Fr_{k}\mid
H^*(A_{\bar{k}}, M^{(a)}_{\chi}))^{-1}
$$
where
$$
\lhat_0(M^{(a)},T)=\prod_{\chi\notin\mcS}\det(1-T^{\deg(\chi)}
\Fr_{k_{\deg(\chi)}}\mid
H^0(A_{\bar{k}}, M^{(a)}_{\chi}))^{-1}.
$$

Note that $\lhat_0(M^{(a)},T)$ is a rational function since
$\lhat(M^{(a)},T)$ is one (Proposition~\ref{pr-lhat}).

The quasi-unipotence property of~$M^{(a)}$ shows that the infinite
product $\lhat_0(M^{(a)},T)$ can be viewed as a formal power series in
$\mcO[[T]]$ for some cyclotomic order~$\mcO$.  We can apply
Lemma~\ref{lm-lambda1} to any non-archimedean place~$\lambda$ of~$\mcO$,
since the eigenvalues of Frobenius on $H^0(A_{\bar{k}},M^{(a)}_{\chi})$
are roots of unity of bounded order for all unramified characters
$\chi$.  This implies that, for any non-archimedean place $\lambda$, the
infinite product $\lhat_0(M^{(a)},T)$ converges in the disc defined by
$|T|_{\lambda}<1$. Taking $\lambda$ to correspond to places above the
characteristic of~$k$, this implies that the rational function
$\lhat_0(M^{(a)},T)$ cannot have a zero or pole which is a $|k|$-Weil
number of positive weight.

Suppose that~$a\in (A\setminus S)(\bar{k})$. Then~$M^{(a)}_e=0$. Hence we
deduce that
\begin{equation}\label{eq-outside}
  1=\prod_{\chi\in\mcS}\wtr_r(\det(1-T\Fr_k\mid
  H^*(A_{\bar{k}},M^{(a)}_{\chi}))).
\end{equation}
Thus, the Frobenius automorphism has no eigenvalue of weight~$-r$ acting
on any of the cohomology spaces $H^i(A_{\bar{k}},M^{(a)}_{\chi})$. By
purity, this translates to the condition
$$
H^{-r}(A_{\bar{k}},M^{(a)}_{\chi})=0
$$
for all~$\chi\in\mcS$. 

On the other hand, suppose that~$a\in U(\bar{k})$. Then $M^{(a)}_e=M_a$ is
pure of weight~$-r$. From the above and
Proposition~\ref{pr-lhat-weights1}, (2), we deduce that
\begin{equation}\label{eq-inside}
  \lhat(M^{(a)},T)=\wtr_r(\lhat(M^{(a)},T))
  =\prod_{\chi\in\mcS}\wtr_r(\det(1-T\Fr_k\mid
  H^*(A_{\bar{k}},M^{(a)}_{\chi}))),
\end{equation}
and since the left-hand side is not~$1$, there exists (by purity again)
at least one~$\chi\in\mcS$ such that
$$
H^{-r}(A_{\bar{k}},M^{(a)}_{\chi})\not=0.
$$

If we combine these two statements, we conclude that~$S=A$. Indeed, the
spaces $H^{-r}(A_{\bar{k}},M^{(a)}_{\chi})$ are independent
of~$a\in A(\bar{k})$ up to isomorphism. Hence, since there exists some
$a_0\in U(\bar{k})$, if one of these spaces is non-zero, then no
$a\in A(\bar{k})$ can satisfy the condition required to have
$a\notin S(\bar{k})$.

Fixing again~$a_0\in U(\bar{k})$, let $\chi\in\mcS$ be such that
$$
H^{-r}(A_{\bar{k}},M^{(a_0)}_{\chi})
$$
is non-zero. Since~$M^{(a_0)}$ is a simple perverse sheaf supported
on~$S=A$, and $r=\dim(S)=\dim(A)$, it follows from Lemma~\ref{lm-morel}
that~$M^{(a_0)}_{\chi}$ is geometrically trivial. This implies that~$M$
is negligible.

\section{The general case}

In this section, we prove Theorem~\ref{th-finite-ab-1} in the general
case. Thus let $M$ be an arithmetically semisimple perverse sheaf of
weight zero in $\Ppintarith(A)$, which we assume is quasi-unipotent and
not punctual. We will show that $M$ is negligible.

It suffices to treat the case of a simple perverse sheaf~$M$
(Lemma~\ref{lem-arith-ss-geo}). 

We denote by $S$ the support of~$M$ and by $r$ its dimension; we have
$r\geq 1$ by our assumption that~$M$ is not punctual. Let $U$ be an open
dense subset of~$S$ contained in the smooth locus of~$S$ such that~$M$
is lisse on~$U$.

Let $(\mcS)_{i\in I}$ be a finite family of \tacs\ such that the set of
ramified characters is contained in the union~$\mcS$ of the
$\mcS_i$. After a finite extension of~$k$, we may assume that each
$\mcS_i$ is defined by a quotient morphism $\pi_i\colon A\to A_i$
defined over~$k$ and a character $\chi_i\in\charg{A}(k)$.

For any subset $J$ of~$I$, we denote by $\mcS_J$ the intersection of
$\mcS_i$ for $i\in J$; this is either empty or a \tac\ of~$A$, also
defined over~$k$ (Lemma~\ref{lem:inter:tac}), in which case we denote by
$\pi_J\colon A\to A_J$ and~$\chi_J$ the corresponding quotient morphism
and character; these are all defined over~$k$. From
Lemma~\ref{lem:inter:tac}, it follows also that $\ker(\pi_J)$ is the
algebraic subgroup of~$A$ generated by the family of subgroups
$(\ker(\pi_i))_{i\in I}$.  We write $[\mcS_J]$ for the set of classes
in~$[\charg{A}]$ of characters in~$\mcS_J$.

Let~$a\in A(k)$. We denote~$M^{(a)}=[x\mapsto x+a]^*M$, so that the
stalk $M^{(a)}_e$ is canonically isomorphic to the stalk~$M_a$ of~$M$
at~$a$. The ramified characters for~$M^{(a)}$ are also contained in the
\tac\ $\mcS$, and the perverse sheaf~$M^{(a)}$ is quasi-unipotent (see
Remark~\ref{rm-quasi-unipotent}, (3)).


We define
$$
\lhat_0(M^{(a)},T)=\prod_{\chi\in[\charg{A}]\setminus
  \mcS}\det(1-T^{\deg(\chi)}\Fr_{k_{\deg(\chi)}}\mid
H^0(A_{\bar{k}},M^{(a)}_{\chi}))^{-1}.
$$
  
By an application of inclusion--exclusion, we have
$$
\lhat(M^{(a)},T)= \lhat_0(M^{(a)},T) \prod_{\substack{\emptyset
    \not=J\subset I\\\mcS_J\not=\emptyset}} \prod_{\chi\in [\mcS_J]}
\det(1-T^{\deg(\chi)}\Fr_{k_{\deg(\chi)}}\mid
H^*(A_{\bar{k}},M^{(a)}_{\chi}))^{(-1)^{|J|}}.
$$
\par
For any $J\subset I$ such that $\mcS_J$ is not empty, we denote
$$
Q_J^{(a)}=R\pi_{J*}M^{(a)}_{\chi_J}.
$$
Proposition~\ref{pr-tac} implies the formula
\begin{equation}\label{eq-aux-lhat}
  \prod_{\chi \in [\mcS_J]} \det(1-T^{\deg(\chi)}\Fr_{k_{\deg(\chi)}}\mid
  H^*(A_{\bar{k}},M^{(a)}_{\chi}))^{-1}= \lhat(Q_J^{(a)},T),
\end{equation}
so that we can rewrite the above expression as
\begin{equation}\label{eq-lhat-full}
  \lhat(M^{(a)},T)= \lhat_0(M^{(a)},T) \prod_{\substack{\emptyset \not=J\subset
      I\\\mcS_J\not=\emptyset}} \lhat(Q_J^{a},T)^{(-1)^{|J|+1}}.
\end{equation}
\par
By Proposition~\ref{pr-lhat}\,\ref{pr-lhat:item2}, this shows in
particular that $\lhat_0(M^{(a)},T)$ is a rational function.  The
quasi-unipotence property of~$M^{(a)}$ shows that the infinite product
$\lhat_0(M^{(a)},T)$ can be viewed as a formal power series in
$\mcO[[T]]$ for some cyclotomic order~$\mcO$.  We can apply
Lemma~\ref{lm-lambda1} to any non-archimedean place~$\lambda$ of~$\mcO$,
since the eigenvalues of Frobenius on $H^0(A_{\bar{k}},M^{(a)}_{\chi})$
are roots of unity of bounded order for all unramified characters
$\chi$.  This implies that, for any non-archimedean place $\lambda$, the
infinite product $\lhat_0(M^{(a)},T)$ converges in the disc defined by
$|T|_{\lambda}<1$. Taking $\lambda$ to correspond to places above the
characteristic of~$k$, this implies that the rational function
$\lhat_0(M^{(a)},T)$ cannot have a zero or pole which is a $|k|$-Weil
number of positive weight.

Since $r\geq 1$, the formula~(\ref{eq-lhat-full}) therefore implies
the formula
\begin{equation}\label{eq-lhat-weight}
  \wtr_r(\lhat(M^{(a)},T))=  \prod_{\substack{\emptyset \not=J\subset
      I\\\mcS_J\not=\emptyset}}
  \wtr_r\left(\lhat(Q_J^{a},T)^{(-1)^{|J|+1}}\right).
\end{equation}

Let $J\subset I$. By proper base change, we have a canonical
isomorphism
$$
Q^{(a)}_{J,e}=(R\pi_{J*}M^{(a)}_{\chi_J})_e\simeq
H^*(\ker(\pi_J)_{\bar{k}},M^{(a)}_{\chi_J}).
$$
\par
Since~$M^{(a)}$, hence also $M^{(a)}_{\chi_J}$, is a perverse sheaf, the
complex $M^{(a)}_{\chi_J}$ is concentrated in degrees between $-r$
and~$r$. Its support is $S-a$, and consequently, the cohomology group
$$
H^i(\ker(\pi_J)_{\bar{k}},M^{(a)}_{\chi_J})= H^i((\ker(\pi_J)\cap
(S-a))_{\bar{k}},M^{(a)}_{\chi_J})
$$
vanishes unless $0\leq i+r\leq 2\dim(\ker(\pi_J)\cap (S-a))$. Since
$M^{(a)}$ has weight~$0$, this space has weight $i$ when it is
non-zero. Using the formula
$$
\lhat(Q^{(a)}_J,T)=\exp\Bigl(\sum_{n\geq
  1}|A_J(k_n)|t_{Q^{(a)}_J}(e;k_n)\frac{T^n}{n}\Bigr)
$$
of Proposition~\ref{pr-lhat}, this means that
\begin{align*}
  \wtr_r(\lhat(Q^{(a)}_J,T))
  &=\det(1-T\Fr_k\mid H^{-r}((\ker(\pi_J)\cap
  (S-a))_{\bar{k}},M^{(a)}_{\chi_J}))^{-1} \\
  &=\det(1-T\Fr_k\mid H^{-r}(((a+\ker(\pi_J))\cap
  S)_{\bar{k}},M_{\chi_J}))^{-1} \\
  &=\det(1-T\Fr_k\mid (R^{-r}\pi_{J*}M_{\chi_J})_{\pi_J(a)})^{-1}=
  \det(1-T\Fr_k\mid (\pi_J^*R^{-r}\pi_{J*}M_{\chi_J})_{a})^{-1}.
\end{align*}

Let~$X=S\setminus U$, so that~$A\setminus X=U\cup (A\setminus S)$.
If $a\in (A\setminus X)(k)$, the left-hand side
of~(\ref{eq-lhat-weight}) is the part of weight $-r$ of
$$
\lhat(M^{(a)},T)=\exp\Bigl(\sum_{n\geq
  1}|A(k_n)|t_{M}(a;k_n)\frac{T^n}{n}\Bigr).
$$
\par
Since $M_a$ is $|k|$-pure of weight~$-r$ (either because $a\in U(k)$, so
that~$M$ is lisse and of weight~$-r$ at~$a$, or because
$a\in (A\setminus S)(k)$, so that~$M_a$ is zero, hence pure of any
weight), we deduce that the equality
\begin{equation}\label{eq-lhat-weight1}
  \det(1-T\Fr_k\mid M_a)= \prod_{\substack{\emptyset \not=J\subset
      I\\\mcS_J\not=\emptyset}}
  \det(1-T\Fr_k\mid (\pi_J^*R^{-r}\pi_{J*}M_{\chi_J})_{a})^{(-1)^{|J|+1}}
\end{equation}
holds. In particular, this gives the equality
$$
t_M(a;k)= \sum_{\substack{\emptyset \not=J\subset
    I\\\mcS_J\not=\emptyset}} (-1)^{|J|+1}
t_{\pi_J^*R^{-r}\pi_{J*}M_{\chi_J}}(a;k)
$$
of values of trace functions for $a\in (A\setminus X)(k)$.

Let $n\geq 1$. Applying this argument to the base change of~$M$
to~$k_n$, we see that the formula
$$
t_M(a;k_n)= \sum_{\substack{\emptyset \not=J\subset
    I\\\mcS_J\not=\emptyset}} (-1)^{|J|+1}
t_{\pi_J^*R^{-r}\pi_{J*}M_{\chi_J}}(a;k_n)
$$
holds for $a\in (A\setminus X)(k_n)$.  By the injectivity of trace
functions (see~\cite[Th.\,1.1.2]{laumon-signes}), this means that we
have an equality
\begin{equation}\label{eq-k-equality}
  M=\sum_{\substack{\emptyset \not=J\subset I\\\mcS_J\not=\emptyset}}
  (-1)^{|J|+1} \pi_J^*R^{-r}\pi_{J*}M_{\chi_J}
\end{equation}
in the Grothendieck group $K(A\setminus X)$.

If~$U=S$ (e.g. if~$M$ is the extension by zero of a lisse sheaf of
weight~$0$ placed in degree~$-r$ on a smooth closed subvariety~$S$,
which will be the case in the applications of
Theorem~\ref{th-finite-ab-1} in Chapter~\ref{sec-jacobian}), then $X$
is empty, so this equality holds in $K(A)$. The right-hand side is a
linear combination of negligible objects (see
Example~\ref{ex-negligible}) so we deduce that~$M$ is negligible by
taking the Euler--Poincaré characteristic (see
Corollary~\ref{cor-grothendieck}).
  
We now consider the general case. Let~$j$ be the open immersion of
$A\setminus X$ in~$A$.  Recall that the classes of simple perverse
sheaves form a basis of the $\Zz$-module $K(A\setminus X)$ (see
Proposition~\ref{pr-injectivity}). Thus, the
equality~(\ref{eq-k-equality}) implies that there exists some~$J$ such
that the simple perverse sheaf~$j^*M$ appears in the decomposition in
simple perverse sheaves of the class of~$j^*N$ in~$K(A\setminus X)$,
where
$$
N=\pi_J^*R^{-r}\pi_{J*}M_{\chi_J}.
$$
Furthermore, this means that there exists $i\in\Zz$ such that~$j^*M$
occurs in the decomposition of the perverse sheaf~$\pH^i(j^*N)$, since
$$
j^*N=\sum_{i\in\Zz}(-1)^i \ \pH^i(j^*N)
$$
in~$K(A\setminus X)$.
\par
The functor $j^*$ is $t$-exact (since~$j$ is smooth of relative
dimension~$0$) so that there exists a canonical isomorphism
$\pH^i(j^*N)\to j^*\pH^i(N)$. Since~$j^*M$ and~$j^*\pH^i(N)$ are pure,
hence geometrically semisimple, this implies the existence of an
injective morphism
$$
f\colon j^*M\to j^*\pH^i(N)
$$
of perverse sheaves.  Applying the functor $j_{!*}$, which preserves
injectivity (e.g., by~\cite[\S.\,2.17]{katz-rls}) and satisfies
$j_{!*}\circ j^*=\mathrm{Id}$ on perverse sheaves, we deduce that there
exists an injective morphism $j_{!*}f\colon M\to \pH^i(N)$. Since~$N$ is
negligible, so is~$\pH^i(N)$, and hence also~$M$.

\begin{remark}\label{rm-negligible}
  A similar argument leads to a proof of the following fact: if $M$ is a
  negligible arithmetically simple perverse sheaf of weight zero in
  $\Ppintarith(A)$, and if $\nunram{M}$ is contained in a finite union
  of \tacs\ of~$A$, then there exists a morphism $\pi\colon A\to B$ of
  abelian varieties with $\dim(\ker(\pi))\geq 1$, a
  character~$\chi\in\charg{A}$ and an object~$N$ of~$\Der(B)$ such that
  $M$ is geometrically isomorphic to~$(\pi^*N)_{\chi}$.
  
  This fact is equivalent (for pure perverse sheaves of weight~$0$) to
  the characterization of negligible objects by
  Weissauer~\cite[Th.\,3]{weissauer_vanishing_2016}, since it is known
  that the assumption on~$\nunram{M}$ is always true
  (Corollary~\ref{cor-tacs-negligible}). However, the proof that this
  is so relies on the generic vanishing theorem
  (Theorem~\ref{thm-gen-van-AV}), which appeals to this result of
  Weissauer, so this remark does not provide a different proof of this
  characterization.

  We sketch the argument nevertheless for the sake of illustration. It
  is relatively elementary that it suffices to prove that the
  isomorphism class of~$M$ is invariant under translation by a
  non-trivial abelian subvariety (this
  is~\cite[Lemma\,6]{weissauer_vanishing_2016}), and we will establish
  this fact.

  To simplify matters, we assume that~$S=U$ in the notation of the
  previous proof. Since~$M$ is negligible, it is quasi-unipotent;
  arguing as in the previous proof, we obtain a finite decomposition
  $$
  M=\sum_{i\in I}n_i \pi_i^*M_i
  $$
  in $K(A)$ for some morphisms $\pi_i\colon A\to A_i$ with
  $\dim(\ker(\pi_i))\geq 1$, some objects $M_i\in\Der(A_i)$ and some
  non-zero $n_i\in\Zz$.
  \par
  Since the classes of simple perverse sheaves form a basis of the
  $\Zz$-module $K(A)$, there exists some $i\in I$ such that
  $$
  \pi_i^*M_i=mM+\sum_{j\in J}m_j M_{i,j}
  $$
  in $K(A)$ for some non-zero integers $m$ and $m_j$ and some simple
  perverse sheaves $M_{i,j}$ not isomorphic to~$M$.

  The isomorphism class of the complex $\pi_i^*M_i$ is invariant under
  translation by elements of $\ker(\pi_i)$, and a fortiori by the
  abelian subvariety $A'=\ker(\pi_i)^{\circ}$. We claim that this
  implies that the same property holds for $M$ and the other
  constituents $M_{i,j}$. Indeed, the $\bar{k}$-valued points of $A'$
  act on the finite set of isomorphism classes of the simple perverse
  sheaves $(M,M_{i,j})$, and thus the stabilizer of any of them is a
  finite index subgroup. Since it is also an algebraic subgroup, it is
  equal to $A'(\bar{k})$, and the assertion follows.  Thus the
  isomorphism class of~$M$ is invariant under translation by the
  non-trivial abelian variety~$\ker(\pi_i)^{\circ}$, as desired.
\end{remark}


\chapter{Stratification and generic Fourier
  invertibility}\label{sec-stratification}

As usual, $k$ is a finite field, with an algebraic closure $\bar{k}$ and
finite extensions $k_n$ of $k$ in $\bar{k}$ of degree~$n$. We fix a
prime $\ell$ distinct from the characteristic of~$k$.

Let $G$ be a connected commutative algebraic group over~$k$, with
dimension~$d$. Given an object $M$ of~$\Der(G)$, an integer~$n\geq 1$
and a character $\chi\in\charg{G}(k_n)$, we set
\[
S(M,\chi)=\sum_{x\in  G(k_n)}\chi(x)t_M(x;k_n).
\]

\section{Stratification for exponential sums}

The results of this section are straightforward consequences of
Theorem~\ref{thm-high-vanish} and Deligne's Riemann Hypothesis. We spell
them out since some of them are likely to be useful for applications to
analytic number theory.

\begin{theorem}\label{th-stratified}
  Let~$M$ be an object of $\Der(G)$. Assume that $M$ is semiperverse and
  mixed of weights $\leq 0$.  There exist subsets
  $\mcS_d\subset\dots\subset \mcS_0=\charg{G}$ such that
  \begin{enumth}
  \item For~$0\leq i\leq d$, the estimate
    $$
    |\mcS_i(k_n)|\ll |k|^{n(d-i)}
    $$
   holds for $n\geq 1$.
  \item The set $\mcS_d$ is empty if $M$ belongs to the category
    $\Ppintarith(G)$. 
  \item For any $n\geq 1$, any integer~$i$ with $1\leq i\leq d$ and any
    $\chi\in\charg{G}(k_n)\setminus \mcS_i(k_n)$, the estimate
    $$
    \sum_{x\in G(k_n)}\chi(x)t_M(x;k_n)\ll c_u(M)|k|^{n(i-1)/2},
    $$
    holds, where the implied constant is independent of~$M$.
  \item If~$G$ is either a torus or an abelian variety, then $\mcS_i$ is
    a finite union of \tacs\ of~$G$ of
    dimension~$\leq d-i$.
  \item If~$G$ is a unipotent group, then $\mcS_i$ is the set of closed
    points of a closed subvariety of dimension $\leq d-i$ of the Serre
    dual $G^\vee$.
  \end{enumth}
\end{theorem}

\begin{proof}
  For $1\leq i\leq d$, let $\mcS_i$ be the set of characters such that
  there exists some $l\geq i$ with $H^l_c(G_{\bar{k}},M_{\chi})\not=0$.

  If $\chi\in \charg{G}\setminus \mcS_i$, then we deduce from the trace
  formula and the Riemann Hypothesis of Deligne, combined with
  Lemma~\ref{lm-sums-betti}, that the estimate
  $$
  \sum_{x\in G(k_n)}\chi(x)t_M(x;k_n)\ll c_u(M)|k|^{n(i-1)/2}
  $$
  holds for $n\geq 1$, which is~(3). We will check that these sets also
  satisfy conditions (1) and~(2).

  Fix an integer~$i$ with $1\leq i\leq d$. For any $l$ and $j$, we have
  the perverse spectral sequence
  $$
  \rmH^j_c(G_{\bar{k}},\pH^{l-j}(M_\chi))\implies
  \rmH_c^{j}(G_{\bar{k}},M_\chi)
  $$
  (see~(\ref{eq-perverse-spectral-seq}))
  so that the condition $\chi\in\mcS_i$ implies that
  $$
  \rmH^j_c(G_{\bar{k}},\pH^{l-j}(M_\chi))\not=0
  $$
  for some $l\geq i$. Since $M$ is semiperverse, so is $M_{\chi}$,
  which means that this condition implies $j\geq l\geq i$.

  Thus, if we denote by $(\mcS_{j,i})_{0\leq i\leq d}$ the sets provided
  by the Stratified Vanishing Theorem~\ref{thm-high-vanish} applied to
  $\pH^j(M)$, we have shown that
  $$
  \mcS_i\subset \bigcup_{i\leq j\leq d}\bigcup_{i\leq l\leq j}
  \mcS_{l-j,j}.
  $$
  \par
  The set $\mcS_{l-j,j}$ has character codimension at least $j$, so that
  $\mcS_i$ has the same property, establishing~(1).  Point (2) follows
  from the fact that
  $$
  H^d_c(G_{\bar{k}},N_{\chi})=0
  $$
  for a geometrically simple perverse sheaf~$N$ which is not
  geometrically isomorphic to~$\mcL_{\chi^{-1}}$.

  Points (4) and (5) follow from the strengthened versions of the
  Stratified Vanishing Theorem for tori, abelian varieties and unipotent
  groups, which are stated in Theorem~\ref{thm-high-vanish}, (4) or
  Proposition~\ref{prop-strat-unip}.
\end{proof}

\begin{remark}
  The following elementary estimate can also sometimes be useful. Fix a
  locally-closed immersion $u\colon G\to \Pp^m$ for some
  integer~$m\geq 1$. Let~$M$ be an $\ell$-adic perverse sheaf on $G$
  that is pure of weight zero. Then by orthogonality of characters, we
  derive that the formula
  $$
  \frac{1}{|G(k_n)|}
  \Bigl|\sum_{\chi\in\charg{G}(k_n)}t_M(x;k_n)\chi(x)\Bigr|^2=
  \sum_{x\in G(k_n)}|t_M(x;k_n)|^2
  $$
  holds for $n\geq 1$. By the Riemann Hypothesis, it follows that the
  estimate
  $$
  \frac{1}{|G(k_n)|} \sum_{\chi\in\charg{G}(k_n)}
  \Bigl|\sum_{\chi\in\charg{G}(k_n)}t_M(x;k_n)\chi(x)\Bigr|^2 \ll c_u(M)
  $$
  holds for $n\geq 1$ (see~Theorem~\ref{th-rh}). Fix then a sequence
  $T=(T_n)$ of positive real numbers, and let $\mcX_T\subset\charg{G}$
  be the set such that $\chi\in\mcX_T( k_n)$ if and only if
  $$
  \Bigl|\sum_{x\in G(k_n)}t_M(x;k_n)\chi(x)\Bigr|>T_n.
  $$
  Then we find by positivity that
  $$
  |\mcX_T(k_n)|\ll c_u(M) |G(k_n)|T_n^{-2}.
  $$
\end{remark}

\begin{corollary}\label{cor-stratified-generic}
  Let~$k$ be a finite field, and let $G$ be a connected commutative
  algebraic group of dimension $d$ over~$k$.  Let $\ell$ be a prime
  distinct from the characteristic of~$k$ and let~$M$ be an $\ell$-adic
  perverse sheaf on $G$ which is pure of weight zero.

  For any generic subsets $\mcX$ and $\mcY$ of $\charg{G}$, the estimate
  $$
  \frac{1}{|G(k_n)|}\sum_{\chi\in (\mcX\setminus \mcY)(k_n)}\Bigl|
  \sum_{x\in k_n}\chi(x)t_M(x;k_n)\Bigr|\ll \frac{1}{|k|^{n/2}}
  $$
  holds for all $n\geq 1$.
\end{corollary}

\begin{proof}
  We may assume that $\mcX=\charg{G}$. Let $(\mcS_i)_{0\leq i\leq d}$ be
  sets of characters as in Theorem~\ref{th-stratified}. We have
  $$
  \frac{1}{|G(k_n)|}\sum_{\chi\in (\charg{G}\setminus \mcY)(k_n)}\Bigl|
  \sum_{x\in k_n}\chi(x)t_M(x;k_n)\Bigr|= S_0+\cdots+S_{d-1},
  $$
  where for each integer $i$ with $0\leq i< d$, we put
  $$
  S_i=\frac{1}{|G(k_n)|}\sum_{\chi\in (\mcS_i\setminus (\mcS_{i+1}\cup
    \mcY))(k_n)}\Bigl| \sum_{x\in k_n}\chi(x)t_M(x;k_n)\Bigr|.
  $$
  \par
  For $i=0$, the exponential sums in the inner sum are $\ll 1$, and
  since $\mcY$ is generic, the set $\mcS_0\setminus (\mcS_{1}\cup \mcY)$
  has character codimension at least~$1$, so that
  $$
  S_0\ll |k|^{n(-d+(d-1))}=|k|^{-n}.
  $$
  \par
  For $i\leq i<d$, we have
  $\mcS_i\setminus (\mcS_{i+1}\cup \mcY)\subset \mcS_i\setminus
  \mcS_{i+1}$, so that by Theorem~\ref{th-stratified}, (1) (for the size
  of $\mcS_i$) and (3) (estimating the exponential sums for
  $\chi\notin \mcS_{i+1}$), we obtain
  $$
  S_i\ll |k|^{-nd+n((d-i)+i/2)}=|k|^{-ni/2}.
  $$
\end{proof}


The next corollary states, intuitively, that for the purpose of
computing the arithmetic Fourier transform of a semiperverse complex
(mixed of weights $\leq 0$), the contribution of any closed (suitably
``transverse'') subvariety is negligible.

\begin{corollary}\label{cor-restriction}
  Let~$k$ be a finite field, and let $G$ be a connected commutative
  algebraic group of dimension $d$ over~$k$.  Let $\ell$ be a prime
  distinct from the characteristic of~$k$ and let~$M$ be an object of
  $\Der(G)$. Assume that $M$ is semiperverse and mixed of weights
  $\leq 0$.

  Let $X\subset G$ be a closed subvariety of~$G$ and let
  $i\colon X\to G$ be the corresponding closed immersion.

  Let $m\geq 0$ be an integer such that for each $j\in\Zz$, the estimate
  $$\dim (X\cap \Supp(\mcH^j(M)))\leq \dim\Supp(\mcH^j(M))-m
  $$
  holds.

  There exists a generic subset $\mcX$ of $\charg{G}$ such that the
  estimate
  $$
  \sum_{x\in X(k_n)}\chi(x)t_M(x;k_n)\ll \frac{c_u(M)}{|k_n|^{m/2}}
  $$
  holds for all $n\geq 1$ and all $\chi\in \mcX(k_n)$.
\end{corollary}

Alternatively, we have
$$
S(M,\chi)=\sum_{x\in (G\setminus X)(k_n)}
\chi(x)t_M(x;k_n)+O(|k_n|^{-m/2}),
$$
which explains the interpretation that $X$ does not contribute
``systematically'' to the arithmetic Fourier transform.

\begin{proof}
  The assumption implies that the complex $N=i_!i^*M[-m](-m/2)$ is
  semiperverse on~$G$, since~$M$ is semiperverse and, for any
  $j\in\Zz$, the support of $\mcH^j(N)$ is
  $X\cap\Supp(\mcH^{j-m}(M))$ so that 
  \begin{align*}
    \dim(\Supp(\mcH^j(N)))= \dim (X\cap\Supp(\mcH^{j-m}(M)))
    &  \leq \dim(\Supp(\mcH^{j-m}(M)))-m\\
    &\leq -(j-m)-m=-j.
  \end{align*}
  \par
  Moreover, the complex~$N$ has weights $\leq 0$. Thus we may apply
  Theorem~\ref{th-stratified} to~$N$. Let $\mcS_0$, \ldots, $\mcS_d$ be
  the corresponding sets of characters, and let
  $\mcX=\charg{G}\setminus \mcS_1$. This is a generic subset of
  $\charg{G}$, and for $n\geq 1$ and $\chi\in \mcX(k_n)$, we have
  $$
  (-1)^m|k_n|^{m/2}\sum_{x\in X(k_n)}\chi(x)t_M(x;k_n)= \sum_{x\in
    G(k_n)}\chi(x)t_N(x;k_n)\ll c_u(M),
  $$
  hence the result.
\end{proof}

\begin{example}
  Let $\mcF$ be a non-zero lisse sheaf on~$G$, pure of weight~$0$, and
  let $M=\mcF[d](d/2)$. We then have $\Supp(\mcH^j(M))=\emptyset$
  except when $j=-d$, in which case the support of $\mcH^{-d}(M)$
  is~$G$.  We can therefore apply the corollary to any closed
  subvariety $X$ of~$G$ of codimension at least~$m$. In particular,
  for any closed subvariety $X\not=G$, hence of codimension at least
  $1$, there exists a generic set of characters $\mcX$ for which the
  estimate
  $$
  \sum_{x\in X(k_n)}\chi(x)t_M(x;k_n)\ll \frac{c_u(M)}{|k_n|^{1/2}}
  $$
  holds for $n\geq 1$ and $\chi\in \mcX(k_n)$.
\end{example}

A uniform version of the stratified vanishing theorem, as in
Remark~\ref{rmk-HV-effective}, would be especially welcome for
stratification estimates, as it would lead to strong potential
applications in analytic number theory (compare with the results of
Fouvry and Katz~\cite{fouvry-katz} based on stratification for the
additive Fourier transform). We state a conditional result of this kind
for emphasis.

\begin{theorem}
  Let $\ell$ be a prime number. Let~$N\geq 1$ be an integer and let
  $(G,u)$ be a quasi-projective commutative group scheme
  over~$\Zz[1/\ell N]$.

  Assume that, for all primes $p\nmid \ell N$, the fiber~$G_p$ of~$G$
  over~$\Ff_p$ is a connected commutative algebraic group such that
  \emph{Theorem~\ref{thm-high-vanish}} holds uniformly with respect to
  the complexity $c_{u_p}(M)$ where $u_p$ is the locally-closed
  immersion of~$G_p$ deduced from~$u$, i.e., such that for a perverse
  sheaf~$M$ on~$G_p$, the sets $\mcS_i$ in \loccit satisfy
  $$
  |\mcS_i(k_n)|\ll |k|^{n(d-i)}
  $$
  where the implied constant depends only on $c_{u_p}(M)$.

  Let $(M_p)_{p\nmid N\ell}$ be a sequence of arithmetically semisimple
  sheaves on $G_p$, pure of weight zero, such that $c_u(M_p)\ll 1$
  for all~$p$.
  
  For each prime $p$, there exist subsets
  $\mcS_d(\Ff_p)\subset\dots\subset \mcS_0(\Ff_p)=\charg{G}_p(\Ff_p)$ such
  that
  \begin{enumth}
  \item For~$0\leq i\leq d$ and $p$ prime, we have
    $$
    |\mcS_i(\Ff_p)|\ll p^{d-i}.
    $$
  \item The set $\mcS_d(\Ff_p)$ is empty if $M_p$ belongs to the
    category $\Ppintarith(G_p)$.
  \item For any prime $p$, any integer~$i$ with $0\leq i\leq d$ and any
    $\chi\in\charg{G}_p(\Ff_p)\setminus \mcS_i(\Ff_p)$, we have
    $$
    \sum_{x\in G(\Ff_p)}\chi(x)t_p(x)\ll p^{(i-1)/2},
    $$
    where $t_p$ is the trace function of~$M_p$ over~$\Ff_p$.
  \item If $G$ is a torus or an abelian variety, then the sets
    $\mcS_{i}(\Ff_p)$ are contained in the union of a bounded number of
    \tacs\ of~$G_{\Ff_p}$ of dimension $\leq d-i$.
  \end{enumth}
\end{theorem}

\begin{remark}
  For $G=\Gg_a^d$, results of this kind are unconditional; see for
  instance~\cite[Th.\,1.1, Th.\,3.1]{fouvry-katz} (note that there the
  sets $\mcS_i$ are points of subschemes defined over~$\Zz$, which we
  cannot hope in the general situation where $M_p$ is allowed to vary
  with~$p$).
\end{remark}

In the case of $\Gg_m^d$ (which is currently conditional), this would
give (for instance) stratification and generic square-root cancellation
for sums of the type
\[
\sum_{x_1,\ldots, x_d\in\Ff_p^{\times}} \chi_1(x_1)\cdots
\chi_d(x_d)e\Bigl(\frac{f(x_1,\ldots,x_d)}{p}\Bigr), 
\]
where $f\in \Zz[X_1,\ldots,X_d]$ is a polynomial and
$\chi_1,\ldots, \chi_d$ are Dirichlet characters modulo~$p$, together
with an a priori algebraic description of the sets of characters where
the sum has size $\asymp p^{i/2}$.

Over finite fields, we can still derive some applications, such as the
following proposition, similar to~\cite[Cor.\,1.4]{fouvry-katz}
(although the vertical direction means that equidistribution is only
in the finite set $(1/p\Zz)/\Zz\subset \Rr/\Zz$, or equivalently
modulo~$p$, as we phrase it.)



\begin{proposition}
  Let~$p$ be the characteristic of~$k$.  Let~$d\geq 1$ and $r\leq d$ be
  integers and let $f=(f_i)\colon \Gg_m^d\to \Aa^r$ be a morphism whose
  image is not contained in an affine hyperplane. For a sequence $(w_n)$
  such that $0\leq w_n< |k|^n-1$ and
  $w_n/(|k_n|^{1/2}\log|k_n|)\to +\infty$, and for an arbitrary
  generator~$y_n$ of~$k_n^{\times}$, the family of residue classes
  $$
  \Tr_{k_n/\Ff_p}(f(y_n^{v_1},\ldots, y_n^{v_d}))\mods{p},\quad\quad
  0\leq v_i\leq w_n\text{ for all } i
  $$
  become uniformly distributed in $(\Zz/p\Zz)^r$.
\end{proposition}

\begin{proof}
  Let $G=\Gg_m^d$ and~$q=|k|$; for $n\geq 1$, denote by $\psi_n$ the
  additive character $x\mapsto e(\Tr_{k_n/\Ff_p}(x)/p)$ of $k_n$. Using
  the generator $y_n$, we can identify the group
  $G(k_n)=(k_n^{\times})^d$ with the group $(\Zz/(q^n-1)\Zz)^d$ and we
  also identify $\charg{G}(k_n)$ with $(\Zz/(q^n-1)\Zz)^d$, the element
  $\beta\in(\Zz/(q^n-1)\Zz)^d$ corresponding to the character~$\chi$
  such that
  $$
  \chi(y_n^{v_1},\ldots, y_n^{v_d})=
  e\Bigl(\frac{1}{q^n-1}(\beta_1v_1+\cdots+\beta_d v_d)\Bigr).
  $$
  \par
  By the Weyl Criterion, we need to prove that
  $$
  \lim_{n\to +\infty} \frac{1}{w_n^d}\sum_{0\leq v_i\leq w_n}
  \psi_n\Bigl(\sum_{i=1}^d h_i f_i(y_n^{v_1},\ldots, y_n^{v_d})\Bigr)=0
  $$
  for any $h\in(\Zz/p\Zz)^r\setminus \{0\}$. Detecting the interval
  $0\leq v\leq w_n$ by Fourier expansion, we have to study the limit
  of
  $$
  \frac{1}{w_n^d} \sum_{\chi\in\charg{G}(k_n)} \widehat{\alpha}_n(\chi)
  \sum_{x\in G(k_n)}\chi(x) \psi_n\Bigl(\sum_{i=1}^d h_if_i(x)\Bigr)
  $$
  where
  $$
  \widehat{\alpha}_n(\chi)=\frac{1}{(q^n-1)^d}\sum_{0\leq v_i\leq w_n}
  \overline{\chi(y_n^{v_1},\ldots,y_n^{v_d})}.
  $$
  \par
  Define $g_h\colon G\to \Aa^1$ by
  $$
  g_h(x)=\sum_{i=1}^dh_i f_i(x).
  $$
  \par
  We can write
  $$
  \frac{1}{w_n^d} \sum_{\chi\in\charg{G}(k_n)} \widehat{\alpha}_n(\chi)
  \sum_{x\in G(k_n)}\chi(x) \psi_n\Bigl(\sum_{i=1}^r h_if_i(x)\Bigr)=
  \frac{1}{w_n^d} \sum_{\chi\in\charg{G}(k_n)} \widehat{\alpha}_n(\chi)
  q^{nd/2}S(M,\chi)
  $$
  for the complex~$M=g_h^*\mcL_{\psi_1}[d](d/2)$ on~$G$, which is a
  simple perverse sheaf, pure of weight~$0$, on~$G$.  
  \par
  We apply Theorem~\ref{th-stratified} to~$M$. Let $(\mcS_i)$ be the
  subsets described there. We have $\mcS_d=\emptyset$ because the image
  of~$f$ is not contained in an affine hyperplane, which implies that
  $g_h$ is non-constant, and hence $M$ is non-trivial, from which the
  fact that it does not coincide with $\mcL_{\chi}$ for any
  character~$\chi\in\charg{G}$ follows.  Moreover, we also know that
  each $\mcS_i$ is a finite union of \tacs\ of~$G$ of
  dimension~$\leq d-i$.
  \par
  The contribution of all $\chi\in(\mcS_0\setminus\mcS_1)(k_n)$ to the
  previous sum satisfies the bound
  $$
  \frac{1}{w_n^d} \sum_{\chi\in(\charg{G}\setminus\mcS_1)(k_n)}
  \widehat{\alpha}_n(\chi) q^{nd/2}S(M,\chi) \ll \frac{q^{nd/2}}{w_n^d}
  \sum_{\chi\in\charg{G}(k_n)} | \widehat{\alpha}_n(\chi)|.
  $$

  It is well-known that the bound
  \begin{equation}\label{eq-log-bound}
    \sum_{\chi\in\charg{G}(k_n)} | \widehat{\alpha}_n(\chi)|\ll
    (\log q)^d
  \end{equation}
  holds for all $n\geq 1$ (see Remark~\ref{rm-log} below), where the
  implied constant depends on~$d$, so that
  $$
  \frac{1}{w_n^d} \sum_{\chi\in(\charg{G}\setminus\mcS_1)(k_n)}
  \widehat{\alpha}_n(\chi) q^{nd/2}S(M,\chi) \ll
  \Bigl(\frac{q^{n/2}\log(q)}{w_n}\Bigr)^d,
  $$
  which converges to~$0$ as $n\to+\infty$ by assumption.
  \par
  We now handle the remaining terms. Let $1\leq j\leq d-2$. By
  Theorem~\ref{th-stratified}, the estimate
  $$
  \frac{1}{w_n^d} \sum_{\chi\in(\mcS_j\setminus\mcS_{j+1})(k_n)}
  \widehat{\alpha}_n(\chi) q^{nd/2}S(M,\chi) \ll
  \frac{q^{n(d+j)/2}}{w_n^d} \sum_{\chi\in \mcS_j(k_n)} |
  \widehat{\alpha}_n(\chi)|
  $$
  holsd for all $n\geq 1$. From Lemma~\ref{lm-boring} below and the fact
  that $\mcS_j$ is a finite union of \tacs\ of codimension at least~$j$,
  we deduce that the estimate
  $$
  \sum_{\chi\in \mcS_j(k_n)} | \widehat{\alpha}_n(\chi)| \ll
  \Bigl(\frac{w_n}{q^n}\Bigr)^{j}(\log q)^d
  $$
  holds for $n\geq 1$. It follows that
  $$
  \frac{1}{w_n^d} \sum_{\chi\in(\mcS_j\setminus\mcS_{j+1})(k_n)}
  \widehat{\alpha}_n(\chi) q^{nd/2}S(M,\chi) \ll
  \frac{q^{n(d+j)/2}}{w_n^d}  \Bigl(\frac{w_n}{q^n}\Bigr)^{j}(\log
  q)^d=
  \Bigl(\frac{q^{n/2}}{w_n}\Bigr)^{d-j}(\log q)^d.
  $$
  \par
  The conclusion follows.
\end{proof}

\begin{lemma}\label{lm-boring}
  With notation as above, for any \tac\ $\mcS$ of~$\Gg_m^d$ of
  dimension~$d-j<d$, we have
  $$
  \sum_{\chi\in \mcS(k_n)} | \widehat{\alpha}_n(\chi)| \ll
  \Bigl(\frac{w_n}{q^n}\Bigr)^{j}(\log q)^{d}.
  $$
\end{lemma}

\begin{proof}
  \emph{Mutatis mutandis}, this is very close
  to~\cite[Lemma~\,9.5]{fouvry-katz}, in the (simpler) case where the
  variety $\mcV$ of loc. cit. is an affine hyperplane (but with the
  primes~$p$ replaced by the sequence $q^n-1$). Indeed, let
  $f\colon \Gg_m^d\to \Gg_m^{d-j}$ and
  $\chi_0=(\chi_{0,1},\ldots,\chi_{0,d})$ be the morphism of tori and
  the character $\chi_0$ defining $\mcS$. There exists a matrix
  $m=(m_{k,l})$ of size $d\times (d-j)$ with integral coefficients of
  rank~$j$ such that
  a character $\eta=(\eta_1,\ldots,\eta_d)$ of~$\Gg_m^d$ belongs to
  $\mcS(k_n)$ if and only if
  $$
  \prod_{l=1}^d\eta_l^{m_{k,l}}=\chi_{0,k}^{-1}
  $$
  for $1\leq k\leq d-j$. When we identify $\charg{\Gg}_m^d(k_n)$ with
  $(\Zz/(q^n-1)\Zz)^d$, this means that $\mcS(k_n)$ is identified with
  the set of solutions $(\xi_1,\ldots,\xi_d)$ in $(\Zz/(q^n-1))^d$ of
  the linear equation
  $$
  \sum_{l=1}^dm_{k,l}\xi_l=y_k
  $$
  for some $y_k\in (\Zz/(q^n-1)\Zz)$.
\end{proof}

\begin{remark}\label{rm-log}
  We recall the proof of~(\ref{eq-log-bound}). We denote $N=q^n-1$ so
  that $\charg{G}(k_n)$ is isomorphic to $(\Zz/N\Zz)^d$. The sum to
  estimate is
  \begin{align*}
    \frac{1}{N^d}\sum_{\xi\in (\Zz/N\Zz)^d} \Bigl| \sum_{0\leq v_i\leq
      w_n} e\Bigl(\frac{\xi_1v_1+\cdots +\xi_dv_d}{N}\Bigr) \Bigr| &
    =\frac{1}{N^d} \sum_{\xi\in (\Zz/N\Zz)^d} \prod_{i=1}^d \Bigl|
    \sum_{0\leq v\leq w_n} e\Bigl(\frac{\xi_iv}{N}\Bigr) \Bigr|
    \\
    &=\frac{1}{N^d}\prod_{i=1}^d \sum_{\xi\in (\Zz/N\Zz)} \Bigl|
    \sum_{0\leq v\leq w_n} e\Bigl(\frac{\xi v}{N}\Bigr) \Bigr|,
  \end{align*}
  which shows that it is enough to handle the case~$d=1$. In this case,
  one uses the bound
  $$
  \Bigl| \sum_{0\leq v\leq w_n} e\Bigl(\frac{\xi v}{N}\Bigr) \Bigr|
  \leq
  \min\Bigl(w_n+1,\frac{e((w_n+1)\xi/N)-1}{e(\xi/N)-1}
  \Bigr)
  \leq \min\Bigl(w_n+1,\frac{1}{2\|\xi/N\|}\Bigr)
  $$
  where~$\|\xi/N\|$ is the distance to the nearest integer of $\xi/N$
  (the sum is a finite geometric sum, and in the last step, we used the
  lower-bound $|\sin(x)|\geq 2\|x\|$, valid for $x\in\Rr$). We then sum
  over the range~$0\leq \xi\leq N-1$; for $\xi=0$, the bound is
  $\leq (w_n+1)\leq N$, and for $1\leq \xi\leq N-1$, we have
  $$
  \|\xi/N\|\geq \min\Bigl(\frac{\xi}{N},\frac{N-\xi}{N}\Bigr),
  $$
  hence
  $$
  \frac{1}{N}\sum_{\xi\in (\Zz/N\Zz)} \Bigl|
  \sum_{0\leq v\leq w_n} e\Bigl(\frac{\xi v}{N}\Bigr) \Bigr|
  \leq 1+\sum_{\xi=1}^{(N-1)/2}\frac{1}{\xi}\ll \log N,
  $$
  as claimed.
\end{remark}

\section{Generic Fourier invertibility}
\label{sec-generic-inversion}

For two semisimple perverse sheaves~$M$ and~$N$,
Proposition~\ref{prop:discreteFouriertransform} implies that if the
arithmetic Fourier transforms of~$M$ and~$N$ coincide, in the sense that
$S(M,\chi)=S(N,\chi)$ for any $\chi\in\charg{G}$, then the trace
functions of $M$ and~$N$ coincide over~$k_n$ for all $n\geq 1$, which
implies that $M$ and~$N$ are isomorphic
(Proposition~\ref{pr-injectivity}; see
also~\cite[Prop.\,4.2.3]{loeser-determinant} for tori).

The stratified vanishing theorem allows us to prove a statement of
``generic Fourier invertibility'' for pure perverse sheaves, which relaxes
the condition of equality of \emph{all} sums $S(M,\chi)$ and $S(N,\chi)$
to a condition for a \emph{generic} set of characters.

\begin{theorem}[Generic Fourier invertibility]\label{th-generic-inversion}
  \index{generic Fourier invertibility}
  Let $G$ be a connected commutative algebraic group over~$k$.  Let $M$
  and~$N$ be arithmetically semisimple $\ell$-adic perverse sheaves
  on~$G$ which are pure of weight zero.

  The perverse sheaves $M_{\intt}$ and~$N_{\intt}$ are arithmetically
  isomorphic if and only if there exists a generic set
  $\mcX\subset \charg{G}$ such that $S(M,\chi)=S(N,\chi)$ for all
  $\chi\in\mcX$.
\end{theorem}

\begin{proof}
  If $M_{\intt}$ is isomorphic to $N_{\intt}$, then the sums $S(M,\chi)$
  and $S(N,\chi)$ coincide for a generic set of characters because
  $S(P,\chi)$ vanishes generically for a negligible object~$P$.
  \par
  To prove the converse, we may assume that $M=M_{\intt}$ and
  $N=N_{\intt}$, i.e., that $M$ and~$N$ are objects of $\Ppintarith(G)$.
  We then argue by induction on the sum~$m$ of the lengths of~$M$
  and~$N$.
  \par
  If $m=0$, then
  the perverse sheaves $M$ and~$N$ are
  both zero.
  \par
  Suppose now that $m\geq 1$ and that the statement holds for all pairs
  $(M_1,N_1)$ of perverse sheaves in $\Ppintarith(G)$ such that the sum
  of the lengths of $M_1$ and $N_1$ is $\leq m-1$. One at least of the
  perverse sheaves~$M$ and~$N$ is non-zero, and (up to exchanging $M$
  and $N$) we may assume that $M$ is non-zero. Let $(Q_i)_{i\in I}$ be
  the simple components (without multiplicity) of the perverse sheaf
  $M\oplus N$, and for $i\in I$, let $\mu_M(i)$ (resp. $\mu_N(i)$) be
  the multiplicity of $Q_i$ in~$M_{\intt}$ (resp. $N_{\intt}$).
  \par
  Let $\mcY$ be the set of Frobenius-unramified characters for the perverse
  sheaf
  $$
  ( M *_{\intt} N^{\vee})\oplus (M *_{\intt} M^{\vee}),
  $$
  viewed as an object of~$\braket{M\oplus N}^{\arith}$.
  \par
  For any integer $n\geq 1$, we consider the sum
  $$
  T_n=\frac{1}{|G(k_n)|} \sum_{\chi\in \mcY(k_n)}S(M*_{\intt}
  N^{\vee},\chi).
  $$
  \par
  Applying Corollary~\ref{cor-schur} after decomposing $M$ and~$N$ in
  terms of the simple perverse sheaves~$Q_i$, we obtain the formula
  $$
  \lim_{N\to +\infty} \frac{1}{N}\sum_{n\leq N}T_n=\sum_{i\in
    I}\mu_M(i)\mu_N(i).
  $$
  \par
  On the other hand, we can write
  $$
  T_n=\frac{1}{|G(k_n)|} \sum_{\chi\in\mcY(k_n)}S(M*_{\intt}
  M^{\vee},\chi) +\frac{1}{|G(k_n)|}
  \sum_{\chi\in \mcY(k_n)}\Bigl(
  S(M*_{\intt} N^{\vee},\chi) - S(M*_{\intt} M^{\vee},\chi)\Bigr)
  $$
  for any $n\geq 1$.  For $\chi$ in the generic set $\mcX\cap \mcY$, the
  assumption implies that
  $$
  S(M*_{\intt} N^{\vee},\chi) = S(M*_{\intt} M^{\vee},\chi).
  $$
  Thus, using Corollary~\ref{cor-stratified-generic}, the assumption
  implies that the bound
  $$
  \frac{1}{|G(k_n)|} \sum_{\chi\in \mcY(k_n)}\Bigl( S(M*_{\intt}
  N^{\vee},\chi) - S(M*_{\intt} M^{\vee},\chi)\Bigr) \ll |k_n|^{-1/2}
  $$
  holds for~$n\geq 1$. Applying Corollary~\ref{cor-schur} once more and
  comparing with the previous computation, we deduce that
  $$
  \sum_{i\in I}\mu_M(i)\mu_N(i)=\sum_{i\in I}\mu_M(i)^2.
  $$
  \par
  The right-hand side is $\geq 1$ since $M$ is non-zero. Hence, there
  exists $i$ such that $\mu_M(i)\mu_N(i)\geq 1$, which means that~$Q_i$
  appears with positive multiplicity in both $M$ and~$N$. Removing one
  occurrence of~$Q_i$ from $M$ and $N$, we obtain perverse sheaves $M_1$
  and~$N_1$ in $\Ppintarith(G)$ for which we can apply the induction
  hypothesis, so that $M_{1}$ is isomorphic to $N_{1}$, and adding the
  simple perverse sheaf $Q_i$ to both sides, we deduce that $M$ is
  isomorphic to $N$.
\end{proof}

\begin{remark}
  In the case of tori, this theorem can be compared with a conditional
  result of Loeser~\cite[Prop.\,4.2.5]{loeser-determinant}.
\end{remark}

\chapter{Independence of~$\ell$}\label{sec-indep}

We consider in this section a connected commutative algebraic group~$G$
over a finite field~$k$. Let~$p$ be the characteristic of~$k$. Since we
will vary the prime $\ell\not=p$, we will indicate it in the
notation. For an object $M$ of $\Der(G,\Qlb)$, we will now denote by
$t_M(x;k_n)$ the $\Qlb$-valued trace function of~$M$, and we will also
specify explicitly the isomorphisms $\iota$ used to define their
complex-valued analogues. In particular, we write $\charg{G}^{(\ell)}$ for
the set of $\ell$-adic characters.
\nomenclature[$G$]{$\charg{G}^{(\ell)}$}{$\ell$-adic characters}

We recall (see, e.g, \cite[Def.\,1.2]{fujiwara} with $E=\Cc$) that if
$A$ is a set of pairs $(\ell,\iota)$ consisting of a prime number $\ell$
different from the characteristic of~$k$ and an isomorphism
$\iota\colon \Qlb\to\Cc$, a family~$(M_{\alpha})_{\alpha\in A}$ of
objects~$M_\alpha$ of $\Perv(G,\Qlb)$ is said to be a \emph{compatible system}\index{compatible system} if for
any $n\geq 1$ and $x\in G(k_n)$, the complex numbers
$\iota(t_{M_{\alpha}}(x;k_n))$ are independent of
$\alpha=(\ell,\iota)\in A$.  This is equivalent to asking that the
eigenvalues of Frobenius for the stalk of~$M_{\alpha}$ at~$x$ are
independent of~$\alpha$.

The question we wish to address is the following:
\begin{question}
  Suppose that we have a compatible system $(M_{\alpha})_{\alpha\in A}$;
  to what extent are the arithmetic and geometric tannakian
  groups of $M_{\alpha}$ independent of~$\alpha$?
\end{question}

We note that the analogue question for the monodromy groups of a
compatible system of lisse sheaves on an algebraic variety $X$ over~$k$
(especially a curve) has been considered in depth by, among others,
Serre~\cite[p.\,1--21]{serre-indep-ell},
Larsen--Pink~\cite{larsen-pink-indep} and
Chin~\cite{chin-indep-ell}. Using Deligne's Fourier transform, this
gives corresponding answers to our question in the case of the
group~$\Gg_a$. We note also that the deepest results (such as that of
Chin) depend on the global Langlands correspondance over function
fields.

In this section, we take a first step in addressing the question. We
will only compare two objects, so for the remainder of this section, we
let $(\ell_1,\iota_1)$ and~$(\ell_2,\iota_2)$ be pairs of prime numbers and
isomorphisms $\iota_j\colon \Qellb{\ell_j}\to \Cc$.  For $j=1$, $2$, we
fix an $\ell_j$-adic arithmetically semisimple perverse sheaf~$M_j$
on~$G$ which is pure of $\iota_j$-weight zero. We assume that $M_1$
and~$M_2$ are compatible, that is we assume that the system with
$A=\{(\ell_1,\iota_1), (\ell_2,\iota_2)\}$ is compatible.

\begin{lemma}\label{lm-simple}
  The following properties hold\emph{:}
  \begin{enumth}
  \item For any $n\geq 1$, the map
  \begin{align*}
  \eta \colon \charg{G}^{(\ell_1)}(k_n) &\longrightarrow \charg{G}^{(\ell_2)}(k_n) \\
  \eta &\longmapsto \iota_2^{-1}\circ \iota_1\circ \chi
  \end{align*}
  is a bijection such that $\iota_1(S(M_1,\chi))=\iota_2(S(M_2,\eta(\chi)))$ holds for all~$\chi$.    
  \item For any $\chi\in\charg{G}^{(\ell_1)}$, the objects
    $(M_1)_{\chi}$ and $(M_2)_{\eta(\chi)}$ are compatible.
  \item The set of weakly unramified characters
    $\chi\in \charg{G}^{(\ell_1)}$ for $M_1$ such that $\eta(\chi)$ is
    weakly unramified for~$M_2$ is generic.
  \end{enumth}
\end{lemma}

\begin{proof}
  This boils down to the computation
  \begin{align*}
    \iota_1(t_{(M_1)_{\chi}}(x;k_n))&=\iota_1(t_{M_1}(x;k_n)\chi(
    N_{k_n/k}(x)))
    \\
    &=
    \iota_1(t_{M_1}(x;k_n))\iota_1(\chi(N_{k_n/k}(x)))
    \\
    &=
    \iota_2(t_{M_2}(x;k_n))\iota_2(\eta(\chi)(N_{k_n/k}(x)))=
    \iota_2(t_{(M_2)_{\eta(\chi)}}(x;k_n))
  \end{align*}
  for any $n\geq 1$ and $x\in G(k_n)$, which follows from the
  definitions, and the fact that $\wunram{M_1}$ and~$\eta^{-1}(\wunram{M_2})$ are both generic, and hence so is their
  intersection in $\charg{G}^{(\ell_1)}$.
\end{proof}

\begin{remark}\label{rm-compatible-sets}
  We will say that two sets
  $\mathcal{A}_j\subset \what{G}^{(\ell_j)}$, defined for $j=1$
  and~$j=2$, are compatible if the bijection~$\eta$ induces bijections
  $\mathcal{A}_1(k_n)\to \mathcal{A}_2(k_n)$ for all~$n\geq 1$.
\end{remark}

\begin{lemma}\label{lem-compatible-rank}
  The tannakian dimensions of $M_1$ and $M_2$ coincide.
\end{lemma}

\begin{proof}
  By Proposition~\ref{prop-dimPerv=dimVect} and the generic vanishing
  theorem, the tannakian dimension of~$M_j$ is equal to the
  Euler--Poincaré characteristic of $(M_{j})_{\chi}$ for $\chi$ in a
  generic set $\mcX_j\subset \what{G}^{(\ell_j)}$. By Lemma~\ref{lm-simple}, we can find
  $\chi\in \mcX_1$ such that $\eta(\chi)\in\mcX_2$. 
  The result then follows from the fact that, since
  $(M_1)_{\chi}$ and
  $(M_{2})_{\eta(\chi)}$ are compatible, they have the same
  Euler--Poincaré characteristic (see, e.g.,
  \cite[Lemma\,6.38]{sawin_conductors}).
\end{proof}

We denote from now on by~$r$ the common tannakian dimension
of~$M_1$ and $M_2$.  We further denote
by~$K_j$ a maximal compact subgroup of
$\iota_j(\garith{M_j})(\Cc)$, and by
$\mu_j$ the probability Haar measure
on~$K_j$. We define the measures $\nu_{cp,j}$
on~$\Un_r(\Cc)^{\sharp}$ and the measures $\nu_j$
on~$\Cc$ as in Theorems~\ref{th-4} and~\ref{th-3}, respectively (the
latter is the Sato--Tate measure of~$M_j$).

\begin{lemma}\label{lm-indep-rho}
  With notation and assumptions as above, the objects~$\rho(M_1)$
  and~$\rho(M_2)$ are compatible for any representation~$\rho$
  of~$\GL_r$.
\end{lemma}

\begin{proof}
  This is clear from the definition since the character of~$\rho$ is a
  symmetric polynomial of the eigenvalues of the matrix argument
  in~$\GL_r$.
\end{proof}

The basic information we have is the following consequence of
equidistribution.

\begin{proposition}\label{pr-sato-tate}
  With notation and assumptions as above, we have
  $\nu_{cp,1}=\nu_{cp,2}$ and $\nu_1=\nu_2$.
\end{proposition}

\begin{proof}
  It suffices to prove the equality $\nu_{cp,1}=\nu_{cp,2}$, and this is
  essentially because the measures~$\nu_{cp,j}$ are both determined by
  equidistribution of ``the same data''.

  To be precise, we first note that by the Peter--Weyl Theorem, it is
  enough to prove that
  \[
  \int_{\Un_r(\Cc)^{\sharp}} \Tr(\rho(g))d\nu_{cp,1}(g)=
  \int_{\Un_r(\Cc)^{\sharp}} \Tr(\rho(g))d\nu_{cp,2}(g)
  \]
  holds for all finite-dimensional representations~$\rho$ of~$\Un_r(\Cc)$. 
  By Theorem~\ref{th-4}, applied to the
  bounded test function $f=\Tr(\rho)$, the equality
  \[
  \int_{\Un_r(\Cc)^{\sharp}} \Tr(\rho(g))d\nu_{cp,j}(g)= \lim_{N\to
    +\infty}\frac{1}{N}\sum_{n\leq N}\frac{1}{|G(k_n)|}
  \sum_{\chi\in\mcX_j(k_n)}\Tr(\rho(\Theta_{M_j,k_n}(\chi)))
  \]
  holds, where $\mcX_j \subset \charg{G}^{(\ell_j)}$ is the set of weakly unramified characters for~$M_j$.

  By Lemma~\ref{lm-indep-rho}, the objects $\rho(M_1)$ and $\rho(M_2)$ are  compatible; hence, by Lemma~\ref{lm-simple} applied to these two perverse sheaves,
  we have
  \begin{equation}\label{eq-ind-2}
    \Tr(\rho(\Theta_{M_1,k_n}(\chi)))=
    \Tr(\rho(\Theta_{M_2,k_n}(\eta(\chi))))
  \end{equation}
  if $\chi\in\mcX_1(k_n)$ is such that~$\eta(\chi)\in\mcX_2(k_n)$.
  Therefore, the difference
  \[
  \int_{\Un_r(\Cc)^{\sharp}} \Tr(\rho(g))d\nu_{cp,1}(g)-
  \int_{\Un_r(\Cc)^{\sharp}} \Tr(\rho(g))d\nu_{cp,2}(g)
  \]
  is equal to
  \[
  \lim_{N\to +\infty}\frac{1}{N}\sum_{n\leq N}\frac{1}{|G(k_n)|} \Bigl(
  \sum_{\chi\in\mcY_1(k_n)}\Tr(\rho(\Theta_{M_1,k_n}(\chi)))-
  \sum_{\chi\in\mcY_2(k_n)}\Tr(\rho(\Theta_{M_2,k_n}(\chi))) \Bigr),
 \]
  where $\mcY_1$ (\resp $\mcY_2$) is the set of $\chi\in \mcX_1$
  satisfying $\eta(\chi)\notin \mcX_2$ (\resp the set of $\chi\in\mcX_2$
  satisfying~$\eta^{-1}(\chi)\notin \mcX_1$).

  Both of the sets~$\mcY_1$ and $\mcY_2$ have positive character
  codimension, and hence we deduce 
  \[
  \int_{\Un_r(\Cc)^{\sharp}} \Tr(\rho(g))d\nu_{cp,1}(g)- \int_{\Un_r(\Cc)^{\sharp}} \Tr(\rho(g))d\nu_{cp,2}(g)=0,
  \]
  which implies the theorem.
\end{proof}

The equality of the characteristic polynomial measure or of the
Sato--Tate measures of objects in a compatible system can provide a
considerable amount of information. In ideal cases, this equality may be
enough to imply that $\garith{M_1}$ and $\garith{M_2}$ are
isomorphic. This does happen, but it is far from being always the case.

\begin{example}
  Let $H$ be a finite group and $H\subset \Un_{|H|}(\Cc)$ be its regular
  representation. Then the Sato--Tate measure is
  $$
  \Bigl(1-\frac{1}{|H|}\Bigr)\delta_0+\frac{1}{|H|}\delta_{|H|},
  $$
  where $\delta_z$ denotes a Dirac mass at a point $z\in\Cc$.  Thus the
  Sato--Tate measure only determines the order of~$H$ in that case.

  For characteristic polynomials,
  Sutherland~\cite[Remark\,1.5]{sutherland} gives examples of
  non-isomorphic transitive finite permutation groups with the same
  distributions of characteristic polynomials.  We refer the reader to Sutherland's
  survey~\cite{sutherland} for more examples and discussion of
  Sato--Tate measures in a more traditional context.
\end{example}

\begin{corollary}\label{cor-indep-ell}
  We continue with the notation and assumptions above.

  \begin{enumth}
  \item The reductive ranks of the reductive groups $\garith{M_1}$ and
    $\garith{M_2}$ are the same.
  \item The group $\garith{M_1}$ is finite if and only if the group
    $\garith{M_2}$ is finite, and in this case, both groups have the
    same order.
  \end{enumth}
\end{corollary}

\begin{proof}
  (1) The reductive rank of $\garith{M_{j}}$ is the dimension of the
  space of characteristic polynomials of $\garith{M_j}$
  (see Serre~\cite[p.\,17]{serre-indep-ell} for this fact), and hence it is equal to
  the dimension of the support of the measure~$\nu_{cp,j}$ on
  $\Un_r(\Cc)^{\sharp}$. The result therefore follows from the
  proposition.

  (2) This holds because the group $\garith{M_j}$ is finite if and only
  if the measure~$\nu_j$ is a finite sum of Dirac masses (for the ``if''
  direction, one can use the same result of Serre as in~(1)), and if
  that is true, then the size of $\garith{M_j}$ is equal to the inverse
  of $\nu_j(\{r\})$.

  (3) This follows from another result of
  Serre~\cite[p.\,19]{serre-indep-ell}, according to which a connected
  compact subgroup of $\Un_r(\Cc)$ cannot induce the same measure on
  characteristic polynomials as a non-connected subgroup.
\end{proof}

\begin{remark}
  In the next chapter, we will also see results which imply that $M_{1}$
  has tannakian group containing $\SL_r$ if and only $M_2$ has the same
  property, and some related statements, following from Larsen's
  Alternative (see Theorem~\ref{th-larsen} or
  Proposition~\ref{pr-gabber} below).
\end{remark}

Using another result of Serre, we can prove that connectedness of the
arithmetic tannakian group is also independent of~$\ell$.

\begin{proposition}
  We continue with the notation and assumptions above. The
  group~$\garith{M_1}$ is connected if and only if the group
  $\garith{M_2}$ is connected.
\end{proposition}

This follows directly by combining Proposition~\ref{pr-sato-tate} and
the following lemma, which is the measure-theoretic version of Serre's
``zero-one law'' of~\cite[p.\,18,\, Théorème]{serre-indep-ell}.

\begin{lemma}\label{lm-serre-zero-one}
  Let~$K$ be a compact subgroup of~$\Un_r(\Cc)$. Let~$\nu$ denote the
  measure on~$\Un_r(\Cc)^{\sharp}$ image of the probability Haar
  measure of~$K$ by the natural map $K\to
  \Un_r(\Cc)^{\sharp}$. Then~$K$ is connected if and only if, for all
  functions $f\in\Zz[a_1,\ldots,a_r]$, the measure
  $\nu(\{g\in K\,\mid\, f(g)=0\})$ is equal to~$0$ or~$1$,
  where~$f(g)$ is computed with $a_i$ replaced by the coefficients of
  the characteristic polynomial of~$g$.
\end{lemma}

\begin{proof}
  (1) If~$K$ is connected, then either the set
  $\{g\in K\,\mid\, f(g)=0\}$ is all of~$K$, or it has
  codimension~$\geq 1$, and has measure zero.

  (2) If~$K$ is not connected, and~$g\in K$ is an element which is not
  in the neutral component, then
  Serre~\cite[p.\,17,\,Lemme\,1]{serre-indep-ell} proves that there
  exists $f\in\Zz[a_1,\ldots,a_r]$ such that~$f$ vanishes on the
  connected component of~$g$, but is non-zero at the identity. Then
  the set
  \[
    \{g\in K\,\mid\, f(g)=0\}
  \]
  contains some connected component of~$K$, say~$p$ of them, but not
  all of them. Its measure is then $p/|\pi^0(K)|$, which is
  neither~$0$ nor~$1$.
\end{proof}

This last proposition suggests that the groups of connected components
of compatible objects should be isomorphic, as Serre proved is the
case for classical monodromy groups
(see~\cite[p.\,15]{serre-indep-ell}).  We hope to come back to this
problem soon.

Finally, under rather strong ``connectedness'' assumptions, we can get a
definitive answer by exploiting deep results of Larsen and
Pink~\cite{larsen-pink-90}.

\begin{proposition}\label{pr-lp}
  We continue with the notation and assumptions above. Denote by~$\Gg_j$
  the connected derived subgroup of~$\garith{M_j}$.  Assume that for all
  representations $\rho\colon \GL_r\to\GL(V)$, the multiplicity of the
  trivial representation in the restrictions of $\rho$ to~$\garith{M_j}$
  and to~$\Gg_j$ are equal.
  
  Then the complex semisimple Lie groups $\iota_1(\Gg_1)$
  and~$\iota_2(\Gg_2)$ are isomorphic.

  Moreover, if $M_1$ or $M_2$ is arithmetically simple, then the groups
  $\iota_1(\Gg_1)$ and $\iota_2(\Gg_2)$ are conjugate in~$\GL_r(\Cc)$.
\end{proposition}

\begin{proof}
  Let~$\rho\colon \GL_r\to \GL(V)$ be any finite-dimensional
  representation of~$\GL_r$ and let~$\mu_j(\rho)$ be the multiplicity of
  the trivial representation in the restriction of~$\rho$
  to~$\garith{M_j}$. By Proposition~\ref{pr-sato-tate}, the following equality holds:
  \[
  \mu_1(\rho)=\int_{\Un_r(\Cc)^{\sharp}}\Tr(\rho)d\nu_{cp,1}=
  \int_{\Un_r(\Cc)^{\sharp}}\Tr(\rho)d\nu_{cp,2}= \mu_2(\rho).
  \]

  By our assumption, the multiplicity $\mu_j(\rho)$ is also the
  multiplicity~$\nu_j(\rho)$ of the trivial representation in the
  restriction of~$\rho$ to~$\Gg_j$, and thus we have
  $\nu_1(\rho)=\nu_2(\rho)$.

  Since this equality holds for all representations~$\rho$, and~$\Gg_1$
  and~$\Gg_2$ are connected semisimple algebraic groups, a theorem of
  Larsen and Pink~\cite[Th.\,1]{larsen-pink-90} implies
  that~$\iota_1(\Gg_1)$ and~$\iota_2(\Gg_2)$ are isomorphic.

  Assume now that $M_1$ is arithmetically simple. Denoting by $\Ad$ the
  adjoint representation of~$\GL_r$, this is equivalent to
  $\mu_1(\Ad)=1$ by Schur's Lemma, and hence we also have $\nu_1(\Ad)=1$ and
  $\nu_2(\Ad)=1$ by the previous results. The result then follows from
  another theorem of Larsen and Pink~\cite[Th.\,2]{larsen-pink-90}.
\end{proof}

\begin{remark}\label{rm-rank}
  (1) Proposition~\ref{pr-lp} applies for instance if one knows that,
  for $j=1$ and $j=2$, the groups~$\garith{M_j}$ are connected and
  semisimple. However, it does not apply in a situation where, say
  $\garith{M_1}=\SL_r$ and~$\garith{M_2}=\GL_r$, since the determinant
  is an example of a representation where the multiplicities
  for~$\garith{M_2}$ and for its derived connected subgroup are not
  the same. (We thank one referee for pointing out this issue in our
  previous version.)

  (2) Larsen and Pink~\cite[Th.\,3\,and\,\S\,3]{larsen-pink-90} give
  examples showing that, in general, the assumption that~$M_1$ (or
  $M_2$) is simple cannot be omitted in the second part of the
  proposition.
\end{remark}




\chapter{Diophantine group theory}\label{sec-larsen}

In order to determine the tannakian (or monodromy) group associated to a
perverse sheaf, Katz has developed essentially two different sets of
methods. The first one (see, e.g.,~\cite{gkm,esde}) relies on
\emph{local monodromy information}, and applies mostly to the additive
group, although there is also a weaker analogue for the multiplicative
group (see~\cite[Ch.\,16]{mellin} and Corollary
\ref{cor-simpler}). However, we are not currently aware of any similar
tools for other groups. The second method, expounded in~\cite{katz:MMP},
is much more global, and exploits the diophantine potential of the
equidistribution of exponential sums to reveal properties of the
underlying group. It turns out that this global method adapts very well
to the tannakian framework, and this will be our fundamental tool.

We denote as usual by~$k$ a finite field, with an algebraic
closure~$\bar{k}$, and finite extensions $k_n$ of degree~$n$
in~$\bar{k}$. We fix a prime $\ell$ different from the characteristic
of~$k$.

\section{The diophantine irreducibility criterion}
\index{diophantine irreducibility criterion}
We first state Katz's criterion for a perverse sheaf to be geometrically
simple in terms of its trace functions.

\begin{proposition}\label{pr-schur}
  Let~$X$ be a quasi-projective algebraic variety over~$k$, and~$M$ an
  $\ell$-adic perverse sheaf on~$X$ which is pure of weight zero. Then
  the equality
  \begin{equation}\label{eq-schur-rh}
    \lim_{n\to +\infty} \sum_{x\in X(k_n)}
    |t_M(x;k_n)|^2=1
  \end{equation}
  holds if and only if $M$ is geometrically simple.
\end{proposition}

See~\cite[Th.\,1.7.2\,(3)]{katz:MMP} for the proof.

\begin{remark}
  This can be seen as a version of Schur's Lemma\index{Schur's Lemma}
  (compare with Corollary~\ref{cor-schur}): intuitively, by
  equidistribution, the limit in the proposition should converge to the
  multiplicity of the trivial representation in the representation
  $\End(\Std)$, where $\Std$ is the standard representation of the
  (usual) geometric monodromy group of the lisse sheaf on an open dense
  subset of the support of~$M$ that is associated to~$M$. The classical
  version of Schur's Lemma states that this multiplicity is equal to~$1$
  if and only if the standard representation is irreducible.
\end{remark}

\section{The Frobenius--Schur indicator}
\index{Frobenius--Schur indicator}

Recall that if~$\Gg$ is an arbitrary group and
$\rho\colon \Gg\to\GL_r(\Cc)$ is a finite-dimensional representation,
one says that $\rho$ is of \emph{orthogonal type} (resp. of
\emph{symplectic type}) if there exists a $G$-invariant non-degenerate
symmetric (resp. alternating) bilinear form on~$\Cc^r$.
\index{representation!of orthogonal type} \index{representation!of
  symplectic type} Suppose that $\rho$ is irreducible. The
Frobenius--Schur indicator $\FS(\rho)$ is defined to be~$1$ if~$\rho$ is
of orthogonal type, $-1$ if~$\rho$ is of symplectic type, and~$0$
otherwise.  If~$\Gg=K$ is a compact group, with probability Haar
measure~$\mu_K$, and if~$\rho$ is irreducible and continuous, then one has
an integral formula
$$
\FS(\rho)= \int_{K}\Tr(\rho(g^2))d\mu_K(g)
$$
(see, e.g.,~\cite[Th.\,6.2.3]{representations}).
\nomenclature{$\FS(\rho)$}{Frobenius--Schur indicator of a
  representation}

As in previous works of Katz (see, e.g.,~\cite[Th.\,9.1]{mellin}
or~\cite[Th.\,1.9.6]{katz:MMP}), there is a diophantine interpretation
of the Frobenius--Schur indicator.

\begin{proposition}\label{pr-fs}
  Let $G$ be a connected commutative algebraic group over~$k$, and let
  $M$ be an arithmetically irreducible $\ell$-adic perverse sheaf on~$G$
  which is pure of weight zero.  Let $\mcX=\wunram{M}$ be the set of
  weakly unramified characters for~$M$.
  
  The Frobenius--Schur indicator of $M$, viewed as a representation of the
  arithmetic tannakian group~$\garith{M}$, is given by the formula
  \[
    \FS(M)=\lim_{N \to +\infty} \frac{1}{N} \sum_{\substack{1\leq n\leq
        N\\\mathscr{X}(k_n)\not=\emptyset}}
    \frac{1}{|\mathscr{X}(k_n)|}\sum_{\chi \in \mathscr{X}(k_n)}
    \Tr(\Theta_{M,k_n}(\chi)^2).
    \label{eq-fs}
  \]
\end{proposition}

The proof is straightforward using the integral formula above and the
equidistribution theorem (Theorem~\ref{th-4}).


\section{Larsen's Alternative}

In this section, $r\geq 1$ is an integer and~$\Gg$ is a reductive
algebraic subgroup of $\GL_r$ over an algebraically closed field of
characteristic zero (recall that reductive groups are not required to be
connected). For each integer $m \geq 1$, the \emph{absolute $2m$-th
  moment}\index{moments of a representation}
\nomenclature{$M_{2m}(\bfG,V)$}{$2m$-th absolute moment of the
  representation~$V$ of~$\bfG$}
of an algebraic representation
$V$ of~$\Gg$ is defined as 
\[
  M_{2m}(\bfG, V)=\dim (V^{\otimes m} \otimes (V^\vee)^{\otimes
    m})^{\bfG}.
\]
When~$V$ is the ``standard'' $r$-dimensional representation given by the inclusion $\Gg \subset \GL_r$ (also
denoted by~$\Std$), we will simply write $M_{2m}(\bfG)$.
\nomenclature{$M_{2m}(\bfG)$}{$M_{2m}(\bfG,\Std)$}

If the base field is~$\Cc$, so that $\bfG$ is a reductive subgroup of $\GL_r(\Cc)$, the moments can be written as integrals over a maximal compact subgroup $K$ of $\Gg$ with Haar probability measure~$\mu_K$. Namely, for all $m \geq 1$, they are given by the integral expression
\begin{equation}\label{eq-int-moments}
  M_{2m}(\bfG)=\int_{K}|\Tr(g)|^{2m}d\mu_K(g). 
\end{equation}

We first note some elementary properties of the moments.
\begin{enumerate}
\item Given a surjective homomorphism $f\colon \Hh\to\Gg$ and a representation
  $\rho\colon \Gg\to \GL(V)$, the equality
  \[
  M_{2m}(\Hh,\rho\circ f)=M_{2m}(\bfG,\rho)
  \]
  holds for all $m \geq 1$ (since
  $(\rho^{\otimes m} \otimes (\rho^\vee)^{\otimes m})^{\bfG}=((\rho\circ
  f)^{\otimes m} \otimes (\rho\circ f)^\vee)^{\otimes m})^{\Hh}$ by
  definition).
\item For groups $\Gg_1$ and~$\Gg_2$ with representations $V_1$ and~$V_2$,
  the equality 
  \begin{equation}\label{eq-larsen-product}
    M_{2m}(\Gg_1\times\Gg_2, V_1\boxtimes V_2)= M_{2m}(\Gg_1,
    V_1)M_{2m}(\Gg_2,V_2)
  \end{equation}
  holds for all~$m \geq 1$ (this might be easiest to see using the
  integral expression~(\ref{eq-int-moments})). 
\item If $\Gg\subset \GL(V)$, and $Z$ is a subgroup of scalar matrices in
  $\GL(V)$, then the equality
  \[
  M_{2m}(\Gg,V)=M_{2m}(Z\Gg,V)
  \]
  holds (because $Z$ acts trivially on the whole space
  $V^{\otimes m} \otimes (V^\vee)^{\otimes m}$).
\item If there exists a $\bfG$-invariant decomposition
  \[
  V^{\otimes m}=\bigoplus_i n_i V_i,
  \]
  then the $2m$-th moment satisfies the inequality 
  \begin{equation}\label{eq-larsen-decomp}
    M_{2m}(\bfG,V)\geq \sum_i n_i^2,
  \end{equation}
  with equality if and only if the $V_i$ are pairwise
  non-isomorphic irreducible representations
  (see~\cite[1.1.4]{katz-larsen}).
  \item If there exists a $\bfG$-invariant decomposition 
  \[
  \End(V)=\bigoplus_i m_i W_i,
  \] then the fourth moment satisfies 
  \begin{equation}\label{eq-larsen-decomp-endomorphisms}
    M_4(\bfG,V)\geq \sum_i m_i^2, 
  \end{equation} with equality if and only if the $W_i$ are pairwise
  non-isomorphic irreducible representations
  (see~\cite[1.1.5]{katz-larsen}).
\end{enumerate}

Since the tensor constructions involved in the definition of the moments
are representations of the ambient group~$\GL(V)$, Theorem~\ref{th-4} immediately yields
a diophantine interpretation of the moments of the arithmetic tannakian
group of a perverse sheaf. 

\begin{proposition}\label{pr-larsen-diophante}
  Let $G$ be a connected commutative algebraic group over~$k$, and
  let $M$ be an arithmetically semisimple $\ell$-adic perverse sheaf
  on~$G$ which is pure of weight zero. 
  For each character $\chi\in\charg{G}(k_n)$, consider the sum
  $$
  S(M,\chi)=\sum_{x\in G(k_n)} t_M(x;k_n)\chi(x).
  $$
  Let $\mcX=\wunram{M}$ be the set of weakly unramified characters
  for~$M$ and let $m\geq 0$ be an integer. 
  
  The absolute moments of $M$, viewed as a representation of the arithmetic tannakian group~$\garith{M}$, satisfy the following: 
  \begin{align}
    M_{2m}(\garith{M}, M)&=\lim_{N \to +\infty} \frac{1}{N}
    \sum_{\substack{1\leq
        n\leq N\\\mathscr{X}(k_n)\not=\emptyset}}
    \frac{1}{|\mathscr{X}(k_n)|}\sum_{\chi \in \mathscr{X}(k_n)} |S(M,
    \chi)|^{2m}, \label{eq-larsen} \\
    M_{2m}(\garith{M}, M)&\leq \liminf_{N \to +\infty} \frac{1}{N}
    \sum_{1\leq
      n\leq N}
    \frac{1}{|G(k_n)|}\sum_{\chi \in \charg{G}(k_n)} |S(M,
    \chi)|^{2m}.
    \label{eq-larsen-2}
  \end{align}
  \par
  Moreover, \emph{if} the limit
  \begin{equation}\label{eq-limit-larsen-geo}
    \lim_{n\to+\infty}\frac{1}{|\mcX(k_n)|}\sum_{\chi\in\mcX(k_n)}
    |S(M,\chi)|^{2m}
  \end{equation}
  exists, then it is equal to the $2m$-th moment $M_{2m}(\ggeo{M},M)$ of
  $M$, viewed as a representation of the geometric tannakian group $\ggeo{M}$, and we have
  $$
  M_{2m}(\ggeo{M},M)=M_{2m}(\garith{M},M).
  $$
\end{proposition}

\begin{proof}
  We use the integral expression
  \[
    M_{2m}(\garith{M},M)=\int_{K}|\Tr(g)|^{2m}d\mu_K(g),
  \]
  where $K \subset \garith{M}(\Cc)$ is a maximal compact subgroup with
  Haar probability measure~$\mu_K$. Recall that to each weakly
  unramified character~$\chi\in\mcX(k_n)$ is associated the unitary
  conjugacy class~$\Theta_{M,k_n}(\chi)$ such that the equality
  $S(M,\chi)=\Tr(\Theta_{M,k_n}(\chi))$ holds, and that these conjugacy
  classes become equidistributed on average as $n \to +\infty$ by
  Theorem~\ref{th-4}. The first formula~(\ref{eq-larsen}) follows from
  this result applied to the test function~$g\mapsto |\Tr(g)|^{2m}$.
  \par
  Moreover, the inequality
  $$
  \frac{1}{N} \sum_{1\leq n\leq N} \frac{1}{|G(k_n)|}\sum_{\chi \in
    \mcX(k_n)} |S(M, \chi)|^{2m} \leq \frac{1}{N} \sum_{1\leq
    n\leq N} \frac{1}{|G(k_n)|}\sum_{\chi \in \charg{G}(k_n)} |S(M, \chi)|^{2m}
  $$
  holds by positivity of $|S(M, \chi)|^{2m}$. Taking the equivalence
  $|G(k_n)|\sim |\mcX(k_n)|$ as $n \to +\infty$ from the generic
  vanishing theorem into account, we deduce the second
  formula~(\ref{eq-larsen-2}).
  \par
  Finally, the last assertion follows from
  Proposition~\ref{pr-conv-weyl-geo}, applied to the representation
  $$
  \rho=\Std^{\otimes m}\otimes(\Std^{\vee})^{\otimes m},
  $$
  and from the fact that if the limit~(\ref{eq-limit-larsen-geo})
  exists, then its value is the same as the limit
  in~(\ref{eq-larsen}).
\end{proof}

We can combine this computation with Larsen's Alternative, a remarkable
criterion that ensures that a reductive subgroup $\bfG\subset \GL_r$
is either finite or contains one of the standard classical
groups,\index{classical group}
provided it has the correct fourth or eighth moment.

\begin{theorem}[Larsen's Alternative]\label{th-larsen}
  \index{Larsen's alternative}
  Let $V$ be a vector space of dimension $r \geq 2$ over an
  algebraically closed field of characteristic zero and let
  \hbox{$\bfG \subset \GL(V)$} be a reductive algebraic
  subgroup. Let~$Z$ denote the center of $\GL(V)$ and $\bfG^\circ$ the
  connected component of the identity of $\bfG$.
  \begin{enumth}
  \item\label{item1:th-larsen} The fourth moment satisfies
    $M_4(\bfG,V)\geq 2$. Furthermore, if $V$ is self-dual and $r\geq 3$,
    then $M_4(\bfG,V)\geq 3$.
  \item If $M_4(\bfG, V)\leq 5$, then the representation of~$\bfG$
    on~$V$ is irreducible.
  \item If $M_4(\bfG, V)=2$, then either $\SL(V) \subset \bfG$ or
    $\bfG \slash (\bfG \cap\ Z)$ is finite. If $\bfG \cap\ Z$ is finite,
    for instance if $\bfG$ is semisimple, then either
    $\bfG^\circ=\SL(V)$ or $\bfG$ is finite.
  \item\label{th-larsen:item2} Assume $r \geq 5$. If $M_4(\bfG,V)=2$ and
    $M_8(\bfG, V)=24$, then $\SL(V) \subset \bfG$.
  \item Assume that there exists a non-degenerate symmetric bilinear
    form~$B$ on~$V$ such that $\bfG$ lies in~$\Ort(B)$.
    \nomenclature[$O$]{$\Ort(B)$}{orthogonal group of~$B$}
    \nomenclature[$SO$]{$\SO(B)$}{special orthogonal group of~$B$}
    If
    $M_4(\bfG, V)=3$, then either $\bfG=\SO(B)$, or $\bfG=\Ort(B)$, or
    $\bfG$ is finite. If $r$ is $2$ or $4$, then $\bfG$ is not contained in $\SO(B)$.
  \item Assume that there exists a non-degenerate alternating bilinear
    form~$B$ on~$V$ such that $\bfG$ lies in~$\Sp(B)$. If $r \geq 4$ and
    $M_4(\bfG, V)=3$, then either $\bfG=\Sp(B)$ or $\bfG$ is finite.
    \nomenclature[$Sp$]{$\Sp(B)$}{symplectic group of~$B$}
  \end{enumth}
\end{theorem}

\begin{proof}
  The first statement concerning the fourth moment is a straightforward consequence of the inequality~(\ref{eq-larsen-decomp}). Indeed, since
  $V^{\otimes 2}\otimes (V^{\vee})^{\otimes 2}$ always contains a
  trivial one-dimensional subrepresentation, the fourth moment can only
  be~$1$ for $V$ of dimension~$1$. Moreover, there is 
  a $\GL(V)$-invariant (and hence $\bfG$-invariant) decomposition
  $$
  V^{\otimes 2}=\syms V \oplus \bigwedge\nolimits^2 V,
  $$
  where the factors are distinct and non-trivial, and of
  dimension~$\geq 2$ if~$r\geq 3$. If $V$ is self-dual, one of the two
  summands contains a proper one-dimensional $\bfG$-invariant subspace,
  so that the fourth moment is at least~$3$ using~(\ref{eq-larsen-decomp})
  again. 
  \par
  The other statements concerning the fourth moment are proved by Katz
  in~\cite[Th.\,1.1.6]{katz-larsen}. The statement about the eighth
  moment was conjectured by Katz in~\cite[2.3]{katz:MMP}, and proved by
  Guralnick and Tiep in~\cite[Th.\,1.4]{GurTiep}. Indeed, according to
  \loccit, a reductive subgroup $\bfG$ of $\GL(V)$ either satisfies
  $M_8(\bfG)>M_8(\GL(V))$ or contains the commutator subgroup
  $[\GL(V), \GL(V)]=\SL(V)$, and the eighth moment of $\GL(V)$ is equal
  to $24$ for $r\geq 4$, for instance in view of the $\GL(V)$-invariant
  decomposition
  \[
    V^{\otimes 4}=\mathrm{Sym}^4\,V \oplus \bigwedge\nolimits^4 V \oplus
    3 S^{(3, 1)}V \oplus 2 S^{(2, 2)}V \oplus 3 S^{(2, 1, 1)}V
  \]
  into pairwise non-isomorphic irreducible representations (see \eg
  \cite[Ex.\,6.5]{fulton-harris}), where $S^{\lambda}$ denotes the Schur
  functor associated to a partition $\lambda$ of $4$.
\end{proof}

In practice, computing a given moment of the arithmetic tannakian group $\garith{M}$ by means of the limit \eqref{eq-larsen} is feasible if there are sufficiently
many independent variables of summation, corresponding to the characters
of~$G$, in comparison with the number of variables involved in the
object~$M$, that is, the dimension of its support. It is
then possible, at least in some cases, to detect a diagonal behavior
that can lead to the asymptotic formula for the moment. This limitation explains why it is difficult to apply
Larsen's alternative when $G$ is one-dimensional, but starting from two-dimensional groups it can be sometimes implemented for objects supported on curves. 

\begin{remark}
  (1) Using typical terminology from geometric group theory, it is
  convenient to summarize the third part of Theorem~\ref{th-larsen} by
  saying that if $\bfG\subset \GL(V)$ has fourth moment equal to~$2$,
  then either $\bfG\supset \SL(V)$ or $\bfG$ is \emph{virtually
    central}\index{virtually central subgroup} in
  $\GL(V)$.
  \par
  (2) The book~\cite{katz:MMP} of Katz develops applications of Larsen's
  alternative which involve sums of the type
  $$
  S(f)=\sum_{x\in X(k)} t_1(x)t_2(f(x)) ,
  $$
  for suitable trace functions $t_1$ and~$t_2$ (on~$X$ and some affine
  space~$\Aa^r$, respectively), parameterized by elements
  $f\colon X\to \Aa^r$ of a ``function space'' $\mcF$. One of the
  conditions that are shown by Katz to ensure that the $2m$-th moment
  can be computed is that the evaluation maps
  $$
  f\mapsto (f(x_1),\ldots,f(x_{2m}))
  $$
  be surjective for distinct $x_i$ in $X(k)$ (see~\cite[\S 1.15,
  Th.\,1.20.2]{katz:MMP} for a precise and more general
  statement).
\end{remark}


\section{Sidon morphisms}

\begin{definition}[Sidon sets and Sidon morphisms]
  Let $A$ be an abelian group. A subset $S\subset A$ is called a
  \emph{Sidon set}\index{Sidon set} if all solutions $x_1$, $x_2$,
  $x_3$, $x_4$ in~$S$ of the equation $x_1x_2=x_3x_4$ satisfy 
  $x_1\in \{x_3,x_4\}$.
  \par
  More generally, let $r\geq 2$ be an integer. We say that $S$ is an
  \emph{$r$-Sidon set}\index{$r$-Sidon set} if all tuples $(x_i)_{1\leq i\leq r}$ and $(y_i)_{1\leq
    i\leq r}$ in $S^r$ such that the equality
  $$
  x_1\cdots x_r=y_1\cdots y_r
  $$
  holds satisfy $\{x_1, \dots, x_r\}=\{y_1, \dots, y_r\}$. A Sidon set
  is thus the same as a $2$-Sidon set.
  \par
  Let $\alpha\in A$. A subset $S \subset A$ is called an \emph{$\alpha$-symmetric
    Sidon set}\index{$\alpha$-symmetric Sidon set} if $S=\alpha S^{-1}$ and all solutions
  \hbox{$x_1, x_2, x_3, x_4 \in S$} of the equation $x_1x_2=x_3x_4$
  satisfy $x_1\in \{x_3,x_4\}$ or $x_2=\alpha x_1^{-1}$.
  \par
  Let $G$ be a connected commutative algebraic group over a field~$k$,
  and let $s\colon X\to G$ be a locally-closed immersion of
  $k$-schemes. We say that $s$ is a \emph{Sidon morphism},\index{Sidon morphism} or that
  $s(X)$ is a \emph{Sidon subvariety of~$G$}\index{Sidon subvariety} if, for any extension $k'$
  of~$k$, the subset $s(X)(k') \subset G(k')$ is a Sidon set. We define
  similarly $r$-Sidon morphisms for any $r\geq 2$.
  \par
  Let $i$ be an involution on~$X$ and $a\in G$. We say that $s$ is
  an \emph{$i$-symmetric Sidon morphism}\index{$i$-symmetric Sidon morphism} if the product morphism
  $(s\circ i)\cdot s\colon X\to G$ is a constant morphism, say
  equal to~$\alpha\in G(k)$, and
  if, for any extension $k'$ of~$k$, the set $s(X(k'))$ is an
  $\alpha$-symmetric Sidon set in $G(k')$.
\end{definition}



The interest of a Sidon morphism $X\to G$ is that it leads to
computations of the fourth moment for objects~$M$ on~$G$ that are pushed
from~$X$. We have two versions, depending on whether we have a Sidon
morphism or a symmetric Sidon morphism.

\begin{proposition}\label{pr-sidon-moments1}
  Let $G$ be a connected commutative algebraic group over a finite
  field~$k$ and let $s\colon X\to G$ be a closed immersion of $k$-schemes. Let $N$ be a geometrically simple $\ell$-adic perverse sheaf on~$X$
  which is pure of weight~$0$, so that the object $M=s_*N=s_!N$
  on~$G$ is a geometrically simple perverse sheaf on~$G$ and is pure
  of weight~$0$. 
  \par
  \begin{enumth}
  \item If $s$ is a Sidon morphism, then the equality
    $$
    M_4(\garith{M}, M)=2
    $$
    holds unless $M$ has tannakian dimension~$\leq 1$.
  \item If $X$ is a curve and $s\colon X\to G$ is a $4$-Sidon
    morphism, then the equality 
    $$
    M_8(\ggeo{M}, M)=M_8(\garith{M}, M)=24
    $$
    holds unless $N$ is geometrically isomorphic to $s^*\mcL_{\chi}[1]$
    for some character~$\chi\in\charg{G}$.
    \end{enumth}
\end{proposition}


\begin{proof}
 Let $n\geq 1$ be an integer. The formula
  $$
  \frac{1}{|G(k_n)|}\sum_{\chi \in \charg{G}(k_n)} |S(M, \chi)|^{4}
  =\sum_{\substack{y_1,\ldots,y_4\in X(k_n)\\
      s(y_1)s(y_2)=s(y_3)s(y_4)}}
  t_N(y_1,k_n)  t_N(y_2;k_n)\overline{t_N(y_3;k_n)t_N(y_4;k_n)}
  $$
  holds by orthogonality of characters. If~$s$ is a Sidon morphism, then we obtain by definition
  $$
  \frac{1}{|G(k_n)|}\sum_{\chi \in \charg{G}(k_n)} |S(M, \chi)|^{4}
  =2\Bigl(\sum_{y\in X(k_n)} |t_N(y,k_n)|^2\Bigr)^2-
  \sum_{y\in X(k_n)}|t_N(y,k_n)|^4,
  $$
  where the second term accounts for the double-counting of the
  solutions of the equation $s(y_1)s(y_2)=s(y_3)s(y_4)$ where
  $y_1=y_2=y_3=y_4$. In particular, we deduce that the inequality
  $$
  \frac{1}{|G(k_n)|}\sum_{\chi \in \charg{G}(k_n)} |S(M, \chi)|^{4}
  \leq 2\Bigl(\sum_{y\in X(k_n)} |t_N(y,k_n)|^2\Bigr)^2
  $$
  holds for all~$n\geq 1$. Since $N$ is geometrically simple, the
  right-hand side of this expression converges to~$2$ as
  $n\to +\infty$ by Proposition~\ref{pr-schur}.  Using the
  inequality~(\ref{eq-larsen-2}) from
  Proposition~\ref{pr-larsen-diophante}, we deduce that
  $$
  M_{4}(\garith{M}, M)\leq 2
  $$
  in the setting of~(1). Hence, the fourth moment is either $\leq 1$
  or equal to~$2$. By Theorem~\ref{th-larsen}\,\ref{item1:th-larsen},
  the former is only possible if $M_{\intt}$ is of tannakian
  dimension~$\leq 1$.
  \par
  Now we assume that~$s$ is a $4$-Sidon morphism. We obtain similarly
  $$
  \frac{1}{|G(k_n)|}\sum_{\chi \in \charg{G}(k_n)} |S(M, \chi)|^{8}\leq
  24\Bigl(\sum_{y\in X(k_n)} |t_N(y,k_n)|^2\Bigr)^4
  $$
  where the right-hand side converges to~$24$ for the same reason as
  before.
  \par
  Assume now that~$X$ is a curve and $s$ is a $4$-Sidon morphism.  We
  apply the Riemann Hypothesis (Theorem~\ref{th-rh}) to the simple perverse
  sheaves $s^*\mcL_{\chi^{-1}}[1](1/2)$ (of weight~$0$) and to~$N$. By assumption, these
  are not geometrically isomorphic, and therefore the estimate
  $$
  S(M,\chi)=\sum_{y\in X(k_n)} \chi(s(y))t_N(y;k_n)\ll 1
  $$
  holds for all characters~$\chi$. We deduce then that the formula
  $$
  \lim_{n\to+\infty} \frac{1}{|G(k_n)|}\sum_{\chi \notin
    \wunram{M}(k_n)} |S(M, \chi)|^{2m} =0
  $$
  holds for any integer $m\geq 1$; we finally conclude from the previous
  computations and the last assertion of
  Proposition~\ref{pr-larsen-diophante} that
  $M_8(\ggeo{M},M)=M_8(\garith{M},M)=24$.
\end{proof}

We now state the version involving symmetric Sidon morphisms.

\begin{proposition}\label{pr-sidon-moments2}
  Let $G$ be a connected commutative algebraic group over a finite
  field~$k$. Let~$X$ be a smooth irreducible algebraic variety over~$k$
  and $i$ an involution on~$X$. Let $s\colon X\to G$ be an $i$-symmetric
  Sidon morphism which is a closed immersion.  Let $\alpha$ be the
  constant value of the morphism $(s\circ i)\cdot s$.
  \par
  Let $N$ be a geometrically simple $\ell$-adic perverse sheaf on~$X$
  which is pure of weight~$0$, so that the object $s_*N=s_!N$
  on~$G$ is a geometrically simple perverse sheaf on~$G$ and is pure of
  weight~$0$.
  \par
  \begin{enumth}
  \item If $i^*N$ is isomorphic to $\dual(N)$, then we have
    $(s_*N)^{\vee}=[\times \alpha^{-1}]^*(s_*N)$, and
    $$
    M_4(\garith{s_*N},s_*N)=3,
    $$
    unless $(s_*N)_{\intt}$ has tannakian dimension~$\leq 2$.
  \item If
    $i^*N$ is not isomorphic to $\dual(N)$, then
    $$
    M_4(\garith{s_*N},s_*N)=2,
    $$
    unless $(s_*N)_{\intt}$ has tannakian dimension~$\leq 2$.
  \end{enumth}
\end{proposition}

\begin{proof}
  Let $M=s_*N$. In the situation of~(1), the definition of~$\alpha$
  means that there is an equality
  $ s\circ i=[\times \alpha]\circ (\inv\circ s)$. Therefore, we obtain
  canonical isomorphisms
  \begin{multline}
    M^{\vee}=\inv^*(\dual(s_*N))=\inv^*(s_*(\dual(N)))
    = \inv^*((s\circ i)_* N)\\
    =\inv^*([\times \alpha]_*(\inv\circ s)_* N)=(\inv \circ [\times
    \alpha]\circ \inv)_*(s_*N)=[\times \alpha^{-1}]^*M.
  \end{multline}
  \par
  We go back to the general case. Arguing as in the proof of the
  previous proposition, we obtain the inequality
  \begin{multline*}
    \frac{1}{|G(k_n)|}\sum_{\chi \in \charg{G}(k_n)} |S(M, \chi)|^{4}
    \leq 2\Bigl(\sum_{y\in X(k_n)} |t_N(y,k_n)|^2\Bigr)^2
    \\
    + \sum_{(y,z)\in X(k_n)^2}
    t_N(y,k_n)t_N(i(y);k_n)\overline{t_N(z;k_n)t_N(i(z);k_n)}
  \end{multline*}
  for all $n\geq 1$, by the definition of symmetric Sidon sets.  The
  second sum is equal to the quantity
  $$
  \Bigl|\sum_{y\in X(k_n)} t_N(y,k_n)t_N(i(y);k_n)\Bigr|^2,
  $$
  which converges to~$1$ under the assumption~(1)
  (using~(\ref{eq-gabber-conjugate})), by Proposition~\ref{pr-schur},
  and to~$0$ under the assumption~(2), by the Riemann Hypothesis. Thus
  we deduce from Proposition~\ref{pr-larsen-diophante} that
  $$
  M_4(\garith{M},M)\leq 3,\quad\quad \text{resp.}\quad\quad
  M_4(\garith{M},M)\leq 2,
  $$
  in case~(1) (resp.~(2)), and we conclude as before from
  Theorem~\ref{th-larsen},~(1).
\end{proof}

\begin{remark}\label{rm-careful}
  The caveats concerning the tannakian dimension of $s_*N$ in these
  statements are necessary.  We will indeed see concrete examples (see
  Example~\ref{ex-fourth-moment-2}\,(1) and
  Remark~\ref{rm-jacobian-careful}\,(1)) where the fourth moment does
  \emph{not} coincide with the limit
  $$
  \lim_{n\to+\infty} \frac{1}{|G(k_n)|}\sum_{\chi \in \charg{G}(k_n)}
  |S(M, \chi)|^{4} 
  $$
  (althouth the latter exists) because of the contribution of some
  special ramified characters.
\end{remark}

The result of Propositions~\ref{pr-sidon-moments1}
and~\ref{pr-sidon-moments2} will be the basis of applications in
Chapters~\ref{sec-product}, \ref{sec-variance} and~\ref{sec-jacobian}.
Here are the relevant cases of Sidon morphisms, together with some
further elementary examples.


\begin{proposition}\label{pr-sidon-morphisms}
  Let~$k$ be a field, not necessarily finite.
  \par
  \begin{enumth}
  \item For any $\alpha\in k^{\times}$, the embedding
    $x\mapsto (x,\alpha x)$ of $\Gg_m$ in $\Gg_m\times \Gg_a$ is a Sidon
    morphism.
  \item Let $C$ be a smooth projective connected algebraic curve of
    genus $g\geq 2$ over~$k$. Let $D$ be a divisor of degree~$1$ on~$C$,
    and let~$A=\mathrm{Jac}(C)$ be the jacobian\index{jacobian of a curve} of~$C$. The closed
    immersion $s\colon x\mapsto x-D$ of~$C$ in~$A$ is a Sidon
    morphism unless~$C$ is hyperelliptic, in which case it is an
    $i$-symmetric Sidon morphism, where $i$ is the
    hyperelliptic\index{hyperelliptic curve}
    involution.
  \item With notation as in the previous item, if the gonality\index{gonality} of~$C$ is
    at least~$5$, then $s$ is a $4$-Sidon morphism.
  \item Let $d\geq 1$ be an integer and let $f$ be a separable
    polynomial of degree~$d$ over~$k$. Let~$Z$ be the set of zeros
    of~$f$. The closed immersion $x\mapsto (z-x)_{z\in Z}$ of
    $\Aa^1[1/f]$ in $\Gg_m^Z$ is a Sidon morphism if $d\geq 2$. It is a
    $4$-Sidon morphism if $d\geq 4$.
  \item Suppose that the characteristic of~$k$ is not~$3$. The graph
    $s\colon x\mapsto (x,x^3)$ from $\Gg_a$ to~$\Gg_a^2$ is an
    $i$-symmetric Sidon morphism, where $i$ is the involution $x\mapsto
    -x$.
  \item The morphism $x\mapsto (x,1-x)$ from $\Gg_m\setminus\{1\}$ to
    $\Gg_m\times\Gg_m$ is a Sidon morphism.
  \end{enumth}
\end{proposition}

\begin{proof}
  (1) For $x_1$, \ldots, $x_4$ in~$\Gg_m$, the equation
  $$
  (x_1,\alpha x_1)\cdot (x_2,\alpha x_2)=(x_3, \alpha x_3)\cdot
  (x_4,\alpha x_4)
  $$
  in $\Gg_a\times \Gg_m$ means that $x_1+x_2=x_3+x_4$ and
  $x_1x_2=x_3x_4$, which implies that $\{x_1,x_2\}=\{x_3,x_4\}$, both
  sets being the solutions of the same quadratic equation.
  \par
  (2) Let $x_1$, \ldots, $x_4$ in~$C$ be solutions of
  $$
  s(x_1)+s(x_2)=s(x_3)+s(x_4).
  $$
  \par
  Assume $x_1\notin\{x_3,x_4\}$. Then the equation implies the existence
  of a rational function on $C$ with zeros $\{x_1, x_2\}$ and poles
  $\{x_3, x_4\}$, which corresponds to a morphism $f \colon C \to \Pp^1$
  of degree at most~$2$. This is not possible unless $C$ is
  hyperelliptic.
  \par
  With the same notation, if~$C$ is hyperelliptic with hyperelliptic
  involution~$i$, then the uniqueness of the morphism
  $f \colon C \to \Pp^1$ of degree~$2$ up to automorphisms (see,
  e.g.,~\cite[Rem.\,4.30]{Liu}) shows that~$i$ exchanges the points of
  the fibers of~$f$, or in other words, that the equalities $x_2=i(x_1)$
  and $x_4=i(x_3)$ hold.
  \par
  (3) The argument is similar: the equation
  $$
  s(x_1)+s(x_2)+s(x_3)+s(x_4)=
  s(x_5)+s(x_6)+s(x_7)+s(x_8)
  $$
  where $\{x_i\}\not=\{y_i\}$ implies the existence of a non-constant
  morphism $f\colon C\to\Pp^1$ of degree at most~$4$, and hence implies that
  $C$ has gonality at most~$4$.
  \par
  (4) Suppose that $x_1$, \ldots $x_4$ satisfy
  $s(x_1)s(x_2)=s(x_3)s(x_4)$. Then we get
  $$
  (x_1-z)(x_2-z)=(x_3-z)(x_4-z)
  $$
  for all $z\in Z$, i.e., the monic polynomials $(x_1-X)(x_2-X)$ and
  $(x_3-X)(x_4-X)$ take the same values at the points of~$Z$. By
  interpolation, they are equal if $|Z|=d\geq 2$. The case of the
  $4$-Sidon property is analogous with polynomials of degree~$4$.
  \par
  (5) Suppose that $x_1$, \ldots, $x_4\in\Gg_a^4$ satisfy
  $$
  \begin{cases}
    x_1+x_2=x_3+x_4\\
    x_1^3+x_2^3=x_3^3+x_4^3.
  \end{cases}
  $$
  If $x_2\not=-x_1$, then these imply that
  $$
  (x_1+x_2)^2-3x_1x_2=(x_3+x_4)^2-3x_3x_4,
  $$
  and therefore $x_1x_2=x_3x_4$ when the characteristic is not~$3$. Now
  we conclude as in~(1).
  \par
  (6) This is again about quadratic equations: let $x_1$, \ldots, $x_4$
  in~$\Gg_m\setminus \{1\}$ be such that
  $$
  \begin{cases} x_1x_2=x_3
    \\
    (1-x_1)(1-x_2)=(1-x_3)(1-x_4).
  \end{cases}
  $$
  Then we get further $x_1+x_2=x_3+x_4$, and conclude as before.
\end{proof}

\begin{remark}
  (1) Example~(1) is classical: it is often attributed to
  Ruzsa~\cite{ruzsa}, but it was pointed out by Eberhard and
  Manners~\cite{eberhard-manners} that it occurs previously in a paper
  of Ganley~\cite[p.\,323]{ganley}, where it is attributed to Spence.
  \par
  Example~(5) was also indicated to us by Eberhard and Manners.
  \par
  (2) There is much work in combinatorics in trying to find the largest
  possible Sidon sets in a finite abelian group~$A$ (for instance, see
  the classification in~\cite{eberhard-manners} of known examples of
  size $\sim |A|^{1/2}$, which they show are all related to finite
  projective planes). A natural analogue geometric question is to
  classify the Sidon morphisms $s\colon X\to G$ such that $\dim(X)$ is
  maximal. The best possible value for a given group $G$ is
  $\dim(X)=\lfloor \tfrac{\dim(G)}{2}\rfloor$.  When can this be
  achieved?

  Note that a subset~$S$ of an abelian group~$A$ is a Sidon set if and
  only if the induced map $S^2/\mathfrak{S}_2\to A$ defined by
  $(x,y)\mapsto x+y$ is injective (where~$\mathfrak{S}_2$ acts by
  permuting the two coordinates).  Consider the variant definition of a
  Sidon morphism $s\colon X\to G$ where we ask that~$s$ be a morphism
  such that the induced map $s^{(2)}\colon X^{(2)}\to G$ from the
  symmetric square of~$X$ to~$G$ is a closed immersion. Again we have
  $2\dim(X)\leq \dim(G)$, but we can see in this case that if~$G$ is an
  abelian variety, then equality is not possible. Indeed, this would
  imply that~$s^{(2)}$ is an isomorphism, which is impossible
  (if~$\dim(X)\geq 2$, because $X^{(2)}$ is then singular, and if~$X$ is
  a curve, because it would have to be smooth of genus~$2$, so~$G$ is an
  abelian surface, but for instance the second cohomology groups do not
  have the same dimension).
  
\end{remark}


The result concerning jacobians of smooth projective curves can be
generalized by considering either Rosenlicht's generalized
jacobians\index{generalized jacobian}
(which appear in geometric class field theory, see the book of
Serre~\cite{serre-alg-groups}), or the Picard group\index{Picard group} of certain singular
curves. The case of generalized jacobians is analyzed in complete
generality in our paper~\cite{ffk-sidon}. We state the result here
(see~\cite[Th.\,1]{ffk-sidon}).


\begin{proposition}\label{pr-sidon-gen-jac}
  Let $k$ be a (not necessarily finite) field and let~$C$ be a smooth
  projective geometrically connected curve of genus~$g$
  over~$k$. Let~$\mfm$ be an effective divisor on~$C$ and $J_{\mfm}$
  the associated generalized jacobian, which is a commutative
  algebraic group of dimension $g+\max(\deg(\mfm)-1,0)$.  Let~$\delta$
  be a divisor of degree~$1$ on~$C$ whose support does not intersect
  that of~$\mfm$.  Let~$s\colon C\setminus\mfm\to J_{\mfm}$ be the
  morphism induced by the map $x\mapsto (x)-\delta$ on divisors.
 
  If $\dim(J_{\mfm})\geq 2$, then~$s$ is either a Sidon morphism or a
  symmetric Sidon morphism.

  If, moreover, $(C\setminus\mfm)(k)$ is non-empty, then it is a
  symmetric Sidon set if and only if one of the following conditions
  hold:
  \begin{enumth}
  \item $g=1$ and~$\deg(\mfm)=2$; in this case, writing $\mfm=(p)+(q)$
    (where~$p$ and~$q$ are not necessarily $k$-points of $C$, but the
    divisor $\mfm$ is assumed to be defined over~$k$), the
    value~$\alpha$ of
    $s(x)+s(p+q-x)$ for $x\in (C\setminus \mfm)(k)$ is independent
    of~$x$ and~$s((C\setminus \mfm)(k))$ is an $\alpha$-symmetric
    Sidon set.
  \item $g\geq 2$, the curve~$C$ is hyperelliptic, and either
    $\deg(\mfm)\leq 1$ or $\mfm=(p)+(i(p))$ for some~$p\in C$,
    where~$i$ is the hyperelliptic involution on $C$. In both of these
    cases, the 
    value~$\alpha$
    of $s(x)+s(i(x))$ for $x\in (C\setminus \mfm)(k)$ is independent
    of~$x$ and~$s((C\setminus \mfm)(k))$ is an $\alpha$-symmetric
    Sidon set.
  \end{enumth}
\end{proposition}

\begin{remark}\label{rm-gener-jac}
  (1) Both the generalized jacobians $J_S$ and the Picard group scheme
  of an irreducible curve are connected commutative algebraic groups
  over~$k$ which may involve all types of groups (unipotent groups,
  abelian varieties and tori).

  More precisely, the following results hold:
  \begin{enumerate}
  \item Let~$C$ be a smooth projective curve of genus $g\geq 0$ over~$k$
    and~$\mfm$ an effective divisor on~$C$. Write $\mfm$ in the form
    $$
    \mfm=\sum_{x \in\Supp(S)}n_x\,(x)
    $$
    with $n_x\geq 1$.  The generalized jacobian $J_{\mfm}$ is an extension
    $$
    0\to L_{\mfm}\to J_{\mfm}\to \Jac(C)\to 0
    $$
    of the (usual) jacobian of~$C$, with kernel
    $L_{\mfm}=R_{\mfm}/\Gg_m$, where $R_{\mfm}$ is isomorphic to a
    product
    $$
    R_{\mfm}=\prod_{x\in\Supp(\mfm)}(\Gg_m\times V_x)
    $$
    with $V_x$ unipotent of dimension $n_x-1$, and with $\Gg_m$
    embedded diagonally in~$R_{\mfm}$ (see, e.g.,~\cite[p.\,2
    and\,V.13,\, V.14]{serre-alg-groups}).
    \par
    In particular, assuming that $g\geq 1$, the group $J_{\mfm}$ has
    non-trivial abelian, toric and unipotent parts as soon as the
    support of ${\mfm}$ contains two distinct points, one of which at
    least has coefficient $\geq 2$.
  \item Let~$C$ be an irreducible projective curve~$C$ over an
    algebraically closed field. Let $\widetilde{C}\to C$ be the
    normalization of~$C$, and for $x\in C(k)$, define $m_x$ to be the
    cardinality of the fiber of $\widetilde{C}\to C$ over~$x$. Then
    $\Pic^0(C)$ has dimension $\dim H^1(C,\mcO_C)$, and it is an
    extension
    $$
    0\to K_C\to \Pic^0(C)\to \Jac(\widetilde{C})\to 0
    $$
    of the jacobian of the normalization $\widetilde{C}$, with kernel
    $K_C$ which is an extension of a torus of dimension
    $$
    \sum_{x\in C\setminus U}(m_x-1)
    $$
    by a unipotent group, of dimension therefore equal to
    $$
    \dim H^1(C,\mcO_C)-g(\widetilde{C})-\sum_{x\in C\setminus
      U}(m_x-1)
    $$
    (see, e.g.,~\cite[Def.\,5.13,\,Th.\,7.5.19,\,
    Lemma\,5.18]{Liu}).
  \end{enumerate}

  Note furthermore that these two classes of algebraic groups are
  closely related (e.g., any generalized jacobian $J_{\mfm}$ is the
  Picard group of some singular curve).

  \par
  (2) All the examples of Sidon morphisms in
  Proposition~\ref{pr-sidon-morphisms} can be interpreted in terms of
  generalized jacobians. For instance, consider the curve $C=\Pp^1$
  over~$k$, and the effective divisor $S=(0)+2(\infty)$, so that
  $U=\Pp^1\setminus \{0,\infty\}=\Gg_m$. According to the above, the
  generalized jacobian $J_S$ is isomorphic to
  $G=(\Gg_m\times (\Gg_m\times\Gg_a))/\Gg_m^{\Delta}$, where the
  subgroup $\Gg_m^{\Delta}$ is embedded diagonally by
  $x\mapsto (x, (x,0))$. An isomorphism $\varphi\colon J_S\to G$ is
  given as follows: given a divisor $E$ of degree~$0$ on~$\Pp^1$,
  represent it as the divisor of a rational function
  $g\colon \Pp^1\to \Pp^1$, and let
  $$
  \varphi(E)=(g(0), (g(\infty),\frac{g'}{g}(\infty)))
  $$
  (this can be checked from the description in~\cite[p.\,2 and\,V.13,\,
  V.14]{serre-alg-groups}).  The morphism
  $G\to \Gg_m\times \Gg_a$ given by
  $(x,(y,a))\mapsto (xy^{-1},a)$ is an isomorphism, and using it to
  identify $G$ with $\Gg_m\times \Gg_a$, the formula above becomes
  $\varphi(E)=(\tfrac{g(1)}{g(\infty)},\tfrac{g'}{g}(\infty))$.

  Consider then the morphism $U=\Gg_m\to J_S$ defined using the divisor
  $D=(1)$. Then the morphism $s_D\colon x\mapsto (x)-(1)$ is given by
  $s_D(x)=(x,1-x)$ (take $g(t)=(t-x)/(t-1)$ to compute
  $\varphi((x)-(1))$). This is a Sidon morphism, the argument for this
  being identical with that of Proposition~\ref{pr-sidon-morphisms}\,(1).

  We refer again to~\cite[\S\,2]{ffk-sidon} for more discussion, in
  particular in comparison with the paper of Eberhard and Manners.
\end{remark}

\section{Gabber's torus trick}

We discuss here another criterion to have a large tannakian group that
also involves Sidon sets, but in a very different manner from their
appearance in the previous sections.  This criterion is difficult to
apply for an individual object, but it leads to simple specialization
results.

We use a version of Gabber's ``torus trick''
(see~\cite[Th.\,1.0]{esde}). The following statement is specialized to
the case of $\SL_r$ and written in the language of compact Lie groups.
\index{Gabber's torus trick}
\begin{theorem}[Gabber]\label{th-gabber}
  Let $V$ be a finite-dimensional complex vector space of dimension
  $r\geq 1$, and let $\bfG$ be a connected semisimple compact subgroup
  of $\GL(V)$ that acts irreducibly on~$V$. Let~$D$ be the subgroup
  consisting of the elements of~$\GL(V)$ that are diagonal with respect
  to some basis, and let $\chi_1, \dots, \chi_r$ be the characters
  $D\to \Cc^{\times}$ giving the coefficients of the elements of~$D$.
  \par
  Let $A\subset D$ be a subgroup of the normalizer of $\bfG$ in
  $\GL(V)$. Let $S\subset \widehat{A}$ be the subset of the group of
  characters of~$A$ given by the restrictions to~$A$ of the diagonal
  characters $\chi_i$. If $|S|=r$ and $S$ is a Sidon set in
  $\widehat{A}$, then $\bfG=\SU(V)$.
\end{theorem}

\begin{remark}
  Properly speaking, Gabber's original result implies here that $G$
  contains a maximal torus of~$\SU(V)$, and the fact that $G$ is
  semisimple and connected then implies that $G$ is $\SU(V)$ (see,
  e.g.,~\cite[p.\,36, prop.\,13]{bourbaki-lie-9}).
\end{remark}

We emphasize that the subgroup $A$ can be arbitrary: it may be finite,
and need not be closed.

We can then deduce the following criterion.

\begin{proposition}\label{pr-gabber}
  Let $G$ be a connected commutative algebraic group over the finite
  field~$k$.  Let $M$ be a simple perverse sheaf on~$G$ which is pure of
  weight~$0$ and of tannakian dimension $r\geq 1$. Assume that $M$ is
  generically unramified.
  \par
  The geometric tannakian group $\ggeo{M}$ contains $\SL_r$ if and only
  there exists an unramified character $\chi\in\charg{G}(k_n)$ for some
  integer $n\geq 1$ such that the eigenvalues of $\Thetaf_{M,k_n}(\chi)$
  are distinct and form a Sidon set in~$\Cc^{\times}$. 
\end{proposition}

\begin{proof}
  Suppose that $\ggeo{M}$ contains $\SL_r$. Let $U\subset \SU_r(\Cc)$ be
  the set of matrices whose eigenvalues are distinct and form a Sidon
  set in~$\Cc^{\times}$. This is an open set (for the Lie group
  topology), so that  equidistribution implies
  $$
  \liminf_{N\to+\infty}\frac{1}{N} \sum_{1\leq n\leq N}
  \frac{1}{|G(k_n)|} \sum_{\Thetaf_{M,k_n}(\chi)\in U}1 >0,
  $$
  and hence there exists $n\geq 1$ and $\chi\in\charg{G}(k_n)$ such that
  $\Thetaf_{M,k_n}(\chi)\in U$.
  \par
  Conversely, if an unramified character $\chi\in\charg{G}(k_n)$ exists
  with $\Thetaf_{M,k_n}(\chi)\in U$, then we can apply
  Theorem~\ref{th-gabber} to the group $A$ generated by a fixed element
  in the conjugacy class $\Thetaf_{M,k_n}(\chi)$, and to the neutral
  component of the geometric tannakian group of~$M$ (which is normalized
  by $\garith{M}$, since $\ggeo{M}$ is normal in $\garith{M}$ by
  Proposition~\ref{pr:geom-vs-arith1} and its neutral component is a
  characteristic subgroup).
\end{proof}

In general, we do not have robust methods to check the existence of a
character with the desired properties. However, we may combine this with
a specialization argument.

\begin{proposition}
  Let $G$ be a connected commutative algebraic group over the finite
  field~$k$.  Let $M$ be a simple perverse sheaf on~$G$ which is pure of
  weight~$0$ and of tannakian dimension $r\geq 1$. Assume that $M$ is
  generically unramified. Let $f\colon G\to H$ be a morphism of
  commutative algebraic groups over~$k$.
  \par
  Suppose that the object $N=Rf_!M$ is a geometrically simple perverse
  sheaf on~$H$ that is pure of weight~$0$, and suppose that
  $\chi\circ f$ is unramified for $M$ whenever $\chi$ is unramified
  for~$N$.
  \par
  If the geometric tannakian group $\ggeo{N}$ contains $\SL_r$, then
  $\ggeo{M}$ contains $\SL_r$.
\end{proposition}

\begin{proof}
  By Proposition~\ref{pr-gabber}, the assumption implies that there
  exists a character~$\chi\in\charg{H}(k_n)$ unramified for~$N$ for
  which $\Thetaf_{N,k_n}(\chi)$ has distinct eigenvalues forming a Sidon
  set. Since $\Thetaf_{M,k_n}(\chi\circ f)$ has the same characteristic
  polynomial, the character $\chi\circ f\in\charg{G}(k_n)$ has the same
  property; by Proposition~\ref{pr-gabber} again, it follows that
  $\ggeo{M}$ contains $\SL_r$.
\end{proof}


\section{Recognition criteria for $\mathbf{E}_6$}

We include here a criterion of Krämer to recognize the exceptional
group~$\mathbf{E}_6$ in one of its $27$-dimensional faithful
representations (we always mean by $\mathbf{E}_6$ the simply-connected
form).
\index{exceptional group~$\mathbf{E}_6$}

\begin{proposition}[Krämer]\label{pr-e6}
  Let $\bfG$ be a connected semisimple linear algebraic group
  over~$\Qlb$ or~$\Cc$ and~$\rho$ an irreducible faithful
  $27$-dimensional representation of~$\bfG$. If the $729$-dimensional
  representation $\End(\rho)$ of~$\bfG$ contains an irreducible
  $78$-dimensional subrepresentation, then $\bfG$ is isomorphic to the
  exceptional group $\mathbf{E}_6$ and~$\rho$ is one of its two
  fundamental $27$-dimensional representations.
\end{proposition}

See~\cite[Lemma\,4]{kramer-e6} for the proof. We will apply this in
Section~\ref{sec-cubic}, although somewhat differently than we use
Larsen's Alternative.  The following criterion is closer to the spirit
of the latter, and might have interesting applications (see again
Section~\ref{sec-cubic} for an attempt).

\begin{proposition}\label{pr-larsen-e6}
  Let $\bfG$ be a connected semisimple linear algebraic group
  over~$\Qlb$ or~$\Cc$, and let $\rho$ be a faithful representation of
  $\bfG$ of dimension~$27$. Then $\bfG$ is isomorphic to the exceptional
  group $\mathbf{E}_6$ and~$\rho$ to one of the two fundamental
  $27$-dimensional representations of~$\bfG$ if and only if
  $M_4(\bfG,\rho)=3$ and $\rho$ is not self-dual.
\end{proposition}

\begin{proof}
  Suppose first that $\bfG=\mathbf{E}_6$ and $\rho$ is one of its
  fundamental representations of dimension~$27$. These representations
  are not self-dual (see~\cite[Table\,1,\,p.\,213]{bourbaki-lie-6}).
  Using the Weyl dimension formula\index{Weyl dimension formula} (see~\cite[Th.\,2,\
  p.\,151]{bourbaki-lie-8} and~\cite[Pl.\,V,\,p.\,260]{bourbaki-lie-6}),
  we see that the dimensions of the irreducible representations
  of~$\mathbf{E}_6$ that may possibly occur in the $729$-dimensional
  representation on $\End(\rho)$ are $1$, $27$, $78$, $351$, $650$. We
  know that the trivial representation appears once in~$\End(\rho)$, and
  that the $78$-dimensional adjoint representation~$\Ad$ appears at
  least once. But the equation
  $$
  729-79=650=27a+78b+351c+650d
  $$
  has the unique non-negative integral solution $(a,b,c,d)=(0,0,0,1)$
  (looking modulo~$3$, it becomes $d\equiv 1\mods{3}$). So we must have
  an isomorphism
  $$
  \End(\rho)\simeq \mathbf{1}\oplus \Ad\oplus \rho_{650},
  $$
  where $\rho_{650}$ has dimension~$650$, and hence the fourth moment
  $M_4(\mathbf{E}_{6},\rho)$ is equal to~$3$. (This is also noted
  without proof by Katz~\cite[Rem.\,1.2.3]{katz-larsen}.)
  \par
  We now prove the converse, and assume that $M_4(\bfG,\rho)=3$ and
  $\rho$ is not self-dual. Since the fourth moment is $\leq 5$, the
  representation $\rho$ is irreducible (Theorem~\ref{th-larsen},
  (1)). Now let
  $$
  \bfG_1\times \cdots \times \bfG_k\to \bfG\fleche{\rho}\GL_{27}
  $$
  be the representation obtained from the decomposition of the algebraic
  universal covering of $\bfG$ in product of almost simple groups. This
  composition decomposes as an external tensor product
  $$
  \rho_1\boxtimes \cdots\boxtimes \rho_k
  $$
  of irreducible representations of~$\bfG_i$.  We then have
  $$
  3=M_4(\bfG,\rho)=\prod_{i=1}^kM_4(\bfG_i,\rho_i)
  $$
  by~(\ref{eq-larsen-product}).
  The condition $M_4(\bfG_i,\rho_i)=1$ is impossible (since it implies
  that $\dim(\rho_i)=1$, and hence $\rho_i$ would be trivial and this
  contradicts the faithfulness assumption), so we have a single factor
  $\bfG_1$.
  \par
  The representation $\rho_1$ is not self-dual, which implies that the
  root system of~$\bfG_1$ (and hence of $\bfG$) can only be of type $E_6$,
  or $A_l$ for $l\geq 2$ or $D_l$ with $l\geq 3$ odd (see,
  e.g.~\cite[p.\,132, prop.\,12]{bourbaki-lie-8}, combined with the fact
  that the longest element of the Weyl group\index{Weyl group} acts by $-\mathrm{Id}$ for
  the other simple root systems).
  \par
  The groups of type $A_l$ with $l\geq 2$ which have a $27$-dimensional
  irreducible representation are of type $A_2$ (the representation with
  highest weight $2\varpi_1+2\varpi_2$, in the standard notation of
  Bourbaki) or $A_{26}$ (the standard representation).
  In the first case, the representation is actually self-dual, and in
  the second case, the fourth moment is equal to~$2$, so these are
  excluded (in particular, groups of type $D_3=A_2$ are also excluded).
  \par
  Let $l\geq 5$ be an odd integer. The representations of groups of type
  $D_l$ which are not self-dual and have smallest possible dimension are
  the half-spin representations\index{half-spin representation} of dimension $2^{l-1}$
  (see~\cite[p.\,210]{bourbaki-lie-8}). Thus only $D_5$ could possibly
  give rise to a representation of dimension~$27$; but one can check
  that there is no representation of this dimension of a group of
  type~$D_5$ (e.g., because of the Weyl Dimension Formula,
  see~\cite[Th.\, 2, p.\,151]{bourbaki-lie-8}).
  \par
  We conclude that the group $\bfG_1$ must be $\mathrm{E}_6$; since its
  $27$-dimensional representations are faithful, the projection
  $\bfG_1\to \bfG$ is an isomorphism.
\end{proof}

\begin{remark}
  This criterion also shows that it may happen that the fourth moment
  $M_4(\bfG,V)$ of a representation of a group~$\bfG$ is equal to~$3$,
  but the representation~$V$ is not self-dual.
\end{remark}

\section{Finiteness of tannakian groups on abelian varieties}
\label{ssec-finite-ab}

The following result strenghtens Theorem~\ref{th-finite-ab-1} in
situations when one can apply Larsen's Alternative to the fourth moment
on abelian varieties.

\begin{proposition}\label{pr-finite-ab-2}
  Let $M$ be a geometrically simple perverse sheaf of weight zero on a
  simple abelian variety~$A$ over~$k$.  Let $d$ be the tannakian
  dimension of~$M$. If the group $\garith{M}$ is virtually central,
  then the object $\End(M)$ in
  $\Ppint(G)$ is punctual and the fourth moment of
  $\garith{M}$ is equal to $d^2$.
\end{proposition}

\begin{proof}
  We observe that $\garith{M}/(Z\cap
  \garith{M})$ is the arithmetic tannakian group of the arithmetically
  semisimple object
  $\End(M)$, and apply Theorem~\ref{th-finite-ab-1} to obtain the first
  conclusion.
  \par
  In particular, this implies that $\End(M)$, as a representation of
  $\garith{M}$, is a direct sum of characters. From~(\ref{eq-larsen-decomp}),
  applied to a decomposition in sum of characters, it follows that
  $$
  M_4(\garith{M})\geq d^2.
  $$
  \par
  On the other hand, let
  $K$ be a maximal compact subgroup of $\garith{M}(\Cc)$, and
  $\mu$ its Haar probability measure. By~(\ref{eq-int-moments}) and
  Schur's Lemma, we derive the inequality
  $$
  M_{4}(\garith{M})=\int_{K}|\Tr(g)|^{4}d\mu(g) \leq
  d^2\int_{K}|\Tr(g)|^{2}d\mu(g) =d^2,
  $$
  which concludes the proof.
\end{proof}

\begin{remark}
  There may exist irreducible subgroups~$\bfG$ of $\GL(V)$ with fourth
  moment equal to~$\dim(V)^2$. Indeed, this is the case, for instance,
  of any group which has the property that all irreducible
  representations with trivial central character have dimension~$1$,
  since only such representations can appear in the decomposition of
  $\End(V)$. A concrete example is given by finite Heisenberg
  groups\index{finite Heisenberg groups} (see, e.g., the
  paper~\cite{gerardin} of Gérardin for the relevant facts).
\end{remark}


\chapter{The product of the additive and the multiplicative groups}
\label{sec-product}

\section{Introduction}

In this chapter, we consider what is perhaps the simplest case of our
equidistribution results beyond those of the additive group and the
multiplicative group, namely the case of
$G=\Gg_m\times\Gg_a$. Concretely, this means that we are looking at the
distribution of two-parameter exponential sums of the type
\begin{equation}\label{eq-sums}
  \frac{1}{p}\sum_{(x,y)\in\Fpt\times\Fp}
  \chi(x)e\Bigl(\frac{ay}{p}\Bigr)t(x,y),
\end{equation}
where~$p$ is a prime number, $\chi$ a complex-valued multiplicative
character of the finite field $\Fp$, and the function~$t$ is a trace
function on $\Gg_m\times\Gg_a$ over~$\Fp$. In practice, we mostly
consider the analogues over extensions of~$\Ff_p$ of degree
$n\to+\infty$, but we will also discuss an horizontal statement in
Corollary~\ref{cor-hor-gagm}.

Throughout this chapter, we denote by~$k$ a finite field with an
algebraic closure~$\bar{k}$, and by~$\ell$ a prime different from the
characteristic of~$k$. We also fix a non-trivial additive
character~$\psi \colon k \to \bQl^\times$. For every $n\geq 1$, we
define~$\psi_n=\psi\circ\Tr_{k_n/k}$, a non-trivial additive character
of the extension~$k_n$ of~$k$ of degree~$n$ in~$\bar{k}$.

We always denote by~$G$ the group $\Gg_m\times\Gg_a$, and we will denote
by $p_1$ and~$p_2$ the projections $G\to \Gg_m$ and~$G\to\Gg_a$.  For
any $n\geq 1$ and any pair $(\chi,a)$ of an $\ell$-adic character of
$k_n^{\times}$ and an element of $k_n$, we will sometimes denote by
$\chag{\chi}{a}$ the character $(x,y)\mapsto \chi(x)\psi_n(ay)$ of
$G(k_n)$, and by $\mcL_{\chi,a}$ the corresponding $\ell$-adic character
sheaf.
\nomenclature[$c$]{$\chag{\chi}{a}$}{character on $\Gg_m\times\Gg_a$}

We first state the specialization of Theorem~\ref{th-3} to this case,
showing that there is always \emph{some} equidistribution statement for
the sums~(\ref{eq-sums}) in the vertical direction.

\begin{theorem}\label{th-1}
  Let~$M$ be an arithmetically semisimple $\ell$-adic perverse sheaf on
  $\Gg_m\times\Gg_a$ over~$k$, with trace function over~$k_n$
  denoted~$t(x,y;k_n)$.  Assume that~$M$ is pure of weight zero.
  \par
  There exist an integer $r\geq 0$ and a reductive
  subgroup~$\bfG\subset \GL_r$
  such that the sums
  $$
  S_n(a,\chi)=\sum_{(x,y)\in k_n^{\times}\times k_n}
  \chi(x)\psi_n(ay)t(x,y;k_n),
  $$
  where~$(a,\chi)$ are pairs of an element of~$k_n$ and a multiplicative
  character of $k_n^{\times}$, become equidistributed on average as
  $n\to +\infty$, with limit measure the image under the trace of the
  Haar probability measure on a maximal compact subgroup of~$\bfG(\Cc)$.
\end{theorem}

With $G=\Gg_m\times \Gg_a$, $\bfG=\garith{M}$ and $r$ the tannakian dimension of $M$, this is
Theorem~\ref{th-3} for the object~$M$.

The remainder of this chapter will be dedicated to the exploration of
special examples.  We consider in particular examples where the object
$M$ (and hence the trace function in~(\ref{eq-sums})) is supported on
the ``diagonal'' $y=x$. Larsen's Alternative will allow us to prove,
with surprisingly little computation, that in this case the
group~$\bfG$ in Theorem~\ref{th-1} is always essentially as large as
possible.

More precisely, we first define
$\Delta\colon \Gg_m\to \Gg_m\times\Gg_a$
\nomenclature[$Delta$]{$\Delta$}{diagonal embedding $\Gg_m\to \Gg_m\times\Gg_a$} \index{diagonal embedding} to be the
diagonal embedding $x\mapsto (x,x)$; this is a closed
immersion. Define the diagonal in $\Gg_m\times\Gg_a$ to be the image
of~$\Delta$, and $j\colon \Gg_m\to\Gg_a$ to be the open immersion.

For any morphism $\lambda\colon \Gg_m\to G$, for an integer $n\geq 1$
and a pair $(\chi,a)\in \charg{\Gg}_m(k_n)\times k_n$, we denote by
$\mcL^{\lambda}_{\chi,a}$ the sheaf
$\lambda_*(\mcL_{\chi}\otimes j^*\mcL_{\psi(ay)})$ on $G_{\bar{k}}$.
\nomenclature[$L$]{$\mcL^{\lambda}_{\chi,a}$}{sheaf $\lambda_*(\mcL_{\chi}\otimes j^*\mcL_{\psi(ay)})$}

\begin{theorem}\label{th-sld1}
  With notation as in Theorem~\emph{\ref{th-1}}, suppose that the input
  object~$M$ is geometrically simple and supported on the diagonal.
  Suppose that~$M$ is not punctual and not geometrically isomorphic to
  $\mcLd_{\eta,b}[1]$ for some
  $\chag{\eta}{b}\in\charg{G}(k_n)$. Then the integer~$r$ is~$\geq 2$
  and the group~$\Gg$ contains~$\SL_r$.
\end{theorem}

We will see that we can in fact fairly often show that~$\Gg=\GL_r$, and
in that setting the sums
$$
S_n(\chi,a)= \sum_{x\in k_n^{\times}}\chi(x)\psi_n(ax)t_M(x,x;k_n)
$$
tend to be distributed like the trace of a random matrix
in~$\Un_r(\Cc)$, almost independently of the input object~$M$.

\begin{example}\label{ex-1}
  Let~$a\in k$, and let~$\chi$ be a multiplicative character of~$k$. The
  \emph{Kloosterman--Salié sums}\index{Kloosterman--Salié sums} over~$k$ are defined by
  $$
  \Ks(\chi,a;k)=\frac{1}{\sqrt{|k|}}\sum_{x\in k^{\times}}
  \chi(x)\psi(ax+x^{-1}).
  $$
  
  These sums have been studied extensively, in particular because of
  their applications in the analytic theory of modular forms (see the
  surveys~\cite{kloosterman-survey} or~\cite{heath-brown}). We may
  fix~$\chi$, obtaining a family of exponential sums parameterized
  by~$a$: this is the discrete additive Fourier transform of the
  function which is~$0$ at~$x=0$ and otherwise maps~$x$ to
  $\chi(x)\psi(x^{-1})$. Alternatively, we may fix~$a$, and then we
  are considering the discrete Mellin transform of the
  function~$x\mapsto \psi(ax+x^{-1})$. These are both well-known
  examples of their respective theories, and their distribution
  properties are as follows:

  \begin{itemize}
  \item If~$\chi$ is the trivial character, we have \emph{Kloosterman
      sums},\index{Kloosterman sums} which are equidistributed with
    respect to the Sato\nobreakdash--Tate measure,\index{Sato--Tate measure} that is, to the image
    of the Haar probability measure on the space of conjugacy classes
    of~$\SU_2(\Cc)$; this reflects the fact that the geometric and
    arithmetic monodromy groups for the $\ell$-adic Fourier transform of
    the extension by zero of $\mcL_{\psi(x^{-1})}$ are both equal
    to~$\SL_2$, by work of Katz~\cite[Thm.\,11.1]{gkm}.
  \item If the characteristic~$p$ of~$k$ is odd and~$\chi$ is the
    character of order~$2$, then we have \emph{Salié sums},\index{Salié
      sums} whose arithmetic monodromy group is a finite subgroup
    of~$\SL_2$, isomorphic to a semi-direct product of~$\Fpt$
    and~$\Zz/2\Zz$ (this can be deduced from~\cite[Cor.\,8.9.2]{esde},
    which shows that the corresponding sheaf is Kummer-induced). The
    finiteness of the group reflects the fact that Salié sums can be
    computed elementarily (see,
    e.g.,~\cite[p.\,288,\,Exerc.\,50]{bourbaki-ts-2}), and is also an
    analogue of the fact that Bessel functions with half-integral index
    are elementary functions (see,
    e.g.,~\cite[p.\,269,\,Exerc.\,20]{bourbaki-ts-2}).
  \item If~$p\geq 7$ and~$\chi$ is fixed, but $\chi^2$ is non-trivial,
    then the neutral component of the geometric monodromy group
    is~$\SL_2$, but the determinant of geometric monodromy group is not trivial, and more precisely
    has order equal to the order of~$\chi$; see~\cite[Th.\,8.11.3, Lemma
    8.11.6]{esde}.
  \item If instead we fix~$a\in\Fpt$ and vary the multiplicative
    character~$\chi$, then the geometric tannakian group (which
    coincides with the one associated by Katz's theory in~\cite{mellin},
    see Appendix~\ref{ch-app-mellin})
    contains $\SL_2$ for all~$a$. Indeed, the sheaf $\mathcal{L}_{\psi(ax+x^{-1})}$ on $\Gg_m$ is not geometrically isomorphic to any of its non-trivial multiplicative translates by the same argument as in the proof of~\cite[Th.~14.2]{mellin}, and hence the characterisation in~\cite[Cor.\,8.3]{mellin} shows that this tannakian group is Lie-irreducible; since it is a subgroup of~$\GL_2$, it necessarily contains $\SL_2$.   If~$a=-1$, then the arithmetic and
    geometric tannakian groups are both equal to~$\SL_2$ (this is the case of
    the Evans sums in~\cite[Th.\,14.2]{mellin}). In general, the
    tannakian determinant is geometrically isomorphic to the skyscraper sheaf at~$\alpha=-1/a$ (so its Mellin transform is proportional to $\chi\mapsto \chi(\alpha)$) by \cite[Th.\,21.1]{mellin}. Letting $n$ denote the order of $\alpha$ in the finite group $\Fpt$, it follows that the geometric tannakian group consists of those matrices whose determinant is an $n$th root of unity. 
   
  \item If~$a=0$ and we vary~$\chi$, we have Gauss sums; the arithmetic
    and geometric tannakian groups are equal to~$\GL_1$.
  \end{itemize}

  The relation with Theorem~\ref{th-1} is the following: we are
  considering the finite field~$k$, and the perverse sheaf~$M$ of weight
  zero is $M=\Delta_*\sheaf{L}[1](1/2)$, where $\mcL$ is the lisse sheaf
  $\sheaf{L}=\mcL_{\psi(x^{-1})}$ of rank one; it is geometrically
  simple, and perverse since~$\Delta$ is a closed immersion
  (Corollary~\ref{cor-closed-immersion}). The group~$\Gg$ of
  Theorem~\ref{th-1} is then~$\GL_2$ (as follows from
  Theorem~\ref{th-sld1}).
  \par
  Note that when we specialize to a fixed character~$\chi$ or a
  fixed~$a$, we obtain a monodromy group or a tannakian group that is a
  subgroup of~$\Gg$ (as seems natural), which has the following
  property: the identity component of the derived group~$\Gg'$ is
  independent of~$\chi$ (resp. $a$), except for a finite exceptional
  set. In fact, the exceptional set for fixed~$\chi$ contains only the
  Legendre character (if~$p$ is odd), and the exceptional set for
  fixed~$a$ contains only~$a=0$.
  \par
  Note also that when we vary~$\chi$ for~$a$ fixed, only the neutral
  component of the identity of the geometric tannakian group is
  independent of~$\chi$, but the tannakian group is usually not
  connected.
  \par
  Finally, observe that here \emph{none} of the ``specialized''
  geometric tannakian groups for either~$\Gg_a$ or~$\Gg_m$ coincides
  with the geometric tannakian group~$\Gg=\GL_2$. However, in an
  intuitive sense, the collection of all of them ``generate'' this
  group.
  \par
  We expect these phenomena to be very general, and we will consider
  such questions in greater generality in later works.
\end{example}

\begin{remark}
  (1) Theorem~\ref{th-sld1} applies for instance to one-variable
  exponential sums of the form
  $$
  \frac{1}{\sqrt{|k|}}\sum_{x\in
    k^{\times}}\chi(x)\eta(g(x))\psi(ax+f(x))
  $$
  for suitable polynomials~$f$ and~$g$ and for a multiplicative
  character~$\eta$.
  \par
  It is worth noting that, even if we are only interested in the
  distribution of these one-variable sums (and not in the more general
  sums allowed by Theorem~\ref{th-1} with a two-variable trace
  function), the \emph{proof} of Theorem~\ref{th-1}, passing through the
  tannakian machinery, requires the consideration of objects supported
  on all of~$\Gg_m\times\Gg_a$, simply because the convolution of two
  objects on $\Gg_m\times \Gg_a$ that are supported on the
  diagonal~$\Delta$ will be supported on the product set
  $\Delta\cdot \Delta=\Gg_m\times\Gg_a$.
  \par
  (2) Remark~\ref{rm-gener-jac}\,(2), suggests a different
  interpretation of Theorem~\ref{th-sld1}. Indeed, using this remark, we
  can view $\Gg_m\times \Gg_a$ as a generalized jacobian of $C=\Pp^1$
  and the diagonal morphism $\Gg_m\to \Gg_m\times \Gg_a$ as a morphism
  of the type $x\mapsto (x)-(D)$ for a suitable divisor~$D$
  on~$\Gg_m$. For a perverse sheaf of weight zero on~$\Gg_m$, we have
  the arithmetic Mellin transform
  $$
  \chi\mapsto \sum_{x\in k_n^{\times}}t_M(x;k_n)\chi(x)
  $$
  as in the work of Katz, which may have a variety of tannakian groups
  (see~\cite[Ch.\,14\,to\,27]{mellin} for examples involving for
  instance $\SL_n$, $\GL_n$, $\Ort_{2n}$, $\SO_n$, $\Sp_{2g}$ and
  $G_2$). Then the further operation of twisting by an additive
  character $\psi$ leads to the sums
  $$
  (\chi,\psi)\mapsto  \sum_{x\in k_n^{\times}}t_M(x;k_n)\chi(x)\psi(x)
  $$
  which correspond to the diagonal object $\Delta_*M$ on the generalized
  jacobian\index{generalized jacobian} $\Gg_m\times \Gg_a$. The theorem is then, analytically, an
  instance of the common situation where twisting an exponential sum by
  a generic additive character leads to ``more random'' exponential sums
  (here, replacing a potentially complicated tannakian group on $\Gg_m$
  by one that in almost all cases contains the special linear
  group). Note however that the tannakian dimension may change when
  adding this extra twist.
\end{remark}

\begin{example}
  The following case of two-variable equidistribution has been studied
  ``by hand'' by Kowalski and Nikeghbali~\cite[\S 4.1,
  Th.\,11]{k-n}. Let $d>5$ be a fixed integer, and consider the sums
  $$
  S(\chi,a)=\frac{1}{\sqrt{|k|}} \sum_{t\in k}\chi(t^d-dt-a)
  $$
  where the character~$\chi$ is extended by~$\chi(0)=0$ if $\chi$ is
  non-trivial and~$\chi(0)=1$ if~$\chi$ is trivial.
  \par
  We can express these sums as Mellin transforms, namely
  $$
  S(\chi,a)=\sum_{x\in k^{\times}}\sum_{y\in k}
  \widehat{S}(x,y)\chi(x)\psi(ay)
  $$
  where
  $$
  \widehat{S}(x,y)=\frac{1}{|G(k)|}
  \dblsum_{\chag{\chi}{a}\in\charg{G}(k)} \overline{\chi(x)}\psi(-ay)S(\chi,a).
  $$
  We compute then
  \begin{align*}
    \widehat{S}(x,y)
    &=
      \frac{1}{\sqrt{|k|}} 
      \frac{1}{|G(k)|}
      \dblsum_{\chag{\chi}{a}\in\charg{G}(k)}
      \overline{\chi(x)}\psi(-ay)
      \sum_{t\in k}\chi(t^d-dt-a)\\
    &=\frac{1}{\sqrt{|k|}} 
      \frac{1}{|G(k)|}
      \sum_{t\in k}\sum_{a\in k}\psi(-ay)
      \sum_{\chi}\overline{\chi(x)}\chi(t^d-dt-a)
    \\
    &=\frac{1}{|k|^{3/2}}\sumsum_{\substack{t\in k,a\in k\\t^d-dt-a=x}}
    \psi(-ay)=
    \frac{1}{|k|^{3/2}}\sum_{t\in k}\psi(-y(t^d-dt-x)).
  \end{align*}

  Note that this trace function is not of diagonal type.  It was proved
  however in~\cite{k-n} that when~$|k|\to +\infty$ (including the
  horizontal case where $k=\Ff_p$ with $p\to+\infty$), the
  sums~$S(\chi,a)$ become equidistributed like the trace of random
  matrices in the unitary group~$\Un_{d-1}(\Cc)$. This was done by
  applying Deligne's equidistribution theorem, and the computation of
  the relevant monodromy group by Katz, for each fixed~$\chi$, and then
  averaging over~$\chi$.

  It would be interesting to recover this result directly from
  Theorem~\ref{th-1} (with $\Gg=\GL_{d-1}$), but it is not obvious how
  to do so: the reader can check that the computation of the fourth
  moment, for instance, is not at all straightforward.
\end{example}

\begin{remark}
  Finally we remark that since the key tool to compute tannakian
  groups for objects supported on the diagonal will be Larsen's
  Alternative combined with the Sidon property of the diagonal, one
  can prove similar results for objects of the form
  $[x\mapsto (x,x^3)]_*M$ on $\Gg_a^2$, for~$M$ on~$\Gg_a$ (in
  characteristic~$\not=3$), and objects of the form
  $[x\mapsto (x,1-x)]_*M$ on~$\Gg_m^2$ for~$M$
  on~$\Gg_m\setminus\{1\}$ (see Proposition~\ref{pr-sidon-morphisms},
  (5) and~(6)).  The corresponding exponential sums are of the form
  $$
  \sum_{x\in k_n}t_M(x;k_n)\psi_n(ax+bx^3)
  $$
  and
  $$
  \sum_{x\in k_n^{\times}\setminus\{1\}}t_M(x;k_n)\chi_1(x)\chi_2(1-x).
  $$
  respectively.
\end{remark}

\section{Tannakian group for diagonal objects}

We first compute the tannakian dimension~$r$ for a perverse sheaf
on~$G=\Gg_m\times \Gg_a$ which is supported on the diagonal.

\begin{lemma}\label{lm-rank-gagm}
  Let $M=\Delta_*(\sheaf{M})[1]$ for some geometrically irreducible
  middle extension sheaf~$\sheaf{M}$ on~$\Gg_m$.
  \par
  \begin{enumth}
  \item The tannakian dimension~$r$ of the object~$M$ is given by the formula
  \begin{equation}\label{eq-rank-gagm}
    r=\sum_{\lambda} \max(0,\lambda -1)+\sum_{x\in\bar{k}^{\times}}
    (\swan_x(\sheaf{M})+\Drop_x(\sheaf{M}))
    +\rank(\sheaf{M})+\swan_0(\sheaf{M}),
  \end{equation}
  where~$\lambda$ runs over the breaks of~$\sheaf{M}$ at infinity, in the
  sense of~\cite[Ch.\,1]{gkm}, counted with multiplicity.
  \item We have $r=1$ if and only if $M=\mcLd_{\eta,b}[1]$ for
    some $\chag{\eta}{b}\in\charg{G}$.
  \item For all but finitely many $a\in\bar{k}$, the tannakian dimension
    of $M_a=p_{1,*}M\otimes j^*\mcL_{\psi(ax)}$ on~$\Gg_{m,k(a)}$ is
    equal to~$r$.
  \end{enumth}
\end{lemma}

\begin{proof}
  (1) By Proposition~\ref{prop-dimPerv=dimVect}, it is enough to
  determine the ``generic'' value of the dimension of the cohomology
  space
  $$
  H^0_c(G_{\bar{k}},M\otimes p_1^*\mcL_{\chi}\otimes p_2^*\mcL_{\psi(ay)})
  $$
  as $\chi$ varies in~$\charg{\Gg}_m$ and $a$ in~$\bar{k}$.  We have a
  canonical isomorphism
  $$
  H^0_c(G_{\bar{k}},M\otimes p_1^*\mcL_{\chi}\otimes
  p_2^*\mcL_{\psi(ay)})= H^1_c(\Gg_{m,\bar{k}},\sheaf{M}\otimes
  \mcL_{\chi}\otimes j^*\mcL_{\psi(ax)}),
  $$
  If $\chi$ is non-trivial, this space is also isomorphic to
  $$
  H^1_c(\Aa^1_{\bar{k}},j_!(\sheaf{M}\otimes \mcL_{\chi})\otimes
  \mcL_{\psi(ax)}).
  $$
  \par
  For all but at most one value of $\chi$, the sheaf
  $j_!(\sheaf{M}\otimes \mcL_{\chi})$ is a Fourier sheaf\index{Fourier
    sheaf} in the sense of~\cite[(7.3.5)]{esde} (i.e., a middle
  extension sheaf $\mcF$ such that Deligne's Fourier transform is also a
  middle extension sheaf). Hence, the space
  $ H^1_c(\Aa^1_{\bar{k}},j_!(\sheaf{M}\otimes \mcL_{\chi})\otimes
  \mcL_{\psi(ax)})$ is the stalk at~$a$ of the Fourier transform
  of~$j_!(\sheaf{M}\otimes \sheaf{L}_{\chi})$, and its generic value
  $r_{\chi}$ as $a$ varies in~$\bar{k}$ is computed in~\cite[Lemma
  7.3.9,\,(2)]{esde}, namely
  $$
  r_{\chi}=\sum_{\lambda} \max(0,\lambda -1)+\sum_{x\in\bar{k}}
  (\swan_x(j_!(\sheaf{M}\otimes
  \sheaf{L}_{\chi}))+\Drop_x(j_!(\sheaf{M}\otimes \sheaf{L}_{\chi})),
  $$
  where $\lambda$ runs over the breaks at $\infty$
  of~$j_!(\sheaf{M}\otimes \sheaf{L}_{\chi})$, counted with
  multiplicity. Since $\mcL_{\chi}$ is lisse on~$\Gg_m$, the formulas
  $$
  \swan_x(j_!(\sheaf{M}\otimes \sheaf{L}_{\chi}))=\swan_x(\sheaf{M})
  \quad\quad \Drop_x(j_!(\sheaf{M}\otimes
  \sheaf{L}_{\chi}))=\Drop_x(\sheaf{M})
  $$
  hold for any $x\in\bar{k}^{\times}$. Since $\mcL_{\chi}$ is tamely
  ramified at~$0$ for $\chi$ non-trivial, we have
  $$
  \swan_0(j_!(\sheaf{M}\otimes \sheaf{L}_{\chi}))=\swan_0(\sheaf{M})
  \quad\quad \Drop_0(j_!(\sheaf{M}\otimes
  \sheaf{L}_{\chi}))=\rank(\sheaf{M})
  $$
  for $\chi$ non-trivial, which leads to~(\ref{eq-rank-gagm}).
  \par
  (2) Since $\rank(\sheaf{M})\geq 1$, and all terms in the
  sum~(\ref{eq-rank-gagm}) are non-negative, we deduce that the
  condition $r=1$ may hold only if $\sheaf{M}$ has rank~$1$ and
  $\sheaf{M}$ is lisse on~$\Gg_m$, tame at~$0$, and has (unique) break
  at most~$1$ at~$\infty$. Twisting by a suitable Kummer sheaf, we may
  then assume that $\sheaf{M}$ is lisse on~$\Aa^1$, and it must then be
  geometrically isomorphic to an Artin--Schreier sheaf, which by
  untwisting implies that $M$ is geometrically isomorphic to some
  $\mcLd_{\eta,b}$.
  \par
  (3) For the object $M_a=\sheaf{M}[1]\otimes j^*\mcL_{\psi(ax)}$
  on~$\Gg_{m,k(a)}$, the tannakian dimension is its compactly-supported
  Euler--Poincaré characteristic, which is equal to
  \begin{multline}\label{eq-ra}
    r_a= \swan_0(\sheaf{M}\otimes j^*\mcL_{\psi(ax)})
    +\swan_{\infty}(\sheaf{M}\otimes
    j^*\mcL_{\psi(ax)})
    \\+\sum_{x\in\bar{k}^{\times}}
    (\swan_x(\sheaf{M}\otimes j^*\mcL_{\psi(ax)})
    +\Drop_x(\sheaf{M}\otimes
    j^*\mcL_{\psi(ax)}))
  \end{multline}
  (see~(\ref{eq-ep-gm})). Since $\mcL_{\psi(ax)}$ is lisse on~$\Gg_a$,
  the formulas
  $$
  \swan_x(\sheaf{M}\otimes
  j^*\mcL_{\psi(ax)})=\swan_x(\sheaf{M}),\quad\quad
  \Drop_x(\sheaf{M}\otimes j^*\mcL_{\psi(ax)}))=\Drop_x(\sheaf{M})
  $$
  hold for $x\in \bar{k}$.
  \par
  Assume that~$a\not=0$. Let $\lambda$ be a break of~$\sheaf{M}$ at
  infinity, and $V_{\lambda}$ the corresponding break-space. Then
  $V_{\lambda}\otimes \mcL_{\psi(ax)}$ coincides with the
  $\mu$-break-space $W_{\mu}$ of $\sheaf{M}\otimes j^*\mcL_{\psi(ax)}$
  where $\mu=\max(1,\lambda)$, except possibly if ~$\lambda=1$ and
  $\mcL_{\psi(-ax)}$ occurs in $V_{\lambda}$.  Thus, for all but
  finitely many~$a$, we have
  \begin{align*}
    \swan_{\infty}(\sheaf{M}\otimes j^*\mcL_{\psi(ax)})
    =\sum_{\mu} \mu \dim W_{\mu}
    &=\sum_{\lambda} \dim V_{\lambda}
    +\sum_{\lambda>1} (\lambda-1)\dim V_{\lambda}\\
    &=\rank(\sheaf{M})+\sum_{\lambda}\max(0,\lambda-1),
  \end{align*}
  which leads to~$r_a=r$ by comparing~(\ref{eq-ra})
  with~(\ref{eq-rank-gagm}).
\end{proof}

\begin{remark}
  We will classify all objects of tannakian dimension~$1$ in
  Section~\ref{ssec-neg-one}, and the diagonal objects of tannakian
  dimension~$2$ in Section~\ref{ssec-rank2}.
\end{remark}

We continue with a lemma to exclude finite tannakian groups in the
diagonal situation. The first step is to exploit the specific shape
of~$G$ to understand the structure of the set of characters which are
not Frobenius unramified for suitable objects (or which are ramified,
for objects which are generically unramified).

\begin{lemma}\label{lm-structure-gmga}
  Let~$M$ be a perverse sheaf on $G=\Gg_m\times\Gg_a$ and~$N$ an object
  of $\braket{M}^{\arith}$ which is arithmetically semisimple and pure
  of weight~$0$.  For all but finitely many $a\in \bar{k}$, the set of
  $\chi\in\charg{\Gg}_m$ such that $\chag{\chi}{a}$ is not
  Frobenius-unramified for~$N$ is finite.
  \par
  In particular, if~$M$ has finite arithmetic tannakian group, then for
  all but finitely many $a\in \bar{k}$, the set of $\chi\in\charg{\Gg}_m$
  such that $\chag{\chi}{a}$ is ramified is finite.
\end{lemma}

\begin{proof}
  The first statement follows immediately from the proof of
  Proposition~\ref{pr-frob-unram} combined with
  Theorem~\ref{thm-vanish-gmga}.
  
  The last statement follows from the first as in the proof of
  Corollary~\ref{cor-finite-gen-unram}.
\end{proof}

\begin{lemma}\label{lm-finite}
  Let~$C\subset G=\Gg_m\times\Gg_a$ be a line given by $y=\alpha x$
  where $\alpha\in k^{\times}$.
  \par
  Let~$M$ be a geometrically simple perverse sheaf
  on~$\Gg_m\times\Gg_a$ supported on~$C$ and of weight zero. Assume that
  the arithmetic tannakian group~$\Gg$ of~$M$ is finite. Then~$M$ is
  punctual.
\end{lemma}

\begin{proof}
  The assumption implies that $M$ is generically unramified by
  Corollary~\ref{cor-finite-gen-unram}.
  \par
  We assume that~$M$ is not punctual to get a
  contradiction. Then $M$ is, up to twist and shift, the pushforward
  to~$G$ of a middle extension sheaf~$\sheaf{M}$
  on~$\Gg_m\simeq C$.

  For all~$a$, we denote
  $\sheaf{M}_a=\sheaf{M}\otimes j^*\sheaf{L}_{\psi(ax)}$; then
  $M_a=\sheaf{M}_a[1](1/2)$ is a perverse sheaf on~$\Gg_m$.

  By Lemma~\ref{lm-structure-gmga}, 
  there exists~$n\geq 1$ and $a\in k_n$ such that for all but finitely
  many $\chi\in\charg{\Gg}_m$, the character $\chag{\chi}{a}$ is
  unramified for~$M$. The action of the Frobenius automorphism of~$k_n$
  on the space
  $$
  \hH^0_c(G_{\bar{k}},M\otimes\sheaf{L}_{\chi,a})
  =\hH^0_c(\Gg_{m, \bar{k}},M_a\otimes\sheaf{L}_{\chi})
  $$
  is then by assumption of finite order bounded independently
  of~$\chi$. The corresponding unitary Frobenius elements
  $\Thetaf_{M_a,k_{nm}}(\chi)$, for $m\geq 1$, are then dense in a
  maximal compact subgroup~$K$ of the complex points of the arithmetic
  tannakian group of the perverse sheaf~$M_a$ on~$\Gg_m$ by
  Corollary~\ref{cor-generate}. It follows that~$K$, and hence
  also~$\garith{M_a}$, is a finite group since a compact real Lie
  group has no non-trivial small subgroup. By Katz's results on finite
  tannakian groups on~$\Gg_m$ (see Theorem~\ref{th-katz-finite}), this
  would imply that the perverse sheaf~$M_a$ is punctual, which is a
  contradiction.
\end{proof}


We will now prove a slightly more general statement than
Theorem~\ref{th-sld1}.

\begin{theorem}\label{th-sld}
  Let~$\lambda\colon \Gg_m\to \Gg_m\times\Gg_a$ be the closed embedding
  $\lambda (x)=(x,\alpha x)$ for some $\alpha\in k^{\times}$ and let~$C$
  be its image.
  \par
  Let~$M$ be a geometrically simple perverse sheaf
  on~$\Gg_m\times\Gg_a$ supported on~$C$ and of weight zero. Assume
  that~$M$ is not punctual, and that $M$ is not geometrically isomorphic
  to $\mcL^{\lambda}_{\eta,b}[1](1/2)$ for some
  $\chag{\eta}{b}\in\charg{G}(k)$.
  \par
  Let~$r\geq 0$ be the tannakian dimension of~$M$ and
  denote~$\Gg=\garith{M}\subset \GL_r$.
  \par
  We then have~$r\geq 2$, the group~$\Gg$ contains~$\SL_r$ and the
  standard representation of~$\Gg$ in~$\GL_r$ is not self-dual.
\end{theorem}

Note that the last item implies in particular that~$\Gg$ cannot be equal
to~$\SL_2$. 

\begin{proof}
  We may assume that $\alpha=1$.  We first note that our assumptions
  and Lemma~\ref{lm-rank-gagm} imply that~$r\geq 2$ (otherwise, $M$
  would be punctual or geometrically isomorphic to some
  perverse sheaf $\mcL^{\lambda}_{\eta,b}[1](1/2)$).

  We will apply Larsen's Alternative. The closed immersion $\lambda$
  is a Sidon morphism (Proposition~\ref{pr-sidon-morphisms}\,(1)), and
  therefore we have $M_4(\Gg)=2$ by
  Proposition~\ref{pr-sidon-moments1} (since~$r\geq 2$). 

  \par
  Our assumptions therefore imply that $M_4(\Gg)=2$.  By Larsen's
  Alternative (Theorem~\ref{th-larsen}\,(3)), it follows that either
  $\Gg$ contains~$\SL_r$, or $\Gg/\Gg\cap Z$ is finite, where
  $Z\subset \GL_r$ is the group of scalar matrices. We must show that
  this second case actually does not arise. We proceed by contradiction,
  assuming therefore that $\Gg/\Gg\cap Z$ is finite.
  \par
  The intersection $\Gg\cap Z$ is either finite or equal to~$Z$. In the
  first case, the group~$\Gg$ would be finite, so that the object~$M$
  would be punctual by Lemma~\ref{lm-finite}, which contradicts our
  assumptions.
  \par
  So we are left with the case $\Gg\cap Z=Z$. The object $\End(M)$ of
  $\braket{M}^{\arith}$ has tannakian group $\Gg/\Gg\cap Z$, which is
  then finite. In particular, this object is generically unramified
  (Corollary~\ref{cor-finite-gen-unram}).

  Let $n\geq 1$. For $a\in k_n$, the complex
  $M_a=p_{1,*}M\otimes j^*\mcL_{\psi(ax)}$ on~$\Gg_{m,k_n}$ is a perverse
  sheaf, geometrically simple and of weight~$0$, since the restriction
  of~$p_1$ to~$C$ is an isomorphism.  For all but a bounded number of
  $a\in k_n$, Lemma~\ref{lm-structure-gmga} implies that $M_a$ has the
  property that
  $$
  S(\End(M_a),\chi)=|S(M_a, \chi)|^2=|S(M,\chag{\chi}{a})|^2
  $$
  take only finitely many values as $\chi\in\charg{\Gg}_{m,k_n}$
  varies. By equidistribution, this is only possible if the arithmetic
  tannakian group of the object $\End(M_a)\in\braket{M_a}^{\arith}$ on
  $\Gg_{m,k_n}$ is finite.  By Katz's Theorem~\ref{th-katz-finite}, this
  implies that~$\End(M_a)$ is punctual, say
  $$
  \End(M_a)=\bigoplus_{s\in S_a}n(a,s) \gamma_{a,s}^{\deg}\otimes
  \delta_{s}
  $$
  for a subset $S_a\subset k_n^{\times}$, integers $n(a,s)\geq 1$ and
  unitary scalars $\gamma_{a,s}$. For all but finitely many
  $a\in\bar{k}$, we know also from Lemma~\ref{lm-rank-gagm}\,(3) that
  $$
  r^2=\dim \End(M)=\dim \End(M_a)=\sum_{s\in S_a}n(a,s).
  $$
  \par
  Since all $\chi\in\charg{G}(k_n)$ are unramified for $\End(M_a)$ and
  $|\gamma_{a,s}|=1$, we compute
  \begin{multline*}
    \frac{1}{|\charg{\Gg}_m(k_n)|}
    \sum_{\chi\in\charg{\Gg}_m(k_n)}|S(M_a,\chi)|^4=
    \frac{1}{|\charg{\Gg}_m(k_n)|}
    \sum_{\chi\in\charg{\Gg}_m(k_n)}|S(\End(M_a),\chi)|^2
    \\
    = \frac{1}{|\charg{\Gg}_m(k_n)|}
    \sum_{\chi\in\charg{\Gg}_m(k_n)}\Bigl|\sum_{s\in
      S_a}n(a,s)\gamma_{a,s}^n\chi(s)\Bigr|^2=\sum_{s\in S_a}n(a,s)^2
    |\gamma_{a,s}^n|^2 \geq r^2.
  \end{multline*}
  \par
  Averaging over~$a\in k_n$, then letting $n\to+\infty$, it follows that
  $M_4(\Gg) \geq r^2\geq 4$, which is a contradiction.
  \par
  Finally, we note that the tannakian dual of~$M$ is supported on the
  image of the diagonal under the inversion map of~$\Gg_m\times\Gg_a$,
  namely on the hyperbola
  $$
  \{(x^{-1}, -x)\,\mid\, x\in\Gg_m\}\subset \Gg_m\times
  \Gg_a.
  $$
  \par
  Since this is not a translate of the diagonal, the tannakian dual
  of~$M$ cannot be geometrically isomorphic to~$M$.
\end{proof}

\begin{example}\label{ex-fourth-moment-2}
  (1) Suppose that~$M=\mcLd_{\eta,b}[1](1/2)$ for some
  $(\eta,b)\in\charg{G}(k)$, which corresponds to the case excluded in
  Theorem~\ref{th-sld}. For $n\geq 1$, denote by $\eta_n$ the character
  $\eta\circ N_{k_n/k}$ of $k_n^{\times}$. Then the sums $S_n(\chi,a)$
  are essentially Gauss sums, namely
  $$
  S_n(\chi,a)=\frac{1}{|k|^{n/2}} \sum_{x\in k_n^{\times}}
  (\chi\eta_n)(x)\psi_n((a+b)x)=\frac{1}{|k|^{n/2}}
  \overline{(\chi\eta_n)}(a+b)\tau(\chi\eta_n,\psi_n)
  $$
  (see~(\ref{eq-gauss-sum}) for the normalization).
  \par
  The equidistribution properties of the Gauss sums are well-known (see
  for instance~\cite[Th.\,9.5]{gkm}), and one deduces easily that the
  arithmetic tannakian group of~$M$ is equal to~$\GL_1$. The fourth
  moment of all sums $S_n(\chi,a)$ converges to~$2$, as we saw in the
  previous proof, but the single contribution to the fourth moment of
  the (ramified) character~$\chag{\eta^{-1}}{-b}$ is
  $(|k_n|-1)^4/|k_n|^{4}\to 1$. (See Proposition~\ref{prop-rk1} for the
  classification of objects of tannakian dimension~$1$ in general.)
  \par
  (2) Let $M_0=\mcK\ell_{2,\psi}(1/2)$ be the Kloosterman complex of
  rank~$2$ on~$\Gg_m$ (see~(\ref{eq-kloosterman-sheaf})) associated
  to~$\psi$, twisted to be pure of weight~$0$
  (see~\cite[Th.\,8.4.13]{esde}). It is of the form
  $\mathcal{M}_0[1](1/2)$ for some middle extension
  sheaf~$\mathcal{M}_0$, pure of weight~$0$ as lisse sheaf on~$\Gg_m$.
  \par
  The object $M_0$ has tannakian dimension~$1$ and geometric tannakian
  group equal to~$\GL_1$ as a $\Gg_m$-object (since it is a
  hypergeometric complex, see Theorem~\ref{th-hypergeometric}). On the
  other hand, the object $M=\Delta_*M_0=\Delta_*\sheaf{M}_0[1](1/2)$ on
  $G$ has tannakian dimension~$2$, and arithmetic tannakian group
  $\GL_2$ by Lemma~\ref{lm-rank-gagm} and Theorem~\ref{th-sld}.
  \par
  We compute the corresponding exponential sums to see the concrete
  meaning of the theorem in this case. For $n\geq 1$ and
  $\chag{\chi}{a}\in\charg{G}(k_n)$, we have the formula
  \begin{align*}
  S_n(\chi,a)&=
  \frac{1}{|k_n|}
  \sum_{x\in k_n^{\times}}
  \Bigl(\sum_{y\in k_n^{\times}}\psi_n(xy+1/y)\Bigr)
  \chi(x)\psi_n(ax)
  \\
  &=
  \frac{1}{|k_n|}
  \sum_{y\in k_n^{\times}}\psi_n(1/y)\sum_{x\in k^{\times}}
  \psi_n((a+y)x)\chi(x).
\end{align*}
For $\chi$ non-trivial, extended by $\chi(0)=0$, this is equal to
$$
S_n(\chi,a)=\frac{\tau(\chi,\psi_n)}{|k_n|} \sum_{y\in
  k_n^{\times}}\overline{\chi(a+y)}\psi_n(1/y).
$$
\end{example}

In order to complete the determination of the tannakian group in the
situation of Theorem~\ref{th-sld}, we need to compute the tannakian
determinant of~$M$.\index{tannakian determinant} There are various tools to do this:
\begin{enumerate}
\item one can attempt to compare the tannakian determinants for $M$
  (supported on a line) with those on~$\Gg_m$, which can often be
  computed using the results of Katz~\cite{mellin};
\item one can use the relation between the tannakian determinant at
  $\chag{\chi}{a}$ and the determinant of Frobenius acting on the
  cohomology group
  $$
  H^0_c(G_{\bar{k}}, M_{\chag{\chi}{a}})\simeq
  H^0_c(\Gg_{m,\bar{k}}, M_a\otimes\mcL_{\chi})
  $$
  (with notation as above). The latter determinant (on a curve) may
  often be computed using the theory of local epsilon
  factors\index{local epsilon factor} of Deligne and Laumon (see
  Appendix~\ref{ch-app-product}). We will not give explicit examples
  here, but we perform a computation of this kind in
  Chapter~\ref{sec-variance} (see
  Proposition~\ref{pr-infinite-order-variance}).
\end{enumerate}

As an example of the first approach, we have for instance the following
criterion:


\begin{proposition}\label{pr-determinant}
  Let~$C\subset G=\Gg_m\times\Gg_a$ be a line defined by $y=\alpha x$
  where $\alpha\in k^{\times}$.  Let~$M$ be a geometrically simple
  perverse sheaf on~$\Gg_m\times\Gg_a$ supported on~$C$ and of weight
  zero. Assume that~$M$ is not punctual, and that the restriction
  of~$M$ to~$C$ is not geometrically isomorphic to $\mcL_{\eta,b}[1]$
  for some multiplicative character~$\eta$ and some~$b$. Let~$r\geq 0$
  be the tannakian dimension of~$M$.
  \par
  Suppose that for all but finitely many~$a$, the tannakian determinant
  of~$p_{1,*}M_{\chag{1}{a}}$ on~$\Gg_m$ is geometrically of infinite
  order. Then we have $\Gg=\GL_r$.
\end{proposition}

\begin{proof}
  Since~$\Gg$ contains~$\SL_r$, it suffices to prove that the
  determinant of~$\Gg$ is arithmetically of infinite order.

  Since $p_1\colon C\to \Gg_m$ is an isomorphism, it follows that for
  any $a\in \Gg_a$, the object $N_a=p_{1,*}M_{\chag{1}{a}}$ on~$\Gg_m$
  is a perverse sheaf, and is arithmetically simple and pure of
  weight~$0$.
  \par
  We claim that the assumption implies that the determinants of
  $\Theta_{M,k_n}(\chag{\chi}{a})$ are equidistributed on average on the
  unit circle, where $\chag{\chi}{a}$ vary among Frobenius-unramified
  classes for the determinant. Indeed, denoting $\mcX$ this set of
  characters, we have for any non-zero integer $h\in\Zz$ the relation
  $$
  \frac{1}{|G(k_n)|}
  \sum_{\chag{\chi}{a}\in \mcX(k_n)}
  \det(\Theta_{M,k_n}(\chag{\chi}{a}))^h
  = \frac{1}{|k_n|}
  \sum_{a\in k_n}
  \frac{1}{|k_n^{\times}|}
  \sum_{\substack{\chi\in \charg{\Gg}_m(k_n)\\\chag{\chi}{a}\in \mcX(k_n)}}
  \det(\Theta_{N_a,k_n}(\chi))^h.
  $$
  \par
  The contribution of those finitely many~$a$ such that $N_a$ has
  geometrically finite-order determinant tends to~$0$. For the other
  values of~$a$, we have
  $$
  \lim_{N\to+\infty} \frac{1}{N}\sum_{n\leq N}
  \frac{1}{|k_n^{\times}|}
  \sum_{\substack{\chi\in
      \charg{\Gg}_m(k_n)\\\chag{\chi}{a}\in \mcX(k_n)}}
  \det(\Theta_{N_a,k_n}(\chi))^h=0
  $$
  by equidistribution, in fact uniformly with respect to~$a$ since the
  complexity of $\det(N_a)$ is bounded independently of~$a$. We deduce
  that
  $$
  \lim_{N\to+\infty} \frac{1}{N}\sum_{n\leq N}
  \frac{1}{|k_n|}\sum_{a\in k_n}
  \frac{1}{|k_n^{\times}|}
  \sum_{\substack{\chi\in
      \charg{\Gg}_m(k_n)\\\chag{\chi}{a}\in \mcX(k_n)}}
  \det(\Theta_{N_a,k_n}(\chi))^h=0,
  $$
  which proves the claim.

  But by Theorem~\ref{th-4}, the determinants of
  $\Theta_{M,k_n}(\chag{\chi}{a})$ are known to be equidistributed on
  average on the subset of the unit circle corresponding to the
  determinant of the arithmetic tannakian group of~$M$; if the latter
  were finite, this would be a finite group of roots of unity. By
  contraposition, the result follows.
\end{proof}

\begin{remark}
  If $\det(M)$ is known to be generically unramified, then it suffices
  to assume that the tannakian determinant of~$p_{1,*}M$ on~$\Gg_m$ is
  geometrically of infinite order, since in this case we can apply
  Proposition~\ref{pr-determinant-tori} to some twist
  $M_{\chag{\chi_1}{a_1}}$ such that the set of characters $\chi$ for
  which the character $\chag{\chi_1}{a_1}\chag{\chi}{0}$ is unramified
  is generic.
\end{remark}

\begin{example}\label{ex-non-unram}
  Proposition~\ref{pr-determinant} applies for instance to objects of
  the form
  $$
  M=\sheaf{L}_{\eta(f)}[1](1/2)
  $$
  where $\eta$ is a non-trivial multiplicative character of~$k$,
  and~$f\in k[X]$ is a polynomial such that~$f(0)\not=0$ with
  degree~$d\geq 2$ such that~$\eta^d$ is non-trivial, as explained by
  Katz in~\cite[Th.\,17.5]{mellin}. Indeed, in this case, the assumption
  of the proposition holds for all~$a\not=0$.

  The dimension formula~(\ref{eq-rank-gagm}) shows that the tannakian
  dimension is~$d+1$.
  Note that~\cite[Th.\,17.5]{mellin} provides the equidistribution for
  the subfamily with~$a=0$, under the assumption that~$f$ is not of the
  form $g(X^b)$ for some $b\geq 2$, but as traces of matrices
  in~$\Un_d(\Cc)$, because the corresponding object on~$\Gg_m$ has
  tannakian dimension~$d$. This means that the characters
  $\chag{\chi}{0}$ are examples of weakly-unramified characters for~$M$
  which are not unramified (since they do not give the ``right''
  dimension).
\end{example}

As explained in Remark~\ref{rm-horizontal}\,(2), we expect that we can
apply Theorem~\ref{th-horiz} unconditionally to~$G$. Thus this
proposition should imply the following result:

\begin{corollary}\label{cor-hor-gagm}
  Let~$\ell$ be a prime number. Assume that
  \textup{Theorem~\ref{th-horiz}} holds for~$G$. For all $p\not=\ell$,
  let~$M_p$ be a geometrically simple perverse sheaf of weight zero on
  $(\Gg_m\times \Gg_a)_{\Ff_p}$ supported on the diagonal with
  $c_u(M_p)\ll 1$, where~$u$ is the natural locally-closed immersion
  $\Gg_m\times\Gg_a\injecte \Aa^2\injecte \Pp^2$. Suppose that the
  tannakian dimension~$r$ of~$M_p$ is independent of~$p$ and that $M_p$
  satisfies the assumption of Proposition~\textup{\ref{pr-determinant}}
  for~$k=\Ff_p$. Then the sums
  $$
  S(\chi,a;p)=\sum_{x\in\Ff_p^{\times}}t_{M_p}(x)\chi(x)
  e\Bigl(\frac{ax}{p}\Bigr),
  $$
  for $\chi$ a multiplicative character of~$\Ff_p$ and~$a\in\Ff_p$,
  become equidistributed according to the trace of a random unitary
  matrix in~$\Un_r(\Cc)$.
\end{corollary}

\section{Diagonal objects of dimension~$2$}\label{ssec-rank2}

The computation of Lemma~\ref{lm-rank-gagm} allow us, for instance, to
classify those sheaves~$\sheaf{M}$ which give rise to geometrically
simple perverse sheaves on the diagonal with tannakian dimension~$r=2$.
Indeed, the (usual) rank of~$\sheaf{M}$ must be either~$1$ or~$2$.

In the first case, one and only one of the following conditions must be
true:
\begin{enumerate}
\item $\sheaf{M}$ is lisse on~$\Gg_m$, tamely ramified at~$0$ and has
  (unique) break at $\infty$ equal to~$2$; if the characteristic
  of~$k$ is not equal to~$2$, then the only such sheaves are
  isomorphic to
  $$
  \sheaf{L}_{\psi(ax^2+bx)}\otimes\sheaf{L}_{\eta}
  $$
  where~$a\not=0$ and~$\eta$ is a multiplicative character. The
  corresponding exponential sums are ``twisted quadratic Gauss sums''.
\item $\sheaf{M}$ is lisse on $\Gg_m$ and has Swan conductor~$1$
  at~$0$ and unique break $\leq 1$ at~$\infty$; the only such sheaves
  are isomorphic to
  $$
  \sheaf{L}_{\psi(a/x+bx)}\otimes\sheaf{L}_{\eta}
  $$
  where~$a\not=0$ and~$\eta$ is a multiplicative character (we recover
  the example of Kloosterman--Salié sums).
\item there exists a unique~$\beta\in\bar{k}^{\times}$ such that
  $\sheaf{M}$ is lisse on~$\Aa^1\setminus\{\beta\}$, it has unique break
  $\leq 1$ at~$\infty$ and is tamely ramified at~$0$ and~$\beta$; the only
  such sheaves are isomorphic to
  $$
  \sheaf{L}_{\eta(x-\beta)}\otimes\mcL_{\xi(x)}\otimes\sheaf{L}_{\psi(\alpha
    x)}
  $$
  where~$\beta\not=0$, $\alpha\in\Gg_a$ and~$\eta$ and~$\xi$ are
  multiplicative characters. The corresponding exponential sums are
  $$
  \frac{1}{|k_n|^{1/2}}\sum_{x\in k_n^{\times}}
  \eta(x-\beta)(\chi\xi)(x)\psi((a+\alpha) x),
  $$
  which can be seen as twisted Jacobi sums.
\end{enumerate}

On the other hand, if $\sheaf{M}$ has rank~$2$, then it must be lisse
on~$\Gg_m$, tamely ramified at~$0$ and have breaks~$\leq 1$
at~$\infty$.  Up to twist by a multiplicative character, we obtain a
sheaf lisse on~$\Aa^1$ with breaks~$\leq 1$ at~$\infty$. Since we
assume $\sheaf{M}$ to be geometrically irreducible, the two breaks
must be equal, say equal to~$\lambda$.  Their sum is the Swan
conductor at~$\infty$, which is also the Euler--Poincaré
characteristic (since~$\sheaf{M}$ is lisse on~$\Gg_m$ and tame at~$0$,
see~(\ref{eq-ep-gm})); thus either $\lambda=1/2$ or $\lambda=1$. The
first case gives Euler--Poincaré characteristic equal to~$1$, so we
have a hypergeometric sheaf of rank~$2$ by Katz's classification (see
Theorem~\ref{th-hypergeometric}, e.g., a Kloosterman sheaf of
rank~$2$, with the corresponding sums described in
Example~\ref{ex-fourth-moment-2}\,(2)). In the second case, we may
have a pullback of such a sheaf by~$x\mapsto x^2$. For the pullback of
the Kloosterman sheaf, the exponential sums are then given by the
formulas
\begin{align*}
  S_n(\chi,a)&=\frac{1}{|k_n|} \sum_{x\in k_n^{\times}} \Bigl(
  \sum_{y\in k_n^{\times}}\psi_n(xy+xy^{-1})\Bigr) \chi(x)\psi_n(ax)
  \\
  &= \frac{1}{|k_n|} \sum_{y\in k_n^{\times}} \sum_{x\in
    k_n^{\times}}\chi(x)\psi_n(x(a+y+y^{-1}))
  \\
  &= \frac{\tau(\chi,\psi_n)}{|k_n|} \sum_{\substack{y\in
      k_n^{\times}\\a+y+y^{-1}\not=0}} \overline{\chi(a+y+y^{-1})},
\end{align*}
for $\chi$ non-trivial.


\section{Negligible objects and objects of dimension one}
\label{ssec-neg-one}

We conclude our discussion of the group $G=\Gg_m\times\Gg_a$ by
classifying the negligible objects as well as the objects of tannakian
dimension~$1$. This may be helpful for further investigations (e.g., to
compute the determinant of the tannakian group in some cases, or to
apply the Goursat--Kolchin--Ribet criterion,
see~\cite[Prop.\,1.8.2]{esde}).  \index{Goursat--Kolchin--Ribet
  criterion}

We will denote by $\ft_{\psi/\Gg_m}$ the relative Fourier transform
functor $\Der(G)\to \Der(G)$, defined by
$$
\ft_{\psi/\Gg_m}(M)=Rq_{2,!}(q_1^*M\otimes\mcL_{\psi(ay)})
$$
where $q_1$ and $q_2$ are the two projections
$\Gg_m\times \Gg_a\times\Gg_a\to \Gg_m\times \Gg_a$, and we use
coordinates $(x,y,a)$ on~$G\times \Gg_a=\Gg_m\times\Gg_a\times
\Gg_a$. This functor satisfies the same basic properties as the Fourier
transform over base fields (see,
e.g.,~\cite[\S\,2]{KL-fourier-exp-som}), and in particular
$\ft_{\psi/\Gg_m}(M)[1]$ is perverse if~$M$ is perverse.

\begin{proposition}
  \label{prop-rk0}
  Let~$M$ be a simple perverse sheaf on~$G$ over~$\bar{k}$.
  \par
  The perverse sheaf $M$ is negligible if and only if~$M$ is isomorphic
  to an object of the form
  \begin{equation}\label{eq-f1}
    p_1^*(N)\otimes \mcL_{\psi(ay)}[1]
  \end{equation}
  for some perverse sheaf~$N$ on~$\Gg_m$ and some $a$, or to an object
  of the form
  \begin{equation}\label{eq-f2}
    \mcL_{\chi}[1]\otimes p_2^*(M),
  \end{equation}
  for some perverse sheaf $M$ on~$\Gg_a$ and some multiplicative
  character~$\chi$.
\end{proposition}

\begin{proof}
  It is elementary that the objects of the two forms in the statement
  are negligible (see Example~\ref{ex-negligible}), so we need to prove
  the converse.
  
  Let $M$ be a simple negligible perverse sheaf on~$G$. We consider the
  (shifted) Fourier transform 
  $$
  F=\ft_{\psi/\Gg_m}(M)[1]
  $$ of $M$ relative
  to $\Gg_m$; this is a perverse sheaf on~$G$. For $a\in \Gg_a$, the
  restriction~$F_a$ of this complex to $\Gg_m\times\{a\}$ is isomorphic
  to $p_{1!}(M_{\psi(ay)})$. Hence, for~$a$ generic, the object
  $F_a=(F\vert\Gg_m\times\set{a})[-1]$ is a perverse
  sheaf 
  by Proposition~\ref{pr-kl}; moreover, if~$a$ is such
  that~$\Gg_m\times \{a\}$ intersects a dense open subset where~$M$ is
  lisse, the generic rank of~$F_a$ is still zero.

  We now distinguish cases according to the dimension $d$ of the support
  of $F$.

  (1) If $d=0$, then $F$ is supported on finitely many points. Since~$M$
  is simple and the Fourier transform preserves simple perverse sheaves,
  $F$ is also simple. This implies that the support of $F$ is
  irreducible, and hence it is a single point $(x,a)$. The point $a$
  correspond to the character $\psi(ay)$ via inverse Fourier
  transform. Hence, $M$ is of the form $p_1^*(N)\otimes\mcL_{\psi(ay)}$,
  where $N$ is a sheaf with finite support in $\Gg_m$, which is an
  object of the form~(\ref{eq-f1}).

  (2) If $d=1$, then the support of $F$ is a curve
  $C\subset \Gg_m\times\Gg_a$. If, for generic $a\in \Gg_a$, the
  intersection of $C$ with $\Gg_m\times\{a\}$ is non-empty, then the
  support of $F_a$ is finite and non-empty, contradicting the fact that
  this sheaf is of generic rank zero. Hence, for generic $a\in \Gg_a$,
  the intersection of $C$ with $\Gg_m\times\{a\}$ is empty.  We then
  deduce that~$F_a=0$ for generic $a$. Hence, $C$ is a finite union of
  horizontal lines. As in (1), $C$ is irreducible, and hence is of the
  form $\Gg_m\times\set{a}$ for some $a\in \Gg_a$. Hence,~$M$ is of the
  form $p_1^*(N)\otimes \mcL_{\psi(ay)}$ for some perverse sheaf $N$ on
  $\Gg_m$; this is again of the form~(\ref{eq-f1}).

  (3) Finally, assume that $d=2$. Let~$\eta$ be the generic point
  of~$\Gg_a$.  Then $F_{\eta}$ is a perverse sheaf with Euler--Poincaré
  characteristic zero on~$\Gg_m$ over~$k(\eta)$. By
  Proposition~\ref{pr-ep-0}, it follows that $F_{\eta}$, viewed as a
  perverse sheaf on $\Gg_m$ over $k(\eta)$, is geometrically isomorphic
  to a Kummer perverse sheaf $\mcL_\chi[1]$ for some multiplicative
  character~$\chi$.  Hence,~$F$ is of the form
  $p_2^*(N')\otimes \mcL_\chi[1]$ for some perverse sheaf~$N'$
  on~$\Gg_a$. Taking the relative inverse (shifted) Fourier transform,
  we find that there exists some object $N$ of~$\Der(\Gg_a)$ such that
  $M$ is isomorphic to $p_2^*(N)\otimes \mcL_{\chi}[1]$.
\end{proof}

We will now classify the objects of tannakian dimension one.

By Proposition~\ref{pr-product}, the most obvious objects of tannakian
dimension one on $\Gg_m\times\Gg_a$ are those of the form
$H\boxtimes N$, for some simple hypergeometric complex~$H$ on~$\Gg_m$
and some simple perverse sheaf~$N$ on~$\Gg_a$ with Fourier transform of
rank one (we refer again to Section~\ref{sec-hypergeometric} for
reminders concerning hypergeometric complexes, which are the objects of
tannakian rank~$1$ on the multiplicative group). The next lemma provides
another class of such objects.

\begin{lemma}\label{lm-mf}
  Let~$f\in\bar{k}(x)^{\times}$ be a rational function and~$U$ a dense
  open set of~$\Gg_a$ where $f$ is defined and non-zero. Let
  $C\subset V=\Gg_m\times U$, with coordinates $(x,a)$, be the curve
  with equation $f(a)=x$. Let $\overline{\Qq}_{\ell,f}$ be the
  intermediate extension to $\Gg_m\times\Gg_a$ of the constant sheaf
  on~$C$ shifted by~$1$, and let~$M_f$ be the inverse relative Fourier
  transform of~$\overline{\Qq}_{\ell,f}$, also shifted by~$1$.

  \begin{enumth}
  \item Write $f=f_1/f_2$ with $f_i\in \bar{k}[x]$
    coprime. Let~$\widetilde{C}\subset G$ be the curve with equation
    $$
    f_1(a)=f_2(a)x,
    $$
    and let~$i\colon \widetilde{C}\to G$ be the closed immersion. We
    then have isomorphisms 
    $$
    \overline{\Qq}_{\ell,f}\simeq i_*\Qlb[1]\simeq i_!\Qlb[1].
    $$
  \item The perverse sheaf $M_f$ on~$G$ has tannakian dimension one.
  \item For any~$y_0\in\Gg_a$, the restriction of $M_f$ to
    $\Gg_m\times \{y_0\}\subset G$ is of the form $\mcG_{y_0}[2]$ for
    some sheaf~$\mcG_{y_0}$ on~$\Gg_m$, identified to a sheaf
    on~$\Gg_m\times\{y_0\}$, of generic rank bounded by
    $\max(\deg(f_1),\deg(f_2))$.
  \item If~$f\in k(x)$, then for~$n\geq 1$, the Fourier transform
    on~$G(k_n)$ of the trace function of~$M_f$ is given by
    $$
    (\chi,b)\mapsto |k_n|\chi(f(b))
    $$
    for~$\chag{\chi}{b}\in \widehat{G}(k_n)$ in a generic set.
  \end{enumth}
\end{lemma}

\begin{proof}
  The curve $\widetilde{C}$ contains $C$, and the assumption that $f_1$
  and $f_2$ are coprime implies that~$\widetilde{C}$ is smooth (since
  the partial derivative with respect to~$x$ is $f_2(a)$, which is
  non-zero on~$\widetilde{C}$). It is irreducible since it is isomorphic
  to the open subset of~$\Gg_a$ where $f_1f_2$ is non-zero by the
  projection $(x,a)\mapsto a$ with inverse $a\mapsto
  f_1(a)/f_2(a)$. Since $i_*=i_!$ for a closed immersion, it follows
  that $i_*\Qlb[1]=i_!\Qlb[1]$ is a perverse sheaf, and since it
  restricts to $\overline{\Qq}_{\ell,f}$ on $V$, these perverse sheaves
  are isomorphic.
  
  We next show that for generic $\chag{\chi}{b}\in\charg{G}$, we have
  $$
  \dim\rmH^0_c(G_{\bar{k}}, (M_f)_{\chag{\chi}{b}})=1,
  $$
  which will prove~(2).
  \par
  This cohomology group can be computed by first taking the relative
  additive Fourier transform~$F$ of $M_f$, restricting it to the line
  $\Gg_m\times\set{b}$, then taking the cohomology
  of~$F\otimes\mcL_{\chi}$ on $\Gg_m\times\set{b}$. Since the Fourier
  transform~$F$ of $M_f$ is~$\overline{\Qq}_{\ell,f}$, there exists a
  dense open set~$U$ of~$\Gg_a$ such that for $b\in U$, the
  restriction of~$F$ to~$\Gg_m\times \{b\}$ is a rank one skyscraper
  sheaf supported on $f(b)$. Such a sheaf, tensored with any character
  $\mcL_\chi$, has its $0$-th cohomology group of dimension~$1$.

  In fact, the same argument shows that if~$\chag{\chi}{b}$ is defined
  over~$k_n$, then the action of Frobenius on the one-dimensional
  space~$\rmH^0_c(G_{\bar{k}}, (M_f)_{\chag{\chi}{b}})$ is
  $|k_n|\chi(f(b))$, which proves the last statement.

  To prove~(3), we observe that, by definition of the Fourier transform,
  yet another description of~$M_f$ is
  $$
  M_f=R\phi_*\mcL_{\psi(-ay)}[2],
  $$
  where $\phi$ is the restriction of the projection $(x,y,a)\mapsto
  (x,y)$ to the subvariety
  $$
  Z=\{(x,y,a)\in\Gg_m\times \Gg_a\times\Gg_a\,\mid\, f_1(a)=xf_2(a)\}
  $$
  of $\Gg_m\times\Gg_a^2$. Since $\phi$ is an affine quasi-finite
  morphism, we obtain~(3) with
  $$
  \mcG=\widetilde{\phi}_*\mcL_{\psi(-ay_0)},
  $$
  where $\widetilde{\phi}$ is the restriction of $\phi$ to
  $Z_{y_0}$. This sheaf has generic rank bounded by the size of the
  fibers of~$\phi$, and is $\leq \max(\deg(f_1),\deg(f_2))$.
\end{proof}

\begin{remark}
  The last statement amounts to the following computation of Fourier
  transform on~$G(k_n)$: by the first part, writing~$f=f_1/f_2$,
  where~$f_i$ are polynomials without common factor, the perverse
  sheaf~$\overline{\Qq}_{\ell,f}$ is the shifted constant sheaf on the
  smooth irreducible curve defined by $f_1(a)=xf_2(a)$ in~$G$.
  Therefore, the trace function of $\overline{\Qq}_{\ell,f}$ at
  $(x,a)\in G(k_n)$ is equal to $1$ if $f_1(a)=xf_2(a)$, and~$0$
  otherwise, so the trace function of $M_f$ has the value
  $$
  \sum_{\substack{a\in k\\f_1(a)=xf_2(a)}}\psi_n(-ay)
  $$
  at $(x,y)\in k_n^{\times}\times k_n$. For $(\chi,b)\in\charg{G}(k_n)$,
  we get the Fourier transform
  $$
  \sum_{(x,y)\in G(k_n)} \chi(x)\psi_n(by) \sum_{\substack{a\in
      k_n\\f_1(a)=xf_2(a)}}\psi(-ay) =\sum_{\substack{a\in
      k_n\\f_1(a)f_2(a)\not=0}} \sum_{y\in k_n} \chi(f(a))\psi_n((b-a)y)
  $$
  and this is equal to $|k_n|\chi(f(b))$ if~$f_1(b)f_2(b)\not=0$,
  and~$0$ otherwise.
\end{remark}




The basic classes of objects of tannakian dimension~$1$ we have just
described turn out to be sufficient to obtain all of them.

\begin{proposition}
  \label{prop-rk1}
  Let~$M$ be a simple perverse sheaf on~$G$ over~$\bar{k}$. Assume
  that~$M$ is in~$\Ppint(G)$.  Then~$M$ has tannakian dimension one if
  and only if there exist a rational function $f$, a simple
  hypergeometric complex~$H$ on~$\Gg_m$ and a perverse sheaf~$N$
  on~$\Gg_a$ with Fourier transform of rank one such that~$M$ is
  isomorphic to the convolution
  $$
  M_f*_{\mathrm{int}} (H\boxtimes N).
  $$
\end{proposition}

\begin{proof}
  Since the tannakian dimension is multiplicative in convolutions, the
  ``if'' assertion follows from Lemma~\ref{lm-mf} and the fact
  that~$H\boxtimes N$ has tannakian dimension~$1$ by an application of
  Proposition~\ref{pr-product}.
  \par
  Conversely, let $M$ be a simple perverse sheaf on $G$ of tannakian
  dimension one. As in the dimension zero case, we consider the
  shifted Fourier transform $F=\ft_{\psi/\Gg_m}(M)[1]$ of $M$ relative
  to~$\Gg_m$. For generic $a\in \Gg_a$, the object
  $F_a=F\vert\Gg_m\times\set{a}$ on $\Gg_m\times\{a\}$ is perverse of
  generic rank one, as in the beginning of the proof of
  Proposition~\ref{prop-rk0}.

  In particular, for the generic point $\eta$ of $\Gg_a$, the object
  $F_{\eta}$, viewed as a perverse sheaf on $\Gg_{m,k(\eta)}$, is of
  tannakian dimension one. By Theorem~\ref{th-hypergeometric}\,(2), it
  is isomorphic to a hypergeometric complex multiplicatively translated
  by a non-zero rational function $f(\eta)$ of $\eta$, and tensored by a
  rank one object on $k(\eta)$.  Thus, there exists a dense open subset
  $W$ of $\Gg_m\times\Gg_a$ and an isomorphism
  \begin{equation}\label{eq-zero-iso}
    F\vert W\simeq p_2^*(N)\vert W\otimes m_f^*(H),
  \end{equation}
  where $N$ is a perverse sheaf on $\Gg_a$ of generic rank one, $H$ is a
  hypergeometric complex on $\Gg_m$ and $m_f\colon W\to \Gg_m$ is the
  morphism
  \[
    (x,a)\mapsto x/f(a)
  \]
  (in particular, $x/f(a)$ is defined and non-zero for~$(x,a)\in
  W$). Using~\cite[Th.\,8.4.10]{esde} and Proposition~\ref{prop-rk0}, we
  may assume that~$H$ is simple.

  Let~$\widehat{N}$ be the Fourier transform of~$N$. We claim that there
  are isomorphisms
  \begin{equation}\label{eq-final-iso}
    M\simeq (H\boxtimes \un)*_{\intt} M_{f}*_{\intt} (\un\boxtimes
    \widehat{N})\simeq M_{f}*_{\intt} (H\boxtimes \widehat{N}),
  \end{equation}
  which will conclude the proof. The second isomorphism follows from
  commutativity and associativity of the convolution combined with the
  isomorphism of Proposition~\ref{pr-product}, hence we need only check
  the first.

  Let~$P=(H\boxtimes \un)*_{\intt} M_{f}$; we need to show that $M$ is
  isomorphic to $P*_{\intt} (\un\boxtimes \widehat{N})$. We will do this
  by showing that the restriction to~$W$ of their relative Fourier
  transforms are isomorphic; since both objects involved are perverse
  sheaves, this will give the result.
  
  Precisely, denote
  $$
  P_!= (H\boxtimes \un)*_{!} M_{f},\quad\quad P_*= (H\boxtimes
  \un)*_{*} M_{f}.
  $$

  We claim that there are isomorphisms
  \begin{gather}
    \label{eq-p-iso}
    P_!\simeq P\simeq P_*,\\
    \label{eq-first-iso}
    \ft_{\psi/\Gg_m}(P)|W\simeq m_f^*(H),
    \\
    \ft_{\psi/\Gg_m}(P*_{\intt} (\un\boxtimes\widehat{N}))\simeq
    p_2^*(N)\otimes \ft_{\psi/\Gg_m}(P),
    \label{eq-second-iso}
  \end{gather}
  where~$M_{f}$ is the complex in Lemma~\ref{lm-mf}, and~$\un$ denotes
  the unit object on~$\Gg_a$.

  Assuming these to be true, it follows by
  combining~(\ref{eq-zero-iso}), (\ref{eq-first-iso})
  and~(\ref{eq-second-iso}) that
  $$
  \ft_{\psi/\Gg_m}(M)|W\simeq p_2^*(N)|W\otimes m_f^*(H) \simeq
  p_2^*(N)|W\otimes \ft_{\psi/\Gg_m}(P)|W \simeq \ft_{\psi/\Gg_m}(P*_!
  (\un\boxtimes\widehat{N}))|W,
  $$
  proving the first part of~(\ref{eq-final-iso}), and thereby concluding
  the proof.
  

  We will begin with the rigorous sheaf-theoretic computations, but we
  include afterwards the (potentially more enlightening) computations of
  trace functions (assuming all objects to be defined over~$k$).

  \textbf{Proof of~(\ref{eq-p-iso}).}  Since $G$ is affine, it suffices 
  to prove that $P_!$ is semiperverse: indeed, it first follows that
  $P_!$ is perverse by~\cite[Lemma\,2.6.7]{katz-rls}; since the dual of
  $P_*$ is $(\dual(H)\boxtimes \un)*_! \dual(M_f)$, this first fact
  (applied to the duals of $H$ and $M_f$, which are of the same type)
  also implies that $P_*$ is perverse, and~(\ref{eq-p-iso}) follows.

  To prove that $P_!$ is semiperverse, we need to estimate the dimension
  of the support of the cohomology sheaves $\mcH^i(P_!)$. Let $i\in\Zz$
  and $(x_0,y_0)\in G$. By the definition of convolution and the proper
  base change theorem, we have an isomorphism
  $$
  \mathcal{H}^i(P_!)_{(x_0,y_0)}\simeq H^i_c(G_{\bar{k}},[(x,y)\mapsto
  (x_0x,y_0+y)]^*(H\boxtimes \un)\otimes [(x,y)\mapsto (x^{-1},-y)]^*M_f).
  $$

  Let $H_{x_0}=[x\mapsto x_0x]^*H$. The last cohomology group is
  isomorphic to
  $$
  H^i_c(G_{\bar{k}},p_1^*H_{x_0}\otimes p_2^*[y\mapsto
  y+y_0]^*\un\otimes [(x,y)\mapsto (x^{-1},-y)]^*M_f).
  $$

  If we denote by $\widetilde{M}_{f,y_0}$ the complex on $\Gg_m$ given
  by $[x\mapsto (x^{-1},-y_0)]^*M_f$, then the projection formula shows
  that the cohomology group is isomorphic to
  $$
  H^i_c(\Gg_{m,\bar{k}},H_{x_0}\otimes \widetilde{M}_{f,y_0}).
  $$

  By Lemma~\ref{lm-mf}\,(3) the object $\widetilde{M}_{f,y_0}$ is of the
  form $\mcG_{y_0}[2]$ for some sheaf on~$\Gg_m$.


  \begin{enumerate}
  \item If $H_{x_0}$ is of the form $\mcF_{x_0}[1]$ for some simple
    middle extension hypergeometric sheaf on~$\Gg_m$, then
    $\mcH^i(P_!)_{(x_0,y_0)}$ is isomorphic to
    $$
    H^{3+i}_c(\Gg_{m,\bar{k}},\mcF_{x_0}\otimes \mcG_{y_0}).
    $$

    It follows immediately that $\mcH^i(P_!)$ is zero for $i\geq 0$.

    For $i=-1$, since $\mcF_{x_0}$ and $\mcG_{y_0}$ are lisse on dense
    open subsets of $\Gg_m$ and $\mcF_{x_0}$ is simple, the stalk at
    $(x_0,y_0)$ is non-trivial if and only if $\mcF_{x_0}$ is a
    Jordan--Hölder factor of the dual of $\mcG_{y_0}$ (on such an open
    set).

    But, for a given~$y_0$, the sheaf $\mcG_{y_0}$ has generic rank
    $\leq \max(\deg(f_1),\deg(f_2))$ according to
    Lemma~\ref{lm-mf}\,(3), hence has at most as many Jordan--Hölder
    factors. The hypergeometric sheaves $\mcF_{x_0}$ are pairwise
    non-isomorphic by~\cite[8.5.6]{esde}, since they are all
    multiplicative translates of a fixed hypergeometric sheaf; for a
    given $y_0$, there are therefore at most $\max(\deg(f_1),\deg(f_2))$
    values of $x_0$ for which the stalk of $\mcH^1(P_!)$ at $(x_0,y_0)$
    is non-zero. This implies that the support of $\mcH^1(P_!)$ is of
    dimension $\leq 1$, so that $P_!$ is semiperverse.


  \item If~$H$ is isomorphic to~$\delta_{x_1}$ for some $x_1\in\Gg_m$,
    then $H_{x_0}\simeq \delta_{x_0x_1}$. Then $\mcH^i(P_!)_{(x_0,y_0)}$
    is isomorphic to
    $$
    H^{2+i}_c(\Gg_{m,\bar{k}},\delta_{x_0x_1}\otimes \mcG_{y_0}),
    $$
    which is zero if $i\not=-2$. Thus $P_!$ is also semiperverse in that
    case.
  \end{enumerate}
  
  \textbf{Proof of~(\ref{eq-first-iso}).}  We compute
  $\ft_{\psi/\Gg_m}(P_!)$, which is the same as $\ft_{\psi/\Gg_m}(P)$ by
  the previous result. Consider
  $$
  X=\{(x,u,v,y,c,d,a)\in \Gg_m^3\times \Gg_a^4\,\mid\, uv=x,\ c+d=y\},
  $$
  and the morphism $\sigma_{x,a}\colon (x,u,v,y,c,d,a)\mapsto (x,a)$
  from~$X$ to~$G$. Then
  $$
  \ft_{\psi/\Gg_m}(P_!)= R\sigma_{x,a,!}(\mcL_{\psi(ay)}\otimes
  H(u)\otimes \un(c)\otimes M_f(v,d)),
  $$
  where we use a shorthand notation for pullbacks, where, e.g.
  $$
  H(u)=[(x,u,v,c,d,y)\mapsto u]^*H.
  $$

  Denoting
  $$
  Y=\{(x,u,v,y,c,d,a)\in X\,\mid\, c=0\},\quad\quad
  Z=\{(x,u,v,y,a)\in\Gg_m^3\times \Gg_a^2\,\mid\, uv=x\},
  $$
  and noting the isomorphism $Y\to Z$ given by
  $$
  (x,u,v,y,c,d,a)\mapsto (x,u,v,y,a)
  $$
  with inverse
  $$
  (x,u,v,y,a)\mapsto (u,x,v,y,0,y,a),$$
  this becomes
  $$
  R\sigma_{x,a,!}(\mcL_{\psi(ay)}\otimes H(u)\otimes M_f(v,y)),
  $$
  with an abuse of notation involving in using the
  notation~$\sigma_{x,a}$ again for the projection from points on~$Z$ to
  $(x,a)$. Factoring $\sigma_{x,a}$ into
  $(x,u,v,y,a)\mapsto (x,u,v,a)\mapsto (x,a)$, we recognize
  $$
  R\sigma_{x,a,!}(H(u)\otimes \ft_{\psi/\Gg_m}(M_f)(v,a))=
  R\sigma_{x,a,!}(H(u)\otimes \overline{\Qq}_{\ell,f}(v,a)),
  $$
  again with some abuse of notation.

  We write $f=f_1/f_2$ as in Lemma~\ref{lm-mf}.  Let
  $$
  Z_f=\{(x,u,v,a)\,\mid\, uv=x,\ vf_2(a)=f_1(a)\}.
  $$

  By Lemma~\ref{lm-mf}, we have an isomorphism
  $$
  R\sigma_{x,a,!}(H(u)\otimes \overline{\Qq}_{\ell,f}(v,a))\simeq
  R\sigma_{x,a,!}(H(u)),
  $$
  and since
  $\sigma_{x,a}\colon \sigma_{x,a}^{-1}(W)\to W\subset \Gg_m\times\Gg_a$
  is an isomorphism with inverse
  $$
  (x,a)\mapsto (x,x/f(a),f(a),a),
  $$
  this is simply $H(x/f(a))=m_f^*(H)$ on~$W$. In other words, we have
  proved that there is an isomorphism
  $$
  \ft_{\psi/\Gg_m}(P_!)|W\simeq m_f^*(H),
  $$
  as desired.


  
  (Formally, we can also show that the trace functions on both sides
  of~(\ref{eq-first-iso}) coincide, disregarding the difference between
  the various convolutions.  For~$(x,a)\in k^{\times}\times k$, the
  object $\ft_{\psi/\Gg_m}(P)$ has trace function at~$(x,a)$ equal to
  $$
  \sum_{y\in k} \sum_{u\in k^{\times}}
  t_H(u)t_{M_{f}}(x/u,y)\psi(ay)=\sum_{u\in k^{\times}}t_H(u)\sum_{y\in
    k}t_{M_{f}}(x/u,y)\psi(ay).
  $$

  The inner sum over~$y$ is the value at~$(x/u,a)$ of the trace
  function of the relative Fourier transform of~$M_{f}$, hence by
  definition it is the trace function at~$(x/u,a)$ of the
  object~$\overline{\Qq}_{\ell,f}$. Writing $f=f_1/f_2$ where~$f_1$
  and~$f_2$ are polynomials without common zeros, this value is~$1$ if
  $f_1(a)=xu^{-1}f_2(a)$, and~$0$ otherwise. Thus,
  provided~$f_1(a)f_2(a)\not=0$, the above expression is
  $$ 
  \sum_{u\in k^{\times}}t_H(u)\sum_{y\in k}t_{M_{f}}(x/u,y)\psi(ay)=
  \sum_{\substack{u\in
      k^{\times}\\f_1(a)=xu^{-1}f_2(a)}}t_H(u)=t_H(x/f(a)).
  $$

  If~$(x,a)\in W$, this is the same as the trace function of~$m_f^*(H)$,
  as desired.)

  
  \textbf{Proof of~(\ref{eq-second-iso}).} The argument is
  similar. Consider
  $$
  X=\{(x,u,v,y,c,d,a)\in \Gg_m^3\times \Gg_a^4\,\mid\, uv=x,\ c+d=y\},
  $$
  and the morphism $\sigma_{x,a}\colon (x,u,v,y,c,d,a)\mapsto (x,a)$
  from~$X$ to~$G$. Then
  $$
  \ft_{\psi/\Gg_m}(P_!*_!  (\un\boxtimes\widehat{N}))\simeq
  R\sigma_{x,a,!}(\mcL_{\psi(ay)}\otimes P_!(u,c)\otimes \un(v)\otimes
  \widehat{N}(d)),
  $$
  where we use the shorthand notation for pullbacks. Denoting
  $$
  Y=\{(x,u,v,y,c,d,a)\in X\,\mid\, u=x,\ v=1,\ c+d=y\}\simeq
  Z=\{(x,c,d,y,a)\,\mid\, c+d=y\},
  $$
  this is
  $$
  R\sigma_{x,a,!}(\mcL_{\psi(ay)}\otimes P_!(x,c)\otimes \widehat{N}(d))
  $$
  (with again an abuse of notation involved in the
  notation~$\sigma_{x,a}$). From the isomorphism
  $$
  Z\simeq \{(x,c,d,a)\in\Gg_m\times \Gg_a^3\}
  $$
  with inverse $(x,c,d,a)\mapsto (x,c,d,c+d,a)$, we get
  $$
  R\sigma_{x,a,!}(\mcL_{\psi(ac)}\otimes \mcL_{\psi(ad)} \otimes
  P_!(x,c)\otimes \widehat{N}(d)).
  $$

  Factor $\sigma_{x,a}$ as $(x,c,d,a)\mapsto (x,c,a)\mapsto (x,a)$; by
  Fourier inversion, we obtain
  \begin{equation}\label{eq-dual-iso-1}
    \ft_{\psi/\Gg_m}(P_!*_!  (\un\boxtimes\widehat{N}))\simeq
    R\sigma_{x,a,!}(P_!(x,c)\otimes N(a)\otimes
    \mcL_{\psi(ac)})=N(a)\otimes \ft_{\psi/\Gg_m}(P_!).
  \end{equation}


   We next claim that we have an isomorphism
  \begin{equation}\label{eq-dual-iso-2}
    \ft_{\psi/\Gg_m,*}(P_**_* (\un\boxtimes\widehat{N}))\simeq
    N(a)\otimes \ft_{\psi/\Gg_m,*}(P_*).
  \end{equation}
  
  Indeed, we can compute the dual of both sides using the fact that
  duality transforms the $*$ convolution in the~$!$ convolution, and
  that both types of Fourier transforms are isomorphic; and the
  isomorphisms~(\ref{eq-p-iso}) and~(\ref{eq-dual-iso-1}) applied to
  the dual of~$M$ (with~$H$ and~$N$ therefore replaced by their
  respective duals) leads to~(\ref{eq-dual-iso-2}).
  
  Again because $P_!\simeq P_*\simeq P$ and the two Fourier transforms
  coincide, we conclude that
  $$
  M\simeq P*_!  (\un\boxtimes\widehat{N})\simeq P*_*
  (\un\boxtimes\widehat{N}) \simeq P *_{\intt}
  (\un\boxtimes\widehat{N}),
  $$
  which establishes~(\ref{eq-second-iso}).

  (Here also we illustrate the result by computing trace functions.  The
  value at~$(x,a)$ of the trace function of the relative Fourier
  transform of $P*_{\intt} (\un\boxtimes \widehat{N})$ is equal to
  $$
  \sum_{y\in k}\psi(ay)\Bigl(\sum_{v\in
    k}t_P(x,v)t_{\widehat{N}}(y-v)\Bigr)=
  \sum_{v\in k}t_P(x,v)\sum_{w\in k}\psi(a(v+w))t_{\widehat{N}}(w).
  $$

  By Fourier inversion, this is the same as
  $$
  t_N(a)\sum_{v\in k}t_P(x,v)\psi(av),
  $$
  which is the value at~$(x,a)$ of the trace function
  of~$p_2^*(N)\otimes \ft_{\psi/\Gg_m}(P)$. By~(\ref{eq-first-iso}),
  this coincides on~$W$ with the trace function
  of~$p_2^*(N)|W\otimes m_f^*(H)$.)
\end{proof}

\begin{remark}
  The trace functions (over~$k$) of simple negligible objects are of the
  form
  $$
  (x,y)\mapsto t(x)\psi(by)
  $$
  for some trace function~$t$ on~$\Gg_m$ and some $b\in k$, or
  $$
  (x,y)\mapsto \chi(x)t(y)
  $$
  for some trace function~$t$ on~$\Gg_a$ and some multiplicative
  character $\chi$.  The trace functions of simple objects of
  tannakian dimension one are convolutions of functions of the three
  types
  $$
  (x,y)\mapsto \sum_{f(z)=x}
  \psi(-yz),\quad\quad (x,y)\mapsto \mathscr{H}(x),\quad\quad
  (x,y)\mapsto t(y),
  $$
  where $f$ is a non-zero rational function, $\mathscr{H}$ is the trace
  function of a hypergeometric sheaf and $t$ is the trace function of an
  object on~$\Gg_a$ whose Fourier transform has generic rank one.  The
  associated exponential sums are (up to normalization by powers
  of~$|k|$) of the form
  $$
  S(\chi,a)= \chi(f(a))\widehat{\mathscr{H}}(\chi)\widehat{t}(a),
  $$
  where $\widehat{t}$ is the trace function of an $\ell$-adic character,
  and $\widehat{\mathscr{H}}$ is a product of monomials in Gauss sums
  (see~(\ref{eq-mellin-hypergeom})).
\end{remark}




\chapter{Variance of arithmetic functions in arithmetic
  progressions}\label{sec-variance}

\section{Introduction}

In this chapter, we will consider some of the first natural concrete
applications of our results to problems which, as stated, do not seem to
refer to algebraic groups, or equidistribution statements of any kind.
These problems are related to one of the most essential questions of
modern analytic number theory, namely the study of arithmetic functions
in arithmetic progressions to large moduli.

Concretely, this means that we are given an arithmetic function~$f$
(i.e., a complex-valued function defined on the set of positive
integers), an integer $q\geq 1$ (the ``modulus'') and $x\geq 2$, and we seek
to understand the quantities
$$
\sum_{\substack{n\leq x\\n\equiv a\mods{q}}}f(n)
$$
for $a$ varying among residue classes modulo~$q$, or only for $a$
coprime to~$q$. The focus is on these sums in settings where both $x$
and~$q$ are large, and the goal is often to obtain asymptotic formulas
valid for $q$ as large as possible in comparison with~$x$.

The literature on this topic is enormous, and the applications cover
almost all of analytic number theory: indeed, this subject encompasses,
almost by definition, all of sieve theory and its applications
(see~\cite{cribro}), and it is in particular at the source of most of
the recent developments in prime number theory, going back to the
Bombieri--Vinogradov Theorem (see, e.g.,~\cite[Ch.\,17]{ant}), and
including such celebrated results as the Green--Tao Theorem, or Zhang's
Theorem~\cite{zhang}, or the Maynard--Tao method (see,
e.g.,~\cite{bourbaki-gaps}).

The problems that we consider here are the analogue for polynomials
over finite fields, and in the limit when the size of the field tends
to infinity, of questions related to the distribution of the
quantities above, and especially of their \emph{variance}, as
functions of~$a$. In other words, we are interested~in
$$
\sum_{a\mods{q}} \Bigl|\sum_{\substack{n\leq x\\n\equiv
    a\mods{q}}}f(n)-\frac{1}{q}\sum_{n\leq x}f(n)\Bigr|^2
$$
or (often more naturally for applications) the variant
\begin{equation}\label{eq-general-variance}
  \sum_{\substack{a\mods{q}\\(a,q)=1}} \Bigl|\sum_{\substack{n\leq x\\n\equiv
      a\mods{q}}}f(n)-\frac{1}{\varphi(q)}\sum_{n\leq x}f(n)\Bigr|^2
\end{equation}
where the sum covers only invertible residue classes.  (In both cases,
the choice of ``expected main term'' is natural, but might require
adjustments, depending on the arithmetic function involved.)

The serious study of these function field analogues has been initiated
especially by Keating and Rudnick and a number of collaborators (see for
instance~\cite{kr1}, in the case where $f$ is the von Mangoldt function,
using results of Katz~\cite{katz-characters}, which themselves relied on
his work on the Mellin transform over finite fields~\cite{mellin}).

It is quite easy to understand the link between a quantity
like~(\ref{eq-general-variance}), in the function field case, and
equidistribution problems of the type considered in
Chapter~\ref{sec:equidis}. Indeed, we are then in the situation
where~$q$ is a polynomial in $k[t]$ for some finite field~$k$, and the
sum over $n\leq x$ is replaced by the sum over monic
polynomials~$g\in k[t]$ of degree~$m$. Then for any complex-valued
function~$f$ defined for polynomials in~$k[t]$, we see (using
orthogonality of characters, or the discrete Plancherel formula) that
the formula
\begin{multline*}
  \sum_{a\in (k[t]/qk[t])^{\times}}
  \Bigl|\sum_{\substack{\deg(g)=m\\g\equiv
      a\mods{q}}}f(g)-\frac{1}{|(k[t]/qk[t])^{\times}}
  \sum_{\deg(g)=m}f(g)\Bigr|^2 =
  \\
  |(k[t]/qk[t])^{\times}|
  \quad\sum_{\chi\not=1} \ \Bigl| \sum_{\deg(g)=m}
  \chi(g)f(g) \Bigr|^2
\end{multline*}
holds, where~$\chi$ runs over non-trivial characters of the group
$(k[t]/qk[t])^{\times}$. These characters can be identified with the
characters of~$G(k)$ for some commutative algebraic group~$G$ (by a
simple special case of geometric class-field theory; in the case which
we will consider, when~$q$ is squarefree, it will be a very explicit
torus). Moreover, for many natural arithmetic functions, the inner sum
over~$g$ monic of degree~$m$ in~$k[t]$ can be interpreted as the value
at~$\chi$ of the arithmetic Fourier transform of some object on this
group~$G$. In the limit where $k$ is replaced by its extensions~$k_n$ of
degree $n\to+\infty$ (and $m$ is fixed), we can therefore expect to
determine the asymptotic behavior of this variance from our
equidistribution theorems.

We will now consider in detail the version of this question when $f$ is
the von Mangoldt function associated to a higher-degree $L$-function
(the classical von Mangoldt function being related to the Riemann zeta
function, which has degree~$1$), in which case Hall, Keating, and
Roddity-Gershon~\cite{variance} have shown that new phenomena appear
(again relying on~\cite{mellin}). These are conjectured to correspond to
new behavior also in the (currently inaccessible) situation over number
fields.  We refer the reader to the introductions of both
papers~\cite{kr1} and~\cite{variance} for extensive discussions of these
motivating conjectures, and for additional references to other papers.

We will see that, as suggested by the discussion above, the
equidistribution theory for arithmetic Fourier transforms on
higher-dimensional tori leads to generalizations, strengthenings, and
better understanding, of these previous results. This leads in
particular to Theorem~\ref{th-variance} in the Introduction, but the
method is suitable for the proof of many similar statements.


In the remainder of this chapter, as before, we denote by~$k$ a finite
field, with an algebraic closure~$\bar{k}$, and for each $n \geq 1$
by~$k_n$ the extension of degree~$n$ of~$k$ in~$\bar{k}$. We fix a prime
$\ell$ distinct from the characteristic of~$k$, and all complexes are
understood to be $\ell$-adic complexes.




  

\section{Equidistribution on tori associated to polynomials}

In what follows, we fix a square-free monic
polynomial~\hbox{$f\in k[t]$} of degree~\hbox{$d\geq 2$}. We denote
by~$B$ the (étale) $k$-algebra $B=k[t]/fk[t]$ of degree~$d$ over~$k$ (in
spite of the notation, $B$ depends on~$f$), by~$Z$ the zero locus
of~$f$, and by $\Aa^1_k[1/f]$ the complement of~$Z$ in the affine line
over~$k$.

We begin with a result of Katz~\cite{katz-characters}.

\begin{proposition}[Katz]\label{pr-katz}
  The functor $ A\mapsto (B\otimes_k A)^{\times}$ on $k$-algebras is
  represented by a torus~$T$ defined over~$k$. This torus splits over any
  extension of~$k$ where~$f$ splits in linear factors.
  \par
  Moreover, the map $x\mapsto t-x$ defines a closed
  immersion
  $$
  i_f\colon \Aa^1_k[1/f]\longrightarrow T,
  $$
  and there exists a morphism of algebraic groups
  $$
  p\colon T\longrightarrow \Gg_m
  $$ satisfying $p\circ i_f=(-1)^{\deg(f)}f$, where
  we view $f$ as defining a morphism \hbox{$\Aa^1[1/f]\to \Gg_m$}.
\end{proposition}

\begin{remark}
  As noted by Katz~\cite[p.\,3224]{katz-characters}, the torus $T$ is
  isomorphic to a generalized jacobian associated to $\Pp^1$ with
  divisor $(\infty)+Z$ (compare Remark~\ref{rm-gener-jac}).
\end{remark}

We call the morphism $p$ the \emph{norm}. If~$f$ splits completely
over~$k$, say
$$
f=\prod_{z\in Z}(t-z),
$$
then the torus $T$ is split by the morphism sending $g$ to
$(g(z))_{z\in Z}$. The norm is then given by
$$
p(g)=\prod_{z\in Z}g(z), 
$$
and in particular one has
$$
p(i_f(x))=\prod_{z\in Z}(z-x)=(-1)^{\deg(f)}f(x). 
$$

We denote by $\widehat{B}^{\times}$ (\resp by $\widehat{k}^{\times}$) the
group of $\ell$-adic characters of the finite group $B^{\times}$
(\resp of~$k^{\times}$).  We extend characters of~$B^{\times}$ to~$k[t]$
by putting $\chi(g)=0$ if $g$ is not coprime to~$f$.
Since $B^{\times}=T(k)$, the
group~$\widehat{B}^{\times}$ of characters of~$B^{\times}$ is also equal
to the group $\charg{T}(k)$ of characters of~$T(k)$ (although we will
sometimes distinguish them to avoid confusion between characters of $B$,
operating on polynomials, and characters of~$T$).

If $f$ splits over~$k$ as above, then the Chinese
Remainder Theorem induces an isomorphism
$(\widehat{k}^{\times})^Z\to \widehat{B}^{\times}$, under which an element $(\chi_z)_{z\in Z} 
\in (\widehat{k}^{\times})^Z$ corresponds to the character~$\chi$
of~$B^{\times}$ that maps $g\in k[t]$ to 
$$
\chi(g)=\prod_{z\in Z}\chi_z(g(z)). 
$$

Let~$M$ be a perverse sheaf on~$\Aa^1_k[1/f]$ which is pure of weight
zero.  We are interested in the distribution properties of families of
one-variable exponential sums of the type
\begin{equation}\label{eq-varsum}
  \sum_{x\in k\setminus Z}t_{M}(x)\chi(t-x)
\end{equation}
for $\chi\in \widehat{B}^{\times}$, or of the underlying $L$-functions
(recall that~$t$ is an indeterminate).

We start by interpreting these sums as Mellin transforms on~$T$ in order
to apply our general equidistribution results.
Let~$\chi\in \widehat{B}^{\times}$.  Let~$\widetilde{\chi}$ be the
character of~$T(k)$ corresponding to~$\chi$. The sum~(\ref{eq-varsum})
takes the form
\begin{equation}\label{eq-sum-variance}
  \sum_{x\in k}t_{M}(x)\chi(t-x) = 
  \sum_{x\in \Aa^1[1/f](k)}t_M(x)\widetilde{\chi}(i_f(x))= 
  \sum_{y\in T(k)}t_{i_{f*}M}(y)\widetilde{\chi}(y).
\end{equation}

Note also that by adapting the argument
of~\cite[Lem.\,1.1]{katz-characters}, for any $n\geq 1$, we have
\begin{equation}\label{eq-katz-1}
  \sum_{x\in k_n}t_{M}(x;k_n)\chi(N_{k_n/k}(t-x))=
  \sum_{y\in T(k_n)}t_{i_{f*}M}(y;k_n)\widetilde{\chi}(N_{k_n/k}(y)).
\end{equation}

The variation with~$\chi\in \widehat{B}^{\times}$ of the
sums~(\ref{eq-varsum}) is therefore governed by the tannakian group of
the perverse sheaf $i_{f*}M$ on~$T$. By Theorem~\ref{th-polite}, this
perverse sheaf is generically unramified.


We first compute the tannakian dimension of the object $i_{f*}M$, in the
most important cases.

\begin{lemma}\label{lm-rank-var}
  Let~$\mcF$ be a middle extension sheaf on~$\Aa^1_k[1/f]$ which is pure
  of weight zero.\footnote{\ Recall (see Example~\ref{ex-weights} (3))
    that this means that the restriction of~$\mcF$ to any dense open set
    where it is lisse is punctually pure of weight~$0$.} Define
  $M=\mcF[1](1/2)$, which is a perverse sheaf of weight zero
  on~$\Aa^1_k[1/f]$. The tannakian dimension~$r$ of~$i_{f*}M$ is given
  by
  $$
  r=(\deg(f)-1)\rank(\mcF)+\sum_{x\in\Pp^1(\bar{k})}
  \swan_x(\mcF)+\sum_{x\in \bar{k}}\Drop_x(\mcF)\geq
  (\deg(f)-1)\rank(\mcF).
  $$
\end{lemma}

\begin{proof}
  The object~$M$ is a perverse sheaf and so is $i_{f*}M$ because $i_f$
  is a closed immersion (see Corollary~\ref{cor-closed-immersion}).  The
  tannakian dimension is the Euler--Poincaré characteristic
  $\chi_c(T_{\bar{k}},(i_{f*}M)_{\chi})$ for a generic
  character~$\chi\in\charg{T}$ (Proposition~\ref{prop-dimPerv=dimVect}).

  For any integer~$i$, we have natural isomorphisms
  $$
  H^i_c(T_{\bar{k}},(i_{f*}M)_{\chi})\simeq H^i_c(\Aa^1[1/f]_{\bar{k}},
  M\otimes i_f^*\mcL_{\chi})\simeq H^i_c(\Aa^1[1/f]_{\bar{k}},
  \mcF[1]\otimes i_f^*\mcL_{\chi}).
  $$
  
  As explained in~\cite[p.\,3227]{katz-characters}, the pullback
  $i_f^*\mcL_{\chi}$ is geometrically isomorphic to the tensor product
  $$
  \mcL=\bigotimes_{z\in Z} \mcL_{\chi_z(z-x)}
  $$
  where~$x$ is the coordinate on~$\Aa^1[1/f]$ and $\chi$ corresponds to
  the tuple $(\chi_z)$ of characters of $k^{\times}$ as above.
  
  Now using the Euler--Poincaré formula on a curve (see
  Theorem~\ref{th-euler-poincare}), we obtain
  \begin{multline*}
    r=
    -\chi_c(\Aa^1[1/f]_{\bar{k}},\mcF\otimes\mcL)=
    -\rank(\mcF)\chi_c(\Aa^1[1/f]_{\bar{k}})
    \\
    +\sum_{x\in\Pp^1}\swan_x(\mcF\otimes\mcL)
    +\sum_{x\in\Aa^1[1/f]}\Drop_x(\mcF\otimes\mcL).
  \end{multline*}

  The first term is equal to $\rank(f)(\deg(f)-1)$ since~$f$ is
  square-free, and the second is the sum of Swan conductors of~$\mcF$,
  since the sheaf~$\mcL$ is everywhere tame. The third is the sum of the
  drops of $\mcF$ on~$\Aa^1[1/f]$, since $\mcL$ is lisse
  on~$\Aa^1[1/f]$.
\end{proof}

We now apply Larsen's Alternative to compute the tannakian group of
such perverse sheaves.

\begin{proposition}
  Let~$\mcF$ be a middle extension sheaf on~$\Aa^1_k[1/f]$ which is pure
  of weight zero and irreducible of rank at least~$2$. Let
  $M=\mcF[1](1/2)$. Assume that $M$ is not geometrically isomorphic to
  $i_f^*\mcL_{\chi}[1]$ for some character~$\chi$ of~$G$.

  Then $i_{f*}M$ is a geometrically simple perverse sheaf, pure of
  weight zero and of tannakian dimension at least~$2$.

  Moreover, if $\deg(f)\geq 2$, then the fourth moment of the tannakian
  group~$\garith{i_{f*}M}$ of~$i_{f*}M$ is equal to~$2$, and if
  $\deg(f)\geq 4$, then the eighth moment is equal to~$24$.
\end{proposition}

\begin{proof}
  The previous lema implies that~$i_{f*}M$ has tannakian dimension~$\geq
  2$. It is geometrically simple since~$M$ is.
  \par
  One argument to obtain the result is to observe that $i_f$ is a Sidon
  morphism when $\deg(f)\geq 2$, and a $4$-Sidon morphism when
  $\deg(f)\geq 4$ (by Proposition~\ref{pr-sidon-morphisms}, (4), since
  these properties can be checked after a finite extension), so that the
  result follows from Proposition~\ref{pr-sidon-moments1} since the
  tannakian dimension is~$\geq 2$, and the assumption on~$M$).
  \par
  For the sake of concreteness, we show also how to perform the
  computation of the eighth moment using the interpretation of the sums
  in terms of Dirichlet characters.  The eighth moment of the full
  family of exponential sums over~$k$ is equal to
  \begin{multline*}
    \frac{1}{|B^{\times}|}\sum_{\chi\in \widehat{B}^{\times}} \Bigl|
    \sum_{x\in k}t_M(x)\chi(t-x) \Bigr|^8= \sum_{x_1,\ldots, x_8}
    \prod_{i=1}^4t_M(x_i)\prod_{i=5}^8\overline{t_M(x_i)} \times
    \\
    \frac{1}{|B^{\times}|}\sum_{\chi\in \widehat{B}^{\times}}
    \chi((t-x_1)\cdots (t-x_4)) \overline{\chi((t-x_5)\cdots (t-x_8))}.
  \end{multline*}

  By orthogonality, the inner sum is~$0$ unless
  $$
  (t-x_1)\cdots (t-x_4)\equiv (t-x_5)\cdots (t-x_8)\mods{f},
  $$
  in which case it is equal to~$|B^{\times}|$. Since the degree of~$f$
  is at least~$4$, this congruence can only occur when
  $$
  (t-x_1)\cdots (t-x_4)= (t-x_5)\cdots (t-x_8)
  $$
  in $k[t]$. We then distinguish according to the size of
  $\{x_1,\ldots, x_4\}$. If this set has four elements, then so does
  $\{x_5,\ldots, x_8\}$, and the two sets are equal. The contribution
  arising from this case is
  $$
  \sum_{x_1,\ldots,x_4}\sum_{\sigma\in \mathfrak{S}_4} t_M(x_1)\cdots
  t_M(x_4)\overline{t_M(x_{\sigma(1)})\cdots t_M(x_{\sigma(4)})}=
  24\Bigl(\sum_{x\in k} |t_M(x)|^2\Bigr)^4.
  $$
  \par
  On the other hand, if the set $\{x_1,\ldots, x_4\}$ has three
  elements, say $x$, $y$ and~$z$, then so does $\{x_5,x_6,x_7,x_8\}$,
  and there are an absolutely bounded number of possibilities for
  $(x_1,x_2,x_3,x_4)$ given $x$, $y$ and~$z$. A similar result holds for
  two or one elements, and since~$t_M(x)\ll |k|^{-1/2}$, one sees that
  these altogether contribute at most
  $$
  \frac{1}{|k|^4}\sum_{x,y,z\in k}1\ll \frac{1}{|k|}.
  $$
  \par
  These computations can be repeated over~$k_n$ for~$n\geq 1$
  using~(\ref{eq-katz-1}), and using Proposition~\ref{pr-schur}, we
  deduce by letting~$n\to+\infty$ that
  $$
  \frac{1}{|B^{\times}|}\sum_{\chi\in \widehat{B}^{\times}} \Bigl|
  \sum_{x\in k}t_M(x)\chi(t-x) \Bigr|^8\to 24
  $$
  as~$|k|\to+\infty$.
  \par
  Finally, the usual argument using the definition of generic sets of
  characters together with~(\ref{eq-sum-variance}) and the Riemann
  Hypothesis imply that
  $$
  \frac{1}{|B^{\times}|}\sum_{\substack{\chi\in
      \widehat{B}^{\times}\\\widetilde{\chi}\text{ ramified}}} \Bigl|
  \sum_{x\in k}t_M(x)\chi(t-x) \Bigr|^8\to 0,
  $$
  so that Proposition~\ref{pr-larsen-diophante} gives the result.
\end{proof}

\begin{corollary}\label{cor-variance-mono}
  Under the assumptions of the proposition, the tannakian group of
  $i_{f*}M$ contains~$\SL_r$, where~$r$ is the tannakian dimension
  of~$i_{f*}M$, if $\deg(f)\geq 4$.
\end{corollary}

\begin{proof}
  By Lemma \ref{lm-rank-var}, the assumption implies~$r\geq 4$, and the
  result follows from Larsen's Alternative, in the form of the eighth
  moment theorem of Guralnick and Tiep
  (see Theorem~~\ref{th-larsen}\,\ref{th-larsen:item2}).
\end{proof}

\section{Application to von Mangoldt functions}

Suppose again that~$M$ is of the form $\mcF[1](1/2)$ for some
middle extension sheaf~$\mcF$ on~$\Aa^1_k[1/f]$ which is pure of
weight zero and geometrically irreducible of rank at least~$2$.

The statement of equidistribution on average for the object $i_{f*}M$
leads automatically to distribution statements of any ``continuous''
function of the polynomials in the variable~$T$ which are the twisted
$L$-functions of~$M$, namely
$$
\det(1-\frob_{k}T\,\mid\,
H^0_c(\Aa^1[1/f]_{\bar{k}},M\otimes\mcL_{\widetilde{\chi}}))
=\det(1-\frob_{k}T\,\mid\, H^0_c(T_{\bar{k}},(i_{f*}M)_{\chi}))
$$
as~$\chi\in \widehat{B}^{\times}$ varies, where~$\widetilde{\chi}$ is
now the character of the fundamental group of~$\Aa^1_k[1/f]$ that
corresponds to~$\chi$ by class-field theory,
and~$\mcL_{\widetilde{\chi}}$ is the associated rank one sheaf. 
\par
For instance, this leads to statements concerning the variance of von
Mangoldt functions in arithmetic progressions, as we now
explain.\index{von Mangoldt function}
\par
Write
$$
L(M,T)=\det(1-\frob_{k}T\,\mid\, H^0_c(\Aa^1[1/f]_{\bar{k}},M))=
\prod_{x}\det(1-\frob_{k_{\deg(x)}}T^{\deg(x)}\,\mid\, \mcF_{x})^{-1},
$$
where~$x$ runs over the set of closed points of~$\Aa^1_k[1/f]$, which
may be identified with the set of irreducible monic polynomials
in~$k[t]$ which are coprime to~$f$. Expanding the logarithmic derivative
of the local factor at a closed point~$x$, corresponding to an
irreducible monic polynomial~$\pi\in k[t]$, we have
$$
-Td\log(\det(1-\frob_{k_{\deg(x)}}T^{\deg(x)}\,\mid\, \mcF_{x})^{-1})=
\sum_{\nu\geq 1}\Lambda_M(\pi^{\nu})T^{\nu\deg(\pi)},
$$
which defines the von Mangoldt function~$\Lambda_M(\pi^{\nu})$ for any
monic irreducible polynomial~$\pi$ coprime to~$f$ and any~$\nu\geq
1$. We further define~$\Lambda_M(g)=0$ if~$g\in k[t]$ is not a power of
such an irreducible polynomial.  The full logarithmic derivative then
has the formal power series expansion
$$
-T\frac{L'(M,T)}{L(M,T)}=\sum_{g}\Lambda_M(g)T^{\deg(g)}
$$
over all monic polynomials~$g\in k[t]$.
\nomenclature[$L$]{$\Lambda_M$}{von Mangoldt function of~$M$}
\par
For an integer $m\geq 1$ and a polynomial $a\in k[t]$, we then define
$$
\psi_M(m;f,a)=\sum_{\substack{\deg(g)=m\\g\equiv
    a\mods{f}}}\Lambda_M(g).
$$
\par
We consider the average
$$
A_M(m;f)=\frac{1}{|B^{\times}|}\sum_{a\in B^{\times}}\psi_M(m;f,a)
$$
and the variance
$$
V_M(m;f)= \frac{1}{|B^{\times}|}\sum_{a\in
  B^{\times}}\abs{\psi_M(m;f,a)-A_M(m;f)}^2.
$$

These are related to exponential sums as follows.

\begin{proposition}
  With assumptions and notation as above, we have 
  $$
  V_M(m;f)=\frac{1}{|B^{\times}|^2}\sum_{\substack{\chi\in
      \widehat{B}^{\times}\\\chi\not=1}} V_M(m;\chi)
  $$
  where
  $$
  V_M(m;\chi)=\Bigl|\sum_{x\in k_m}t_M(x;k_m)\chi(N_{k_m/k}(t-x))\Bigr|^2.
  $$
  \par
  In particular, if $\chi$ is weakly unramified for~$i_{f*}M$, then we
  have
  $V_M(m;\chi)=|\Tr(\Theta_{M}(\chi)^m)|^2$.
\end{proposition}

\begin{proof}
  The first part is proved, using the orthogonality of characters,
  exactly like~\cite[\S 6, (6.3.4)]{variance}. The second assertion then
  follows from Lemma~\ref{lm-trace} and~(\ref{eq-katz-1}).
\end{proof}

\begin{remark}
  The von Mangoldt function can be replaced by many other arithmetic
  functions in this argument; we refer to the discussion by Sawin
  in~\cite{sawin-witt} (which proves analogue equidistribution
  statements to ours for the case of ``short intervals'', which amounts
  to considering a unipotent group instead of a torus) and
  to~\cite{sawin-factorization} for a discussion of how classical
  arithmetic functions which are related to ``factorization functions''
  (functions of polynomials~$g$ that depend only on the factorization
  type of~$g$) can be interpreted as trace functions using
  representation theory of the symmetric groups.
\end{remark}

We now obtain a formula for the variance, with some additional
assumption.

\begin{corollary}\label{cor-variance}
  In addition to the assumptions of this section, assume that $m\geq 2$,
  and that the tannakian determinant of~$M$ is geometrically of infinite
  order. Then
  $$
  \lim_{|k|\to+\infty}|B^{\times}|^2V_M(m;f)=\min(m,r),
  $$
  where $r$ is the tannakian dimension of~$i_{f*}M$.
\end{corollary}

\begin{proof}
  Combined with Corollary~\ref{cor-variance-mono}, the assumption
  implies that the arithmetic and geometric tannakian groups
  of~$i_{f*}M$ are both equal to~$\GL_r$. Thus the limit exists by
  Theorem~\ref{th-2bis} and is equal to
  $$
  \int_{\Un_r(\Cc)}|\Tr(g^m)|^2d\mu(g)
  $$
  where $\mu$ is the Haar probability measure. This matrix integral is
  equal to $\min(m,r)$ by work of Diaconis and
  Evans~\cite[Th.\,2.1]{dia-ev}.
\end{proof}

To check the assumption on the tannakian determinant, we have a first
general criterion, which is however quite restricted.

\begin{proposition}
  \label{pr-local-monodromy-variance}
  With notation and assumptions as above, suppose that there exists
  $z\in Z$ such that the local monodromy representation\index{local monodromy representation} of~$\mcF$ at~$z$
  has a non-zero unipotent tame component while the local monodromy at
  infinity has no unipotent tame component. Then the tannakian
  determinant of~$i_{f*}M$ is geometrically of infinite order.
\end{proposition}

\begin{proof}
  We apply Corollary~\ref{cor-simpler} to the norm morphism
  $p\colon T\to\Gg_m$. Indeed, $p\circ i_f$ coincides with the finite
  morphism $\eps f\colon \Aa^1[1/f]\to \Gg_m$, where
  $\eps=(-1)^{\deg(f)}$ (Proposition~\ref{pr-katz}), so that the
  equalities
  $Rp_!(i_{f*}M)=Rp_!(i_{f!}M)=(\eps f)_*M=((\eps f)_!\mcF)[1](1/2)$
  hold, and the sheaf
  $((\eps f)_!\mcF)[1](1/2)=((\eps f)_*\mcF)[1](1/2)$ has no tame
  unipotent local monodromy at infinity, but has some non-trival tame
  unipotent monodromy at~$0$ in view of the canonical isomorphism
  $$ 
  ((\eps f)_*\mcF)_0\simeq \bigoplus_{z\in Z}\mcF_z.
  $$
  Hence, the tannakian determinant of the object $i_{f*}M$ is
  geometrically of infinite order.
\end{proof}

We now explain the proof of Theorem~\ref{th-variance}, where we will
also use a different approach to checking that the tannakian determinant
has infinite order, which may be useful in other contexts.

Let~$\pi\colon \mathcal{E}\to \Pp^1$ be the morphism which ``is'' the
Legendre elliptic curve.\index{Legendre elliptic curve} We start with
the sheaf
$$
\mcF=R^1\pi_*\bQl(1/2).
$$
\par
This is a middle extension sheaf on~$\Aa^1_k$.  It is pure of weight
zero and geometrically irreducible of rank~$2$ (in particular, its
$H^2_c$ vanishes), and is tamely ramified at~$0$, $1$ and~$\infty$, with
drop equal to~$1$ at~$0$ and~$1$.  Using Lemma~\ref{lm-rank-var}, we
compute that the tannakian dimension is $r=2\deg(f)-2+a$, where~$a$ is the
degree of the gcd of~$f$ and $t(t-1)$.

Now the pullback of~$\mcF$ to~$\Aa^1_k[1/f]$ is a middle extension
sheaf, geometrically irreducible of rank~$2$ and pure of weight~$0$, for
which we keep the same notation.
We can then apply Corollary~\ref{cor-variance} to~$\mcF$, using the
following proposition. In order to conclude after doing so, we check
that the contribution of the local factors at~$z\in Z$ to the
$L$-functions (which might not be of weight~$0$) is negligible
(compare~\cite[Prop.\,6.5.3]{variance}).

\begin{proposition}\label{pr-infinite-order-variance}
  Let $M=\mcF[1](1/2)$. The tannakian determinant of~$i_{f*}M$ is
  geometrically of infinite order.
\end{proposition}

\begin{proof}
  If $f$ is not coprime to $t(t-1)$, then we can apply
  Proposition~\ref{pr-local-monodromy-variance}, since $\mcF$ has
  non-trivial tame unipotent monodromy at $0$ and~$1$, and none at
  infinity.  So we assume that $f$ is coprime with $t(t-1)$.
  \par
  We may assume that the polynomial $f$ splits in linear factors
  over~$k$ and that $k\not=Z\cup\{0,1\}$. Fix a non-trivial additive
  character $\psi$ of~$k$. We will then prove in
  Proposition~\ref{pr-local} below, using the theory of local constants,
  that there exists a generic set of characters $\mcX\subset \unram{M}$
  and elements $\xi_z\in \Aa^1[1/f]$ such that for $n\geq 1$ and
  $\chi\in \mcX(k_n)$, the equality
  $$
  \det(\Theta_{M,k_n}(\chi))= \gamma^n \, 
  H_1(\chi)^{-1} \, H_2\Bigl(\prod_{z\in Z}\chi_z^{-1}\Bigr)^{-1}
  $$
  holds, for some number $\gamma$ independent of~$\chi$ and~$n$, where
  the functions $H_1$ and~$H_2$ are products of Gauss sums described
  in~(\ref{eq-h1}) and~(\ref{eq-h2}) below.
  \par
  On~$\Gg_m$, the function
  $$
  \chi_z\mapsto \chi_z(\xi_z) \frac{1}{|k|}\Bigl(\sum_{y\in
    k^{\times}}\chi_z(y)\psi(y)\Bigr)^{2}
  $$
  coincides for $\chi_z$ non-trivial with the arithmetic Mellin
  transform of the multiplicative translated hypergeometric complex
  $\Hyp_{\xi_z}(!,\psi, 1,1;\emptyset)(1/2)$
  (see~(\ref{eq-mellin-hypergeom}) for this; in this case, this is a
  shifted and translated Kloosterman sheaf). Since the function
  $\chi\mapsto H_1(\chi)^{-1}$ is the product of these functions
  over~$z\in Z$, it coincides generically with the Mellin transform
  on~$T$ of the tensor product
  $$
  \bigotimes_{z\in Z} p_z^*\Hyp_{\xi_z^{-1}}(!;\psi,1,1;\emptyset)(1/2),
  $$
  where $p_z$ is the projection from~$T$ to the $z$-component in the
  splitting $g\mapsto (g(z))$ of the torus~$T$. (Indeed, this reflects
  the formula
  $$
  \sum_{x\in T(k)}
  \chi(x)\prod_{z\in Z}f_z(p_z(x))=
  \sum_{(x_z)\in (k^{\times})^Z}\prod_{z\in Z} \chi_z(x_z)\prod_{z\in
    Z}f_z(x_z)= \prod_{z\in Z}\sum_{x\in k^{\times}}\chi_z(x)f_z(x)
  $$
  for arbitrary functions $f_z$ on~$k^{\times}$.)
  \par
  Similarly, the function $\chi\mapsto H_2(\prod \chi_z^{-1})^{-1}$,
  which only depends on the product~$\eta$ of the component characters
  $(\chi_z)$, coincides (for $\eta$ non-trivial) with the arithmetic
  Mellin transform of the object $\Delta_*L$, where
  $L=\Hyp(!,\psi,\lambda_2,\lambda_2;\emptyset)(1/2)$ and
  $$
  \Delta \colon \Gg_m\to \Gg_m^Z\simeq T
  $$
  is the closed immersion $x\mapsto (x^{-1},\ldots, x^{-1})$. This
  reflects the fact that~$\Delta$ is a morphism of algebraic groups, and
  that the dual $\widehat{\Delta}$ on~$\charg{T}(k)$ is given by
  $$
  (\chi_z)_{z\in Z}\mapsto \prod_{z\in Z}\chi_z^{-1}.
  $$
  \par
  By Theorem~\ref{th-generic-inversion}, the
  formula~(\ref{eq-root-number2}) therefore implies that the tannakian
  determinant of~$M$ is geometrically isomorphic in $\Ppb(T)$ to the
  perverse sheaf
  $$
  D=(\Delta_*L) * \Bigl(\bigotimes_{z\in Z}
  p_z^*\Hyp_{\xi_z}(!;\psi,1,1;\emptyset)(1/2)\Bigr).
  $$
  \par
  The object~$D$ visibly has infinite geometric tannakian group since
  for any~$m\geq 1$, we have
  $$
  D^{* m}=(\Delta_*L)^{* m} * \Bigl(\bigotimes_{z\in Z}
  p_z^*\Hyp_{\xi_z^{-1}}(!;\psi,1,1;\emptyset)^{*m}(1/2)\Bigr),
  $$
  in $\Ppb(T)$, and the $m$-th convolution powers on $\Gg_m$ of the
  hypergeometric complexes that appear are not geometrically trivial
  (see Theorem~\ref{th-hypergeometric}).
\end{proof}

We complete this section by proving the formula for the determinant.

\begin{proposition}\label{pr-local}
  Suppose that $f$ splits in linear factors over~$k$.
  For $z\in Z$, define
  $$
  \xi_z=z(z-1) \prod_{x\in Z\setminus\{z\}}(z-x)^{2}\in k^{\times}.
  $$
  \par
  There exist numbers $\eps_0$, $\eps_1$ 
  with the following property. For a character $\chi\in\unram{M}$ such
  that all components $\chi_z$ for $z\in Z$ are non-trivial, and such
  that the product of the components is not of order at most~$2$, we
  have
  \begin{equation}\label{eq-root-number2}
    \det(\Thetaf_M(\chi))^{-1}= (-1)^r|k|
    \,\eps_0\eps_1
    \,    H_1(\chi)
    \,    H_2\Bigl(\prod_{z\in Z}\chi_z^{-1}\Bigr)
  \end{equation}
  where
  \begin{align}
    \label{eq-h1}
    H_1(\chi)&= \prod_{z\in Z} \chi_z(\xi_z^{-1})|k|\Bigl(\sum_{y\in
      k^{\times}}\chi_z(y)\psi(y)\Bigr)^{-2},
    \\
    H_2(\chi)&=|k|\Bigl(\sum_{y\in
      k^{\times}}(\lambda_2\chi)(y)\psi(y)\Bigr)^{-2}.
    \label{eq-h2}
  \end{align}
\end{proposition}

\begin{proof}
  Let $j\colon \Aa^1[1/f]\to \Pp^1$ be the open immersion. Let
  $\chi\in\widehat{B}$ be a Dirichlet character and $\mcL_{\chi}$ the
  lisse rank~$1$ sheaf on~$\Aa^1[1/f]$ that corresponds to it. The
  $L$-function of $j_!(M\otimes \mcL_{\chi})$ satisfies a functional
  equation of the form
  $$
  L(j_!(M\otimes\mcL_{\chi}),T)=\eps(\chi)T^aL(\dual(j_!(M\otimes
  \mcL_{\chi})), T^{-1})
  $$
  where $a=-\chi(j_!(M\otimes \mcL_{\chi}))=-r$ is an integer and
  $$
  \eps(\chi)= \det(-\Frob_k\mid H^0(\Pp^1_{\bar{k}}, j_!(M\otimes
  \mcL_{\chi})))^{-1} =\det(-\Frob_k\mid H^0_c(\Aa^1_{\bar{k}}[1/f],
  M\otimes \mcL_{\chi}))^{-1}
  $$
  (see, e.g.,~\cite[(3.1.1.3),\,(3.1.1.5)]{laumon-signes} or the
  reminder in Section~\ref{sec-product-formula}).
  \par
  By Lemma~\ref{lm-trace}, if $\chi\in\charg{T}$ is unramified
  for~$M$, then we deduce that
  \begin{equation}\label{eq-root-number}
    \det(\Thetaf_M(\widetilde{\chi}))=(-1)^r\eps(\chi)^{-1},
  \end{equation}
  where $r$ is the tannakian dimension of~$i_{f*}M$.  By a theorem of
  Laumon,\footnote{\ Which, in the case we use it, goes back to
    Deligne~\cite[Th.\,9.3]{deligne-constantes};
    see~\cite[3.2.1.9]{laumon-signes} for references.} we can express
  the constant $\eps(\chi)$ as a product over closed points
  $$
  \eps(\chi)=|k|^{-2}\prod_{x\in|\Pp^1|}\eps_x(\chi)
  $$
  of local constants, previously defined by
  Deligne~\cite{deligne-constantes} and characterized by the properties
  of~\cite[Th.\,3.1.5.4]{laumon-signes}. Precisely, fixing a non-trivial
  additive character~$\psi$ of~$k$ and a non-zero meromorphic
  differential $1$-form $\omega$ on~$\Pp^1$, we can then define
  $$
  \eps_x(\chi)=\eps(\Pp^1_{(x)},j_*(M\otimes\mcL_{\chi})|\Pp^1_{(x)},
  \omega\mid \Pp^1_{(x)})
  $$
  with the notation of \loccit See again
  Section~\ref{sec-product-formula}; in particular the factor $|k|^{-2}$
  above is given by~(\ref{eq-product-formula}), namely the exponent is
  obtained by the computation
  $$
  -2=1\cdot (1-0)\cdot (-2),
  $$
  where $-2$ is the generic rank of the object~$M$ (a sheaf of
  rank~$2$ in degree~$-1$).
  \par
  We take $\omega=dt$, where $t$ is the standard coordinate
  on~$\Pp^1$. The data of~$\psi$ and $\omega$ allows us to define
  non-trivial additive characters~$\psi_x$ of the completed local
  field at any closed point $x\in|\Pp^1|$ by the recipe
  in~\cite[Th.\,3.1.5.4,\,(v)]{laumon-signes}. For all closed points
  $x\in\Aa^1$, the character $\psi_x$ is of conductor zero since
  $\omega$ is regular at~$x$ (see~\cite[3.1.3.6]{laumon-signes}).  For
  $x=\infty$, we have $c(\psi_{\infty})=-2$ since $\omega$ has a
  double pole at $\infty$.
  \par
  The main tool to compute the local constants is the
  formula~(\ref{eq-3156}) for twisting by a lisse sheaf: for any
  closed point $x$, if $K$ is an $\ell$-adic complex on the
  \emph{trait} $\Pp^1_{(x)}$ and $F$ is a lisse $\Qlb$-sheaf
  on~$\Pp^1_{(x)}$ of rank~$r(F)$, then we have
  \begin{equation}\label{3.1.5.6}
    \eps(\Pp^1_{(x)},(K\otimes F)|\Pp^1_{(x)},\omega|\Pp^1_{(x)})=
    \det(\frob_x\mid F)^{a(\Pp^1_{(x)},K,\omega|\Pp^1_{(x)})}
    \eps(\Pp^1_{(x)},K,\omega|\Pp^1_{(x)})^{r(F)},
  \end{equation}
  where the local exponent $a(\Pp^1_{(x)},K,\omega|\Pp^1_{(x)})$ is
  defined in~(\ref{eq-3151}) and~(\ref{eq-3152}). Moreover, we will
  often use the formula
  \begin{align*}
    \eps(\Pp^1_{(x)},K[1],\omega)= \eps(\Pp^1_{(x)},K,\omega)^{-1}
  \end{align*}
  (see~(\ref{eq-shift})).
  \par
  Let $(\chi_z)_{z\in Z}$ be the tuple of characters corresponding
  to~$\chi$. We recall that $\mcL_{\chi}$ is isomorphic to
  $\bigotimes_{z\in Z} \mcL_{\chi_z(z-t)}$.
  \par
  We now compute the local constants, distinguishing between
  the cases 
  $x\in \Aa^1\setminus (\{0,1\}\cup Z)$,
  $x\in\{0,1\}$, $x\in Z$ and $x=\infty$.
  \par
  \medskip
  \par
  \textbf{Case 1}.  Let $x\in\Aa^1$ and $x\notin Z\cup \{0,1\}$. In this
  case, $M\otimes\mcL_{\chi}$ is a lisse sheaf shifted by~$1$, and since
  $c(\psi_x)=0$, we find
  \begin{equation}\label{eq-elsewhere}
    \eps_x(\chi)=1
  \end{equation}
  by~(\ref{3.1.5.6}).
  \par
  \smallskip
  \par
  \textbf{Case 2.} Let $x\in \{0,1\}$.
  Then
  $\mcL_{\chi}$ is a lisse sheaf at~$x$, since we assumed that $f$ is
  coprime with $t(t-1)$.  We find
  $$
  \eps_x(\chi)=\eps_x\,
  t_{\mcL_{\chi}}(x)^{a(\Pp^1_{(x)},M|\Pp^1_{(x)},dt)}
  $$
  by~(\ref{3.1.5.6}) with $F=\mcL_{\chi}$, where
  $\eps_x=\eps(\Pp^1_{(x)},M,dt)$, which is independent of~$\chi$. We
  further compute that
  $$
  a(\Pp^1_{(x)},M|\Pp^1_{(x)},dt)= -
  a(\Pp^1_{(x)},\mcF(1/2)|\Pp^1_{(x)},dt)= -(2-1+0)=-1
  $$
  by~(\ref{eq-3151}) and~(\ref{eq-3152}), since $\mcF$ has drop~$1$
  at~$x$ (see, e.g.,~\cite[p.\,73]{mellin}) and $dt$ is regular
  at~$x$. Hence,
  \begin{equation}\label{eq-01}
    \eps_x(\chi)=\eps_x \prod_{z\in Z}\chi_z(z-x)^{-1}.
  \end{equation}
  \par
  \smallskip
  \par
  \textbf{Case 3.} Let $x\in Z$. Then we can write
  $$
  M\otimes\mcL_{\chi}=\mcF[1](1/2)\otimes \mcL^{(x)}\otimes
  \mcL_{\chi_x(t-x)}= (\mcF(1/2)\otimes \mcL^{(x)}\otimes
  \mcL_{\chi_x(t-x)})[1],
  $$
  where $\mcF$ and $\mcL^{(x)}$ are both lisse sheaves
  at~$x$. Applying~(\ref{3.1.5.6}) after an inversion due to the shift,
  we get
  \begin{align*}
    \eps_x(\chi)&= \eps(\Pp^1_{(x)},\mcF[1](1/2)\otimes
    \mcL^{(x)}\otimes \mcL_{\chi_x(t-x)},dt)^{-1}
    \\
    &=\det(\Fr_x\mid \mcF(1/2)\otimes \mcL^{(x)})^{-a}
    \eps(\Pp^1_{(x)},\mcL_{\chi_x(t-x)},dt)^{-2}
  \end{align*}
  where
  $$
  a= a(\Pp^1_{(x)},\mcL_{\chi_x(t-x)},dt)=1+0-0=1
  $$
  if $\chi_x$ is non-trivial by~(\ref{eq-3151}) and~(\ref{eq-3152})
  again.
  \par
  We have
  $$
  \det(\Fr_x\mid \mcF(1/2)\otimes \mcL^{(x)})=
  \frac{1}{|k|}\prod_{\substack{z\in Z\\z\not=x}}\chi_z(z-x)^2,
  $$
  and by~(\ref{eq-3531}), we find that
  $$
  \eps(\Pp^1_{(x)},\mcL_{\chi_x(t-x)},dt)=
  \eps_0(\Pp^1_{(x)},\mcL_{\chi_x(t-x)},dt)= -\chi_x(-1) \sum_{y\in
    k^{\times}}\chi(y)\psi(y)
  $$
  if $\chi_x$ is not trivial (here we also use the fact that
  $x\in k$).
  \par
  These computations imply that
  \begin{equation}
    \label{eq-Z}
    \eps_x(\chi)=
    \prod_{\substack{z\in Z\\ z\not=x}}\chi_z(z-x)^{-2}\,
    |k|\, \Bigl(  \sum_{y\in k^{\times}}\chi(y)\psi(y)\Bigr)^{-2},
  \end{equation}
  if $\chi_x$ is not trivial.
  \par
  \smallskip
  \par
  \textbf{Case 4.} Let $x=\infty$.  Write $u=1/t$, a uniformizer
  at~$\infty$, so that $dt=-u^{-2}du$. Then
  $ \mcL_{\chi}=\mcL^{(\infty)}\otimes\mcL_{\eta(u)}$ where
  $$
  \mcL^{(\infty)}=\bigotimes_{z\in Z}\mcL_{\chi_z(uz-1)},\quad\quad
  \eta=\prod_{z\in Z}\chi_z^{-1}.
  $$
  \par
  The sheaf~$\mcL^{(\infty)}$ is lisse at $\infty$ and the local
  eigenvalue of Frobenius there is equal to~$(-1)^{\deg(f)}$.  On the
  other hand, we have $M=\mcF[1](1/2)$, and $\mcF$ is of rank~$2$,
  tamely ramified at~$\infty$ with local monodromy isomorphic to
  $\mathcal{L}_{\lambda_2}\otimes \mathrm{Unip(2)}$, where $\lambda_2$
  is the Legendre character and $\mathrm{Unip(2)}$ is a unipotent Jordan
  block of size~$2$ (see, e.g.,~\cite[p.\,73]{mellin}).
  \nomenclature[$U$]{$\mathrm{Unip(n)}$}{unipotent Jordan block of size~$n$}
  \par
  Computing first as in the previous case, we get
  \begin{align*}
    \eps_{\infty}(\chi)&=
    \eps(\Pp^1_{(\infty)},\mcF(1/2)\otimes\mcL^{(\infty)} \otimes
    \mcL_{\eta(u)},-u^{-2}du)^{-1}
    \\
    &= \det(\Fr_{\infty}\mid \mcL^{(\infty)})^{-a}
    \eps(\Pp^1_{(\infty)},\mcF(1/2)\otimes
    \mcL_{\eta(u)},-u^{-2}du)^{-1}
  \end{align*}
  where
  \begin{align*}
    a&= a(\Pp^1_{(\infty)},\mcF(1/2)\otimes \mcL_{\eta(u)},-u^{-2}du)
    \\
    & =a(\Pp^1_{(\infty)},\mcF(1/2)\otimes \mcL_{\eta(u)})-2\times 2
    =(2+0-2)-4=-4
  \end{align*}
  if $\eta$ is non-trivial (see again~(\ref{eq-3151})
  and~(\ref{eq-3152})).
  Note then that
  $$
  \det(\Fr_{\infty}\mid \mcL^{(\infty)})^{-a}=\prod_{z\in
    Z}\chi_z(-1)^{4}=1.
  $$
  \par
  The shape of the local monodromy and the multiplicativity property
  under extensions shows that if $\lambda_2\eta$ is not trivial, then
  the formula
  $$
  \eps(\Pp^1_{(\infty)},\mcF(1/2)\otimes \mcL_{\eta(u)},-u^{-2}du)=
  \eps(\Pp^1_{(\infty)},\mcL_{(\lambda_2\eta)(u)},-u^{-2}du)^2
  $$
  holds. Indeed, in this case, the stalk at~$\infty$
  of~$\mcF\otimes \mcL_{\eta}$ and of its semisimplification both
  vanish, so that
  $$
  \eps(\Pp^1_{(\infty)},\mcF(1/2)\otimes \mcL_{\eta(u)},-u^{-2}du)=
  \eps_0(\Pp^1_{(\infty)},j_*\mcF(1/2)\otimes \mcL_{\eta(u)},-u^{-2}du),
  $$
  where $\eps_0$ is the local factor defined by~(\ref{eq-eps0}), and
  $j$ is the inclusion of the generic point of~$\Pp^1_{(\infty)}$, and
  one can apply~(\ref{eq-3157});
  compare~\cite[8.12]{deligne-constantes}.
  \par
  Let $\beta$ be the character of the local field at infinity
  associated to $\lambda_2\eta$ by local class field theory.
  Using~(\ref{eq-3155}), we derive the formula
  $$
  \eps(\Pp^1_{(\infty)},\mcL_{(\lambda_2\eta)(u)},-u^{-2}du)
  =\beta(-u^{-2})|k|^{-2}\eps(\Pp^1_{(\infty)},\mcL_{(\lambda_2\eta)(u)},du).
  $$
  \par
  From~(\ref{eq-3531}), we deduce further that if $\lambda_2\eta$ is
  non-trivial, then
  $$
  \eps(\Pp^1_{(\infty)},\mcL_{(\lambda_2\eta)(u)},-u^{-2}du) = |k|^{-2}
  \sum_{y\in k^{\times}}(\lambda_2\eta)(y)\psi(y).
  $$
  \par
  The final outcome is that
  \begin{equation}\label{eq-infinity}
    \eps_{\infty}(\chi)=|k|^4
    \Bigl(\sum_{y\in k^{\times}}(\lambda_2\eta)(y)\psi(y)\Bigr)^{-2},
  \end{equation}
  if $\eta\notin\{1,\lambda_2\}$.
  
  \par
  
  We now simply combine the formulas (\ref{eq-elsewhere}),
  (\ref{eq-01}), 
  (\ref{eq-Z}) and (\ref{eq-infinity}) to conclude the proof, noting
  that the contribution of all $x\in Z$ involves the product
  $$
  \prod_{x\in Z} \prod_{\substack{z\in Z\\z\not=x}} \chi_z(z-x)^{-2}=
  \prod_{z\in Z}\chi_z\Bigl(\prod_{\substack{x\in
      Z\\x\not=z}}(z-x)^{-2}\Bigr).
  $$
\end{proof}

\begin{remark}
  It it also certainly possible to perform this computation by
  automorphic methods (using the global case of the $\GL_2$-Langlands
  correspondence\index{Langlands correspondence} over~$k(t)$, first proved by Drinfeld). However, more
  general situations might be easier to handle using these geometric
  arguments.
  \par
  Yet another possible approach, which would be well-suited for
  generalizations, would be to use Loeser's general computation of the
  tannakian determinant for an arbitrary perverse sheaf on a
  torus~$\Tt$ (see~\cite[Th.\,3.6.1]{loeser-determinant}), which can
  be identified with an element of the hypergeometric group
  $\Hh_{\intt}(\Tt)$ of Gabber and Loeser (see
  Example~\ref{ex-gl-rank-1}). This group is isomorphic (\loccit) to
  $\Tt(\bar{k})\times \Zz^{S}$ for some explicit set~$S$ (related to
  sub-tori of dimension~$1$ in~$\Tt$ and tame $\ell$-adic characters
  of~$\Gg_m$. It would then be enough to show that there exists some
  $s\in S$ such that the $s$-component of $\det(M)$ is non-zero to
  deduce that $\det(M)$ has infinite order (without computing exactly
  the determinant).
\end{remark}


\chapter{Equidistribution on abelian varieties}\label{sec-jacobian}

In this chapter, we consider some aspects of equidistribution on
abelian varieties. We denote as before by $k$ a finite field, and
by~$\bar{k}$ an algebraic closure of~$k$. We denote by $k_n$ the
extension of degree~$n$ in~$\bar{k}$. The prime $\ell$ is different from
the characteristic of~$k$.

\section{Equidistribution in the jacobian of a curve}

The main result of this section is a generalization of a theorem
announced by Katz during a talk at a workshop held at the University of
Zürich in September~$2012$~\cite{KatzETH}, answering a question of
Tsimerman.

Let $C$ be a smooth projective geometrically connected curve of genus
$g \geq 2$ over~$k$, and let $A=\mathrm{Jac}(C)$ be its jacobian.  We
recall that $C$ may not have $k$-rational points but always has a
$k$-rational divisor of degree one. We fix such a divisor~$\Delta$ and
we denote by~$s_{\Delta} \colon C \hookrightarrow A$ the closed
immersion obtained by sending a point $x$ to the class of the divisor
$(x)-\Delta$. Recall that the functor $s_{\Delta\ast}=s_{\Delta!}$
preserves perversity (Corollary~\ref{cor-closed-immersion}).

\begin{theorem}[Katz]\label{thm:equidistribution-on-jacobians}
  Let $\Delta$ be a divisor of degree one on $C$. Let $M_0$ be a
  geometrically simple perverse sheaf on $C$ of generic rank~$r\geq 1$
  which is pure of weight zero.  Let $M=s_{\Delta*}M_0$ and let $d$ denote
  the tannakian dimension of~$M$.
  \begin{enumth}
  \item We have $d\geq (2g-2)r\geq 2$.
  \item\label{thm:equidistribution-on-jacobians:item1} Assume that $C$
    is hyperelliptic, that $\Delta=(0_C)$ for some $k$-rational point
    $0_C \in C(k)$ fixed by the hyperelliptic involution~$i$, and
    that~$\dual(M_0)$ is geometrically isomorphic to~$i^\ast M_0$. Then,
    up to conjugacy, there are inclusions
    \begin{displaymath}
      \ggeo{M}=\garith{M}=\Sp_d
      \text{ or }\SO_d \subset \ggeo{M} \subset
      \garith{M} \subset \Ort_d.
    \end{displaymath}
  \item\label{item2-thmKatz} If $C$ is not hyperelliptic, or if $C$ is
    hyperelliptic but $\dual(M_0)$ is not geometrically isomorphic
    to~$i^\ast M_0$, then there are inclusions
    \begin{displaymath}
      \SL_d \subset \ggeo{M} \subset \garith{M} \subset \GL_d.
    \end{displaymath}
  \end{enumth}
\end{theorem}


\begin{proof}
  We write~$s=s_{\Delta}$ for simplicity.  Since $A$ is an abelian
  variety, the dimension $d$ is the Euler--Poincaré characteristic of
  $M_{\chi}$ for \emph{any} $\chi\in\charg{A}$ (see
  Proposition~\ref{pr-chi}), in particular for the trivial character,
  which means that $d=\chi(A_{\bar{k}},M)=\chi(C_{\bar{k}},M_0)$. Write
  $M_0=\mcF_0[1](1/2)$ for some middle extension sheaf~$\mcF_0$ on~$C$
  of generic rank~$r$; using the Euler--Poincaré characteristic formula
  on a curve (see~(\ref{eq-ep-sheaf}), for instance), it follows that
  \begin{equation}\label{eq-tan-jac-2}
    \chi(C_{\bar{k}},M_0)=\chi(C_{\bar{k}},\mcF_0[1])=(2g-2)r+\sum_{x\in
      C(\bar{k})}(\swan_x(\mcF_0)+\Drop_x(\mcF_0)) \geq (2g-2)r.
  \end{equation}
  \par
  According to Proposition~\ref{pr-sidon-morphisms}\,(2), the embedding
  $s$ is a Sidon morphism if~$C$ is not hyperelliptic, and is an
  $i$-symmetric Sidon morphism in the hyperelliptic situation of~(2).
  \par
  Suppose first that~$C$ is not hyperelliptic. Using the fact
  that~$d\geq 2$, we deduce from Proposition~\ref{pr-sidon-moments1}
  that $M_4(\garith{M})=2$. Thus, by Larsen's Alternative
  (Theorem~\ref{th-larsen}\,(3)), either $\garith{M}$ is virtually
  central, \ie $\garith{M}/\garith{M}\cap Z$ is finite, or $\garith{M}$
  contains $\SL_d$.  Proposition~\ref{pr-finite-ab-2} shows that the
  first case is not possible, since $M_4(\garith{M})=2$ is not the
  square of an integer.  Then the fact that~$\garith{M}$ contains
  $\SL_d$ implies that $\ggeo{M}$ also contains $\SL_d$ (indeed, the
  intersection $\ggeo{M} \cap \SL_d$ is a normal subgroup of
  $\garith{M}$ by Proposition~\ref{pr:geom-vs-arith1}, and hence is a
  normal subgroup of~$\SL_d$; it is therefore either equal to $\SL_d$,
  or is contained in the center~$\mmu_d$; but since $d\geq 2$, the
  latter would imply that $\garith{M}/\ggeo{M}$ is not abelian).
  \par
  We now assume that~$C$ is hyperelliptic. First we consider the case
  when $d\geq 3$.
  \par
  If $\dual(M_0)$ is not geometrically isomorphic to $i^*M_0$, then
  Proposition~\ref{pr-sidon-moments2}\,(2) implies 
  \hbox{$M_4(\garith{M})=2$} since we assume that $d\geq 3$; as previously, we
  then conclude that $\garith{M}$ contains~$\SL_d$.
  \par
  If the conditions of (2) hold, then the constant morphism
  $(s\circ i)+ s$ is given by
  $$
  s(i(x))+ s(x)=(x)+i(x)-2(0_C)=0,
  $$
  the identity element of~$A$.  Proposition~\ref{pr-sidon-moments2}\,(1)
  implies then that $M$ is self-dual and has $M_4(\garith{M})=3$, again
  from our assumption that $d\geq 3$. We conclude in that case by
  Larsen's Alternative (Theorem~\ref{th-larsen}\,(5)), combined with the
  fact that $\garith{M}$ is infinite by Theorem~\ref{th-finite-ab-1}.
  \par
  There remains to consider the case when $d=2$ (and~$C$ hyperelliptic).
  Since $d=\chi(C_{\bar{k}},M_0)$, formula~(\ref{eq-tan-jac-2}) shows
  that this situation can only occur if~$(g,r)=(2,1)$ and if the sheaf
  $\mcF_0$ is lisse on~$C$. Thus the curve~$C$ has genus~$2$, and the
  sheaf~$\mcF_0$ is a rank~$1$ sheaf corresponding to a character of the
  fundamental group of~$C$.  As we will recall below in general, there
  exists then a character~$\chi_0\in \charg{A}(k)$ such that $\mcF_0$ is
  geometrically isomorphic to $s^*\mcL_{\chi_0}$ on~$C$.  The duality
  condition $M_0\simeq i^*\dual(M_0)$ is then always satisfied.

  We claim that in this situation, the fourth moment $M_4(\garith{M})$
  is still equal to~$2$. Indeed, from the proof of
  Proposition~\ref{pr-sidon-moments2}, we know that
  $$
  \frac{1}{|A(k_n)|}\sum_{\chi \in \charg{A}(k_n)} |S(M, \chi)|^{4}
  $$
  converges to~$3$ as $n\to+\infty$. The contribution of the
  character~$\chi_0^{-1}$, which is the only ramified character, is
  $$
  \frac{1}{|A(k_n)|} |S(M,\chi_0^{-1})|^4=\frac{1}{|A(k_n)|}\Bigl|
  \sum_{x\in C(k_n)}t_{M_0}(x;k_n)\overline{\chi_0(x)}
  \Bigr|^4=\frac{|C(k_n)|^4}{|k_n|^2|A(k_n)|}
  $$
  which converges to~$1$ as $n\to+\infty$. We then conclude from
  Larsen's Alternative that the group $\garith{M}$ contains $\SL_2=\Sp_2$.
\end{proof}

\begin{remark}\label{rm-jacobian-careful}
  (1) Note that the last case provides a concrete example where the
  limit
  \[
  \lim_{n\to+\infty} \frac{1}{|A(k_n)|}\sum_{\chi \in \charg{A}(k_n)}
  |S(M, \chi)|^{4}
  \]
  exists, where the sum ranges over all characters, but its value is
  \emph{not} the fourth moment of the standard representation of the
  tannakian group (see Remark~\ref{rm-careful}).
  \par
  (2) If the curve $C$ has gonality at least~$5$, then the inclusions
  \begin{displaymath}
    \SL_d \subset \ggeo{M} \subset \garith{M} \subset \GL_d
  \end{displaymath}
  can be deduced without appealing to
  Proposition~\ref{pr-finite-ab-2}. Indeed, the immersion~$s_{\Delta}$
  is then a $4$-Sidon morphism by
  Proposition~\ref{pr-sidon-morphisms}\,(3), so we deduce from
  Proposition~\ref{pr-sidon-moments1}\,(2) that $\garith{M}$ (and hence
  also~$\ggeo{M}$, as before) contains $\SL_d$. (Precisely, we are in
  the excluded case of this statement, but we can observe that there are
  only finitely many ramified characters here, and that the assumption
  implies that the genus of~$C$ is at least five, so that the
  contribution to the $8$-th moment of the ramified characters is
  $$
  \ll \frac{1}{|k_n|^g}|k_n|^{8/2}\to 0,
  $$
  so that we do obtain the correct $8$-th moment.)
\end{remark}

\begin{remark} 
  In characteristic zero, Krämer and
  Weissauer~\cite{kramer-weissauer-larsen} have obtained closely
  related results, using more geometric methods in the case of the
  object $M=s_{\Delta*}\Qlb[1]$.
\end{remark} 

We now explain how Theorem \ref{thm:equidistribution-on-jacobians}
answers a question of Tsimerman, which was Katz's original
motivation. Let $\rho \colon \pi(C)^{\ab} \to \Cc^\times$ be a character
of finite order. By the Riemann hypothesis for curves over finite
fields, the Artin $L$-function\index{Artin $L$-function} $L_C(\rho, s)$
is a polynomial of degree $2g-2$ in the variable $T=q^{-s}$ all of whose
reciprocal roots have absolute value $\sqrt{q}$. We can then write
\[
L(\rho, \sfrac{T}{\sqrt{q}})=\det\left(1-T\Theta_{C\slash k, \rho}\right)
\] for a unique conjugacy class $\Theta_{C/k, \rho}$ in the unitary
group $\Un_{2g-2}(\Cc)$.

\begin{question}[Tsimerman]
  How are these conjugacy classes distributed as $\rho$ varies?
\end{question}

From now on, we shall normalize the characters as follows: we fix a
divisor $\Delta=\sum n_i x_i$ of degree one on $C$ and we only consider
those characters $\rho$ satisfying
\begin{displaymath}
  \prod \rho(\Frob_{\kappa(x_i), x_i})^{n_i}=1. 
\end{displaymath}
Through the isomorphism $\pi_1(C)^{\ab} \simeq \pi_1(A)$ induced by
$s_{\Delta}\colon C \hookrightarrow A=\mathrm{Jac}(C)$, such normalized
characters correspond to characters
$\rho \colon \pi_1(A) \to \Cc^\times$ satisfying
$\rho(\Frob_{k, 0_A})=1$. Since they are in addition supposed to be of
finite order, they arise via the Lang isogeny from the elements
of~$\charg{A}(k)$. 
Replacing $k$ with $k_n$, we obtain the corresponding characters in
$\charg{A}(k_n)$. Thus the following statement answers Tsimerman's
question when considering conjugacy classes associated to normalized
characters over~$k_n$ and taking $n\to+\infty$.

\begin{corollary}
  Let $C$ be a smooth projective geometrically connected curve of genus
  $g \geq 2$ over~$k$ with jacobian~$A$.
\begin{enumth}
\item\label{corKatz:item1} If $C$ is hyperelliptic, the hyperelliptic
  involution has a fixed $k$-point \hbox{$0 \in C(k)$,} and we use this
  point to define the embedding $C\to A$, then the conjugacy classes
  $(\Theta_{C/k,\chi})_{\chi\in\charg{G}(A_n),\chi\not=1}$
  are conjugacy classes in~$\USp_{2g-2}(\Cc)$ and become equidistributed
  with respect to the image of the Haar probability measure on the set
  of conjugacy classes.
\item\label{cor:curves2} If $C$ is not hyperelliptic and $(2g-2)\Delta$
  is a canonical divisor on $C$, then the conjugacy classes
  $(\Theta_{C/k,\chi})_{\chi\in\charg{A}(k_n),\chi\not=1}$ are conjugacy
  classes in~$\SU_{2g-2}(\Cc)$ and become equidistributed with respect
  to the image of the Haar probability measure on the set of conjugacy
  classes.
\end{enumth}
\end{corollary} 

\begin{proof}
  Consider the weight zero perverse sheaf
  $M_0=\overline\Qq_\ell(1/2)[1]$ on $C$ and set
  $M=s_{\Delta\ast}M_0$. For each rank one $\ell$-adic lisse sheaf $\Lb$
  on $A$, there are isomorphisms
  \begin{displaymath}
    H^i(A_{\bar k}, M \otimes \Lb) \simeq H^i(A_{\bar k},
    s_{\Delta\ast}(M_0 \otimes s_{\Delta}^\ast\Lb)) \simeq H^i(C_{\bar k},
    M_0 \otimes s_{\Delta}^\ast \Lb)
    \simeq H^{i+1}(C_{\bar k}, s_{\Delta}^\ast \Lb(1/2))
  \end{displaymath}
  by the projection formula and the exactness of $s_{\Delta\ast}$. It
  follows that $M$ has tannakian dimension
  \[
    -\chi(C_{\bar k}, s_{\Delta}^\ast \Lb)=2g-2,
  \]
  and moreover that all non-trivial characters are unramified for~$M$
  (since we are considering an abelian variety). By Theorem~\ref{th-2},
  it suffices therefore to prove that the arithmetic and geometric
  tannakian groups of $M$ coincide and are equal to $\Sp_{2g-2}$ in case
  \ref{corKatz:item1} and to $\SL_{2g-2}$ in case \ref{cor:curves2}.
  
  Assume $C$ is hyperelliptic with hyperelliptic involution $i$. Then
  the sheaf $M_0$ is geometrically isomorphic to $i^\ast \DD(M_0)$ and,
  because of the shift by $1$ in the definition of $M_0$, the
  corresponding self-duality is symplectic. Therefore, if in addition we
  assume that $i$ has a fixed $k$-point $0 \in C(k)$, which we use as
  divisor $\Delta$, Theorem
  \ref{thm:equidistribution-on-jacobians}\,\ref{thm:equidistribution-on-jacobians:item1}
  gives the equality $\ggeo{M}=\garith{M}=\Sp_{2g-2}.$
 
  If $C$ is not hyperelliptic and $(2g-2)\Delta$ is a canonical divisor
  on $C$, in view of Theorem
  \ref{thm:equidistribution-on-jacobians}\,\ref{item2-thmKatz}, it
  suffices to show that the arithmetic group $\garith{M}$ lies in
  $\SL_{2g-2}$. For this, we compute the determinant of the action of
  Frobenius on $H^1(C_{\bar k}, \Lb(1/2))$. Since this cohomology is
  even-dimensional, this is also the determinant of $-\mathrm{Fr}_k$,
  which is the constant in the functional equation for the $L$-function
  of $\Lb(1/2)$. By a classical result of
  Weil~\cite{weil-basic},\footnote{\ Which can also easily be recovered
    from the theory of local constants, applying the results of Deligne
    and Laumon (see Appendix~\ref{ch-app-product}).} in the case of
  $\Lb$ this constant is given by $q^{1-g}\rho_{\Lb}(\text{can})$ for a
  canonical divisor $\text{can}$, where $\rho_{\Lb}$ is the character
  associated to~$\Lb$, which factors through the jacobian.  Taking the
  half-Tate twist into account, along with the fact that
  $\rho_{\Lb}(\text{can})=1$ since $(2g-2)\Delta$ is a canonical divisor and
  characters are normalized to take the value~$1$ at $\Delta$, it follows
  that the determinant is trivial, as claimed.
\end{proof}


We conclude this section by a (partial) generalization of
Theorem~\ref{thm:equidistribution-on-jacobians} to the setting of
generalized jacobians arising in geometric class-field theory. This
gives a natural example of an application of our results where the
algebraic group~$G$ is not restricted to being either a torus, an
abelian variety or a unipotent group, but may involve all three of these
fundamental building blocks (see Remark~\ref{rm-gener-jac}). For
simplicity, we will only deal with the case where $C$ is not
hyperelliptic.

\begin{theorem}\label{th-generalized-jacobians}
  Assume that the curve~$C$ is not hyperelliptic.  Let $S$ be an
  effective divisor on the curve~$C$. Let $U$ be the complement of the
  support of~$S$ in~$C$. Let $\Delta$ be a divisor of degree one on $U$.  Let
  $J_S$ be the generalized jacobian of~$C$ relative to the divisor~$S$,
  and let $s_\Delta\colon U\to J_S$ be the natural immersion defined by
  $x\mapsto (x)-\Delta$.
  \par
  Let $M_0$ be a semiperverse object on~$U$, mixed of weights $\leq 0$
  and put $M=s_{\Delta!}M_0$. Let $d$ be the tannakian dimension of the
  semisimplification $\tilde{M}$ of the part of $\pH^0(M)$ which is pure
  of weight~$0$. Assume that $\tilde{M}$ is non-zero.
  \par
  Then we have $d\geq 2$, and either the arithmetic tannakian group
  of~$\tilde{M}$ contains $\SL_d$ or $\garith{\tilde{M}}$ is virtually
  central in $\GL_d$.
\end{theorem}

\begin{proof}
  We note that $M$ is a semiperverse object on~$J_S$ since $s_D$ is
  quasi-finite, and is mixed of weights $\leq 0$ by the Riemann
  Hypothesis.
  \par
  To check that $d\geq 2$, we use the general Euler--Poincaré
  characteristic formula (see Theorem~\ref{th-euler-poincare})
  as in~(\ref{eq-tan-jac-2}), to conclude.  We then need only observe
  that $s_D$ is a Sidon morphism by Proposition~\ref{pr-sidon-gen-jac},
  and apply Larsen's Alternative.
\end{proof}

\begin{remark}
  (1) Since we do not know in general if perverse sheaves on the group
  $J_S$ are generically unramified, the corresponding equidistribution
  statement is currently restricted to the distribution of the
  arithmetic Fourier transforms
  $$
  \sum_{x\in U(k_n)}t_M(x;k_n)\chi(x)
  $$
  for $\chi\in\charg{J}_S(k_n)$.
  \par
  (2) Again because the group $J_S$ is \emph{a priori} fairly arbitrary
  here, we can not exclude the possibility that $\garith{\tilde{M}}$ is
  virtually central (e.g., finite), since we do not have currently a
  general version of Proposition~\ref{pr-finite-ab-2}. (In our case,
  since the jacobian of $C$ is a non-trivial quotient of $J_S$, we can
  expect that the statement should indeed extend.)
  \par
  (3) It is possible that $\widetilde{M}$ is zero; in this case, we have
  of course $d=0$, and the tannakian group is trivial.
\end{remark}

\section{The intermediate jacobian of a cubic threefold}
\label{sec-cubic}

Our second application involving abelian varieties is related to a very
classical and important construction in algebraic geometry, that of the
intermediate jacobian\index{intermediate jacobian} of a smooth cubic
threefold,\index{smooth cubic threefold} which was used by
Clemens and Griffiths to prove that these threefolds, over~$\Cc$, are
not rational (although they are unirational).

The geometric setting, which over finite fields goes back at least to
the work of Bombieri and Swinnerton-Dyer~\cite{bombieri-sd} (computing
the zeta function of smooth cubic threefolds over finite fields) is the
following.
\par
Let $k$ be a field of characteristic different from~$2$, and
let~$X\subset \Pp^4_k$ be a smooth cubic threefold over~$k$. We denote
by~$F(X)$\nomenclature{$F(X)$}{Fano variety of lines} the Fano scheme of
lines\index{Fano variety of lines} in~$X$, which is a smooth projective
and geometrically connected surface over~$k$ (see,
e.g.,~\cite[\S\,4]{beauville-theta} or~\cite[Lem.\,3]{bombieri-sd}
or~\cite[Cor.\,1.12, Th.\,1.16]{altman-kleiman}; this uses the fact that
the characteristic is different from~$2$). Let then
$A(X)$\nomenclature{$A(X)$}{Albanese variety of~$F(X)$}\index{Albanese
  variety of the Fano variety} be the
Albanese variety of~$F(X)$, which is known to be isomorphic to the
Picard variety of~$F(X)$\index{Picard variety of the Fano variety} (see, e.g.~\cite[Cor.\,4.3.3]{huybrechts}). It
has dimension~$5$, and if the base field is contained in~$\Cc$, then the
analytification of~$A(X)$ is canonically isomorphic to the intermediate
jacobian $J(X)$ of Griffiths, which is defined analytically in terms of
Hodge theory (this is due to Murre;
see~\cite[Prop.\,9]{beauville-theta}).

The Albanese morphism $s\colon F(X)\to A(X)$\index{Albanese morphism} is a closed immersion,
according to a theorem of Beauville~\cite[p.\,201,
cor.]{beauville-theta}.  If we view $A(X)$ as the Picard variety, then
the morphism $s$ can be identified geometrically with the map
sending a line $l\in F(X)$ to the divisor defined by the curve $C_s$
which is the Zariski-closure in $F(X)$ of the set of lines $l'\not=l$
such that $l'\cap l$ is not empty.

The problem we consider is then the following: if~$k$ is a finite field
of odd characteristic, what is the arithmetic tannakian group of the
perverse sheaf $M=s_*\Qlb[2](1)$ on~$A(X)$? (It is perverse because $s$
is a closed immersion, as in previous similar examples.) The
corresponding exponential sums are then
$$
S(M,\chi)=\frac{1}{|k_n|}\sum_{l\in F(X)(k_n)}\chi(s(l))
$$
for a character $\chi\in \charg{A(X)}(k_n)$.

Up to correcting a small oversight, the following answer is the analogue
over finite fields of a result of Krämer over~$\Cc$
(see~\cite[Th.\,2]{kramer-e6}).

\begin{proposition}\label{pr-kramer}
  Let $k$ be a finite field of characteristic different
  from~$2$. Let~$X$ be a smooth cubic threefold over~$k$, and denote
  by~$F(X)$ the Fano scheme of lines in~$X$, by~$A(X)$ the Albanese
  variety of~$F(X)$, and by
  $$
  s \colon F(X)\to A(X)
  $$
  the natural closed immersion.
  \par
  Let $\ell$ be a prime different from the characteristic of~$k$, and
  let $M$ be the object $M=s_*\Qlb[2](1)$ on~$A(X)$.  The connected
  derived subgroup of the arithmetic tannakian monodromy group of the
  object $M$ of the category $\Ppbarith(A(X))$ is isomorphic to the
  exceptional group~$\mathbf{E}_6$.
\end{proposition}

For the proof, we will use the following lemma, whose proof was
communicated to us by Beauville.

\begin{lemma}[Beauville]\label{lm-beauville}
  With notation as above, there is no~$x\in A(X)$ such that
  $-s(F(X))=x+s(F(X))$, and there is no non-zero~$x\in A(X)$ such that
  $s(F(X))=x+s(F(X))$.
\end{lemma}

\begin{proof}
  We argue by contradiction.

  For the first assertion, if $x$ existed such that
  $-s(F(X))=x+s(F(X))$, then the involution $a\mapsto -x-a$ of~$A(X)$
  would induce an involution $i$ of the variety $F(X)$ with a finite
  number of fixed points. The quotient variety $F(X)/i$ is then a
  normal variety with only isolated ordinary double points as
  singularities. In particular, it is Gorenstein (see,
  e.g.,~\cite[Cor.\,21.19]{eisenbud}), so its canonical divisor
  $K_{F(X)/i}$, defined as the direct image of the canonical divisor
  of the smooth locus of~$F(X)/i$, is a Cartier divisor (see
  e.g~\cite[p.\,79]{kollar}).
  Since the projection
  $p\colon F(X)\to F(X)/i$ is étale outside of the set of fixed
  points, the canonical divisor of $F(X)$ is $K=p^*(K_{F(X)/i})$. This
  implies that $K^2=2(K_{F(X)/i}^2)$ is even. However, it is known
  that $K^2=45$, which is odd (see,
  e.g.,~\cite[Prop.\,4.6]{huybrechts}).
  \par
  For the second assertion, note that $s(F(X))=x+s(F(X))$ would
  imply that
  $$
  s(F(X))-s(F(X))=x+s(F(X))-s(F(X)),
  $$
  so that the theta divisor $\Theta(X)=s(F(X))-s(F(X))$
  \nomenclature[$T$]{$\Theta(X)$}{theta divisor}%
  \index{theta divisor}%
  satisfies $\Theta(X)=x+\Theta(X)$. However,
  Beauville~\cite[\S\,3,\,Prop.\,2]{beauville-theta} showed that
  $\Theta(X)$ is smooth except for a single singularity, so this
  equality can only happen if~$x=0$.
\end{proof}

\begin{remark}
  The cohomological analogue of this proposition is not true: for
  $\eps\in\{-1,1\}$, the cohomology class of $\eps s(F(X))$
  in~$H^6(A(X))$ is $\Theta^3/6$, where $\Theta$ is the cohomology
  class of the symmetric theta divisor $s(F(X))-s(F(X))$ (the fact
  that~$s(F(X))$ has the same class as $-s(F(X))$ is due to the fact
  that $x\mapsto -x$ acts trivially on even-degree cohomology groups;
  the computation in terms of~$\Theta$ is explained, e.g.,
  in~\cite[Cor.\,5.3.12,\,(i)]{huybrechts}).
\end{remark}

We now give a proof of Proposition~\ref{pr-kramer} adapting Krämer's
argument over~$\Cc$, the key point being the recognition criterion of
$\mathbf{E}_6$ in Proposition~\ref{pr-e6}.

\begin{proof}
  Since $F(X)$ is a smooth, projective and geometrically connected
  surface, the object~$M$ is a simple perverse sheaf on~$A(X)$. The
  tannakian dimension of~$M$ is equal to the Euler--Poincaré
  characteristic of $M$ over $\bar{k}$ (Proposition~\ref{pr-chi}), which
  is equal to the Euler--Poincaré characteristic of the Fano surface
  $F(X)$, which is~$27$ (a result of Fano, see,
  e.g.,~\cite[Prop.\,1.23]{altman-kleiman}).
  \par
  Let $\Theta(X)$ be the theta divisor $s(F(X))-s(F(X))$
  in~$A(X)$, and $i\colon \Theta(X)\to A(X)$ the closed immersion.  The
  object $M * M^{\vee}$ contains the object $N=i_*\Qlb[1]$ by the
  decomposition theorem (see~\cite[proof of Th.\,2]{kramer-e6}). This is
  also a simple perverse sheaf since $\Theta$ is a geometrically
  irreducible divisor (see, e.g.,~\cite[Prop.\,2]{beauville-theta}). The
  tannakian dimension of~$N$ can be computed as in~\cite[Cor.\,6]{kramer-e6}
  (or by lifting to characteristic~$0$, as can be done as in~\cite[Proof
  of Lemma\,5]{bombieri-sd}), and is equal to~$78$.
  \par
  To conclude using Proposition~\ref{pr-e6}, applied to the connected
  derived subgroup~$\bfG$ of $\garith{M}$, it suffices therefore to
  check that~$\bfG$ still acts irreducibly on the $27$-dimensional
  representation corresponding to~$M$.
  \par
  To see this, note that the neutral component~$(\garith{M})^{\circ}$
  acts irreducibly by Corollary~\ref{cor-translate-irred-var-ab}
  combined with Lemma~\ref{lm-beauville}. Then its derived group~$\bfG$ must
  also act irreducibly since
  $$
  (\garith{M})^{\circ}=C\cdot \bfG
  $$
  for some torus~$C$, which is central by irreducibility.
\end{proof}





It is natural to ask whether this proposition can also be proved using
the fourth moment criterion of Proposition~\ref{pr-larsen-e6} instead of
Krämer's criterion.
  
We have not fully succeeded in doing so, but we can show that the
question translates to an interesting geometric property of the cubic
threefolds. Conversely, this property follows in fact from the previous
proof, as we will now explain. 
  
In order to apply Proposition~\ref{pr-larsen-e6}, we need to check that
the object $M$ is not self-dual, that it has tannakian dimension~$27$
and that its fourth moment is $M_4(\garith{M})=3$.




Lemma~\ref{lm-beauville} implies that~$M$ is not self-dual. The second
property is derived as in the beginning of the previous proof. Now we
attempt to compute the fourth moment.
\par
We write $F=F(X)$ and~$A=A(X)$. We use the diophantine interpretation of
the fourth moment. Summing over all characters, we find as usual using
orthogonality that for $n\geq 1$, the formula
$$
\frac{1}{|A(k_n)|}\sum_{\chi\in\charg{A}(k_n)}|S(M,\chi)|^4=
\frac{1}{|k_n|^4} \sum_{\substack{(l_1,\ldots,l_4)\in
    F(k_n)^4\\s(l_1)-s(l_2)=s(l_3)-s(l_4)}}1
$$
holds. We rewrite this in the form
$$
\frac{1}{|k_n|^4} \sum_{(l_1,l_2)\in   F(k_n)^2}
N(l_1,l_2)
$$
where
$$
N(l_1,l_2)=|\{(l_3,l_4)\in F(k_n)^2\,\mid\,
s(l_1)-s(l_2)=s(l_3)-s(l_4)\}|.
$$
\par
For $s\in F(k_n)$, we have
$$
N(s,s)=|F(k_n)|,
$$
and hence
$$
\frac{1}{|k_n|^4} \sum_{\substack{(l_1,l_2)\in F(k_n)^2}} N(l_1,l_2)
=\frac{|F(k_n)|^2}{|k_n|^4} + \frac{1}{|k_n|^4} \sum_{l_1\not=l_2\in
  F(k_n)} N(l_1,l_2)
$$
\par
Since $F$ is a geometrically irreducible surface over~$k$, the first
term converges to~$1$ as $n\to+\infty$. To handle the second term,
consider the morphism given by the first projection
$$
f\colon F^2\times_A F^2\to F^2,
$$
where the fiber product is defined by the morphisms
$(l_1,l_2)\mapsto s(l_1)-s(l_2)$ and
$(l_3,l_4)\mapsto s(l_3)-s(l_4)$. 
\par
We then have
$$
N(l_1,l_2)=|f^{-1}(l_1,l_2)(k_n)|=|f^{-1}(l_1,l_2)(\bar{k})^{\Frob_{k_n}}|,
$$
the number of fixed points of the Frobenius of~$k_n$ acting on the
fiber of~$f$. The fiber product $F^2\times_A F^2$ contains the
diagonal $\Delta=\{(l_1,l_2,l_1,l_2)\}$, and we denote by
$\widetilde{f}$ the morphism obtained by restriction
$$
\widetilde{f}\colon (F^2\times_AF^2)\setminus \Delta\to F^2.
$$
\par
We then deduce that
$$
N(l_1,l_2)=1+\widetilde{N}(l_1,l_2)\quad\text{ where }\quad
\widetilde{N}(l_1,l_2)=|\widetilde{f}^{-1}(l_1,l_2)(\bar{k})^{\Frob_{k_n}}|,
$$
and hence
$$
\frac{1}{|k_n|^4} \sum_{\substack{(l_1,l_2)\in F(k_n)^2}} N(l_1,l_2)
=\frac{|F(k_n)|^2}{|k_n|^4} +\frac{|F(k_n)|(|F(k_n)|-1)}{|k_n|^4} +
\frac{1}{|k_n|^4} \sum_{l_1\not=l_2\in F(k_n)} \widetilde{N}(l_1,l_2).
$$
\par
A theorem of Beauville~\cite[Prop.\,8]{beauville-theta} implies that
$\widetilde{f}$ is generically finite of degree~$5$. By the Chebotarev
Density Theorem\index{Chebotarev density theorem} (see, e.g.,~\cite[Th.\,9.7.13]{katz-sarnak}) it
follows that
$$
\lim_{n\to+\infty}
\frac{1}{|k_n|^4} \sum_{l_1\not=l_2\in F(k_n)} \widetilde{N}(l_1,l_2)
$$
is equal to the number of orbits of the Galois group of the Galois
closure of $\widetilde{f}$ in its permutation
representation\index{permutation representation} on the
generic fiber of $\widetilde{f}$.
\par
The generic point~$\eta$ of $F^2$ is a pair of two disjoint lines
$\eta=(\widetilde{s}_1,\widetilde{s}_2)$. Beauville showed that the
points $(\widetilde{s}_3,\widetilde{s}_4)\in F^2$ such that
$(\widetilde{s}_1,\widetilde{s}_2,\widetilde{s}_3,\widetilde{s}_4)$ is
in the fiber over~$\eta$ are such that the lines $\widetilde{s}_3$ and
$\widetilde{s}_4$ are contained in the intersection~$S$ of $X$ and of
the projective $3$-space spanned by
$(\widetilde{s}_1,\widetilde{s}_2)$. Thus~$S$ is a smooth cubic surface,
and the lines $\widetilde{s}_3$ and $\widetilde{s}_4$ are elements of
the set $\Lambda$ of the five lines in~$S$ intersecting both
$\widetilde{s}_1$ and~$\widetilde{s}_2$; for these geometric facts,
see~\cite[p.\,203, rem.\,2]{beauville-theta} or~\cite[proof of
Cor.\,4.3.9]{huybrechts}.

We claim that the subgroup of the Galois group of the $27$
lines\index{$27$ lines of a smooth cubic surface} which
fixes the two lines $\widetilde{s}_1$ and $\widetilde{s}_2$ is the
Galois group of the set of seven lines
$\{\widetilde{s}_1,\widetilde{s}_2\}\cup \Lambda$ (see Lemma~\ref{lm-27}
below). Now it follows
from the work of Harris on Galois groups of enumerative problems
(see~\cite[p.\,718]{harris-enumerative}) that this permutation
representation is indeed transitive, in fact that it has image
isomorphic to the whole symmetric group $\mathfrak{S}_5$, \emph{if we
  take the base field to be $\Cc$ and the cubic threefold to be
  general}.
\par
One may expect this to also be true in our situation:
\begin{itemize}
\item the restriction of the base field should not be problematic
  (indeed, the fact that the ``generic'' Galois group of the $27$ lines
  on a smooth cubic surface is isomorphic to the Weyl group\index{Weyl group} of
  $\mathbf{E}_6$, which is the starting point of Harris's work, is known
  in all odd characteristics, by work of
  Achter~\cite[Prop.\,4.8]{achter-27});
\item the (four-dimensional) family of hyperplane
  sections\index{hyperplane sections} that we consider is dominant over
  the (also four-dimensional) space of hyperplane sections of the cubic
  hypersurface~$X$ (indeed, for any hyperplane~$H$ in $\Pp^4$
  intersecting $X$ in a smooth surface, we can pick two distinct lines
  $(l_1,l_2)$ in $X\cap H$, and the corresponding section is $H\cap X$);
\item and the family of all hyperplane sections of~$X$ is probably
  general enough for the result of Harris to extend. (This is in fact
  the most delicate point.)
\end{itemize}

If we assume that this expectation holds for~$X$, then we would deduce
that
\begin{equation}\label{eq-harris}
  \lim_{n\to+\infty}
  \frac{1}{|k_n|^4} \sum_{l_1\not=l_2\in F(k_n)} \widetilde{N}(l_1,l_2)
  =1
\end{equation}
holds, and hence conclude that
$$
\lim_{n\to+\infty} \frac{1}{|k_n|^4} \sum_{\substack{(l_1,l_2)\in
    F(k_n)^2}} N(l_1,l_2)=3.
$$
\par
Under this assumption, we therefore derive from
Proposition~\ref{pr-larsen-diophante} that $M_4(\garith{M})\leq 3$.
Since~$M$ is not of tannakian dimension~$1$, the fourth moment is 
equal to either~$2$ or~$3$. We can at least partially exclude the
first possibility as follows:

\begin{enumerate}
\item For ``most'' cubic threefolds, the abelian variety $A$ is
  absolutely simple (see Lemma~\ref{lm-simple-27} below for a precise
  statement). In this case, there are only finitely many characters
  $\chi\in\charg{A}$ which are not weakly-unramified, and for which
  $$
  |S(M,\chi)|=\Bigl|
  \frac{1}{|k_n|}\sum_{x\in F(k_n)}\chi(s(x))\Bigr|
  \ll |k_n|,
  $$
  so that
  $$
  \frac{1}{|\charg{A}(k_n)|}
  \sum_{\chi\notin \wunram{A}(k_n)}
  |S(M,\chi)|^4\ll \frac{|k_n|^4}{|k_n|^5}\to 0
  $$
  as $n\to+\infty$, and from Proposition~\ref{pr-larsen-diophante},
  the computation we have performed actually means that the fourth
  moment is equal to~$3$.
\item We may use the beginning of Krämer's proof to deduce that
  $M*M^{\vee}$ contains an irreducible summand of dimension~$78$,
  which excludes the possibility that the fourth moment be equal
  to~$2$.
\end{enumerate}

So under the above assumptions, we conclude that $M_4(\garith{M})=3$ and
we can then apply Proposition~\ref{pr-larsen-e6} (as in the previous
argument, we use Corollary~\ref{cor-translate-irred-var-ab} to deduce
that the neutral component of~$\garith{M}$ still acts irreducibly).

Now, going backwards, if we use Proposition~\ref{pr-kramer}, then we
\emph{do} know that the fourth moment of~$\garith{M}$ is equal to~$3$,
since the tannakian group is~$\mathbf{E}_6$. It follows that, at least
in the first of the above two situations, the limit
formula~(\ref{eq-harris}) must be true.

\begin{remark}
  In contrast with Theorem~\ref{thm:equidistribution-on-jacobians},
  Proposition~\ref{pr-kramer} will not extend to compute the fourth
  moment for perverse sheaves of the form $s_*M$ for a more general
  simple perverse sheaf $M$ on $F(X)$. One can expect that, in this
  case, the fourth moment should be equal to~$2$, but this seems
  difficult to prove.
\end{remark}

We now state and prove the two lemmas we used above. The first one is
certainly a standard fact in the study of the $27$ lines.

\begin{lemma}
  \label{lm-27}
  Let $S$ be a smooth cubic hypersurface in~$\Pp^3$ over an
  algebraically closed field. Let $l_1$ and $l_2$ be two
  disjoint lines in~$S$.  Let $\Lambda$ be the set of the five lines
  in~$S$ intersecting both~$l_1$ and~$l_2$. Any
  Galois-automorphism of the twenty seven lines that fixes the lines
  in $\{l_1,l_2\}\cup \Lambda$ is the identity.
\end{lemma}

\begin{proof}
  The key point in this computation is the fact that no line on~$S$ is
  disjoint from all lines in $\Lambda$.  More precisely, we use the
  classical description of $S$ as a blow-up of $\Pp^2$ in six points
  which are in general position (see,
  e.g.,~\cite[Prop.\,3.2.3]{huybrechts}), and the resulting partition
  of the $27$ lines in subsets
  \begin{gather*}
    E_1, \ldots, E_6\\
    L_{ij},\quad 1\leq i<j\leq 6\\
    L_1,\ldots, L_6,
  \end{gather*}
  with incidences described as follows:
  \begin{gather*}
    E_i\cap L_j\not=\emptyset\text{ if and only if } i\not=j,\\
    E_i\cap L_{i,j}\not=\emptyset\text{ for any } j,\\
    L_i\cap L_{i,j}\not=\emptyset\text{ for any } j,\\
    L_{i,j}\cap L_{k,l}\not=\emptyset\text{ for } \{i,j\}\cap
    \{k,l\}=\emptyset,
  \end{gather*}
  all other pairs of lines being disjoint (see,
  e.g.,~\cite[Rem.\,3.2.4, 3.3.1]{huybrechts}).
  \par
  We choose the blow-up, as we may, so that $l_1=E_1$
  and~$l_2=E_2$ (see~\cite[3.3.2]{huybrechts}). We then have
  $\Lambda=\{L_{12},L_3,L_4,L_5,L_6\}$.
  \par
  Let $\sigma$ be a Galois automorphism which fixes the seven given
  lines. Since $\sigma$ respects incidence relations, we see:
  \begin{enumerate}
  \item For any $i$, we have $\sigma(E_i)=E_i$. Indeed, assume
    that~$i=3$ for simplicity, since all cases are similar. Then $E_3$
    meets $L_4$, $L_5$, $L_6$, which implies that $\sigma(E_3)$ also
    intersects these three lines. But the only lines with this
    property are $E_1$, $E_2$ and~$E_3$; since $\sigma$ fixes the
    first two of these, we have $\sigma(E_3)=E_3$.
  \item For any $i$, we have $\sigma(L_i)=L_i$. Indeed, assume that
    $i=1$; from the previous point, the lines $L_{12}$, $E_2$, \ldots,
    $E_6$, which all meet $L_1$, are fixed by $\sigma$, so that
    $\sigma(L_1)$ fixes all of them. We see that the only line with this
    property is $L_1$, so that $\sigma(L_1)=L_1$.
  \item For any $i<j$, we have $\sigma(L_{i,j})=L_{i,j}$. We consider
    the example of $L_{1,3}$, the other cases being similar. The lines
    $E_1$, $E_3$, $L_1$, $L_3$ all meet $L_{1,3}$, and hence (by the first
    two points) also intersect $\sigma(L_{1,3})$. But this means that
    $\sigma(L_{1,3})$ must of one of the $L_{i,j}$, and the only one
    that has the desired property is $L_{1,3}$.
  \end{enumerate}
\end{proof}

The second lemma concerns the ``generic'' simplicity of the intermediate
jacobian. Explicit examples that show that this property is not always
valid are given for instance by Debarre, Laface and
Roulleau~\cite[Cor.\,4.12]{d-l-r}; for the Fermat threefold
$$
x_0^3+\cdots+x_4^3=0
$$
over $\Ff_p$, with $p\geq 5$, the intermediate jacobian is isogenous to
$E^5$, where $E$ is the Fermat curve $y_0^3+y_1^3+y_2^3=0$.
\index{Fermat curve and threefold}

\begin{lemma}\label{lm-simple-27}
  Let $k$ be a finite field of characteristic~$p>11$.  Let~$\mathcal{M}$
  be the coarse moduli space of smooth projective cubic threefolds
  over~$k$. For any integer~$n\geq 1$, let $ \mathcal{M}_s(k_n)$ be the
  subset of $X\in\mathcal{M}(k_n)$ such that the abelian variety $A(X)$
  is simple over~$k_n$.
  \par
  There exists $\delta>0$ such that the asymptotic formula
  $$
  |\mathcal{M}_s(k_n)|=|\mathcal{M}(k_n)|(1+O(|k_n|^{-\delta}))
  $$
  holds for $n\geq 1$.
\end{lemma}

\begin{proof}
  Fix a prime $\ell$ invertible in~$k$.  Let $\mcF$ be the lisse
  $\ell$-adic sheaf on~$\mathcal{M}$ parameterizing the cohomology
  group~$H^1(A(X)_{\bar{k}},\bQl)$. The geometric monodromy group
  of~$\mcF$ is the symplectic group~$\Sp_{10}$ by a result of
  Achter~\cite[Th.\,4.3]{achter-27} (based on semicontinuity of
  monodromy\index{semicontinuity of monodromy} and the extension to positive odd characteristic of a result
  of Collino~\cite{collino} over~$\Cc$, which states that the
  Zariski-closure of the image of~$\mathcal{M}$ in the moduli space
  $\mathcal{A}_5$ of principally polarized abelian varieties of
  dimension~$5$ contains the locus~$\mathcal{H}_5$ of jacobians of
  hyperelliptic curves of genus~$5$).

  Using the method in~\cite[\S\,6,\,\S\,8]{kowalski-sieve}, this implies
  that there exists $\delta>0$ such that the set $\mathcal{M}_i(k_n)$ of
  threefold $X\in\mathcal{M}(k_n)$ for which the characteristic
  polynomial of Frobenius acting on $H^1(A(X)_{\bar{k}},\bQl)$ is
  irreducible in~$\Qq[X]$ satisfies the asymptotic
  $$
  |\mathcal{M}_i(k_n)|=|\mathcal{M}(k_n)|(1+O(|k_n|^{-\delta}))
  $$
  for $n\geq 1$, and one deduces the lemma since
  $\mathcal{M}_i(k_n)\subset \mathcal{M}_s(k_n)$.
\end{proof}

\begin{remark}
  A qualitative form of the result, namely the equality
  $$
  \lim_{n\to +\infty}\frac{|\mathcal{M}_s(k_n)|}{|\mathcal{M}(k_n)|}=1,
  $$
  can be proved, \emph{mutatis mutandis}, for finite fields of all odd
  characteristic.  It also possible to improve this estimate to obtain
  absolute simplicity.
\end{remark}

\chapter{``Much remains to be done''}\label{sec-problems}

We conclude this book with a selection of open problems (related to
the results of this text) and questions (concerning potential
generalizations and more speculative possibilities).


\section{Problems}

\begin{enumerate}
\item \label{pb1} Prove a version of the vanishing theorem where the size
  of the exceptional sets is controlled by the complexity in all cases,
  and moreover where those sets have a clear algebraic or geometric
  structure (similar to that of \tacs\ for tori and abelian varieties).
\item Prove that any object is generically unramified for a general
  group, or find a counter-example to this statement.
\item Establish functoriality properties relating tannakian groups of
  $M$ on~$G$ (resp. $N$ on~$H$) with those of $f_*M$ (resp. of $f^*N$)
  when we have a morphism $f\colon G\to H$ of commutative algebraic
  groups.
\item Study the situation in families over a base like $\Spec(\Zz)$.
\item Find additional and robust tools to compute the tannakian group,
  or at least to determine some of its properties, which are applicable
  when Larsen's Alternative is not. In particular, find analogues (if
  they exist) of the local monodromy techniques for the Fourier
  transform on $\Gg_a$ (i.e, of the local Fourier transform functors of
  Laumon).

  We note that recent work of Lawrence and Sawin~\cite{LS} and
  Javanpeykar, Krämer, Lehn and Maculan~\cite{jklm} computes the
  tannakian group of many objects of the form $i_*\Qlb[d]$ on abelian
  varieties over~$\Cc$, where $i\colon X\to A$ is a closed immersion
  which is either a hypersurface (in the case of~\cite{LS}) or has
  dimension $<(\dim(A)-1)/2$ (in the case of~\cite{jklm}). It should be
  possible to extend their results to the situation over finite fields,
  and it would be interesting to see if it also leads to more cases with
  other perverse sheaves with similar support conditions. Moreover,
  Ji~\cite{ji} has also performed similar computations for $\Gg_m^n$.
  
\item Construct interesting concrete perverse sheaves where the
  tannakian group is an exceptional group. In this direction, we note
  that automorphic methods have been used by Heinloth, Ngô and
  Yun~\cite{h-n-y} to construct sheaves on~$\Gg_a$ with any of the
  exceptional groups as geometric monodromy groups, hence (taking
  inverse Fourier transforms of these) also to sheaves with these as
  tannakian groups. Moreover, Katz has shown the existence of examples
  involving $G_2$, for $\Gg_a$~\cite[Th.\,11.1]{gkm},
  $\Gg_m$~\cite[Ch.\,26,\,27]{mellin} and on some elliptic
  curves~\cite[Th.\,4.1]{katz_elliptic}. In the case of~$\Gg_m$, his
  result is of a ``statistic'' nature: in a certain family of objects
  whose trace functions are related to hypergeometric sums
  \emph{evaluated at a fixed~$a$}, he shows that for ``most'' values
  of~$a$ in~$\Gg_m$, the tannakian group is~$G_2$. After the first
  draft of this book was written, Zurbuchen~\cite{zurbuchen} improved
  this result by showing that (as expected by Katz) any (non-zero)
  value of~$a$ has the desired property, provided the characteristic
  of the finite field is large enough.
\item Find further applications!
\end{enumerate}

\section{Questions}

Many of the following questions are rather speculative and much more
open-ended that the problems above. They may not have any interesting
answer, but we find them intriguing.

\begin{enumerate}
\item \textit{For a given~$G$, what are the tannakian groups that may
    arise?}
  \par
  This is motivated in part by the striking difference concerning finite
  groups between $\Gg_a$ and $\Gg_m$ or abelian varieties. In the former
  case, the solution of Abhyankar's Conjecture gives a characterization
  of which finite groups will appear, and a recent series of works of
  Katz, Rojas-Léon and Tiep has shown that there are many possibilities,
  even when one restricts attention to Fourier transforms of general
  hypergeometric sheaves (see for instance~\cite{kt}) and $\Gg_m$). On
  the other hand, we have already mentioned that Katz proved that finite
  cyclic groups are the only possible finite geometric tannakian groups
  on~$\Gg_m$, and Corollary~\ref{cor-support-fingp} is a statement going
  in a similar direction for abelian varieties (although not yet as
  precise since it requires \emph{a priori} that the arithmetic
  tannakian group be finite).
\item \textit{For a given group~$G$, if $M$ is a
    semisimple perverse sheaf associated to a semisimple lisse
    sheaf~$\mcF$ on an open dense subset of~$G$, what (if any) are the
    relations between the ``ordinary'' monodromy group~$\Gg$ of~$\mcF$
    and its tannakian groups?}
  \par
  In particular, suppose that $\mcF$ has finite monodromy group; what
  constraints does that impose on the tannakian group of~$M$?  We note
  that there is one ``obvious'' relation: the tannakian group acts
  irreducibly on its standard representation if and only if the lisse
  sheaf~$\mcF$ is irreducible.
  \par
  Since this last fact amounts to the discrete Plancherel
  formula,\index{discrete Plancherel formula} or
  equivalently to a relation between the second moments of both
  groups, a more specific question could be: are there non-trivial
  inequalities between the moments of the monodromy group of~$\mcF$
  and those of the tannakian group? For instance, does there exist a
  constant~$c\geq 0$, independent of the size of the finite field $k$,
  such that
  $$
  M_{4}(\garith{M})\leq cM_4(\Gg),\quad \text{ and (or) }\quad
  M_{4}(\Gg)\leq cM_4(\garith{M})\, ?
  $$
  \par
  One can get trivial bounds, similar to the bounds for the norm of the
  discrete Fourier transform on $G(k_n)$ when viewed as a map from
  $L^{2m}(G(k_n))$ to $L^{2m}(G(k_n))$ for $m>1$ and $n$ varying, but
  this norm has been determined by Gilbert and
  Rzeszotnik~\cite[Th.\,2.1]{gilbert-rzeszotnik} and depends on~$n$.
  (On the other hand, a referee has pointed out that certain heuristic
  examples indicate that the question is most likely ``No'', and it will
  be interesting to confirm this rigorously.)

  Of course, the meaning of ``relation'' between the ordinary and
  tannakian groups could encompass very different aspects, and it is of
  interest to note that the papers of Lawrence and Sawin and of
  Javanpeykar, Krämer, Lehn and Maculan use their computations of
  tannakian groups of certain objects on abelian varieties to study the
  ordinary monodromy groups of families of subvarieties of abelian
  varieties (see~\cite[Th.\,5.6]{LS} and~\cite[Th.\,4.10]{jklm} for
  details).
\item \textit{Can one construct a ``natural'' fiber functor $\omega$ on the
    tannakian category for~$G$, similar to Deligne's fiber
    functor\index{Deligne's fiber functor} for $\Gg_m$? }
  \par
  This would lead to a definition of Frobenius conjugacy classes for all
  characters (by considering for any~$\chi$ the conjugacy class
  in~$\garith{M}$ corresponding to the fiber functor defined by
  $M\mapsto \omega(M_{\chi})$), and potentially provide useful extra
  information to help determine the tannakian group. This is not even
  clear in the case of~$\Gg_a$, but C. Ji~\cite[Th.\,2.9]{ji} has
  constructed an analogue of Deligne's fiber functor for $\Gg_m^n$ for
  all integers~$n\geq 1$.
\item \textit{Can one find an \emph{a priori} characterization of the
    families $(f_n)_{n\geq 1}$, where
    $f_n\colon \widehat{G}(k_n)\to \Cc$ is a function, that arise as the
    arithmetic Fourier transforms of trace functions of complexes, or of
    perverse sheaves, on~$G$?}
  \par
  More generally, is there a natural ``geometric'' object, with
  appropriate notions of sheaves, etc, on the ``space'' of characters
  of~$G$? A crucial test for such a geometric interpretation of the
  discrete Fourier transforms would be the definition of an inverse
  transform.

  This geometric description exists when $G$ is unipotent, since the
  Serre dual $G^{\vee}$ is also a commutative algebraic group, and the
  Fourier transform is defined as a functor from $\Der(G)$ to
  $\Der(G^{\vee})$, but such a strong ``algebraicity'' property does not
  hold for other commutative algebraic groups (see for
  instance~\cite[Example\,1.8]{boyarchenko-drinfeld}, or
  Remark~\ref{rm-lhat-different}).

  There are however some hints in a more positive direction:
  \begin{enumerate}
  \item Gabber and Loeser~\cite[Th.\,3.4.7]{GL_faisc-perv} have
    characterized perverse sheaves on tori in terms of the structure of
    their (coherent) Mellin transforms (which can also be defined for
    semiabelian varieties), and Loeser~\cite[Ch.\,4]{loeser-determinant}
    has defined a variant over finite fields taking the Frobenius
    automorphism into account.
    \par
    It would be of considerable interest to understand better the
    (essential) image of these Mellin transforms, and to obtain a
    geometric form of Mellin inversion in this context.
  \item Considering the well-established analogy of $\ell$-adic
    sheaves with $\mathcal{D}$-modules (the basic setup of Katz's
    work~\cite{esde}), it is well-understood in the complex setting
    that the Mellin transform of a $\mathcal{D}$-module is a
    difference equation (e.g., the Mellin transform $\Gamma(s)$ of the
    exponential satisfies the difference equation
    $\Gamma(s+1)=s\Gamma(s)$); see for instance the paper of Loeser
    and Sabbah~\cite{loeser-sabbah}.
  \end{enumerate}
    
\item \textit{Is there an analogous theory for non-commutative algebraic groups?}
  \par
  For instance, let $G$ be a reductive group over a finite field~$k$,
  such as~$\SL_d(k)$. Deligne--Lusztig Theory parameterizes the
  irreducible representations of $G(k_n)$ (or some other more convenient
  basis of the $\ell$-adic representation ring) in terms of pairs
  $(\Tt,\theta)$ of a maximal torus of $G$ over~$k$ and a
  character~$\theta$ of $\Tt(k)$ (see for
  instance~\cite[Ch.\,7]{carter}), and the corresponding series of
  representations have (essentially) constant dimension as $\theta$
  varies, so that the character values in such series are suggestively
  sums of a fixed number of roots of unity. The theory of character
  sheaves of Lusztig gives a geometric form of this theory.
  \par
  Are there equidistribution statements for the Fourier coefficients of
  suitably algebraic conjugacy-invariant functions on $G(k_n)$?  In the
  case of characteristic functions of conjugacy classes, this might lead
  to interesting consequences concerning the error term in the
  Chebotarev Density Theorem for Galois extensions with Galois group of
  the form $G(k_n)$.
  \par
  In the case of (possibly non-commutative) unipotent groups, the Serre
  dual still exists as a unipotent group; a theory of character sheaves,
  and of the Fourier transform has been studied by Lusztig and
  Boyarchenko--Drinfeld (see for instance the
  survey~\cite{boyarchenko-drinfeld}).
\item \textit{Is there an analogue of automorphic duality for other
    groups than $\Gg_a$?}
  \par
  What we mean by this is the following: in the case of a simple
  middle extension sheaf~$\mcF$ on~$\Gg_a$ over a finite field~$k$ that
  is pure of weight zero and not geometrically isomorphic to an
  Artin--Schreier sheaf, there is (by the Langlands
  correspondence,\index{Langlands correspondence} due to Lafforgue in
  this generality) an automorphic representation~$\pi$\index{automorphic
    representation} on some general
  linear group over the adèle ring of~$k(t)$ such that (among other
  properties) the $L$-functions of (twists of)~$\pi$ coincide (up to
  normalization) with those of (twists of) the Fourier transform
  of~$\mcF$. Automorphic methods and results are then available to study
  the Fourier transform of~$\mcF$.
  \par
  If $G$ is a commutative algebraic group which is different
  from~$\Gg_a$, are there objects of a similar nature as automorphic
  forms and representations that would ``correspond'' to the arithmetic
  Fourier transform of suitable perverse sheaves on~$G$?  Such objects
  would presumably have some kind of $L$-function, which would coincide
  with the $\lhat$-function that we have defined. In particular, is
  there such a theory for $\Gg_m$?
\end{enumerate}


\appendix
\numberwithin{equation}{chapter}

\chapter{Review of perverse sheaves}
\label{ch-app-perverse}

In this appendix, we summarize the basic definitions and facts about
$\ell$-adic perverse sheaves.  The fundamental reference for this
material is the work of Beillison, Bernstein, Deligne and
Gabber~\cite{BBD-pervers}. Other useful summaries of perverse sheaves
are provided by Katz in~\cite[\S 2.1\,to\,2.3]{katz-rls} and
in~\cite[\S\,1.1,\,1.2,\,1.5]{katz:MMP}. For basic material on trace
functions in this context, see also~\cite[\S\,1.1]{laumon-signes}.

\section{Complexes of $\ell$-adic sheaves}

In this appendix, we work over a field $k$ of characteristic $p$ and
fix a prime $\ell\neq p$. For $X$ a separated scheme of finite type
over $k$, one can define the triangulated category of complexes of
$\ell$-adic sheaves $\Der(X)=\Der(X,\Qlb)$.

For $M\in\Der(X)$, we write $\mcH^n(M)$ for the $n$-th cohomology sheaf
of $M$, which is an $\ell$-adic constructible sheaf. We denote by
$\tau^{\leq n}$ and $\tau^{\geq n}$ the truncation functors; for every
object~$M$, we have canonical maps $\tau^{\leq n}(M)\to M$ and
$M\to \tau^{\geq n}(M) $. The composite functor
$\tau^{\geq 0}\circ \tau^{\leq 0}$ is canonically isomorphic to
$M\mapsto \mcH^0(M)$.

For varying $X$, the categories $\Der(X)$ satisfy all the properties of
Grothendieck's formalism of the six functors (see~\cite[1.12]{D-WeilII}
or~\cite[2.2.18]{BBD-pervers} in the case when $k$ is finite or
algebraically closed, which suffices for this book).

More precisely, $\Der(X)$ is endowed with two bifunctors 
\[
  (M,N)\mapsto \mathrm{RHom}(M,N),\quad\quad (M,N)\mapsto M\otimes N
\]
from $\Der(X)\times \Der(X)$ to $\Der(X)$, and for a morphism
$f \colon X\to Y $ of finite type, we have functors
\[
  M\mapsto Rf_*M\quad\quad M\mapsto Rf_!M
\]
from $\Der(X)$ to $\Der(Y)$, and functors
\[
M\mapsto f^*M\quad\quad M\mapsto f^!M
\]
from $\Der(Y)$ to $\Der(X)$. These functors satisfy the usual
compatibilities and adjunctions.

The dualizing complex for $X$ is defined to be $s^!\Qlb$, where
$s\colon X\to \Spec(k)$ is the structure morphism, and the Verdier dual
of $M\in \Der(X)$ is $\dual(M)=\mathrm{RHom}(M,s^!\Qlb)$. When $X$ is
smooth of pure dimension $d$, there is a canonical isomorphism
\begin{equation}
\label{eqA:duality}
s^!\Qlb\simeq \Qlb(d)[2d].
\end{equation}

Let $s\colon X\to \Spec(k)$ be the structure morphism. For any
object~$M$ of~$\Der(X)$ and~$i\in\Zz$, the $i$-th \emph{cohomology
  group} of~$X$ with coefficients in~$M$ (resp. \emph{cohomology group
  with compact support of~$X$ with coefficients in~$M$}) is given by
$$
\rmH^i(X,M)=\mcH^i(s_*M),\quad\quad
\rmH^i_c(X,M)=\mcH^i(s_!M),
$$
where we identify $\ell$-adic sheaves on~$\Spec(k)$ with $\Qlb$-vector
spaces.

When $X$ is a smooth curve, two other important results (the
Euler--Poincaré characteristic formula and Laumon's product formula for
epsilon factors) which are used in this book will be reviewed in
Appendix~\ref{ch-app-product}.

\section{Perverse sheaves}

\begin{definition}
  A complex $M\in \Der(X)$ is said to be \emph{semiperverse} if its
  cohomology sheaves satisfy
\[\dim \supp(\mcH^i(M))\leq -i \text{ for every }i\in \Zz,
\]
and $M$ is said to be \emph{perverse} if both $M$ and $\dual(M)$ are
semiperverse (see~\cite[(4.0.1')]{BBD-pervers}).

We denote by $\Perv(X)$ the full subcategory of perverse sheaves
in~$\Der(X)$, by ${}^\mathfrak{p}D^{\leq 0}(X)$ the full subcategory
of semiperverse sheaves, and by ${}^\mathfrak{p}D^{\geq 0}(X)$ the
full subcategory of objects $M$ such that $\dual(M)$ is semiperverse.
\end{definition}

\begin{theorem}
  The data of ${}^\mathfrak{p}D^{\leq 0}(X)$ and
  ${}^\mathfrak{p}D^{\geq 0}(X)$ give rise to a $t$-structure on
  $\Der(X)$. Its heart
  $\Perv(X)={}^\mathfrak{p}D^{\leq 0}(X)\cap {}^\mathfrak{p}D^{\geq
    0}(X)$ is therefore an abelian category.
\end{theorem}

\begin{example}
  Suppose that $X$ is smooth of pure dimension $d$, and let $\mcF$ be
  a lisse $\ell$-adic sheaf on $X$. Then the complex $\mcF[d]$ (i.e.,
  the sheaf~$\mcF$ put in degree~$-d$) is a perverse sheaf.

  Indeed, $\mcF[d]$ is clearly semiperverse and by (\ref{eqA:duality}),
  we see that $\dual(\mcF[d])=\mcF^\vee(d)[d]$, where $\mcF^\vee$ is the
  dual lisse sheaf of $\mcF$, so the dual of $\mcF[d]$ is also
  semiperverse.
\end{example}

If~$M$ is a complex in~$\Der(X)$ with support~$Y\subset X$, then there
exists an open dense subset~$U$ of~$Y$ such that the restriction
of~$M$ to~$U$ is lisse, i.e., all of the cohomology sheaves of~$M|U$
are lisse sheaves. We then say that \emph{$M$ is lisse on~$U$}.

One also defines
\[
{}^\mathfrak{p}D^{\leq n}(X)={}^\mathfrak{p}D^{\leq 0}(X)[n]
\text{ and }
{}^\mathfrak{p}D^{\geq n}(X)={}^\mathfrak{p}D^{\geq 0}(X)[n].
\]

The inclusion functors ${}^\mathfrak{p}D^{\leq n}(X)\subset \Der(X)$
and ${}^\mathfrak{p}D^{\geq n}(X)\subset \Der(X)$ admit right and left
adjoints, called the \emph{perverse truncation functors}, which are
denoted
\[
{}^\mathfrak{p}\tau^{\leq n}\colon \Der(X)\to {}^\mathfrak{p}D^{\leq n}(X)
\text{ and }
{}^\mathfrak{p}\tau^{\geq n}\colon \Der(X)\to {}^\mathfrak{p}D^{\geq n}(X).
\]
 
\begin{definition}
  The $n$-th \emph{perverse cohomology sheaf} of a complex
  $M\in \Der(X)$ is the perverse sheaf
  \[
    \pH^n(M)=\tau^{\leq 0}\tau^{\geq 0}(M[n])\in \Perv(X).
  \]
\end{definition}
 
Given a distinguished triangle $M\to N\to L\to$ in $\Der(X)$, we have a
long exact sequence
\begin{equation}
  \cdots \to \pH^i(M)\to \pH^i(N)\to \pH^i(L)\to \pH^{i+1}(M)\to \cdots
\end{equation}
of perverse cohomology sheaves.

Let~$M$ be a perverse sheaf on~$X$. From general principles, there are
convergent spectral sequences
\begin{equation}\label{eq-perverse-spectral-seq}
  E_2^{p,q}=H^p(X, \pH^q(M))\Longrightarrow H^{p+q}(X,M),\quad\quad
  E_2^{p,q}=H^p_c(X, \pH^q(M))\Longrightarrow H^{p+q}_c(X,M),
\end{equation}
which are called the \emph{perverse spectral sequences}.

We also have an equality
\begin{equation}\label{eq-decomp-phi}
M=\sum_{j\in\Zz}(-1)^j\,\pH^j(M)  
\end{equation}
in the Grothendieck group~$K(X)$ (see,
e.g.,~\cite[(0.8)]{laumon-signes}).

As with the standard $t$-structure, perverse cohomology sheaves give a
criterion to check whether a complex is semiperverse.

\begin{lemma}
  A complex $M\in \Der(X)$ is semiperverse if and only if $\pH^i(M)=0$
  for all integers $i\geq 1$.
\end{lemma}

See~\cite[Prop.\,1.3.7]{BBD-pervers} for the proof.

\begin{definition}
  An exact\footnote{\ Namely, a functor that commutes with shift and
    preserves distinguished triangles.} functor from $\Der(X)$ to
  $\Der(Y)$ is said to be \emph{left $t$-exact} (resp. \emph{right
    $t$-exact}) if it sends ${}^\mathfrak{p}D^{\geq 0}(X)$ to
  ${}^\mathfrak{p}D^{\geq 0}(Y)$ (resp. ${}^\mathfrak{p}D^{\leq 0}(X)$
  to ${}^\mathfrak{p}D^{\leq 0}(Y)$). It is said to be \emph{$t$-exact}
  if it is both left and right $t$-exact.
\end{definition}

The following important result is a direct consequence of Artin's
cohomological vanishing theorem (see~\cite[Th.\,4.1.1]{BBD-pervers}).

\begin{theorem}
  Let $f\colon X\to Y$ be an affine morphism, then $Rf_*$ is right
  $t$-exact and $Rf_!$ is left $t$-exact.
\end{theorem}

Since a closed immersion~$i$ is affine and proper (so that $Ri_*=i_!$),
we obtain as corollary:

\begin{corollary}\label{cor-closed-immersion}
  If~$i$ is a closed immersion, then~$i_*$ is $t$-exact.
\end{corollary}

More generally (see~\cite[Cor.\,4.1.3]{BBD-pervers}), the functors $f_!$
and $f_*$ are $t$-exact if $f$ is quasi-finite and affine.

A central result is the construction of the intermediate extension,
see~\cite[Cor.\,1.4.25]{BBD-pervers}.

\begin{proposition}\label{pr-intermediate-ext}
  Let $j\colon U\to X $ be a locally closed immersion. Let $M$ be a
  perverse sheaf on $U$. Then there exists a unique perverse sheaf
  $j_{!*}(M)$ on $X$, called the \emph{middle extension} or
  \emph{intermediate extension} of~$M$, such that
  \begin{itemize}
  \item There exists an isomorphism $j^*j_{!*}(M)\simeq M$.
  \item The perverse sheaf $j_{!*}(M)$ is supported on the closure
    $\overline{U}$ of $U$.
  \item The perverse sheaf $j_{!*}(M)$ has no subobject and no quotient
    supported on $\overline{U} \setminus U$.
  \end{itemize}
\end{proposition}

The most important example of this construction is when
$j\colon U\to X$ is a dense open immersion, with $U$ smooth of pure
dimension $d$, and $M=\mcF[d]$ for a lisse sheaf $\mcF$. Note that the
uniqueness implies that $\dual
(j_{!*}\mcF[d])=j_{!*}\mcF^\vee(d)[d]$. When $\mcF=\bQl$ is the
constant sheaf on~$U$, then $j_{!*}\bQl[d]$ is called the
\emph{intersection complex} of $X$.

\begin{example}\label{ex-perverse-curve}
  Let $X$ be a curve, $U$ a dense open subset of $X$ contained in the
  smooth locus of $X$ and $\mcF$ a lisse sheaf on $U$. Then
  $j_{!*}\mcF[1]=R^0j_*\mcF[1]$, where $j\colon U\to X$ is the open
  immersion.
\end{example}

The fundamental result concerning the category of perverse sheaves is
the following theorem~\cite[Th.\,4.3.1]{BBD-pervers}.

\begin{theorem}\label{thm:perv:simple}
  The category $\Perv(X)$ is artinian and noetherian, i.e., all objects
  are of finite length. Its simple objects are of the form
  $j_{!*}\mcF[d]$ where $j\colon U\to X$ is a locally closed immersion
  with $U$ smooth irreducible of dimension $d$ and~$\mcF$ is an
  irreducible lisse sheaf on $U$.
\end{theorem}

\begin{example}\label{ex-middle-ext}
  Let~$X$ be a smooth and geometrically connected curve.  Following
  Katz~\cite[\S\,7.3]{esde}, a constructible sheaf~$\mcF$ on~$X$ is
  called a \emph{middle extension sheaf}\index{middle extension sheaf}
  if, for any dense open set~$U$ of~$X$ such that~$\mcF$ is lisse
  on~$U$, with open immersion $j\colon U\to X$, the canonical morphism
  $\mcF\to j_*j^*\mcF$ is an isomorphism.
  \par
  There is a one-to-one correspondence between irreducible middle
  extension sheaves and simple perverse sheaves on~$X$ with support
  equal to~$X$; for a middle extension sheaf~$\mcF$, the corresponding
  simple perverse sheaf is~$\mcF[1]$. Conversely, for a simple
  perverse sheaf~$M$ with support equal to~$X$, of the form
  $j_{!*}\mcF[1]$ as in the theorem, the corresponding (irreducible)
  middle extension sheaf is $j_*\mcF$.
\end{example}

For simple perverse sheaves, the bounds on the dimension of the support
of the cohomology sheaves have an ``automatic improvement'' from the
bound given by the semi-perversity, except for~$\mcH^{-\dim(X)}$.

\begin{proposition}\label{pr-support-property}
  Let $M$ be a simple perverse sheaf on~$X$ which is not
  punctual. Then for any~$i\not=-\dim(\supp(X))$, we have
  \[
    \dim \supp(\mcH^i(M))\leq -i-1.
  \]
\end{proposition}

\begin{proof}
  This results from the classification of simple perverse sheaves and
  from the general description of the intermediate extension functor
  in~\cite[Prop.\,2.1.11]{BBD-pervers}.
\end{proof}

\begin{example}
  In the case of a curve, this property can be seen from
  Example~\ref{ex-perverse-curve}, since in that case any simple
  perverse sheaf which is not punctual is supported on a dense open
  subset.
\end{example}

We thank S. Morel for communicating us a proof of the following lemma
(see also \cite[Sublemma\,1.10.5]{katz:MMP}).

\begin{lemma}\label{lm-morel}
  Let $k$ be an algebraically closed field. Let $X$ be an irreducible
  projective variety of dimension $d$ over $k$, and let $M$ be a
  simple perverse sheaf on $X$ such that $\rmH^{-d}(X, M)$ is
  non-zero. Then the support of~$M$ is~$X$ and there exists an open
  immersion $j \colon U \hookrightarrow X$ such that
  $M=j_{!\ast}\Qlb[d]$.
\end{lemma}

\begin{proof}
  By the classification of simple perverse sheaves in
  Theorem~\ref{thm:perv:simple}, there exists a locally closed
  immersion $j \colon U \hookrightarrow X$ and a simple lisse sheaf
  $\mathcal{F}$ on $U$ such that
  $M=j_{!\ast}\mathcal{F}[\dim(U)]$. The cohomolog $\rmH^i(X,M)$
  vanishes unless $|i|\leq \dim(U)$ (see~\cite[4.2.4]{BBD-pervers}),
  so the assumption implies that~$\dim(U)=d$.  Besides, the formula
  for intermediate extensions from \cite[2.1.11]{BBD-pervers} implies
  the vanishing $\mcH^i(M)=0$ for $i<-d$. From the spectral sequence
  \[
    E_2^{p, q}=H^p(X, \mcH^q(M)) \Longrightarrow H^{p+q}(X, M),  
  \]
  we then get an isomorphim
  $\rmH^{-d}(X, M)\simeq H^0(X, \mathcal{H}^{-d}(M))$. The
  non-vanishing of this cohomology group implies that
  $\mathcal{H}^{-d}(M)$ has a global section. Hence, $\mathcal{F}$ has
  a global section and is therefore trivial.
\end{proof}

\section{Weights}\label{sec-weights}

In this section we assume that $k$ is a finite field of characteristic
$p$, and denote by $\bar k$ an algebraic closure. We also fix an
isomorphism $\iota \colon \Qlb \to \Cc$.

Let $q$ be a prime power and let $w\in \Zz$ be an integer. An element
$x\in \Qlb$ is said to be a \emph{$q$-Weil number of weight
  $w$}\index{$q$-Weil number of weight~$w$} if it
is algebraic over~$\Qq$, and if all the complex conjugates of
$\iota(x)$ are complex numbers with modulus $q^{w/2}$.

Let $X$ be a separated scheme of finite type over $k$ and $\mcF$ a
$\Qlb$-sheaf on $X$. Let $x$ be a closed point of $X$, with residue
field $k(x)$. Viewing $k(x)$ as a subfield of the fixed algebraic
closure $\bar k$ of $k$ defines a geometric point
$\tilde x\colon \Spec(\bar k)\to X$ supported at $x$. The geometric
Frobenius automorphism, inverse of $y\mapsto y^{{k(x)}}$ in
$\mathrm{Gal}(\bar k/ k)$, acts on the stalk $\mcF_{\tilde x}$ of $\mcF$
at $\tilde x$. We denote by $\frob_x$ this endomorphism of
$\mcF_{\tilde x}$, which is well-defined up to conjugacy.

\begin{definition}[{\cite[1.2]{D-WeilII}, \cite[5.1.5]{BBD-pervers}}]
  Let $X$ be a separated scheme of finite type over $k$, $\mcF$ a
  $\Qlb$-sheaf on $X$, and $M$ an object of~$\Der(X)$.
  \begin{enumerate}
  \item The sheaf $\mcF$ is \emph{punctually pure of weight $w$} if
    for every $x\in \abs{X}$, the eigenvalues of $\frob_x$ are
    $\abs{k(x)}$-Weil numbers of weight $w$.
    \index{punctually pure sheaf}
    
  \item The sheaf $\mcF$ is \emph{mixed} if it admits a finite
    filtration with successive quotients that are punctually pure. The
    weights of the non-zero quotients are called the punctual weights
    of~$\mcF$.\index{punctual weight}

  \item The complex $M$ is \emph{mixed} if all its cohomology sheaves
    are mixed. It is \emph{mixed of weights $\leq w$} if for every
    $i\in \Zz$, the sheaf $\mcH^i(M)$ is mixed with punctual weights
    $\leq w+i$.  It is \emph{mixed of weights $\geq w$} if its Verdier
    dual $D(M)$ is mixed of weights $\leq -w$.  \index{mixed complex of
      weights~$\leq w$}

  \item The complex $M$ is \emph{pure of weight $w$} if it is both mixed
    of weights $\leq w$ and of weights $\geq w$.  \index{pure complex of
      weight~$w$}
  \end{enumerate}
\end{definition}

\begin{remark}
  Deligne also defines $\iota$-weights and $\iota$-pure or mixed sheaves
  and complexes for any fixed isomorphism $\iota$; the notion above
  means that the objects are $\iota$-pure for all $\iota$
  (see~\cite[1.2.6]{D-WeilII}).
\end{remark}

We write $D_{\leq w}(X)$ and $D_{\geq w}(X)$ for the full subcategories
of $\Der(X)$ of objects mixed of weights $\leq w$ and $\geq w$. Thanks
to the shift in the definition, one has in particular
$D_{\leq w}[1]=D_{\leq w+1}$.

\begin{example}\label{ex-weights}
  (1) Suppose that $X$ is smooth of pure dimension $d$, and that
  $M\in \Der(X)$ is such that all its cohomology sheaves are lisse
  on~$X$. Then $M$ is pure of weight $w$ if and only if each sheaf
  $\mcH^i(M)$ is punctually pure of weight $w+i$.
  \par
  (2) The characterization of (1) does not apply in general. For
  instance, let~$X=\Aa^1$ be the affine line, and $j\colon \Gg_m\to X$
  the open immersion. Let~$M=(j_*\hypk_2)[1](1/2)$ be the Kloosterman
  sheaf of rank~$2$ shifted to be in degree~$-1$ and Tate-twisted to be
  of weight~$0$ (see~(\ref{eq-kloosterman-sheaf})). Then~$M$ is pure of
  weight~$0$. However, the cohomology
  sheaf~$\mcH^{-1}(M)=j_*\hypk_2(1/2)$ is not punctually pure of
  weight~$-1$: indeed, the stalk of this sheaf at~$0$ has rank~$1$ with
  a Frobenius eigenvalue of weight~$-2$.
  \par
  (3) If~$\mcF$ is a middle extension sheaf on~$X$ (see
  Example~\ref{ex-middle-ext} for the definition), we say that~$\mcF$ is
  pure of weight~$w$ if the perverse sheaf $M[\dim(X)](\dim(X)/2)$ is
  pre of weight~$w$. This is equivalent to the condition that the
  restriction of~$\mcF$ to any dense open set where it is lisse is
  punctually pure of weight~$w$.
\end{example}

Deligne's main theorem in~\cite[3.3.1,\ 6.2.3]{D-WeilII}, which directly
implies the most general form of the Riemann Hypothesis over finite
fields, is the following:

\begin{theorem}[Deligne]\label{th-deligne-riemann}
  Let $f\colon X\to Y$ be a separated morphism of schemes of finite
  type over~$k$. Then the functor $Rf_!$ sends $D_{\leq w}(X)$ to
  $D_{\leq w}(Y)$.
\end{theorem}

Using duality, one gets the following list of compatibilities of the
different functors on~$\Der(X)$ (see~\cite[5.1.14]{BBD-pervers}):
\begin{enumerate}
\item $Rf_!$ and $f^*$ preserve $D_{\leq w}$;
\item $Rf_*$ and $f^!$ preserve $D_{\geq w}$;
\item $\otimes$ sends $D_{\leq w}\times D_{\leq w'}$ to
  $D_{\leq w+w'}$;
\item $\mathrm{RHom}$ sends $D_{\leq w}\times D_{\geq w'}$ to $D_{\geq -w+w'}$;
\item Verdier duality exchanges $D_{\leq w}$ and $D_{\geq -w}$.
\end{enumerate}

\section{Trace functions}

We continue with the notation of the previous section, so that $X$ is an
algebraic variety over a finite field~$k$. 

Let $M$ be a complex in $\Der(X)$. For any integer $n\geq 1$ and
$x\in X(k_n)$, the stalk~$M_{\bar{x}}$ of~$M$ at a geometric
point~$\bar{x}$ above~$x$ is a complex of finite-dimensional
$\Qlb$-vector spaces, with only finitely many non-zero cohomology
spaces. The geometric Frobenius~$\Fr_{k_n}$ of~$k_n$ (the inverse of the
automorphisme $a\mapsto a^{|k_n|}$ of $k_n$) acts on $M_{\bar{x}}$, and
this action is independent of the choice of~$\bar{x}$ up to
conjugacy. We denote
$$
t_M(x;k_n)=\sum_{i\in\Zz} (-1)^i\Tr(\Fr_{k_n}\mid \mcH^i(M)_{\bar{x}}),
$$
which is also independent of~$\bar{x}$ above~$x$.

Whenever we have fixed the isomorphism $\iota_0\colon \bQl\to \Cc$ (as
in the whole of the main text, see Section~\ref{sec-conventions}), we
will view the trace function as a function $X(k_n) \to \Cc$ whenever
convenient.

\begin{definition}
  The \emph{trace function} $t_M$ of~$M$ is the data of the whole family
  of functions $(t_M(\cdot;k_n))_{n\geq 1}$.
\end{definition}

\begin{remark}
  We will sometimes write simply $t_M(x)$ for $t_M(x;k)$, when $x\in
  X(k)$.
  \par
  Viewing $X(k_n)$ as a subset of $X(\bar{k})$, we will also sometimes
  denote the stalk of~$M$ simply by~$M_x$, instead of introducing
  explicitly a specific geometric point over~$x$.
\end{remark}

Let $f\colon X\to Y$ be a morphism of algebraic varieties over~$k$.  The
following properties holds for objects $M_i$ and~$M$ of~$\Der(X)$ and
$N$ of~$\Der(Y)$:
\begin{align*}
  t_{\bQl}&=1\quad\text{($\bQl$ in degree~$0$)}\\
  t_{M[k]}&=(-1)^kt_M,\quad\quad t_{M(w)}=q^{-w/2}t_M\\
  t_{M_2}&=t_{M_1}+t_{M_3}\quad \text{ for any distinguished triangle }
  M_1\to M_2\to M_3\to, \\
  t_{M_1\otimes M_2} &= t_{M_1}t_{M_2}\\
  t_{f^*N}&=t_N\circ f,\quad\text{ i.e. }\quad t_{f^*N}(x;k_n)=
  t_N(f(x);k_n)\text{ for all $n\geq 1$ and $x\in X(k_n)$}
  \\
  t_{Rf_!M}(y;k_n)&= \sum_{\substack{x\in X(k_n)\\f(x)=y}}t_M(x;k_n).
\end{align*}

The last of these properties is a form of the Grothendieck--Lefschetz
trace formula (see~\cite[Exp.\,III, \S4]{SGA5}). Applied to a
complex~$M$ and to the structure morphism $X\to \Spec(k)$, it takes the
customary form
\begin{equation}\label{eq-trace-formula}
  \sum_{x\in X(k_n)}t_M(x;k_n)=
  \sum_{i\in\Zz}(-1)^i\Tr(\Fr_{k_n}\mid H^i_c(X_{\bar{k}},M)).
\end{equation}

Suppose that~$M$ is a semisimple perverse sheaf which is pure of
weight~$0$. Then by a result of Gabber
(see~\cite[proof\,of\,Prop.\,6.40]{sawin_conductors}), the equality
\begin{equation}\label{eq-gabber-conjugate}
  t_{\dual(M)}(x;k_n)=\overline{t_M(x;k_n)}
\end{equation}
holds for all~$n\geq 1$ and~$x\in X(k_n)$.

We also recall a useful injectivity statement:

\begin{proposition}\label{pr-injectivity}
  Let~$M_1$ and~$M_2$ be objects of~$\Der(X)$. The trace functions
  of~$M_1$ and~$M_2$ coincide, in the sense that
  $$
  t_{M_1}(x;k_n)=t_{M_2}(x;k_n)
  $$
  for all~$n\geq 1$ and all~$x\in X(k_n)$, if and only if the classes
  of~$M_1$ and~$M_2$ in the Grothendieck group~$K(X)$ are equal. In
  particular, if $M$ and~$N$ are semisimple perverse sheaves, then~$M$
  and~$N$ are isomorphic.
  \par
  Moreover, the classes of simple perverse sheaves form a basis of the
  $\Zz$-module $K(X)$.
\end{proposition}

This is proved in~\cite[Th.\,1.1.2]{laumon-signes}.

\chapter{The arithmetic Mellin transform over finite fields}
\label{ch-app-mellin}

We summarize here the most important results of Katz~\cite{mellin}
concerning the arithmetic Mellin transform on~$\Gg_m$. These results are
used in various places in the book.

\section{The category $\oldcal{P}$}\label{sec-category-p}

Let $k$ be a finite field with algebraic closure $\bar{k}$ and with
finite extensions $k_n/k$ for $n\geq 1$.

Katz defines a category $\oldcal{P}$ as the full subcategory of the
category of perverse sheaves on $\Gg_m$ over $\bar{k}$ whose objects are
perverse sheaves $N$ such that, for any perverse sheaf $M$ on~$\Gg_m$,
the objects $M *_! N$ and $M *_* N$ are both perverse
(see~\cite[Ch.\,2]{mellin} and~\cite[2.6.2]{katz-rls}).  Katz proved
that a perverse sheaf $N$ is an object of~$\oldcal{P}$ if and only if it
admits no shifted Kummer sheaf\index{Kummer sheaf} $\mcL_{\chi}[1]$ as
either subobject or quotient (this follows, e.g, from the combination
of~\cite[Lemma\,2.6.13, Lemma\,2.6.14, Cor.\,2.6.15]{katz-rls}).

The category $\oldcal{P}_{\arith}$ is defined as the full subcategory of
perverse sheaves on $\Gg_m$ over $k$ whose objects are those perverse
sheaves $N$ such that the base change of~$N$ to $\bar{k}$ is an object
of $\oldcal{P}$ (\cite[Ch.\,4]{mellin}).

Using the correct notion of exactness from the work of Gabber and
Loeser, the categories $\oldcal{P}_{\arith}$ and $\oldcal{P}$ are
neutral tannakian categories with the middle convolution
$$
M*_{\intt} N=\Imag(M *_! N\to M *_*N)
$$
as tensor operation (see~\cite[p.\,535]{GL_faisc-perv}).

The tannakian dimension of an object of~$\oldcal{P}$ is its
Euler--Poincaré characteristic. 

\section{Deligne's fiber functor and Frobenius conjugacy classes}
\label{sec-deligne-fiber-functor}

One remarkable canonical fiber functor on the tannakian
category~$\oldcal{P}$ is given by a theorem of Deligne.

\begin{theorem}[Deligne]\label{th-deligne-fiber-functor}
  Let $k$ be a finite field with algebraic closure $\bar{k}$. 
  Let $j_0\colon \Gg_m\to \Aa^1$ be the open immersion. Then the functor
  $$
  \omega_{\mathrm{Del}}\colon M\mapsto H^0(\Aa^1_{\bar{k}},j_{0!}M)
  $$
  is a fiber functor on the category $\oldcal{P}$.
\end{theorem}
\nomenclature[$omega$]{$\omega_{\mathrm{Del}}$}{Deligne's fiber functor}
\index{Deligne's fiber functor}

This is~\cite[Th.\,3.1 and Appendix]{mellin}.

Let $N$ be an object of $\oldcal{P}_{\arith}$ which is arithmetically
semisimple and pure of weight~$0$. Let $\garith{N}$ be the tannakian
group of the tannakian subcategory of~$\oldcal{P}_{\arith}$ generated by
$N$. Using Deligne's fiber functor and the tannakian formalism, Katz
defines a Frobenius conjugacy class $\Frf_{N,k_n}(\chi)$ in $\garith{N}$
for any $n\geq 1$ and any $\ell$-adic character $\chi$ of $k_n^{\times}$
by considering the fiber functor
$\omega_{\chi}\colon M\mapsto \omega_{\mathrm{Del}}(M_{\chi})$
(see~\cite[Ch.\,5]{mellin}).

\section{Finite tannakian groups}

\begin{theorem}[Katz]\label{th-katz-finite}
  Let $k$ be a finite field with algebraic closure $\bar{k}$. Let $N$ be
  a perverse sheaf in the category $\oldcal{P}_{\arith}$. Assume that $N$
  is arithmetically semisimple and pure of weight~$0$.
  \begin{enumth}
  \item If every Frobenius conjugacy class of~$N$ is quasi-unipotent,
    then the object $N$ is punctual.
  \item If the geometric tannakian group of $N$, i.e., the tannakian
    group of the tannakian subcategory of $\oldcal{P}$ generated by
    $N\otimes \bar{k}$, is \emph{finite}, then the object $N$ is
    punctual.
  \end{enumth}
\end{theorem}

These statements are~\cite[Th\,6.2 and Th.\,6.4]{mellin}.

\section{Hypergeometric complexes and sheaves}
\label{sec-hypergeometric}

Katz has also classified the perverse sheaves on $\Gg_m$ with tannakian
dimension~$0$ and~$1$. Indeed, since the tannakian dimension is equal to
the Euler--Poincaré characteristic in this case, the question is to
classify simple perverse sheaves~$M$ on~$\Gg_m$ with $\chi(M)=0$
or~$1$.

For Euler--Poincaré characteristic zero, we have:

\begin{proposition}\label{pr-ep-0}
  Let~$k$ be an algebraically closed field of characteristic $p>0$ with
  $p\not=\ell$. Let~$M$ be a simple perverse sheaf on~$\Gg_m$ with
  $\chi(\Gg_m,M)=0$. Then there exists a tame character~$\chi$
  of~$\Gg_m$ such that $M$ is isomorphic to~$\mcL_{\chi}[1]$.
\end{proposition}

This is~\cite[Prop.\,8.5.2]{esde}.

Katz has furthermore shown that the objects with Euler--Poincaré
characteristic~$1$ are exactly the \emph{hypergeometric complexes} on
$\Gg_m$, defined in~\cite[8.2, 8.3]{esde}.  \index{hypergeometric
  sheaf}\index{hypergeometric complex}

We recall the definition and notation for hypergeometric complexes.
Let~$k$ be a field of positive characteristic. Fix a pair $(m,n)$ of
non-negative integers and a non-trivial $\ell$-adic additive
character~$\psi$ of a finite subfield of $k$. Denote by
$j\colon \Gg_{m}\to \Aa^1$ the open immersion. Let
$$
\uple{\chi}=(\chi_1,\dots,\chi_n),\quad\quad
\uple{\rho}=(\rho_1,\dots,\rho_m)
$$
be two tuples of tame $\ell$-adic continuous characters
$\pi_1^t(\Gm)\to \Qlbt$.  Denote by $\bar \psi$ the inverse of
$\psi$ and write
$$
\bar{\uple{\chi}}=(\chi_1^{-1},\dots,\chi_n^{-1}),\quad\quad
\bar{\uple{\rho}}=(\rho_1^{-1},\dots,\rho_m^{-1}).
$$

The \emph{hypergeometric complex}
$\Hyp(!,\psi, \uple{\chi};\uple{\rho})$ in $\Der(\Gm)$ is then defined
inductively as follows:
\nomenclature[$H$]{$\Hyp(#!,\psi, \uple{\chi};\uple{\rho})$}{hypergeometric
  sheaf}

\begin{enumerate}
\item If $(m,n)=(0,0)$ then $\Hyp(!,\psi, \emptyset;\emptyset)$ is the
  skyscraper sheaf supported at 1.
\item If $(m,n)=(1,0)$ then
  $\Hyp(!,\psi, \chi;\emptyset)=j^*(\mcL_\psi) \otimes \mcL_\chi[1]$.
\item If $(m,n)=(0,1)$ then
  $\Hyp(!,\psi, \emptyset, \rho)=\inv^*(j^*(\mcL_{\bar \psi}) \otimes
  \mcL_{\bar \rho})[1]$.
\item If $(m,n)=(m,0)$ with $m\geq 2$ then
  $\Hyp(!,\psi, \uple{\chi};\emptyset)$ is the convolution
\[
  \Hyp(!,\psi, \chi_1;\emptyset)*_!\dots *_! \Hyp(!,\psi,
  \chi_n;\emptyset).
\]
\item If $(m,n)=(0,n)$ with $n\geq 2$ then
  $\Hyp(!,\psi, \emptyset;\uple{\rho})$ is the convolution
\[
  \Hyp(!,\psi, \emptyset;\rho_1)*_!\dots *_! \Hyp(!,\psi,
  \emptyset;\rho_n).
\]
\item In the general case, we have
\[
  \Hyp(!,\psi, \uple{\chi};\uple{\rho})=\Hyp(!,\psi,
  \uple{\chi};\emptyset)*_!\Hyp(!,\psi, \emptyset;\uple{\rho}).
\]
\end{enumerate}
\par
For $\lambda\in k^\times$, define also
\[
  \Hyp_\lambda(!,\psi, \uple{\chi};\uple{\rho})=[x\mapsto \lambda
  x]_*\Hyp(!,\psi, \uple{\chi};\uple{\rho}).
\]

It follows from these definitions that the general convolution
formula
$$
\Hyp_\lambda(!,\psi, \uple{\chi};\uple{\rho})*_!  \Hyp_\mu(!,\psi,
\uple{\chi}';\uple{\rho}')= \Hyp_{\lambda\mu}(!,\psi,
\uple{\chi},\uple{\chi}';\uple{\rho},\uple{\rho}')
$$
holds.

Let $K$ be an extension of~$k$. We say that a complex $M$ on $\Gg_m$
over~$K$ is \emph{hypergeometric over~$k$} if there exists
$\lambda\in k^{\times}$, an additive character $\psi$, and families of
tame multiplicative characters $\uple{\chi}$ and $\uple{\rho}$ over~$k$
such that $M\otimes K$ is isomorphic to
$\Hyp_\lambda(!,\psi, \uple{\chi};\uple{\rho})$, where the characters
involved are defined on~$\Gg_m$ over~$K$ by composition with the
canonical morphism $\pi_1^t((\Gg_m)_K)\to \pi_1^t((\Gg_m)_k)$.

Before stating some of the main results concerning hypergeometric
sheaves, we need a further definition: the tuples $\uple{\chi}$ and
$\uple{\rho}$ are said to be \emph{disjoint} if $(n,m)\not=(0,0)$ and
$\chi_i\not=\rho_j$ for all $i$ and~$j$.\index{disjoint tuples of
  characters}

\begin{theorem}[Katz]\label{th-hypergeometric}
  Assume that $k$ is algebraically closed.
  \begin{enumth}
  \item If the tuples $\uple{\chi}$ and $\uple{\rho}$ are disjoint, then
    $\Hyp_\lambda(!,\psi, \uple{\chi};\uple{\rho})$ is a simple
    nonpunctual perverse sheaf of Euler characteristic~$1$ on~$\Gg_m$.
  \item Let $K$ be an extension of $k$ and $\bar{K}$ an algebraic
    closure of~$K$. Let $M$ be a simple perverse sheaf~$M$ on $\Gg_m$
    over~$K$ with Euler--Poincaré characteristic equal to~$1$.  Then the
    base change $M\otimes \bar{K}$ of $M$ to $\bar{K}$ is hypergeometric
    over~$K$.
  \item Let $k$ be a finite field and
    $\bar{k}$ an algebraic closure
    of~$k$.  If the tuples $\uple{\chi}$ and
    $\uple{\rho}$ are disjoint, then the tannakian group of the
    hypergeometric object $\Hyp_\lambda(!,\psi,
    \uple{\chi};\uple{\rho})$ on~$\Gg_m$ over $\bar{k}$ is $\GL_1$.
  \end{enumth}
\end{theorem}

\begin{proof}
  The first statement follows from \cite[Th.\,8.4.2]{esde}, and the
  third is explained, e.g., in~\cite[proof of Cor.\,6.3]{mellin}.
  \par
  The second statement is \cite[Th.\,8.5.3]{esde} if
  $K=\bar{K}=k$.  Applied to
  $\bar{K}$ instead of an algebraic closure
  of~$k$, this gives the result except that we only know \emph{a priori}
  that $\lambda\in
  \bar{K}^{\times}$.  We need to check that in fact $\lambda\in
  K^{\times}$. To do this, we check the steps of the proof of \loccit,
  which is easily seen to provide this extra information.
  \par
  Say that $M$ is of type $(m,n)$ if
  $m$ is the dimension of the tame part of~$M$ at~$0$ and
  $n$ the dimension of the tame part at infinity.  The strategy of the
  proof is to reduce by induction to the case $m> n$, then to
  $n=0$ and finally to the case $m=n=0$.
  \par
  Each of these reduction steps follows a similar pattern.  First, up to
  tensoring $M$ by
  $\mcL_\Lambda$ for some tame continuous character
  $\Lambda$ of
  $\pi_1^t({\Gg}_{m,k})$, one can assume that the trivial character
  occurs in the local monodromy
  at~$0$. From Kummer theory, we have an isomorphism
  $\pi_1^t({\Gg}_{m,k})\simeq \widehat{\Zz}(1)_{p'}$; since
  $M$ is defined over $k(\eta)$, the character
  $\Lambda$ must be of finite order and hence is a character of
  $k^\times$. All the characters $\chi$ and
  $\rho$ appear as such $\Lambda$.
  \par
  After this tensoring step, one considers the Fourier transform
  $\ft_\psi(j_*M)$, and one checks that it is of type $(n, m-1)$, and is
  still a geometrically simple perverse sheaf of Euler
  characteristic~$1$.
  \par
  At the end of the induction, one is left either with a skyscraper
  sheaf, which must be supported on some $\lambda\in
  K^{\times}$ since $M$ is geometrically simple
  over~$K$, or with a perverse sheaf that is geometrically isomorphic to
  $\mcL_{\psi(\lambda x)}$ for some $\lambda\in
  \bar{K}^\times$, and since this sheaf is defined over
  $K$, we must have $\lambda\in K$, as desired.
\end{proof}

\begin{remark}\label{rm-hypergeo}
  (1) In~\cite[Ch.\,8]{esde}, Katz has also determined the geometric
  monodromy group of almost all hypergeometric sheaves.  We observe in
  passing that this computation has recently been used by Fresán and
  Jossen~\cite{fresan-jossen} to construct examples of $E$-functions
  that are not related to hypergeometric functions, answering a question
  raised by Siegel in his fundamental paper~\cite{siegel}.
  \par
  (2) Theorem~\ref{th-hypergeometric} is a key ingredient in the proof
  of the theorem of Gabber and Loeser that determines the
  group~$\Hh_{\intt}(\Gg^r_{m,\bar{k}})$ of isomorphism classes of
  objects on~$\Gg^r_{m,\bar{k}}$ with tannakian rank~$1$, which is
  explained in Example~\ref{ex-gl-rank-1}.  In fact, it is not very
  difficult to deduce from Theorem~\ref{th-hypergeometric}, (3), that
  the group~$\Hh_{\intt}(\Gg_{m,\bar{k}})$
  is isomorphic to
  $$
  \bar{k}^{\times}\times \Zz^{\Pi(\Gg_{m,\bar{k}},\bQl)},
  $$
  where we recall that $\Pi(\Gg_{m,\bar{k}},\bQl)$ denotes the set of
  continuous tame characters of~$\Gg_{m,\bar{k}}$ (see
  Section~\ref{sec:Fourier-Mellin}).
  \par
  An isomomorphism $\Phi$ between these groups is determined as follows:
  given $\lambda \in\bar{k}^{\times}$ and a function
  $f\in\Zz^{\Pi(\Gg_{k,\bar{k}},\bQl)}$, let $\uple{\chi}$ be the tuple
  whose distinct elements are the characters $\chi$ such that
  $f(\chi)\geq 1$, each repeated with multiplicity $f(\chi)$, and let
  $\uple{\rho}$ be the tuple whose distinct elements are the characters
  $\chi$ such that $f(\chi)\leq -1$, each repeated with multiplicity
  $-f(\chi)$. Then one has
  $$
  \Phi(\lambda, f)=\Hyp_\lambda(!,\psi, \uple{\chi};\uple{\rho}).
  $$
  \par
  Conversely, the function $f$ can be recovered from an element $M$ of
  $\Hh_{\intt}(\Gg_m)$ by looking at the tame characters
  appearing in the local monodromy at~$0$ and~$\infty$, and their
  multiplicities.
\end{remark}

Let now $k$ be a finite field, with $\psi$ a non-trivial additive
character of~$k$. Let $\lambda\in k^{\times}$ and let 
$$
M=\Hyp_\lambda(!,\psi, \uple{\chi};\uple{\rho})
$$
for tuples $\uple{\chi}$ and $\uple{\rho}$ of tame characters associated
to multiplicative characters $k^{\times}\to \bQl^{\times}$ (denoted in
the same manner).  The trace function of $M$ is then given by
$$
t_M(x;k)=(-1)^{m-n}
\sum_{\substack{(x_i)\in (k^{\times})^n,\ (y_j)\in (k^{\times})^m\\
    x_1\cdots x_n=\lambda^{-1}xy_1\cdots y_m}}
\psi\Bigl(\sum_{i=1}^n x_i-\sum_{j=1}^m y_j\Bigr) \prod_{i=1}^n \chi_i(x_i)
\prod_{j=1}^m \rho_j(y_j),
$$
with the obvious analogue for finite extensions of~$k$
(see~\cite[(8.2.7)]{esde}). 

For a multiplicative character $\chi\colon k^{\times}\to \Qlbt$, let
$$
\tau(\psi,\chi)=\sum_{x\in k^{\times}}\psi(x)\chi(x)
$$
denote the Gauss sums over~$k$. Then the arithmetic Mellin transform of
the hypergeometric complex~$M$ is
\begin{equation}\label{eq-mellin-hypergeom}
  \sum_{x\in k^{\times}}\chi(x)t_M(x;k)=\chi(\lambda)
  \prod_{i=1}^m\tau(\psi,\chi\chi_i)
  \prod_{j=1}^n\tau(\bar{\psi},\bar{\chi}\bar{\rho}_j)
\end{equation}
for $\chi\colon k^{\times}\to \Qlbt$ (a monomial in Gauss sums;
see~\cite[(8.2.7),\,(8.2.8)]{esde}).\index{Gauss sum}

In particular, if $n \geq 1$ and $\chi_i=1$ for all $i$, and if
$\uple{\rho}$ is empty and $\lambda=1$, we obtain the unnormalized \emph{hyper-Kloosterman sums}\index{hyper-Kloosterman sums}
$$
(-1)^n\sum_{\substack{x_1,\ldots,x_n \in
    k^{\times}\\x_1\cdots x_n=x}}\psi(x_1+\cdots+x_n).
$$
The corresponding hypergeometric complex
\begin{equation}\label{eq-kloosterman-sheaf}
  \mcK\ell_{n,\psi}=\Hyp(!,\psi,(1,\ldots,1);\emptyset)
\end{equation}
is called a \emph{Kloosterman complex}; it is of the form $\hypk_n[1]$
for a lisse sheaf~$\hypk_n$ of rank~$n$ on~$\Gg_m$, called the
\emph{Kloosterman sheaf} of rank~$n$ (see~\cite[Rem.\,8.4.3]{esde}).

\chapter{The product formula for epsilon factors}
\label{ch-app-product}

We recall in this Appendix the formula of Laumon~\cite{laumon-signes}
for the epsilon factor of an object of $\Der(X)$ on a curve~$X$, and
recall the main parts of the formalism of local epsilon factors. We also
include the general Euler--Poincaré characteristic formula. 

\section{The product formula}\label{sec-product-formula}

The results in this section are quoted directly
from~\cite[\S\,3]{laumon-signes}.

Let $k$ be a finite field of characteristic~$p$, with $k_n/k$ the
extension of~$k$ of degree~$n$ in an algebraic closure~$\bar{k}$
of~$k$.

Let~$X$ be a smooth projective curve over~$k$. We denote by~$[X]$
\nomenclature[$X$]{$[X]$}{set of closed points of~$X$}
the
set of closed points of~$X$. For a complex $M$ in $\Der(X)$, the
$L$-function of~$M$ is defined by the product
$$
L(M,T)=\prod_{x\in [X]}\det(1-T^{\deg(x)}\Fr_{k_{\deg(x)}}\mid
M_x)^{-1}.
$$
It satisfies the relation
$$
L(M,T)=\det(1-T\Frob_k \mid H^*(X_{\bar{k}},M))^{-1}
$$
and the functional equation\index{functional equation}
$$
L(M,T)=\eps(M)T^{a(M)}L(\dual(M),T^{-1}),
$$
where
\[
  a(M)=-\chi(X_{\bar{k}},M),\quad
  \eps(M)=\det(-\Frob_k \mid H^*(X_{\bar{k}},M))^{-1}.
\]

Laumon's product formula, which had been conjectured by Deligne, is an
expression for $\eps(M)$ in terms of local epsilon factors.\index{local epsilon factor}

Consider a fixed non-trivial $\ell$-adic additive character $\psi$
of~$\Ff_p$, and denote $\psi_k=\psi\circ \Tr_{k/\Ff_p}$. Furthermore,
consider a fixed non-zero meromorphic $1$-form $\omega$ on~$X$. 

\begin{theorem}[Laumon]\label{th-laumon}
  Suppose that~$X$ is connected. Let~$g$ be the common genus of all
  the connected components of~$X_{\bar{k}}$, and~$n\geq 1$ the number
  of these connected components.
  \par
  Let $M$ be an object of~$\Der(X)$ of generic rank~$r(M)$.  There
  exist specific local constants $\eps_x(M)$, depending on the choice
  of~$\omega$, such that
  \begin{equation}\label{eq-product-formula}
    \eps(M)=|k|^c\prod_{x\in |X|}\eps_x(M)
  \end{equation}
  where $c=n(1-g)r(M)$.
\end{theorem}

This is~\cite[Th.\,3.2.1.1]{laumon-signes}, defining (in the notation
of \loccit) the local factors by
\begin{equation}\label{eq-local-global}
  \eps_x(M)=\eps(X_{(x)},M|X_{(x)},\omega|X_{(x)}).
\end{equation}
\nomenclature[$eps$]{$\eps(X_{(x)},M|X_{(x)},\omega|X_{(x)})$}{local epsilon factor}

\section{Local epsilon factors}

We summarize here the basic identities and formal properties of the
local epsilon factors~$\eps_x(M)$ in Laumon's
Theorem~\ref{th-laumon}. The existence and uniqueness of these local
factors, subject to certain conditions, are given precisely by Laumon
in~\cite[Th\,3.1.5.4]{laumon-signes}; they were defined earlier by
Deligne~\cite{deligne-constantes}.

The local epsilon factors are attached to a triple $(T,M,\omega)$,
where $T$ is a strictly henselian local ring of equal characteristic
with residue field containing~$k$, $M$ is an object of $\Der(T)$
and~$\omega$ is a non-zero meromorphic $1$-form on~$T$.

The notation $\eps(X_{(x)},M|X_{(x)},\omega|X_{(x)})$
in~(\ref{eq-local-global}) refers to these factors with the subscript
$(x)$ referring to strict localization at~$x$.

We now recall the local exponents $a(T,M,\omega)$ and $a(T,M)$, which
require additional notation (see~\cite[3.1.5]{laumon-signes}):
\begin{enumerate}
\item We denote by~$v$ the valuation of~$T$, extended to $1$-forms by
  $v(adb)=v(a)$ if $v(b)=1$.
\item We denote by $t$ the closed point of~$T$ and by $\eta$ the
  generic point.
\item We denote by $\bar{t}$ (resp. $\bar{\eta}$) a geometric
  generic point of~$T$ above~$t$ (resp. above $\eta$).
\item We denote by $k_t$ the residue field of~$T$ at~$t$.
\item For an object $M$ of $\Der(T)$, we denote by $r(M_{\bar{\eta}})$
  (resp. $r(M_{\bar{t}})$) the generic rank of~$M$ (resp. the rank of
  the stalk at the closed point) and by $s(M_{\bar{\eta}})$ the Swan
  conductor; all of these are defined for an étale sheaf first and
  extended by additivity, see~\cite[\S\,2.2.1]{laumon-signes}.
  \index{Swan conductor}
\end{enumerate}

With these notation, the local conductor exponents are defined by the
formulas
\begin{align}
  a(T,M)&=r(M_{\bar{\eta}})+s(M_{\bar{\eta}})-r(M_{\bar{t}}),
  \label{eq-3152}\\
  a(T,M,\omega)&=a(T,M)+r(M_{\bar{\eta}})v(\omega).
  \label{eq-3151}
\end{align}
(see~\cite[(3.1.5.1),\,(3.1.5.2)]{laumon-signes}).%
\nomenclature{$a(T,M,\omega)$}{local exponent}%
\nomenclature{$a(T,M)$}{local exponent}%

In the global case, we will denote
$$
a_x(M,\omega)=a(X_{(x)},M|X_{(x)},\omega|X_{(x)}).
$$

Furthermore, for a lisse $\bQl$-sheaf $\mcF$ on the generic point~$\eta$
of~$T$, one defines
\begin{equation}\label{eq-eps0}
\eps_0(T,\mcF,\omega)=\eps(T,j_!\mcF,\omega),
\end{equation}
where $j\colon \{\eta\}\to T$ is the open immersion
(see~\cite[3.1.5.6,\, p.\,187]{laumon-signes}).

For a short exact sequence $0\to \mcF'\to \mcF\to \mcF''\to 0$, we have
\begin{equation}\label{eq-3157}
  \eps_0(T,\mcF,\omega)=
  \eps_0(T,\mcF',\omega)
  \eps_0(T,\mcF'',\omega).
\end{equation}
\nomenclature[$eps$]{$\eps_0(T,\mcF,\omega)$}{local epsilon factor}

The local epsilon factors satisfy (among other things) the following
properties (see, respectively, formula (3.1.5.6), formula (3.1.5.5) and
section 3.5.3.1 in~\cite{laumon-signes}):

\begin{enumerate}
\item For any lisse $\bQl$-sheaf~$\mcF$ of rank~$r$ on~$T$, the formula
  \begin{equation}\label{eq-3156}
    \eps(T, M\otimes \mcF,\omega)=
    \det(\frob\mid \mcF)^{a(T,M,\omega)}
    \eps(T,M,\omega)^{r}
  \end{equation}
  holds, where $\frob$ denotes the geometric Frobenius automorphism at
  the closed point~$t$ of~$T$.
\item For a non-zero rational function $a$ on~$T$, the formula
  \begin{equation}\label{eq-3155}
    \eps(T,M,a\omega)=
    \chi(a)|k_{t}|^{r(M_{\bar{\eta}})v(a)}
    \eps(T,M,\omega)
  \end{equation}
  holds, where $\chi$ is the character of the completion of the residue
  field at~$\eta$ associated, by local class field theory,\footnote{\
    Normalized as explained in~\cite[(3.1.4)]{laumon-signes}.} to the
  lisse sheaf $\det(M)|\eta$ on~$\eta$, viewed as a character of the
  local Galois group.
\item For a non-trivial multiplicative character~$\chi$ of the residue
  field~$k_t$ and the corresponding lisse Kummer sheaf~$\mcL_{\chi}$
  on~$\{\eta\}$, and for a uniformizer~$\pi$ at~$x$, we have
  \begin{equation}\label{eq-3531}
    \eps_0(T,\mcL_{\chi},d\pi)=\chi(-1)\sum_{a\in
      k_t^{\times}}\chi(a)\psi(\Tr_{k_t/\Ff_p}(a)).
  \end{equation}
\end{enumerate}

We also have the elementary shift formula
\begin{equation}\label{eq-shift}
  \eps(T,M[1],\omega)=
  \eps(T,M,\omega)^{-1}.
\end{equation}

\section{The Euler--Poincaré characteristic
  formula}\label{sec-euler-poincare}

We keep the notation of Section~\ref{sec-product-formula}. In
particular, $X$ is a smooth projective curve over a finite field $k$
with algebraic closure $\bar{k}$. We assume that~$X$ is geometrically
connected, and denote by~$g$ the genus of~$X$.

Let $M$ be a complex in~$\Der(X)$. For any point $x\in X(\bar{k})$, the
Swan conductor\index{Swan conductor}
\nomenclature[$swan$]{$\swan_x(M)$}{Swan conductor of a complex at~$x$}
$\swan_x(M)$ is defined by additivity from the case of a
$\bQl$-sheaf (in which case, it is defined for instance
in~\cite[(2.1.2.5)]{laumon-signes} or~\cite[Ch.\,1]{gkm}). Similarly,
the drop $\Drop_x(M)$ is defined by additivity from the drop
$$
\Drop_x(\mcF)=\rank(\mcF)-\dim(\mcF_x)
$$
of a $\bQl$-sheaf $\mcF$.
\nomenclature[$drop$]{$\Drop_x(M)$}{drop of a complex at~$x$}

\begin{theorem}[Grothendieck--Ogg--Shafarevich]\label{th-euler-poincare}
  Let $U\subset X$ be an open dense subset. Let $M$ be a complex
  in~$\Der(U)$, let $V$ be an open dense subset of~$U$ on which $M$ is
  lisse of generic rank~$r(M)$.
  \par
  We have
  $$
  \chi(U_{\bar{k}},M)=\chi(U_{\bar{k}},\bQl)r(M)-\sum_{x\in X(\bar{k})}
  \swan_x(M) -\sum_{x\in U(\bar{k})}\Drop_x(M),
  $$
  where $\chi(U_{\bar{k}},\bQl)=(2-2g)-|(X\setminus U)|$.
\end{theorem}

This statement follows from~\cite[Th.\,2.2.1.2]{laumon-signes}, which
corresponds to~$X=U$ (up to changes in notation) by applying this
result to $j_*M$, where~$j\colon U\to X$ is the open immersion, and
using the additivity of the Euler--Poincaré characteristic, in the
sense that
$$
\chi(X_{\bar{k}},j_*M)=
\chi(U_{\bar{k}},M)+\chi((X\setminus U)_{\bar{k}},i^*j_*M)
$$
with $i$ the closed immersion of $X\setminus U$ in~$X$.
\par
For the case of a $\Qlb$-sheaf, the statement is also given for
instance in~\cite[Ch.\,14]{mellin}.

We consider some special cases that appear in this book.

(1) If $U=X$ and $M=\mcF[1]$ for some $\Qlb$-sheaf~$\mcF$ of generic rank~$r$
on~$X$, then the formula becomes
\begin{equation}\label{eq-ep-sheaf}
  \chi(X_{\bar{k}},M)=(2g-2)r+\sum_{x\in
    X(\bar{k})}(\swan_x(\mcF)+\Drop_x(\mcF)).
\end{equation}

(2) If $U=\Gg_m\subset X=\Pp^1$ and $M=\mcF[1]$ for some
$\Qlb$-sheaf~$\mcF$ of generic rank~$r$ on~$\Gg_m$, then
\begin{equation}\label{eq-ep-gm}
  \chi((\Gg_m)_{\bar{k}},M)=
    \swan_0(\mcF)
    +\swan_{\infty}(\mcF)
    +\sum_{x\in\bar{k}^{\times}}    (\swan_x(\mcF)+\Drop_x(\mcF)).
\end{equation}

\chapter{Deligne's letter to Kazhdan}
\label{ch-app-letter}

\textit{We reproduce below the content of Deligne's letter to Kazhdan,
  in which the $\ell$-adic Fourier transform was defined for the first
  time (the typography is not faithfully reproduced).  } 

\par
\vskip 0.2cm
\hrule
\vskip 0.1cm
\par


\begin{flushright}
29-11-76
\end{flushright}

Dear \foreignlanguage{russian}{Каждан},

This is perhaps a partial answer to an old letter of yours. I thought to
the matter again because of some estimations of trigonometrical sums
Hooley asked me about. As I am in a hurry to continue writing up Weil
II, I will leave many open ends and soon turn to French.

\par

\underline{Theme}: many functions correspond to sheaves, and operations
on functions to operations on sheaves. What about harmonic analysis on
$\mathbb{G}_a$?

\medskip

\noindent\textcircled{a} If $X$ is a scheme $/\mathbb{F}_q$, we will
consider
\begin{enumerate}
\item[$\alpha$)] objects of the derived category
  $\mathrm{D}^b(X, \overline{\mathbb{Q}}_\ell)$
  \par
  $\downarrow$ by $\sum (-1)^i \mathrm{H}^i$
\item[$\beta$)] virtual $\ell$-adic sheaves [this means either: elements
  of the Grothendieck group of the abelian category of constructible
  sheaves --- or if possible and useful, objects of some Picard category
  having this $K^0$ as set of isomorphism classes of objects]
  \par
  $\downarrow$ by $\mathrm{Tr}(F_x^\ast, \mathcal{F}_{\bar x})$ 
  (this map is \underline{injective}) 
\item[$\gamma$)] ``functions'': a system of functions on the $X(\mathbb{F}_{q^n})$
\end{enumerate}

\medskip

\noindent\textcircled{b} Here are corresponding operations:
\par
On functions: $+$, $\cdot$, $\sum$ \hspace{2cm} On
$\alpha), \beta)$: $\oplus$, $\otimes$, $\mathrm{R}\pi_!$

\noindent
Convolution of functions: if $G$ is a group, and $K, L \in
\mathrm{D}^b(X, \overline{\mathbb{Q}}_\ell)$, one considers the product
$$
\pi \colon G \times G \to G,\quad\text{and}\quad K \boxtimes
L=\mathrm{pr}_1^\ast K \otimes \mathrm{pr}_2^\ast L,\text{ and }
$$
\[
K \ast L=\mathrm{R}\pi_!(K \boxtimes L) 
\]

\noindent
Kernel: given $Z\to X \times Y$ and
$K \in \mathrm{D}^b(Z, \overline{\mathbb{Q}}_\ell)$, this defines an
operation $\mathrm{D}^b(X) \to \mathrm{D}^b(Y)$
\[
  L_X \mapsto \mathrm{Rp}_{2!}(K \otimes \mathrm{Rp}_{1}^\ast
  L_X).
\]

\medskip


\noindent\textcircled{c} Now I want to consider Fourier transform.
 \par
 Let us choose
 $\psi \colon \mathbb{F}_p \to \overline{\mathbb{Q}}_\ell^\ast$.  If $f$
 is a function on $X$, we get a sheaf $\mathcal{F}(\psi f)$.\footnote{\
   Pull back by $f$ of the sheaf on $\mathbb{G}_a$, rank~$1$, defined by
   $\psi$ and Artin--Schreier $T^p-T=X$.}  Fourier transform, on
 $\mathbb{G}_a$, is given by the
kernel $\mathcal{F}(\psi(xy))$ on $\mathbb{G}_a \times \mathbb{G}_a$.

\medskip

\underline{Definition}:
$\underline{F}(K)=\mathrm{Rpr}_{2!}(\mathcal{F}(\psi(xy)) \otimes
\mathrm{Rpr}_{1}^\ast K)$

\medskip

\noindent
\underline{Proposition 1}:
$\underline{F}(K \ast L)=\underline{F}(K) \otimes \underline{F}(L)$
\hspace{2cm} (from
$\mathcal{F}(\psi(x(y'+y''))=\mathcal{F}(\psi(xy')) \otimes
\mathcal{F}(\psi(xy''))$)

\smallskip
\noindent
\underline{Proposition 2}:
$\underline{F}\underline{F}(K)=K^\vee(-1)[-2]$: \quad $\vee$ is for
``image by $x \mapsto -x$'', $(-1)$ for a Tate twist, and $[-]$ for décalage.

Kernels compose like expected: we have to compute
$\mathrm{R}\pi_! \mathcal{F}(\psi(x+z)y)$ for
$\pi \colon \mathbb{G}_a \times \mathbb{G}_a \times \mathbb{G}_a
\stackrel{(13)}{\longrightarrow} \mathbb{G}_a \times \mathbb{G}_a$,
one gets
$$
\begin{cases}
  \mathbb{Q}_\ell(-1)& \text{ on the diagonal, in degree $2$}
  \\
  0& \text{elsewhere}
\end{cases}
$$
hence the result.

\medskip
\noindent
It is convenient in such computations to forget writing $\psi$ and
writing $\int \cdots dy$ for a $\mathrm{R}\pi_!$.

\smallskip
\noindent
\underline{Remark}: this defines, via prop 1, an isomorphism
$\underline{F}(K \otimes L)(-1)[-2]=\underline{F}(K) \ast
\underline{F}(L)$.

For Plancherel formula, one suffer somewhat of not having complex
conjugation.  Let $\overline{F}$ be $F$ defined using $\psi(-x)$. Then

\noindent
a) inner product: $\langle K, L \rangle=\mathrm{R}\Gamma(K \otimes L)$

\noindent
b) \underline{Proposition} $\langle FK, \overline{F}L \rangle=\langle K, L \rangle(-1)[-2]$.

This boils down to the usual
$$
\int \psi((x'-x'')y)K(x)L(x'')dx' dx''
dy\underset{\substack{\uparrow\\\text{by} \int_y}}{=}\int
\delta^{(-1)[-2]}(x'-x'')K(x')L(x'')dx'dx''.$$

\medskip

Everything done above can be generalized to any abelian connected
unipotent group $U$. The dual $U^\ast$ is to be taken in Serre's sense
(it is natural only up to inseparable isogenies, but this does not
matter. For $n$ large enough, one has a pairing
\[
U \times U^\ast \fleche{\boldsymbol{\cdot}} W_n
\]
(better: the pairing is in the cowitt vectors
$W_{-\infty}=\displaystyle{\varinjlim_{\text{by }V}} W_n$). Given
$\psi \colon \mathbb{Q}_p \slash \mathbb{Z}_p=W_{-\infty}(\mathbb{F}_p)
\to \overline{\mathbb{Q}}_\ell^\ast,$ and using the sheaf given by the
Lang covering of $W_{-\infty}/\mathbb{F}_p$ and $\psi$, everything can
be repeated, \underline{with} $(-1)[-2]$ \underline{replaced by}
$(-d)[-2d]$ where $d$ is the dimension.

\medskip

This requires to be careful if one wants to consider $\mathbb{Q}_p$ as a
(ind pro quasi) unipotent algebraic group $/\mathbb{F}_p$.

\medskip

\noindent
\textcircled{d} Where $\underline{F}$ is, there should also be an action
of the metaplectic group! (here
symplectic). Let me work for $\mathbb{G}_a$, and for $p \neq 2$. The
most precise way of speaking I see is working over $\mathbb{F}_p$, with
kernels. [It gives more than actions of $\mathrm{SL}(2, k)$ on
$\mathrm{D}^b(\mathbb{G}_a, \mathbb{Q}_\ell), k/\mathbb{F}_q$.]

\medskip
\noindent
\underline{Wanted}:
$P \in \mathrm{D}^b(\mathrm{SL}(2) \times \mathbb{G}_a \times
\mathbb{G}_a)$, viewed as a family of kernels on
$\mathbb{G}_a \times \mathbb{G}_a$ parametrized by $\mathrm{SL}(2)$.
Plus ``$P_{g'} \cdot P_{g''}=P_{g'\cdot g''}$''

\noindent
We know what is wanted for generators:
\par\noindent
\smallskip
\begin{tabular}{ccccl}
  $U^-$ &  $\begin{pmatrix}1 & 0\\a&1 
  \end{pmatrix}$ & $\mapsto$ & $(\ \ \otimes
  \mcF(\psi(\frac{ax^2}{2}))$ & (noyau sur la diagonale)
  \\
  \noalign{\bigskip}  $H$ &  $\begin{pmatrix}\lambda & 0\\0&\lambda^{-1}
  \end{pmatrix}$ & $\mapsto$ & $(x\to \lambda x)_*(\ \ )$ & (noyau sur
  $y=\lambda x$)
  \\
  \noalign{\bigskip} $a\not=0$\ $U_0^+$ & $\begin{pmatrix}1 & a\\0&1
  \end{pmatrix}$ & $\mapsto$ &
  $\Bigl(\int_x\mcF\psi(\frac{a^{-1}x^2}{2})\Bigr)^{-1}
  \mcF\psi(\frac{a^{-1}x^2}{2})\,*$
  & (noyau: faisceau loc c${}^{\underline{t}}$, de rg~$1$, 
  \\
  &&&&en degré~$-1$)
\end{tabular}

An explanation: $\mathrm{R}\Gamma \mathcal{F}\psi(a^{-1} \frac{x^2}{2})$
is of dimension $1$, and degree $1$, and I take the dual one dimensional
vector space [in $a$: a sheaf], in degree $-1$.

\medskip
\noindent
En français:

Raisonnons un peu a priori. Comme ``fonctions'', on sait ce que sont les
noyaux cherchés. On cherche des faisceaux leur donnant naissance. Sur
$U^- \times H \times U_0^+ \times U^- \times \mathbb{G}_a \times
\mathbb{G}_a$, composant les générateurs, on trouve un faisceau
localement constant de rang $1$, placé en degré $-1$, qui convient. En
chaque point de $U^- \times H \times U_0^+ \times U^-$, comme fonction
de $x, y$, il est de la forme $\psi f$, pour $f$ une fonction qui, en
$x, y$ (sur $\mathbb{G}_a \times \mathbb{G}_a$) est quadratique
homogène. Regardons la surjection
\[
U^- \times H \times U_0^+ \times U^- \longrightarrow G-B^- \qquad\qquad (G=\mathrm{SL}(2), B^-=\left(\begin{smallmatrix} \ast & 0 \\ \ast & \ast \end{smallmatrix}\right)) 
\]
Puisque comme ``fonctions'' ce que nous cherchons existe, le faisceau
obtenu est constant sur les fibres de (cette application
$\times \mathbb{G}_a \times \mathbb{G}_a$).

\medskip
\noindent
\underline{Obtenu}: un faisceau de rang $1$, en degré $-1$, localement
constant, sur $(G-B^-) \times \mathbb{G}_a \times \mathbb{G}_a$.

\medskip

Pour compléter ce tableau, il est bon de comprendre en quel sens, pour
$a \to 0$, on a
$$
  \Bigl(\int \mathcal{F}\psi(a^{-1}\frac{x^2}{2})dx \Bigr)^{-1} \cdot \
  \ \mathcal{F}\psi(a^{-1} \frac{x^2}{2}) \quad\longrightarrow \quad
  \delta \quad (\text{faisceau $\mathbb{Q}_\ell$ en $x=0$})
$$
[où]
$$
\int \mathcal{F}\psi(a^{-1}\frac{x^2}{2})dx
$$
est un faisceau de rang $1$ (degré $-1$) sur la droite de $a$; ce
faisceau se trivialise sur le revêtement de la droite de $a$ donné par
$\sqrt{a}$, car
\[
  \int \mathcal{F}\psi(a^{-2}\frac{x^2}{2})dx=\int
  \mathcal{F}\psi(\frac{(a^{-1}x)^2}{2}dx=\int
  \mathcal{F}\psi(\frac{x^2}{2})dx \quad \text{par
    ch${}^{\underline{\text{nt}}}$ de variable}
\]
\noindent
Il correspond à une somme de Gauss; sur $\int \cdots$, \quad
$|\text{Frobenius}|=q^{1/2}$.

\noindent
Traçons le plan $a, x$\quad ; le faisceau considéré est défini pour
$a \neq 0$; il se ramifie (sauvagement) le long de $a=0$, et la
ramification est équisingulière pour $x\not=0$. Si $j$ est l'inclusion
de $a \neq 0$ dans le plan, on a
$$
\begin{cases}
  j_\ast(\text{faisceau})=j_!(\text{faisceau})& \text{ nul pour $a=0$}
   \\
   \mathrm{R}^1j_\ast(\text{faisceau})& \text{concentré en $(0, 0)$, où
     c'est $\delta$}
   \\
   \mathrm{R}^2j_\ast(\text{faisceau}) &\text{concentré en $(0,0)$}
\end{cases}
$$

\medskip

Ceci se vérifie assez facilement en éclatant $2$ fois $(0, 0)$, la 2ème
fois en éclatant (courbe exceptionnelle) $\cap$ (transformé pur de l'axe
des $x$) : on utilise
\[
\xymatrix{ 
& (\text{plan éclaté}) \ar[dr]^{\pi} & \\
(a \neq 0) \ar@{^{(}->}[ur]^{\widetilde{j}} \ar@{^{(}->}[rr] & & (\text{plan}) 
} \qquad \mathrm{R}j_\ast=\mathrm{R}\pi_\ast \mathrm{R}\widetilde{j}_\ast 
\]

\medskip
On contrôle en projetant sur la droite des $a$ : si $p_a$ est cette
projection, on a
$$
Rp_{a*}Rj_*=Rj_* Rp_{a*}
$$
[où] $Rp_{a*}$ donne $\mathbb{Q}_{\ell}$ sur la droite $a$, $\setminus
\{0\}$,  et $Rj_*$ sur cette droite $(a\not=0)\injecte a=0$, donc
$$
\begin{cases}
  \text{deg. } 0 & \mathbb{Q}_{\ell}\\
  \text{deg. } 1 & \mathbb{Q}_{\ell}(-1) \text{ en } 0.
\end{cases}
$$

\medskip
\noindent
Ceci nous dit ce que nous devons faire pour construire $P$:
\begin{enumerate}
\item[(a)] sur $U^+$, le noyau s'obtient à partir de 
\[
  \tau_{\leq 0}\Bigl(\mathrm{R}j_\ast\Bigl(\Bigl(\int
  \mathcal{F}\psi(a^{-1}\frac{x^2}{2})dx \Bigr)^{-1}
  \mathcal{F}(\psi(a^{-1}\frac{x^2}{2})) \Bigr)
\] sur $U^+ \times \mathbb{G}_a$, comme convolution.

\item[(b)] sur $G \times \mathbb{G}_a \times \mathbb{G}_a$, on prend le
  noyau déjà construit sur
  $(G-B^-)\times \mathbb{G}_a \times \mathbb{G}_a,$ et pour $j$
  l'inclusion dans $G \times \mathbb{G}_a \times \mathbb{G}_a,$ on lui
  applique $\tau_{\leq 0} \mathrm{R}j_\ast$.
\end{enumerate}

\bigskip
\noindent
Je me suis convaincu que la formule $P_g \cdot P_{g'}=P_{gg'}$ vaut au
sens le plus fort possible:
\begin{enumerate}
\item[a)] sur $G \times G \times \mathbb{G}_a \times \mathbb{G}_a$\quad , on prend $P_{g''}(y, z)P_{g'}(x, y)$.
\item[b)] on intègre par rapport à $y$:
  $(P \cdot P)_{g', g''}=\int dy \cdots$\quad sur\
  $G \times G \times \mathbb{G}_a \times \mathbb{G}_a$
\item[c)] si $\pi$ est $G \times G \to G \colon g', g'' \mapsto g'g''$,
  on a un isomorphisme
  $$
  (P\cdot P)=\pi^\ast P
  $$
\item[d)] on a une compatibilité pour un composé triple [en c), on a
  unicité à une constante près, et on normalise par ce qui se passe à
  l'origine].
\end{enumerate}

\bigskip

Bien sûr, tout ceci devrait valoir pour un espace vectoriel $V$, et
$\mathrm{Sp}(V \oplus V^\ast)$. Il est facile de se convaincre qu'on a
en tout cas un noyau $P_g(v, v')$ qui est un faisceau virtuel, et que
sur la cellule des $g \in \mathrm{Sp}$ où
$gV^\ast \cap V^\ast=0$, il est donné de façon naturelle par un faisceau
de rang $1$, localement constant, en degré $-n$. J'espère que le noyau
lui-même s'en déduit par une suite d'opérations $\tau_{\leq}j_\ast$,
avec un résultat localement constant de rang $1$ sur un sous-espace, en
degré $-k$, sur la strate $\dim(V^\ast/V^\ast \cap gV^\ast)=k$... (qu'on
ait un noyau ainsi stratifié doit pouvoir se vérifier par Fourier).

\bigskip
\noindent
\underline{Question}: Le foncteur
$K \mapsto (x \mapsto -x)_\ast \underline{\mathrm{RHom}}(K,
\mathbb{Q}_\ell)$ commute-t-il à l'action de $\mathrm{SL}(2)$\ ?
\par
\medskip
\noindent
\underline{Question bis}: pour $P_g$ le noyau, et $K$ sur
$\mathbb{G}_a$, a-t-on
\[
  \mathrm{Rpr}_{2!}(P_g \otimes \mathrm{pr}_1^\ast K)
  \stackrel{\sim}{\longrightarrow} \mathrm{Rpr}_{2\ast}(P_g \otimes
  \mathrm{pr}_1^\ast K) \quad ?
\]

\medskip
\noindent
\underline{Question ter}: y commute-t-il virtuellement -- au moins
virtuellement sur $\overline{\mathbb{F}}$\quad\ ?

\par
\vspace{2cm}
\begin{center}
  Bien à toi,
  \par
  \bigskip\bigskip
  \par
  \qquad\qquad P. Deligne
\end{center}

\chapter{Intuition for analytic number theorists}
\label{ch-app-analytic}

The goal of this informal appendix is to provide readers who have a
background in analytic number theory with some intuition and feeling for
objects such as $\ell$-adic complexes, perverse sheaves, or tannakian
categories, all of which are essential tools in this book.

The focus here concerns trace functions of \emph{more than one
  variable}.  On the other hand, the theory of trace functions in
\emph{one} variable is more accessible, as the algebraic objects can be
presented more concretely using Galois theory of function fields.  Some
familiarity with this point of view will certainly also be very helpful
in developing intuition. A very concise introduction can be found in the
Pisa survey of Fouvry, Kowalski and Michel~\cite{pisa}, and a more
detailed treatment is contained in the lectures of Michel at the
2016 Arizona Winter School~\cite{arizona}.

We fix a finite field $k$, and denote by $k_n$ the extension of $k$ of
degree~$n$ inside a fixed algebraic closure~$\bar{k}$.  For simplicity
of notation, we will mostly speak about trace functions on the affine
space~$\Aa^m$ for some integer $m\geq 0$. However, it will be implicit
that most of what we discuss can be done for any algebraic variety $Y$
over $k$ (and this is needed, for instance because we often naturally
wish to restrict a trace function to a subvariety, where some particular
property holds), for instance for powers of the multiplicative group
$\Gg_m$ (i.e., $Y$ such that $Y(k_n)=(k_n^{\times})^d$ for some
$d\geq 0$). The reader should keep in mind that for such a subvariety,
of dimension~$d\leq m$, the size of the finite set~$Y(k_n)$ of points
of~$Y$ with coordinates in~$k_n$ is approximately $|k_n|^{d}$ when $n$
is large.

Throughout, we fix a non-trivial additive character
$\psi\colon k\to \Cc^{\times}$ and, for $n\geq 1$, we define
\begin{align*}
\psi_n \colon k_n &\longrightarrow \Cc^{\times} \\
x &\longmapsto \psi(\Tr_{k_n/k}(x)).
\end{align*}
We finally note that we will completely ignore (here) the distinction
between $\Qlb$ and~$\Cc$.

\section{Trace functions}

The concrete origin for the use of methods of algebraic geometry and
étale cohomology in analytic number theory lies in trace functions, and
especially in exponential sums. Properly speaking, a trace function
on~$\Aa^m$ is the data of a family $(t_n)_{n\geq 1}$ of functions
$k_n^m\to \Cc$, and it is associated to some algebraic object~$M$, which
we call a ``coefficient object''.  \index{coefficient object} This
object is not uniquely determined by~$(t_n)$, but we will not worry
about this matter in this appendix.

The first examples of trace functions arise from polynomials
$f\in k[X_1,\ldots,X_m]$ by means of 
\begin{equation}\label{eq-simplest-trace}
  t_n(x_1,\ldots,x_m)=\psi_n(f(x_1,\ldots,x_m));
\end{equation}
the corresponding coefficient object is denoted by $\mcL_{\psi(f)}$. Many other examples are then obtained by applying various operations,
which are known to preserve the set of trace functions (these are
operations on the coefficient objects, which are reflected in a specific
operation at the level of trace functions). These operations include the
following, where we indicate the algebraic notation for the
corresponding coefficient objects:
\begin{itemize}
\item The constant function $1$ is associated to the coefficient object 
  $M=\bQl$.
\item The sum of the trace functions associated to $M_1$ and $M_2$ is
  associated to $M_1\oplus M_2$.
\item If $(t_n)$ is a trace function associated to~$M$, then
  $((-1)^kt_n)$ is a trace function for each integer $k\in \Zz$,
  associated to a coefficient denoted by $M[k]$ and called a ``shift''
  of $M$.
 
\item If $(t_n)$ is a trace function associated to~$M$, then
  $(|k_n|^r t_n)$ is a trace function for each integer~$r \in \Zz$,
  associated to a coefficient denoted by $M(-r)$ and called a ``(Tate)
  twist'' of $M$.
  
\item The product of the trace functions associated to $M_1$ and $M_2$
  is associated to $M_1\otimes M_2$.
\item If $f=(f_1,\ldots,f_d)\colon \Aa^m\to \Aa^d$ is a tuple of
  polynomials in $k[X_1,\ldots,X_m]$, and $s=(s_n)$ is a trace
  function on~$\Aa^d$ associated to a coefficient~$N$, then
  $$
  t_n(x_1,\ldots,x_m)=s_n(f(x_1,\ldots,x_m))
  $$
  defines a trace function $(t_n)$ on~$\Aa^m$, which we also denote by 
  $s\circ f$. The corresponding coefficient is $f^*N$.
\item If $f=(f_1,\ldots,f_d)\colon \Aa^m\to \Aa^d$ is a tuple of
  polynomials in $k[X_1,\ldots,X_m]$, and $t=(t_n)$ is a trace function
  on~$\Aa^m$, associated to a coefficient object~$M$, then
  \begin{equation}\label{eq-sum-fiber}
    s_n(y_1,\ldots, y_d)=\sum_{\substack{x\in k_n^m\\f(x)=y}}t_n(x)
  \end{equation}
  defines a trace function on~$\Aa^d$; the associated coefficient object is
  denoted by $Rf_!M$.
\end{itemize}

\begin{example}[Fourier transform]
  This formalism is already sufficient to explain Deligne's Fourier
  transform. Let $m\geq 1$ be an integer, and consider the projections
  $$
  p_1,\ p_2\colon \Aa^{2m}\to \Aa^m
  $$
given by 
  $$
  p_1(x_1,\ldots,x_m,y_1,\ldots,y_m)=(x_1,\ldots,x_m),\quad\quad
  p_2(x_1,\ldots,x_m,y_1,\ldots,y_m)=(y_1,\ldots,y_m).
  $$
  We write 
  \[
  X\cdot Y=X_1Y_1+\cdots+X_mY_m
  \]
  for variables $X_i$ and $Y_j$. This is a polynomial with coefficients
  in~$k$, so the functions
  $$
  F_n(x,y)=\psi_n(x_1y_1+\cdots +x_my_m)
  $$
  define a trace function $F=(F_n)$ on~$\Aa^{2m}$, associated to the
  coefficient object $\mcL_{\psi(X\cdot Y)}$.
  
  Let $t=(t_n)$ be a trace function on~$\Aa^m$ with coordinates $(x_1, \ldots, x_m)$. Then the discrete
  Fourier transforms $(\widehat{t}_n)$, which are defined for $n\geq 1$
  and $y\in k_n^m$ by
  $$
  \widehat{t}_n(y)=\sum_{x\in k_n^m}t_n(x)F_n(x,y)=
  \sum_{x\in k_n^m}t_n(x)\psi_n(x\cdot y), 
  $$
  also define a trace function $\widehat{t}=(\widehat{t}_n)$. Indeed,
  for any $y$, the set of all $x\in k_n^m$ can be identified with the
  set of $(x,y)\in k_n^{2m}$ such that $p_2(x,y)=y$, and we have
  $t_n(x)=t_n(p_1(x,y))$, so that if~$t$ is associated to the
  coefficient object~$M$, then the formalism above shows that $\widehat{t}$ is
  associated to
  $$
  \widehat{M}=Rp_{2!}(p_1^*M\otimes \mcL_{\psi(X\cdot Y)}).
  $$
\end{example}

\section{Weights and purity: lisse sheaves}

The formalism of trace functions is useful in analytic
number theory \emph{because} of Deligne's Riemann hypothesis over finite
fields. This also leads to some understanding of the important
qualitative differences between various types of trace functions---corresponding to classes of coefficients which may (for instance) be
lisse sheaves, constructible sheaves, complexes of constructible
sheaves, or perverse sheaves. We will try in this and the following
sections to provide the readers with some intuition of the concrete meaning of these
notions.

The key concept (due to Deligne) is that of a coefficient~$M$ which is
\emph{punctually pure}, or \emph{pure}, of some weight~$w\in\Zz$. The
main conceptual difficulty is that the meaning of this property for the
corresponding trace function is not straightforward in general.

The simplest case (from which the others will be derived) is that of $M$
which is a single ``lisse sheaf''. \emph{In that case}, the concrete
meaning\footnote{\ But not exactly the precise definition.} of $M$ being
\emph{punctually pure of weight~$w$}, in terms of the trace function $t=(t_n)$,
is that there exist
\begin{itemize}
\item an integer $r\geq 0$, the \emph{rank} of~$M$, 
\item for each $n\geq 1$ and~$x\in k_n^m$, a unitary matrix
$\Theta_M(x;k_n)\in \Un_r(\Cc)$, well-defined up to conjugacy, 
\end{itemize}
so that the following equality holds:  
$$
t_n(x)=|k_n|^{w/2}\Tr(\Theta_M(x;k_n)).
$$
In particular, note that this implies the estimate
$$
|t_n(x)|\leq r|k_n|^{w/2}
$$
for all~$n$ and~$x\in k_n^m$.

In the remainder of this appendix, we will sometimes say that a lisse
sheaf, or its trace function, is ``pure'' instead of the more correct
``punctually pure''.

\begin{remark}
  The matrix $\Theta_M(x;k_n)$ is not arbitrary in~$\Un_r(\Cc)$. For
  instance, its eigenvalues (which of course determine the trace) are
  Weil numbers of weight~$0$, i.e., algebraic numbers in~$\Cc$ for which
  all Galois conjugates have modulus~$1$. Moreover, if $n'$ is a multiple of~$n$, then $x\in k_n^m$ can also be
  viewed as an element of $k_{n'}^m$ through the inclusion $k_n \subset k_{n'}$, and the formula
  $$
  \Theta_M(x;k_{n'})= \Theta_M(x;k_{n})^{n'/n}
  $$
  holds (\ie the eigenvalues of the matrix $\Theta_M(x;k_{n'})$ are
  those of $\Theta_{M}(x;k_n)$ raised to the power~$n'/n$).
\end{remark}

As one can expect, the trace functions defined by the
formulas~(\ref{eq-simplest-trace}), associated to $\mcL_{\psi(f)}$, are of this type, with $r=1$, $w=0$, and the matrix $\Theta(x;k_n)$ reduced
to the single complex number of modulus one $\psi_n(f(x))$. Moreover, it
is also intuitively clear (and true) that some of the operations
discussed above will respect the special class of trace functions
associated to pure lisse sheaves.

For instance:
\begin{itemize}
\item If $t$ and~$t'$ are trace functions associated to objects $M$
  and~$N$ which are both lisse sheaves pure of (the same) weight~$w$,
  then $t+t'$ is also pure of weight~$w$; we have 
  \[
  \Theta_{M\oplus N}(x;k_n)=\Theta_M(x;k_n)\oplus \Theta_N(x;k_n).
  \]
\item If $t$ and~$t'$ are trace functions associated to objects $M$
  and~$N$ which are both lisse sheaves pure of weights~$w$ and~$w'$,
  respectively, then $tt'$ is also pure of weight~$w+w'$. In other
  words, $M\otimes N$ is still a lisse sheaf, pure of that weight; in
  fact, we have 
  \[
  \Theta_{M\otimes N}(X;k_n)=\Theta_M(x;k_n)\otimes \Theta_N(x;k_n).
  \]
\item If $f=(f_1,\ldots,f_d)\colon \Aa^m\to \Aa^d$ is a tuple of
  polynomials in $k[X_1,\ldots,X_m]$, and $s$ is a trace function
  on~$\Aa^d$ associated to a lisse sheaf of weight~$w$, then $s\circ f$
  is also pure of weight~$w$. In other words, $f^*N$ is still a lisse
  sheaf, pure of weight~$w$; in fact, we have
  \[
  \Theta_{f^*N}(x;k_n)=\Theta_N(f(x);k_n).
  \]
\end{itemize}

But elementary examples show that the crucially important operation of
``summing over the fiber'' (see~(\ref{eq-sum-fiber})) does not always
send a single lisse sheaf to a lisse sheaf, and may also not map a trace
function which is pure of some weight to another one.

\begin{example}\label{ex-trace-fns}
  (1) Let $m=d=1$ and $f\in k[X]$ a polynomial of degree $2$,
  viewed as a map from $\Aa^1$ to itself. We consider the trace function
  $(t_n)$ with $t_n(x)=\psi_n(x)$, associated to the lisse sheaf
  $\mcL_{\psi(X)}$ (of weight~$0$), and the trace function $(s_n)$
  defined by
  $$
  s_n(x)=\sum_{\substack{y\in k_n\\ f(y)=x}}t_n(y)=\sum_{\substack{y\in
      k_n\\ f(y)=x}}\psi_n(y),
  $$
  for $n\geq 1$ and $x\in k_n$, which is associated to the coefficient object 
  $Rf_!\mcL_{\psi(X)}$. For most $x$, the value of $s_n(x)$ is either $0$ (if $f(y)=x$ has no solutions in $k_n$) or a sum of two
  roots of unity, but for the single point $x_0=f(y_0)$, where $y_0$ is
  the unique zero of the derivative of~$f$, the value $s_n(x_0)$ is a
  single root of unity (note that $y_0$, and hence $x_0$, belongs to $k$, so it also belongs to $k_n$ for all $n$, but the value of $s_n(x_0)$ does vary with $n$). 

  (2) We consider $m=2$ and the trace function $(t_n)$ defined by
  $t_n(x,y)=\psi_n(xy^2)$ for $(x,y)\in k_n^2$. It is associated to the coefficient object 
  $\mcL_{\psi(XY^2)}$, which is pure of weight~$0$. Let $d=1$ and
  $f=X$. Then $Rf_!\mcL_{\psi(XY^2)}$ has the trace function $(s_n)$
  such that
  $$
  s_n(x)=\sum_{y\in k_n} \psi_n(xy^2)=\begin{cases}
    \text{a quadratic Gauss sum}& \text{ if } x\not=0,\\
    |k_n|&\text{ if } x=0.
  \end{cases}
  $$
\end{example}

Neither of these examples of trace functions are associated to a single
punctually pure lisse sheaf. However, it turns out that the underlying
reason is not the same. In Example~(1), the issue is that $(s_n)$ is
associated to a single constructible sheaf which is ``not lisse'' at the
point~$x_0$. In Example~(2), the issue is that $(s_n)$ is associated to
a ``complex'' of constructible sheaves, i.e., not to a single sheaf.

\section{Weights and purity: constructible sheaves and complexes}

In fact, the most general source of trace functions are \emph{(bounded)
  mixed complexes of constructible sheaves}.  We now try to outline the
concrete interpretation of these more general conditions.

The first step goes from a single lisse sheaf to a \emph{single
  constructible sheaf}. Such a sheaf is \emph{(punctually) pure of
  weight~$w$} if there is a ``stratification''
$$
\emptyset=X_0\subset X_1\subset \cdots\subset X_q=\Aa^m
$$
of $\Aa^m$, where $X_i$ is a closed subvariety of~$X_{i+1}$, so
that the restriction of~$M$ to each of the pieces~$X_{i+1}\setminus X_i$
is a single lisse sheaf, punctually pure of weight~$w$, and of some
rank~$r_i\geq 0$ (which in general depends on~$i$).

Concretely, for a given $x\in k_n^m$, there exists a unique $i$ such that
$x\in X_{i+1}\setminus X_i$, and then there exists a unitary matrix
$\Theta_M(x;k_n)$ of size $r_i$ such that
$$
t_n(x)=|k_n|^{w/2}\Tr(\Theta_M(x;k_n)).
$$

\begin{example}
  Example~(1) above is of this kind, with the stratification
  $$
  \emptyset\subset \{x_0\} \subset \Aa^1,
  $$
  and with $r_0=1$ and $r_1=2$. On $\{x_0\}$, the unique eigenvalue is
  $s_n(x_0)=\psi_n(y_0)$, viewing $x_0$ as belonging to $k_n$. On $\Aa^1\setminus \{x_0\}$, the two
  eigenvalues are either opposite (hence the trace is zero) if
  $x\notin f(k_n)$, or are given by $\psi_n(y)$, for $y$ ranging over
  the two roots of the quadratic equation~$f(y)=x$.
\end{example}

More generally, Deligne defined a \emph{mixed constructible sheaf} of
weights $\leq w$ by the condition that there is a filtration with
associated punctually pure quotients~$M_j$, each of some weight
$w_j\leq w$. Concretely, this implies that the trace function $t=(t_n)$
is given by
$$
t_n(x)=\sum_{j\in J}t_{n,j}(x)
$$
for some finite set~$J$, where each family $(t_{n,j})_{n\geq 1}$ is the
trace function of a constructible sheaf which is pure of weight
$w_j\leq w$.
 
Finally, the most general type of trace functions arises from
objects~$M$ that are \emph{complexes of constructible sheaves}. Such a complex gives in particular rise to a \emph{sequence} $(\mcH^i(M))_{i\in\Zz}$ of
constructible sheaves, with $\mcH^i(M)=0$ for all but finitely many $i$,
in such a way that
$$
t_n(x)=\sum_{i\in\Zz}(-1)^it_{n,i}(x)
$$
for all $n\geq 1$ and $x\in k_n$, where $(t_{n,i})_{n\geq 1}$ is the
system of trace functions for the constructible sheaf~$\mcH^i(M)$. (These sheaves are called the \emph{cohomology sheaves} of
the complex~$M$.)

\begin{example}
  Example~(2) above is obtained from a complex of constructible
  sheaves~$M$, where there are two non-zero pieces, namely $\mcH^1(M)$
  and~$\mcH^2(M)$.

  The sheaf~$\mcH^1(M)$ is constructible for the stratification
  $$
  \emptyset\subset \{0\}\subset \Aa^1,
  $$
  with the piece on~$\{0\}$ of rank~$0$, and the piece
  on~$\Aa^1\setminus \{0\}$ of rank~$1$, pure of weight~$1$, with the
  corresponding unique eigenvalue equal to the quadratic Gauss sum
  $$
  \sum_{y\in k_n} \psi_n(xy^2)
  $$
  for $x\in k_n\setminus \{0\}$.

  The sheaf $\mcH^2(M)$ is also constructible, for the same
  stratification (but this is not a general feature), with
  the lisse sheaf of rank~$0$ on $\Aa^1\setminus \{0\}$, and a piece of
  rank~$1$ of weight~$2$ at $\{0\}$, with eigenvalue $|k_n|$.
\end{example}

However, for a complex~$M$, the definition of what it means that $M$ is
\emph{pure of weight~$w$} is much more subtle than for a single
sheaf. In particular, it does \emph{not} mean that each piece $\mcH^i(M)$ is itself a punctually pure sheaf of
  weight~$w$. More precisely, one defines first the \emph{mixed
  complexes of weights $\leq w$}, which are those such that $\mcH^i(M)$
is a mixed constructible sheaf of weights $\leq w+i$ for any
$i\in\Zz$. There is then furthermore defined another complex $\dual(M)$,
called the \emph{Verdier dual} of~$M$, and $M$ is said to be pure of
weight~$w$ if $M$ is mixed of weights~$\leq w$ and~$\dual(M)$ is mixed
of weights $\leq -w$.

\begin{remark}
  (1) For a single lisse sheaf~$M$ which is punctually pure of
  weight~$0$, the corresponding complex has $\mcH^0(M)=M$ and
  $\mcH^i(M)=0$ for all $i\not=0$. One can prove that the Verdier dual
  is a complex~$\dual(M)$ such that $\mcH^{-2m}(\dual(M))$ is a lisse
  sheaf which is pure of weight $-2m$ and all the other cohomology
  sheaves vanish, so that the two definitions of purity coincide for
  lisse sheaves. In fact, the trace function of $\dual(M)$ is \emph{in
    this case} the complex conjugate of the trace function of~$M$.
  \par
  (2) In practice, if an analytic number theorist is interested in a
  single trace function (e.g., one that represents a concrete family of
  exponential sums which one is interested in estimating) and one is not applying further
  operations like $Rf_!$, then one can quite often reduce to the case of
  a single lisse sheaf. This is for example the case for the
  hyper-Kloosterman sums in two variables
  $$
  \Kl_3(x;k_n)=\frac{1}{|k_n|}\sum_{\substack{a,b,c\in
      k_n^{\times}\\abc=x}} \psi_n(a+b+c),
  $$
  or the famous sums
  $$
  FI(x,y;k_n)=\sum_{z\in k_n^{\times}} \Kl_3(xz;k_n)
  \Kl_3(yz;k_n)\psi_n(z)
  $$
  which arose in the work of Friedlander and Iwaniec on the ternary
  divisor function~\cite{friedlander-iwaniec}, and reappeared in the
  work of Zhang~\cite{zhang}.
  
  Indeed, if the exponential sum is mixed, this will often be clear from
  the definition, or from a preliminary analysis, and one can
  ``isolate'' the part of most interest (of highest weight usually),
  which will be associated to a punctually pure constructible
  sheaf. Then by restricting the set of definition according to a
  suitable stratification, one will ensure that one handles a lisse
  sheaf.
  
  For $m=1$, this second step means avoiding finitely many values of $x$
  where the sheaf has unusual behavior; for $m\geq 2$, this means
  avoiding those that satisfy some non-trivial polynomial equation
  $g(x_1,\ldots,x_m)=0$. These special parameters can then be handled
  separately---giving rise to a kind of inductive process which
  reflects exactly the algebraic stratification of the corresponding
  coefficient~$M$.
\end{remark}

One good explanation for the focus on mixed objects with bounded weights
can be found (a posteriori) from the statement of Deligne's most general
form of the Riemann hypothesis. In our context, it can be stated as
follows:

\begin{theorem}[Deligne]\label{thm:DRH}
  Let $(t_n)$ be a trace function on~$\Aa^m$ associated to a complex~$M$
  which is mixed of weights~$\leq w$. Let $f=(f_1,\ldots, f_d)$ be a
  tuple of polynomials in $k[X_1,\ldots,X_m]$.  The complex $Rf_!M$ is
  \emph{mixed} of weights $\leq w$ , and so its trace functions
  $$
  s_n(y)=\sum_{\substack{x\in k_n^m\\f(x)=y}}t_n(x)
  $$
  are also mixed of weights~$\leq w$.
\end{theorem}

\begin{remark}
  On the other hand, even if $M$ is a single lisse sheaf, punctually
  pure of weight~$w$, it is \emph{not always the case} that $Rf_!M$ is
  pure.
\end{remark}

A benefit of introducing these more general definitions is that all
operations now respect the property of being mixed for any trace
function, with a good understanding of how the weights may~change:
\begin{itemize}
\item The lisse sheaf $M=\bQl$ is pure of weight~$0$.
\item If $M_1$ and $M_2$ have weights $\leq w_1$ and $\leq w_2$,
  respectively, then $M_1\oplus M_2$ has weights \hbox{$\leq \max(w_1,w_2)$}
  and $M_1\otimes M_2$ has weights $\leq w_1+w_2$.
\item If $M$ has weights $\leq w$, then for any $k\in\Zz$, the shifted
  complex $M[k]$ has weights $\leq w-k$.
  \item If $M$ has weights $\leq w$, then for any $r\in\Zz$, the twisted
  complex $M(r)$ has weights $\leq w-2r$.
\item If $f=(f_1,\ldots,f_d)\colon \Aa^m\to \Aa^d$ is a tuple of
  polynomials in $k[X_1,\ldots,X_m]$, and $s=(s_n)$ is a trace function
  on~$\Aa^d$ associated to a mixed complex~$N$ of weights~$\leq w$, then
  $f^*N$ has weights $\leq w$.
\item If $f=(f_1,\ldots,f_d)\colon \Aa^m\to \Aa^d$ is a tuple of
  polynomials in $k[X_1,\ldots,X_m]$, and if $M$ has weights $\leq w$,
  then   $Rf_!M$ has weights $\leq w$ (this is again Deligne's Theorem \ref{thm:DRH}).
\end{itemize}
\par
All objects that occur in practice in analytic number theory\footnote{\
  And indeed more generally in algebraic geometry.} are mixed
complexes. This means that any trace function $(t_n)$ has a
decomposition
$$
t_n=\sum_{a\leq w\leq b} t_{n,w}
$$
for some $a$ and $b$ (independent of~$n$), where $(t_{n,w})_{n\geq 1}$
is a trace function associated to a complex which is pure of weight~$w$.

\section{Perverse sheaves}

There remains the task of attempting to explain a further fundamental
subclass of trace functions (hence of complexes), those associated to
\emph{perverse sheaves}. This is a distinguished class of complexes with
remarkable geometric and arithmetic properties. For analytic purposes,
the most important of these is maybe that the \emph{simple} perverse
sheaves provide a \emph{canonical basis} of the abelian group of trace
functions, and that if we restrict to pure perverse sheaves, then this
is in a natural sense a \emph{quasi-orthogonal basis} for the trace
functions of pure complexes of weight $0$. We will now explain these
properties.

The rigorous definition of perverse sheaves is of a similar nature to
that of pure complexes: it is the combination for both the complex~$M$
and its Verdier dual $\dual(M)$ of a relatively simple condition, called
\emph{semiperversity}.\footnote{\ The complication is that the Verdier
  dual is often difficult to compute.} The condition of
semiperversity concerns the size of the support of the cohomology
sheaves $\mcH^i(M)$ (which are intuitively the points~$x$ where
$\mcH^i(M)$ does not vanish; in the stratification in terms of lisse
sheaves, this is where these sheaves have non-zero rank): for any
$i\in\Zz$, the support of $\mcH^i(M)$ should be of dimension at most
$-i$. (In particular, if $i\geq 1$, then the support should be empty, so
$\mcH^i(M)$ should be zero then.)

Remarkably, this condition can be recovered intuitively from basic
analytic intuition (which highlights that it is extremely natural).
  
Thus consider a trace function $t=(t_n)$ associated to a complex $M$
on~$\Aa^m$ and assume that it is mixed of weights $\leq 0$. From the
analytic point of view, we are often in the situation where the
mean-square of the values of the trace function $t_n$ are bounded (after
some normalization maybe), and bounded away from zero, i.e., for $n$
large enough, we have
\begin{equation}\label{eq-perverse-norm}
  \sum_{x\in k_n^m}|t_n(x)|^2\asymp 1.
\end{equation}

For $i\in\Zz$, the cohomology sheaf $\mcH^i(M)$ should be
``essentially'' pure of weight~$i$ (rigorously, we only know that it is
mixed of weights $\leq i$). So the contribution to the sum above of the
$x$ in the support~$S_i$ of $\mcH^i(M)$ should be expected to be of
order of magnitude
$$
|k_n|^{2\cdot i/2}\times |S_i(k_n)|\approx |k_n|^{i+d_i}
$$
if $S_i$ has dimension $d_i$. Hence the
estimate~(\ref{eq-perverse-norm}) only has a chance to hold if
$i+d_i\leq 0$ for all~$i$, and this is \emph{precisely} the
semiperversity condition.

\begin{example}
  Consider a family of exponential sums of type
  $$
  \frac{1}{|k_n|^{m}} \sum_{y\in k_n^m}\psi_n(f(y)+x_1y_1\cdots +x_my_m)
  $$
  with parameters $(x_1,\ldots,x_m)\in k_n^m$ (these functions of $x$
  are the trace functions of a complex~$M$ which is a normalized form of
  Deligne's Fourier transform of the lisse sheaf $\mcL_{\psi(f)}$).
  \par
  We expect ``generic'' square-root cancellation,
  \index{square-root cancellation} so as $n$ varies, for ``most'' choices of
  $x\in k_n^m$, this sum should be of size about $|k_n|^{-m/2}$. Since
  $\mcH^i(M)$ is of weight~$\leq i$, and hence contributes terms of size
  typically expected to be~$|k_n|^{i/2}$, this expectation corresponds
  to the fact that $\mcH^i(M)$ should be ``generically'' zero
  unless~$i=m$, while $\mcH^{-m}(M)$ contributes a fixed number of
  complex numbers of modulus $\leq |k_n|^{-m/2}$.

  But for special values of $x$, those satisfying some non-trivial
  polynomial equation $g(x)=0$, one may obtain a larger sum than
  square-root cancellation. Experience teaches that usually this size
  only jumps by one factor $|k_n|^{1/2}$ (so the sum is about
  $|k_n|^{-m/2+1/2}$) if only this one condition is imposed; if it is
  bigger (say of size $|k_n|^{-m/2+1}$), this should mean that a second
  (independent) equation $h(x)=0$ holds, and so on.

  This ``stratification'' of bounds getting steadily worse only on
  smaller subsets corresponds to cohomology sheaves $\mcH^i(M)$
  (contributing terms of size $|k_n|^{i/2}$) vanishing outside of
  subvarieties of dimension at most $-i$.

  In the extreme case, the exponential sum is of size~$1$ (i.e., there
  is no cancellation at all) at worse for finitely many values of the
  parameters, corresponding to $\mcH^0(M)$ being supported on finitely
  many points.
  \par
  This particular example is at the root of the results of Katz, Laumon
  and Fouvry on stratification for additive exponential
  sums~\cite{fouvry,KL-fourier-exp-som,fouvry-katz}.  It should suggest
  to analytic readers that semiperversity is a relatively easy
  condition to check, and that it should be natural and ubiquitous in
  analytic number theory.
\end{example}

The following statement provides a concrete illustration of the
advantages of perverse sheaves.

\begin{theorem}\label{th-appendix-basis}
  The $\Zz$-module of trace functions on~$\Aa^m$ over~$k$ is generated
  by the trace functions of perverse sheaves, and the trace functions of
  \emph{simple} perverse sheaves form a basis.
\end{theorem}

The first statement is in fact very explicit. Indeed, if $t=(t_n)$ is an
arbitrary trace function, associated to a complex~$M$, one can define
(in addition to its ``usual'' cohomology sheaves $\mcH^i(M)$) its
\emph{perverse cohomology sheaves} $\pH^i(M)$, which are perverse
sheaves, zero for $|i|>m$, such that their trace functions
$(\pt_{i,n})_{n\geq 1}$ satisfy the equation
$$
t_n=\sum_{i\in\Zz}(-1)^i\ \pt_{i,n}
$$
for all $n\geq 1$. Furthermore, a complex $M$ is mixed of
weights~$\leq w$ if and only if each $\pH^i(M)$ is also mixed of
weights~$\leq w+i$ (similarly to the cohomology sheaves;
see~\cite[Th.\,5.4.1]{BBD-pervers}).

\begin{remark}
  To say that a complex $M$ is perverse is to say that its perverse cohomology sheaves are $M=\pH^0(M)$ and
  $\pH^i(M)=0$ for all~$i\not=0$.
\end{remark}

Up to the terminology and notation, the second statement of
Theorem~\ref{th-appendix-basis} is proved by Laumon
in~\cite[Th.\,1.1.2]{laumon-signes} (it was already mentioned by Deligne
in his letter to Kazhdan; see Appendix~\ref{ch-app-letter}). To
understand it, one must explain what are the simple perverse sheaves
which are mentioned there. We will content ourselves with stating the
quasi-orthonormality property which holds for a simple perverse sheaf
that is pure of weight~$0$.  It is another consequence of Deligne's
Riemann Hypothesis, proved by Katz, that if $t=(t_n)$ is the trace
function of a perverse sheaf~$M$, then
\begin{equation}\label{eq-quasi-ortho}
  \limsup_{n\to +\infty}\sum_{x\in k_n^m}|t_n(x)|^2=1
\end{equation}
\emph{if and only if} $M$ is \emph{simple}.

\begin{remark}
  One of the fundamental results of Beilinson, Bernstein, Deligne and
  Gabber \cite[Cor.\,5.3.4]{BBD-pervers} is that a simple perverse sheaf
  which is mixed, as a complex, is in fact \emph{pure} of some weight;
  since non-mixed complexes do not appear in practice, this means that
  simple perverse sheaves in analytic number theory are always pure of
  some weight, and the quasi-orthonormality characterization can be
  extended to all simple perverse sheaves, up to normalization.
\end{remark}

\begin{example}
  We can illustrate how useful this quasi-orthonormality statement can
  be to guess or understand some properties of perverse sheaves by
  noting that it strongly suggests a non-trivial property of simple
  perverse sheaves. Namely, let $M$ be a simple perverse sheaf, pure of
  weight~$0$, and generically non-zero (i.e., the support of~$M$ is all
  of~$\Aa^m$). If we repeat the argument leading to the guess of the
  semiperversity condition, we see that we expect that the contribution
  to
  $$
  \sum_{x\in k_n^m}|t_n(x)|^2
  $$
  of each non-zero cohomology sheaf $\mcH^i(M)$ should be of size
  $$
  \alpha_i |k_n|^{i+d_i}
  $$
  for some integer $\alpha_i\geq 1$, and comparison
  with~(\ref{eq-quasi-ortho}) indicates that $i+d_i$ will be $<0$ except
  for one single value of~$i$. Moreover, one knows that the cohomology
  sheaf $\mcH^{-m}(M)$ is generically non-zero, so this value must be
  $i=-m$, so that we expect that
  $$
  d_i\leq -i-1\quad \text{ for }\quad i\not=-m,
  $$
  which is stronger than the condition $d_i\leq -i$ derived from
  semiperversity only. This is indeed true (it is the improved support condition of
  Proposition~\ref{pr-support-property}).
\end{example}

\section{Tannakian categories}

The results of this book also rely in an essential way on another tool
that is most likely unfamiliar to analytic number theorists: the
formalism of tannakian categories.  In very rough terms, this refers to
a method to construct or define a \emph{group} (which in our case will
be the ``symmetry group''\index{symmetry group} that governs the
equidistribution properties of a trace function), by recovering it from
the way it acts on finite-dimensional $K$-vector spaces, for some 
algebraically closed field $K$ of characteristic zero (which can be considered to be~$\Cc$). That this is possible is indicated by the following
result:

\begin{theorem}[Tannaka]
  Let $\Gg$ be a compact group. Assume that for every
  finite-dimensional complex vector space $V$ on which the group~$\Gg$
  acts linearly, via a continuous homomorphism
  $\rho\colon \Gg\to \GL(V)$, we are given an invertible linear
  transformation $\alpha(\rho)\colon V\to V$, and suppose that these
  data satisfy the following conditions\emph{:}
  \par
  ``Whenever $\Gg$ acts by $\rho$ on~$V$ and by $\pi$ on $W$, we have
  $$
  \alpha(1)=\mathrm{Id}_{\Cc},\quad\quad
  \alpha(\rho\otimes \pi)=\alpha(\rho)\otimes\alpha(\pi)\, ;
  $$
  whenever $\Gg$ acts by~$\rho$ on~$V$, we have
  $$
  \alpha(\bar{\rho})=\overline{\alpha(\rho)},
  $$
  where $\bar{\rho}$ is the same action as~$\rho$ but viewed as a
  representation on the conjugate vector space; and whenever we have a
  linear map $u\colon V\to W$ such that
  $$
  u(\rho(g)v)=\pi(g)u(v)
  $$
  for all $g\in \Gg$ and $v\in V$, then we have
  $$
  u\circ \alpha(\rho)=\alpha(\pi)\circ u
  $$
  as linear maps from $V$ to~$W$.''
  \par
  Then there exists a unique element $g\in \Gg$ such that
  $\alpha(\rho)=\rho(g)$ for all actions $\rho$ of~$\Gg$.
\end{theorem}

More generally, note that the ``set'' of all data of all $\alpha(\rho)$
of the type considered in this theorem can naturally be used to form a
group (with $(\alpha\beta)(\rho)=\alpha(\rho)\circ \beta(\rho)$), and
then the result identifies the group~$\Gg$ with these data. 

In a converse direction, the main theorem of the theory of tannakian
categories establishes a list of conditions on a suitable category which
guarantees that it is ``equivalent'' to the category of representations
of a group~$\Gg$ (although the context is that of algebraic groups, such
as $\GL_n(\Cc)$, instead of compact groups). A key property to apply the
``reconstruction theorem'' is that one must be able to associate to each
object~$M$ a finite-dimensional vector space $\omega(M)$ (corresponding
to the abstract space on which the group acts), and one needs to have
defined a bilinear operation on these objects, say $M\star N$, in such a
way that $\omega(M\star N)=\omega(M)\otimes\omega(N)$. Such an
``assignment'' $\omega$ is called a \emph{fiber functor}; it is not
unique, and its construction may be a delicate matter.

In the applications in this book (following the idea of Katz
in~\cite{mellin}), the objects that will correspond in this abstract way
to the actions of~$\Gg$ on vector spaces are certain perverse sheaves,
and the operation $\star$ is a form of algebraic convolution which
respects the corresponding usual convolution operation on trace
functions.

For the classical form of Tannaka duality for compact groups, we refer
to the presentation by Joyal and
Street~\cite[\S\,1]{joyal-street}. For an accessible treatment of
tannakian categories, emphasizing the natural evolution from Galois
theory, we refer to the book~\cite{szamuely} of Szamuely.

\section{Frequently asked questions}

We conclude by trying to answer some natural questions that an
analytically-minded reader of little faith may raise:

\begin{itemize}
\item \emph{Is it possible to describe trace functions (or the
    underlying algebraic objects) ``by generators and relations'', by
    listing a number of basic examples and a list of operations
    preserving trace functions, so that all trace functions are obtained
    from these basic data in finitely many steps?}
  \par
  \medskip
  \par
  It is true that in many applications to analytic number theory, the
  sheaves or trace functions which occur are constructed precisely in
  such a way (e.g., starting from an additive character, replacing the
  variable by a polynomials, taking the Fourier transform,~etc).

  However, it seems extremely unlikely that one could provide a
  satisfactory and rigorous version of such an idea, for instance
  because it is known that there are $q$ geometrically irreducible
  middle extension sheaves of rank~$2$ on the projective line
  over~$\Ff_q$ with~$4$ singular points and principal tame local
  monodromy at each point (see for
  instance~\cite[Prop.\,7.1]{deligne-flicker}; the proof of this fact
  relies on automorphic methods).  All these sheaves have bounded
  complexity as $q$ varies. However, only six such sheaves are
  explicitly known (they are associated to certain elliptic curves
  over~$\Ff_q(t)$ with four singular fibers), as shown by
  Beauville~\cite{beauville-stables}.

  Since operations on sheaves tend to increase the complexity in general
  (although in a controllable manner), it seems very difficult to
  imagine how one could construct the ``other'' $q-6$ sheaves in a
  straighforward way.
  
\item \emph{Why are perverse sheaves essential to the results of this
    book?  Why can one not (even in the simplest cases, such as
    exponential sums parameterized by multiplicative characters) work
    around the requirement to use such objects in a way similar to the
    previous papers of Fouvry, Kowalski and Michel?}
  \par
  \medskip
  \par
  The simplest reason for this (not the only one) is that the use of
  tannakian methods (which is the only way we know to produce the
  symmetry group for arithmetic Fourier transforms) depends on applying
  many times a number of operations which will have uncontrollable
  effect on the type of complex we work with, even when starting with a
  single lisse sheaf.
  \par
  More technically, the same tannakian idea requires the construction of
  an \emph{abelian} category (which will ``be'' the category of
  representations of the symmetry group); general complexes do not form
  an abelian category, whereas perverse sheaves form one---certainly
  the best known abelian category beyond that of lisse sheaves.
\item \emph{Conversely, if perverse sheaves are so natural and have such
    remarkable properties, and suffice to describe all trace functions,
    why not dispense with general complexes then?}

\medskip

  Here the issue is that, although perverse sheaves and their trace
  functions are individually wonderful things, they are not \emph{in
    toto} stable by all the operations that one might want to apply.  In
  particular, if $M$, $M_1$, $M_2$ and $N$ are perverse sheaves, then it
  is \emph{not true} in general that $M_1\otimes M_2$, or $f^*N$, or
  $Rf_!M$, are perverse sheaves (on their respective affine spaces). (A
  significant and highly non-trivial exception, however, is that if $M$
  is perverse on~$\Aa^m$, then its Fourier transform in the sense of
  Deligne is still perverse.) In the case of our applications, the
  problem appears in the definition of the algebraic convolution that is
  used to apply the tannakian formalism---\emph{a priori}, even for $M$
  and~$N$ perverse, their algebraic convolution is simply a complex of
  constructible sheaves.

\item \emph{Why is there no normalization by the size of the sum in a
    formula like~\emph{(\ref{eq-quasi-ortho})}\,?}
  
  \medskip
  
  It is a useful property of perverse sheaves, although surprising at
  first sight, that the definition itself implies a normalization for
  these sums. If $M$ is a perverse sheaf with support~$\Aa^m$ which is
  pure of weight~$0$, then the local eigenvalues at a ``generic'' point
  $x$ of~$k_n^m$ are of weight $-m$, i.e., they are typically of size
  $|k|^{-m/2}$. So the sum~(\ref{eq-quasi-ortho}) is naturally expected
  to be of bounded size, \emph{without normalizing}.
\end{itemize}


\backmatter

\printindex

\printnomenclature[45mm]

\bibliographystyle{abbrv}
\bibliography{bibliography_convolution.bib}

\end{document}